\documentclass[12pt,
               letterpaper, 
               twoside      
               ]{report}

\usepackage{blkarray}
\usepackage{bbold}
\usepackage{array}
\usepackage{wrapfig}
\usepackage[dvipsnames]{xcolor}
\usepackage[utf8]{inputenc}
\usepackage[english]{babel}
\usepackage{mathtools}
\usepackage{amsmath}
\usepackage{amsfonts}
\usepackage{amssymb}
\usepackage{amsthm}
\usepackage{mathrsfs}
\usepackage{graphicx}
\usepackage{tikz}
\usepackage{tikz-cd}
\usepackage{colortbl}
\usepackage{yhmath}
\usetikzlibrary{positioning}
\usetikzlibrary{calc}
\usepackage{url}
\usepackage{multicol}
\usepackage[]{algorithm2e}
\usepackage{float} 
\usepackage{pdfpages} 
\usepackage{makecell}
\usepackage{subcaption}

\usepackage{tablefootnote} 

\usepackage{fancybox}

\usepackage[left=2cm,right=2cm,top=2cm,bottom=2cm]{geometry}
\usepackage{mathtools}

\usepackage{url}
\usepackage[colorlinks,pdftex]{hyperref}

\usepackage{lscape} 
\usepackage{longtable} 

\usepackage{totcount} 

\usepackage{makeidx} 
\makeindex

\usepackage{fancyhdr}
\usepackage{titlesec}

\definecolor{wpicrimson}{cmyk}{0.07,1,0.82,0.26}
\definecolor{wpigray}{cmyk}{0.21,0.11,0.09,0.23}

\RequirePackage[helvetica]{quotchap}

\definecolor{chaptergrey}{cmyk}{0.07,1,0.82,0.26}

\hypersetup{
    colorlinks=true,
    linkcolor=wpicrimson,  
    citecolor=wpicrimson, 
    filecolor=wpigray, 
    urlcolor=wpicrimson  
}

\fancypagestyle{main-style}{
  \fancyhf{} 
  \fancyhead[LE]{\nouppercase{\leftmark}}
  \fancyhead[RO]{\nouppercase{\rightmark}}
  \fancyfoot[LE,RO]{\thepage} 
}

\fancypagestyle{plain-intro}{
  \fancyhf{} 
  \fancyfoot[LE,RO]{\thepage} 
}

\newcommand{\green}[1]{\textcolor{ForestGreen}{#1}}

\newcommand{\red}[1]{\textcolor{Red}{#1}}
\newcommand{\blue}[1]{\textcolor{NavyBlue}{#1}}

\newcommand{\N}{\mathbb{N}}
\newcommand{\Z}{\mathbb{Z}}
\newcommand{\Q}{\mathbb{Q}}
\newcommand{\R}{\mathbb{R}}
\newcommand{\C}{\mathbb{C}}
\newcommand{\F}{\mathbb{F}}

\newcommand{\s}{\mathbb{S}}
\newcommand{\p}{\mathbb{P}}

\newcommand{\inc}{\mathbin{I}}

\DeclareMathOperator{\outdeg}{outdeg}
\DeclareMathOperator{\indeg}{indeg}
\DeclareMathOperator{\disc}{disc}
\DeclareMathOperator{\lcm}{lcm}

\DeclareMathOperator{\Gal}{Gal}

\DeclareMathOperator{\Hom}{Hom}

\DeclareMathOperator{\Mat}{Mat}

\DeclareMathOperator{\tr}{tr}

\DeclareMathOperator{\Span}{span}

\DeclareMathOperator{\EA}{EA}
\DeclareMathOperator{\AG}{AG}
\DeclareMathOperator{\PG}{PG}
\DeclareMathOperator{\GL}{GL}
\DeclareMathOperator{\GQ}{GQ}

\DeclareMathOperator{\OA}{OA}

\DeclareMathOperator{\GH}{GH}
\DeclareMathOperator{\BH}{BH}
\DeclareMathOperator{\GHM}{GHM}
\DeclareMathOperator{\QUH}{QUH}
\DeclareMathOperator{\GW}{GW}
\DeclareMathOperator{\W}{W}

\DeclareMathOperator{\SRG}{SRG}
\DeclareMathOperator{\ssp}{sp}
\DeclareMathOperator{\re}{Re}

\newcommand{\q}{\mathfrak{q}}

\DeclareMathOperator{\diag}{diag}
\DeclareMathOperator{\Char}{char}

\newcommand{\genlegendre}[4]{%
  \genfrac{(}{)}{}{#1}{#3}{#4}%
  \if\relax\detokenize{#2}\relax\else_{\!#2}\fi
}
\newcommand{\legendre}[3][]{\genlegendre{}{#1}{#2}{#3}}

\newtheorem{definition}{Definition}[section]
\newtheorem*{definition*}{Definition}
\newtheorem*{remark*}{Remark}
\newtheorem{remark}{Remark}[section]
\newtheorem{example}{Example}[section]
\newtheorem{lemma}{Lemma}[section]
\newtheorem{proposition}{Proposition}[section]
\newtheorem*{proposition*}{Proposition}
\newtheorem{corollary}{Corollary}[section]
\newtheorem{theorem}{Theorem}[section]
\newtheorem*{theorem*}{Theorem}

\newtheorem{research-problem}{Research problem}
\regtotcounter{research-problem}

\newcounter{Protocol} 
\newtheorem{protocol}[Protocol]{Protocol}{\bf}{\rmfamily}

\newcolumntype{P}[1]{>{\centering\arraybackslash}p{#1}}

\newcommand\undermat[2]{%
  \makebox[0pt][l]{$\smash{\underbrace{\phantom{%
    \begin{matrix}#2\end{matrix}}}_{\text{$#1$}}}$}#2}

\makeatletter
\renewcommand{\cleardoublepage}{%
    \clearpage\if@twoside
    \ifodd\c@page\else
    \thispagestyle{empty} 
     \textit{This page is intentionally left blank.}
    \newpage\if@twocolumn\hbox{}\newpage\fi
    \fi
    \fi
}
\makeatother

\usepackage[symbols,nogroupskip,sort=none]{glossaries-extra} 

\glsxtrnewsymbol[description={The transpose of the matrix \ensuremath{A}}]{At}{\ensuremath{A^{\intercal}}}

\glsxtrnewsymbol[description={The conjugate transpose of the matrix \ensuremath{A}}]{A*}{\ensuremath{A^*}}

\glsxtrnewsymbol[description={The entrywise inverse of the matrix \ensuremath{A}}]{A(-)}{\ensuremath{A^{(-)}}}

\glsxtrnewsymbol[description={The entrywise inverse transpose of the matrix \ensuremath{A}}]{A-}{\ensuremath{A^{-}}}

\glsxtrnewsymbol[description={The tensor product of \ensuremath{A} and \ensuremath{B} whose blocks are \ensuremath{[A\cdot b_{ij}]_{ij}}}]{tens}{\ensuremath{A\otimes B}}

\glsxtrnewsymbol[description={The entrywise product, or Schur product, of the matrices \ensuremath{A} and \ensuremath{B}}]{schur}{\ensuremath{A\circ B}}

\glsxtrnewsymbol[description={The identity matrix of order \ensuremath{n}}]{In}{\ensuremath{I_n}}

\glsxtrnewsymbol[description={The all-ones matrix of order \ensuremath{n}}]{Jn}{\ensuremath{J_n}}

\glsxtrnewsymbol[description={The all-ones column vector of length \ensuremath{n}}]{1n}{\ensuremath{\mathbf{1}_n}}

\glsxtrnewsymbol[description={The quadratic (Hermitian) space generated by the symmetric (Hermitian) matrix \ensuremath{A}}]{<A>}{\ensuremath{\langle A\rangle }}

\glsxtrnewsymbol[description={The quadratic (Hermitian) space generated by the diagonal (Hermitian) matrix \ensuremath{\diag(\alpha_1,\dots,\alpha_n)}}]{<alph>}{\ensuremath{\langle \alpha_1,\dots,\alpha_n\rangle }}

\glsxtrnewsymbol[description={The orthogonal space of  the subset \ensuremath{\mathcal{S}} of a quadratic space}]{orth}{\ensuremath{\mathcal{S}^{\perp}}}

\glsxtrnewsymbol[description={The discriminant of the matrix \ensuremath{A}}]{disc}{\ensuremath{\delta(A),\delta_A}}
\glsxtrnewsymbol[description={The signature of the symmetric rational matrix \ensuremath{A}}]{sig}{\ensuremath{\sigma(A)}}
\glsxtrnewsymbol[description={The Hasse-Minkowski invariant at prime \ensuremath{p} of the rational symmetric matrix \ensuremath{A}}]{HM}{\ensuremath{\varepsilon_p(A)}}
\glsxtrnewsymbol[description={The Hasse-Pall invariant at prime \ensuremath{p} of the rational symmetric matrix \ensuremath{A}}]{HP}{\ensuremath{c_p(A)}}

\glsxtrnewsymbol[description={Euler's totient function}]{tot}{\ensuremath{\varphi(n)}}
\glsxtrnewsymbol[description={The Legendre symbol of \ensuremath{a} at the prime \ensuremath{p}}]{Leg}{\ensuremath{\legendre{a}{p}}}

\glsxtrnewsymbol[description={A primitive complex \ensuremath{m}-th root of unity}]{zet}{\ensuremath{\zeta_m}}
\glsxtrnewsymbol[description={The group of complex \ensuremath{m}-th roots of unity}]{mum}{\ensuremath{\mu_m}}

\glsxtrnewsymbol[description={The \ensuremath{R}-module of \ensuremath{R}-homomorphisms from \ensuremath{M} to \ensuremath{N}}]{Hom}{\ensuremath{\Hom_{R}(M,N)}}

\glsxtrnewsymbol[description={The set of natural numbers, starting at \ensuremath{0}}]{N}{\ensuremath{\N}}
\glsxtrnewsymbol[description={The set of integers}]{Z}{\ensuremath{\Z}}
\glsxtrnewsymbol[description={The set of rational numbers}]{Q}{\ensuremath{\Q}}
\glsxtrnewsymbol[description={The set of real numbers}]{R}{\ensuremath{\R}}
\glsxtrnewsymbol[description={The set of complex  numbers}]{C}{\ensuremath{\C}}
\glsxtrnewsymbol[description={The set of \ensuremath{p}-adic integers}]{Zp}{\ensuremath{\Z_p}}
\glsxtrnewsymbol[description={The set of \ensuremath{p}-adic numbers}]{Qp}{\ensuremath{\Q_p}}

\author{Guillermo N. Ponasso}
\title{Combinatorics of Complex Maximal Determinant Matrices}

\begin{document}

\pagenumbering{roman}
\begin{titlepage}
    \begin{center}

        \vspace{1.5cm}
        
        {\Large\textbf{Combinatorics of Complex Maximal Determinant Matrices}
        }\large
        \vspace{1cm}
        
        by \\
        Guillermo Nuñez Ponasso \\
        
        \vspace{1cm}
        A Dissertation\\
        Submitted to the Faculty\\
        of the\\
        WORCESTER POLYTECHNIC INSTITUTE\\
         in partial fulfillment of the requirements for the\\
        Degree of Doctor of Philosophy\\
        in\\
         Mathematical Sciences\\
         \vspace{1.5cm}
         \begin{minipage}{0.4\textwidth}\centering
         \hrulefill\\
         August 2023
         \end{minipage}

\begin{flushleft}
APPROVED:
\end{flushleft}
\vspace{.5cm}
\begin{tabular}{@{}p{0.45\textwidth}@{\hspace{0.1\textwidth}}p{0.45\textwidth}@{}}
      \begin{flushleft}
        \hrulefill \\
        {\normalfont Padraig Ó Catháin,  Advisor\\
        Worcester Polytechnic Institute \& \\
      Dublin City University }
      \end{flushleft}
      &
      \begin{flushright}
        \hrulefill \\
        William J. Martin\\
        Committee chair\\
       	Worcester Polytechnic Institute \\
      \end{flushright}
\end{tabular}
\begin{tabular}{@{}p{0.45\textwidth}@{\hspace{0.1\textwidth}}p{0.45\textwidth}@{}}
      \begin{flushleft}
        \hrulefill \\
        John Bamberg \\
        University of Western Australia \\
      \end{flushleft}
      &
      \begin{flushright}
        \hrulefill \\
        Ada Chan \\
        York University\\
      \end{flushright}
\end{tabular}
\begin{tabular}{@{}p{0.45\textwidth}@{\hspace{0.1\textwidth}}p{0.45\textwidth}@{}}
      \begin{flushleft}
        \hrulefill \\
        Gábor N. Sárközy \\
        Worcester Polytechnic Institute \\
      \end{flushleft}
      &
      \begin{flushright}
        \hrulefill \\
        Adam Zsolt Wagner \\
        Worcester Polytechnic Institute \\
      \end{flushright}
\end{tabular}
        
        \vfill
        
    \end{center}
\end{titlepage}

\pagebreak
\thispagestyle{empty}
\textit{This page is intentionally left blank}
\pagebreak
\tableofcontents

\glsaddall
\printunsrtglossary[type=symbols,style=long]

\chapter*{Introduction}
\addcontentsline{toc}{chapter}{Introduction}
\thispagestyle{plain-intro}

 Our starting point is Hadamard's determinant inequality, which states that an $n\times n$ matrix $M$,  whose entries are taken from the complex unit disk, satisfies
\[|\det(M)|\leq n^{n/2}.\]

A matrix $H$ meets Hadamard's bound with equality if and only if the entries of $H$ all have absolute value equal to $1$, and $HH^*=nI_n$. Such a matrix $H$ is called an Hadamard matrix. If the entries of $H$ are restricted to some subset of the complex unit circle, for example $\{+1,-1\}$, then Hadamard's bound cannot always be achieved. It is easy to see that a $\pm 1$ matrix of order $n$ cannot be Hadamard if $n$ is odd, and in fact for $n>2$, the order $n$ must be a multiple of $4$. Hadamard's maximal determinant problem asks to find the maximum value of the determinant of a $\pm 1$ matrix. A matrix achieving the maximum is called a maximal determinant matrix, or $D$-optimal design.\\ 

Real Hadamard matrices, i.e. $\pm 1$ Hadamard matrices, are not only interesting for their own sake, but also because of their wide range of applicability. The original motivation of Hadamard's bound came from the classical Fredholm theory of integral equations. During the 20th century, $\pm 1$ Hadamard matrices found an impressively wide range of applications, ranging from signal processing, coding theory, and cryptography, to the statistical theory of design of experiments. Maximal determinant matrices are also applied in statistics. Certain experimental designs are described by $\pm 1$ matrices, and maximising the determinant corresponds to minimising the variance of the error of the estimators  \cite{Greek-17}.\\

A natural generalisation of real Hadamard matrices is the class of Butson-type Hadamard matrices. These are Hadamard matrices whose  entries are roots of unity. A Butson-type Hadamard matrix of order $n$ with entries in the set $\mu_m$, of $m$-th roots of unity, is called a $\BH(n,m)$ matrix. In particular, the set of $\BH(n,2)$ matrices is precisely the set of real Hadamard matrices of order $n$. Butson-type Hadamard matrices are perhaps the most important class of complex Hadamard matrices. For this reason, one of the main topics of this dissertation will be an extension of Hadamard's maximal determinant problem to matrices with entries in $\mu_m$.\\
 
From the point of view of applications, complex Hadamard matrices have been gaining more relevance in recent years. To mention a few applications, complex Hadamard matrices have been used to disprove conjectures in harmonic analysis \cite{Tao-FugledeConjecture}, and they have also been applied in operator theory \cite{Popa-Subfactors}. The most notable application of complex Hadamard matrices occurs in the fields of quantum information theory and quantum computation, where these matrices play a fundamental role \cite{Banica-Invitation, Bengtsson-Zyczkowski-GeoQuantumStates}. Very recently, Butson-type matrices have also found applications in coding theory \cite{Ronan-ButsonCodes}.\\

 To study maximal determinant matrices over the $m$-th roots we use a wide range of theoretical tools, among which quadratic forms and algebraic number theory are the most prevalent. Other techniques we use belong to matrix analysis, representation theory, character theory, association schemes, finite geometry, Diophantine approximation, and algebraic geometry. The last chapter of the thesis has quite a different flavour from the rest, as it consists of an application of finite geometry to privacy in communications. Nonetheless, quadratic forms will make an appearance in all chapters, forming the main theoretical backbone of the thesis. When presenting classical material, and especially the material on quadratic forms, we have made great efforts to make it accessible to an audience of combinatorialists.\\
 
 It appears to be a tradition in the area of Hadamard matrices to include a list of research problems in theses or books on the subject. We partake in this tradition by including a total of \total{research-problem} research problems, to keep the interested reader and ourselves busy.\\

We highlight the material that is taken from one of our papers:

\begin{itemize}
\item Theorem \ref{thm-MyHM} together with our proof of the Bruck-Ryser-Chowla Theorem in Chapter \ref{chap-BRC}, will appear in \cite{InvariantsPaper}.
\item Theorem \ref{thm-QUH-Nonexistence} and part of the exposition in Section \ref{sec-SplittingIdeals} and Section \ref{sec-NonExQUH} appeared in \cite{QUH-paper}.
\item Theorem \ref{thm-QUHMorphism} appeared in \cite{QUH-paper}.
\item Most of the material in Chapter \ref{chap-Maxdet} and \ref{chap-ASMaxdet}, is in preparation to be submitted for publication as a single-author paper.
\item Most of the results in Chapter \ref{chap-UPIR}, with the exception of Section \ref{sec-PIR} and Section \ref{sec-GQ} appeared in the paper \cite{UPIR-paper}.
\end{itemize}


Every chapter contains a new contribution to the literature. Our convention for the attribution of results is as follows:
\begin{itemize}
\item Unattributed results are the work of the present author. Some of these results are new proofs of known results, when this is the case we include a remark for clarification.
\item In some cases, we have included folklore results or straightforward results without attribution, since it may be hard to point out a particular source for these. In these cases we have added a remark explaining that the results are known.
\end{itemize}
 Below we give an outline of each chapter highlighting their main results.

\section*{Chapter 1: Non-solvability of Gram equations}
\thispagestyle{plain-intro}
The main motivation for this chapter is to present tools to study the solvability of matrix equations of the type
\[XX^*=M,\]
where $M$ is a given Hermitian positive-definite matrix, and $X$ is a square matrix with entries in some subfield of $\C$. Particularly we focus on the case where $M$ is symmetric, and $X$ has rational entries. The study of the equation $XX^{\intercal}=M$, is very interesting from the combinatorial point of view. Indeed, using incidence matrices, many combinatorial structures are equivalent to solutions of such an equation, provided that the entries of $X$ are integral. For example projective planes, and symmetric designs,  can be seen to be equivalent to a $\{0,1\}$ solution to a Grammian equation. Additionally, solutions to Gram matrix equations over the alphabet $\{\pm 1\}$ are interesting in the study of maximal determinant matrices.\\

Using the language of quadratic forms, it is easy to see that the Grammian problem $XX^{\intercal}=M$ is equivalent to a problem of equivalence of quadratic forms. Quadratic forms are in bijection with symmetric matrices, and two quadratic forms given by symmetric matrices $A$ and $B$ are equivalent if and only if the matrices $A$ and $B$ are \index{congruent matrices}{\textit{congruent}}, i.e. there exists an invertible matrix $X$ such that $XAX^{\intercal}=B$. This observation has been tremendously successful in design theory, since the introduction of the powerful machinery of quadratic forms provided many new non-existence results. The main example of such results is the \index{Bruck-Ryser-Chowla Theorem}{Bruck-Ryser-Chowla} (BRC) Theorem \cite{Bruck-Ryser,Ryser-Chowla},

\begin{theorem*}[Bruck-Ryser-Chowla] Suppose that there is a symmetric $2$-$(v,k,\lambda)$ design. Then,
\begin{enumerate}
\item if $v$ is even, $k-\lambda$ must be a perfect square, and
\item if $v$ is odd, there must be a non-trivial solution to the Diophantine equation
\[(k-\lambda)x^2+(-1)^{(v-1)/2}\lambda y^2=z^2.\]
\end{enumerate}
\end{theorem*}

In the design theory literature, most of the expositions of the Bruck-Ryser-Chowla Theorem avoid introducing the theory of quadratic forms, and instead make use of ad-hoc arguments. While some of these arguments can be very elegant, the drawback to them is that they tend to obscure the proof, and hide the fact that determining equivalence of rational quadratic forms is a fairly straightforward and mechanical computational task, not significantly harder than computing a Jordan canonical form. The reason that quadratic forms are sometimes avoided is that the theory can become quite technical if presented in full generality. However, many times in combinatorial applications we may restrict to positive-definite forms over the rational numbers. This restriction avoids several technicalities that may be unpleasant for the non-specialist.\\

In this chapter we introduce the theory of quadratic forms over a general field, and then focus on rational quadratic forms to set up the scene to prove the Bruck-Ryser-Chowla Theorem in the following chapter. We present the invariants of rational quadratic forms, namely the discriminant, the signature, and the Hasse-Minkowski invariants. These are complete invariants of quadratic forms, but we emphasise that for combinatorial applications we only require partial invariants, since many times we only need to disprove the existence of a rational solution to a Grammian equation.\\

 There are two main original contributions to this chapter:
\begin{enumerate}
\item We give an elementary motivation for the Hilbert symbol $(a,b)_K$ over a field $K$ by studying the equation $XX^{\intercal}=M$ for $2\times 2$ matrices over $K$.
\item We give a new elementary, matrix-theoretic proof that the Hasse-Minkowski invariants are partial invariants of quadratic forms.
\item All other results are accompanied by a citation to either the original author of the result, or a reference text that includes it.
\end{enumerate} 

This new proof will be included in a paper in collaboration with Oliver Gnilke and Padraig Ó Catháin, that is currently in preparation \cite{InvariantsPaper}.

\section*{Chapter 2: Invariants of Quadratic Forms in Design Theory}
\thispagestyle{plain-intro}
In this section we introduce basic concepts from design theory, and give new proofs of the Bruck-Ryser-Chowla Theorem and the\index{Bose-Connor Theorem} Bose-Connor Theorem, one of the new proofs of the Bruck-Ryser-Chowla Theorem that we give will appear in the paper \cite{InvariantsPaper}. Additionally, we present an application of the Bose-Connor Theorem to the existence problem of certain $\pm 1$ maximal determinant matrices.\\

The Bose-Connor Theorem \cite{Bose-Connor} extends the Bruck-Ryser-Chowla Theorem to the class of group-divisible designs. Both the proofs of the Bruck-Ryser-Chowla Theorem and the Bose-Connor Theorem establish non-existence conditions by assuming the existence of a design and deriving a convenient Gram matrix through a series of manipulations. However, this approach has two disadvantages. Firstly, the results in the Bruck-Ryser-Chowla Theorem and the Bose-Connor Theorem may hold more generally as statements concerning the congruence of two rational matrices. This poses a problem since in certain variant applications of the Bruck-Ryser-Chowla or the Bose-Connor Theorem, our matrices may not satisfy some of the properties of incidence matrices of designs, such as constant row-sum. Secondly, the matrix manipulations in the proof of the Bruck-Ryser-Chowla Theorem are non-obvious, and those in proof of the Bose-Connor Theorem are particularly intricate and hard to follow.\\

In our original contribution, we used ideas from the theory of association schemes to give a unified method to prove both the Bruck-Ryser-Chowla Theorem and the Bose-Connor Theorem. We believe that our new proof of the Bose-Connor Theorem is more straightforward and natural than the original.

\section*{Chapter 3: Hermitian Forms and Determinant Obstructions}
\thispagestyle{plain-intro}
In this chapter, we extend the theory of quadratic forms presented in Chapter 1 to Hermitian forms. Our motivation for this extension is to study the equation $XX^*=M$, where now $M$ and $X$ can have complex entries. To study Hermitian forms, we use a reduction to quadratic forms due to Jacobson \cite{Jacobson-Reduction}. While this reduction is well-known to number theorists, its application in the context of combinatorics seems to be a novel contribution. Hermitian forms had already been considered in combinatorics, and in \cite{Brock} an equivalent, but less effective, approach to study Hermitian forms had been developed. The method in \cite{Brock} gives several conditions for the solvability of $XX^*=M$ where $M$ is Hermitian and has coefficients in an imaginary number field $K$. We show that only one of those conditions is necessary and sufficient. Namely, the equation $XX^*=M$ is solvable over $K=k[\sqrt{-d}]$, where $k$ is a number field, if and only if $\det(M)$ can be written as $x^2+dy^2$, where $x,y\in k$.\\

Previous results in the literature only worked with quadratic extensions. For imaginary quadratic extensions of $\Q$, the theory of rational quadratic forms gives a simple answer to determine the solvability of $x^2+dy^2$ with $x,y\in\Q$. However, we will be interested in cyclotomic extensions of degree $>2$, and in biquadratic extensions. The study of these cases will require knowledge of the splitting of prime ideals in each extension.  For this purpose we will give a brief introduction to algebraic number theory, which gives us general techniques to study the behaviour of prime ideals over an extension of number fields.\\

There are two novel contributions in this chapter:\\

In \cite{Winterhof-NonexistenceButson} Winterhof gave necessary conditions that $n$ must satisfy in order for a $\BH(n,p^f)$ or $\BH(n,2p^f)$ matrix to exist, whenever $p\equiv 3\pmod{4}$ is a prime, and $f\geq 1$ is an integer. We extended Winterhof's result to include also the case $p\equiv 1\pmod{4}$ and give a unified proof for both cases.

\begin{theorem*}
Let $p$ be an odd prime, and $f\geq 1$ an integer. Suppose that $n=p^{\ell}a^2 m$ is odd, where $p\nmid m$, and $m$ is square free. Then if $q\mid m$ and $q^t\equiv -1\pmod{p^f}$ for some integer $t$, then there cannot exist a $\BH(n,p^f)$ or a $\BH(n,2p^f)$.
\end{theorem*}

We give non-existence conditions $\QUH(n,m)$ matrices, introduced by Fender, Kharaghani, and Suda in \cite{Fender-Kharaghani-Suda}. These are complex Hadamard matrices with entries in the set
\[\left\{\frac{1\pm \sqrt{-m}}{\sqrt{m+1}},\frac{-1\pm\sqrt{-m}}{\sqrt{m+1}}\right\}.\]
This is the first non-existence result for a non-Butson-type class of Hadamard matrices, and it exhibits the great generality of the techniques we present. This result appeared published in our paper \cite{QUH-paper}.

\begin{theorem*}
Let $m$ be a positive integer, such that neither $m$ nor $m+1$ are perfect squares. Write $m=(m_0)^2 a$ and $m+1=(m_0')^2 b$, where $a,b>1$ are square-free. Let $n=(n_0)^2t$ be an odd integer, where $t$ is square-free. Suppose $p$ is an odd prime, coprime to both $m$ and $m+1$ and $p\mid t$. If
\[\legendre{-a}{p}=-1,\text{ and } \legendre{b}{p}=1,\]
then there cannot exist a $\QUH(n,m)$.
\end{theorem*}

\section*{Chapter 4: A survey on Butson-type Hadamard matrices}
\thispagestyle{plain-intro}
This chapter provides a survey on the existence of Butson-type Hadamard matrices. While our focus is on $\BH(n,p)$ matrices, we also explore general constructions, and have a section dedicated to morphisms of Hadamard matrices. The word morphism is used to refer to a mapping or partial mapping between different sets of Hadamard matrices, often obtained through isomorphisms of algebras.\\

Many of the constructions shown here are previously known results, it is worth noting however that the literature on Butson-type Hadamard matrices is quite scattered, making it challenging to find constructions and examples for  a specific $\BH(n,m)$ matrix. We include tables summarising the current state of the art in terms of existence of $\BH(n,m)$ for $m=3,4,5,6$, to the best of our knowledge. For each matrix listed, we either present the construction method within the chapter, or provide a reference to the source where the construction can be found. We hope that in this way the reader can obtain each example on their own, as well as most of the known $\BH(n,m)$ matrices even when not tabulated.\\

There are four novel contributions in this chapter:\\

 In the late 19th century, Italian mathematician Umberto Scarpis found a construction for real Hadamard matrices of order $q(q+1)$, where $q$ is a prime power and $q+1$ is the order of a real Hadamard matrix \cite{Scarpis}. This construction was later rediscovered by Seberry \cite{Seberry-GeneralisedHadamard-1980}, and applied to generalised Hadamard matrices, or GH matrices. We show here that the Scarpis construction can be also applied to Butson-type Hadamard matrices. The theorem in its full generality reads as follows
 
\begin{theorem*} 
Let $H$ be a $\BH(n+1,m)$, and suppose that there is a $\GH(n,G)$ where $|G|=n$. Then there is a $\BH(n(n+1),m)$ matrix.
\end{theorem*}

As a corollary, we have that if there is a $\BH(q+1,m)$ matrix and $q$ is a prime power, then there is a $\BH(q(q+1),m)$ matrix. For example, we can show the existence of $\BH(90,6)$ matrices which, to the best of our knowledge, was previously unknown.\\

In an unpublished paper, Warwick de Launey showed the existence of $\BH(2^t\cdot 3,3)$ matrices for all $t\geq 1$. This construction is only outlined in his paper \cite{DeLauney-GHMSurvey}, and the present author has been unable to find the precise construction. We present an alternative formulation of de Launey's construction, and show that it gives the existence of $\BH(2^t\cdot 3,3)$ matrices provided that a certain sequence of matrices with entries in the third roots of unity exists.

\begin{proposition*}
If there is a sequence of matrices $K_t$ of order $2^t$  for $t\geq 0$ with entries in $\{1,\omega,\omega^2\}$, satisfying $K_t(K_{t-1}^*\otimes J_2)=(-1)^{t}J_{2^{t}}$, and the following recurrent Gram matrix equations $K_0K_0^*=1$,
\[K_tK_t^*= \left[
\begin{array}{cc}
2 K_{t-1}K_{t-1}^* & (-1)^t J_{2^{t-1}}\\
(-1)^t J_{2^{t-1}} & 2K_{t-1}K_{t-1}^*
\end{array}
\right]\text{ for } t\geq 1.\]
 Then the matrix
\[H_t=\left[
\begin{array}{cc}
K_{t-2}^{(2\times 2)} & K_{t-1}^{(1\times 2)}\\
K_{t-1}^{(2\times 1)} & K_t
\end{array}
\right]=\left[
\begin{array}{cc}
K_{t-2}\otimes J_2 & K_{t-1}\otimes J_{1,2}\\
K_{t-1}\otimes J_{2,1} & K_t
\end{array}
\right],
\]
is a $\BH(2^t\cdot 3, 3)$ for every $t\geq 2$.
\end{proposition*}

We construct a new morphism from the class of $\QUH$ matrices to real Hadamard matrices, provided the existence of skew Hadamard matrices. Skew Hadamard matrices are real Hadamard matrices with the property that $(H-I_n)^{\intercal}=-(H-I_n)$. This is based on our paper \cite{QUH-paper}, in collaboration with Heikoop, Pugmire, and Ó Catháin.
\begin{theorem*}
If there exists a skew Hadamard matrix of order $m+1$, then there is a morphism $\QUH(n,m)\rightarrow\BH(n(m+1),2)$.
\end{theorem*}

We show the existence of $\BH(12p,p)$ matrices for every $p>197$, improving on the previous lower bound of $104857600=(10\cdot 2^{10})^2$.  We do this computationally, by checking the existence of $\BH(12p,p)$ for all primes between $197$ and $(10\cdot 2^{10})^2$.
 \begin{theorem*}
 There is a $\BH(12p,p)$ matrix for all primes $p>197$.
 \end{theorem*}
\thispagestyle{plain-intro}
\section*{Chapter 5: Complex Maximal Determinant Matrices}
\thispagestyle{plain-intro}
Hadamard's maximal determinant problem asks us to find the maximum value of the determinant of a $\pm 1$ matrix of order $n$. In this chapter we extend this problem to matrices with entries over the $m$-th roots of unity. For general values of $m$ we find the following lower bound:

\begin{theorem*}
If there is a $\BH(n,m)$, then there is a matrix of order $n^2+1$ with entries in the $m$-th roots of unity $M$ such that
\[|\det(M)|\geq (n+1)n^{n^2}.\]
\end{theorem*}

Additionally, we show that the  determinant upper bound of Barba \cite{Barba-Bound} holds for complex matrices. In particular, letting 

\[\sigma_m(n):=\min\left\{\left|\sum_{i=1}^{n}\zeta_m^{a_i}\right|: a_i\in \{0,\dots,m-1\}, \text{ for } 1\leq i \leq n\right\},\]

we have

\begin{theorem*}Let $M$ be an $n\times n$ matrix with entries in the set $\mu_m$ of $m$-th roots of unity. Suppose that $\sigma_m(n)$ is positive. Then,
\[|\det M|\leq \sqrt{(n+(n-1)\sigma_m(n))}(n-\sigma_m(n))^{(n-1)/2}.\]
Furthermore there is equality in the bound if and only if there exists a diagonal matrix $\Delta$ with non-zero entries of modulus $1$, such that $B=\Delta^* M$ satisfies $BB^*=(n-\sigma_m(n))I_n+\sigma_m(n)J_n$.
\end{theorem*}
When $m=2$, $3$, and $4$, we have that $\sigma_m(n)\in\{0,1\}$ for all $n$. In these cases, a matrix $M$ meets the Barba bound with equality if and only if $M$ is equivalent to a matrix $B$ satisfying $BB^*=(n-1)I_n+J_n$. Matrices satisfying this matrix equation are called Barba matrices.\\

The case $m=2$ corresponds to Hadamard's maximal determinant problem, and we give an outline of this case before turning into the cases $m=3$ and $m=4$. The study of the case $m=3$ is novel, along with all the results presented here. We believe that this case is the most challenging and interesting, so we dedicate more attention to it. We include the following results:
\begin{enumerate}
\item Lower bounds for the determinant of a matrix with entries in $\{1,\omega,\omega^2\}$ at certain orders $n\equiv 1\pmod{3}$ and $n\equiv 2\pmod{3}$, using techniques from cyclotomy.
\item A classification of Barba matrices over the third roots which belong to the Bose-Mesner algebra of a strongly regular graph.
\item Several examples of maximal determinant matrices at small orders.
\end{enumerate}
 The case $m=4$ had been investigated by Cohn \cite{Cohn-ComplexDOptimal}, where by means of the Turyn morphism, he related the maximal determinant problem over $\pm 1$ matrices to matrices over the fourth roots. Here, we apply this idea to find the following infinite family of Barba matrices from a known family of $\pm 1$ maximal determinant matrices:
\begin{theorem*}
Let $q$ be a prime power, then there is a Barba matrix of order $q^2+q+1$ over the fourth roots.
\end{theorem*} 
We also include some sporadic examples of small maximal determinant matrices over the fourth roots, found computationally.\\
 
 We conclude the chapter with a discussion of techniques to prove the maximality of a candidate matrix. In particular, we show that the pruning technique of Moyssiadis and Kounias in \cite{Greek-17} extends to the complex case, and we find the following arithmetic condition for candidate Gram matrices:
 
\begin{proposition*}\normalfont Let $M=XX^*$, where $X$ is an $n\times n$ matrix with entries in $\{1,\omega,\omega^2\}$. Let 
 \[p_M(x)=x^n-n^2x^{n-1}+a_2x^{n-2}+\dots+a_{n-1}x+a_n.\]
 Then, $a_i\in \Z$ for all $i=2,\dots,n$, and $3^{i-1}\mid a_i$.
 \end{proposition*} 

\section*{Chapter 6: Maximal Determinants in Association Schemes}
\thispagestyle{plain-intro}
This chapter studies the existence of certain types of maximal determinant matrices belonging to the Bose-Mesner algebra of an association scheme. To do this, we consider the problem of solving  $XX^* =M$, where both $X$ and $M$ belong to some Bose-Mesner algebra, and the entries of $X$ have modulus $1$. We characterise the solutions to this problem with the following result

\begin{theorem*}
Let $M=\sum_{k=0}^{d} \alpha_kA_k$ be a matrix in the Bose-Mesner algebra $\mathcal{A}$ of a $d$-class association scheme. Then, $M=NN^*$ where $N=\sum_k \beta_k A_k$ if and only if for all $k=0,1,\dots, d$,
\[\beta^*(WP_k)\beta = \alpha_k,\]
where $\beta=(\beta_0,\beta_1,\dots,\beta_d)^{\intercal}$, and $W$ is the permutation matrix given by the involution $i\mapsto i'$.
\end{theorem*}
Here the matrices $A_i$ are the adjacency matrices of the association scheme, and the value $i'$ is the unique index in $\{0,1,\dots, d\}$ such that $A_{i}^{\intercal}=A_{i'}$. The matrices $P_k$ above are given by 
\[[P_k]_{ij}=p_{ij}^k,\]
where $p_{ij}^k$ are the intersection parameters of the association scheme.\\

The result above shows that the problem $XX^*=M$ over a Bose-Mesner algebra is equivalent to a system of quadratic polynomial equations. We study this problem computationally using Gröbner basis. With this, we reproduce some of the results of Chan in \cite{Chan-ComplexHadamard} on the existence of Hadamard matrices, and of Ikuta and Munemasa in \cite{Ikuta-Munemasa-Bordered} on the existence of Bordered Hadamard matrices. Additionally, we study the existence of Barba matrices on strongly regular graphs. For instance, we have

\begin{proposition*}\normalfont Let $\{I,A,J-I-A\}$ be the adjacency matrices of a conference graph of order $v$. Let 
\[M=I+\alpha A+\beta(J-I-A).\]
Then,
\begin{itemize}
\item[(i)]  $M$ is the core of a bordered Hadamard matrix if and only if $\alpha=\pm i$ and $\beta=\mp i$ or $\alpha=\overline{\beta}$ has the minimal polynomial 
\[p(x)=x^2+\frac{2}{t}x+1,\]
where $t=k=(v-1)/2$, (cf. Ikuta and Munemasa \cite{Ikuta-Munemasa-Bordered}).
\item[(ii)] $M$ is a Barba matrix if and only if 
\[\alpha=\frac{-1\pm i\sqrt{t^2-1}}{t},\]
and $\beta=\overline{\alpha}$, where $t^2+(t+1)^2=v$.
\item[(iii)] $M$ is an Hadamard matrix if and only if 
\[\alpha=\frac{-1\pm i\sqrt{t^2-1}}{t},\]
and $\beta=\overline{\alpha}$, where $(t+1)^2=v$.
\end{itemize}
\end{proposition*}

We characterise Hadamard matrices in the Bose-Mesner algebra of an asymmetric $2$-class association scheme:

\begin{theorem*}
Let $\mathcal{X}$ be an asymmetric $2$-class association scheme with parameters $(v,k,\lambda,\mu)=(4r+3,2r+1,r,r+1)$. Let $\{I,A,A^{\intercal}\}$ be the $01$-generators of the Bose-Mesner Algebra of $\mathcal{X}$, then the matrix
\[H=I+\alpha A+ \beta A^{\intercal},\]
is a complex Hadamard matrix if and only if

\begin{itemize}
\item[(i)] One of $\alpha$ or $\beta$ has value $1$, and the other has minimal polynomial
\[p_r(t)=t^2+\frac{2r+1}{r+1}t+1.\]
\end{itemize}
\item[(ii)] $H=I_3+\omega (J_3-I_3)$, where $\omega$ is a primitive third root of unity.
\end{theorem*}

\section*{Chapter 7: User-Private Information Retrieval and Finite Geometry}
\thispagestyle{plain-intro}
In this chapter, quite different in spirit, we study applications of finite geometry to privacy. The material is mostly based on our paper \cite{UPIR-paper}, in collaboration with Gnilke, Greferath, Hollanti, Ó Catháin, and Swartz. We begin with a short introduction to Private Information Retrieval (PIR), and discuss some of its shortcomings. These motivate the introduction of an alternative paradigm known as User-Private Information Retrieval (UPIR). In UPIR we consider a network of users who wish to retrieve information from a server. In order to preserve their privacy, users submit queries on behalf of each other. The goal of UPIR is to keep private the identity of the users who originate each request.\\

The privacy of the users may be compromised by a coalition of eavesdroppers in the network. The underlying structure by which communications between users are established becomes very important in preserving privacy. A UPIR system is based on an incidence structure, where two users are able to communicate directly if and only if they lie in the same block of the incidence structure. We investigate previous protocols based on BIBDs and projective planes, and exhibit some of their vulnerabilities. To overcome the issues with these protocols we propose a novel one based on geometries known as generalised quadrangles (GQs).\\

 In the chapter, we introduce the reader to the basics on GQs, and show how to construct the classical families of GQs using quadratic and Hermitian forms over finite fields. After, we proceed to analyse the security of the GQ-UPIR schemes, and show that they provide a much more secure communication scheme than the ones previously considered in the literature.
\thispagestyle{plain-intro}

\chapter*{Acknowledgments}
\addcontentsline{toc}{chapter}{Acknowledgments}

\vspace{-1cm}
I must express my gratitude to many people who have helped me during my journey as a doctoral student. First and foremost I must thank my advisor, Professor Padraig Ó Catháin, for his constant support and guidance over the years. I met Padraig as a research student in Finland shortly after my bachelor's, and ever since he has helped me grow as a mathematician, a writer, and a researcher. His commitment and patience during my PhD helped me continue to progress despite the great challenge of being in different countries during the larger part of my studies.\\

A very heartfelt word of gratitude goes to Professor William J. Martin, who has been extremely supportive and encouraging during my years in WPI. I am thankful for his commitment and enthusiasm in organising the \textit{WPI Discrete Mathematics Seminar}, where I had the opportunity to learn much, and meet great mathematicians. I am also very grateful for the time that he took in teaching me many topics in combinatorics, and quantum information theory, and for his generosity and hospitality in hosting me and other PhD students for beer and research evenings!\\

I am also very grateful to Prof. John Bamberg, Prof. Ada Chan, Prof. Gábor Sárközy and Prof. Adam Zsolt Wagner for serving as committee members and taking the time to read this thesis. Their careful and insightful reviews have improved many aspects of this thesis. A special thank you goes to Prof. Sárközy for his kindness in teaching me about the Szemerédi regularity Lemma.\\

Among the staff and faculty in the WPI Mathematical Sciences department I must thank Mike Malone for helping me get started with the \texttt{math2018} and \texttt{math2022} high-performance computers at WPI. Many of the computational results in this thesis ran in those machines. I thank as well Rhonda Podell for her help with administrative issues during these years, and for her kindness and support. I would also like to thank Prof. Sarah Olson for her excellent job as department head and for her invaluable help during these years. An honourable mention goes to Greg Aubin, our cheerful custodian, who always has a word of support for a fatigued PhD student trying to make things happen!\\

I would also like to thank the staff at the WPI Library for their amazing work finding bibliographical sources. Thanks to them I was able to gain access to historical documents on the maximal determinant problem, and to some more recent but obscure papers.\\

I have a truly marvelous list of friends, which unfortunately this page is too narrow to contain - I thank you all. A special mention goes to Mason DiCicco, Derek Drumm, Shirshendu Ganguly, Forrest Miller, Elisa Negrini, Matteo Pintonello, Ieva Savickyt\.{e}, Jidapa Thadajarassiri, Nick Vea, Tony Vuolo, and Pooya Yousefi for the wonderful moments together and the many memories we share. I thank Ben Gobler, Jessica Wang and Ethan Washock, for making me proud to be their unofficial mentor for a couple of years.\\

Last but not least I would like to thank my family. Gracias a mi madre, Laura, por enseñarme el valor del esfuerzo y a encontrar la belleza en todas las cosas, y gracias a mis abuelos, Lela y Lelo, por las largas charlas sosegadas tomando mate en el jardín.


\cleardoublepage
\chapter{ Non-solvability of Gram equations}\label{chap-GramEquations}

\pagenumbering{arabic}

Many of the standard objects of study in combinatorics and design theory, such as designs, graphs and finite geometries, can be described by incidence matrices. This allows for the application of powerful algebraic techniques to combinatorial problems. A basic but very useful fact is that the \index{matrix!Gram}{\textit{Gram matrix}} of an incidence matrix counts pairwise intersections. Recall that the Gram matrix of a matrix $X$ is $G=XX^*$, where ${X^*}$ is the conjugate-transpose of $X$. \vspace{12pt}

\begin{minipage}{0.2\textwidth}\hspace{24pt}
\begin{tikzpicture}[scale=1.5]
  \coordinate (1) at (0,0);
  \coordinate (2) at (0,1);
  \coordinate (3) at (1,1);
  \coordinate (4) at (1,0);

  \draw (1) -- (2) -- (3) -- (4) -- cycle;
  \draw (1) -- (3);
  \draw (2) .. controls (1.75,2) and (2,1.75) .. (4);

  \foreach \point/\position in {1/below left, 2/above left, 3/above right, 4/below right}
    \fill (\point) circle (2pt) node[\position] {$\point$};
\end{tikzpicture}
\end{minipage}
\begin{minipage}{0.8\textwidth}
\[\hspace{-12pt}
X=
\begin{blockarray}{ccccc}
&1 & 2 & 3 & 4\\
\begin{block}{c[cccc]}
12 &\phantom{-}1 & 1 & 0 & 0\phantom{-}\\
34 &\phantom{-}0 & 0 & 1 & 1\phantom{-}\\
13 &\phantom{-}1 & 0 & 1 & 0\phantom{-}\\
24 &\phantom{-}0 & 1 & 0 & 1\phantom{-}\\
14 &\phantom{-}1 & 0 & 0 & 1\phantom{-}\\
23 &\phantom{-}0 & 1 & 1 & 0\phantom{-}\\
\end{block}
\end{blockarray},\ \ 
G=XX^*=
\begin{bmatrix}
2 & 0 & 1 & 1 & 1 & 1\\
0 & 2 & 1 & 1 & 1 & 1\\
1 & 1 & 2 & 0 & 1 & 1\\
1 & 1 & 0 & 2 & 1 & 1\\
1 & 1 & 1 & 1 & 2 & 0\\
1 & 1 & 1 & 1 & 0 & 2
\end{bmatrix}
\]
\end{minipage}
In the example above, we see the affine plane of order two, its incidence matrix $X$, and the Gram matrix $G$ of $X$. In this case the entries of $G$ count the cardinality of line intersections. Many combinatorial objects can be characterised by the Gram matrix of their incidence matrix. So, determining the existence of the object of interest becomes equivalent to finding a solution to the equation $XX^*=M$ for a given $M$. Typically, the solution $X$ is required to have entries in the set $\{-1,0,1\}$, or in some other finite subset of $\C$. The celebrated \index{Bruck-Ryser-Chowla Theorem} Bruck-Ryser-Chowla Theorem (often abbreviated as BRC Theorem) \cite{Bruck-Ryser,Ryser-Chowla} gives non-existence conditions for symmetric $2$-$(v,k,\lambda)$ designs by studying the solvability of the equation $XX^*=(k-\lambda)I_n+J_n$. This was a very successful and original application of algebraic and number-theoretical techniques to combinatorics. We will study this theorem, and related results, in detail in Chapter \ref{chap-BRC}.\\

The first condition we find that $M$ must satisfy is the following: Recall that a matrix $M$ with complex coefficients is Hermitian, if and only if $M=M^*$. Then clearly $M$ must be Hermitian, if $M=XX^*$. Furthermore, recall that regarding $\C^n$ as an inner product space with the standard inner product, an Hermitian matrix $M$ is positive-definite if and only if $\langle x,Mx\rangle =x^*Mx>0$ for all non-zero $x\in\C^n$. It is easy to see as well that if $M=XX^*$ and $\det(M)\neq 0$, then $M$ is positive-definite. This is standard linear algebra, but we will recall the details in Section \ref{sec-MatAnalysis}.\\
 
In this chapter we will investigate a series of  quite versatile techniques to decide the solvability of the equation $XX^*=M$ for a given Hermitian positive-definite matrix $M$, and where $X$ has its entries in some subring $R$ of $\C$.  The problem of solving $M=XX^*$ can be stated in the language of quadratic and Hermitian forms, and solved completely over certain classes of fields, most importantly over the rational field $\Q$.\\

 In the design theory literature, the proof of the BRC theorem is often presented in a way that avoids discussing the theory of quadratic forms substantially. This is perhaps due to the belief that computations involving the invariants of quadratic forms are complicated: the only two instances of the word ``troublesome'' in Hall's \textit{Combinatorial Theory} \cite{Hall-CombinatorialTheory} appear on page 143, when carrying computations involving the invariants of quadratic forms. Although some of the more ad-hoc proofs of the BRC theorem presented in the literature can be very elegant, avoiding the theory of quadratic forms comes unfortunately at the expense of presenting a general technique to study the equation $XX^{\intercal}=M$ over the rationals. One of the goals of this chapter is to convince the reader that deciding the solvability of $XX^{\intercal}=M$ over the field $\Q$ is a very straightforward computational task, not much harder than computing a Jordan canonical form. Then in Chapter \ref{chap-BRC}, we will present a very brief proof of the BRC theorem using precisely the invariants of quadratic forms. We will also present a new proof of the Bose-Connor Theorem \index{Bose-Connor Theorem} \cite{Bose-Connor}, which admittedly involves troublesome computations, although not as troublesome as the ones in the original proof.\\
 
 Another goal of this chapter is to point out that the theory of rational quadratic forms gives us much more than what we need for combinatorial applications. For example, to determine non-solvability conditions for Grammian equations we do not need complete invariants, we only need partial invariants. In addition, the assumption that $M$ is positive-definite also simplifies the theory a great deal. We have made a significant effort to  present an accessible account of the theory of quadratic forms geared towards combinatorialists. We omit some of the technical details that are not essential to combinatorial applications, but still present the theory in sufficient generality to be used flexibly in this type of application. In Chapter \ref{chap-HermitianForms} we extend this analysis to the theory of Hermitian forms, illustrated with several applications and new non-existence results for certain families of complex Hadamard matrices.\\

Part of the exposition in this chapter was done in collaboration with Oliver Gnilke and Padraig Ó Catháin, currently in preparation \cite{InvariantsPaper}.

\section{Positive-definite matrices and matrix analysis}\label{sec-MatAnalysis}
Over $\R$ or $\C$ the Grammian problem can be solved in a straightforward way using matrix-analytical techniques. The reader can find more information on matrix analysis in the textbooks of Bhatia \cite{Bhatia-MatrixAnalysis} or Horn and Johnson \cite{Horn-Johnson}. We recall some definitions:\\

A square matrix $M$ with complex entries is said to be \index{matrix!Hermitian}{\textit{Hermitian}} if $M=M^*$ where $M^*$ denotes the conjugate-transpose of $M$. In particular, a real symmetric matrix is Hermitian. A square matrix $M$ is said to be \index{matrix!normal}{\textit{normal}} if $MM^*=M^*M$. Therefore, every Hermitian matrix is normal. A matrix $U$ is called \textit{unitary} if $UU^*=I$.\\

 The following result holds for general Hilbert spaces, but we formulate here in the finite dimensional case in the language of matrices.
\begin{theorem}[Spectral theorem, cf. Theorem 2.5.3. \cite{Horn-Johnson}]
Every normal matrix $M$ is unitarily diagonalisable, i.e. if $M$ is normal then there exists a unitary matrix $U$ such that
\[U^*MU=D,\]
where $D$ is a diagonal matrix.
\end{theorem}
The eigenvalues of an Hermitian matrix are real. Indeed, if $M$ is Hermitian, and $v$ is a non-zero vector such that $Mv=\lambda v$, then 
\[\|v\|^2\lambda=v^*Mv=v^*M^*v=(v^*Mv)^*=\|v\|^2\overline{\lambda},\]
from which it follows that $\lambda=\overline{\lambda}$, and $\lambda\in \R$. An Hermitian matrix $M$ of order $n$ is said to be \index{matrix!positive-definite}{\textit{positive-definite}} if and only if for every non-zero column vector $x\in\C^n$ we have $x^*Mx>0$. Equivalently, an Hermitian matrix $M$ is positive-definite if and only if all its eigenvalues are positive. The following is a very useful characterisation of positive-definite matrices:
\begin{theorem}[Sylvester's Criterion, Theorem 7.2.5 \cite{Horn-Johnson}]\label{thm-SylvesterCriterion}\index{Sylvester!criterion} An Hermitian matrix $M$ is positive-definite if and only if all its leading principal minors are positive.
\end{theorem}

\begin{theorem}[Sylvester's law of inertia, Theorem 4.5.8 \cite{Horn-Johnson}]\label{thm-SylvesterInertia}\index{Sylvester!law of inertia}
Let $A$ and $B$ be two real symmetric matrices, then there exists a matrix $X\in \GL_n(\R)$ such that
\[X^{\intercal}AX=B,\]
 if and only if $A$ and $B$ have the same number of positive eigenvalues and the same number of negative eigenvalues.
\end{theorem}

From these results we obtain the following well-known fact:

\begin{theorem}\label{thm-GramPD}
An Hermitian matrix $M$ is positive-definite if and only if there exists an invertible square matrix $X$ with complex entries such that $XX^*=M$.
\end{theorem}
\begin{proof}
Suppose that $M=XX^*$ for some invertible matrix $X$, then $M$ is Hermitian, and positive definite. Since whenever $x$ is a non-zero vector, we have that $X^*x\neq 0$ and $x^*Mx=x^*(XX^*)x=(X^*x)^*(X^*x)=\|X^*x\|>0$. Conversely, let $M$ be Hermitian and positive-definite, then by the spectral theorem there is a unitary matrix $U$ such that $M=U^*DU$, for some diagonal matrix $D$ with real and positive non-zero entries. We can then define $D^{1/2}$ to be the matrix obtained from $D$ by taking the positive square root of each of its entries. The matrix $D^{1/2}$ clearly satisfies $D^{1/2}D^{1/2}=D$, therefore
\[M=U^*DU=U^*D^{1/2}D^{1/2}U=(D^{1/2}U)^*(D^{1/2}U).\qedhere \]
\end{proof}

The theorem above shows that if $M$ is positive-definite, then $M=XX^*$ can be solved over the complex numbers. In our combinatorial setting, the critical condition that we are introducing is that we require the entries of $X$ to belong to some proper subring $R$ of $\C$. We will show in the next section that the Grammian problem is completely solved when $R=\Q$. If $R$ is not a field, the difficulty of the problem can increase tremendously. As an example we mention the following theorem by Ryser.
\begin{theorem}[Ryser, Chapter 8, Theorem 4.2 \cite{Ryser-CombinatorialMathematics}]
Let $M=(k-\lambda)I_v+\lambda J_v$ where $J_v$ is the $v\times v$ all-ones matrix, and assume that the Grammian equation $XX^*=M$ is solvable for some $v\times v$ matrix $X$ with integer entries. If $\gcd(k,\lambda)$ is square-free and $k-\lambda$ is odd, then $X$ or $-X$ is the incidence matrix of a symmetric $2-(v,k,\lambda)$ design.
\end{theorem}
In other words, in some cases the solvability of Grammian equations over rings implies that the entries of $X$ are very restricted. In the theorem above, the entries are forced to be $0$ or $1$. The parameters of finite projective planes and Hadamard designs satisfy the conditions on $v,k$ and $\lambda$  of the theorem above. This indicates that solving integral Grammian equations is \textit{at least as hard} as finding finite projective planes and real Hadamard matrices, and both of these objects are notably hard to obtain in general, see \cite{Lam-SearchForPlane}.

\section{Quadratic forms}\label{sec-QF}
In this section, we introduce the theory of quadratic forms over an arbitrary field $k$ of characteristic $\neq 2$, and $V$ a finite-dimensional vector space over $k$. For an accessible introduction to the arithmetic theory of rational quadratic forms see Serre's book \cite{Serre-ACourseInArithmetic}, and see O'Meara for quadratic forms on general number fields \cite{OMeara-QuadraticForms}. A more elementary and self-contained exposition can be found in Jones' book \cite{Jones-ArithmeticQF}. For a more modern treatment with a focus on the algebraic theory of quadratic and Hermitian forms we refer the reader to Scharlau's book \cite{Scharlau-QuadraticHermitian}.

\begin{definition}\normalfont \label{def-SymBilForm}
A \index{form! symmetric bilinear}{\textit{symmetric bilinear form}} on $V$ is a function $b:V\times V\rightarrow k$ which is linear on both of its arguments, and satisfies $b(x,y)=b(y,x)$ for all $x,y\in V$. The pair $(V,b)$ is called a \textit{symmetric bilinear space}.
\end{definition}
\begin{definition}\normalfont \label{def-QuadraticForm} A \index{form!quadratic}{\textit{quadratic form}} on $V$ is a function $q:V \rightarrow k$ satisfying the axioms\\\\
 QF 1. $q(\alpha x)=\alpha^2q(x)$ for each $x\in V$, and $\alpha\in k$.\\
 QF 2. The mapping $(x,y)\mapsto \frac{1}{2}\left(q(x+y)-q(x)-q(q)\right)$ is a symmetric bilinear form.\\\\
The pair $(V,q)$ is called a \index{quadratic space}{\textit{quadratic space}}.
\end{definition}
One can obtain a quadratic form $q_b(x)=b(x,x)$ from every bilinear form $b$, and conversely one can obtain a bilinear form $b_q$ from a quadratic form by letting
\[b_q(x,y)=\frac{1}{2}(q(x+y)-q(x)-q(y)).\]
This establishes a one-to-one correspondence between quadratic forms and bilinear forms. Indeed, it is a straightforward exercise to show
\[b_{q_b}(x,y)=b(x,y),\text{ and } q_{b_q}(x)=q(x), \]
for all $x,y\in V$. Because of this equivalence, we will use the terms symmetric bilinear space and quadratic space interchangeably. If there is no possibility of confusion, instead of $b$ and $b_q$ or $q$ and $q_b$, we will simply use the notation $b$ and $q$ for the symmetric bilinear form and quadratic bilinear form of a given quadratic space.\\

Having introduced the objects that we will study, it is time to introduce the transformations between such objects.

\begin{definition}\normalfont \label{def-Isometry}  An \index{isometry}{\textit{isometry}} between quadratic spaces $(V,b)$ and $(V',b')$ is an injective linear mapping $\sigma: V\rightarrow V'$ such that
\[b(x,y)=b'(\sigma x,\sigma y).\]
If in addition $\sigma$ is bijective, then $(V,b)$ and $(V',b')$ are said to be \textit{isomorphic} spaces (denoted $(V,b)\simeq (V',b')$), and in this case the forms $b$ and $b'$ are called \textit{equivalent}.
\end{definition}

\begin{definition}\normalfont
An isometry from $(V,b)$ into itself is called an \index{autometry}{\textit{autometry}}. The set of autometries of $(V,b)$ forms a group called the \index{orthogonal group}{\textit{orthogonal group}}, denoted $O(V;b)$. Whenever we consider a fixed bilinear form $b$ on $V$ we abbreviate $O(V):=O(V;b)$.
\end{definition}

The definition of isometry can also be given in terms of quadratic spaces. An isometry between quadratic spaces $(V,q)$ and $(V',q')$ is an injective linear mapping $\sigma:V\rightarrow V'$ satisfying 
\[q(x)=q'(\sigma x).\]
Whenever $\sigma$ is bijective, we say that $q$ and $q'$ are \textit{equivalent} quadratic forms.\\

The connection between congruence of symmetric matrices, and equivalence of quadratic forms is established by choosing a basis of $V$. The following  well-known fact is a straightforward consequence of the definitions.
\begin{proposition} Let $(V,b)$ be a quadratic space, then for every choice of a basis $\mathcal{B}$ of $V$ there is a unique symmetric matrix $A$ such that
\[b(x,y)=x^{\intercal}Ay,\]
where in the right-hand side $x$ and $y$ are expressed as column vectors given by their coordinates in the basis $\mathcal{B}$. Conversely, for each choice of basis of $k^n$, a symmetric matrix $A$ gives rise to a unique symmetric bilinear form.
\end{proposition}
\begin{proof}
Let $\mathcal{B}=\{x_1,\dots,x_n\}$ be a basis of $V$. Define the matrix $A$ by $a_{ij}:=b(x_i,x_j)$, and assume that $x=\sum_i t_ix_i$ and $y=\sum_j r_jx_j$. Then by bilinearity,
\[b(x,y)=b\left(\sum_i t_ix_i,\sum_j r_j x_j\right)=\sum_{ij}t_ib(x_i,x_j)r_j=\sum_{ij}t_ia_{ij}r_j=x^{\intercal}Ay.\]
The matrix $A$ is clearly symmetric since $a_{ij}=b(x_i,x_j)=b(x_j,x_i)=a_{ji}$. Conversely, if $A$ is a symmetric matrix of order $n$, then writing any pair of vectors $x,y\in k^n$ in terms of a basis of $k^n$ the form $b_A(x,y)=x^{\intercal}Ay$ is a symmetric bilinear form.
\end{proof}
Again, the result above has a reformulation in terms of quadratic forms. Upon choice of a basis $\mathcal{B}$ for $V$, and a quadratic form $q$ on $V$, there is a unique symmetric matrix $A$ such that
\[q(x)=x^{\intercal}Ax.\]
And given a symmetric matrix $A$, the form $q_A(x)=x^{\intercal}Ax$ is quadratic.\\

By taking $x$ to be a vector of indeterminates, we can also interpret quadratic forms as given by homogeneous quadratic polynomials
\[q(x_1,\dots,x_n)=a_{11} x_1^2 + 2a_{12} x_1x_2+\dots+ 2a_{1n} x_1x_n+a_{22}x_2^2+\dots+a_{nn}x_n^2.\]

\begin{example}\normalfont We illustrate the different equivalent formulations for a quadratic space. Letting $k=\Q$, and $V=\Q^3$ we can define a bilinear form $b$ for every symmetric matrix. For example
\[A=\begin{bmatrix}
1 & 1 & 0\\
1 & 0 & 1\\
0 & 1 & 1
\end{bmatrix},
\]
gives the bilinear form
\begin{align*}
b(x,y)&=x^{\intercal}Ay\\
&=[x_1,x_2,x_3]\begin{bmatrix}
1 & 1 & 0\\
1 & 0 & 1\\
0 & 1 & 1
\end{bmatrix}
\begin{bmatrix}
y_1\\
y_2\\
y_3
\end{bmatrix}\\
&=x_1y_1+x_1y_2+x_2y_1+x_2y_3+x_3y_2+x_3y_3.
\end{align*}
Likewise, we obtain the quadratic form
\[q(x,y,z)=b((x,y,z),(x,y,z))=x^2+2xy+2yz+z^2.\]
Conversely, from a quadratic form of the type
\[q(x,y,z)= ax^2+bxy+cxz+dy^2+eyz+fz^2,\]
we obtain a matrix
\[A_q=\begin{bmatrix}
a & b/2 & c/2\\
b/2 & d & e/2\\
c/2 & e/2 & f
\end{bmatrix},
\]
such that $q(x)=x^{\intercal}A_q x$.
\end{example}
We now interpret the notion of isometry in terms of the matricial representation of quadratic forms. Let $q$ and $q'$ be quadratic forms on a vector space $V$ given by symmetric matrices $A$ and $B$ with respect to bases $\mathcal{B}$ and $\mathcal{B}'$. Recall that $q$ and $q'$ are isometric if and only if there is an injective linear mapping $\sigma:V\rightarrow V$ such that $q(x)=q'(\sigma x)$. Then if $P$ is the matrix of $\sigma$ with respect to the bases $\mathcal{B}$ and $\mathcal{B}'$ we have that
\[q'(\sigma x)=(\sigma x)^{\intercal}B(\sigma x)=(Px)^{\intercal}B(Px)=x^{\intercal}(P^{\intercal}BP)x=q(x)=x^{\intercal}Ax.\]
Therefore $A=P^{\intercal}BP$ for some invertible matrix $P$. This shows that two quadratic forms $q$ and $q'$ (represented by $A$ and $B$ respectively) are equivalent if and only if the matrices $A$ and $B$ are congruent.\\

We now present a few essential results in the general theory of quadratic forms. The proofs can be given in a purely matrix-theoretical way: For this approach we refer the reader to our paper \cite{InvariantsPaper} on applications of quadratic forms to combinatorics. Here instead, we use the geometric notions of quadratic space, and orthogonality more extensively, following the style of the expositions of Scharlau \cite{Scharlau-QuadraticHermitian} and Cassels \cite{Cassels-RationalQuadraticForms}.\\

Two vectors $x$ and $y$ in a quadratic space are said to be \textit{orthogonal} if and only if $b(x,y)=0$. If $(V,b)$ is a bilinear space and $S$ is a subset of $V$ then the \textit{orthogonal complement} of $S$ is defined as
\[S^{\perp}:=\{x\in V: b(x,y)=0,\text{ for all }y\in S\}.\]
By bilinearity, $S^{\perp}$ is a vector subspace of $V$. It is easy to check that if $S_1\subset S_2$ then $S_2^{\perp}\subset S_1^{\perp}$. A quadratic space $(V,b)$ is \index{quadratic space! regular}{\textit{regular}} if $V^{\perp}=0$, i.e. if $x=0$ is the only vector orthogonal to all vectors of $V$. In other words, if a quadratic space $(V,b)$ is regular, then for all non-zero $x\in V$ there is a $y\in V$ such that $b(x,y)\neq 0$.
\begin{lemma}[cf. Chapter 1, Corollary 3.2. \cite{Scharlau-QuadraticHermitian}]\normalfont\label{lemma-RegMatrix} $(V,b)$ is regular if and only if the matrix of $b$ with respect to some basis of $V$ is invertible.
\end{lemma}
\begin{proof}
Suppose that $(V,b)$ is regular, and let $A$ be the matrix of $b$ with respect to some basis. Let $x\neq 0$ be an arbitrary vector in $V$, then there is a vector $y\in V$ such that $b(x,y)\neq 0$. Therefore $b(y,x)=y^{\intercal}Ax\neq 0$, which implies that $Ax\neq 0$. Since $x$ is arbitrary, this shows that the endomorphism of $V$ induced by $A$ is injective, and so this endomorphism is necessarily bijective. This in turn implies that $A$ is invertible. Conversely if $A$ is not invertible, then there is a non-zero $x\in V$ such that $Ax=0$. But then $b(x,y)=b(y,x)=y^{\intercal}Ax=0$ for all $y\in V$, hence $(V,b)$ is not regular.
\end{proof}

\begin{lemma}[cf. Chapter 1, Corollary 3.2. \cite{Scharlau-QuadraticHermitian}]\normalfont Let $b_x:V\rightarrow k$ be given by $b_x(y)=b(x,y)$ for $y\in V$. Then $(V,b)$ is regular if and only if the linear mapping
\begin{align*}
V & \rightarrow V^*=\Hom_k(V,k)\\
x&\mapsto b_x
\end{align*}
is an isomorphism.
\end{lemma}
\begin{proof}
Let $\mathcal{B}=\{x_1,\dots,x_n\}$ be a basis for $V$, and let $\mathcal{B}^*=\{\delta_1,\dots,\delta_n\}$ be its dual basis, i.e. $\delta_{i}(x_j)=\delta_{ij}$. Let $A$ be the  matrix of $b$ with respect to $\mathcal{B}$, then the matrix of the mapping $x\mapsto b_x$ with respect to the bases $\mathcal{B}$ and $\mathcal{B}^*$ is also $A$. Indeed,
\[b_{x_i}(x_j)=b(x_i,x_j)=a_{ij}=a_{ij}\delta_j(x_j)=\sum_{\ell}a_{i\ell}\delta_{\ell}(x_j),\]
which implies $b_{x_i}=\sum_j a_{ij}\delta_j$. Now the result follows from Lemma \ref{lemma-RegMatrix}.
\end{proof}

\begin{proposition}[cf. Chapter 1, Lemma 3.4. \cite{Scharlau-QuadraticHermitian}]\label{prop-OrthDS} Let $W$ be a regular subspace of $(V,b)$. Then $V$ is the direct sum of $W$ and $W^{\perp}$, i.e. $V= W\oplus W^{\perp}$
\end{proposition}
\begin{proof}
Let $x\in V$, then $b_x\in V^*=\Hom_k(V,k)$ and the restriction ${b_{x}}_{|_W}$ of $b_x$ to $W$ is an element of $W^*$. By regularity of $W$ it follows that there is an element $y\in W$ such that $b_y = {b_{x}}_{|_W} $. In other words
\[b(y,z)=b(x,z),\]
for every $z\in W$. Therefore, if $z\in W$ then
\[b(x-y,z)=b(x,z)-b(y,z)=0.\]
So $x=y+(x-y)$, where $y\in W$ and $x-y\in W^{\perp}$. Since $W$ is regular, we have $W\cap W^{\perp}=\{0\}$ which implies that the above decomposition of $x$ is unique. It follows that $V=W\oplus W^{\perp}$.
\end{proof}
 From Proposition \ref{prop-OrthDS} we derive the following elementary, but important result. 

\begin{theorem}[Polarisation theorem, cf. Chapter 1, Theorem 3.5. \cite{Scharlau-QuadraticHermitian}]\label{thm-Polarisation}
Every bilinear space $(V,b)$ is an orthogonal direct sum of 1-dimensional spaces.
\end{theorem}
\begin{proof}
If $b=0$, then the result follows trivially since every decomposition of $V$ into a direct sum of $1$-dimensional subspaces will be also orthogonal. If $b\neq 0$, then there is a pair of vectors $x,y\in V$ such that $b(x,y)\neq 0$. Now from
\[b(x,y)=\frac{1}{2}(b(x+y,x+y)-b(x,x)-b(y,y)),\]
it follows that some $z\in \{x,y,x+y\}$ satisfies $b(z,z)\neq 0$. Therefore $W=\Span(z)$ is a one-dimensional regular subspace of $V$, by Proposition \ref{prop-OrthDS} it follows that $V=W\oplus W^{\perp}$. We can apply induction on $W^{\perp}$, and the result follows. 
\end{proof}
A quadratic form $q$ is \index{form!quadratic!polarised}{\textit{polarised}} with respect to a basis $\mathcal{B}$ if and only if the matrix of $q$ with respect to $\mathcal{B}$ is diagonal. The polarisation theorem says then that every quadratic form can be polarised. From the point of view of matrices, the polarisation theorem says that every symmetric matrix is congruent to a diagonal matrix, this result can be obtained by a symmetric row and column reduction of the matrix.

\begin{example}\normalfont
Even if this is familiar to any student of linear algebra we illustrate, for clarity, the process of polarisation of a matrix by row-reduction. Consider the matrix 
\[ S = \begin{bmatrix} 1 & 2 & 3 \\ 2 & 4 & 5 \\ 3 & 5 &-1 \end{bmatrix}\,. \] 
We begin by eliminating the off-diagonal entries in the first row, and then the first column. 
\[ \begin{bmatrix} 1 & 0 & 0 \\ -2 &1 & 0 \\ -3 & 0 &1 \end{bmatrix}
\begin{bmatrix} 1 & 2 & 3 \\ 2 & 4 & 5 \\ 3 & 5 &-1 \end{bmatrix}
\begin{bmatrix} 1 & -2 & -3 \\ 0 & 1 & 0 \\ 0 & 0 &1 \end{bmatrix} = 
\begin{bmatrix} 1 & 0 & 0 \\ 0 & 0 & -1 \\ 0 & -1 &-10 \end{bmatrix}\,.\]
Since we have a zero pivot in position $(2,2)$, we swap the second and third rows. We can also multiply the second row and column by $10$, to achieve the matrix 
 \[\begin{bmatrix} 1 & 0 & 0 \\ 0 & 0 & 1 \\ 0 & 10 &0 \end{bmatrix}
 \begin{bmatrix} 1 & 0 & 0 \\ 0 & 0 & -1 \\ 0 & -1 &-10 \end{bmatrix}
 \begin{bmatrix} 1 & 0 & 0 \\ 0 & 0 & 10 \\ 0 & 1 &0 \end{bmatrix} = \begin{bmatrix} 1 & 0 & 0 \\ 0 & -10 & -10 \\ 0 & -10 &0 \end{bmatrix}\]
 Finally, subtracting the second row from the third and likewise for columns leaves the diagonal matrix $\diag(1, -10, 10)$. The matrix $X$ such that $X^{\intercal}SX=\diag(1,-10,10)$ can be computed explicitly by multiplying out the row operation matrices. 
\end{example}

\section{Witt's Lemma}\label{sec-WittLemma}
In this section we will prove \index{Witt's lemma}{Witt's lemma}, which is a fundamental tool for the general study of quadratic forms. This result is non-essential for the combinatorial application we are considering, but we will need it in Chapter \ref{chap-HermitianForms} and Chapter \ref{chap-UPIR} as a theoretical tool. Let us first introduce a convenient notation to express quadratic spaces: Let $A$ be a symmetric matrix, then we denote by $\langle A\rangle$ the symmetric bilinear (or quadratic) space generated by $A$. If $A$ is congruent to the matrix $\diag(\alpha_1,\dots,\alpha_n)$, then we write  $\langle \alpha_1,\dots,\alpha_n\rangle :=\langle A\rangle$ for the quadratic space generated by $A$. The direct sum of quadratic spaces  $\langle A\rangle$ and $\langle B\rangle$ is defined as 
\[\langle A\rangle\oplus \langle B\rangle := \langle A\oplus B\rangle,\]
where $A\oplus B$ is the block-matrix
\[A\oplus B=\left[
\begin{array}{c|c}
A & 0\\
\hline
0 & B
\end{array}
\right].
\]
Therefore, if $\langle\alpha_1,\dots,\alpha_n\rangle=\langle A\rangle$ and $\langle\beta_1,\dots,\beta_m\rangle=\langle B\rangle$, then 
\[\langle\alpha_1,\dots,\alpha_n\rangle\oplus \langle\beta_1,\dots,\beta_m\rangle=\langle\alpha_1,\dots,\alpha_n;\beta_1,\dots,\beta_m\rangle.\]
It is clear that the equivalence class of the space $\langle \alpha_1,\dots,\alpha_n\rangle$ does not depend on the order in which the $\alpha_i$ are listed, nor on multiplication of the $\alpha_i$ by square factors. Furthermore, it is easy to show the properties hold for the direct sum of quadratic spaces:
\begin{itemize}
\item $\varphi\oplus \psi\simeq \psi\oplus \varphi$,
\item If $\varphi\simeq \varphi'$ and $\psi\simeq \psi'$, then $\varphi\oplus \psi\simeq \varphi'\oplus \psi'$.
\end{itemize}
To prove Witt's theorem we first require some knowledge of the orthogonal group $O(V)$, and its action on vectors of $V$. Let $W$ be a regular subspace of $V$, then by regularity $V=W\oplus W^{\perp}$ and we can define an autometry $\sigma$ by letting
\[\sigma(v)=\begin{cases}
-v &\text{ if } v \in W\\
v &\text{ if } v\in W^{\perp}
\end{cases}.
\] 
The map $\sigma$ is linear, and for $x,y\in V$ we can write $x=w+u$ and $y=w'+u'$ for unique $w,w'\in W$ and $u,u'\in W^{\perp}$, which implies
\begin{align*}b(\sigma x,\sigma y)&=b(-w+u,-w'+u')\\
&=b(-w,-w')-b(w,u')-b(u, w')+b(u,u')\\
&=b(w,w')+b(u,u')\\
&=b(w+u,w'+u')=b(x,y).
\end{align*}
So $\sigma$ is an autometry of $V$. In particular if $W=\Span(w)$ where $b(w,w)=q(w)\neq 0$,  we have an autometry $\tau_w$ defined by $\tau_w(w)=-w$, and $\tau_w(v)=v$ if $b(w,v)=0$. This autometry has the closed form
\[\tau_w(v)=v-2\frac{b(v,w)}{b(w,w)}w.\]
\begin{lemma}[cf. Lemma 4.2. \cite{Cassels-RationalQuadraticForms}]\normalfont \label{lemma-FibreTransitivity} The group $O(V)$ acts transitively on the fibres $q^{-1}(\alpha)$ for $\alpha\in k-\{0\}$, i.e. if $q(v)=q(w)\neq 0$ then there is an autometry $\sigma$ such that $\sigma(v)=w$.
\end{lemma}
\begin{proof}
Assume that $v,w\in V$ are such that $0\neq q(v)=b(v,v)=b(w,w)=q(w)$. First consider the case where $q(v-w)\neq 0$.  Then the space $\Span(v-w)$ is regular and $\tau_{v-w}$ is an autometry of $V$. We have that 
\[0\neq b(v-w,v-w)=b(v,v)-2b(v,w)+b(w,w)=2(b(v,v)-b(v,w))=2b(v,v-w).\]
Therefore,
\[\tau_{v-w}(v)=v-\frac{2b(v,v-w)}{b(v-w,v-w)}(v-w)=v-(v-w)=w.\]
Since $q(v)\neq 0$, we find that
\begin{align*}
b(v+w,v+w)+b(v-w,v-w)&=2(b(v,v)+b(v,w) -b(v,w)+b(w,w))\\
&=2(b(v,v)+b(w,w))\\
&=4q(v)\neq 0.
\end{align*}
So if $q(v-w)=0$, then $q(v+w)\neq 0$. And from $0\neq b(v+w,v+w)=2b(v,v+w)$ we find
\[\tau_{v+w}(v)=v-2\frac{b(v,v+w)}{b(v+w,v+w)}(v+w)=-w.\]
This implies that $\tau_w\circ\tau_{v+w}(v)=w$.\qedhere
\end{proof}
\begin{theorem}[Witt, cf. Theorem 4.1. \cite{Cassels-RationalQuadraticForms}]\label{thm-WittThm}
Let $W$ and $W'$ be isomorphic regular subspaces of $V$, where the isomorphism is given by an isometry
\[\rho: W\rightarrow W'\]
then there is an autometry $\sigma\in O(V)$ that extends $\rho$. In other words, the diagram
\[\begin{tikzcd}
V \arrow[r, "\sigma", dashed]             & V                        \\
W \arrow[r, "\rho"'] \arrow[u, "i", hook] & W' \arrow[u, "i"', hook]
\end{tikzcd}
\]
is commutative, i.e. $\sigma\circ i=i\circ \rho$, where $i$ denotes the inclusions of $W$ and $W'$ in $V$.
\end{theorem}
\begin{proof}
Since $W$ is regular, there is a vector $w\in W$ such that $q(w)\neq 0$. Since $\rho$ is an isometry $q(\rho w)=q(w)$, and Lemma \ref{lemma-FibreTransitivity} implies the existence of an autometry $\lambda$ such that $\lambda(\rho(w))=w$.  We have trivially that $i_{\lambda W'}\circ\lambda=\lambda\circ i_{W'}$, so if we show that there is a $\sigma\in O(V)$ such that $\sigma\circ i_{W}= i_{\lambda W'}(\lambda\circ\rho)=\lambda\circ(i_{W'}\circ\rho)$, then
\[(\lambda^{-1}\sigma)\circ i_{W}=i_{W'}\circ \rho.\]
Thus, we can replace $\rho$ with $\lambda\circ\rho$ and $W'$ with $\lambda W'$, and assume without loss of generality that $\rho(w)=w$ and $w\in W\cap W'$. If $\dim W=1$, then the identity autometry is clearly an extension of $\rho$. Otherwise, we proceed by induction: If $\dim W>1$, let $W_0=\Span(w)$,
\[U=W\cap W_0^{\perp},\text{ and } U'=W'\cap W_0^{\perp}.\]
We have that $\rho(U)=\rho(W\cap W_0^{\perp})=W'\cap W_0^{\perp}=U'$, hence the restriction of $\rho$ to $U$ is an isometry. By the induction hypothesis, there is an autometry $\tau$ extending $\rho_{|U}$ to $V$. Let $\sigma(w)=w$, and $\sigma (u)=\tau (u)$ for all $u\in W_0^{\perp}$, then $\sigma$ is an autometry and it extends $\rho$ to $V$.
\end{proof}
\begin{corollary}[Witt's Lemma, cf. Theorem 4.1. \cite{Cassels-RationalQuadraticForms}]\label{cor-WittLemma}
Suppose that $(V,b)$ and $(V',b')$ are isomorphic quadratic spaces, and that $W\subseteq V$ and $W'\subseteq V'$ are isomorphic regular subspaces. Then the orthogonal complements $W^{\perp}$ of $W$ in $V$ and ${W'}^{\perp}$ of $W'$ in $V'$, are isomorphic.
\end{corollary}
\begin{proof}
By assumption, there is a bijective isometry $\rho: V'\rightarrow V$. Taking $\rho(W')$ instead of $W'$, and $\rho(V')=V$ instead of $V'$, we may assume without loss of generality that $V=V'$ and $b=b'$. Since there is an isomorphism, say $\mu:W\rightarrow W'$, between $W$ and $W'$, Theorem \ref{thm-WittThm} implies the existence of an autometry $\sigma$ which extends $\mu$. For $\sigma$ we have that $\sigma(W^{\perp})=W'^{\perp}$, and this gives the required isomorphism.
\end{proof}
Witt's Lemma is also known as \textit{Witt cancellation}, since it implies that the direct sum of quadratic spaces has the following cancellation property: Let $\psi$, $\psi'$ and $\varphi$ denote quadratic spaces, if 
\[\psi\oplus \varphi\simeq \psi'\oplus \varphi,\]
then $\psi\simeq \psi'$. This property shows that the set of isometry classes of quadratic spaces forms an abelian semigroup with cancellation. There is a canonical way to embed this semigroup into a group, known as the \index{Grothendieck group}{\textit{Grothendieck group}} of the field $k$. Considering in addition the tensor product of quadratic forms, we can form a ring known as the \index{Grothendieck-Witt ring}{\textit{Grothendieck-Witt ring}} $W(k)$ of $k$. With this point of view, the problem of classification of quadratic forms over the field $k$ is equivalent to computing $W(k)$. This is one of the fundamental ideas of the algebraic theory of quadratic forms, and structural results on $W(k)$ can be used to learn about the theory of quadratic forms over a general field. For more on this approach see Chapter 2 of Scharlau \cite{Scharlau-QuadraticHermitian}.
\\

As mentioned before, many of the arguments presented can be approached in a purely matrix-theoretic way. For completeness, we present our own proof of Witt's cancellation Lemma using elementary methods.

\begin{lemma}[Witt cancellation]\normalfont
Let $A, A', B$ and $C$ be symmetric matrices. If $A$ is congruent to $A'$ (denoted $A\simeq A'$), and
\[
\left[
\begin{array}{c|c}
A & 0\\
\hline
0 & B
\end{array}
\right]
\simeq
\left[
\begin{array}{c|c}
A' & 0\\
\hline
0 & C
\end{array}
\right],
\]
then $B\simeq C$.
\end{lemma}
\begin{proof}
In view of Theorem \ref{thm-Polarisation} it suffices to show that  if
\[
\left[
\begin{array}{c|c}
\alpha & 0\\
\hline
0 & B
\end{array}
\right]
\simeq
\left[
\begin{array}{c|c}
\beta & 0\\
\hline
0 & C
\end{array}
\right],\]
and  $\alpha=t^2\beta$ for some $t\in k^{\times}$ (i.e. the forms $\alpha x^2$ and $\beta y^2$ are equivalent), then $B\simeq C$. By hypothesis there exists a matrix 
\[T=\left[
\begin{array}{c|c}
\lambda & u^{\intercal}\\
\hline
v &P
\end{array}
\right],\]
such that $T^{\intercal}(\alpha \oplus B)T=(t^2\alpha)\oplus C$. Therefore,
\[\left[
\begin{array}{c|c}
\lambda & v^{\intercal}\\
\hline
u& P^{\intercal}
\end{array}
\right]
\left[
\begin{array}{c|c}
\alpha & 0\\
\hline
0 & B
\end{array}
\right]
\left[
\begin{array}{c|c}
\lambda & u^{\intercal}\\
\hline
v &P
\end{array}
\right]
=
\left[
\begin{array}{c|c}
t^2\alpha & 0\\
\hline
0 &C
\end{array}
\right].
\]
Computing the product in the left-hand-side we find
\[
\left[
\begin{array}{c|c}
\lambda^2\alpha + v^{\intercal}Bv & \lambda\alpha u^{\intercal} + v^{\intercal}BP\\
\hline
\lambda\alpha u + P^{\intercal}Bv & \alpha uu^{\intercal} + P^{\intercal}BP
\end{array}
\right]=
\left[
\begin{array}{c|c}
t^2\alpha & 0\\
\hline
0 & C
\end{array}
\right].
\]
This equation is equivalent to the following system of matrix equations
\begin{align*}
&v^{\intercal}Bv = (t^2-\lambda^2)\alpha,\\
&v^{\intercal}BP + \lambda\alpha u^{\intercal} = 0,\\
&P^{\intercal}Bv + \lambda\alpha u = 0,\\
&P^{\intercal}BP+\alpha uu^{\intercal} = C.
\end{align*}
Let $\epsilon = \pm 1$ be chosen so that $t+\epsilon\lambda\neq 0$. Let $S=P-\epsilon r(vu^{\intercal})$, where $r=(t+\epsilon\lambda)^{-1}$. We show that $S^{\intercal}BS=C$:
\begin{align*}
S^{\intercal}BS&=(P^{\intercal}-\epsilon r(uv^{\intercal}))B(P-\epsilon r (vu^{\intercal}))\\
&=P^{\intercal}BP -\epsilon r(uv^{\intercal}BP)-\epsilon r(P^{\intercal}Bvu^{\intercal}) + r^2(uv^{\intercal}Bv u^{\intercal})\\
&=P^{\intercal}BP +2\epsilon\lambda (r\alpha uu^{\intercal})+r^2(t^2-\lambda^2)\alpha(uu^{\intercal})\\
&=P^{\intercal}BP + (2\epsilon\lambda + r(t^2-\lambda^2))r\alpha uu^{\intercal}\\
&=P^{\intercal}BP + (2\epsilon\lambda+(t-\epsilon\lambda))r\alpha uu^{\intercal}\\
&= P^{\intercal}BP + \alpha uu^{\intercal} = C.
\end{align*}
This gives the required congruence between $B$ and $C$.
\end{proof}

\section{Hilbert symbols}\label{sec-HilbertSym}
Hilbert symbols are the main ingredient to define the local invariants of quadratic forms. In this section we will motivate them and study their properties. The property of bilinearity of the symbol is very important and to study it we will briefly discuss $p$-adic numbers.\\

We begin with the study of equivalence of quadratic spaces in low dimensions. A regular quadratic space of dimension $1$ is given by a $1\times 1$ matrix $(\alpha)$, for $\alpha\in k^{\times}$. It is clear then that $\langle \alpha\rangle\simeq \langle \beta\rangle$ if and only if there is some $t\in k^{\times}$, such that
\[t\alpha t=\alpha t^2 =\beta.\]
In other words, $\alpha$ and $\beta$ are in the same coset of the \index{square class group}\textit{square class group} $\Gamma(k):=k^{\times}/(k^{\times})^2$. To be precise, two elements in $\alpha,\beta\in k^{\times}$ are in the same \textit{square class} if and only if $\alpha=t^2\beta$ for some $t\in k^{\times}$.\\

The square class group is elementary abelian of characteristic $2$, since clearly for every $a,b\in k^{\times}$ we have $ab=ba$ and $a^2\equiv 1$ in $\Gamma(k)$.

\begin{example}\normalfont
The square class group $\Gamma(\C)$ of the complex numbers is trivial, since $\C$ is algebraically closed, and the equation $x^2=\alpha$ has a solution for every $\alpha\in \C$.\\

$\Gamma(\R)$ has order $2$, and it is generated by $+1$ and $-1$. This is because every non-zero real number $x$ can be written as 
\[x=\begin{cases}
+\sqrt{|x|}^2 &\text{ if } x>0\\
-\sqrt{|x|}^2 &\text{ if } x<0
\end{cases}
.\] 

The square class group $\Gamma(\Q)$ is generated by $-1$ and all rational prime numbers $p$. This is because $\frac{a}{b}\equiv \frac{a}{b}b^2=ab$ in $\Gamma(\Q)$, and by the fundamental theorem of arithmetic every integer $n$ can be expressed as a product of primes 
\[n=\pm p_1^{e_1}\dots p_r^{e_r},\]
in a unique way up to relabelling of the $p_i$.
\end{example}

Our next goal is then to find conditions for the  equivalence of regular quadratic spaces of dimension $2$. We begin with the  particular case of solving $XX^{\intercal}=M$. Notice that taking the transpose of $X$, the solvability of $X^{\intercal}X=M$ is equivalent to the solvability of $XX^{\intercal}=M$.

\begin{proposition}\label{prop-2x2Gram} \normalfont Let $k$ be a field with $\Char(k)\neq 2$ and $a,b\in k^{\times}$. Then the equation
\[
X^{\intercal}X=\begin{bmatrix}
a & 0\\
0 & b
\end{bmatrix}
\]
has a solution for some $X\in \GL_2(k)$ if and only if 
\begin{itemize}
\item[(i)] $ab$ is a square in $k^{\times}$, and
\item[(ii)] $ax^2+by^2=1$ has a solution in $k$.
\end{itemize}

\end{proposition}
\begin{proof}
Let $M=\begin{bmatrix}
a & 0\\
0 & b
\end{bmatrix},$ 
where $a,b\in k^{\times}$. 
The condition $M=Y^{\intercal}Y$ for $Y\in \GL_n(k)$ is equivalent to $(Y^{-1})^{\intercal}MY^{-1}=I$. So we may show instead that $X^{\intercal}MX=I$ if and only if $ab$ is a square, and $ax^2+by^2=1$ has a solution. Assume that there exist a matrix $X\in \GL_n(k)$ such that $X^{\intercal}MX=I$. Letting $X=\begin{bmatrix}
x_1 & x_2\\
y_1 & y_2
\end{bmatrix}$, this equation is rewritten as  
\[X^{\intercal}MX=\begin{bmatrix}
x_1 & y_1\\
x_2 & y_2
\end{bmatrix}
\begin{bmatrix}
a & 0\\
0 & b
\end{bmatrix}
\begin{bmatrix}
x_1 & x_2\\
y_1 & y_2
\end{bmatrix}=
\begin{bmatrix}
ax_1^2+by_1^2 & ax_1x_2+by_1y_2 \\
ax_1x_2+by_1y_2 & ax_2^2+by_2^2
\end{bmatrix}=
\begin{bmatrix}
1 & 0\\
0 & 1
\end{bmatrix}
,\]
for some $x_1,x_2,y_1,y_2\in k$. Therefore, the congruence above is equivalent to the system of equations
\[
\begin{cases}
ax_1^2 + by_1^2=1\\
ax_2^2 + by_2^2=1\\
ax_1x_2+ by_1y_2=0
\end{cases}.
\]
In particular, there is a solution to $ax^2+by^2=1$. Taking determinants in the expression $M=X^{\intercal}X$ find that $ab=\det(M)=\det(X^{\intercal}X)=\det(X)^2$, so $ab$ must be a square in $k^{\times}$.\\
Conversely, suppose that there is a solution $(x_1,y_1)$ to $ax^2+by^2=1$, and that $ab$ is a square. If $y_1=0$ then $ax_1^2=1$, so $a$ is a square in $k^{\times}$. Now, $ab$ is a square which implies that $b$ is a square, and there is some $x_2\in k^{\times}$ such that $bx_2^2=1$. Therefore,
\[\begin{bmatrix}
x_1 & 0\\
0 & x_2
\end{bmatrix}
\begin{bmatrix}
a & 0 \\
0 & b
\end{bmatrix}
\begin{bmatrix}
x_1 & 0\\
0 & x_2
\end{bmatrix}
=\begin{bmatrix}
1 & 0\\
0 & 1
\end{bmatrix}.
\]
 Suppose then that $y_1\neq 0$: we find values for $(x_2,y_2)$ so that $ax_1x_2+by_1y_2=0$ and $ax_2^2+by_2^2=1$. Since $b\neq 0$ we let $y_2=-ax_1x_2/by_1$, and substituting into the equation $1=ax_2^2+by_2^2$ we find
\[by_1^2=aby_1^2x_2^2+ax_1^2x_2^2=ax_2^2(ax_1^2+by_1^2)=ax_2^2.\]
So $x_2^2=\frac{by_1^2}{a}$, and the right-hand-side is a square by our assumption that $ab$ is a square. Hence $x_2$ and $y_2$ belong to $k$, and are determined from $x_1$ and $y_1$ up to sign.
\end{proof}

\begin{remark}\normalfont The element $ab\in k^{\times}$ in the proposition above, interpreted as an element of $\Gamma(k)$ is known as the \textit{discriminant} of $\langle a,b\rangle$. In the next section we will give the general definition of the discriminant of a quadratic form.
\end{remark}

It is an easy exercise to show that for $a,b\in k^{\times}$, the equation $ax^2+by^2=1$ has a solution in $k$ if and only if $ax^2+by^2=z^2$ has a \textit{non-trivial} solution in $k$. This motivates the following definition:

\begin{definition}\normalfont
Let $k$ be a field, and let $a,b\in k^{\times}$. The \index{Hilbert symbol} \textit{Hilbert symbol} of $a$ and $b$ is defined as 
\[(a,b)_k:=\begin{cases}
+1 & \text{ if the equation } ax^2+by^2=z^2\text{ has a non-trivial solution over } k\\
-1 & \text{ otherwise }
\end{cases}
\]
If the field $k$ is clear from the context we simply write $(a,b)$ instead of $(a,b)_k$.
\end{definition}
 
 \begin{example}\normalfont
If $k=\R$, then $(a,b)_{\R}=1$ if and only if $a$ and $b$ are not both negative: The square class group $\R^{\times}/(\R^{\times})^2$ is isomorphic to $\{\pm 1\}$. Therefore we may assume that $a,b\in \{\pm 1\}$. If either $a$ or $b$ are $1$, then $(a,b)_{\R}=1$. And for $a=b=-1$, we have $(a,b)_{\R}=-1$, since the equation
\[-x^2 - y^2 =z^2\]
has no real solutions. Let $\Gamma(\R)=\R^{\times}/(\R^{\times})^2\simeq \{\pm 1\}$, then regarding the Hilbert symbol as a function $(\cdot,\cdot)_{\R}:\Gamma(\R)\times\Gamma(\R)\rightarrow \R$, we have the following table of values of the Hilbert symbol:
\[
\begin{array}{c|cc}
(a,b)_{\R} & +1 & -1\\
\hline
+1 & 1 & 1\\
-1 & 1 & -1
\end{array}
\]
Our computation of the real Hilbert symbols, and Proposition \ref{prop-2x2Gram} tell us that $\langle a,b\rangle$ is isomorphic to $\langle 1,1\rangle$ if and only if $a$ and $b$ are both positive. This is a particular case of Sylvester's law of inertia.
\end{example}

\begin{example}\normalfont \label{ex-FFSymbol} Let $k=\F_q$, where $q$ is an odd prime-power. Then, the square class group $\Gamma(\F_q)=(\F_q^{\times})/(\F_q^{\times})^2$ has order $2$. Let $x\in \F_q^{\times}$ be an arbitrary non-square, then $x$ is a sum of two squares in $\F_q^{\times}$. Otherwise, for every $a\in \F_q^{\times}$,
\[x-a^2\not\in (\F_q^{\times})^2.\]
But there are exactly $(q-1)/2$ distinct elements in the set $\mathcal{C}=\{x-a^2: a\in \F_q^{\times}\}$, which implies that $\mathcal{C}$ is the set of non-squares of $\F_{q}^{\times}$. This is a contradiction, since then $x-a^2=x$ for some $a\in k^{\times}$ yet $a^2\neq 0$. This implies that $(a,b)_{\F_q}=1$ for all $a,b\in\F_q^{\times}$: If $a$ or $b$ are squares, then clearly $(a,b)_{\F_q}=1$, so assume that both $a$ and $b$ are non-squares. Then $t=a^{-1}$ is a non-square, and the equation $ax^2+by^2=z^2$ is equivalent to 
\[x^2+tby^2=tz^2.\]
Now $tb$ is a square so there is a non-trivial solution to $x^2+tby^2=tz^2$ if and only if there is a non-trivial solution to $x^2+y^2=tz^2$. Since $t\in \F_q^{\times}$, then $t=c^2+d^2$ for some $c,d\in\F_q^{\times}$, which implies that $(c,d,1)$ is a solution of $x^2+y^2=tz^2$.\\
To summarise, the Hilbert symbols at finite fields have the following table of values
\[
\begin{array}{c|cc}
(a,b)_{\F_q} & 1 & r\\
\hline
1 & 1 & 1\\
r & 1 & 1
\end{array}
\]
where $r$ is a non-square in $\F_q^{\times}$. Proposition \ref{prop-2x2Gram} then implies that the regular quadratic space $\langle a,b\rangle$ over $\F_q$ is isomorphic to $\langle 1,1\rangle$ if and only if $ab$ is a square in $\F_q$.
\end{example}

We hope that the above examples illustrated how the theory of quadratic forms depends heavily on the structure of the square class group $\Gamma(k)=k^{\times}/(k^{\times})^2$. In both our examples above, the square class field is finite. But for other fields, such as the rationals, this group is infinite and the theory of quadratic forms becomes more involved, even in the $2\times 2$ case.
\begin{example}\normalfont The rational quadratic space $\langle 5,20\rangle$ is isomorphic to $\langle 1,1\rangle$. In other words, there is a matrix $X\in \GL_2(\Q)$ such that
\[X^{\intercal}\begin{bmatrix}
5 & 0\\
0 & 20
\end{bmatrix}
X=
\begin{bmatrix}
1 & 0\\
0 & 1
\end{bmatrix}.
\]
The discriminant is $5\cdot 20=100\equiv 1$ in $\Q^{\times}/(\Q^{\times})^2$, since $100$ is a square. Also, $5x^2+20y^2=z^2$ has a non-trivial solution $(x,y,z)=(1,1,5)$, so $(5,20)_{\Q}=1=(1,1)_{\Q}$. We can reproduce the proof of Proposition \ref{prop-2x2Gram} with $a=5$, and $b=20$. From our solution to $5x^2+20y^2=z^2$ we find a solution $(x_1,y_1)=(1/5,1/5)$ to $5x_1^2+20y_1^2=1$. We let $y_2=-ax_1x_2/by_1=-x_2/4$, and substitute $y_2$ into the equation $5x_2^2+20y_2^2=1$. Operating we find that
\[25x_2^2=4,\]
and we can choose for example the solution $x_2=2/5$. From here we find $y_2=-1/10$. Indeed, we can check that
\[
\begin{bmatrix}
1/5 & 1/5\\
2/5 & -1/10
\end{bmatrix}
\begin{bmatrix}
5 & 0\\
0 & 20
\end{bmatrix}
\begin{bmatrix}
1/5 & 2/5\\
1/5 & -1/10
\end{bmatrix}
=\begin{bmatrix}
1 & 0\\
0 & 1
\end{bmatrix}.
\]
And, taking inverses we find an expression for $\diag(5,20)$ as a rational Gram matrix.
\[\begin{bmatrix}
1 & 2\\
4 & -2
\end{bmatrix}
\begin{bmatrix}
1 & 4\\
2 & -2
\end{bmatrix}
=
\begin{bmatrix}
5 & 0\\
0 & 20
\end{bmatrix}.
\]
\end{example}

\begin{example}\normalfont
The rational quadratic form given by $\langle 3,3\rangle$ is not isomorphic to $\langle 1,1\rangle$. In other words,
\[\begin{bmatrix}
3 & 0\\
0 & 3
\end{bmatrix}\text{ is \textit{not} rationally congruent to }
\begin{bmatrix}
1 & 0\\
0 & 1
\end{bmatrix}.
\]
We have that the discriminant $3\cdot 3=9$ is a square, so we show that $(3,3)_{\Q}\neq (1,1)_{\Q}=1$. In other words, we show that $3x^2+3y^2=z^2$ has no non-trivial rational solutions. If this equation had rational solutions, then multiplying by a common denominator we find that it has integer solutions. If $(x,y,z)$ is a non-trivial integer solution, we may assume without loss of generality that $x$, $y$ and $z$ have no common factors, otherwise dividing by their greatest common divisor we find a coprime solution. Since $3x^2+3y^2=3(x^2+y^2)=z^2$, and $3$ is prime it follows that $z$ is divisible by $3$. We can then write $z=3z_0$, and then $3x^2+3y^2=9z_0^2$, so we find
\[x^2+y^2=3z_0^2.\]
Now reducing the above equation modulo $3$, we find that $x^2+y^2\equiv 0\pmod{3}$. But the squares modulo $3$ are $0$ and $1$, which implies that both $x$ and $y$ are divisible by $3$. This contradicts the assumption that $x$,$y$ and $z$ are coprime. Therefore, $(3,3)_{\Q}=-1$, and by Theorem \ref{thm-2x2HM} a rational congruence between the matrices does not exist. 
\end{example}

\begin{lemma}[cf. Chapter III, Proposition 2 \cite{Serre-ACourseInArithmetic}] \normalfont\label{lemma-Hilbert-Sym-Props}
The Hilbert Symbol satisfies the following properties
\begin{itemize}
\item[(i)] $(a,b)_k=(b,a)_k$,
\item[(ii)] $(a,1)_k=1$,
\item[(iii)] $(a\cdot t^2,b)_k=(a,b)_k$,
\item[(iv)] If $(a_1,b)_k=1$, then $(a_1a_2,b)_k=(a_2,b)_k$, 
\item[(v)] For any $d\in k^{\times}$ with $a\neq d^2$, $(a,d^2-a)_k=1$.
\item[(vi)] $(a,b)_k=(a,-ab)_k$.
\end{itemize}
\end{lemma}
\begin{proof}
All properties are straightforward to show from the definition, except for property (iv). To prove this first notice that if $b$ is a square in $k$, then (iv) holds for arbitrary $a_1,a_2\in k^{\times}$. So we may assume that $b$ is not a square in $k$. Then, $K:=k[T]/(T^2-b)$ is a field extension of $k$ of degree $2$. Now, $(a,b)_k=1$ if and only if 
\[a=(z/x)^2-b(y/x)^2=N((z/x)+\sqrt{b}(y/x)).\]
Here we identify $\sqrt{b}$ with the class of $T$ in $K$, and $N$ denotes the norm of $K$ over $k$, i.e. $N(r+\sqrt{b}s)=(r+\sqrt{b}s)(r-\sqrt{b}s)=r^2-bs^2$. In other words, $(a,b)_k=1$ if and only if $a$ is a norm in the quadratic extension $K/k$, and the assumption $(a_1,b)_k=1$ implies $N(\alpha_1)=a_1$ for some $\alpha_1\in K^{\times}$. So if $(a_1a_2,b)_k=1$, then there is an $\alpha\in K^{\times}$ so that $a_1a_2=N(\alpha)$, but then by the multiplicativity of the norm
\[a_2=N(\alpha_1)^{-1}N(\alpha)=N(\alpha_1^{-1}\alpha),\]
and $(a_2,b)_k=1=(a_1a_2,b)_k$. Conversely, if $(a_1a_2,b)=-1$ then $(a_2,b)=-1$, otherwise $a_2=N(\alpha_2)$ for some $\alpha_2\in K^{\times}$ which then implies $a_1a_2=N(\alpha_1)N(\alpha_2)=N(\alpha_1\alpha_2)$.
\end{proof}

Property (iv) is \textit{almost} a property of bilinearity of the Hilbert symbol, in the sense that if $(a_1a_2,b)_k=-(a_2,b)_k$ when $(a_1,b)_k=-1$ then we have
\[(a_1a_2,b)=(a_1,b)(a_2,b).\]
 However, this property does \textit{not} hold over a general field.

\begin{example}\normalfont The rational Hilbert symbol is \textit{not} bilinear. For example, we show that $(3,2)_{\Q}=-1$ and  $(11,2)_{\Q}=-1$ yet $(33,2)_{\Q}=-1$. 
If $(3,2)_{\Q}=1$ then, without loss of generality, suppose that the equation $3x^2+2y^2=z^2$ has a non-trivial integral solution with $x,y$ and $z$ coprime. Then reducing modulo $3$ we find that $2y^2\equiv z^2\pmod{3}$, but $2$ is not a square modulo $3$. This implies that $3$ divides both $y$ and $z$. Therefore $y=3y_0$ and $z=3z_0$ for some $y_0,z_0\in\Z$, now
\[3x^2+2\cdot 3^2y_0^2=3^2z_0^2,\]
and dividing by $3$ we see that $x$ is also a multiple of $3$. This contradicts our assumption that $x$, $y$, and $z$ are coprime, therefore $(3,2)_{\Q}=-1$.\\
Since $2$ is also not a square modulo $11$, we can show analogously that $(11,2)_{\Q}=-1$. The argument above can be reproduced verbatim to show that $(33,2)_{\Q}=-1$.
\end{example}

In our examples above we showed how to determine that a rational quadratic equation has \textit{no} non-trivial solutions by reducing the equation modulo a prime and showing the resulting equation has no non-trivial solutions in $\F_p\simeq \Z/p$. Here we explore this in more detail, consider for example the rational equation
\[7x^2+35y^2=z^2.\]
We know from Example \ref{ex-FFSymbol} that $(7,35)_{\F_p}=1$, whenever $7$ and $35$ are units in $\F_p$ (upon identifying $\F_p\simeq \Z/p$). Hence, the obstructions reducing modulo a prime can only come from the primes $p=5$ and $p=7$. For $p=5$ we find reducing modulo $5$ that
\[7x^2\equiv 2x^2\equiv z^2\pmod{5},\]
however $2$ is a non-square residue modulo $5$, so the equation has no solutions in $\F_5$, hence no solutions in $\Q$. The situation at the prime $p=7$ is more nuanced, if we reduce modulo $7$ we find
\[0\equiv z^2\pmod{7}.\]
Therefore $z=7z_1$, for some $z_1\in \Z$. To find obstructions, we must then consider the equation modulo $7^2=49$. Here we find $7x^2+35y^2\equiv 7^2z_1^2\pmod{49}$, which is equivalent to 
\[x^2+5y^2\equiv 7z_1^2\equiv 0\pmod{7}.\]
This has a non-trivial solution $x=1$, $y=2$, and indeed
\[49\mid (7\cdot 1^2 +35\cdot 2^2 -7^2z_1^2)=147-49z_1^2, \text{ for any value of }z_1.\]
Perhaps we may find an obstruction by looking at the equation modulo $7^3=343$. Any solution modulo $7^3$ reduces to a solution modulo $7^2$ via the ring homomorphism
\begin{align*}
\Z/(7^3)&\rightarrow \Z/(7^2)\\
x &\mapsto x\pmod{7^2}
\end{align*}
Then without loss of generality we may begin by extending our existing solution. Namely, we let
\begin{align*}
x &= 1+7x_1,\\
y &= 2+7y_1,\\
z &= 0+7z_1,
\end{align*}
and substitute. Again $7x^2+35y^2=7^2z_1^2\pmod{7^3}$ if and only if $x^2+5y^2\equiv 7z_1^2\pmod{7^2}$. We find then,
\[(1+7x_1)^2+5(2+7y_1)^2\equiv 21+7\cdot 2x_1+7\cdot 20y_1\equiv 7z_1^2\pmod{7^2}.\]
The left-hand side is divisible by $7$, so this is equivalent to
\[3+2x_1+6y_1\equiv z_1^2\pmod{7}.\]
A possible solution is $x_1=y_1=1$, and $z_1=2$. Letting $x=1+1\cdot 7=8$, $y=2+1\cdot 7=9$ and $z_1=2\cdot 7$ we have
\[343\mid 343\cdot 9=3087=(7x^2+35y^2-z^2).\]
The reader may ask then if this process can go on indefinitely, and we find no obstructions from the prime $7$. A result known as \index{Hensel's lemma} \textit{Hensel's lemma} tells us that this is in fact the case. Under certain conditions on a multivariate integer polynomial $f$ and its formal derivatives, we can guarantee that a non-trivial solution to $f(a)=0$ modulo $p^k$ lifts to a solution modulo $p^n$ for all $n\geq k$. In this way, we obtain solutions to our equation as formal power series in $p$:
\begin{align*}
x=\sum_{n=0}^{\infty} x_np^n,\ 
y=\sum_{n=0}^{\infty} y_np^n,\text{ and }
z=\sum_{n=0}^{\infty} z_np^n,
\end{align*}
where $x_n,y_n,z_n\in \{0,\dots,p-1\}$. 

\begin{definition}\normalfont\label{def-AbsVal}
Let $\R_{\geq 0}$ denote the set of non-negative real numbers. An \index{absolute value} \textit{absolute value} on a field $k$ is a mapping $|\cdot|:k\rightarrow\R_{\geq 0}$ satisfying the following properties:
\begin{itemize}
\item[(i)] $|x|=0$ if and only if $x=0$,
\item[(ii)] $|xy|=|x||y|$ for all $x,y\in k$, and
\item[(iii)] $|x+y|\leq |x|+|y|$.
\end{itemize}
If an absolute value $|\cdot|$ satisfies the stronger property (iii)' $|x+y|\leq \max(|x|,|y|)$ then it is called \index{absolute value! non-archimedean}\textit{non-archimedean}. Otherwise $|\cdot|$ is \index{absolute value! archimedean}\textit{archimedean}.
\end{definition}

\begin{example}\normalfont If $k$ is an arbitrary field, the \index{absolute value! trivial}\textit{trivial absolute value} is defined as $|x|=1$ for all $x\in k-\{0\}$, and $|0|=0$. Clearly the trivial absolute value satisfies properties (i), (ii), and (iii)' so it is a non-archimedean absolute value.
\end{example}
 The power series above can be shown to be convergent under the following non-archimedean absolute value

\begin{definition}\normalfont\label{def-pAdicAbs} Let $x$ be a non-zero integer, then the \index{absolute value! $p$-adic} \textit{$p$-adic absolute value} of $x$ is defined as 
\[|x|_p=p^{-v_p(x)},\]
 where $v_p(x)$ is the largest power of $p$ dividing $x$. For $x=0$ the $p$-adic absolute value is defined as $|0|_p=0$. For a rational number $a/b\in \Q$ we let
 \[|a/b|_p=\frac{|a|_p}{|b|_p}.\]
\end{definition}
For example, $|300/23|_5=5^{-2}=1/25$ since $(300/23)=2^2\cdot 3\cdot 5^2\cdot 23^{-1}$, and $|300/23|_{23}=23$. Clearly if $p$ is coprime to both $a$ and $b$ in the fraction $a/b$, then $|a/b|_p=1$.\\

An absolute value $|\cdot|$ induces a metric $d:k\rightarrow\R$ given by $d(x,y)=|x-y|$ for $x,y\in k$, and with this metric the usual analytic notions can be defined.
\begin{definition} \normalfont Let $|\cdot|$ be an absolute value on a field $k$ and $\{a_n\}_{n=0}^{\infty}$ be a sequence of elements of $k$. Then $\{a_n\}$ is a \index{sequence!Cauchy}\textit{Cauchy sequence} (with respect to $|\cdot|$) if for all $\varepsilon\in\R$ with $\varepsilon>0$ there exists an integer $N>0$ such that
\[|a_m-a_{n}|<\varepsilon,\]
whenever $m,n\geq N$.
\end{definition}

\begin{definition}\normalfont\label{def-AbsEq}
Let $|\cdot|$ and $|\cdot|'$ be two absolute values on a field $k$. Then  $|\cdot|$ and $|\cdot|'$ are \textit{equivalent} if and only if for any sequence $\{a_n\}$ in $k$,
\[\{a_n\}\text{ is Cauchy with respect to } |\cdot| \text{ if and only if } \{a_n\}\text{ is Cauchy with respect to } |\cdot|'.\]
\end{definition}

\begin{definition}\normalfont \label{def-Place} Let $k$ be a field. A \index{place}\textit{place} of $k$ is an equivalence class of absolute values on $k$.
\end{definition}

\begin{definition}\normalfont A sequence $\{a_n\}_{n=0}^{\infty}$ in a field $k$ is said to be \index{sequence!convergent}\textit{convergent} (with respect to an absolute value $|\cdot|$) if and only if there is an $\ell\in k$ such that for all $\varepsilon>0$ there exists an integer $N>0$ such that
\[|a_n-\ell|<\varepsilon,\]
whenever $n\geq N$. In such a case $\ell$ is called the \textit{limit} of $\{a_n\}$.
\end{definition}

\begin{definition}\normalfont A field $k$ is \textit{complete} with respect to an absolute value $|\cdot|$ if and only if every Cauchy sequence in $k$ is convergent.
\end{definition}

Given a $p$-adic absolute value $|\cdot|_p$ we can construct the \index{field completion}\textit{completion} $\Q_p$ of $\Q$ with respect to $|\cdot|_p$. This is done by taking $\Q_p$ to be the set of equivalence classes of Cauchy sequences, where two Cauchy sequences $\{a_n\}$ and $\{b_n\}$ are in the same class if and only if $\{a_n-b_n\}$ is a convergent sequence with limit $0$. This is analogous to the construction of the real numbers $\R$ from $\Q$. The field $\Q_p$ is called the field of \index{p-adic numbers@$p$-adic numbers}\textit{$p$-adic numbers}, and contains $\Q$ as a subfield via the mapping $x\mapsto \{x\}_{n=0}^{\infty}$ for all $x\in\Q$.\\

  With this notion, the power series obtained by Hensel's Lemma give us an exact solution to $f(x,y,z)=0$ in $\Q_p$. In this setting Hensel's Lemma can be interpreted as the $p$-adic analogue of the \textit{Newton-Rhapson method} for finding successive approximations to the real roots of a polynomial.\\

One may ask then whether or not the field of $p$-adic numbers is isomorphic to the field of real numbers, and if there are other completions of $\Q$ other than these.
 The following theorem of Ostrowski characterises these completions
\begin{theorem}[Ostrowski, Theorem 1 \cite{Koblitz-pAdic}] Let $|\cdot|$ be a non-trivial absolute value on $\Q$. Then $|\cdot|$ is equivalent to $|\cdot|_p$ for some prime $p$, or $|\cdot|$ is equivalent to the usual absolute value on $\Q$.
\end{theorem}
It is not too difficult to show that $\Q_p\not\simeq \Q_q$ whenever $p$ and $q$ are \textit{distinct} primes. Likewise, it is easy to show that $\R\not\simeq \Q_p$ for all primes $p$. Hence, Ostrowski's theorem characterises \textit{all} possible completions of $\Q$.  To make notation uniform one typically denotes $\R=\Q_{\infty}$. In other words, the places of $\Q$ are in correspondence with prime numbers $p$ or $p=\infty$.\\

By Ostrowski's theorem, there could be a hope of finding a rational solution after determining that there are no obstructions in $\Q_p$ for any place $p$ of $\Q$. Remarkably, the following theorem of Hasse and Minkowski theorem shows that this is the case for quadratic forms: If a rational quadratic homogeneous polynomial has a root in $\Q_p$ for all places $p$, then it has a root in $\Q$.

\begin{theorem}[(Strong) Hasse local-global principle, Chapter IV, Theorem 8 \cite{Serre-ACourseInArithmetic}]\index{Hasse local-global principle! strong}\label{thm-StrongLocalGlobal}
Let $q$ be a rational quadratic form then $q(x)=0$ has a non-trivial solution in $\Q$ if and only if $q(x)=0$ has a non-trivial solution in $\Q_p$ for all places $p$ of $\Q$.
\end{theorem}

\begin{remark}\normalfont The Hasse local-global principle does not hold for general polynomial equations. In fact it already fails for certain cubics: For example, Selmer showed \cite{Selmer-Cubics} that
\[3x^3+4y^3+5z^3=0,\]
has a zero in $\Q_p$ for all places $p$, yet it has no rational solutions.
\end{remark}

Henceforth we denote $(a,b)_p:=(a,b)_{\Q_p}$, for every place $p$ of $\Q$. Using the strong local-global principle for quadratic forms, we can characterise the rational Hilbert symbol.

\begin{corollary}Let $a,b\in\Q^{\times}$, then $(a,b)_{\Q}=1$ if and only if $(a,b)_p=1$ for all places $p$.
\end{corollary}

There are several advantages to working with the ``local'' symbols $(a,b)_p$ instead of $(a,b)_{\Q}$. First, it can be shown that the square class group $\Gamma(\Q_p)=\Q_p^{\times}/(\Q_p^{\times})^2$ is finite for all places $p$ (we have seen this already for $p=\infty$). Furthermore, we have a closed formula for the Hilbert symbol in $\Q_p$. First, we introduce some notation:

\begin{definition}\normalfont 
For $a\in \Z$ and $p$ a prime, the \index{Legendre symbol}\textit{Legendre symbol} is defined as
\[\legendre{a}{p}=\begin{cases}
+ 1 & \text{ if } a \text{ is a square residue modulo } p\\
\phantom{+}0 & \text{ if } p\mid a\\
-1 & \text{ otherwise } 
\end{cases}
\]
\end{definition}

\begin{proposition}[Euler's Criterion, Proposition 5.1.2, \cite{Ireland-Rosen}] Given $a\in \Z$,
\[\legendre{a}{p}=a^{(p-1)/2}\pmod{p}.\]
\end{proposition}
From this it follows easily that the Legendre symbol is multiplicative, i.e. $\legendre{ab}{p}=\legendre{a}{p}\legendre{b}{p}$. We also have the following important relationship due to Gauss,

\begin{theorem}[Quadratic reciprocity, Chapter 5, Theorem 1 \cite{Ireland-Rosen}]\index{reciprocity!quadratic} Let $p$ and $q$ be odd primes, then
\[\legendre{p}{q}\legendre{q}{p}=(-1)^{(p-1)(q-1)/4}.\]
\end{theorem}
The theorem of quadratic reciprocity sometimes is presented alongside with the following \textit{supplements},
\begin{align*}
&\legendre{-1}{p}=(-1)^{(p-1)/2}, \text{ and}\\
&\legendre{2}{p}=(-1)^{(p^2-1)/8}.
\end{align*}

 Notice that since the Hilbert symbol is invariant under multiplication by squares, we may assume that $a$ and $b$ are square-free when computing $(a,b)_p$. We have
 
 \begin{proposition}[cf. Serre, Chapter III, Theorem 1, \cite{Serre-ACourseInArithmetic}]\normalfont\label{prop-HilbertSymbolFormula} Let $a,b\in \Z-\{0\}$. For a prime $p$, let $\alpha$ and $\beta$, $u$ and $v$ be integers such that 
 \[a=p^{\alpha}u,\text{ and } b=p^{\beta}v,\]
 where $p\nmid u$ and $p\nmid v$. Then,
\begin{itemize}
\item[(i)] if $p$ is odd,
 \[(a,b)_p=(-1)^{\alpha\beta\varepsilon(p)}\legendre{u}{p}^{\beta}\legendre{v}{p}^{\alpha},\]
 where $\varepsilon(x)=(x-1)/2$.
 \item[(ii)] If $p=2$, 
 \[(a,b)_2=(-1)^{\varepsilon(u)\varepsilon(v)+\alpha\omega(v)+\beta\omega(u)},\]
 where $\omega(x)=(x^2-1)/8$.
\end{itemize} 
For the archimedean place $p=\infty$ we have that $(a,b)_{\infty}=1$ if and only if $a,b>0$.

 \end{proposition} 

 We can also express the values of $(a,b)_p$ as a table, see \cite{Cassels-RationalQuadraticForms}. For an odd prime $p$, we have that a full set of representatives for the elements of $\Gamma(\Q_p)$ is $\{1,r,p,pr\}$ where $r$ is a non-square residue modulo $p$. Then,
\[
\begin{array}{c|cccc}
(a,b)_p & 1 & r & p & pr\\
\hline
1 & +1 & +1 & +1 & +1\\
r & +1 & +1 & -1 & -1\\
p & +1 & -1 & \varepsilon & -\varepsilon\\
pr& +1 & -1 &-\varepsilon & \varepsilon
\end{array}
\]
For $p=2$, the square class group $\Gamma(\Q_2)$ has order $8$ and a full set of representatives of its elements is $\{\pm 1,\pm 5, \pm 2, \pm 10\}$
\[
\begin{array}{c|cccccccc}
(a,b)_2 & 1 & 5 & -1 & -5 & 2 & 10 & -2 & -10\\
\hline
1 & +1 & +1 & +1 & +1 & +1 & +1 & +1 & +1\\
5 & +1 & +1 & +1 & +1 & -1 & -1 & -1 & -1\\
-1& +1 & +1 & -1 & -1 & +1 & +1 & -1 & -1\\
-5& +1 & +1 & -1 & -1 & -1 & -1 & +1 & +1\\
2 & +1 & -1 & +1 & -1 & +1 & -1 & +1 & -1\\
10& +1 & -1 & +1 & -1 & -1 & +1 & -1 & +1\\
-2& +1 & -1 & -1 & +1 & +1 & -1 & -1 & +1\\
-10&+1 & -1 & -1 & +1 & -1 & +1 & +1 & -1\\
\end{array}
\]
Recall also that for $p=\infty$, $\Gamma(\Q_{\infty})=\Gamma(\R)\simeq \{\pm 1\}$, and
\[
\begin{array}{c|cc}
(a,b)_{\infty} & 1 & -1\\
\hline
1 & +1 & +1\\
-1& +1 & -1
\end{array}
\]
From the closed formulas or the tables one can conclude the following:
\begin{theorem}[Chapter III, Theorem 2 \cite{Serre-ACourseInArithmetic}]
The symbol $(a,b)_p$ is bilinear for all places $p$. In other words,
\[(a_1a_2,b)_p=(a_1,b)_p(a_2,b)_p.\]
\end{theorem}
Additionally, we see that if $a$ and $b$ are coprime to $p$, then $(a,b)_p=1$. So there are only \textit{finitely many} values of $p$ for which $(a,b)_p\neq 1$. Finally, the local symbols satisfy the following local-global relation

\begin{theorem}[Hilbert reciprocity, Chapter III, Theorem 3 \cite{Serre-ACourseInArithmetic}]\label{thm-HilbertReciprocity}\index{reciprocity!Hilbert}
For every $a,b\in\Q^{\times}$,
\[\prod_{p}(a,b)_p=1,\]
where the product is taken over all places of $\Q$.
\end{theorem}

With the bilinearity property of the local symbols we can complete our discussion of rational congruences in the $2\times 2$ case. This time we find conditions to determine the existence of $X\in \GL_2(\Q)$ such that
\[X^{\intercal}\begin{bmatrix}
a & 0\\
0 & b
\end{bmatrix}
X=
\begin{bmatrix}
c & 0\\
0 & d
\end{bmatrix}.\]
\begin{theorem}\normalfont \label{thm-2x2HM} The regular rational quadratic spaces $\langle a,b\rangle$ and $\langle c,d\rangle$ are isomorphic, if and only if $ab\equiv cd$ in $\Gamma(\Q)=\Q^{\times}/(\Q^{\times})^2$ and $(a,b)_{p}=(c,d)_{p}$ for all places $p$.
\end{theorem}
\begin{proof}
Suppose that the rational quadratic spaces $\langle a,b\rangle$ and $\langle c,d\rangle$ are isomorphic. Then there is a matrix $X=\begin{bmatrix}
x_1 & x_2\\
y_1 & y_2
\end{bmatrix}\in \GL_2(\Q)$ such that $X^{\intercal}\diag(a,b)X=\diag(c,d)$. Computing this matrix product we find
\[\begin{bmatrix}
ax_1^2+by_1^2 &ax_1x_2+by_1y_2\\
ax_1x_2+by_1y_2 & ax_2^2+by_2^2
\end{bmatrix}
=
\begin{bmatrix}
c & 0\\
0 & d
\end{bmatrix}.
\]
In particular, there is a non-trivial rational solution to the equation $ax_1^2+by_1^2=c$. By regularity of $\langle c,d\rangle$ one has $c\neq 0$, and dividing this equation by $c$ we find
\[\frac{a}{c}x_1^2+\frac{b}{c}y_1^2=1.\]
Hence $(a/c,b/c)_{\Q}=(ac,bc)_{\Q}$, which implies $(ac,bc)_{p}=1$ for all places $p$. Taking determinants in the expression $X^{\intercal}\diag(a,b)X=\diag(c,d)$ we find that $ab\equiv cd$ in $\Gamma(\Q)$ (hence in $\Gamma(\Q_p)$). Multiplying by $bd$ in both sides we find that $ad\equiv bc$ in $\Gamma(\Q_p)$. By bilinearity of the local symbols
$1=(ac,bc)_p=(a,bc)_{p}(c,bc)_{p}$, so $(a,bc)_p=(c,bc)_p$. Using bilinearity again and the fact that $ad\equiv bc$ we find
\[(a,b)_{p}(a,c)_{p}=(a,bc)_p=(c,bc)_p=(c,ad)_p=(c,a)_{p}(c,d)_{p}.\]
By symmetry we cancel $(a,c)_p$ in the left-hand-side with $(c,a)_p$ in the right-hand-side, and it follows that $(a,b)_{p}=(c,d)_{p}$.\\

Conversely, suppose that $ab\equiv cd$ in $\Q^{\times}/(\Q^{\times})^2$, and that $(a,b)_{p}=(c,d)_{p}$ for all places $p$. Then 
\begin{align*}
(a,b)_{p}&=(c,-dc)_{p}=(c,-ab)_{p}=(c,-a)_{p}(c,b)_{p}.
\end{align*}
From where it follows
\[(ac,b)_{p}=(c,-a)_{p}=(c,ac)_{p}.\]
Therefore $(ac,bc)_{p}=(a/c,b/c)_p=1$ for all places $p$. The local-global principle (Theorem \ref{thm-StrongLocalGlobal}) implies that there is a rational solution to $ax_1^2 +by_1^2=c$. If $y_1=0$, then $a\equiv c$ in $\Gamma(\Q)$, and this implies that $b\equiv d$ in $\Gamma(\Q)$. Therefore, there is an $x_2\in\Q^{\times}$ such that $bx_2^2=d$, and in this case $X=\diag(x_1,x_2)$ is the sought matrix. So we may  assume that $y_1\neq 0$. In this case we may let $y_2=-ax_1x_2/(by_1)$, where $x_2$ is an indeterminate. Substituting $y_2$ into the expression $ax_2^2+by_2^2=d$, we find
\[bdy_1^2=ax_2^2(by_1^2+ax_1^2)=acx_2^2.\]
And since $bd/ac$ is a square, we find that $x_2$ is in $\Q$ and determined up to sign from $y_1$. Then, $y_2$ is determined uniquely from $x_1,y_1$, and $x_2$, and from this it follows that $X=\begin{bmatrix}
x_1 &x_2\\
y_1 &y_2
\end{bmatrix}$ is a rational solution to the congruence equation $X^{\intercal}\diag(a,b)X=\diag(c,d)$.
\end{proof}

\section{Invariants of quadratic forms}\label{sec-InvariantsQF}

In general, taking determinants in the equation $A=X^{\intercal}BX$, we have $\det(A)=\det(X)^2\det(B)$. This implies that the determinant, as an element of $\Gamma(k)=k^{\times}/(k^{\times})^2$, is an \textit{invariant} for the equivalence of quadratic forms over an arbitrary field $k$.

\begin{definition}\normalfont \label{def-Discriminant} Let $A$ be the matrix of a quadratic form $q$ with respect to some basis. The class of the determinant $\det(A)$ in $\Gamma(k)$ is called the \textit{discriminant} of $q$ and it is denoted by $\delta(q)$, or simply $\delta$ when there is no chance of confusion.
\end{definition}
 In the case $k=\Q$, we found that the $p$-adic Hilbert symbols $(a,b)_p$ give us, together with the discriminant, a \textit{complete} set of invariants of quadratic forms in dimension $2$. More generally, we have that the following is a complete set of invariants for the equivalence of rational quadratic forms of \textit{any} dimension:

\begin{itemize}
\item The \index{discriminant} discriminant $\delta=\delta(q)\in \Q^{\times}/(\Q^{\times})^2$.
\item The \index{signature} \textit{signature} $\sigma=\sigma(q)$ of the quadratic form when considered as a \textit{real} quadratic form.
\item The \index{Hasse-Minkowski!invariants}\textit{local Hasse-Minkowski invariants} $\varepsilon_p(q)$ for every place $p$ of $\Q$. 
\end{itemize}

\begin{definition}\label{def-HMInvariants}\normalfont
Let $q$ be a quadratic form represented by a diagonal matrix $A\simeq \diag(\alpha_1,\dots,\alpha_n)$. We define the \textit{Hasse-Minkowski invariants} of $q$  as follows:
\[\varepsilon_p(q)=\varepsilon_p(A)=\prod_{i<j}(a_i,a_j)_p\]
For a $1\times 1$ matrix $(\alpha)$ we define $\varepsilon_p(\alpha)=1$ for all places $p$.
\end{definition}
A real symmetric matrix $M$ is congruent to a diagonal matrix $D$ with entries $\pm 1$ (which are a full set of representatives of $\Gamma(\R)$). Then the signature of $M$ is defined as $\sigma(M)=(n_{+},n_{-})$ where $n_+$ is the number of $+1$s in $D$, and $n_{-}$ is the number of $-1$s. Notice that by the spectral theorem, $n_{+}$ coincides with the number of positive eigenvalues of $M$. The signature is an invariant of real quadratic forms by Sylvester's law of inertia (Theorem \ref{thm-SylvesterInertia}). The reader may have noticed that we do not mention signatures in Theorem \ref{thm-2x2HM}, this is because if two $2\times 2$ real matrices $A=\diag(a,b)$ and $B=\diag(c,d)$ have the same discriminant and $(a,b)_{\infty}=(c,d)_{\infty}$, then $A$ and $B$ have the same signature.\\

By complete set of invariants, we mean that the following holds:
\begin{theorem}[Hasse-Minkowski, Chapter IV, Theorem 9 \cite{Serre-ACourseInArithmetic}]\label{thm-HasseMinkowski}\index{Hasse-Minkowski!theorem}
Two rational quadratic forms $q$ and $q'$ of the same dimension are equivalent if and only if $\delta(q)=\delta(q')$, $\sigma(q)=\sigma(q')$ and $\varepsilon_p(q)=\varepsilon_p(q')$ for all primes $p$.
\end{theorem}  

This is known as the\index{Hasse local-global principle!weak} \textit{(weak) Hasse local-global principle}, because it is implied by the strong Hasse local-global principle. The proofs of both these theorems are not inaccessible, but they require a fair amount of background material.  This technical difficulty is mostly present in showing that equality in the invariants implies equivalence of the forms. However, for combinatorial applications, a non-integral solution to $XX^{\intercal}=M$ is typically uninteresting. For us it is sufficient to show that the above is a set of \textit{partial} invariants for the equivalence of rational quadratic forms, since this is enough to determine the non-solvability of Grammian equations. Furthermore, we are only interested in conditions to decide the equivalence of a quadratic form $q_M(x)=x^{\intercal} M x$ to the \textit{standard quadratic form} \index{form!quadratic!standard}
\[\iota_n(x):=q_{I_n}(x)=x^{\intercal}x=x_1^2+x_2^2+\dots+x_n^2.\] 
In particular we must have $\delta(M)=\delta(I)=1$, and the assumption that $M$ is positive-definite already implies that its signature is $\sigma(M)=(n,0)=\sigma(I)$. Therefore we only need to consider the local invariants $\varepsilon_p(M)$. Positive-definiteness also implies $\varepsilon_{\infty}(M)=1$, hence by Hilbert reciprocity (Theorem \ref{thm-HilbertReciprocity}), if $(a,b)_2=-1$ then $(a,b)_p=-1$ for some odd prime $p$. We summarise this as follows

\begin{theorem}Let $M$ be a rational, positive-definite matrix. If $XX^{\intercal}=M$ for some rational matrix $X$, then 
\begin{itemize}
\item $\det(M)$ is a square, and
\item $\varepsilon_p(M)=1$ for all odd primes $p$.
\end{itemize}
\end{theorem}

For completeness we present a matrix-theoretic proof of the fact that $\varepsilon_p(A)$ are invariants for the equivalence of quadratic forms for all primes $p$. Namely we show
\[\varepsilon_p(X^{\intercal}AX)=\varepsilon_p(A).\]
To do so, we use a generalisation of the Hasse-Minkowski invariants due to Pall \cite{Pall-Invariants}.

\begin{definition}\normalfont \label{def-PallInvariants}
Let $A$ be a rational symmetric matrix, the \textit{Hasse-Pall} invariants \index{Hasse-Pall invariants} are defined for every place $p$ as
\[c_p(A)=(-1,-\delta_n)_p\prod_{i=1}^{n-1}(\delta_i,-\delta_{i+1}),\]
where $\delta_i$ is the $i$-th leading principal minor of $A$.
\end{definition}

\begin{proposition}\normalfont If $A$ is a rational diagonal matrix with discriminant $\delta(A)=1$, then for all odd primes $p$ 
\[\varepsilon_p(A)=c_p(A).\]
\end{proposition}
\begin{proof}
If $A$ is diagonal, say $A=\diag(a_1,\dots,a_n)$ we have that the $i$-th leading principal minor of $A$ is $\delta_i=a_1\dots a_i$. Using $\delta_n=\delta(A)=1$, and bilinearity we find
\begin{align*}
c_p(A)&=(-1,\delta(A))_p\prod_{i=1}^{n-1}(\delta_i,-\delta_i \cdot a_{i+1})_p\\
&=(-1,-1)_p\prod_{i=1}^{n-1}(\delta_i,-\delta_i)_p(\delta_i,a_{i+1})_p
\end{align*}
If $p$ is odd then $-1\equiv p-1\pmod{p}$ is coprime to $p$, hence $(-1,-1)_p=1$. From the relation $(a,-a)_p=1$ we then find
\[c_p(A)=\prod_{j=1}^{n-1}(\delta_j,a_{j+1})_p=\prod_{j=2}^n\prod_{i=1}^{j-1}(a_i,a_j)_p=\prod_{i<j}(a_i,a_j)_p=\varepsilon_p(A).\qedhere\]
\end{proof}

To give an elementary proof that the Pall invariants are indeed invariants under rational congruence we will require a lemma on determinants. This lemma was notably used by Jacques Hadamard in the original proof of his celebrated determinant bound \cite{Hadamard-Determinants}.

\begin{lemma}\normalfont\label{lemma-DeterminantLemma} Let $M$ be an $n\times n$ symmetric positive-definite matrix. For $i\neq j$, let $M_{[i,j]}$ be the $(n-2)\times (n-2)$ submatrix obtained from $M$ by removing the $i$-th and $j$-th rows and the $i$-th and $j$-th columns. Then
\[\det(M)\det(M_{[i,j]})=M_{i,i}M_{j,j}-(M_{i,j})^2,\]
where $M_{i,j}$ denotes the $(i,j)$-th minor of $M$.
\end{lemma}
\begin{proof}
Write $M$ as a block-matrix of the type
\[
M=\begin{bmatrix}
M_1 & M_2\\
M_3 & M_4
\end{bmatrix}.
\]
Since $M$ is positive-definite, it is invertible. Letting $N$ be the inverse of $M$, we can write 
\[
N=\begin{bmatrix}
N_1 & N_2\\
N_3 & N_4
\end{bmatrix}.
\]
Then, it follows that
\[\begin{bmatrix}
N_1 & N_2\\
N_3 & N_4
\end{bmatrix}
\begin{bmatrix}
M_1 & 0\\
M_3 & I
\end{bmatrix}
=
\begin{bmatrix}
I & N_2\\
0 & N_4
\end{bmatrix}.
\]
Taking determinants, it follows that $\det(N)\det(M_1)=\det(N_4)$. Hence
\[\frac{\det(M_1)}{\det(M)}=\det(N_4).\]
Let $N_4$ be a $2\times 2$ submatrix. Since the determinant of $M$ is unchanged after a \textit{symmetric} row/column permutation, we may assume without loss of generality that $M_1=M_{[i,j]}$ and that
\[M_4=\begin{bmatrix}
m_{ii} & m_{ij}\\
m_{ij} & m_{jj}
\end{bmatrix}.\]
Hence, using the cofactor formula for the inverse $N=M^{-1}$ we find that
\[N_4=\frac{1}{\det(M)}\begin{bmatrix}
M_{i,i} & -M_{i,j}\\
-M_{i,j} & M_{j,j}
\end{bmatrix}.\]
Therefore
\[\frac{\det(M_{[i,j]})}{\det(M)}=\det(N_4)=\frac{1}{\det(M)^2}(M_{i,i}M_{j,j}-(M_{i,j})^2).\]
Multiplying by $\det(M)^2$, we conclude the proof.\qedhere
\end{proof}

Below we present our matrix-theoretic proof of the Hasse-Minkowski theorem for positive-definite matrices, this result will appear in the paper \cite{InvariantsPaper}.
\begin{theorem}[cf. \cite{Jones-ArithmeticQF}, and \cite{Pall-Invariants}]\label{thm-MyHM}
If $M$ is an $n\times n$ symmetric positive-definite rational matrix, then for each $X\in \GL_n(\Q)$
\[c_p(M)=c_p(X^{\intercal}MX)\]
for all rational places $p$.
\end{theorem}
\begin{proof}
The group $\GL_n(\Q)$ is generated by permutation matrices, row-multiplying matrices, and elementary row-operation matrices. It is then sufficient to prove the claim for each generator of $\GL_n(\Q)$:\\

Let $X=\diag(1,\dots,\lambda,\dots,1)$ be a row-multiplying matrix, where $\lambda$ appears in the $i$-th diagonal element. Let $M'=X^{\intercal}MX$, then the leading principal minors $\delta_j'$ of $M'$ clearly satisfy $\delta_j'=\delta_j$ for $j<i$, and $\delta_j'=\lambda^2 \delta_j$ for $j\geq i$. But since $(a,b)_p=(a\cdot \lambda^2,b)_p$ for all $p$, it follows that $c_p(N)=c_p(X^{\intercal}MX)=c_p(M)$.\\

Now, let $X=P$ be a permutation matrix. It is sufficient to show that the claim holds whenever $P$ corresponds to the permutation of two \textit{consecutive} indices. Let $P$ be the permutation matrix corresponding to the permutation $(i,i+1)\in S_n$. If $M'=P^{\intercal}MP$, then all leading principal minors $\delta'_j$ of $N$ coincide with those of $M$ except perhaps for $\delta'_i$ and $\delta_i$, which may differ. Thus it suffices to show that
\[(\delta_{i-1},-\delta'_{i})_p(\delta'_i,-\delta_{i+1})_p=(\delta_{i-1},-\delta_{i})_p(\delta_{i},-\delta_{i+1})_p.\]
The bilinearity of the local symbols implies that $(\delta_{i-1},\delta'_i\delta_i)_p(\delta'_i\delta_i,-\delta_{i+1})_p=1$, then by symmetry and bilinearity again
\[(-\delta_{i-1}\delta_{i+1},\delta'_i\delta_i)_p=1.\]
To prove the above identity, we apply Lemma \ref{lemma-DeterminantLemma} to the $(i+1)$-th leading principal submatrix of $M$, denoted $M(i+1)$. We have that
\[M(i+1)=
\left[
\begin{array}{c|cc}
M(i-1) & \alpha & \beta\\
\hline
\alpha^{\intercal} & m_{ii} & m_{i,i+1}\\
\beta^{\intercal} & m_{i,i+1} &m_{i+1,i+1}
\end{array}
\right], \text{ and }
M'(i+1)=
\left[
\begin{array}{c|cc}
M(i-1) & \beta & \alpha\\
\hline
\beta^{\intercal} & m_{i+1,i+1} & m_{i,i+1}\\
\alpha^{\intercal} & m_{i,i+1} &m_{i,i}
\end{array}
\right].
\]
Therefore,
\[\det(M(i+1))\det(M(i-1))=
\det
\begin{bmatrix}
M(i-1) & \alpha\\
\alpha^{\intercal} & m_{ii}
\end{bmatrix}
\det
\begin{bmatrix}
M(i-1) &  \beta\\
\beta^{\intercal} & m_{i+1,i+1}
\end{bmatrix}
-d^2,
\]
for some $d\in \Q$. Hence,
\[\delta_{i+1}\delta_{i-1}=\delta_i\delta'_i-d^2,\]
and by positive-definiteness $\delta_i\delta'_i-d^2=\delta_{i+1}\delta_{i-1}\neq 0$. Therefore,
\[(-\delta_{i-1}\delta_{i+1},\delta'_i\delta_i)_p=(d^2-\delta'_i\delta_i,\delta'_i\delta_i)_p=1,\]
since $(d^2-a,a)_p=1$ whenever $a,d^2-a\in \Q^{\times}$. It remains to show the claim when $X$ is an elementary row-operation matrix. But since $c_p$ is invariant under permutations, we may assume that $X$ changes only the last row of $M$. Since $\det(X)=1$, all minors of $M'=X^{\intercal}MX$ are unchanged. This concludes the proof.
\end{proof}

\begin{remark*}\normalfont
Notice that this result remains true whenever the Hilbert symbol over the field $k$ is bilinear.
\end{remark*}
Now that we have presented the basic theory of quadratic forms and given a set of invariants for their equivalence, it is time to put these tools to practice. In the next chapter we will prove the BRC and Bose-Connor theorems.

\cleardoublepage
\chapter{Invariants of Quadratic forms in Design Theory}\label{chap-BRC}

In the last chapter, we studied the theory of quadratic forms, and explained how to use this theory to decide the solvability of the Gram matrix equation $XX^*=M$ over the rationals. Now we will apply these tools to obtain non-existence conditions for families of combinatorial designs. Here we will assume that the reader is familiar with Proposition \ref{prop-HilbertSymbolFormula} and with the contents of Section \ref{sec-InvariantsQF}, particularly Theorem \ref{thm-HasseMinkowski}.\\

Combinatorial designs, or just designs, are finite structures consisting of points and blocks that are ``balanced'' in some sense. This could mean for example that every point is in the same number of blocks, or that any pair of blocks has the same number of points in common. Such properties are typically called regularity properties. Combinatorial designs receive their name from their widespread use in the statistical theory of design of experiments since the early 20th century. However, the origins of design theory trace back at least to antiquity, see the nice historical account in Part I of the Handbook of Combinatorial Designs \cite{HandbookOfDesigns}.\\

As we remarked in the introduction to Chapter \ref{chap-GramEquations}, many combinatorial structures, including designs, can be characterised by the Gram equation of their incidence matrix. For example, there exists a  symmetric $2$-$(v,k,\lambda)$ design if and only if there is a square matrix $N$, with entries $0$ or $1$, such that its Gram matrix is:
\[NN^{\intercal}=(k-\lambda)I_v+\lambda J_v,\]
 The BRC Theorem \cite{Bruck-Ryser,Ryser-Chowla} assumes the existence of a $2$-$(v,k,\lambda)$ design, to find such a Gram equation, and then extracts necessary conditions that $v$, $k$, and $\lambda$ must satisfy whenever said equation has a solution over the rationals. In this way, we can rule out several families of parameters $(v,k,\lambda)$.\\
 
 We use the theory of quadratic forms to give two new proofs of the BRC theorem. The first one is inspired by several existing proofs in the literature, and will appear in the paper \cite{InvariantsPaper}. And the second one actually shows a slightly stronger statement. Namely, we extract the same conditions as in the BRC Theorem on $(v,k,\lambda)$, but without the assumption that a $2$-$(v,k,\lambda)$ design exists. This is important, because if $N$ is the incidence matrix of a design, then $N$ has constant row-sum. However, in some other applications, the assumption of constant row-sum for a solution $X$ to $XX^{\intercal}=(k-\lambda)I_v+\lambda J_v$ may not need to hold.\\

We will begin the chapter by giving a brief self-contained review of design theory. Then, we will present our two proofs the BRC theorem. After, we will give a proof of the Bose-Connor Theorem \cite{Bose-Connor}, which is an extension of the BRC Theorem to the class of group-divisible designs. We remark, that our second proof of the BRC Theorem, and our proof of the Bose-Connor Theorem both use ideas from the theory of association schemes. This puts both theorems under a common framework, which has the advantage of providing a more systematic approach to both. Finally, we present an application of the Bose-Connor Theorem to the theory of $\pm 1$ maximal determinant matrices, due to Tamura \cite{Tamura-DOptimal}.

\section{Design theory}
Since our goal application is the BRC theorem, which deals with symmetric $2$-designs, we present a brief introduction to design theory. For texts on design theory we refer the reader to \cite{Beth-Jungnickel-Lenz, Stinson-Designs, VanLint-Wilson}.\\

An \index{incidence structure} \textit{incidence structure} consists of a set of points $\mathcal{P}$ and a set of blocks $\mathcal{B}$ together with an incidence relation $I\subseteq \mathcal{P}\times \mathcal{B}$, which specifies which points are incident with which blocks. Namely, we say that a point $p$ is ``incident to the block'' $B$  if and only if $(p,B)\in I$, also written as $p\inc B$. Let $\mathcal{S}=(\mathcal{P},\mathcal{B},I)$ be an incidence structure. Fixing an ordering of $\mathcal{P}$ and $\mathcal{B}$, we define the \index{matrix!incidence} \textit{incidence matrix} of $\mathcal{S}$ with respect to this ordering as the $|\mathcal{P}|\times |\mathcal{B}|$ matrix,
\[(A_{\mathcal{S}})_{p,B}=\begin{cases}
1 & \text{ if } (p,B)\in I\\
0 & \text{ otherwise }
\end{cases}.
\]
If $A$ and $A'$ are two incidence matrices for $\mathcal{S}$, then there exist permutation matrices $P$ and $Q$ such that
\[PAQ=A'.\]
By directly computing the matrix product one can see that 
\[(A_{\mathcal{S}}A_{\mathcal{S}}^{\intercal})_{p,q}=\#\{B\in \mathcal{B}: (p,B)\in I,\text{ and } (q,B)\in I\}.\]
Thus, the Gram matrix of an incidence matrix counts the number of blocks that are incident to two given points, and we can characterise the Gram matrix of $\mathcal{A}_{S}$ using regularity properties of $\mathcal{S}$. 

\begin{definition}\normalfont
A \index{design!$2$-design} \textit{$2$-$(v,k,\lambda)$ design }(or $2$-design) is an incidence structure $(\mathcal{P},\mathcal{B},I)$ with $|\mathcal{P}|=v$, where each block is incident to $k$ points, and every pair of points is incident to $\lambda$ blocks.
\end{definition}
 
 More generally, we can define $t$-$(v,k,\lambda)$ designs, or \textit{$t$-designs} for short. A $t$-$(v,k,\lambda)$ design is an incidence structure on $v$ points for which every block is incidence to $k$ points, and every $t$-subset of points is incident to exactly $\lambda$ blocks, i.e. if $S\subseteq \mathcal{P}$ and $|S|=t$, then
\[\#\{B: p \inc B,\text{ for all } p\in S\}=\lambda.\]
We can always find designs at every order if we allow $k$ to be $1$ or $v$, but such designs are uninteresting. We say that a $t$-design is \index{design!trivial}\textit{trivial} if $k\in\{v-1,v\}$ or if $k\leq 1$. 
\begin{example}\normalfont Let $\mathcal{P}$ be the set of non-zero vectors of $\F_2^3$, and let $\mathcal{B}=\{\{x,y,x+y\}: x,y\in\mathcal{P}\}$. Define an incidence relation by $(x,\ell)\in I\subseteq \mathcal{P}\times \mathcal{B}$ if and only if $x\in \ell$. If we are given a pair $(x,y)$ of vectors in $\mathcal{P}$ with $x\neq y$ , then there is a unique vector $z$ such that $\{x,y,z\}\in\ell$, namely $z=x+y$. This shows that $(\mathcal{P},\mathcal{B},I)$ is a $2$-$(7,3,1)$ design. This design is known as the \textit{Fano plane}, pictured below
\begin{figure}[H]
\centering
\begin{tikzpicture}[
mydot/.style={
  draw,
  circle,
  fill=black,
  inner sep=1.5pt}
]
\draw
  (0,0) coordinate (A) --
  (3cm,0) coordinate (B) --
  ($ (A)!.5!(B) ! {sin(60)*2} ! 90:(B) $) coordinate (C) -- cycle;
\coordinate (O) at
  (barycentric cs:A=1,B=1,C=1);
\draw (O) circle [radius=3cm*1.717/6];
\draw (C) -- ($ (A)!.5!(B) $) coordinate (LC); 
\draw (A) -- ($ (B)!.5!(C) $) coordinate (LA); 
\draw (B) -- ($ (C)!.5!(A) $) coordinate (LB); 
\foreach \Nodo in {A,B,C,O,LC,LA,LB}
  \node[mydot] at (\Nodo) {};    
\end{tikzpicture}%
\caption{The Fano plane.}
\end{figure}
More generally consider the incidence structure $(\mathcal{P},\mathcal{B},I)$ where $\mathcal{P}$ and $\mathcal{B}$ are the set of all $1$-dimensional, and $2$-dimensional vector subspaces of $\F_q^3$ respectively. If we let $(\ell,\pi)\in I$ if and only if $\ell$ is a subspace of $\pi$, we obtain a $2$-$(q^2+q+1,q+1,1)$ design known as a \textit{projective plane} of order $q$.
\end{example}

In general we define projective planes as follows

\begin{definition}\normalfont \index{projective plane} A \textit{projective plane} is an incidence structure $(\mathcal{P},\mathcal{L},I)$ consisting of points and lines such that:
\begin{itemize}
\item[PP1.] For any pair $x,y$ of distinct points of $\mathcal{P}$ there is a unique line incident to both $x$ and $y$.
\item[PP2.] Every pair of distinct lines meets at a unique point.
\item[PP3.] There are four points in $\mathcal{P}$ such that no three of them lie in the same line.
\end{itemize}
\end{definition}
If a line of a projective plane $\Pi$ is incident to exactly $n+1$ points then \textit{all} lines of $\Pi$ are incident to exactly $n+1$ points, and the number $n$ is called the \textit{order} of $\Pi$. It is easy to see that a projective plane of order $n$ gives rise to a $2-(n^2+n+1,n+1,1)$ design.\\

In our definition of design there is no mention to the number of blocks of $\mathcal{B}$. It turns out that the imposed regularity conditions are enough to determine the number of blocks. We have
\begin{lemma}\label{countblocks} The number of blocks of a $t$-$(v,k,\lambda)$ design is
\[b=\lambda{v\choose t}/{k\choose t}.\]
\end{lemma}
\begin{proof}
Count in two different ways the number of pairs $(T,B)$ where $T$ is a $t$-subset of points of the design, and $B$ is a block incident to all points of $T$. This yields
\[{v\choose t}\lambda = b{k\choose t},\]
and the result follows.
\end{proof}

This shows in particular that the parameters of a $t$-$(v,k,\lambda)$ design are not independent. Indeed we have the following stronger relation.

\begin{lemma}\label{general-blockcount} Let $0\leq i\leq t$, then for a $t$-$(v,k,\lambda)$ design, the number of blocks incident to all points of any $i$-subset $I\subseteq \mathcal{P}$ is
\[\lambda_i = \lambda{v-i\choose t-i}/{k-i\choose t-i}.\]
In particular a $t$-$(v,k,\lambda)$ design is also an $i$-$(v,k,\lambda)$ design for $0\leq i\leq t$.
\end{lemma}
\begin{proof} Count in two ways the number of pairs $(T,B)$ with $I\subseteq T$, $|T|=t$ and $B\in\mathcal{B}$ incident with all points of $T$. If we let $\lambda_I$ the number of blocks which are incident with all points of $I$ we have
\[{v-i\choose t-i}\lambda = \lambda_I{k-i\choose t-i}.\]
From here we find that $\lambda_I=\lambda{v-i\choose t-i}/{k-i\choose t-i},$ and this expression only depends on the cardinality of $I$.  
\end{proof}
Note that in the result above $\lambda_0$ is simply the number of blocks $b$ of the $t$-$(v,k,\lambda)$ design and we recover Lemma \ref{countblocks}. The value $\lambda_1$ is the number of blocks incident with any point $p$ of the design, this is known as the \textit{replication number} of the design, and it is denoted $r$. With this notation Lemma \ref{general-blockcount} gives
\[r=\lambda{v-1 \choose t-1}/{k-1\choose t-1}.\]
For the case of a $2$-design Lemma \ref{general-blockcount} gives
\[bk=rv,\]
and 
\[\lambda(v-1)=r(k-1).\]
These results give strong arithmetic
 restrictions to the existence of $t$-design. The following example, taken from \cite{VanLint-Wilson}, demonstrates this.
\begin{example}\normalfont If a $3$-$(v,6,1)$ design exists, then $v\equiv 2,6\pmod{20}$. Let $\mathcal{D}$ be a design with these parameters, then by Lemma \ref{general-blockcount} we have three non-trivial conditions
\begin{itemize}
\item $b=b_0={v\choose 3}/{6\choose 3}=v(v-1)(v-2)/120\in \Z$,
\item $r=b_1={v-1\choose 2}/{5\choose 2}=(v-1)(v-2)/20\in\Z$, and
\item $b_2={v-2\choose 1}/{4\choose 1}=(v-2)/4\in \Z$.
\end{itemize}
From the last condition we find that $v\equiv 2 \pmod{4}$, hence $v\equiv 2,6,10,14,18\pmod{20}$. The condition $r=(v-1)(v-2)/20$ implies that $(v-1)(v-2)\equiv 0\pmod{20}$. Out of the five possibilities for $v\pmod{20}$ the only ones that satisfy $(v-1)(v-2)\equiv 0\pmod{20}$ are $v\equiv 2,6\pmod{20}$, from which the claim follows. Notice that the first condition $v(v-1)(v-2)\equiv 0\pmod{120}$ is always satisfied when $v\equiv 2,6\pmod{20}$, so no further restrictions on the congruence class of $v$ modulo $20$ can be found in this way.
\end{example}

\begin{theorem}[Fisher's Inequality]\label{Thm-Fisher-Inequality} \index{Fisher's inequality} In a non-trivial $2$-$(v,k,\lambda)$ design the number of blocks $b$ satisfies the inequality
\[b\geq v.\]
\end{theorem}
\begin{proof}
 Since the design is not trivial, we have that $k<v$. From the equation $\lambda (v-1)=r(k-1)$ we find that $r>\lambda$. Now, let $N$ be the $v\times b$ incidence matrix of a $2$-$(v,k,\lambda)$ design. Then $NN^{\intercal}=(r-\lambda)I_v+\lambda J_v$. Since the eigenvalues of $J_v$ are $v$ and $0$ with multiplicity $1$ and $v-1$ respectively, we find that the eigenvalues of $(r-\lambda)I_v+\lambda J_v$ are $(r-\lambda)+v\lambda$ and $(r-\lambda)$ with multiplicities $1$ and $v-1$, respectively. Therefore, taking determinants we have
\[\det(NN^{\intercal})=\det((r-\lambda)I_v+\lambda J_v)=(r+\lambda(v-1))(r-\lambda)^{v-1}=rk(r-\lambda)^{v-1}>0.\]
This implies that $N$ must have full column rank, and so the number of rows of cannot exceed the number of columns. This is equivalent to $b\geq v$.
\end{proof}

A symmetric $2$-$(v,k,\lambda)$ design \index{design!symmetric} is a design for which $v=b$. For such a design, its incidence matrix is square. For such a design the following important fact follows
\begin{theorem}[Chapter 8, Theorem 2.1 \cite{Ryser-CombinatorialMathematics}]\label{thm-SymmDesignNormal}
The incidence matrix $N$ of a symmetric $2$-$(v,k,\lambda)$ design is normal, i.e.
\[NN^{\intercal}=N^{\intercal}N=(k-\lambda)I_v+J_v.\]
\end{theorem}


\section{The Bruck-Ryser-Chowla Theorem}\label{sec-BRC}
The Bruck-Ryser-Chowla Theorem (or BRC Theorem) is a fundamental non-existence result in the theory of symmetric designs. The precursor to this theorem appeared in 1948 in the paper of Bruck and Ryser \cite{Bruck-Ryser}. Here the authors give necessary conditions for the existence of projective planes of order $n$. Namely, it is proven that an odd prime $p\equiv 3\pmod{4}$ cannot divide the square-free part of $n$ when $n\equiv 1\text{ or } 2\pmod{4}$. From this follows in particular that there is no projective plane of order $6$, something that had been proved in a purely combinatorial way by Gaston Tarry in \cite{Tarry-I, Tarry-II} as a consequence of the non-existence of a solution to Euler's 36 officers problem. Here we will present two proofs of the Bruck-Ryser-Chowla Theorem, one closer to the original proof assuming the existence of a symmetric $2$-$(v,k,\lambda)$ design, and  a stronger version of this result that does not require this assumption.  which is inspired by the proofs of the BRC Theorem in \cite{Beth-Jungnickel-Lenz} and \cite{VanLint-Wilson}.\\

Recall the notation $\langle a_1,\dots,a_n\rangle$ for the quadratic form (or quadratic space) induced by the diagonal matrix $\diag(a_1,\dots,a_n)$. The following equivalence of quadratic forms is well-known, we give a new proof of this result based on the proofs in \cite{Beth-Jungnickel-Lenz}, and \cite{VanLint-Wilson}.

\begin{proposition}[cf. \cite{Beth-Jungnickel-Lenz,VanLint-Wilson}]\label{prop-BRCEquivalence}
If there is a symmetric $2$-$(v,k,\lambda)$ design, then there is the following equivalence of rational quadratic forms of rank $v+1$
\[\langle 1,\dots,1,n\lambda\rangle\simeq\langle n,\dots,n,\lambda\rangle,\]
where $n=k-\lambda$.
\end{proposition}
\begin{proof}
Let $X_1=\diag(1,\dots,1,n\lambda)$ and $X_2=\diag(n,\dots,n,\lambda)$ be diagonal matrices of order $v+1$. We will produce two explicit invertible matrices $S$ and $P$ such that $S^{\intercal}X_1S=P^{\intercal}X_2P$, which will give us the desired equivalence of quadratic forms. Suppose there is a symmetric $2$-$(v,k,\lambda)$ design $\mathcal{D}$, and let $n=k-\lambda$. Let $N$ be the incidence matrix of $\mathcal{D}$, and define a $(v+1)\times(v+1)$ block matrix $S$ by
\[S=
\left[
\begin{array}{c|c}
N & (\lambda/k)\mathbf{1}_v\\
\hline
\mathbf{0}^{\intercal}_v & 1/k
\end{array}
\right],
\]
where $\mathbf{0}_v$ and $\mathbf{1}_v$ denote the all-zeroes and all-ones column vectors of dimension $v$, respectively. Direct computation shows that
\[
S^{\intercal}X_1S=
S^{\intercal}
\left[
\begin{array}{c|c}
I_{v} & 0\\
\hline
0 & n\lambda
\end{array}
\right]
S = 
\left[
\begin{array}{c|c}
N^{\intercal}N & (\lambda/k)N^{\intercal}\mathbf{1}_v\\
\hline
(\lambda/k)\mathbf{1}_v^{\intercal}N & \frac{\lambda^2}{k^2}\mathbf{1}_v^{\intercal}\mathbf{1}_v + n\lambda/k^2
\end{array}
\right].
\]
Since $\mathcal{D}$ is a symmetric $2$-design, Theorem \ref{thm-SymmDesignNormal} shows that $N^{\intercal}N=(k-\lambda)I_v+\lambda J_v$ and, by definition, $NJ=JN=kJ$ so that $\mathbf{1}_v^{\intercal}N=k\mathbf{1}_v^{\intercal}$ and $N^{\intercal}\mathbf{1}_v=k\mathbf{1}_v$. Also note that $\mathbf{1}_v^{\intercal}\mathbf{1}_v=v$, therefore $S^{\intercal}X_1S$ expands as
\begin{equation*}\label{brc-cong-1}
S^{\intercal}
\left[
\begin{array}{c|c}
I_{v} & 0\\
\hline
0 & n\lambda
\end{array}
\right]
S = 
\left[
\begin{array}{c|c}
(k-\lambda)I_v+\lambda J_v & \lambda\mathbf{1}_v\\
\hline
\lambda\mathbf{1}_v^{\intercal} & (\lambda^2 v+n\lambda)/k^2
\end{array}
\right]=
\left[
\begin{array}{c|c}
(k-\lambda)I_v+\lambda J_v & \lambda \mathbf{1}_v\\
\hline
\lambda\mathbf{1}_v^{\intercal} &\lambda
\end{array}
\right].
\end{equation*}
The last equality follows from the fact that the parameters of $\mathcal{D}$ satisfy $\lambda(v-1)=k(k-1)$. This implies $\lambda v + n=\lambda v + k-\lambda=k^2$, and so $(\lambda^2 v+\lambda n)/k^2=\lambda$. Let $P$ be the following $(v+1)\times (v+1)$ block matrix
\[P=\left[
\begin{array}{c|c}
I_v & \mathbf{0}_v\\
\hline
\mathbf{1}_v^{\intercal} & 1
\end{array}
\right],
\]
then we have that
\begin{equation*}\label{brc-cong-2}
P^{\intercal}X_2P=
P^{\intercal}
\left[\begin{array}{c|c}
nI_v & \mathbf{0}_v\\
\hline
\mathbf{0}_v^{\intercal} & \lambda
\end{array}
\right]
P
=
\left[\begin{array}{c|c}
nI_v+\lambda \mathbf{1}_v\mathbf{1}_v^{\intercal} & \lambda\mathbf{1}_v\\
\hline
\lambda\mathbf{1}_v^{\intercal} & \lambda
\end{array}
\right]
=
\left[\begin{array}{c|c}
(k-\lambda)I_v+\lambda J_v & \lambda\mathbf{1}_v\\
\hline
\lambda\mathbf{1}_v^{\intercal} & \lambda
\end{array}
\right].
\end{equation*}
It follows from equations (\ref{brc-cong-1}) and (\ref{brc-cong-2}) that we have the following congruence relation of matrices
\[\diag(1,\dots,1,n\lambda)\simeq 
\left[\begin{array}{c|c}
(k-\lambda)I_v+\lambda J_v & \lambda\mathbf{1}_v\\
\hline
\lambda\mathbf{1}_v^{\intercal} & \lambda
\end{array}
\right]
\simeq 
\diag(n,\dots,n,\lambda),
\]
which in turn gives the desired equivalence of quadratic forms.
\end{proof}

\begin{theorem}[Bruck-Ryser-Chowla, cf. \cite{Bruck-Ryser,Ryser-Chowla}]\index{Bruck-Ryser-Chowla Theorem}
Suppose that a symmetric $2$-$(v,k,\lambda)$ design exists, then
\begin{itemize}
\item[(i)] If $v$ is even, then $n:=k-\lambda$ is a perfect square.
\item[(ii)] If $v$ is odd, then for all odd primes $p$
\[(n,(-1)^{(v-1)/2}\lambda)_p=1.\]
\end{itemize}
\end{theorem}

\begin{proof}[Proof 1.]
The idea of the proof is as follows: The existence of a symmetric $2$-design gives (Proposition \ref{prop-BRCEquivalence}) an equivalence of rational quadratic forms 
\[\phi_1:=\langle 1,\dots,1,n\lambda\rangle \simeq \langle n,\dots,n,\lambda\rangle=:\phi_2,\]
where $n=k-\lambda$. By the Hasse-Minkowski Theorem \ref{thm-HasseMinkowski}, the discriminants and Hasse-Minkowski invariants of $\phi_1$ and $\phi_2$ should coincide, i.e. $\delta(\phi_1)=\delta(\phi_2)$ and $\varepsilon_p(\phi_1)=\varepsilon_p(\phi_2)$ for all odd primes $p$. A case analysis of the parity of $v$ will give us the conditions in the theorem statement. Let us begin the proof.\\

First we compute the discriminants (Definition \ref{def-Discriminant}) of $\phi_1$ and $\phi_2$. These are
\begin{align*}
&\delta(\phi_1)=1^{v}\cdot n\lambda=n\lambda,\text{ and }\\
&\delta(\phi_2)=n^{v}\lambda.
\end{align*}
The discriminants $\delta(\phi_i)$ are interpreted as elements of the square class group $\Gamma(\Q)$, and so they are equal if and only if their product is a rational square, we have
\[\delta(\phi_1)\delta(\phi_2)=n^{v+1}\lambda^2\equiv n^{v+1}\pmod{(\Q^{\times})^2}.\]
From here it follows that if $v$ is even, then $n=k-\lambda
$ must be a perfect square.\\

The Hasse-Minkowski invariants (Definition \ref{def-HMInvariants}) of $\phi_1$ and $\phi_2$ are
\begin{align*}
&\varepsilon_p(\phi_1)=(1,1)_p^{v\choose 2}(1,n\lambda)_p^v=(1,n\lambda)_p=1.\\
&\varepsilon_p(\phi_2)=(n,n)^{v\choose 2}(n,\lambda)^v.
\end{align*}
If $v$ is even, we saw that $n$ must be a perfect square and thus all symbols vanish and we find no further conditions. If $v$ is odd, then we must have $\varepsilon_p(\phi_2)=\varepsilon_p(\phi_1)=1$, and the invariant $\varepsilon_p(\phi_2)$ reduces to
\[
\varepsilon_p(\phi_2)=(n,n)_p^{v\choose 2}(n,\lambda)_p=(n,n)_p^{{v\choose 2}-1}(n,n)_p(n,\lambda)_p=(n,n)_p^{{v\choose 2}-1}(n,n\lambda)_p.
\]
For $v$ odd, the binomial coefficient ${v\choose 2}$ is even if and only if $v\equiv 1\pmod{4}$, hence we have
\[\varepsilon_p(\phi_2)=(n,n)_p^{{v\choose 2}-1}(n,n\lambda)_p
\begin{cases}
(n,n\lambda)_p &\text{ if } v\equiv 3\pmod{4}\\
(n,n)_p(n,n\lambda)_p & \text{ if } v\equiv 1\pmod{4}
\end{cases}.
\] 
Using the properties of the Hilbert symbols (Lemma \ref{lemma-Hilbert-Sym-Props}) we have $\varepsilon_p(\phi_2)=(n,n\lambda)_p=(n,-\lambda)_p$ when $v\equiv 3\pmod{4}$, and $\varepsilon_p(\phi_2)=(n,n)_p(n,n\lambda)_p=(n,\lambda)_p$ when $v\equiv 1\pmod{4}$. So we find that in any case
\[\varepsilon_p(\phi_2)=(n,(-1)^{(v-1)/2}\lambda)_p.\]
And the condition 
\[\varepsilon_p(\phi_2)=(n,(-1)^{(v-1)/2}\lambda)_p=1,\]
for all odd primes $p$, follows.\qedhere
\end{proof}
\begin{remark}\normalfont
By the strong Hasse local-global principle, Theorem \ref{thm-StrongLocalGlobal}, we have that the condition $(n,(-1)^{(v-1)/2}\lambda)_p=1$ for all odd primes $p$, together with $n=k-\lambda>0$ (non-triviality of the design), imply that $z^2=nx^2+(-1)^{(v-1)/2}\lambda y^2,$
has a non-trivial rational solution. Multiplying by a common denominator of $x, y,$ and $z$ we find a non-trivial integral solution to the Diophantine equation
\[z^2=nx^2+(-1)^{(v-1)/2}\lambda y^2.\]
This is how the Bruck-Ryser-Chowla Theorem is typically presented in the design theory literature. However, this formulation has the downside that it does not indicate how one can systematically find obstructions to a quadratic Diophantine equation \index{Diophantine equation}. On the other hand we have precise tables and formulas to compute the local Hilbert symbols $(n,(-1)^{(v-1)/2}\lambda)_p$ (Proposition \ref{prop-HilbertSymbolFormula}), and then the obstructions become explicitly computable.
\end{remark}

The proof of the Bruck-Ryser-Chowla presented above uses the trick of Proposition \ref{prop-BRCEquivalence}. Namely, one can use the existence of a certain incidence structure to find a convenient rational congruence to which the Hasse-Minkowski Theorem can be applied. This same approach is taken in the proof of the Bose-Connor Theorem \cite{Bose-Connor}. However, we will develop a more general and systematic approach to these types of theorems. On the one hand we will find combinatorial structures with incidence matrices $X$ satisfying the same Gram equation $XX^{\intercal}=nI_v+\lambda J_v$, where $X$ does not necessarily have a constant row-sum (all incidence matrices $N$ of $2$-designs satisfy $NJ=kJ$). On the other hand, we show that one can work directly with the target Gram equation $nI_v+\lambda J_v$ and that it is not necessary to find a congruence relation using the putative incidence matrices $X$. Furthermore, we will see that the computation of the local invariants for the matrix $\alpha I_v+\beta J_v$ (without assumptions on $\alpha$ and $\beta$) is not much harder than the one using Proposition \ref{prop-BCRationalCongruence}. \\

First, we present a summary of straightforward results on the Hasse-Minkowski invariants. One can find analogue statements for the Hasse-Pall invariants (Definition \ref{def-PallInvariants}) in \cite{Bose-Connor}.\\

\textbf{Notation:} To ease readability, in what follows we will abbreviate the discriminant $\delta(A)$ of a symmetric matrix $A$ by $\delta_A$.\index{discriminant}
\begin{lemma}[cf. \cite{Bose-Connor}]\label{lemma-DirectSum-Symbol}\normalfont Let $A$ and $B$ be symmetric matrices, then
\[\varepsilon_p(A\oplus B)=\varepsilon(A)_p\varepsilon(B)_p(\delta_A,\delta_B)_p.\]
\end{lemma}
\begin{proof}
Assume $A\sim \langle a_1,\dots,a_n\rangle$ and $B\sim \langle b_1,\dots,b_m\rangle$, then $A\oplus B \sim \langle a_1,\dots,a_n,b_1,\dots,b_m\rangle$. By bilinearity of the local Hilbert symbols, we find the following expansion:
\begin{align*}
\varepsilon_{p}(A\oplus B)&=\prod_{1\leq i<j\leq n} (a_i,a_j)_p\prod_{i,j}(a_i,b_j)_p\prod_{1\leq i<j\leq m}(b_i,b_j)_p
\end{align*}
where the product $\prod_{ij}(a_i,b_j)$ ranges through all possible values of $i$ and $j$. This implies
\[\varepsilon_p(A\oplus B)= \varepsilon_p(A)\varepsilon_p(B)(\delta_A,\delta_B)_p.\qedhere\]
\end{proof}
Since $\varepsilon_p(\alpha)=1$ for a $1\times 1$ matrix $(\alpha)$ it follows that
\begin{corollary}[cf. \cite{Bose-Connor}]\normalfont \label{cor-addOne-Symbol} If $A=(\alpha)$ is a $1\times 1$ matrix, then
\[\varepsilon_p(\alpha\oplus B)=(\alpha,\delta_B)_p\varepsilon_p(B).\]
\end{corollary}

\begin{corollary}[cf. \cite{Bose-Connor}]\normalfont \label{cor-rFold-Symbol} Let $\Delta_r=\bigoplus_{i=1}^r A$ be the $r$-fold direct sum of $A$, then 
\[\varepsilon_p(\Delta_r)=\varepsilon_p(A)^r (\delta_A,-1)_p^{r\choose 2}.\]
\end{corollary}
\begin{proof}
This is a straightforward induction proof. When $r=2$ we have by 
Lemma \ref{lemma-DirectSum-Symbol} that
\[\varepsilon_p(\Delta_2)=\varepsilon_p(A\oplus A)=\varepsilon_p(A)^2 (\delta_A,\delta_A)_p=\varepsilon_p(A)^2(\delta_A,-1)_p.\]
We used above that $(a,a)_p=(a,-1)_p$, which follows from the definition of the Hilbert symbol. Assume that for $r\geq 2$, $\varepsilon_p(\Delta_r)=\varepsilon_p(A)^r(\delta_A,-1)_p^{r\choose 2}$, then
\begin{align*}
\varepsilon_p(\Delta_{r+1})&=\varepsilon_p(\Delta_r\oplus A)\\
&=\varepsilon_p(\Delta_r)\varepsilon_p(A) (\delta_{\Delta_r},\delta_A)_p\\
&=\varepsilon_p(A)^{r+1}(\delta_A,-1)_p^{r\choose 2}(\delta_A^r,\delta_A)_p\\
&=\varepsilon(A)_p^{r+1}(\delta_A,-1)_p^{r\choose 2}(\delta_A,\delta_A)_p^r\\
&=\varepsilon_p(A)^{r+1}(\delta_A,-1)_p^{r+1\choose 2}.\qedhere
\end{align*}
\end{proof}

\begin{lemma}[cf. \cite{Bose-Connor}]\normalfont \label{lemma-ScaledInvariant} Let $\gamma$ be a rational number and $A$ a symmetric matrix of order $d$. The Hasse-Minkowski invariant of $\gamma A$ is
\[\varepsilon_p(\gamma A)=(\gamma,-1)_p^{d\choose 2}(\gamma,\delta_A)_p^{d-1}\varepsilon_p(A).\]
\end{lemma}
\begin{proof}
Assume $A\sim\langle a_1,\dots,a_d\rangle$, then $\gamma A\sim\langle\gamma a_1,\dots, \gamma a_d\rangle$. Using bilinearity and symmetry we find
\[\varepsilon_p(\gamma A)=\prod_{i<j}(\gamma a_i,\gamma a_j)_p=\prod_{i<j}(\gamma,\gamma)_p(\gamma,a_ia_j)_p(a_i,a_j)_p=(\gamma,\gamma)_p^{d\choose 2}(\gamma,\prod_{i<j}a_ia_j)_p\varepsilon_p(A).\]
In the product $\prod_{i<j}a_ia_j$ each $a_i$ appears exactly $d-1$ times. A way to see this is by labelling the complete graph on $d$ vertices using the elements $a_i$. Therefore $(\gamma,\prod_{i<j}a_ia_j)_p=(\gamma,\delta_A^{d-1})_p=(\gamma,\delta_A)_p^{d-1}$. Plugging this back into the equation for $\varepsilon_p(\gamma A)$ above yields the result.\qedhere
\end{proof}
 
Below we give a simple computation of the local invariants of $I_d+J_d$. At a first glance it may seem like this computation does not have far reaching consequences, but rather counter-intuitively it gives us a basic building block to compute the invariants of the Bruck-Ryser-Chowla Theorem, and the Bose-Connor Theorem.

\begin{lemma}[cf. \cite{InvariantsPaper}]\label{lemma-PolarIJ}\normalfont
The $d \times d$ matrices $I_d+ J_d$ and $\diag( 2,6,\dots,d(d+1))$ are rationally congruent. 
\end{lemma}

\begin{proof} 
Since every vector is an eigenvector of $nI$, it suffices to choose an orthogonal eigenbasis for $J$ in which 
each basis vector has rational entries. This may be accomplished as follows: 
\[ f_{i} = (1,1,\ldots,1, -i, 0, \ldots, 0), \,\,\,\text{ for } 1\leq i < d,\text{ and } f_d =(1,1,\ldots,1).\] 
where $f_{i}$ contains $-i$ in co-ordinate $i+1$, with $1$'s to the left and $0$'s to the right. 
By linearity, $(I_d + J_d)f_{d} = (d+1)f_{d}$ and $(I +J)f_{i} = f_{i}$. Let $F$ be the matrix with $f_{i}$ in the $i^{\textrm{th}}$ column. 
Since $f_i^{\intercal}f_j=i(i+1)\delta_{ij}$ for $1\leq i,j<d$ and $f_i^{\intercal}f_d=d\delta_{id}$, it follows that  $D = F^{\top} (I_d + J_d) F$ is diagonal, with $D_{ii} = i(i+1)$ for $1\leq i <d$ and $D_{dd} = d(d+1)$. 
\end{proof} 

\begin{proposition}[cf. \cite{InvariantsPaper}]\label{lemma-SymbolIJ}\normalfont At any prime $p$, the Hasse-Minkowski invariant of $I_d+J_d$ is 
\[\varepsilon_p(I_d+J_d)=(d,d+1)_p\]
\end{proposition}

\begin{proof}
By Lemma \ref{lemma-PolarIJ} we have that
\[\varepsilon_p(I_d+J_d)=\prod_{i<j}(i(i+1),j(j+1))_p=\prod_{i=1}^{d-1} \prod_{j=i+1}^{d}(i(i+1),j(j+1))_p.\]
The product $\prod_{j=i+1}^{d}(i(i+1),j(j+1))_p$ is telescoping, and we find
\[\prod_{j=i+1}^d(i(i+1),j(j+1))_p=\prod_{j=i+1}^d(i(i+1),j)_p(i(i+1),j+1)_p=(i(i+1),i+1)_p(i(i+1),d+1)_p.\]
Notice that $(i+1,i(i+1))_p=(i+1,-i(i+1)^2)_p=(i+1,-i)_p=1$, because $(i+1)x^2-iy^2=z^2$ has the non-trivial solution $x=y=z=1$. Therefore
\begin{align*}
\prod_{i<j}(i(i+1),j(j+1))_p &= \prod_{i=1}^{d-1}(i(i+1),d+1)_p\\
&=\prod_{i=1}^{d-1}(i,d+1)_p(i+1,d+1)_p\\
&=(1,d+1)_p(d,d+1)_p=(d,d+1)_p.\qedhere
\end{align*}
\end{proof}
\begin{theorem} \label{thm-PolarIJ-General} At any rational place $p$, the Hasse-Minkowski invariant of $\alpha I_d+\beta J_d$ is
\[\varepsilon_p(\alpha I_d+\beta J_d)=((\alpha+\beta d)d,\alpha^{d-1} d)_p (\alpha,-1)_p^{d-1\choose 2}(\alpha,d)_p^{d}(d-1,d)_p.\]
\end{theorem}
\begin{proof}
Let
$P=\left[\begin{array}{c|c}
1 & -\mathbf{1}_{d-1}^{\intercal}\\
\hline
\mathbf{1}_{d-1} & I_{d-1}
\end{array}\right]$, the result is a consequence of the following congruence:
\[P^{\intercal}(\alpha I_d+\beta J_d)P=
\left[
\begin{array}{c|c}
(\alpha+\beta d)d & \mathbf{0}\\
\hline
\mathbf{0} & \alpha(I_{d-1}+J_{d-1})
\end{array}
\right],
\]
which follows from the fact that the columns of $P$ are eigenvectors for $J_d$. Now we can apply the lemmas we obtained before to find the local invariants of this block matrix. From the fact that $\det(I_{d-1}+J_{d-1})=d$, and Corollary \ref{cor-addOne-Symbol} we find
\[\varepsilon_p(\alpha I_d+J_d)=((\alpha+\beta d)d,\alpha^{d-1} d)_p\, \varepsilon_p(\alpha(I_{d-1}+J_{d-1})).\]
The invariant in the right-hand-side can be computed with the formula for invariants of scaled matrices of Lemma \ref{lemma-ScaledInvariant}, this gives
\[\varepsilon_p(\alpha(I_{d-1}+J_{d-1}))=(\alpha,-1)_p^{d-1\choose 2}(\alpha,d)_p^{d-2}(d-1,d)_p.\]
Putting this together gives
\[\varepsilon_p(\alpha I_{d-1}+\beta J_{d-1})=((\alpha+\beta d)d,\alpha^{d-1} d)_p (\alpha,-1)_p^{d-1\choose 2}(\alpha,d)_p^{d-2}(d-1,d)_p.\qedhere\]
\end{proof}
We note that one can obtain this result directly through the congruence
\[F^{\intercal}(\alpha I_d+\beta J_d)F=\langle 2\alpha,6\alpha,\dots,d(d+1)\alpha,(\alpha+\beta d)d\rangle,\]
where $F$ is the matrix in the proof of Lemma \ref{lemma-PolarIJ}. But then the computation of the invariants is, in our opinion, harder to carry than with the approach of Theorem \ref{thm-PolarIJ-General}.\\

From Theorem \ref{thm-PolarIJ-General}, we obtain a rational converse of the Bruck-Ryser-Chowla Theorem, see also section 10.4 of Hall's combinatorial theory \cite{Hall-CombinatorialTheory}.  The difference between our proof and the one in \cite{Hall-CombinatorialTheory} is that Hall computes the Pall invariants of $(k-\lambda)I_v+J_v$ directly, see  Definition \ref{def-PallInvariants}. Our approach using the congruence of Theorem \ref{thm-PolarIJ-General} involves computations that are easier to carry.

\begin{theorem}[cf. Section 10.4 \cite{Hall-CombinatorialTheory}]\label{thm-BRC-General}Let $v,k,\lambda$ be positive integers such that $n=k-\lambda>0$, and  $k(k-1)=\lambda(v-1)$. The matrix $nI_v+\lambda J_v$, is a rational Gram matrix if and only if
\begin{itemize}
\item $n=k-\lambda$ is a square when $v$ is even,
\item $(n,(-1)^{(v-1)/2}\lambda)_p=1$, for all odd $p$ when $v$ is odd.
\end{itemize}
\end{theorem}
\begin{proof}
The discriminant of $nI_v+\lambda J_v$ is equal to $(k-\lambda+\lambda v)n^{v-1}=k^2n^{v-1}$. So if $v$ is even, it is necessary that $n=k-\lambda$ is a square. Apply Theorem \ref{thm-PolarIJ-General} with $\alpha=(k-\lambda)$, $\beta=\lambda$ and $d=v$ to find that
\begin{align*}
\varepsilon_p(nI_v+\lambda J_v)&=(k^2 v,n^{v-1} v)_p(n,-1)_p^{v-1\choose 2}(n,v)_p^v (v-1,v)_p.
\end{align*} 
If $v$ is even, then $n$ is a square and the expression above simplifies to
\[\varepsilon_p(nI_v+\lambda J_v)=(v,v)_p(v-1,v)_p=(v(v-1),v)_p=(-(v-1),v)_p=1.\]
On the other hand, if $v$ is odd we find
\[\varepsilon_p(nI_v+\lambda J_v)=(n,-1)_p^{v-1\choose 2}(n,v)_p.
\]
From the equation $\lambda v=k^2-(k-\lambda)=k^2-n$ we find that $(n,v)_p=(n,\lambda)_p$. Indeed,
\[(n,\lambda)_p(n,v)_p=(n,\lambda v)_p=(n,k^2-n)_p=1,\]
since the equation $nx^2+(k^2-n)y^2=z^2$ has the non-trivial solution $x=y=1$, $z=k$. Therefore $nI_v+\lambda J_v$ is a rational Gram matrix if and only if
\[\varepsilon_p(nI_v+\lambda J_v)=(n,(-1)^{(v-1)/2}\lambda)_p=1\]
for all $p$.\qedhere
\end{proof}
 From this result the Bruck-Ryser-Chowla Theorem follows immediately:

\begin{proof}[Proof 2 (of BRC Theorem).]\index{Bruck-Ryser-Chowla Theorem}
The incidence matrix $N$ of a symmetric $2$-$(v,k,\lambda)$ design satisfies $NN^{\intercal}=(k-\lambda)I_v+\lambda J_v$, and the parameters $v$, $k$ and $\lambda$ satisfy the equation $\lambda(v-1)=k(k-1)$. Hence Theorem \ref{thm-BRC-General} can be applied, and in the odd case we find that for all primes $p$, $(k-\lambda,(-1)^{(v-1)/2}\lambda)_p=1$.
\end{proof}

There is a very interesting coding-theoretic proof of the Bruck-Ryser-Chowla Theorem in the odd case due to Eric Lander \cite{Lander-SymmetricDesigns}. In his proof, Lander shows that the existence of a symmetric $2$-$(v,k,\lambda)$ implies the constructibility of certain \textit{self-dual} $\F_p$\textit{-codes}. This in turn yields two conditions that are implied by the condition $(n,(-1)^{(v-1)/2}\lambda)_p=1$ for all odd $p$. If $k$ and $\lambda$ are coprime, then these two conditions are equivalent to the Hilbert symbol condition. When this is not the case, we do not have a complete coding-theoretic interpretation of the BRC theorem. 

\begin{research-problem}\normalfont Complete Lander's argument in Chapter 2 of \cite{Lander-SymmetricDesigns} to include the case where $k$ and $\lambda$ are not coprime. Give an interpretation of the equivalence of forms over $\Q_p$ in coding-theoretic terms.
\end{research-problem}

Recall that a \index{projective plane} projective plane of order $n$ is a $2$-$(n^2+n+1,n+1,1)$ design. Then $v=n^2+n+1$ is always odd, and the BRC Theorem implies that 
\[(n,(-1)^{n(n+1)/2})_p=(n,(-1)^{n+1\choose 2})_p=1,\]
for all $p$ odd. The binomial coefficient ${n+1\choose 2}$ is odd if and only if $n\equiv 1,2\pmod{4}$. And in this case the condition $(n,-1)_p=1$ for all odd primes $p$ is equivalent to the existence of a non-trivial integral solution to
\[nx^2-y^2=z^2.\]
Multiplying by a common denominator, say $a$, of $x$, $y$ and $z$ we find that $na$ is a sum of two integer squares. We can show that this implies that $n$ is a sum of two squares by using the following theorem:

\begin{theorem}[Sum of two squares, cf. Chapter 17, Section 6, Corollary 1, \cite{Ireland-Rosen}]\label{thm-2sq}
Let $n$ be a natural number, then $n$ is a sum of two integer squares if and only if every odd prime factor $p$ of $n$ with $p\equiv 3\pmod{4}$ divides $n$ with even multiplicity.
\end{theorem}
\begin{proof}
Suppose $n$ is a sum of two integer squares, say $n=x^2+y^2$. Then
\[
\begin{bmatrix}
x & -y\\
y & x
\end{bmatrix}
\begin{bmatrix}
x & y\\
-y & x
\end{bmatrix}
=
\begin{bmatrix}
x^2+y^2 & 0\\
0 & x^2 + y^2
\end{bmatrix}
=
\begin{bmatrix}
n & 0\\
0 & n
\end{bmatrix}.
\]
Therefore by the Hasse-Minkowski theorem, Theorem \ref{thm-HasseMinkowski}, we have that $\varepsilon_p(\langle n,n\rangle )= (n,n)_p=(n,-1)_p=1$ for all primes $p$. If $p$ does not divide $n$, or if $p$ divides $n$ with even multiplicity then trivially $(n,-1)_p=1$. So assume that $p$ divides $n$ with odd multiplicity, in this case by Proposition \ref{prop-HilbertSymbolFormula}
\[(n,-1)_p=\legendre{-1}{p}.\]
From which it follows that all odd prime factors of $n$ appearing with odd multiplicity must be $p\equiv 1\pmod{4}$. Conversely assume that $n$ has a prime factorisation
\[n=p_1^{e_1}\dots p_r^{e_r},\]
where $p_i\equiv 1\pmod{4}$ whenever both $p_i$ and $e_i$ are odd. By Diophantus' identity
\[(a^2+b^2)(c^2+d^2)=(ac-bd)^2+(ad+bc)^2,\]
it suffices to show that each prime $p\equiv 1\pmod{4}$ can be written as the sum of two integer squares. But this is a consequence of Theorem \ref{thm-2x2HM} and the fact that $(p,p)_p=(p,-1)_p=\legendre{-1}{p}$. Indeed, since $(p,p)=1$ for $p\equiv 1\pmod{4}$, we have that $\langle p,p\rangle=\langle 1,1\rangle$, and in particular $p$ is the sum of two integer squares. Finally $2=1^2+1^2$, and since each $p_i^{e_i}$ is either an integer square of a sum of two integer squares, Diophantus' identity implies that $n$ is itself a sum of integer squares.\qedhere
\end{proof}

\begin{remark}\normalfont The proof given above is non-constructive. We remark that using the Euclidean algorithm, one can efficiently decompose a prime $p\equiv 1\pmod{4}$ as a sum of two squares.
\end{remark}

We obtain the following corollary:

\begin{corollary}[Bruck and Ryser\cite{Bruck-Ryser}] \label{cor-ProjPlanNonExistence} If a projective plane of order $n$ exists and $n\equiv 1,2\pmod{4}$, then $n$ must be the sum of two integer squares.
\end{corollary}
\begin{proof}
As outlined above, for a projective plane of order $n$ the BRC Theorem implies that if $n\equiv 1,2\pmod{4}$, then
\[nx^2=y^2+z^2,\]
must have a non-trivial rational solution. Dividing by $x^2$ we find that $n=(y/x)^2+(z/x)^2$, and so $n$ is the sum of two rational squares. Without loss of generality we can write
\[n=\left(\frac{a}{c}\right)^2+\left(\frac{b}{c}\right)^2,\]
where $a, b$ and $c$ are integers. Therefore
\[n\cdot c^2=a^2+b^2.\]
Suppose that there is a prime factor $p$ of $n$ with $p\equiv 3\pmod{4}$. Since $c^2$ is a square, by the fundamental theorem of arithmetic the parity of the multiplicity of $p$ as a factor of $n$ coincides with the parity of $p$ as a factor of $a^2+b^2$. By the theorem on sums of two squares, Theorem \ref{thm-2sq}, $p$ must divide $a^2+b^2$ with even multiplicity, and so $p$ must also divide $n$ with even multiplicity. Applying Theorem \ref{thm-2sq} again, we find that $n$ is the sum of two integer squares.\qedhere
\end{proof}

\begin{example}\normalfont. Suppose there exists a projective plane of order $6$. Then we have that $n\equiv 2\pmod 4$, and in this case Corollary \ref{cor-ProjPlanNonExistence} implies that $n=6$ must be a sum of two squares. But this is a contradiction, since $6$ is not the sum of two squares. Therefore a projective plane of order $6$ does not exist. 
\end{example}
Using the same argument we find that
\[\{6, 14, 21, 22, 30, 33, 38, 42, 46, 54, 57, 62, 66, 69, 70, 77, 78, 86, 93, 94\},\]
is the set of orders $n\leq 100$ for which a projective plane of order $n$ cannot exist by the BRC theorem.

\section{The Bose-Connor Theorem}

The Bose-Connor Theorem gives conditions for the non-existence of \textit{group-divisible} designs. The incidence matrix of group-divisible designs has a Gram matrix of the type
\[D_{\alpha,\beta,\gamma}(a,b)=((\alpha-\beta)I_a+(\beta-\gamma)J_a)\otimes I_b+\gamma J_{ab},\]
where $a$ and $b$ are positive integers, and $\alpha,\beta,\gamma\in\Q$. So $D_{\alpha,\beta,\gamma}(a,b)$ is a block matrix, where the integer $a$ represents the size of the diagonal blocks, which are equal to $(\alpha-\beta)I_a+\beta J_a$, and $b$ represents the number of diagonal blocks. All off-diagonal blocks are of the type $\gamma J_a$. For example, the matrix $D_{\alpha,\beta,\gamma}(2,3)$ is as follows
\[D_{\alpha,\beta,\gamma}(2,3)=\left[
\begin{array}{cc|cc|cc}
\alpha & \beta & \gamma & \gamma & \gamma & \gamma\\
\beta & \alpha & \gamma & \gamma & \gamma & \gamma\\
\hline
\gamma & \gamma & \alpha & \beta & \gamma & \gamma\\
\gamma & \gamma & \beta & \alpha & \gamma & \gamma\\
\hline
\gamma & \gamma & \gamma & \gamma & \alpha & \beta\\
\gamma & \gamma & \gamma & \gamma & \beta & \alpha
\end{array}\right].
\]
Later on in this section we will give the precise definition of group-divisible designs. For now, we will work with $D_{\alpha,\beta,\gamma}(a,b)$ without assuming the existence of such designs.\\

To compute the invariants of $D_{\alpha,\beta,\gamma}(a,b)$ we will make use of some concepts from the theory of association schemes that we now introduce:

\begin{definition}\normalfont The \textit{Bose-Mesner algebra} \index{association scheme ! Bose-Mesner algebra of an} of a $d$-class \textit{association scheme} \index{association scheme}$\mathcal{X}$ is the complex matrix algebra spanned by a collection of $(0,1)$-matrices $\{A_0,A_1,\dots,A_d\}$ of order $v$ satisfying the following properties
\begin{itemize}
\item [(i)] $A_0=I$,
\item [(ii)] $\sum_{i=0}^{d}A_i=J_v$,
\item [(iii)] $A_i^{\intercal}=A_{i'}$ for some $i'\in \{0,1,\dots,d\}$,
\item [(iv)] There exist natural numbers $p_{ij}^k$ such that
\[A_iA_j=\sum_k p_{ij}^k A_k,\]
\item[(v)] $A_iA_j=A_jA_i$.
\end{itemize}
\end{definition}
The matrices $A_i$ are called the \textit{adjacency matrices}\index{matrix!adjacency} of the association scheme $\mathcal{X}$. If $A_i^{\intercal}=A_i$ for all $i$, we say that the association scheme is  \textit{symmetric}.

\begin{example}\normalfont The Bose-Mesner algebra of the \textit{trivial association scheme} \index{association scheme! trivial} is $\mathcal{A}=\Span_{\C}\{I,J-I\}$. This is a symmetric association scheme, and $A_1=J-I$ is the adjacency matrix of the complete graph $K_v$ on $v$ vertices.
\end{example}

\begin{definition}\normalfont\label{def-SRG}
A \textit{strongly regular graph},\index{graph!strongly regular} $\SRG(v,k,\lambda,\mu)$, is a $k$-regular graph, different from $K_v$ or $\overline{K_v}$, such that
\begin{itemize}
\item[(i)] for every pair $x\sim y$ of adjacent vertices there are exactly $\lambda$ vertices $z$ such that $x\sim z\sim y$, and
\item[(ii)] for every pair $x\neq y$ of non-adjacent vertices, there are exactly $\mu$ vertices $z$ such that $x\sim z\sim y$.
\end{itemize}
\end{definition}
A good reference text on the theory of strongly regular graphs is the book by Brouwer and Van Maldeghem, \cite{Brouwer-VanMaldeghem}. It is then easy to check that if $A$ is the adjacency matrix of an $\SRG(v,k,\lambda,\mu)$ then
\[A^2=kI_v+\lambda A+\mu(J_v-I_v-A).\]
Therefore, $\Span_{\C}\{I,A,J-I-A\}$ is the Bose-Mesner algebra of a $2$-class symmetric association scheme. \index{association scheme !symmetric}

\begin{example}\normalfont 
\begin{itemize}
\item[(i)]  The \textit{Petersen graph} is an $\SRG(10,3,0,1)$.
\item[(ii)] For any Latin $L$ square of order $n$, we can obtain a strongly regular graph on $n^2$ vertices indexed by the entries of $L$.  This is done by joining vertices $(i,j)$ with $(i',j')$ whenever $i=i'$, $j=j'$ or $L_{ij}=L_{i'j'}$. The resulting graph is known as a \textit{Latin square graph} and it is an $\SRG(n^2,3(n-1),n,6)$. More generally, Latin square graphs may be defined from \textit{transversal designs}, see Section 8.4.2 of \cite{Brouwer-VanMaldeghem}.
\item[(iii)]Let $q\equiv 1\pmod{4}$ be an odd prime power, we define a graph on $q$ vertices indexed by elements of $\F_q$ by letting $x\sim y$ if and only if $x-y\in (\F_q^{\times})^2$. The resulting graph is known as a \textit{Paley graph} and it is an $\SRG(q,(q-1)/2,(q-1)/4-1,(q-1)/4)$.\index{graph!Paley}
\item[(iv)] The disjoint union $bK_a$ of $b$ copies of the complete graph $K_a$ is a strongly regular graph with $\lambda=a-2$ and $\mu=0$.
\end{itemize}
\end{example}
If $\mu=k$ or $\lambda=k-1$, then the strongly regular graph is called \textit{imprimitive}.\index{graph!strongly regular!imprimitive} Therefore the imprimitive strongly regular graphs are the disjoint union of complete graphs $bK_a$, or its complement the complete multipartite graph $K_{a,\dots,a}$ (where there are $b$ parts of size $a$).\\

  Notice that the matrix $nI+\lambda J$ of the BRC Theorem belongs to the Bose-Mesner algebra of the complete graph, and the matrix $D_{\alpha,\beta,\gamma}(a,b)$ belongs to the Bose-Mesner algebra of $bK_a$. Therefore, the Bruck-Ryser-Chowla Theorem and the Bose-Connor Theorem can be seen as particular instances of the following problem
  
\begin{research-problem}\normalfont
Let $\mathcal{A}=\Span\{A_0,A_1,\dots,A_d\}$ be the Bose-Mesner algebra of a symmetric $d$-class association scheme. For arbitrary $\alpha_i\in \Q$ determine when 
\[M=\alpha_0I+\alpha_1A_1+\dots+\alpha_d A,\]
is of the form $XX^{\intercal}$ for some $X\in\GL_v(\Q)$.
\end{research-problem}

With this point of view, the BRC Theorem and Bose-Connor Theorem follow from the classification of rational quadratic forms on the Bose-Mesner algebra of the trivial association scheme, and the imprimitive strongly regular graph respectively.
\begin{research-problem}\normalfont Classify rational quadratic forms in the Bose-Mesner algebra of an arbitrary \textit{primitive} strongly regular graph.
\end{research-problem}

For an introduction to the theory of association schemes the reader can consult the book by Bannai and Ito \cite{Bannai-Ito}. The only fact about general Bose-Mesner algebras that we will use here is the following: The matrices $A_i$ are normal (by properties (iii) and (iv)), and they commute so they are simultaneously diagonalisable by a \textit{complex} unitary matrix (or real orthogonal, if the matrices are symmetric). Therefore, there is a basis of $V=\C^v$ consisting of common eigenvectors for the matrices $A_i$. In particular, there is a decomposition $V=V_0\oplus \dots \oplus V_d$ into \textit{maximal} common eigenspaces. By maximal eigenspaces we mean that if $i\neq j$, then there is a matrix $A_k$ whose eigenvalue on $V_i$ is different from its eigenvalue on $V_j$.\\

Since the matrices $A_k$ are normal, any two eigenvectors associated to distinct eigenvalues are orthogonal. In particular, the maximal common eigenspaces are mutually orthogonal. Therefore, if we can find a \textit{rational} basis for the spaces $V_i$ (not necessarily consisting of unitary vectors), then there is a rational matrix $P$ such that
\[P^{\intercal}A_kP=X_0^{(k)}\oplus\dots\oplus X_d^{(k)},\]
for all $k$. In this way, we can reduce the computation of the Hasse-Minkowski invariants of $\sum_{i}\alpha_iA_i$ to the computation of the Hasse-Minkowski invariants of $\sum_{i} \alpha_i X_{j}^{(i)}$, for all $i$.\\

Recall that in the proof of Theorem \ref{thm-PolarIJ-General}, we obtained a rational congruence 

\[P^{\intercal}(\alpha I_{d}+\beta J_{d})P=
\left[
\begin{array}{c|c}
(\alpha+\beta d)d & \mathbf{0}\\
\hline
\mathbf{0} & \alpha(I_{d-1}+J_{d-1})
\end{array}
\right],
\]
by using the matrix $P=\left[
\begin{array}{c|c}
1 & -\mathbf{1}_{d-1}^{\intercal}\\
\hline
\mathbf{1}_{d-1} & I_{d-1}
\end{array}
\right]
$. Notice that the columns of the matrix $P$ form a basis for the maximal common eigenspaces of the Bose-Mesner algebra generated by $\{I,J-I\}$.\\

 We will use this same approach now to find a nice congruence relation for the matrix $D_{\alpha,\beta,\gamma}(a,b)$. Let $C_a$ be the $(a+1)\times a$ matrix given by
\[C_a=\left[\begin{array}{c}
-\mathbf{1}_a^{\intercal}\\
\hline
I_a
\end{array}\right]=
\left[
\begin{array}{ccc}
-1 & \dots &-1\\
\hline
1 & & \\
 &\ddots &\\
 &&1
\end{array}
\right],
\]
then we have
\begin{lemma}\normalfont \label{lemma-BCRationalEigenbasis}The columns of the matrix 
\[F=\left[
\begin{array}{c|c|c}
\mathbf{1}_{ab} & \mathbf{1}_a\otimes C_{b-1} & C_{a-1}\otimes I_b 
\end{array}
\right],
\]
form a rational basis of common eigenvectors for the matrices $A_0=I_{ab}$, $A_1=(J_a-I_a)\otimes I_b$, and $A_2=J_{ab}-A_0-A_1$. 
\end{lemma}
\begin{proof}
We show that the blocks of $F$ consist of common eigenvectors for the matrices $A_1=(J_a-I_a)\otimes I_b$ and $J_{ab}$. Clearly, $J_{ab}\mathbf{1}_{ab}=ab\mathbf{1}_{ab}$, and by the mixed-product property of the Kronecker product
\begin{align*}
&J_{ab}(\mathbf{1}_a\otimes C_{b-1})=(J_a\otimes J_b)(\mathbf{1}_a\otimes C_{b-1})=(J_a\mathbf{1}_a\otimes J_bC_{b-1})=0,\text{ and }\\
&J_{ab}(C_{a-1}\otimes I_b)=(J_a\otimes J_b)(C_{a-1}\otimes I_b)=(J_aC_{a-1}\otimes J_b)=0.
\end{align*}
For $A_1=(J_a-I_a)\otimes I_b$, we have again by the mixed product property that 

\begin{align*}
&A_1\mathbf{1}_{ab}=((J_a-I_a)\otimes I_b)(\mathbf{1}_a\otimes \mathbf{1}_b)=((J_a-I_a)\mathbf{1}_a)\otimes \mathbf{1}_b=(a-1)\mathbf{1}_{ab},\\
&A_1(\mathbf{1}_a\otimes C_{b-1})=((J_a-I_a)\mathbf{1}_a)\otimes C_{b-1}=(a-1)(\mathbf{1}_a\otimes C_{b-1}), \text{ and }\\
&A_1(C_{a-1}\otimes I_b)=((J_a-I_a)C_{a-1})\otimes I_b=-(C_{a-1}\otimes I_b).
\end{align*}
Now, by linearity we have that for $A_2=J_{ab}-A_1-I$,
\begin{align*}
&A_2\mathbf{1}_{ab}=J_{ab}\mathbf{1}_{ab}-A_1\mathbf{1}_{ab}-\mathbf{1}_{ab}=a(b-1)\mathbf{1}_{ab},\\
&A_2(\mathbf{1}_a\otimes C_{b-1})=(0-(a-1)-1)(\mathbf{1}_a\otimes C_{b-1})=-a(\mathbf{1}_a\otimes C_{b-1}),\text{ and }\\
&A_2(C_{a-1}\otimes I_b)=(0-(-1)-1)(C_{a-1}\otimes I_b)=0.\qedhere
\end{align*}
\end{proof}

This shows that the blocks in $F$ give a rational basis of eigenvectors for each of the maximal common eigenspaces of the Bose-Mesner algebra $\Span\{I,A_1=(J_a-I_a)\otimes I_b,J-A_1-I\}$. As an immediate corollary we have

\begin{corollary}\label{cor-BCDeterminant}
\[\det D_{\alpha,\beta,\gamma}(a,b)=[(\alpha-\beta)+a(\beta-\gamma)+ab\gamma][(\alpha-\beta)+a(\beta-\gamma)]^{b-1}(\alpha-\beta)^{(a-1)b}.
\]
\end{corollary}
\begin{proof}
Notice that
\[D_{\alpha,\beta,\gamma}(a,b)=\alpha I_{ab}+\beta A_1+\gamma A_2.\]
By Lemma \ref{lemma-BCRationalEigenbasis}, we know that we have three common eigenspaces $V_0$, $V_1$, and $V_2$ for $A_1$ and $A_2$, of dimensions $1$, $b-1$ and $(a-1)b$ respectively. In $V_0$ the eigenvalues of $I$, $A_1$ and $A_2$ are $1$, $(a-1)$ and $a(b-1)$ respectively, hence the eigenvalue of $D_{\alpha,\beta,\gamma}(a,b)$ in $V_0$ is
\[\alpha+(a-1)\beta+a(b-1)\gamma=(\alpha-\beta)+a(\beta-\gamma)+ab\gamma.\]
Likewise, the eigenvalues of $D_{\alpha,\beta,\gamma}(a,b)$ in $V_1$ and $V_2$ are
\begin{align*}
& \alpha + (a-1)\beta -a\gamma=(\alpha-\beta)+a(\beta-\gamma),\text{ and }\\
& \alpha + (-1)\beta + 0\gamma =(\alpha-\beta),
\end{align*}
respectively. The result follows directly from this.\qedhere

\end{proof}

\begin{proposition}\normalfont \label{prop-BCRationalCongruence} There is a rational congruence
\[D_{\alpha,\beta,\gamma}(a,b)\simeq x_0\oplus X_1\oplus X_2,\]
where
\begin{align*}
& x_0=ab[(\alpha-\beta)+a(\beta-\gamma)+ab\gamma],\\
& X_1=a[(\alpha-\beta)+a(\beta-\gamma)](I_{b-1}+J_{b-1}), \text{ and }\\
& X_2=(\alpha-\beta)(I_{a-1}+J_{a-1})\otimes I_b.
\end{align*}
\end{proposition}
\begin{proof} Let $F=[F_0|F_1|F_2]$ be the block-matrix of Lemma \ref{lemma-BCRationalEigenbasis}. Since the columns of $F_i$ are in the common eigenspace $V_i$, it follows that $F_i^{\intercal}F_j=0$ whenever $i\neq j$. From the fact that $C_m^{\intercal}C_m=I_m+J_m$, we have that
\begin{align*}
&F_0^{\intercal}F_0=\mathbf{1}_{ab}^{\intercal}\mathbf{1}_{ab}=ab,\\
&F_1^{\intercal}F_1=(\mathbf{1}_a^{\intercal}\otimes C_{b-1}^{\intercal})(\mathbf{1}_a\otimes C_{b-1})=(a\otimes C_{b-1}^{\intercal}C_{b-1})=a(I_{b-1}+J_{b-1}),\text{ and }\\
&F_2^{\intercal}F_2=(C_{a-1}^{\intercal}\otimes I_b)(C_{a-1}\otimes I_b)=(C_{a-1}^{\intercal}C_{a-1})\otimes I_b=(I_{a-1}+J_{a-1})\otimes I_b.
\end{align*}
Now, since the columns of $F_i$ are eigenvectors of $A_1$ and $A_2$, we have
\begin{align*}
&A_1F=[(a-1)F_0|(a-1)F_1|-F_2],\text{ and }\\
&A_2F=[a(b-1)F_0|-aF_1 |\ 0],
\end{align*}
from which it follows that
\[D_{\alpha,\beta,\gamma}(a,b)F=[((\alpha-\beta)+a(\beta-\gamma)+ab\gamma)F_0|((\alpha-\beta)+a(\beta-\gamma))F_1|(\alpha-\beta)F_2].\]
Hence,
\[F^{\intercal}D_{\alpha,\beta,\gamma}(a,b)F=
\left[
\begin{array}{c|c|c}
x_0 & &\\
\hline 
& X_1 &\\
\hline
&&X_2
\end{array}
\right]=x_0\oplus X_1\oplus X_2,
\]
where
\begin{align*}
& x_0= ((\alpha-\beta)+a(\beta-\gamma)+ab\gamma)F_0^{\intercal}F_0=ab[(\alpha-\beta)+a(\beta-\gamma)+ab\gamma],\\
& X_1=((\alpha-\beta)+a(\beta-\gamma))F_1^{\intercal}F_1=a[(\alpha-\beta)+a(\beta-\gamma)](I_{b-1}+J_{b-1}),\text{ and }\\
& X_2=(\alpha-\beta)F_2^{\intercal}F_2=(\alpha-\beta)(I_{a-1}+J_{a-1})\otimes I_b.\qedhere
\end{align*}
\end{proof}
In the paper by Bose and Connor \cite{Bose-Connor}, the authors consider the following class of incidence structures:
\begin{definition}\normalfont A \textit{group-divisible design} or \textit{GDD}, \index{design!group-divisible} is an incidence structure on $v$ points and $b$ blocks, such that each point is in $r$ blocks  each of size $k$. Additionally, the points can be divided into $m$ ``groups'', with $n$ points each, in such a way that each pair of points in the same group are incident to $\lambda_1$ blocks, and each pair of points in distinct groups are incident to $\lambda_2$ blocks.
\end{definition}
In the case when $\lambda_1=\lambda_2$, the definition of GDDs coincides with that of $2$-designs. It is easy to check that there is an indexing of the points of a GDD such that if $N$ is the incidence matrix with respect to it, then 
\[NN^{\intercal}=D_{r,\lambda_1,\lambda_2}(n,m).\]
The parameters of the GDD satisfy the following combinatorial relations, see \cite{Bose-Connor}:
\begin{align*}
& v=mn,\ bk=vr,\\
& (n-1)\lambda_1+n(m-1)\lambda_2=r(k-1), \text{ and }\\
& r\geq \max(\lambda_1,\lambda_2).
\end{align*}
In particular, $rk=(r-\lambda_1)-n(\lambda_2-\lambda_1)+v\lambda_2$, and by Corollary \ref{cor-BCDeterminant}, we have that
\[\det(N)^2=rk(rk-v\lambda_2)^{m-1}(r-\lambda_1)^{m(n-1)}.\]

Therefore, the study of GDDs splits into the following cases,
\begin{itemize}
\item[(i)] \textit{Singular GDDs}: $r=\lambda_1$,
\item[(ii)]\textit{Semi-regular GDDs}: $r>\lambda_1$, $rk-v\lambda_2=0$, and
\item[(iii)]\textit{Regular GDDs}: $r>\lambda_1$, $rk-v\lambda_2>0$.
\end{itemize}

\begin{example}[Example 7.1.9 \cite{Stinson-Designs}]\normalfont
The matrix 
\[N=
\begin{bmatrix}
0 & 0 & 1 & 1 & 1 & 0 & 0 & 0\\
0 & 0 & 0 & 0 & 0 & 1 & 1 & 1\\
1 & 0 & 0 & 1 & 0 & 0 & 1 & 0\\
0 & 1 & 0 & 0 & 1 & 0 & 0 & 1\\
1 & 0 & 0 & 0 & 1 & 1 & 0 & 0\\
0 & 1 & 1 & 0 & 0 & 0 & 1 & 0\\
1 & 0 & 1 & 0 & 0 & 0 & 0 & 1\\
0 & 1 & 0 & 1 & 0 & 1 & 0 & 0
\end{bmatrix},
\]
is the incidence matrix of a regular GDD with parameter set $(n,m,k,\lambda_1,\lambda_2)=(2,4,3,0,1)$. This can be seen by taking $NN^{\intercal}$ and checking that
\[NN^{\intercal}=
\left[
\begin{array}{cc|cc|cc|cc}
3 &0 &1 &1 &1 &1 &1 &1\\
0 &3 &1 &1 &1 &1 &1 &1\\
\hline
1 &1 &3 &0 &1 &1 &1 &1\\
1 &1 &0 &3 &1 &1 &1 &1\\
\hline
1 &1 &1 &1 &3 &0 &1 &1\\
1 &1 &1 &1 &0 &3 &1 &1\\
\hline
1 &1 &1 &1 &1 &1 &3 &0\\
1 &1 &1 &1 &1 &1 &0 &3
\end{array}
\right].
\]
\end{example}
Both singular and semi-regular GDDs have straightforward characterisations. In \cite{Bose-Connor}, the authors consider \textit{symmetric} regular GDDs. These are regular GDDs for which $v=b$, and we have the following conditions between parameters,
\begin{align*}
&v=b=mn,\ r=k,\\
&(n-1)\lambda_1+n(m-1)\lambda_2=r(r-1),\\
&\nu:=r-\lambda_1>0, \mu:=r^2-v\lambda_2>0
\end{align*}

In this case, we have that $\det(N)^2=r^2\mu^{m-1}\nu^{m(n-1)}$. Therefore, we assume that $\mu^{m-1}\nu^{m(n-1)}$ is a perfect square. We remark that our notation $\mu$ and $\nu$ is not standard, in \cite{Bose-Connor} the authors use the notation $Q$ and $P$ respectively.
\begin{remark}\normalfont In the \textit{Handbook of Combinatorial designs} \cite{HandbookOfDesigns} the definition we give of GDD corresponds to a \textit{uniform} $k$-GDD of index $\lambda=\lambda_2$. Notice that given $n,m,r,k$ and $\lambda_2$, then $\lambda_1$ is uniquely determined. 
\end{remark}

\begin{theorem}[Bose and Connor, cf. \cite{Bose-Connor}]\label{thm-BC}\index{Bose-Connor Theorem}
Suppose that $v,r,k,\lambda_1,\lambda_2$ satisfy the relations of the parameters of a symmetric, regular GDD.  Then, the local Hasse-Minkowski invariants of $D_{r,\lambda_1,\lambda_2}(n,m)$ are
\[\varepsilon_p(D_{r,\lambda_1,\lambda_2}(n,m))=(\mu,nm)_p(\mu,n)_p^m(\nu,n)_p^{m}(\mu,-1)_p^{m\choose 2}(\nu,-1)_p^{m(n-1)\choose 2}.\]
\end{theorem}
\begin{proof}
From Proposition \ref{prop-BCRationalCongruence}, and the relations between parameters we find that 
\[D_{r,\lambda_1,\lambda_2}(n,m)\simeq X\oplus Y,\]
where
\begin{align*}
&X=x_0\oplus X_1=
\left[\begin{array}{c|c}
vr^2 & \\
\hline
      & n\mu(I_{m-1}+J_{m-1})
\end{array}
\right],\text{ and }\\
&Y=X_2=\nu[(I_{n-1}+J_{n-1})\otimes I_m].
\end{align*}
We compute the Hasse-Minkowski invariants of $X$ and $Y$ using the formulas for the invariants of the direct sum and the product by a scalar (see Lemma \ref{lemma-DirectSum-Symbol} and its corollaries, together with Lemma \ref{lemma-ScaledInvariant}). From the fact that $\det (I_{d}+J_{d})=d+1$, and $\varepsilon_p(I_{d}+J_{d})=(d,d+1)_p$ (Proposition \ref{lemma-SymbolIJ}) we have
\begin{align*}
\varepsilon_p(X)&=\varepsilon_p(x_0\oplus X_1)\\
&=(x_0,\delta_{X_1})_p\varepsilon_p(X_1)\\
&=(nm,(n\mu)^{m-1}m)_p(n\mu,-1)_p^{m-1\choose 2}(n\mu,m)_p^m(m-1,m)_p.
\end{align*}
We use bilinearity to rewrite the symbol above. On the one hand we have that
\begin{align*}
(nm,(n\mu)^{m-1}m)_p&=(nm,m)_p(nm,n)_p^{m-1}(nm,\mu)_p^{m-1}\\
&=(n,m)_p(m,-1)_p(n,-1)_p^{m-1}(n,m)_p^{m-1}(n,\mu)_p^{m-1}(m,\mu)_p^{m-1}\\
&=(m,-1)_p(n,-1)_p^{m-1}(n,m)_p^m(n,\mu)_p^{m-1}(m,\mu)_p^{m-1}.
\end{align*}
Expanding the term $(n\mu,m)^m$ as $(n,m)_p^m(\mu,m)_p^{m-1}(\mu,m)_p$ and using the fact that $(m,-1)_p(m-1,m)_p=1$, we can cancel terms and find that
\begin{align*}
\varepsilon_p(X)&=(n,-1)_p^{m-1}(n\mu,-1)_p^{m-1\choose 2}(\mu,n)_p^{m-1}(\mu,m)_p\\
&=(n,-1)_p^{m\choose 2}(\mu,-1)_p^{m-1\choose 2}(\mu,n)_p^{m-1}(\mu,m)_p.
\end{align*}
On the other hand,
\begin{align*}\varepsilon_p(Y)=\varepsilon_p(X_2)&=(\nu,-1)_p^{m(n-1)\choose 2}(\nu,n^m)_p^{m(n-1)-1}(n-1,n)_p^m(n,-1)_p^{m\choose 2}\\
&=(\nu,-1)_p^{m(n-1)\choose 2}(\nu^{m(n-1)},n^m)_p(\nu,n)_p^m(n-1,n)_p^m(n,-1)_p^{m\choose 2}.
\end{align*}
Since $\nu^{m(n-1)}\mu^{m-1}$ is a square, we find that $(\nu^{m(n-1)},n^m)_p=(\mu^{m-1},n^m)_p
=(\mu,n)_p^{m(m-1)}$. Now, $m(m-1)$ is always even, so $(\mu,n)_p^{m(m-1)}=1$. This implies, 
\[\varepsilon_p(Y)=(\nu,-1)_p^{m(n-1)\choose 2}(\nu,n)_p^m(n-1,n)_p^m(n,-1)_p^{m\choose 2}.\]
Finally,
\begin{align*}
(\delta_X,\delta_Y)_p&=(nm(n\mu)^{m-1}m,\nu^{m(n-1)}n^m)_p\\
&=(n^m\mu^{m-1},n^m\nu^{m(n-1)})_p\\
&=(n^m\mu^{m-1},-\mu^{m-1}\nu^{m(n-1)})_p.
\end{align*}
Using that $\mu^{m-1}\nu^{m(n-1)}$ is a square we find
\[(\delta_X,\delta_Y)_p=(n^m\mu^{m-1},-1)_p=(n,-1)_p^m(\mu,-1)_p^{m-1}.\]
Lemma \ref{lemma-DirectSum-Symbol} tells us that $\varepsilon_p(X\oplus Y)=\varepsilon_p(X)\varepsilon_p(Y)(\delta_X,\delta_Y)_p$. Putting all terms together we have the following cancellations: $(n,-1)_p^m$ in $(\delta_X,\delta_Y)_p$ cancels with $(n-1,n)_p^m$ in $\varepsilon_p(Y)$, and the term $(n,-1)_p^{m\choose 2}$ appears in both $\varepsilon_p(X)$ and $\varepsilon_p(Y)$ so these vanish. We find,
\begin{align*}
\varepsilon_p(X\oplus Y)&=(\mu,-1)_p^{m-1\choose 2}(\mu,-1)_p^{m-1} (\mu,n)_p^{m-1}(\mu,m)_p(\nu,-1)_p^{m(n-1)\choose 2}(\nu,n)_p^m\\
&=(\mu,n)_p^{m-1}(\mu,-1)_p^{m\choose 2}(\nu,n)_p^{m}(\nu,-1)_p^{m(n-1)\choose 2}(\mu,m)_p.
\end{align*}
Since $(\mu,n)_p^2=1$, we may multiply by $(\mu,n)_p$ twice and gather the terms $(\mu,n)_p^{m-1}(\mu,n)_p=(\mu,n)_p^m$, and $(\mu,m)_p(\mu,n)_p=(\mu,mn)_p$. We obtain then
\[\varepsilon_p(D_{r,\lambda_1,\lambda_2}(n,m))=\varepsilon_p(X\oplus Y)=(\mu,nm)_p(\mu,n)_p^m(\nu,n)_p^m(\mu,
-1)_p^{m\choose 2}(\nu,-1)_p^{m(n-1)\choose 2}.\qedhere\]
\end{proof}

\begin{remark}\normalfont
The term $(\mu,nm)_p$ in our formula for $\varepsilon_p(X)$ appears as $(\mu,\lambda_2)_p$ in Equation (9.9) of the paper by Bose and Connor \cite{Bose-Connor}. We can identify these two terms using the following Hilbert symbol relation, see \cite{Tamura-DOptimal}, 
\[(a,bc)_p=(a+bc,-abc)_p.\]
To prove the equation above notice that $(a,bc)_p(a+bc,-abc)_p=(a,-abc)_p(a+bc,-abc)_p=(a^2+abc,-abc)_p$, and the equation $(a^2+abc)x^2+(-abc)y^2=z^2$ has the non-trivial solution $(x,y,z)=(1,1,a)$. So, $(a,bc)_p(a+bc,-abc)_p=1$, and hence $(a,bc)_p=(a+bc,-abc)_p$. Now, recall that $\mu=r^2-nm\lambda_2$, and apply the formula above with $a=r^2$, $b=-nm$ and $c=\lambda_2$. We find
\[1=(r^2,-nm\lambda_2)_p=(r^2-nm\lambda_2,r^2nm\lambda_2)_p=(\mu,nm\lambda_2)_p.\]
Now, using bilinearity we find $(\mu,nm)_p(\mu,\lambda_2)_p=1$, and so
\[(\mu,nm)_p=(\mu,\lambda_2)_p.\]
\end{remark}

\begin{remark}\normalfont
It is not at all harder to compute the general form of the Hilbert symbols for $D_{\alpha,\beta,\gamma}(a,b)$, without any assumptions on the parameters. The hardest part of the work that we did in Theorem \ref{thm-BC} consisted in simplifying the symbols using the relations between parameters.
\end{remark}

\begin{corollary}[Bose and Connor, \cite{Bose-Connor}]\normalfont\label{cor-BCNonEx}
Suppose there exists a symmetric regular GDD with parameters $n,m,r,\lambda_1,\lambda_2$. Then
\begin{itemize}
\item[(i)] If $m$ is even, then $\mu$ must be a perfect square. Furthermore, if $m\equiv 2\pmod{4}$ and $n$ is even then $(\nu,-1)_p=1$ for all odd primes $p$.
\item[(ii)] If $m$ is odd and $n$ is even, then $\nu$ must be a perfect square and
\[(\mu,-1^{m\choose 2}m)_p=1,\]
for all odd primes $p$.
\item[(iii)] If $m$ and $n$ are both odd, then 
\[(\mu,(-1)^{m\choose 2}m)_p(\nu,(-1)^{n\choose 2}n)_p=1,\]
for all odd primes $p$.
\end{itemize}
\end{corollary}
\begin{proof}
If a GDD exists with the given parameters, then letting $N$ be its incidence matrix we have that 
\[NN^{\intercal}=D_{r,\lambda_1,\lambda_2}(n,m).\]
And so $D_{r,\lambda_1,\lambda_2}(n,m)$ is rationally congruent to the identity matrix. By the Hasse-Minkowski Theorem (Theorem \ref{thm-HasseMinkowski}) the discriminant of $D_{r,\lambda_1,\lambda_2}(n,m)$ must be a square, and all local Hasse-Minkowski invariants should be equal to $1$. Recall that $\det(D_{r,\lambda_1,\lambda_2}(n,m))=r^2\mu^{m-1}\nu^{m(n-1)}$. By Theorem \ref{thm-BC}
\[\varepsilon_p(D_{r,\lambda_1,\lambda_2}(n,m))=(\mu,nm)_p(\mu,n)_p^m(\nu,n)_p^{m}(\mu,-1)_p^{m\choose 2}(\nu,-1)^{m(n-1)\choose 2}.\]
In case (i) the fact that $\det(D_{r,\lambda_1,\lambda_2}(n,m))$ must be a perfect square implies that $\mu$ must be a perfect square, and the Hasse-Minkowski invariant reduces to $(\nu,-1)^{m(n-1)\choose 2}$. If $m\equiv 2\pmod{4}$, and $n$ is even, then $m(n-1)\equiv 2\pmod{4}$, and thus ${m(n-1)\choose 2}$ is odd. Therefore, we have
\[\varepsilon_p(D_{r,\lambda_1,\lambda_2}(n,m))=(\nu,-1)_p=1,\]
for all odd primes $p$.\\

In case (ii) the determinant condition implies that $\nu$ must be a perfect square, and we have
\[\varepsilon_p(D_{r,\lambda_1,\lambda_2}(n,m))=(\mu,nm)_p(\mu,n)_p(\mu,-1)_p^{m\choose 2}=(\mu,(-1)^{m\choose 2}m)_p=1,\]
for all odd primes $p$.\\

In case (iii) the determinant condition does not imply that $\mu$ or $\nu$ are perfect squares. We are left only with Hasse-Minkowski obstructions. Gathering terms we may write
\[\varepsilon_p(D_{r,\lambda_1,\lambda_2}(n,m))=(\mu,(-1)^{m\choose 2}m)_p(\nu,(-1)^{m(n-1)\choose 2}n)_p.\]
Notice that for $m$ and $n$ odd, the binomial coefficient ${m(n-1)\choose 2}$ is odd if and only if $n\equiv 3\pmod{4}$. Therefore, the parity of ${m(n-1)\choose 2}$ coincides with that of ${n\choose 2}$ and we find
\[\varepsilon_p(D_{r,\lambda_1,\lambda_2}(n,m))=(\mu,(-1)^{m\choose 2}m)_p(\nu,(-1)^{n\choose 2}n)_p=1,\]
for all odd primes $p$.\qedhere
\end{proof}

\subsection{Non-feasibility tables for parameters of GDDs}
In this subsection we present a table of non-existence results obtained using the different parts of Corollary \ref{cor-BCNonEx}. There is a large number of parameters $n,m,r,\lambda_1,$ and $\lambda_2$ satisfying the conditions 
\begin{align*}
&(n-1)\lambda_1+n(m-1)\lambda_2=r(r-1),\\
&\nu=r-\lambda_1>0,\text{ and }\\
&\mu=r^2-nm\lambda_2>0.
\end{align*}
Therefore, we will restrict to the case $\lambda_2=1$ to reduce the number of combinations. The case $\lambda_2=1$ is also of special interest: it is a singled-out case in many design theory reference books (see 1.5 in Part IV of \cite{HandbookOfDesigns}, and Definition 6.1 in Chapter I of  \cite{Beth-Jungnickel-Lenz}), and some authors \cite{Stinson-Designs} only consider the case $\lambda_2=1$ in the definition of GDD. Some of the impossible parameters that we present already appeared  in \cite{Bose-Connor}. But to the best of our knowledge most of them have not appeared tabulated in this way before.\\

A GDD with $\lambda_1=0$ is called \textit{resolvable}. The family of resolvable GDDs is of special interest, so we include additional tables of non-existence for GDDs with $\lambda_1=0$ and $\lambda_2=1$ in each case.\\

Each table includes:
\begin{itemize}
\item[1.] A first column with a \textit{reference number}
\item[2.] Five columns, each with the values of parameters $n$, $m$, $r$, $\lambda_1$ and $\lambda_2$, respectively
\item[3.] A last column with the \textit{reason for infeasibility} of the parameter set (See below for the correct interpretation of this column).
\end{itemize}
The last column includes one of the symbols $\mu$, $\nu$ or $p$ followed by an equal sign and an integer. Whenever the symbols $\mu$ or $\nu$ appear, the reason for infeasibility is that $\mu$ or $\nu$ are not perfect squares, and the value of $\mu$ or $\nu$ follows. If instead the symbol $p$ appears, then the number following is a prime, and the reason for infeasibility is that $\varepsilon_p(D_{r,\lambda_1,\lambda_2}(n,m))=-1$ for the value of $p$ indicated in the table.\\


\newpage
The tables below correspond to the case $m\equiv 2\pmod{4}$ and $n$ even. This is a particular case of Case (i) in Corollary \ref{cor-BCNonEx}.

\begin{table}[H]
\parbox{.45\linewidth}{
\centering
\begin{tabular}{|c|ccccc|c|}
\hline
No. & $n$ & $m$ & $r$ & $\lambda_1$ & $\lambda_2$ & \makecell{Reason for\\ infeasibility}\\
\hline
1 & 2 & 10 & 5 & 2 & 1 &  $\mu=5$\\
2 & 2 & 14 & 6 & 4 & 1 &  $\mu=8$\\
3 & 2 & 22 & 7 & 0 & 1 &  $\mu=5$\\
4 & 2 & 26 & 8 & 6 & 1 &  $\mu=12$\\
5 & 2 & 34 & 9 & 6 & 1 &  $\mu=13$\\
6 & 2 & 46 & 10 & 0 & 1 &  $\mu=8$\\
7 & 4 & 22 & 10 & 2 & 1 &  $\mu=12$\\
8 & 4 & 34 & 12 & 0 & 1 &  $\mu=8$\\
9 & 4 & 34 & 13 & 8 & 1 &  $\mu=33$\\
10 & 4 & 42 & 14 & 6 & 1 &  $\mu=28$\\
11 & 4 & 46 & 15 & 10 & 1 &  $\mu=41$\\
12 & 6 & 18 & 12 & 6 & 1 &  $p=3$\\
13 & 6 & 22 & 13 & 6 & 1 &  $\mu=37$\\
14 & 6 & 26 & 15 & 12 & 1 &  $\mu=69$\\
15 & 6 & 38 & 17 & 10 & 1 &  $\mu=61$\\
16 & 6 & 42 & 18 & 12 & 1 &  $\mu=72$\\
17 & 8 & 6 & 11 & 10 & 1 &  $\mu=73$\\
18 & 8 & 10 & 13 & 12 & 1 &  $\mu=89$\\
19 & 8 & 14 & 12 & 4 & 1 &  $\mu=32$\\
20 & 8 & 22 & 14 & 2 & 1 &  $\mu=20$\\
21 & 8 & 34 & 18 & 6 & 1 &  $\mu=52$\\
22 & 8 & 38 & 20 & 12 & 1 &  $\mu=96$\\
23 & 8 & 50 & 21 & 4 & 1 &  $\mu=41$\\
24 & 8 & 50 & 22 & 10 & 1 &  $\mu=84$\\
25 & 10 & 22 & 15 & 0 & 1 &  $\mu=5$\\
26 & 10 & 34 & 21 & 10 & 1 &  $\mu=101$\\
\hline

\end{tabular}
\caption{Infeasible parameter sets for GDDs with $\lambda_2=1$, $m\equiv 2\pmod{4}$ and $n$ even. $2\leq n\leq 10$, $2\leq m\leq 50$.}
}
\hfill
\parbox{.45\linewidth}{
\centering
\begin{tabular}{|c|ccccc|c|}
\hline
No. & $n$ & $m$ & $r$ & $\lambda_1$ & $\lambda_2$ & \makecell{Reason for\\ infeasibility}\\
\hline
1 & 2 & 22 & 7 & 0 & 1 &  $\mu=5$\\
2 & 2 & 46 & 10 & 0 & 1 &  $\mu=8$\\
3 & 4 & 34 & 12 & 0 & 1 &  $\mu=8$\\
4 & 6 & 58 & 19 & 0 & 1 &  $\mu=13$\\
5 & 6 & 78 & 22 & 0 & 1 &  $p=11$\\
6 & 8 & 70 & 24 & 0 & 1 &  $p=3$\\
7 & 10 & 22 & 15 & 0 & 1 &  $\mu=5$\\
8 & 10 & 94 & 31 & 0 & 1 &  $\mu=21$\\
9 & 14 & 34 & 22 & 0 & 1 &  $\mu=8$\\
10 & 14 & 86 & 35 & 0 & 1 &  $\mu=21$\\
11 & 20 & 22 & 21 & 0 & 1 &  $p=3$\\
12 & 26 & 58 & 39 & 0 & 1 &  $\mu=13$\\
13 & 28 & 46 & 36 & 0 & 1 &  $\mu=8$\\
14 & 30 & 70 & 46 & 0 & 1 &  $p=23$\\
15 & 30 & 86 & 51 & 0 & 1 &  $\mu=21$\\
16 & 32 & 34 & 33 & 0 & 1 &  $p=3$\\
17 & 40 & 78 & 56 & 0 & 1 &  $p=7$\\
18 & 42 & 94 & 63 & 0 & 1 &  $\mu=21$\\
19 & 56 & 58 & 57 & 0 & 1 &  $p=3$\\
20 & 68 & 70 & 69 & 0 & 1 &  $p=3$\\
21 & 76 & 78 & 77 & 0 & 1 &  $p=7$\\
22 & 92 & 94 & 93 & 0 & 1 &  $p=3$\\
\hline
\end{tabular}
\caption{Infeasible parameter sets for GDDs with $\lambda_1=0$, $\lambda_2=1$, $m\equiv 2\pmod{4}$ and $n$ even. $2\leq n,m\leq 100$.}
}
\end{table}

\pagebreak
The tables below corresponds to infeasible parameter sets for GDDs with $n$ even, $m$ odd and $\lambda_2=1$. This corresponds to Case (ii) in \ref{cor-BCNonEx}. Notice that if $\lambda_2=1$ and $\lambda_1=0$, then $\nu=r-\lambda_1=r$ must be a perfect square.

\begin{table}[H]
\parbox{.45\linewidth}{
\centering
\begin{tabular}{|c|ccccc|c|}
\hline
No. & $n$ & $m$ & $r$ & $\lambda_1$ & $\lambda_2$ & \makecell{Reason for\\ infeasibility}\\
\hline
1 & 2 & 11 & 5 & 0 & 1 & $\nu=5$\\
2 & 2 & 21 & 7 & 2 & 1 & $\nu=5$\\
3 & 2 & 27 & 8 & 4 & 1 & $p=5$\\
4 & 2 & 29 & 8 & 0 & 1 & $\nu=8$\\
5 & 2 & 33 & 9 & 8 & 1 & $p=3$\\
6 & 2 & 35 & 9 & 4 & 1 & $\nu=5$\\
7 & 2 & 43 & 10 & 6 & 1 & $p=7$\\
8 & 2 & 45 & 10 & 2 & 1 & $\nu=8$\\
9 & 4 & 7 & 7 & 6 & 1 & $p=3$\\
10 & 4 & 15 & 8 & 0 & 1 & $\nu=8$\\
11 & 4 & 19 & 10 & 6 & 1 & $p=3$\\
12 & 4 & 31 & 12 & 4 & 1 & $\nu=8$\\
13 & 4 & 45 & 14 & 2 & 1 & $\nu=12$\\
14 & 6 & 11 & 10 & 6 & 1 & $p=17$\\
15 & 6 & 23 & 12 & 0 & 1 & $\nu=12$\\
16 & 6 & 27 & 13 & 0 & 1 & $\nu=13$\\
17 & 6 & 31 & 15 & 6 & 1 & $p=13$\\
18 & 6 & 33 & 17 & 16 & 1 & $p=13$\\
19 & 6 & 43 & 17 & 4 & 1 & $\nu=13$\\
20 & 6 & 47 & 18 & 6 & 1 & $\nu=12$\\
21 & 8 & 15 & 15 & 14 & 1 & $p=5$\\
22 & 8 & 17 & 13 & 4 & 1 & $p=11$\\
23 & 8 & 35 & 17 & 0 & 1 & $\nu=17$\\
24 & 8 & 43 & 22 & 18 & 1 & $p=5$\\
25 & 8 & 45 & 20 & 4 & 1 & $p=5$\\
26 & 10 & 3 & 8 & 4 & 1 & $p=17$\\
27 & 10 & 13 & 15 & 10 & 1 & $\nu=5$\\
28 & 10 & 19 & 18 & 14 & 1 & $p=67$\\
29 & 10 & 31 & 22 & 18 & 1 & $p=3$\\
30 & 10 & 39 & 20 & 0 & 1 & $\nu=20$\\
31 & 10 & 43 & 21 & 0 & 1 & $\nu=21$\\
32 & 10 & 43 & 25 & 20 & 1 & $\nu=5$\\
\hline

\end{tabular}
\caption{Infeasible parameter sets for GDDs with $\lambda_2=1$, $m$ odd and $n$ even. $2\leq n\leq 10$, $1\leq m< 50$.}
}
\hfill
\parbox{.45\linewidth}{
\centering
\begin{tabular}{|c|ccccc|c|}
\hline
No. & $n$ & $m$ & $r$ & $\lambda_1$ & $\lambda_2$ & \makecell{Reason for\\ infeasibility}\\
\hline
1 & 4 & 151 & 25 & 0 & 1 & $p=3$\\
2 & 6 & 211 & 36 & 0 & 1 & $p=3$\\
3 & 14 & 91 & 36 & 0 & 1 & $p=11$\\
4 & 22 & 451 & 100 & 0 & 1 & $p=3$\\
5 & 24 & 271 & 81 & 0 & 1 & $p=3$\\
6 & 30 & 43 & 36 & 0 & 1 & $p=3$\\
7 & 30 & 331 & 100 & 0 & 1 & $p=7$\\
8 & 44 & 331 & 121 & 0 & 1 & $p=11$\\
9 & 70 & 547 & 196 & 0 & 1 & $p=7$\\
10 & 78 & 491 & 196 & 0 & 1 & $p=59$\\
11 & 88 & 235 & 144 & 0 & 1 & $p=7$\\
12 & 130 & 295 & 196 & 0 & 1 & $p=11$\\
13 & 182 & 211 & 196 & 0 & 1 & $p=7$\\
14 & 228 & 571 & 361 & 0 & 1 & $p=7$\\
15 & 252 & 771 & 441 & 0 & 1 & $p=3$\\
16 & 280 & 571 & 400 & 0 & 1 & $p=3$\\
17 & 308 & 631 & 441 & 0 & 1 & $p=7$\\
18 & 420 & 463 & 441 & 0 & 1 & $p=3$\\
19 & 462 & 507 & 484 & 0 & 1 & $p=11$\\
20 & 480 & 691 & 576 & 0 & 1 & $p=3$\\
21 & 520 & 751 & 625 & 0 & 1 & $p=3$\\
\hline
\end{tabular}
\caption{Infeasible parameter sets for GDDs with $r$ a perfect square, $\lambda_1=0$, $\lambda_2=1$, $m$ odd and $n$ even. $m\leq 800$.}\label{tab-mOdd_nEv_rSq}
}
\end{table}

\pagebreak
Finally we include a table of impossible parameters for GDDs with both $n$ and $m$ odd, this corresponds to Case (iii) of Corollary \ref{cor-BCNonEx}. Recall that in this case, neither $\nu$ nor $\mu$ need to be perfect squares, so the only obstructions appearing come from the Hasse-Minkowski invariants.

\begin{table}[H]
\parbox{.45\linewidth}{
\centering
\begin{tabular}{|c|ccccc|c|}
\hline
No. & $n$ & $m$ & $r$ & $\lambda_1$ & $\lambda_2$ & \makecell{Reason for\\ infeasibility}\\
\hline
1 & 3 & 5 & 5 & 4 & 1 & $p=5$\\
2 & 3 & 11 & 6 & 0 & 1 & $p=3$\\
3 & 3 & 13 & 7 & 3 & 1 & $p=5$\\
4 & 3 & 21 & 9 & 6 & 1 & $p=3$\\
5 & 3 & 23 & 9 & 3 & 1 & $p=3$\\
6 & 3 & 31 & 10 & 0 & 1 & $p=5$\\
7 & 3 & 41 & 12 & 6 & 1 & $p=7$\\
8 & 3 & 43 & 12 & 3 & 1 & $p=3$\\
9 & 3 & 45 & 13 & 12 & 1 & $p=17$\\
10 & 5 & 7 & 6 & 0 & 1 & $p=3$\\
11 & 5 & 13 & 9 & 3 & 1 & $p=3$\\
12 & 5 & 19 & 10 & 0 & 1 & $p=5$\\
13 & 5 & 19 & 11 & 5 & 1 & $p=13$\\
14 & 5 & 21 & 12 & 8 & 1 & $p=13$\\
15 & 5 & 29 & 13 & 4 & 1 & $p=3$\\
16 & 5 & 31 & 14 & 8 & 1 & $p=3$\\
17 & 5 & 35 & 15 & 10 & 1 & $p=5$\\
18 & 5 & 39 & 15 & 5 & 1 & $p=5$\\
19 & 5 & 41 & 16 & 10 & 1 & $p=17$\\
20 & 5 & 43 & 15 & 0 & 1 & $p=3$\\
21 & 7 & 7 & 10 & 8 & 1 & $p=3$\\
22 & 7 & 13 & 13 & 12 & 1 & $p=13$\\
23 & 7 & 31 & 15 & 0 & 1 & $p=3$\\
24 & 7 & 31 & 16 & 5 & 1 & $p=13$\\
25 & 7 & 37 & 18 & 9 & 1 & $p=13$\\
26 & 7 & 37 & 19 & 15 & 1 & $p=17$\\
27 & 7 & 43 & 19 & 8 & 1 & $p=3$\\
28 & 7 & 45 & 20 & 12 & 1 & $p=5$\\
29 & 9 & 3 & 7 & 3 & 1 & $p=11$\\
30 & 9 & 19 & 15 & 6 & 1 & $p=3$\\
31 & 9 & 23 & 19 & 18 & 1 & $p=11$\\
32 & 9 & 39 & 22 & 15 & 1 & $p=7$\\
33 & 9 & 43 & 23 & 16 & 1 & $p=71$\\
\hline

\end{tabular}
\caption{Infeasible parameter sets for GDDs $n$ and $m$ both odd, $\lambda_2=1$.\\ $3\leq n< 10$, $3\leq m< 50$.}
}
\hfill
\parbox{.45\linewidth}{
\centering
\begin{tabular}{|c|ccccc|c|}
\hline
No. & $n$ & $m$ & $r$ & $\lambda_1$ & $\lambda_2$ & \makecell{Reason for\\ infeasibility}\\
\hline
1 & 3 & 11 & 6 & 0 & 1 & $p=3$\\
2 & 3 & 31 & 10 & 0 & 1 & $p=5$\\
3 & 3 & 53 & 13 & 0 & 1 & $p=5$\\
4 & 3 & 71 & 15 & 0 & 1 & $p=3$\\
5 & 5 & 7 & 6 & 0 & 1 & $p=3$\\
6 & 5 & 19 & 10 & 0 & 1 & $p=5$\\
7 & 5 & 43 & 15 & 0 & 1 & $p=3$\\
8 & 5 & 77 & 20 & 0 & 1 & $p=3$\\
9 & 5 & 85 & 21 & 0 & 1 & $p=3$\\
10 & 7 & 31 & 15 & 0 & 1 & $p=3$\\
11 & 7 & 61 & 21 & 0 & 1 & $p=3$\\
12 & 7 & 67 & 22 & 0 & 1 & $p=3$\\
13 & 11 & 43 & 22 & 0 & 1 & $p=11$\\
14 & 13 & 15 & 14 & 0 & 1 & $p=7$\\
15 & 13 & 51 & 26 & 0 & 1 & $p=13$\\
16 & 15 & 29 & 21 & 0 & 1 & $p=7$\\
17 & 15 & 59 & 30 & 0 & 1 & $p=3$\\
18 & 17 & 71 & 35 & 0 & 1 & $p=5$\\
19 & 19 & 75 & 38 & 0 & 1 & $p=19$\\
20 & 19 & 79 & 39 & 0 & 1 & $p=13$\\
21 & 21 & 23 & 22 & 0 & 1 & $p=11$\\
22 & 21 & 83 & 42 & 0 & 1 & $p=7$\\
23 & 29 & 31 & 30 & 0 & 1 & $p=3$\\
24 & 33 & 61 & 45 & 0 & 1 & $p=3$\\
25 & 33 & 91 & 55 & 0 & 1 & $p=5$\\
26 & 35 & 71 & 50 & 0 & 1 & $p=5$\\
27 & 35 & 89 & 56 & 0 & 1 & $p=3$\\
28 & 37 & 39 & 38 & 0 & 1 & $p=19$\\
29 & 41 & 43 & 42 & 0 & 1 & $p=3$\\
30 & 45 & 47 & 46 & 0 & 1 & $p=23$\\
31 & 45 & 67 & 55 & 0 & 1 & $p=5$\\
\hline
\end{tabular}
\caption{Infeasible parameter sets for GDDs with $n$ and $m$ both odd, $\lambda_1=0$, and $\lambda_2=1$. $3\leq n<50$, $3\leq m< 100$.}\label{tab-mOdd_nOdd}
}
\end{table}

\section{An application to maximal determinant matrices}\label{sec-MaxdetApp}

In this section we apply the Bose-Connor Theorem to compute the local invariants of the $7m\times 7m$ matrices 
\[D(m)=((7m-3)I_m+4J_m)\otimes I_7 -J_{7m}.\]
These matrices were introduced by Ehlich in \cite{Ehlich-3mod4} to obtain an upper bound for the determinant of a $\pm 1 $ matrix of order $n\equiv 3\pmod{4}$. In particular, a $\pm 1$ matrix $X$ of order $7m$ satisfying $XX^{\intercal}=D(m)$ attains the maximum possible  absolute value of the determinant among all $\pm 1$ matrices of its order. As shown by Tamura in \cite{Tamura-DOptimal}, the smallest value of $m$ for which $D(m)$ can be a Gram matrix is $m=511/7= 73$, and hence the smallest order at which Ehlich's determinant bound can be met with equality is $n=511$.\\

The original form of the Bose-Connor Theorem is not directly applicable to compute the invariants in $D(m)$ since their proof assumes the existence of a group-divisible design (GDD) to compute the invariants of $D_{\alpha,\beta,\gamma}(a,b)$. Having the application to the theory of maximal determinants in mind this assumption needs to be dropped a priori, since the putative maximal determinant matrices need not have constant row-sum. Tamura mentioned in his paper \cite{Tamura-DOptimal} that his computation for the local invariants of $D_{\alpha,\beta,\gamma}(a,b)$ is \textit{almost the same} as the proof of the Bose-Connor Theorem, and by this Tamura may have meant that the assumptions on existence of designs needed to be removed, and that the result holds more generally, although this is not explicitly stated. Tamura's result is nonetheless true in its form, since he proved that if a $\pm 1$ matrix $X$ meets the Ehlich bound, then $\frac{1}{2}(J-X)$ is the incidence matrix of a GDD to which the Bose-Connor Theorem applies.\\

Since $D(m)=((7m-3)I_m+4J_m)\otimes I_7 -J_{7m}=D_{7m,3,-1}(m,7)$, plugging in the values $a=m, b=7, \alpha=7m,\beta=3,$ and $\gamma=-1$ in Proposition \ref{prop-BCRationalCongruence}, we have that $x_0=7m(4m-3)$, $X_1=m(11m-3)(I_6+J_6)$, and $X_2=(7m-3)(I_{m-1}+J_{m-1})\otimes I_7$, hence
\[D(m)\simeq \left[\begin{array}{c|c}
X &\\
\hline
& Y
\end{array}\right],\]
where 
\[X=\left[\begin{array}{c|c}
7m(4m-3) &\\
\hline
& (11m-3)m(I_6+J_6)
\end{array}
\right]
\text{, and }
Y=(7m-3)(I_{m-1}+J_{m-1})\otimes I_7.
\]
Proceeding analogously to the proof of Theorem \ref{thm-BC} we find

\begin{theorem}[cf. Tamura \cite{Tamura-DOptimal}]\label{thm-EhlichSymbols}\normalfont
If $m$ is odd, the Hasse-Minkowski invariant of $D(m)$ at a prime $p$ is
\[\varepsilon_p(D(m))=(4m-3,7m)_p(11m-3,-7)_p(7m-3,m)_p.\]
\end{theorem}

We conclude with Tamura's result on the non-existence of Ehlich-type maximal determinant matrices.
\begin{corollary}[Tamura, \cite{Tamura-DOptimal}] If there is a $\pm 1$ matrix of order $7m\equiv 3\pmod{4}$ meeting the Ehlich bound, then
\begin{itemize}
\item $4m-3$ is a square, and
\item $(11m-3,-(7m-3))_p=1$ for all $p$ odd.
\end{itemize}
\end{corollary}
\begin{proof}
A $\pm 1$ matrix $X$ of order $7m\equiv 3\pmod{4}$ meeting the Ehlich bound satisfies $X^{\intercal}X=D(m)$, where $m\equiv 1\pmod{4}$. Therefore Theorem \ref{thm-EhlichSymbols} applies. The determinant of $D(m)$ is
\[\det D(m)=(4m-3)(11m-3)^6(7m-3)^{7(m-1)},\]
by Corollary \ref{cor-BCDeterminant}. Therefore, $4m-3$ must be a square and the term $(4m-3,7m)_p$ in $\varepsilon_p(D(m))$ vanishes. We find the conditions
\[\varepsilon_p(D(m))=(11m-3,-7)_p(7m-3,m)_p=1.\]
We can simplify this expression further following an argument by Tamura. Since $4m-3$ is a square, $(11m-3)\cdot 1^2-7m\cdot 1^2=4m-3=z^2$ for some integer $z$, so we find that $(11m-3,-7m)_p=1$. By bilinearity $(11m-3,-7)_p=(11m-3,m)_p$. Using the identity $(a,bc)_p=(a+bc,-abc)_p$ holds,  we find
\[(7m-3,m)_p=(7m-3,4m)_p=(7m-3+4m,-(7m-3)4m)_p=(11m-3,-(7m-3)m)_p.\]
Therefore
\begin{align*}
\varepsilon_p(D(m))&=(11m-3,-7)_p(7m-3,m)_p\\
&=(11m-3,m)_p(11m-3,-(7m-3)m)_p\\
&=(11m-3,-(7m-3))_p=1.\qedhere
\end{align*}
\end{proof}
We remark that the condition $(11m-3,-(7m-3))_p=1$ for all $p$, is equivalent to the existence of a non-trivial solution to the Diophantine equation
\[(11m-3)x^2 -(7m-3)y^2=z^2.\]

The complete list of values $m<10^5$ for which $D(m)$ may be the Gram matrix of a $\pm 1$ is
\[73,241,757, 1057, 1561, 14281, 14521, 17557, 20881, 25441, 28057, 3673,50401,57841,78121, 97657 .
\]
\begin{research-problem}\normalfont Determine the maximal value of the determinant of a $\pm 1$ matrix of order $511$. Can the Ehlich bound be met?
\end{research-problem}

Tamura's non-existence result for maximal determinant matrices is interesting because it reminds us that the applicability of the BRC Theorem and the Bose-Connor Theorem is not limited to study matrices with entries $0$ or $1$. However, so far we only developed techniques to deal with rational matrices. In the next chapter we will extend our techniques to be able to the study of complex matrices as well.

\cleardoublepage
\chapter{Hermitian Forms and Determinant Obstructions}\label{chap-HermitianForms}

In Chapter \ref{chap-BRC} we applied the theory of quadratic forms to show the non-existence of certain matrices with entries in the set $\{-1,0,1\}$. We will extend these techniques to allow complex entries, such as roots of unity. To do this, we study Hermitian forms. Here we rely heavily on the material of Section \ref{sec-WittLemma} on Witt's theorem, Section \ref{sec-HilbertSym} on the Hilbert symbol and its properties, and Section \ref{sec-InvariantsQF} on the invariants of quadratic forms and the Hasse-Minkowski theorem. Furthermore, we assume that the reader is familiar with the basics of field theory and Galois theory, see \cite{Milne-FT} for a nice introduction to these topics. \\

First, we will present a reduction of the theory of Hermitian forms to the theory of quadratic forms due to Jacobson \cite{Jacobson-Reduction}. This reduction is very concrete and elementary, and it will show us that the obstructions arising from local invariants of quadratic forms do not appear in the theory of Hermitian forms. All obstructions to the solvability of $XX^*=M$, in the framework of Hermitian forms, come from the determinant. For rational quadratic forms, this means that the determinant of $M$ must be a square. In the Hermitian case, the answer is much more nuanced, and it depends on the behaviour of primes in field extensions.\\

For this, we will require the machinery of algebraic number theory, which we will introduce omitting most of the proofs. The interested reader can find more about this beautiful area of mathematics in the books by Ireland and Rosen \cite{Ireland-Rosen}, Marcus \cite{Marcus-NumberFields}, or  Neukirch \cite{Neukirch-ANT}.\\

We include here two novel results on the non-existence of certain complex maximal determinant matrices. The first is an extension of the non-existence results of Winterhof in \cite{Winterhof-NonexistenceButson} for Butson-type Hadamard matrices, and the second is a non-existence result for quaternary-unit Hadamard matrices which appeared in our paper \cite{QUH-paper}, in collaboration with Heikoop, Pugmire and Ó Catháin.

\section{Hermitian forms}\label{sec-HF}
We consider Hermitian forms over a subfield $K$ of $\C$ with $K\not\subset\R$. For such a field, complex conjugation induces a non-trivial field automorphism in $K$, which we denote $\tau$, i.e. $\tau(z)=\overline{z}$ for each $z\in K$. Let
\[k:=K^{\tau}=\{x\in K: \tau(x)=x\},\]
be the subfield of $K$ fixed by $\tau$. In particular, $k\subset \R$ is the maximal real subfield of $K$. Recall the following theorem of Artin
\begin{theorem}[Artin, Chapter VI Theorem 1.8. \cite{Lang-Algebra}] Let $K$ be a field, and $G$ a finite group of automorphisms of $K$. Let $k=K^{G}=\{x\in K: \sigma(x)=x,\text{ for all } \sigma\in G\}$ be the fixed field of $G$. Then $k\subset K$ is a Galois extension with Galois group $G$, and $[K:k]=|G|$.
\end{theorem}
 For $k=K^{\tau}$, Artin's Theorem implies that $|K:k|=2$. As such, $K=k[\sqrt{-d}]$ for some $d>0$ in $k$. We define the \textit{norm} mapping $N:K\rightarrow k$ by $N(\alpha)=\alpha\alpha^{\tau}\in k$. Writing $\alpha=a_0+a_1\sqrt{-d}$ for some $a_0,a_1\in k$, we find that
\[N(a_0+a_1\sqrt{-d})=(a_0+a_1\sqrt{-d})(a_0-a_1\sqrt{-d})=a_0^2 +a_1^2d.\]

We will show that in this setting, a theorem of Jacobson \cite{Jacobson-Reduction} reduces the theory of Hermitian forms over $K=k[\sqrt{-d}]$ to the theory of quadratic forms over $k$.\\

\begin{definition}\normalfont
Let $V$ be a finite dimensional $K$-vector space, an \index{form!Hermitian} Hermitian form over $K=k[\sqrt{-d}]$ is a mapping $h:V\times V\rightarrow K$ satisfying
\begin{itemize}
\item [(i)] $h(x_1+x_2,y)=h(x_1,y)+h(x_2,y)$, and $h(x,y_1+y_2)=h(x,y_1)+h(x,y_2)$,
\item[(ii)] $h(x,\alpha y)=\alpha h(x,y)$,
\item[(iii)] $h(x,y)=h(y,x)^{\tau}$,
\end{itemize}
where $\tau: K\rightarrow K$ is the non-trivial automorphism induced in $K$ by complex conjugation.
\end{definition}
Notice that $h$ is linear in its second argument, and in its first argument $h$ is additive and the scalar product is ``twisted'' by $\tau$, i.e $h(\alpha x,y)=\alpha^{\tau}h(x,y)$, such a form is called \textit{sesquilinear}\index{form!sesquilinear}. By the sesquilinearity of the form, if we fix a basis $\mathcal{B}$ of $V$ there is a unique matrix $A$ such that,
\[h(x,y)=x^*Ay,\]
and condition (iii) implies that $A=A^*$. If $B$ is the matrix of $h$ with respect to another basis $\mathcal{B}'$ of $V$, then
\[X^*A X= B,\]
where $X$ is the change of basis matrix from $\mathcal{B'}$ to $\mathcal{B}$. So, in analogy with the theory of quadratic forms, we say that two Hermitian forms represented by matrices $A$ and $B$ respectively, are \textit{equivalent} if and only if there is a non-singular matrix $X$ such that $X^*AX=B$, i.e.  the matrices $A$ and $B$ are \textit{$*$-congruent}\index{congruent matrices}. An Hermitian form $h$ (represented by $A$) is said to be \textit{regular} if and only if $\det(A)\neq 0$.\\

Suppose that the $K$-vector space $V$ is $n$-dimensional, then since $K=k[\sqrt{-d}]$, $V$ can be regarded as a $2n$-dimensional $k$-vector space. Indeed if $\{x_1,\dots,x_n\}$ is a basis for $V$ as a $K$-vector space, then $\{x_1,\dots,x_n;\gamma x_1,\dots,\gamma x_n\}$, where $\gamma=\sqrt{-d}$ is a basis for $V$ as a $k$-vector space. From a Hermitian form $h$ of degree $n$ we construct a quadratic form $q_h$ of degree $2n$ called the \textit{trace form} \index{form!quadratic!trace form} of $h$ in the following manner:
\[q_h(x)=h(x,x),\]
where $x$ is interpreted in the left-hand-side as a $2n$-vector and the right-hand side as an $n$-vector. Coordinate-wise:
\[q((a_1\dots,a_n; b_1,\dots,b_n))=h(\sum_i (a_i+b_i\gamma)x_i,\sum_j (a_j+b_j\gamma)x_j).\]
Clearly $q_h(x,x)\in k$, since $h(x,x)=h(x,x)^{\tau}$, so $q$ is well-defined as a $k$-quadratic form.

\begin{example}\normalfont Let $K=\Q[\omega]$, where $\omega$ is a complex third-root of unity, for example $\omega=e^{2\pi i/3}=\frac{-1+\sqrt{-3}}{2}$. Then $K=\Q[\sqrt{-3}]$, and $k=\Q$. Consider the $K$-hermitian form $h$ given by the matrix
\[\begin{bmatrix}
2 & \omega\\
\omega^2 & 2 
\end{bmatrix}=\begin{bmatrix}
2 & (-1+\sqrt{-3})/2\\
(-1-\sqrt{-3})/2 & 2
\end{bmatrix}\]
We can compute the trace form as follows:
\begin{align*}
q_h(x)&=h(x,x)\\
&=[x_1-\sqrt{-3}y_1,x_2-\sqrt{-3}y_2]\begin{bmatrix}
2 & (-1+\sqrt{-3})/2\\
(-1-\sqrt{-3})/2 & 2
\end{bmatrix}
\begin{bmatrix}
x_1+\sqrt{-3}y_1\\
x_2+\sqrt{-3}y_2
\end{bmatrix}\\
&=2x_1^2 - x_1x_2 - 3x_1y_2 + 6y_1^2 + 3y_1x_2 - 3y_1y_2 + 2x_2^2 + 6y_2^2.
\end{align*}
Therefore, $q_h$ is the $\Q$-quadratic form given by the $4\times 4$ matrix
\[\begin{bmatrix}
2 & 0 & -1/2 & -3/2\\
0 & 6 & 3/2 & -3/2\\
-1/2 & 3/2 & 2 & 0\\
-3/2 & -3/2 & 0 &6
\end{bmatrix}.\]
\end{example}

\begin{proposition}[cf. \cite{Jacobson-Reduction}]\normalfont \label{prop-HFQFCorrespondence} Hermitian forms on $K$ are in one-to-one correspondence with quadratic forms over $k$ satisfying the equation
\[q(x\alpha)=N(\alpha)q(x),\]
for all $\alpha\in K$. 
\end{proposition}
\begin{proof}
The trace form $q_h$ satisfies the property
\[q_h(x\alpha)=h(x\alpha,x\alpha)=\alpha\alpha^{\tau}h(x,x)=N(\alpha)q_h(x).\]
Conversely, if $q$ is a quadratic form of degree $2n$ over $k$ satisfying
\[q(x\alpha)=N(\alpha)q(x),\]
where $x\alpha$ is to be interpreted as a $2n$-dimensional vector in ${}_kV$. For the symmetric bilinear form $b_q(x,y)=\frac{1}{2}(q(x+y)-q(x)-q(y))$ associated to $q$, we have that $b_q(x\alpha,y\alpha)=N(\alpha)b_q(x,y)$. Indeed,
\[b_q(x\alpha,y\alpha)=\frac{1}{2}(q((x+y)\alpha)-q(x\alpha)-q(y\alpha))=\frac{1}{2}N(\alpha)(q(x+y)-q(x)-q(y))=N(\alpha)b_q(x,y).\]
Therefore,
\[b_q(x\alpha^{\tau},y)=\frac{1}{N(\alpha)}b_q(x\alpha^{\tau}\alpha,y\alpha)=\frac{1}{N(\alpha)}b_q(xN(\alpha),y\alpha)=b_q(x,y\alpha).\]
The identity $b_q(x,\gamma y)=-b_q(x\gamma,y)$ holds because
\[N(\gamma)b_q(x,\gamma y)=b_q(\gamma x,\gamma^ 2 y)=b_q(\gamma x,-N(\gamma) y)=-N(\gamma)b_q(\gamma x,y),\]
and $N(\gamma)>0$. Define
\[h_q(x,y):=b_q(x,y)-\frac{\gamma}{N(\gamma)}b_q(x,\gamma y),\]
then $h_q$ is an Hermitian form. Clearly $h_q$ is bilinear, and the Hermitian condition holds since
\[h_q(y,x)=b_q(y,x)-\frac{\gamma}{N(\gamma)}b_q(y,\gamma x)=b_q(x,y)+\frac{\gamma}{N(\gamma)}b_q(y\gamma,x)=b_q(x,y)-\frac{\gamma^{\tau}}{N(\gamma)}b_q(x,\gamma y)=h_q(x,y)^{\tau}.\]

To show that the correspondence is one-to-one we check that $q_{h_q}=q$ and $h_{q_h}=h$. Since $N(\gamma)b_q(x,\gamma x)=-N(\gamma)b_q(x,\gamma x)$, we find $b_q(x,\gamma x)=0$, so
\[q_{h_q}(x)=h_q(x,x)=q(x)-\frac{\gamma}{N(\gamma)}b_q(x,\gamma x)=q(x).\]
From $b_q(x,y)=\frac{1}{2}(h(x,y)+h(x,y)^{\tau})$ we have that
\begin{align*}
h_{q_h}(x,y)&=b_{q_h}(x,y)-\frac{\gamma}{N(\gamma)}b_{q_h}(x,\gamma y)\\
&=\frac{1}{2}(h(x,y)+h(x,y)^{\tau}-\frac{\gamma}{N(\gamma)}h(x,\gamma y)-\frac{\gamma}{N(\gamma)}h(x,\gamma y)^{\tau})\\
&=\frac{1}{2}(h(x,y)+h(x,y)^{\tau}-\frac{\gamma^2}{N(\gamma)}h(x,y)-\frac{\gamma\gamma^{\tau}}{N(\gamma)}h(x,y)^{\tau})\\
&=\frac{1}{2}(h(x,y)+h(x,y)^{\tau}+h(x,y)-h(x,y)^{\tau})\\
&=h(x,y).\qedhere
\end{align*}
\end{proof}

For example, if $h(x,y)=x^*y=x^* I y$ is the Hermitian form represented by the identity matrix, then for each basis vector $x_i$, $q_h(x_i)=1$ and $q(\gamma x_i)=\gamma^{\tau}\gamma=N(\gamma)=d$. Therefore, $q_h$ is the quadratic form $\langle 1,\dots, 1; d,\dots, d\rangle$ where the $1$s and $d$s appear exactly $n$ times.\\

In a complete analogy to the case of quadratic forms, Hermitian forms can be polarised by a series of (Hermitian) elementary row and column operations. The entries of a polarised Hermitian matrix are necessarily elements of the field $k$. So if $h$ is represented by the diagonal matrix $\diag(a_1,\dots,a_n)$, where each $a_i\in k$, then the corresponding trace form $q_h$ is given by $\langle a_1,\dots,a_n; N(\gamma)a_1,\dots, N(\gamma)a_n\rangle=\langle a_1,\dots,a_n; da_1,\dots,da_n\rangle.$ In particular, we find that if $h$ is a regular Hermitian form, then $q_h$ is a regular quadratic form.\\

The following theorem of Jacobson \cite{Jacobson-Reduction} shows that the one-to-one correspondence of Proposition \ref{prop-HFQFCorrespondence} respects equivalence of forms. We reproduce Jacobson's proof below:

\begin{theorem}[Jacobson's reduction, \cite{Jacobson-Reduction}]\label{thm-JacobsonReduction}\index{Jacobson's reduction}
Two regular hermitian forms $h$ and $h'$ are equivalent as $K$-hermitian forms if and only if their trace forms $q_h$ and $q_{h'}$ are equivalent as $k$-quadratic forms.
\end{theorem}
\begin{proof}
We prove this result by induction on the dimension $n$ of the Hermitian forms $h$ and $h'$. Two Hermitian forms $h$ and $h'$ of dimension $1$ on $K$ are given by scalars $a$ and $b$ in $k$, and their corresponding trace forms are $\langle a,da\rangle$ and $\langle b,db\rangle$. We show that $(a)\simeq (b)$ as $K$-Hermitian forms if and only if $\langle a,ad\rangle\simeq \langle b,bd\rangle$ as $k$-quadratic forms. The Hermitian equivalence of $(a)$ and $(b)$ is equivalent to the existence of a scalar $\lambda\in K^{\times}$ such that $\lambda^{\tau}a\lambda=b$. Writing $\lambda=x+\gamma y$ with $x,y\in k$, this is equivalent to
\[(x^2+dy^2)a=b.\]
But then, we have the congruence
\[\begin{bmatrix}
x & y\\
-dy & x
\end{bmatrix}
\begin{bmatrix}
a & 0\\
0 & ad
\end{bmatrix}
\begin{bmatrix}
x & -dy\\
y & x
\end{bmatrix}
=
\begin{bmatrix}
(x^2+dy^2)a & 0\\
0 & (x^2+dy^2)ad
\end{bmatrix}
=
\begin{bmatrix}
b & 0\\
0 & bd
\end{bmatrix}.
\]
Conversely, from the equivalence $\langle a,ad\rangle\simeq \langle b,bd\rangle$, we find a congruence
\[
\begin{bmatrix}
b & 0\\
0 & bd
\end{bmatrix}
=
\begin{bmatrix}
x & y\\
z & t
\end{bmatrix}
\begin{bmatrix}
a & 0\\
0 & ad
\end{bmatrix}
\begin{bmatrix}
x & z\\
y & t
\end{bmatrix}
=\begin{bmatrix}
x^2 a + y^2 a d & x z a + y t a d\\
x z a + y t a d & z^2 a + t^2 a d
\end{bmatrix},
\]
which in turn implies the existence of a solution $x,y\in k$ to the equation $(x^2 + dy^2)a=b$. This establishes the base case. Now, assume that any two Hermitian forms of dimension $<n$ are equivalent if and only if their corresponding trace forms are equivalent:\\

 Let ${}_{K}V$ and ${}_{K}W$ be $n$-dimensional $K$-vector spaces, denote by ${}_kV$ and ${}_kW$ these vector spaces regarded as $2n$-dimensional $k$-vector spaces. If $h:{}_{K}V\times {}_{K}V\rightarrow K$, and $h':{}_{K}W\times {}_{K}W\rightarrow K$ are equivalent Hermitian forms, then there is an invertible $K$-linear mapping $\phi:{}_{K}V\rightarrow {}_{K}W$ such that 
\[h'(\phi(x),\phi(y))=h(x,y).\]
It is easy to check that the induced $k$-linear mapping $\phi:{}_kV\rightarrow{}_kW$ is invertible as well, and 
\[q_{h'}(\phi(x))=h'(\phi(x),\phi(x))=h(x,x)=q_h(x),\]
which implies that $q_h$ and $q_{h'}$ are equivalent as quadratic forms over $k$.\\

Conversely, suppose that $q_h$ and $q_{h'}$ are equivalent as quadratic forms over $k$, i.e. there is a linear invertible mapping $\sigma:{}_kV\rightarrow {}_kV$, such that $q_{h'}(x)=q_{h}(\sigma(x))$, for all $x\in {}_k V$. Since $h$ is regular, then $q_h$ is regular, so there is a vector $x_0\in {}_kV$ such that 
\[q_h(x_0)=q_{h'}(\sigma(x_0))=\alpha\neq 0,\]
for some $\alpha\in k-\{0\}$. Therefore, $h(x_0,x_0)=h'(\sigma(x_0),\sigma(x_0))\neq 0$. Let ${}_KW$ and ${}_KW'$ be the $K$-vector spaces orthogonal to $x_0$ and $\sigma(x_0)$ relative to $h$ and $h'$, respectively. Let ${}_KU=\Span\{x_0\}$ and ${}_KU'=\Span\{\sigma(x_0)\}$. If $h(x,y)=0$ then $b_{q_h}(x,y)=\frac{1}{2}(h(x,y)+h(x,y)^{\tau})=0$, so we have that ${}_V={}_kU\oplus{}_kW$, and ${}_kV'={}_kU'\oplus {}_kW'$, and this direct sum is orthogonal. Now, the matrix of $q_h$ and $q_{h'}$ in ${}_k U$ and ${}_k U'$ is $\diag(\alpha,d\alpha)$, and by assumption $q_h$ and $q_{h'}$ are equivalent over ${}_kV$ and ${}_kV'$. Hence Witt's cancellation Lemma (see Theorem \ref{thm-WittThm}) implies that the restrictions of $q_h$ and $q_{h'}$ to ${}_kW$ and ${}_kW'$ respectively are equivalent. By our induction hypothesis, this implies that the restrictions of $h$ and $h'$ to ${}_KW$ and ${}_KW'$ are equivalent, and since $h(x,x)=h'(\sigma(x),\sigma(x))=\alpha\neq 0$ the forms $h$ and $h'$ are also equivalent over ${}_KU$ and ${}_KU'$. The result then follows, since the sums ${}_KV={}_KU\oplus {}_KW$ and ${}_KV'={}_KU'\oplus {}_KW'$ are orthogonal with respect to $h$ and $h'$.\qedhere
\end{proof}

\begin{remark} \normalfont Jacobson's reduction holds true in more situations. For example, let $q$ be an odd prime power. If $k=\F_q$ is the finite field of $q$ elements and $K=\F_{q^2}$, then Hermitian forms can be defined using the involutory automorphism $\tau:\F_{q^2}\rightarrow\F_{q^2}$ given by $\tau(x)=x^q$ for all $x\in \F_{q^2}$. Then, since all local-symbols $(a,b)_{\F_q}$ are trivial, we know that the trace forms $\langle a,ad\rangle$ and $\langle b,bd\rangle$ are equivalent, as they have the same discriminant. Witt's Theorem holds for general fields of characteristic $\neq 2$, so the proof of Jacobson's reduction remains true in the case of finite fields.
\end{remark}
The theory of quadratic forms over a general subfield of $\C$ can be quite complicated, so we will assume that $k$ is a number field, i.e. a \textit{finite-degree} extension of $\Q$. So for example, we do not consider fields like $\Q(\pi)$, where $\pi$ is a transcendental real number.\\

The theory of quadratic forms over number fields is very similar to that of the rationals. We can define a ``global'' symbol $(a,b)_k$, which will take value $1$ if and only if $(a,b)_{\mathfrak{p}}=1$ at all ``local'' symbols (this is the Hasse local-global principle over number fields, see Chapter VI, 66:3 and 66:4 in O'Meara \cite{OMeara-QuadraticForms}). The local symbols correspond to equations over completions $k_{\mathfrak{p}}$ of $k$, so there is a symbol for every \index{place} place on $k$ (recall that a place is an equivalence class of absolute values on $k$, Definition \ref{def-Place}). These completions are known as \textit{local fields}, \index{field!local field} and in general over a local field the bilinearity of the Hilbert symbol holds. Notice that all our previous arguments held formally for a bilinear symbol over an arbitrary field $k$ with $\Char(k)\neq 2$, so they are still true for number fields $k$.\\

Finally we mention that, in analogy to the rational case, the non-archimedean places of $k$ are in one-to-one correspondence with each non-zero prime ideal $\mathfrak{p}$ in the \textit{ring of integers} \index{ring of integers} $\mathcal{O}_k$ of $k$, i.e. the ring consisting of all elements of $k$ satisfying a monic equation with coefficients in $\Z$. Since $k$ is the fixed field of $K$ under the automorphism induced by complex conjugation, all embeddings of $k$ into $\C$ are real. Hence, the archimedean places correspond to each possible embedding of $k$ into $\R$.

\begin{example}\normalfont Let $k=\Q[t]/(t^2-2)\simeq \Q[\sqrt{2}]$, then there are two archimedean places: one for each embedding of $k$ into $\R$. Namely,
\begin{align*}
|x+ty|_1&=|x+\sqrt{2}y|,\text{ and }\\
|x+ty|_2&=|x-\sqrt{2}y|,
\end{align*}
where $|\cdot|$ denotes the usual absolute value in $\R$. The rational prime $7$ splits in $\Q[\sqrt{2}]$ as $7=(3+\sqrt{2})(3-\sqrt{2})$. It is easy to check that the ring of integers  of $k=\Q[\sqrt{2}]$ is $\mathcal{O}_k=\Z[\sqrt{2}]$. Given an arbitrary element $a+b\sqrt{2}\in \Z[\sqrt{2}]$, we have that $a+b\sqrt{2}\mp b(3\pm \sqrt{2})=a\mp 3 b$, and from $7=(3+\sqrt{2})(3-\sqrt{2})$, it follows that $\Z[\sqrt{2}]/(3\pm\sqrt{2})\simeq \Z/7\Z$. In particular, $(3\pm\sqrt{2})$ is a prime ideal. There are then two non-archimedean places $|\cdot|_{(3+\sqrt{2})}$ and $|\cdot|_{(3-\sqrt{2})}$, instead of the single rational place associated to the prime $7$. The rational prime $p=5$ stays irreducible when considered as an element of $\Q[\sqrt{2}]$ (this can be seen by taking the norm of $\Q[\sqrt{2}]$ over $\Q$ in an equation $5=ab$ with $a$ or $b$ non-units), and the principal ideal $(5)$ is prime. Therefore, there is exactly one archimedean place associated to the prime $5$.
\end{example}

The following result is a consequence of Jacobson's reduction Theorem \ref{thm-JacobsonReduction} (see Chapter 10, Remark 1.4 of Scharlau \cite{Scharlau-QuadraticHermitian}). Combining this with the Hasse-Minkowski Theorem for number fields one can decide equivalence of any pair of Hermitian forms over $K$.

\begin{proposition}[cf. Scharlau, Chapter 10, Remark 1.4. \cite{Scharlau-QuadraticHermitian}] \normalfont \label{prop-TraceFormInvariants} Let $h$ be an Hermitian form of order $n$ over $K=k[\sqrt{-d}]$. Then the following hold for $q_h$,
\[\delta(q_h)=d^{n},\text{ and } \varepsilon_{\mathfrak{p}}(q_h)=(d,-1)_{\mathfrak{p}}^{n\choose 2}(-d,\det(M))_{\mathfrak{p}},\]
for all places $\mathfrak{p}$ of $k$.
\end{proposition}
\begin{proof} After polarisation, we may assume that $h$ is represented by a diagonal matrix with coefficients in $k$, say $M=\diag(a_1,\dots,a_n)$. Then $q_h$ is represented by $\diag(a_1,\dots,a_n;da_1,\dots,da_n)=M\oplus dM$, which implies that 
\[\delta(q_h)=\det(M\oplus dM)=\det(M)\det(dM)=d^n\det(M)^2\equiv d^n \text{ in } k^{\times}/(k^{\times})^2.\]
The Hilbert symbol at any of the completions of $k$ is bilinear, hence we can apply Lemma \ref{lemma-DirectSum-Symbol} and Lemma \ref{lemma-ScaledInvariant},
\begin{align*}
\varepsilon(q_h)&=\varepsilon(M\oplus dM)\\
&=\varepsilon(M)\varepsilon(dM)(\det(M),d^n\det(M))\\
&=\varepsilon(M)(d,-1)^{n\choose 2}(d,\det(M))^{n-1}\varepsilon(M)(\det(M),d^n\det(M))\\
&=(d,-1)^{n\choose 2}(d^{n-1},\det(M))(\det(M),d\det(M))(\det(M),d^{n-1})\\
&=(d,-1)^{n\choose 2}(\det(M),d\det(M))\\
&=(d,-1)^{n\choose 2}(\det(M),-d).\qedhere
\end{align*}
\end{proof}
The notion of positive-definiteness \index{matrix!positive-definite} is not well-defined for a matrix over an abstract number field, since this depends on the embedding in $\R$ that we choose. For example, if $k=\Q[t]/(t^2-2)\simeq \Q[\sqrt{2}]$, then
\[\begin{bmatrix}
1+t & 0\\
0 & 1+t
\end{bmatrix},
\]
is positive-definite with the embedding $t\mapsto +\sqrt{2}$, but it is negative-definite with the embedding $t\mapsto-\sqrt{2}$ since $1-\sqrt{2}\approx -0.4142\ldots <1$. Hence, we assume to have a fixed embedding of $k$ into $\R$, for which our matrix is positive-definite.
\begin{theorem}\label{thm-GramNorm}
Let $M$ be an Hermitian positive-definite matrix with coefficients in $K=k[\sqrt{-d}]\subset\C$. Then, there is a matrix $X\in \GL_n(K)$ such that $XX^*=M$ if and only if $\det(M)$ is a norm, i.e. $\det(M)\in N(K^{\times})$.
\end{theorem}
\begin{proof}
If $M=X^*X$ for some $X\in\GL_n(K)$, then the Hermitian form $h:=h_M$ represented by $M$ is equivalent to the Hermitian form $h_I$ represented by $I$. Let $q$ and $q_I$ be the trace forms of $h$ and $I$ respectively. Then $q$ and $q_I$ are equivalent as $k$-quadratic forms. Therefore, we must have that $\delta(q)=\delta(q_I)$ and $\varepsilon(q)=\varepsilon(q_I)$. The first condition is vacuous by Proposition \ref{prop-TraceFormInvariants}, and since $\det(I)=1$ the condition $\varepsilon_{\mathfrak{p}}(q)=\varepsilon_\mathfrak{p}(q_I)$ reduces to
\[(-d,\det(M))_{\mathfrak{p}}=1,\]
for all places $\mathfrak{p}$ of $k$. This is equivalent to $(-d,\det(M))_k=1$, i.e. to the existence of a non-trivial solution on $k$ to
\[\det(M)x^2-dy^2=z^2.\]
Since $-d$ is not a square in $k$, this is equivalent to $\det(M)=(z/x)^2+d(y/x)^2=N((z/x)+\sqrt{-d}(y/x))\in N(K^{\times})$.\qedhere
\end{proof}
\begin{remark}\normalfont
Hermitian forms had been studied by Brock in \cite{Brock}, in the context of combinatorics. In this paper, the author extracts conditions for the solvability of certain Hermitian Gram matrix equations. One of these conditions coincides with the one in Theorem \ref{thm-GramNorm}, but some additional restrictions are listed. Our characterisation shows that those additional conditions in \cite{Brock} are redundant.
\end{remark}

With this we can determine the Hermitian Gram matrices over the cyclotomic fields of degree $2$ over $\Q$:
\begin{corollary}\label{cor-Norm3}\normalfont
Let $M$ be an Hermitian positive-definite matrix with entries in $K=\Q[\omega]$, where $\omega=\exp(2\pi i/3)$. Write $\det(M)=a^23^rm$ where $a\in \Q^{\times}$, $m\in \Z$ is square-free and $3\nmid m$. Then there is a matrix $X\in \GL_n(K)$ such that $XX^*=M$ if and only if every odd prime factor $p$ of $m$ satisfies $p\equiv 1\pmod{3}$.
\end{corollary}
\begin{proof}
 If $K=\Q[\omega]=\Q[\sqrt{-3}]$, then $k=\Q$. Therefore a positive-definite matrix $M$ with coefficients in $K$ satisfies $M=XX^*$ if and only if $\det(M)\in N(K^{\times})$. We saw this is equivalent to 
\[(\det(M),-3)_p=1\]
for all places $p$ of $\Q$. Under the hypothesis of the statement, let $p$ be an odd prime with $p\mid m$ then
\[(\det(M),-3)_p=(3^r\cdot p,-3)_p=(3^r,-3)_p(p,-3)_p.\]
Now, since $p\neq 3$ we have that $(3^r,-3)_p=1$, and
\[(p,-3)_p=\legendre{-3}{p}=\legendre{-1}{p}\legendre{3}{p}.\]
By quadratic reciprocity, we have that
\[\legendre{3}{p}\legendre{p}{3}=(-1)^{(p-1)/2}=\legendre{-1}{p}.\]
From which it follows that $(\det(M),-3)_p=\legendre{p}{3}=1$ if and only if $p$ is a square residue modulo $3$, i.e. $p\equiv 1\pmod{3}$.
Finally we evaluate $(\det(M),-3)_3$, we have
\begin{align*}
(3^rm-3)_3 &=(3^r,-3)_3(m,-3)_3\\
&=(3^r,-3)_3(m,-1)_3(m,3)_3.
\end{align*}
Since $m$ and $-1$ are coprime to $3$ it follows that $(m,-1)_3=1$, and from the relation $(a,b)=(a,-ab)_3$ we have that $(3^r,-3)_3=(3^r,3^{r+1})_3=(3,3)^{r(r+1)}$. The integer $r(r+1)$ is always even, so $(3^r,-3)_3=1$. It follows that 
\[(\det(M),-3)_3=(m,3)_3,\]
but from our discussion above we have that every odd prime factor of $M$ is a square modulo $3$, hence $m$ is a square modulo $3$ and $(\det(M),-3)_3=1$. From $\det(M)>0$ it follows that $(\det(M),-3)_{\infty}=1$, so Hilbert reciprocity implies that $(\det(M),-3)_2=1$ as well.\qedhere
\end{proof}
\begin{corollary}\normalfont\label{cor-Norm4} Let $M$ be an Hermitian positive-definite matrix with entries in $K=\Q[i]$, where $i=\sqrt{-1}$. Write $\det(M)=a^22^rm$ where $a\in \Q^{\times}$, and $m\in \Z$ is square-free and odd. Then there is a matrix $X\in \GL_n(K)$ such that $XX^*=M$ if and only if every prime factor $p$ of $m$ satisfies $p\equiv 1\pmod{4}$.
\end{corollary}
\begin{proof}
If $K=\Q[i]=\Q[\sqrt{-1}]$, then $k=\Q$, and $XX^*=M$ if and only if
\[(\det(M),-1)_p=1\]
for all places $p$ of $\Q$. If $p$ is an odd prime, and $p\mid \det(M)$, then
\[(\det(M),-1)_p=(p,-1)_p=\legendre{-1}{p}=1\text{ if and only if } p\equiv 1\pmod{4}.\]
Since $\det(M)>0$, we have that $(\det(M),-1)_{\infty}=1$ and Hilbert reciprocity implies that $(\det(M),-1)_2=1$.\qedhere
\end{proof}

The conclusion of this section is that the theory of equivalence of Hermitian forms reduces to the study of the determinant. The following section is devoted to studying the solvability of norm equations involving determinants.

\section{Splitting of prime ideals}\label{sec-SplittingIdeals}
When the degree $|k:\Q|$ in the tower of field extensions $\Q\subseteq k\subset K$ is greater than $1$, we may not have at our disposal formulas for the local Hilbert symbols of $k$ in terms of Legendre symbols. To deal with this case, we will study what is known as the \textit{prime ideal decomposition} of certain elements of $k$. We recall below some concepts and facts from ring theory and algebraic number theory. For an introduction to basic algebraic number theory the reader can consult the following \cite{ Ireland-Rosen, Kato-FermatDream, Marcus-NumberFields, Neukirch-ANT}.\\ 

Dedekind introduced the theory of \textit{ideals} to recover in some sense the property of unique prime factorisation, which fails over ring extensions of $\Z$. For example in $\Z[\sqrt{-5}]$ the element $6$ does not factor uniquely into primes elements, since

\[2\cdot 3 = 6 = (1+\sqrt{-5})(1-\sqrt{-5}).\]
  Recall that an ideal \index{ideal} $A$ of a commutative ring $R$ is a subgroup of the additive group $(R,+)$ with the property that if $r\in R$ and $x\in A$, then $rx\in A$. The ideal
\[(x)R=\{rx: r\in R\},\]
is called the \textit{principal ideal}\index{ideal!principal} generated by $x$. More generally, given elements $a_1,\dots,a_m$ in $R$, the ideal \textit{generated} by the $a_i$ is 
\[(a_1,\dots,a_m)R=\{r_1a_1+\dots+r_ma_m:r_i\in R\}.\]
 If the ring $R$ is clear from the context, then the principal ideal $(x)R$ is denoted $(x)$, and $(a_1,\dots,a_m)R$ is denoted $(a_1,\dots,a_m)$. An ideal $A$ is said to be \textit{prime}\index{ideal!prime} if $ab\in A$ implies that $a\in A$ or $b\in A$, equivalently $A$ is prime if and only if the quotient ring $R/A$ is a domain. If $R$ is a domain, then $(0)$ is a prime ideal of $R$. In what follows we will use the term prime ideal to refer to non-zero prime ideals, and we will denote these using Gothic letters: $\mathfrak{p},\mathfrak{q}$ etc.\\
 
 As we mentioned above, $6$ does not factor uniquely in $\Z[\sqrt{-5}]$, but the ideal $(6)$ factors uniquely into prime ideals in $\Z[\sqrt{-5}]$ as
\[(6)=(2,1+\sqrt{-5})^2(3,1+\sqrt{-5})(3,1-\sqrt{-5}).\]
Therefore, ideals are the right concept to work with to study factorisation over number fields.\\
 
 Let $K$ be a number field, recall that the ring of integers \index{ring of integers} $\mathcal{O}_K$ of $K$ is the set of elements $\alpha\in K$ that are roots of a \textit{monic} polynomial with coefficients on $\Z$. For example, the ring of integers of $\Q$ is precisely $\Z$. Because of this, $\mathcal{O}_K$ plays the analogue role in $K$ as $\Z$ in $\Q$.
 
\begin{example}\normalfont Over $K=\Q[\sqrt{-3}]$ every element of the type $a+b\sqrt{-3}$ is in $\mathcal{O}_K$, but the element $\omega=\frac{-1}{2}+\frac{\sqrt{-3}}{2}$ is also in $\mathcal{O}_K$ since 
\[\omega^2+\omega+1=0.\]
In particular, the ring of integers of a simple extension $\Q[\alpha]$ is not always equal to $\Z[\alpha]$.
\end{example}
For quadratic extensions, rings of integers are characterised as follows:
\begin{proposition}[\cite{Marcus-NumberFields}, Theorem 1, Corollary 2]
 Let $K=\Q[\sqrt{d}]$ with $d$ a square-free integer, then
\[\mathcal{O}_K=\begin{cases}
\Z[\sqrt{d}] & \text{if } d\equiv 2,3\pmod{4}\\
\Z[\frac{1+\sqrt{d}}{2}] & \text{if } d\equiv 1\pmod{4}
\end{cases},
\]
\end{proposition}

Another important family of extensions are \textit{cyclotomic extensions},\index{field!cyclotomic} these are the fields $\Q[\zeta_n]$ obtained from $\Q$ by appending $\zeta_n$, a primitive $n$-th root of unity. For cyclotomic extensions we have the following:
\begin{proposition}[\cite{Marcus-NumberFields}, Theorem 12, Corollary 2]\normalfont The ring of integers of $\Q[\zeta_n]$ for $\zeta_n$ a primitive $n$-th root of unity is $\Z[\zeta_n]$.
\end{proposition}
 
\begin{definition}\normalfont A \textit{Dedekind domain} \index{Dedekind domain} is an integral domain $R$ where every non-zero prime ideal $I$ can be written in a unique way as a product of prime ideals.
\end{definition} 
We remark that in most textbooks on algebraic number theory, the definition of Dedekind domain is different than the one given above. However, both definitions are equivalent (see Theorem 10.6 in \cite{Jacobson-BasicAlgebraII}).

\begin{theorem}[cf. \cite{Marcus-NumberFields}]\label{thm-RofI-DD}The ring of integers of a number field is a Dedekind domain.
\end{theorem}

 We give a brief summary without proofs of some general results on Dedekind domains:\\

Given ideals $A$ and $B$ in a commutative ring $R$, we say that $A$ \textit{divides} $B$ if and only if $AC=B$ for some ideal $C$ of $R$. If $A$ divides $B$ we write $A\mid B$.

\begin{proposition}[\cite{Marcus-NumberFields} Theorem 19]\normalfont Let $K\subset L$ be an extension of number fields. If $\mathfrak{p}$ is a prime ideal in $\mathcal{O}_K$ and $\mathfrak{P}$ is a prime ideal in $\mathcal{O}_L$ then the following are equivalent
\begin{itemize}
\item[(i)] $\mathfrak{P}\mid \mathfrak{p}\mathcal{O}_L$,
\item[(ii)] $\mathfrak{P}\supset \mathfrak{p}\mathcal{O}_L$,
\item[(iii)] $\mathfrak{P}\supset \mathfrak{p}$,
\item[(iv)] $\mathfrak{P}\cap \mathcal{O}_K=\mathfrak{p}$,
\item[(v)] $\mathfrak{P}\cap K=\mathfrak{p}$.
\end{itemize}
\end{proposition}
If any of the equivalent conditions of the theorem above hold for primes $\mathfrak{P}$ and $\mathfrak{p}$ we say that $\mathfrak{P}$ \textit{ lies above } $\mathfrak{p}$ or that $\mathfrak{p}$ \textit{ lies under } $\mathfrak{P}$.

\begin{theorem}[cf. \cite{Marcus-NumberFields}, Theorem 20]\label{thm-UniquePrimeAbove}
Let $K\subset L$ be an extension of number fields. If $\mathfrak{P}$ is a prime ideal of $\mathcal{O}_L$, then there is a unique prime ideal $\mathfrak{p}$ of $\mathcal{O}_K$ such that 
\[\mathfrak{p}\subset\mathfrak{P}.\] 
Conversely, if $\mathfrak{p}$ is a prime ideal of $\mathcal{O}_K$ then there is at least one prime ideal $\mathfrak{P}$ of $\mathcal{O}_L$ such that $\mathfrak{p}\subset\mathfrak{P}$.
\end{theorem}

In the case where $K$ is the \textit{splitting field} of an irreducible polynomial in $\Q$ we have the following behaviour of primes $p$ in $\Z$

\begin{itemize}
\item $(p):=p\mathcal{O}_K$ is a prime ideal in $\mathcal{O}_K$. In which case we say $p$ is \textit{inert}.\index{ideal!prime!inert}
\item $(p)$ decomposes as $(p)=\prod_{i=1}^r\mathfrak{p}_i^e$, where $e\geq 1$ and the $\mathfrak{p}_i$ are \textit{distinct} prime ideals of $\mathcal{O}_K$. If $e>1$ we say that $p$ is \textit{ramified}, otherwise $p$ \textit{splits}.\index{ideal!prime!ramified}\index{ideal!prime!split}
\end{itemize}
Over more general number fields, we may find different multiplicities for each prime factor. But for our purposes this restricted scenario is enough. We mention that in a general number field, a rational prime ramifies in $\mathcal{O}_K$ only if it divides the \textit{discriminant}\index{field!discriminant of a} of $K$. It can be shown that for a number field $K$, the abelian group $(\mathcal{O}_K,+)$ is free of rank $n$, where $n=[K:\Q]$ is the degree of the extension $\Q\subset K$. Therefore there are algebraic integers $\alpha_i$ such that
\[(\mathcal{O}_K,+)\simeq \alpha_1\Z\oplus\dots\oplus\alpha_n\Z.\]
The discriminant of $K$ is then defined as
\[\disc(K)=\det(\sigma_i(\alpha_j))^2,\]
where  $\sigma_1,\dots,\sigma_n$ are the embeddings of $K$ into $\C$. The discriminant is independent of the choice of $\alpha_i$.

\begin{example}\normalfont
Let $K=\Q[\sqrt{d}]$, then if $d\equiv 2,3\pmod{4}$ then $\mathcal{O}_K=\Z[\sqrt{d}]$, hence
\[(\mathcal{O}_K,+)\simeq \Z\oplus \sqrt{d}\Z.\]
Then
\[\disc(K)=\left(\det\begin{bmatrix}
1 & \sqrt{d}\\
1 & -\sqrt{d}
\end{bmatrix}\right)^2
=(-2\sqrt{d})^2=4d.\]
If instead, $d\equiv 1\pmod{4}$ then $\mathcal{O}_K=\Z[\frac{1+\sqrt{d}}{2}]$, and 
\[\disc(K)=\left(\det\begin{bmatrix}
1 & \frac{1+\sqrt{d}}{2}\\
1 & \frac{1-\sqrt{d}}{2}
\end{bmatrix}\right)^2=(-\sqrt{d})^2=d.\]
Therefore
\[\disc(\Q[\sqrt{d}])=
\begin{cases}
4d & \text{ if } d\equiv 2,3\pmod{4}\\
d & \text{ if } d\equiv 1 \pmod{4}
\end{cases}.
\]
\end{example}

Over quadratic fields, the behaviour of primes is controlled by the Legendre symbol.
\begin{theorem}[\cite{Marcus-NumberFields}, Theorem 25]\label{thm-QuadraticSplitting}
Let $p$ be an odd prime, and $m$ a square-free integer. Then over the ring of integers of $K=\Q[\sqrt{m}]$ we have
\begin{itemize}
\item[(i)] If $p\mid m$, then $p$ ramifies as $(p)\mathcal{O}_K=(p,\sqrt{m})^2$.
\item[(ii)] If $p\nmid m$ and $\legendre{m}{p}=-1$, then $p$ is inert.
\item[(iii)] If $p\nmid m$ and $\legendre{m}{p}=1$, then $p$ splits completely as
\[(p)\mathcal{O}_K=(p,n+\sqrt{m})(p,n-\sqrt{m}),\]
where $m\equiv n^2 \pmod{p}$.
\end{itemize}
\end{theorem}

Notice that the theorem above is a generalisation of Corollary \ref{cor-Norm3} and Corollary \ref{cor-Norm4}. For general splitting fields of irreducible polynomials we have the following powerful result:
\begin{theorem}[cf. Theorems 21, 23 and 24 in \cite{Marcus-NumberFields}]\label{thm-SplittingTheorem}
Let $K$ be the splitting field of an irreducible polynomial in $\Q[x]$. Let $n=[K:\Q]$ be the degree of the extension $\Q\subset K$. If a rational prime $q$ is ramified in $\mathcal{O}_K$ , then $q\mid \disc(K)$. And if $q\nmid \disc(K)$, then we have
\[(q)\mathcal{O}_K = \mathfrak{q}_1\dots\mathfrak{q}_r,\]
where $r\mid n$. Furthermore, the action of the Galois group $\Gal(K/\Q)$ on $\{\mathfrak{q}_1,\dots,\mathfrak{q}_r\}$ is transitive.
\end{theorem}

 The proposition below contains the main tool to determine necessary conditions for $\det(M)$ to be a norm.
 
 \begin{proposition}\normalfont \label{prop-DetNormCondition} Let $K$ be a number field, and $\alpha\in \Z$ be an integer. Suppose that $\alpha$ is a norm in $K$, i.e. $\alpha=\beta\beta^{\tau}$ for some $\beta\in\mathcal{O}_K$. If $\q\subset\mathcal{O}_K$ is a prime ideal fixed by complex conjugation, then $\q$ must divide $(\alpha)\mathcal{O}_K$ with even multiplicity.
 \end{proposition}
 \begin{proof}
 Since $\q$ divides $(\alpha)$ we have that $(\alpha)=\q^e\cdot A$ for some ideal $A\subset\mathcal{O}_K$ not divisible by $\q$. The equation $\alpha=\beta\beta^{\tau}$ in $\mathcal{O}_K$ implies that
 \[(\alpha)=(\beta\beta^{\tau})=(\beta)(\beta)^{\tau}.\]
 By Theorem \ref{thm-RofI-DD} the ring $\mathcal{O}_K$ is a Dedekind domain, and prime ideal factorisations are unique. This implies that $\q$ divides $(\beta)$ or $(\beta)^{\tau}$. Suppose that $\q$ divides $(\beta)$, then
 \[(\beta)=\q^{\ell}B,\]
 where $B\subset\mathcal{O}_K$ is an ideal, not divisible by $\q$. Applying the complex conjugation automorphism $\tau$ we find
 \[(\beta)^{\tau}=(\q^{\ell}B)^{\tau}=\q^{\ell}B^{\tau}.\]
The prime $\q$ does not divide $B^{\tau}$, otherwise applying $\tau$ we would find that $\q$ divides $B$. Therefore $\q$ divides both $(\beta)$ and $(\beta)^{\tau}$  with multiplicity $\ell$. The uniqueness of prime ideal factorisations over $\mathcal{O}_K$ then implies that $e=2\ell$, and so $\q$ divides $(\alpha)$ with even multiplicity.\qedhere
 \end{proof}

  \begin{example}\normalfont Over the field $K=\Q[\sqrt{-3}]$, the integer $10$ cannot be written as $10=\beta\beta^{\tau}$ for $\beta\in \Z[\frac{1+\sqrt{-3}}{2}]=\mathcal{O}_K$. From Theorem \ref{thm-QuadraticSplitting}, we know that $(5)$ is inert in $\mathcal{O}_K$ since
 \[\legendre{-3}{5}=\legendre{-1}{5}\legendre{3}{5}=-1\cdot +1=-1.\]
 Therefore $(5)$ is a prime ideal in $\mathcal{O}_K$, and clearly $(5)$ is fixed by complex conjugation. However $(5)$ appears with multiplicity $1$ in the decomposition of $(10)$, which is an odd multiplicity. The claim follows from Proposition \ref{prop-DetNormCondition}.
 \end{example}
 
 Our strategy will be to apply the proposition above with $\alpha=\det(M)$, which satisfies $\det(M)=\det(X)\det(X)^{\tau}$ under the assumption that $M=XX^*$. We illustrate this with some example applications.
 
 \section{Non-existence of Butson-type Hadamard matrices}\label{sec-NonExButson}

\begin{definition}\normalfont A \textit{complex Hadamard matrix} is an $n\times n$ matrix $H$ with entries in the complex unit disk such that $HH^*=nI_n$.\index{Hadamard matrix}
\end{definition}
 \begin{definition}\normalfont\label{def-ButsonMatrix} A \textit{Butson-type Hadamard matrix}, or \textit{Butson matrix}, is a complex Hadamard matrix with all of its entries consisting of complex roots of unity. We denote by $\BH(n,k)$ the set of Butson matrices with entries in the $k$-th roots of unity.\index{Hadamard matrix!Butson-type}
\end{definition}

A well-known condition for non-existence of Butson-type Hadamard matrices can be found in Winterhof's paper \cite{Winterhof-NonexistenceButson}. In this paper, the author obtained non-existence conditions by examining norm equations over the field extension $\Q[\zeta_p]$, where $p$ is an odd prime. The main idea of Winterhof's proof is to reduce norm equations on the extension $\Q[\zeta_p+\zeta_p^{-1}]\subset\Q[\zeta_p]$ to norm equations on $\Q\subset\Q[\sqrt{-p}]$ through the following lemma:

\begin{lemma}[Winterhof \cite{Winterhof-NonexistenceButson}] \label{lemma-WinterhofReduction} Let $p$ be a prime $p \equiv 3 \pmod{4}$. Suppose $n=p^{\ell}a^2m$ is odd, that  $p\nmid m$ and that $m$ is square-free. If there exists a $\BH(n,p^f)$ or a $\BH(n,2p^f)$ then there exists an $x\in \Q[\sqrt{-p}]$ such that
$N_{\Q[\sqrt{-p}]}(x) = m.$
\end{lemma}

The norms over quadratic extensions are characterised in Theorem \ref{thm-QuadraticSplitting}. Hence, Winterhof concludes

\begin{theorem}[Winterhof \cite{Winterhof-NonexistenceButson}]
Let $p \equiv 3 \pmod {4}$ be a prime. Suppose $n=p^{\ell}a^2 m$ is odd, where $p\nmid m$
and $m$ is square-free. Then no $\BH(n,p^f)$ and no $\BH(n, 2p^f)$ exist if $\legendre{q}{p}=-1$ for some
prime $q\mid m$.
\end{theorem}

Winterhof's use of Lemma \ref{lemma-WinterhofReduction} has the advantage of giving an elementary proof of non-existence of BH matrices, and we will later see (Proposition \ref{prop-SelfConjp3}) that no obstructions arising from the splitting of primes are lost by using this reduction. However, Lemma \ref{lemma-WinterhofReduction} does not apply to the case $p\equiv 1\pmod{4}$. We extended  Winterhof's results to the case $p\equiv 1\pmod{4}$ by carrying a full examination of prime ideal decompositions on cyclotomic integers. This has also the advantage of treating the cases $p\equiv 1$ and $\equiv 3\pmod{4}$ uniformly. The following result gives a beautifully simple description of the splitting of primes over $\Q[\zeta_n]$.

\begin{theorem}[\cite{Neukirch-ANT} Chapter I, \S 10, Prop. 10.3] \label{thm-CyclotomicSplitting}
Let $n=\prod_p p^{\nu_p}$ be the prime factorisation of $n\in\Z$, where the product is taken over all primes and $\nu_p=0$ for all but finitely many $p$. Let $p$ be a prime, denote by $f_p$ the multiplicative order of $p$ in $(\Z/(n/p^{\nu_p})\Z)^{\times}$. Then the prime ideal factorisation of $p$ in the ring of integers $\Z[\zeta_n]$ of $\Q[\zeta_n]$ is of the type
\[(p)=(\mathfrak{p}_1\dots\mathfrak{p}_r)^{\varphi(p^{\nu_p})},\]
where $\varphi$ is Euler's totient function and $r=\varphi(n)/f_p$.
\end{theorem}
Here we use the convention that $\varphi(1)=1$.
\begin{corollary}\normalfont Let $p$ and $q$ be distinct odd primes, and $f\geq 1$ a rational integer. Then the prime ideal decomposition of $(q)$ in $\Z[\zeta_{p^f}]$ is
\[(q)=\q_1\dots\q_r,\]
where $r$ is the index of $\langle q\rangle$ in $(\Z/p^f\Z)^{\times}$, i.e. $r=\varphi(p^f)/f_q$ where $f_q$ is the multiplicative order of $q$ modulo $p^f$.
\end{corollary}

\begin{proposition}\normalfont \label{prop-CycloRealPrimeFactor}\index{prime number!self-conjugate}Let $p$ and $q$ be distinct odd primes and $f\geq 1$ an integer. Then the following are equivalent
\begin{itemize}
\item[(i)] There is a prime ideal $\q$ in $\Z[\zeta_{p^f}]$ lying above $(q)$ which is fixed by complex conjugation.
\item[(ii)] All prime ideals in $\Z[\zeta_{p^f}]$ lying above $(q)$ are fixed by complex conjugation.
\item[(iii)] The prime $q$ is \textit{self-conjugate} modulo $p^{f}$, i.e. there is an integer $t$ such that
\[q^t\equiv -1\pmod{p^{f}}.\]
\end{itemize}
\begin{proof} We first prove that (i) is equivalent to (ii) by showing that the action of the Galois group of the cyclotomic field acts semi-regularly on the prime ideals above $(q)$. The Galois group $G=\Gal(\Q[\zeta_{p^f}]/\Q)$ is isomorphic to $(\Z/p^f\Z)^{\times}$, which is cyclic as $p^f$ is an odd prime power. If $\gamma$ is a generator of $(\Z/p^f\Z)^{\times}$, then $\sigma:\zeta_{p^f}\mapsto \zeta_{p^f}^{\gamma}$ is a generator of $G$. The automorphism $\tau$ given by complex conjugation is an involution, so necessarily $\tau=\sigma^{\varphi(p^f)/2}$. Since $q\neq p$, Theorem \ref{thm-CyclotomicSplitting} implies that $q$ is not ramified, i.e. the prime ideal decomposition of $(q)$ in $\Z[\zeta_{q^f}]$ is of the type $(q)=\q_1\dots\q_r$. Then, by Theorem \ref{thm-SplittingTheorem} the action of $G$ on $\{\q_1,\dots,\q_r\}$ is transitive. Now $G=\langle\sigma\rangle$, so we may assume without loss of generality that
\[\q_i^{\sigma}=\q_{i+1} \text{ for } 1\leq i<r,\text{ and } \q_r^{\sigma}=\q_1.\]
In other words, $\sigma$ induces the cycle $(1,\dots,r)$ on the indices of the prime factors $\q_i$. From this observation it follows that if a power of $\sigma$ fixes one of the $\q_i$ then it must fix all, this is in particular true for $\tau=\sigma^{\varphi(p^f)/2}$. To prove (ii) is equivalent to (iii) notice that Theorem \ref{thm-CyclotomicSplitting} applied to $n=p^f$ says that
\[(q)=\q_1\dots\q_r,\]
where $r=\varphi(p^f)/f_q$, and $f_q$ is the order of $q$ in $(\Z/p^f\Z)^{\times}$. Since $\tau=\sigma^{\varphi(p^f)/2}$ and $\sigma$ induces the permutation $(1\dots r)$ on the set of prime ideals $\{\q_1,\dots,\q_r\}$ it follows that $\tau$ fixes all primes above $(q)$ if and only if $r$ divides $\varphi(p^f)/2$. So there is an integer $t$ such that
\[t\cdot \varphi(p^f)/f_q=t\cdot r=\varphi(p^f)/2.\]
From here it follows that $f_q=2t$. So that $q^{2t}\equiv 1\pmod{p^f}$, and this is equivalent to 
\[q^{t}\equiv -1\pmod{p^f}.\qedhere\]
\end{proof}

\end{proposition}

The following theorem is our extension of Winterhof's results in \cite{Winterhof-NonexistenceButson}. In our result, the condition $p\equiv 3\pmod{4}$ required by Winterhof, is relaxed, and now we only require $p$ to be a prime. We will show in Proposition \ref{prop-SelfConjp3} that in the case $p\equiv 3\pmod{4}$ we obtain the same obstructions as Winterhof. A particular case of our extension had been obtained by de Launey in \cite{DeLauney-GHMSurvey}, where he obtained similar non-existence conditions for generalised Hadamard matrices over elementary abelian groups. The conditions of de Launey's apply to $\BH(n,p)$ matrices, and our result extends these to non-existence conditions on $\BH(n,p^f)$ matrices and $\BH(n,2p^f)$ matrices. Since prime ideal decompositions had been applied to obtain non-existence results for difference sets by authors like Arasu, Pott \cite{Arasu-CyclotomicDiffSets}, and Schmidt \cite{Schmidt-GroupInvariantBHSurvey}, it is possible that the following result was known. However, to the best of our knowledge the result below has never appeared in print before.
\begin{theorem}\label{thm-ButsonNonEx}
Let $p$ be an odd prime, and $f\geq 1$ an integer. Suppose that $n=p^{\ell}a^2 m$ is odd, where $p\nmid m$, and $m$ is square free. Then if $q\mid m$ and $q^t\equiv -1\pmod{p^f}$ for some integer $t$, then there cannot exist a $\BH(n,p^f)$ or a $\BH(n,2p^f)$.
\end{theorem}
\begin{proof}
Let $q$ satisfy the hypotheses in the statement. Then by Proposition \ref{prop-CycloRealPrimeFactor} one of the prime factors of $(q)$ in $\Z[\zeta_{p^f}]$ is fixed by complex conjugation. If $H$ is a $\BH(n,p^f)$ or $\BH(n,2p^f)$ then $HH^*=nI_n$. Taking determinants 
\[n^n=\det(HH^*)=\det(H)\det(H^*)=\det(H)\det(H)^{\tau}.\]
We have that $n^n=(p^{\ell}a^2)^nm^n$, where $n$ is odd, $p\nmid m$ and $m$ is square-free. If $q\mid m$ and $q^t\equiv -1\pmod{p^f}$ for some $t$, then by Proposition \ref{prop-CycloRealPrimeFactor} there is a prime ideal $\q$ lying above $(q)$ which is fixed by complex conjugation. Applying Proposition \ref{prop-DetNormCondition} with $\alpha=n^n$ and $\beta=\det(H)\in\Z[\zeta_{p^f}]$, we find that the multiplicity of $\q$ in the decomposition of $(n^n)$ must be even. From the unique factorisation of prime ideals over $\Z[\zeta_{p^f}]$ and the fact that $q\neq p$, it follows that the multiplicity of $\q$ in the decomposition of $(m)$ is also even. However, $m$ is square-free and $n$ is odd; and since every prime of $\mathcal{O}_K$ lies over a unique prime of $\Z$ (Theorem \ref{thm-UniquePrimeAbove}), the multiplicity of $\q$ in the decomposition of $m$ is odd. This gives a contradiction. \qedhere
\end{proof}

To better illustrate Proposition \ref{prop-CycloRealPrimeFactor} and Theorem \ref{thm-ButsonNonEx} we will study the prime decomposition patterns in $\Q[\zeta_{61}]$ in detail. The reason we choose the prime $61$ is that it is the smallest prime $p \equiv 1\pmod{4}$ for which $\varphi(p)=p-1$ is not a semiprime. Since $2$ is a generator of $(\Z/61\Z)^{\times}$ we have that $\sigma:\zeta_{61}\mapsto \zeta_{61}^2$ generates the Galois group $G=\Gal(\Q[\zeta_{61}]/\Q)$. From the fact that $\varphi(61)=60$, we find that $\tau=\sigma^{\varphi(61)/2}=\sigma^{30}$. Let $q\neq 61$ be an odd prime then by Theorem \ref{thm-CyclotomicSplitting} $(q)=\q_1\dots\q_r$, for some $r\mid \varphi(61)=60$. Let $O_{f_q}$ be the set of elements of order $f_q$ in $(\Z/61\Z)^{\times}$, then we have the following tables of values of $r$ and $\# O_{f_q}$ in terms of $f_q$:
\begin{table}[H]
\centering
\begin{tabular}{|c|cccccccccccc|}
\hline
$r$&$60$ & $30_*$ & $20$ & $15_*$ & $12$ & $10_*$  & $6_*$ & $5_*$ & $4$ & $3_*$ & $2_*$ & $1_*$\\
\hline
$f_q$ & $1$ & $2$ & $3$ & $4$ & $5$ & $6$ & $10$ & $12$ & $15$ & $20$ & $30$ & $60$\\
\hline
$\# O_{f_q}$ & $1$ & $1$ & $2$ & $2$ & $4$ & $2$ & $4$ & $4$ & $8$ & $8$ & $8$ & $16$\\
\hline
\end{tabular}
\end{table}
Here we highlight with the subscript $*$ those values of $r$ that divide $\varphi(61)/2=30$. Recall that $\sigma$ induces the cycle $(1\dots r)$ on the prime ideals above $(q)$, and hence $\q_i^{\tau}=\q_i$ for some $\q_i$ above $(q)$ if and only if $(1\dots r)^{\varphi(61)/2}=e$, which is equivalent to $r\mid \varphi(61)/2=30$. Notice that $\# O_{f_q}=\varphi(f_q)$, the reason for this is that in $(\Z/p\Z)^{\times}$ the set of elements of order $m$ is 
\[O_m=\{\gamma^{t\varphi(p)/m}: \gcd(t,m)=1\}.\]
Using this observation, we can compute the sets $O_{f_q}$ for which $r=60/f_q$ divides $30$:
\begin{align*}
&O_{2}=\{60\},\\
&O_{4}=\{11,50\},\\
&O_{6}=\{ 14, 48 \},\\
&O_{10}=\{ 3, 27, 41, 52 \},\\
&O_{12}=\{ 21,29,32,40 \},\\
&O_{20}=\{8,23,24,28,33,37,38,53\},\\
&O_{30}=\{4,5,19,36,39,45,46,49\},\\
&O_{60}=\{2, 6, 7, 10, 17, 18, 26, 30, 31, 35, 43, 44, 51, 54, 55, 59 \}.
\end{align*}
Now suppose that $n=61^{\ell}a^2m$ is odd, where $m$ is square-free and $61\nmid m$. Then if $q\neq 61$ is an odd prime congruent modulo $61$ to an element of the sets $O_{2},O_{4},\dots,O_{60}$, it follows from Theorem \ref{thm-ButsonNonEx} that a $\BH(n,61)$ or $\BH(n,122)$ cannot exist.
\begin{remark*}\normalfont There is a surjective ring homomorphism $\Z/(p^f)\Z\rightarrow\Z/p\Z$, given by $x\mapsto x\pmod{p}$. Therefore, if $q^t \equiv -1 \pmod{p^f}$ for some $t$, then $q^t\equiv -1\pmod{p}$.
\end{remark*}

In the style of Winterhof \cite{Winterhof-NonexistenceButson}, we compile a list of non-existence results for $\BH(n,p^f)$ and $\BH(n,2p^f)$ matrices with small $p$.

\begin{corollary}\normalfont \label{cor-BH5NE} Suppose $n=5^{\ell}p_1^{k_1}\dots p_r^{k_r}$ is odd, with $p_i\neq 5$. Then if $p_i\equiv 3,7,9\pmod{10}$ and $k_i$ is odd there cannot be a $\BH(n,5^f)$ or a $\BH(n,2\cdot 5^f)$.

\end{corollary}
\begin{proof}
The elements $x$ in $(\Z/5\Z)^{\times}$ that satisfy $x^t\equiv -1\pmod{5}$ for some $t$ are $2,3$ and $4\pmod{5}$ . Apply Theorem \ref{thm-ButsonNonEx} with $p=5$ and $f$ arbitrary: if $p_i^t\equiv -1\pmod{5^f}$ for some $t$, then there is no $\BH(n,5^f)$ or $\BH(n,2\cdot 5^f)$. Now under the surjection $\Z/5^f\Z\rightarrow\Z/5\Z$, $p_i$ projects to $2$, $3$ or $4$ modulo $5$. Therefore, modulo $10$ we find obstructions for primes in the congruence classes $2,3,4,7,8,9\pmod{10}$, however the even residues classes modulo $10$ contain no primes (except perhaps for $p=2$ but then $p$ cannot be a factor of $n$). We then find obstructions for primes $p\equiv 3,7,9\pmod{10}$.\qedhere
\end{proof}
\begin{remark}\normalfont The result above could have also been presented with $p_i\equiv 2,3,4\pmod{5}$ instead of $p_i\equiv 3,7,9\pmod{10}$. However, the second formulation is slightly better from a computational point of view since the classes modulo $10$ (and in general modulo $2p$) avoid even numbers.
\end{remark}
Notice that $4$ is a square residue modulo $5$ so the obstruction we found for primes $q\equiv 4\pmod{5}$ could not have been inferred using Winterhof's method. We illustrate this non-existence result with a concrete example.
\begin{example}\normalfont
There is no $\BH(95,5)$ or $\BH(95,10)$: We have that $95=19\cdot 5$. The prime ideal decomposition of $(19)$ in $\Z[\zeta_5]$ is
\[(19)=(19,4-(\zeta_5+\zeta_5^4))(19,14-(\zeta_5+\zeta_5^4)).\]
Notice that $\zeta_5+\zeta_5^4$ is fixed by conjugation, in particular depending on the embedding of $\Q[\zeta_5]$ in $\C$, $\zeta_5+\zeta_5^4$ is either $-\phi$ or $\frac{1}{\phi}$, where $\phi=\frac{1+\sqrt{5}}{2}$ is the golden ratio. So the prime ideals in the decomposition of $(19)$ are fixed under conjugation. The ideal $(5)$ splits as
\[(5)=(5,4+\zeta_5)^4.\]
Therefore, $(95)^{95}$ factorises as
\[(95)^{95}=(5,4+\zeta_5)^{380}(19,4-(\zeta_5+\zeta_5^4))^{95}(19,4-(\zeta_5+\zeta_5^4))^{95}.\]
However, if a $\BH(95,5)$ or $\BH(95,10)$ say $H$ exists then letting $\alpha=\det(H)\in\Z[\zeta_5]$ satisfies 
\[(\alpha)(\alpha)^{\tau}=(95)^{95}.\]
And by Proposition \ref{prop-DetNormCondition}, the multiplicity of $(19,4-(\zeta_5+\zeta_5^4))$ must be even and we find a contradiction.
\end{example}

The proof of Corollary \ref{cor-BH5NE} holds true for any odd prime $p$. For $p=13$ and $p=17$ we have:

\begin{corollary}\normalfont \label{cor-BH13NE} Suppose $n=13^{\ell}p_1^{k_1}\dots p_r^{k_r}$ is odd, with $p_i\neq 13$. \\Then if $p_i\equiv 5, 7, 11, 15, 17, 19, 21, 23, 25\pmod{26}$ and $k_i$ is odd there cannot be a $\BH(n,13^f)$ or a $\BH(n,2\cdot 13^f)$.
\end{corollary}
\begin{corollary}\normalfont \label{cor-BH17NE} Suppose $n=17^{\ell}p_1^{k_1}\dots p_r^{k_r}$ is odd, with $p_i\neq 17$.\\ Then if $p_i\equiv 3, 5, 7, 9, 11, 13, 15, 19, 21, 23, 25, 27, 29, 31, 33 \pmod{34}$ and $k_i$ is odd there cannot be a $\BH(n,17^f)$ or a $\BH(n,2\cdot 17^f)$.
\end{corollary}

From the proposition below, we see that  the method we presented recovers the results of Winterhof for non-existence of $\BH(n,p^f)$ when $p\equiv 3\pmod{4}$.

\begin{proposition}\normalfont\label{prop-SelfConjp3} Let $p\equiv 3\pmod{4}$ be a prime. Then an odd prime $q\neq p$ is self-conjugate modulo $p$ if and only if $q$ is a non-square residue modulo $p$.
\end{proposition}
\begin{proof}
Recall that $q$ is self-conjugate modulo $p$ if and only if 
$q^t\equiv -1\pmod{p},$
for some integer $t$. If $q$ is square-residue modulo $p$, then $q\equiv x^2\pmod{p}$ for some $x$, but then $(x^t)^2\equiv -1\pmod{p}$ and this is a contradiction since $p\equiv 3\pmod{4}$. Conversely, assume that $q$ is a non-square residue modulo $p$, we will show that $q$ has even multiplicative order modulo $p$. Let $\gamma$ be a generator of $(\Z/p\Z)^\times$, and suppose that $q$ has order $m$ in $(\Z/p\Z)^{\times}$, then 
\[q=\gamma^{r(p-1)/m},\]
for some $r$ coprime to $m$. Since $q$ is not a square residue and $r(q-1)$ is even, then $m$ must be even. Otherwise $r(q-1)/m$ would also be even and then $q$ would be a square-residue. Therefore $m=2t$ and necessarily $q^t\equiv -1\pmod{p}$.\qedhere
\end{proof}

 For completeness, we include here some small cases when $p\equiv 3\pmod{4}$:
\begin{corollary}[cf. Example 2 \cite{Winterhof-NonexistenceButson}]\normalfont \label{cor-BH3NE} Suppose $n=3^{\ell}p_1^{k_1}\dots p_r^{k_r}$ is odd, with $p_i\neq 3$. Then if $p_i\equiv 5\pmod{6}$ and $k_i$ is odd there cannot be a $\BH(n,3^f)$ or a $\BH(n,2\cdot 3^f)$.
\end{corollary}
\begin{corollary}[cf. Example 3 \cite{Winterhof-NonexistenceButson}]\normalfont \label{cor-BH7NE} Suppose $n=7^{\ell}p_1^{k_1}\dots p_r^{k_r}$ is odd, with $p_i\neq 7$. Then if $p_i\equiv 3,5,13\pmod{14}$ and $k_i$ is odd there cannot be a $\BH(n,7^f)$ or a $\BH(n,2\cdot 7^f)$.
\end{corollary}
\begin{corollary}[cf. Example 4 \cite{Winterhof-NonexistenceButson}]\normalfont \label{cor-BH11NE} Suppose $n=11^{\ell}p_1^{k_1}\dots p_r^{k_r}$ is odd, with $p_i\neq 11$. Then if $p_i\equiv 7, 13, 17, 19, 21\pmod{22}$ and $k_i$ is odd there cannot be a $\BH(n,11^f)$ or a $\BH(n,2\cdot 11^f)$.
\end{corollary}

\section{Non-existence of quaternary unit Hadamard matrices}\label{sec-NonExQUH}
The method we described above is very general and can be applied to several other number fields and interesting classes of matrices. The present author studied non-existence conditions for the family of \textit{quaternary unit Hadamard matrices}\index{Hadamard matrix!quaternary unit}, which was introduced by Fender, Kharaghani and Suda in \cite{Fender-Kharaghani-Suda}. These results have been published in a paper in collaboration with Heikoop, Pugmire and Ó Catháin in the \textit{Bulletin of the ICA}, see \cite{QUH-paper}. A quaternary unit Hadamard matrix, or $\QUH(n,m)$, is a matrix of order $n$ with entries in the set
\[X_m:=\left\{\frac{1\pm \sqrt{-m}}{\sqrt{m+1}},\frac{-1\pm\sqrt{-m}}{\sqrt{m+1}}\right\}.\]
The common minimal polynomial of the elements of $X_m$ is $g_m(X)=X^4+\frac{2(m-1)}{m+1}X^2+1$, so they belong to a biquadratic field extension of $\Q$, namely $K_m=\Q[\sqrt{-m},\sqrt{m+1}]\simeq \Q[X]/g_m(X)$. 

From Theorem \ref{thm-QuadraticSplitting} it is easy to determine which primes split in biquadratic extensions. First we have that if $K=\Q[\sqrt{a},\sqrt{b}]$, with $\gcd(a,b)=1$ then the lattice of subfields of $K$ is as follows

\begin{center}
\begin{tikzpicture}[node distance =2cm]
\node(K) {$K=\Q\left[\sqrt{a},\sqrt{b}\right]$};
\node(K2)[below=.75cm of K]{$K_2=\Q\left[\sqrt{b}\right]$};
\node(K1)[left=1.25cm of K2] {$K_1=\Q\left[\sqrt{a}\right]$};
\node(K3)[right=1.25cm of K2]{$K_3=\Q\left[\sqrt{ab}\right]$};
\node(Q)[below =2.25cm of K]{$\Q$};
\foreach \x in {1,2,3}{
	\draw(K\x)--(K);
	\draw(Q)--(K\x);
}
\end{tikzpicture}
\end{center}
We have also that $\disc(K)=\disc(K_1)\disc(K_2)\disc(K_3)$, (see Exercise 42 of \cite{Marcus-NumberFields}). We have that $\disc(\Q[\sqrt{a}])=a$ or $4a$. In particular for $K=\Q[\sqrt{-m},\sqrt{m+1}]$ the possible prime factors of $\disc(K)$ are $2$ and the prime factors of $m$ or $m+1$. In \cite{Fender-Kharaghani-Suda}, the authors prove the existence of $\QUH(q^f,q)$ for $f\geq 1$, for $q\equiv 3\pmod{4}$ a prime power. We study the non-existence in the case $m=p$ where $p\equiv 3\pmod{4}$.

\begin{proposition}[\cite{QUH-paper}]\normalfont \label{prop-BiQSplit} Let $a$ and $b$ be both positive integers, coprime and square-free, and let $q$ be an odd prime such that $q\nmid a$ and $q\nmid b$. Let $K=\Q[\sqrt{-a},\sqrt{b}]$, then the prime ideal factorisation of $(q)$ in $\mathcal{O}_K$ contains an ideal $\q$ fixed by complex conjugation if and only if
\[\legendre{-a}{q}=-1,\text{ and } \legendre{b}{q}=1.\]

\end{proposition}
\begin{proof}
Since $q\nmid a$ and $q\nmid b$ and $q$ is odd, then $q\nmid \disc(K)$. By Theorem \ref{thm-SplittingTheorem} there are the following possibilities for the splitting of $(q)$ in $\mathcal{O}_K$:
\begin{itemize}
\item[(i)] $(p)$ is inert,
\item[(ii)] $(p)=\q_1\q_2$, and
\item[(iii)] $(p)=\q_1\q_2\q_3\q_4$.
\end{itemize}
Let $K_1=\Q[\sqrt{-a}]$, $K_2=\Q[\sqrt{b}]$ and $K_3=\Q[\sqrt{-ab}]$. Then case (i) does not take place. If $(q)$ splits in one of the intermediate rings $\mathcal{O}_{K_1}$, $\mathcal{O}_{K_2}$ or $\mathcal{O}_{K_3}$, then it splits in $\mathcal{O}_K$. Suppose that $(q)$ does not split in $\mathcal{O}_{K_1}$ and $(q)$ does not split in $\mathcal{O}_{K_2}$, then by Theorem \ref{thm-QuadraticSplitting} we have that $\legendre{-a}{q}=-1$ and $\legendre{b}{q}=-1$, therefore
\[\legendre{-ab}{q}=\legendre{-a}{q}\legendre{b}{q}=(-1)^2=1.\]
So $(q)$ splits in $K_3$. Likewise we can show that if $(q)$ is inert in any two of the intermediate fields then it must split in the third.
The Galois group of $K$ over $\Q$ is $\Gal(K/Q)=\{e,\tau,\sigma,\tau\sigma\}\simeq \Z/2\Z\oplus \Z/2\Z$, where
\begin{align*}
& \tau:\sqrt{-a}\mapsto -\sqrt{-a},\ \sqrt{b}\mapsto\sqrt{b}\\
& \sigma:\sqrt{-a}\mapsto \sqrt{-a},\ \sqrt{b}\mapsto-\sqrt{b}
\end{align*}
In case (ii), using the transitivity of the Galois group it is easy to see there is a unique non-trivial element $\epsilon$ of $\Gal(K/\Q)$ that fixes $\q_1$ and $\q_2$. This implies that $(q)=\q_1\q_2$ splits in the fixed field of the automorphism $\epsilon$, and that $(q)$ is inert in the two other intermediate fields. Therefore, we find that $(q)=\q_1\q_2$ has prime factors fixed by $\tau$ if and only if $q$ splits in $K_2=\Q[\sqrt{b}]$ (which is the fixed field of $\tau$) and $q$ is inert in $K_1$. By Theorem \ref{thm-QuadraticSplitting} this is equivalent to $\legendre{-a}{q}=-1$, and $\legendre{b}{q}=1$. Finally in case (iii) we have that $(q)=\q_1\q_2\q_3\q_4$ and by transitivity of the Galois group we have that up to relabelling of the prime factors
\begin{align*}
\q_1^{\tau}=\q_2,\ \q_1^{\sigma}=\q_3,\text{ and }\q_1^{\tau\sigma}=\q_4.
\end{align*}
In particular the action of $\tau$ on the prime induces the permutation $(12)(34)$ which has no fixed-points. Hence if $(q)$ splits completely in $\mathcal{O}_K$, then there none of the prime ideal factors of $q$ are fixed by complex conjugation.\qedhere
\end{proof}

\begin{theorem}[\cite{QUH-paper}] \label{thm-QUH-Nonexistence}Let $m$ be a positive integer, such that neither $m$ nor $m+1$ are perfect squares. Write $m=(m_0)^2 a$ and $m+1=(m_0')^2 b$, where $a,b>1$ are square-free. Let $n=(n_0)^2t$ be an odd integer, where $t$ is square-free. Suppose $p$ is an odd prime, coprime to both $m$ and $m+1$ and $p\mid t$. If
\[\legendre{-a}{p}=-1,\text{ and } \legendre{b}{p}=1,\]
then there cannot exist a $\QUH(n,m)$.
\end{theorem}
\begin{proof}
First note that $m$ and $m+1$ are coprime, so $a$ and $b$ must also be coprime. Then we have that $K=\Q[\sqrt{-m},\sqrt{m+1}]=\Q[\sqrt{-a},\sqrt{b}]$, and the hypotheses of Proposition \ref{prop-BiQSplit} hold for the prime $p$. So we find that $p$ splits as $p=\q_1\q_2$ with $\q_i^{\tau}=\q_i$ for $i=1,2$. The elements of
\[X_m=\{\pm\alpha_m,\pm\alpha_m^*\}:=\left\{\frac{1\pm \sqrt{-m}}{\sqrt{m+1}},\frac{-1\pm\sqrt{-m}}{\sqrt{m+1}}\right\}\]
are not necessarily algebraic integers, since their minimal polynomial $g_m(x)=x^4+2\frac{m-1}{m+1}x^2+1$, does not always have integral coefficients. However, multiplying by $(m+1)$, we have that $\pm(m+1)\alpha_m$ and $\pm(m+1)\alpha_m^*$ are all algebraic integers. If there exists a $\QUH(n,m)$, say $H$, then $(m+1)H$ has coefficients in $\mathcal{O}_K$ and $\det((m+1)H)\det((m+1)H)^{\tau}=(m+1)^{2n} n^n=s^2t$, for some $s\in\Z$. By Proposition \ref{prop-DetNormCondition}, the primes $\q_1$ and $\q_2$ must appear with even multiplicity in the equation. However $n$ is odd, $q$ divides $t$ with multiplicity $1$ and both $\q_1$ and $\q_2$ lay above the prime $p$ and no other rational prime. This implies that the multiplicity of $\q_1$ and $\q_2$ is odd, and this is a contradiction. \qedhere
\end{proof}

In \cite{Fender-Kharaghani-Suda} the authors give a construction for $\QUH(q^r,q)$ for $r\geq 1$. For $\QUH(n,q)$ matrices, where $q\equiv 3\pmod{4}$ is an odd prime, we obtain the following table of non-existence.

\begin{center}
\begin{tabular}{|l |l |}
\hline
$q$ & $n$ \\
\hline
 $7$ & $17,31,41,47,51,73,85,89,93,97,103,119,123,141,\dots$\\
 $11$ & $13,39,61,65,73,83,91,107,109,117,131,143,167,\dots$\\
 $19$ & $29,31,41,59,71,79,87,89,93,109,123,145,151,\dots$\\
 $23$ & $5,15,19,35,43,45,53,55,57,65,67,85,95,97,105,\dots$\\
 $31$ & $17,23,51,69,73,79,85,89,115,119,127,137,151,\dots$\\
 $43$ & $5,7,15,19,21,35,37,45,55,57,63,65,77,85,89,91,\dots$\\
 \hline
\end{tabular}\\\vspace{0.5cm}

\textit{Pairs $(n,q)$ such that $\QUH(n,q)$ is empty.}
\end{center}

\begin{research-problem}\normalfont Find examples of $\QUH(n,q)$ matrices where $n$ is not a power of $q$.
\end{research-problem}

\cleardoublepage
\chapter{A Survey on Butson-type Hadamard Matrices}\label{chap-BHMats}

In this chapter we study constructions of Butson-type Hadamard matrices. Our exposition here is essentially self-contained: the only theorem that we will require from a previous chapter is Theorem \ref{thm-ButsonNonEx}. Generalised Hadamard matrices (GHMs) are closely related to Butson matrices, and they will make a brief appearance here. We included additional material on GHMs and their relationship to projective planes in Appendix \ref{app-GHMs}.\\

 Hadamard matrices are square matrices with entries in the complex unit circle, whose rows and columns are pairwise orthogonal. These matrices are a type of maximal determinant matrix, meaning that they achieve the maximum absolute value of the determinant among a certain set of matrices. A Butson-type Hadamard matrix of order $n$ and with entries over the $m$-th roots of unity is denoted $\BH(n,m)$.  In particular, the set of real Hadamard matrices of order $2$ is precisely the set of $\BH(n,2)$ matrices. The Butson-type families of Hadamard are particularly interesting because they are closed under taking the tensor product. This gives a series of tensor-like constructions for BH matrices, which we survey.\\

The families of Butson matrices $\BH(n,4)$ and $\BH(n,6)$ have been surveyed in \cite{Szollosi-PhDThesis}, so one of our goals in this chapter will be to complement this survey, by having a special focus on results for $\BH(n,p)$ matrices where $p$ is a prime number. For example, we include a discussion on de Launey's existence result for $\BH(2^t\cdot 3,3)$ where $t\geq 1$.\\

The class $\BH(n,4)$ has received special attention in the literature on Hadamard matrices, partly due to the existence of the Turyn morphism, which is a mapping from $\BH(n,4)$ matrices to $\BH(2n,2)$ matrices. In \cite{Compton-Craigen-DeLauney}, Compton, Craigen, and de Launey showed that there is a partial morphism from $\BH(n,6)$ to $\BH(4n,2)$. The study of morphisms establishes  relationships between different sets of Hadamard matrices, and these can sometimes give more insight into the constructions of such matrices. A theory of morphisms between Butson-type matrices has been developed by Egan, Ó Catháin, and Swartz \cite{Egan-OCathain-Swartz-Spectra}, and by Österg\aa rd and Paavola \cite{Ostergard-Paavola-Morphisms}. We will briefly survey this part of the literature, and include one of our new contributions, which consists of a morphism from certain classes of non-Butson Hadamard matrices into real Hadamard matrices. This is interesting, since previously the only known morphisms were between Butson classes. Our result appeared published in the paper \cite{QUH-paper}, in collaboration with Heikoop, Pugmire, and Ó Catháin.\\

Finally, we will discuss the results of de Launey and Dawson on the asymptotic existence of $\BH(hp,p)$ matrices, where $p$ is prime and $h$ is the order of a real Hadamard matrix \cite{DeLauney-Dawson-AsymptoticExistence}. Our main contribution here is an improvement on the lower bound on $p$ for the existence of $\BH(12p,p)$ matrices from $p>104857600=(10\cdot 2^{10})^2$, to $p>263$. This was obtained by computational methods.

\section{Hadamard matrices}

\begin{definition} \normalfont \label{def-HadamardMatrix} An \index{Hadamard matrix} Hadamard matrix $H$ of order $n$ is an $n\times n$ matrix with complex entries of modulus $1$, satisfying the matrix equation
\[HH^*=nI_n.\]
\end{definition}

In particular, a \index{Hadamard matrix! real}\textit{real Hadamard matrix} is a $\pm 1$ matrix $H$ satisfying $HH^{\intercal}=nI_n$. In the literature on Hadamard matrices there is conflicting terminology that the reader must be aware of: Historically, real Hadamard matrices have been the family that received the most attention. Because of this, the term Hadamard matrix is used to refer to real Hadamard matrices in most of the literature. An Hadamard matrix in the sense of Definition \ref{def-HadamardMatrix} is sometimes called complex Hadamard matrix. However, other authors reserve the term complex Hadamard matrix for matrices with entries in the set $\{\pm 1,\pm i\}$, where $i^2=-1$.\\

The study of Hadamard matrices dates back at least to J. J. Sylvester's 1867 paper colourfully entitled \textit{Thoughts on inverse orthogonal matrices, simultaneous sign successions, and tessellated pavements in two or more colours, with applications to Newton's rule, ornamental tile-work, and the theory of numbers} \cite{Sylvester-InverseOrthogonal}. Here Sylvester studied a family of matrices that are known as type II matrices, which are a generalisation of Hadamard matrices.
\begin{definition}\normalfont
A \index{matrix!type II} \textit{type II} matrix $M$ is a matrix with complex non-zero entries such that 
\[MM^-=nI_n,\]
where $M^{-}$ is the entrywise inverse transpose of $M$, i.e. the $(i,j)$ entry of $M^{-}$ is given by $(M^-)_{ij}=1/M_{ji}$.
\end{definition}
 The term Hadamard matrix comes from Hadamard's celebrated determinant bound

\begin{theorem}[Hadamard, 1893 \cite{Hadamard-Determinants}]\index{determinant inequality!Hadamard}\label{thm-HadamardTheorem}
Let $M$ be an $n\times n$ matrix with entries taken from the complex unit disk, then
\[|\det(M)|\leq n^{n/2}.\]
Furthermore, the bound is met with equality if and only if $MM^*=nI_n$.
\end{theorem}

In particular, Hadamard matrices are maximal determinant matrices. A straightforward combinatorial argument shows that real Hadamard matrices can only exist at orders $n=1$, $2$ or $n$ a multiple of $4$. 

\begin{research-problem}[Hadamard conjecture]\normalfont Show that real Hadamard matrices exist at orders $4n$ for every integer $n\geq 1$.
\end{research-problem} 

Currently, the smallest open case for the existence of a real Hadamard matrix is $n=668$.\\

Because of the wide range of applicability of real Hadamard matrices, the Hadamard conjecture has received much attention, and there are many surveys of real Hadamard matrices that the reader can consult, see for example \cite{Horadam-HadamardBook}. \\

In this chapter, we focus instead on the class of \textit{Butson Hadamard matrices} \cite{Butson}. We will also mention the closely related concept of \textit{generalised Hadamard matrices}, introduced by Drake in \cite{Drake}. Generalised Hadamard matrices have a close connection to projective planes, for more on this topic see Appendix \ref{app-GHMs}.

\section{Butson-Type Hadamard matrices}

\begin{definition}\normalfont
A \textit{Butson matrix} or \textit{Butson-type Hadamard matrix} is a complex Hadamard matrix with entries taken from the set of $m$-th roots of unity $\mu_m:=\{1,\zeta_m,\zeta_m^2,\dots,\zeta_m^{m-1}\}$. 
The set of Butson matrices of order $n$ with entries in $\mu_m$ is denoted by $\BH(n,m)$.
\end{definition}

Recall that a\index{matrix!monomial} \textit{monomial matrix} is a matrix $P$ with exactly one non-zero entry in each row and column. In particular, permutation matrices are monomial. 
If $P$ and $Q$ are monomial matrices with unimodular entries, and $H$ is Hadamard then $P^*HQ$ is Hadamard since the entries of $P^*HQ$ have modulus $1$ and,
\[(P^* HQ)(P^*HQ)^*=P^*HQQ^*HP=P^*HH^*P=nP^*P=nI_n.\]
This motivates the following,
\begin{definition}\normalfont \label{def-MonomialEquivalence} Two $\BH(n,m)$ matrices $H_1$ and $H_2$ are \index{monomial equivalence}\textit{monomially equivalent}, or simply \textit{equivalent}, if and only if there exist two monomial matrices $P$ and $Q$ with non-zero entries in $\mu_m$ such that
\[P^*H_1Q=H_2.\]
\end{definition}

The following is a well-known non-existence condition for Butson-type Hadamard matrices:
\begin{lemma}\normalfont \label{lemma-BHpDiv} Let $p$ be a prime number. If there exists a $\BH(n,p)$, then $p\mid n$.
\end{lemma}
\begin{proof}
Any $\BH(n,p)$ matrix $H$ is equivalent to a matrix whose first row consists of all ones. Suppose that the second row of $H$ is given by the vector,
\[(\zeta_p^{a_1},\dots,\zeta_p^{a_n}),\]
for some $0\leq a_i\leq p-1$. Then taking inner product of the first row with the second we find
\[\zeta_p^{a_1}+\dots+\zeta_p^{a_n}=0.\]
Let $f(x)=x^{a_1}+\dots+x^{a_n}\in\Z[x]$, then $f(\zeta_p)=0$, and since the cyclotomic polynomial
\[\Phi_p(x)=1+x+x^2+\dots+x^{p-1},\]
is the minimal polynomial of $\zeta_p$, there exists a polynomial $g\in\Z[x]$ such that $\Phi_p(x)g(x)=f(x)$. Now, $f(1)=n$ and $\Phi_p(1)=p$, so evaluating at $1$ we find
\[\Phi_p(1)g(1)=pg(1)=n=f(1).\]
And thus, $p$ divides $n$.\qedhere
\end{proof}
We remark that the fact that $p$ is prime is essential in the proof. Since in general, for the $m$-th cyclotomic polynomial $\Phi_m(x)$ we cannot guarantee that $\Phi_m(1)=m$. For example
\[\Phi_6(x)=x^2-x+1,\]
hence $\Phi_6(1)=1$. In fact we have several patterns of vanishing sixth roots of unity, for example
\[1+\omega+\omega^2+1+(-1)=0,\]
where $\omega$ is a primitive third root of unity, is a vanishing sums of $5$ sixth roots of unity.\\

We will need some notions from character theory, see also Babai's lecture notes \cite{Babai-FourierAnalysisFG}.
\begin{definition}\normalfont \label{def-MultCharacter} Let $G$ be a finite abelian group, written multiplicatively. \index{character!linear} A \textit{linear character}, or simply \textit{character}, of $G$ is a homomorphism $\chi:G\rightarrow \C^{\times}$, in other words, 
\[\chi(ab)=\chi(a)\chi(b)\]
for all $a,b\in G$.
\end{definition}
In particular, $\chi(e)=1$ where $e$ is the identity element of $G$. Note that if the exponent of $G$ is $n$, i.e. if $x^n=e$ for all $x\in G$, then $\chi(G)\subseteq \mu_n$ for any character $\chi$ of $G$. Indeed, for any $x\in G$, $\chi(x)^n=\chi(x^n)=\chi(e)=1$.

\begin{example}\normalfont Let $G$ be an arbitrary abelian group, then the function $\varepsilon:G\rightarrow\C^{\times}$ given by $\varepsilon(x)=1$ for all $x\in G$ is a character of $G$. The character $\varepsilon$ is known as the \index{character!trivial}\textit{trivial character} of $G$.
\end{example}

\begin{example}\normalfont Let $G=C_n$ be the cyclic group on $n$ elements. Let $\gamma$ be a generator of $G$. Then, the function $\chi(\gamma^a)=\zeta_n^{a}$ is a character of $G$. Furthermore, any character of $G$ is a power of $\chi$.
\end{example}


The product $\chi\psi$ of two (linear) characters $\chi$ and $\psi$ of a group $G$ is itself a character, since 
\[(\chi\psi)(ab)=\chi(ab)\psi(ab)=\chi(a)\chi(b)\psi(a)\psi(b)=\chi(a)\psi(a)\chi(b)\psi(b)=(\chi\psi)(a)(\chi\psi)(b).\]
Likewise, the complex conjugate $\overline{\chi}$ of a character $\chi$ is itself a character. Additionally, $\chi\overline{\chi}=\varepsilon$. This implies that the set of characters of a group $G$ is a group, called the \textit{dual group} of $G$, and denoted $\hat{G}$.

\begin{lemma} \normalfont \label{lemma-VanishingCharSum} Let $\chi$ be a non-trivial character of a finite group $G$, then 
\[\sum_{x\in G} \chi(x)=0.\]
\end{lemma}
\begin{proof}
Let $S=\sum_{x\in G}\chi(x)$. Since $\chi$ is a non-trivial character, there is an element $y\in G$ such that $\chi(y)\neq 1$. Then,
\[\chi(y)S=\chi(y)\sum_{x\in G}\chi(x)=\sum_{x\in G}\chi(yx).\]
The mapping $G\rightarrow G$ given by $y\mapsto yx$ is invertible, with inverse given by the mapping $x\mapsto y^{(-1)}x$. Therefore, we have $\{yx: x\in G\}=G$, and 
\[\chi(y)S=\sum_{x\in G}\chi(yx)=\sum_{x\in G}\chi(x)=S.\]
It follows that $(\chi(y)-1)S=0$, but we know that $\chi(y)\neq 1$, so we must have $S=\sum_{x\in G} \chi(x)=0$.\qedhere
\end{proof}
\begin{corollary}\normalfont \label{cor-InnerProdChars}
Let $\chi$ and $\psi$ be two distinct characters of $G$, then
\[\sum_{x\in G}\chi(x)\overline{\psi(x)}=0.\]
\end{corollary}
\begin{proof}
Since $\chi$ and $\psi$ are distinct characters, the character $\chi\overline{\psi}$ is non-trivial, therefore by Lemma \ref{lemma-VanishingCharSum}, we have that
\[\sum_{x\in G}(\chi\overline{\psi})(x)=\sum_{x\in G}\chi(x)\overline{\psi(x)}=0.\]
\end{proof}

 Our first example of a $\BH(n,n)$ is given by the Fourier matrix:\index{matrix!Fourier}
\begin{lemma}[cf. Example 4.1.1. \cite{Horadam-HadamardBook}]\label{lemma-FourierMatrix} Let $F_n$ be the $n\times n$ matrix given by $F_n=(\zeta_n^{ij})_{ij}$. Then $F_n\in \BH(n,n)$.
\end{lemma}
\begin{proof}
The result follows by direct computation
\[(F_nF_n^*)_{ij}=\sum_k (F_n)_{ik}(F_n^*)_{kj}=\sum_k \zeta_n^{ik}\zeta_n^{-jk}=\sum_k \zeta_n^{(i-j)k}.\]
If $i=j$ then $(F_nF_n^*)_{ij}=n$, and if $i\neq j$ then $(F_nF_n^*)_{ij}=0$ as $\sum_k \zeta_n^{(i-j)k}$ is the character sum of a non-trivial character of $\Z/n\Z$ (Lemma \ref{lemma-VanishingCharSum}). Another way to see that $\sum_{k}\zeta_n^{(i-j)k}$ vanishes is the following: For $d>1$ let $\zeta_d$ be a primitive $d$-th root of unity, then from the polynomial identity 

\[x^d-1=\prod_k (x-\zeta_d^k),\]
it follows that the coefficient of $x^{d-1}$ in the right-hand side vanishes, namely $(-\sum_{k}\zeta_d^k)=0$. Now let $d=\frac{n}{\gcd(i-j,n)}$, then if $i-j$ is not a multiple of $n$ we have that $d>1$ and $\zeta_n^{(i-j)}=\zeta_d$ so
\[\sum_{k=0}^{n-1}\zeta_n^{(i-j)k}=\frac{n}{d}\sum_{k=0}^{d-1}\zeta_d^k=0.\]
This shows that when $i\neq j$ the $i$-th and $j$-th row of $F_n$ are orthogonal, hence $F_nF_n^*=nI_n$.
\end{proof}
More generally, the character table of a finite abelian group of order $n$ and exponent $m$ gives an example of a $\BH(n,m)$. The following is a well-known fact in the character theory of finite groups:

\begin{lemma} Let $G$ be an abelian group of order $n$ and exponent $m$. Then the character table of $G$ is a $\BH(n,m)$.
\end{lemma}
\begin{proof}
Let $G=\{g_1,\dots,g_n\}$. Since $G$ is abelian, $G$ has $n$ irreducible linear characters $\chi_1,\dots,\chi_n$ and since the exponent of $G$ is $m$ the values $\chi_i(g_j)$ are all $m$-th roots of unity. Let $C=(\chi_i(g_j))_{ij}$ be the character table of $G$, the inner product of two rows of $C$ is of the type
\[\sum_k \chi_i(g_k)\overline{\chi_j(g_k)}=\sum_k (\chi_i\overline{\chi_j})(g_k).\]
The product $\chi_i\overline{\chi_j}$ of characters of $G$ is a non-trivial character if and only if $i\neq j$, thus of $C$ are orthogonal and this shows $CC^*=nI_n$.\qedhere
\end{proof}

\section{Tensor-like constructions for Hadamard matrices}

Let $\s^1(\C)=\{z\in\C: |z|=1\}$ be the group of complex numbers of unit modulus. Let $\Lambda$ be a finite multiplicatively closed subset of $\s^1(\C)$, then $\Lambda$ is the multiplicative group of $m$-th roots of unity for some $m$. Indeed since $\Lambda$ is finite it is easy to see that it must be a group, and furthermore $\Lambda$ has a finite exponent $m$ so that $\alpha^m =1$ for all $\alpha\in \Lambda$, in particular $\Lambda$ is a subgroup of the group of $m$-th roots of unity. We will mention below some constructions for families of complex Hadamard matrices that hold whenever the entries are taken from a finite multiplicatively closed subset $\Lambda$ of $\s^1(\C)$ (which by the above remark such families are always of the type $\BH(n,m)$ for some $m$). These are the tensor-product-like constructions in which we are only allowed to multiply elements within $\Lambda$ and permute rows or columns.\\

We begin with the first such construction, which goes back to J. J. Sylvester \cite{Sylvester-InverseOrthogonal}. Although simple it is one of the most important existence results for  Hadamard matrices as it allows us to combine them multiplicatively.

\begin{proposition}[Sylvester, \cite{Sylvester-InverseOrthogonal}]\normalfont \label{prop-SylvesterConstruction}
If $H_1$ and $H_2$ are $\BH(n,k)$ and $\BH(m,\ell)$ matrices, then their Kronecker product $H_1\otimes H_2$ is a $\BH(nm,\lcm(k,\ell))$.
\end{proposition}
\begin{proof}
Since $H_1$ and $H_2$ are Hadamard, we have $H_iH_i^*=n_iI_{n_i}$. Therefore
\[(H_1\otimes H_2)(H_1\otimes H_2)^*=H_1H_1^*\otimes H_2H_2^*=n_1n_2 I_{n_1n_2}.\]
$H_1\otimes H_2$ is a block matrix with $(i,j)$ block given by $(H_1)_{ij}H_2$, hence the entries of $H_1\otimes H_2$ are $\lcm(k,\ell)$-th roots of unity. \qedhere
\end{proof}

Subsequent generalisations of the tensor product have appeared, which give additional constructions for Hadamard matrices.
\begin{proposition}[Di\c{t}\u{a} construction, \cite{Dita-HadParametrisation}]\normalfont
Let $H$ be an Hadamard matrix of order $n$ and let $L_1,\dots,L_n$ be Hadamard matrices of order $m$, then the matrix
\[H\otimes [L_1,\dots,L_n] = 
\left[
\begin{array}{c c c c}
h_{11} L_1 & h_{12} L_2 &\dots & h_{1n}L_n\\
h_{21} L_1 & h_{22} L_2 &\dots & h_{2n}L_n\\
\vdots &\vdots & \ddots & \vdots\\
h_{n1}L_1 & h_{n2}L_2 &\dots & h_{nn}L_n
\end{array}
\right]\]
is an Hadamard matrix of order $nm$.
\end{proposition}
\begin{proof}
Let $M:=H\otimes [L_1,\dots,L_n]$. Taking inner products by row-blocks we find that the block $(i,i)$ of $MM^*$ is
\[\sum_j h_{ij}L_jL_j^* \overline{h_{ij}}=\sum_j mI_m=nmI_m,\]
and the block $(i,j)$ with $i\neq j$ of $MM^*$ is
\[\sum_k h_{ik}L_kL_k^*\overline{h_{jk}}=\left(\sum_k h_{ik}\overline{h_{jk}}\right)mI_m=0.\qedhere\]
\end{proof}

Hosoya and Suzuki \cite{Hosoya-Suzuki-TypeIIBMA} found a more general tensor product, encompassing Di\c{t}\u{a}'s construction, and identified algebraic conditions on the \textit{Nomura algebra} of an Hadamard matrix that determine if the matrix can be expressed as a generalised tensor product.
\index{Nomura algebra}

\begin{definition}\normalfont Let $(U_1,U_2,\dots,U_m)$ be square matrices of size $n$, and let $(V_1,V_2,\dots,V_m)$ be square matrices of size $n$. We define the \textit{generalised tensor product} of $U_1,\dots,U_m$ and $V_1,\dots,V_n$ as the matrix $(U_1,U_2,\dots,U_m)\otimes (V_1,V_2,\dots,V_n)$ whose block $(i,j)$ is the matrix
\[\Delta_{ij}V_j,\]
where $\Delta_{ij}$ is the diagonal matrix whose $h$-th diagonal entry is the $(i,j)$ entry of the matrix $U_h$, namely the $(h,k)$ entry of $\Delta_{ij}$ is $\Delta_{ij}[h,k]=\delta_{hk}U_h[i,j]$ .
\end{definition}

\begin{theorem}[Hosoya-Suzuki, Lemma 4.1 \cite{Hosoya-Suzuki-TypeIIBMA}]
Let $U_1, U_2, \ldots, U_m$ be square matrices of size $n$, and $V_1, V_2, \ldots, V_n$ be square matrices of size $m$. Then the following are equivalent.
\begin{itemize}
  \item[(i)] $(U_1,U_2,\dots,U_m)\otimes (V_1,V_2,\dots,V_n)$ is a type II matrix.
  \item[(ii)] $U_1, U_2, \ldots, U_m, V_1, V_2, \ldots, V_n$ are type II matrices.
\end{itemize}

\end{theorem}

So in particular, the generalised tensor product of two sequences of Hadamard matrices \linebreak $(H_1,\dots,H_m)$ and $(H_1',\dots,H_n')$, of orders $n$ and $m$ respectively, is an Hadamard matrix.\\

There have been more recent developments coming from the physics community that use the concept of \textit{mutually unbiased bases}, or \textit{MUBs}:\index{mutually unbiased bases}

\begin{definition}\normalfont Two orthonormal bases in a Hilbert space $\C^d$, $\mathcal{B}=\{e_1,\dots,e_d\}$ and $\mathcal{B}'=\{f_1,\dots,f_d\}$ are \textit{mutually unbiased} if $|\langle e_i,f_j\rangle|^2=1/d$ for all $1\leq i,j\leq d$. A set of orthonormal bases $\{\mathcal{B}_1,\mathcal{B}_2,\dots,\mathcal{B}_k\}$ is called \textit{mutually unbiased} if any pair of bases in the set is mutually unbiased.
\end{definition}

The following result is well-known and a straightforward consequence of the definition of MUBs.
\begin{lemma} \normalfont \label{lemma-MUBMatrix}Let $\mathcal{B}=\{e_1,\dots,e_d\}$ and $\mathcal{B}'=\{f_1,\dots,f_d\}$ be two orthonormal bases in $\C^{d}$. Let $K=[e_1|\dots|e_d]$ and $L=[f_1|\dots|f_d]$ be $d\times d$ matrices given by the columns of each basis.  Then $\mathcal{B}$ and $\mathcal{B}'$ are mutually unbiased if and only if $\sqrt{d}K^*L$ is an Hadamard matrix. 
\end{lemma}
\begin{proof}
Suppose that $\mathcal{B}$ and $\mathcal{B}'$ are mutually unbiased. Then, since each basis is orthonormal, the matrices $K$ and $L$ are unitary, i.e. $K^*K=KK^*=L^* L=LL^*=I_d$. From the identity $|e_i^*f_j|=|\langle e_i,f_j\rangle|= 1/\sqrt{d}$, we find that $|\sqrt{d}e_i^*f_j|=1$. Therefore, the matrix
\[\sqrt{d}K^*L= \sqrt{d}\begin{bmatrix}
e_1^*f_1 & e_1^*f_2 & \cdots & e_1^*f_d \\
e_2^*f_1 & e_2^*f_2 & \cdots & e_2^*f_d \\
\vdots & \vdots & \ddots & \vdots \\
e_d^*f_1 & e_d^*f_2 & \cdots & e_d^*f_d \\
\end{bmatrix},\]
has entries of modulus $1$. Direct computation shows 
\[(\sqrt{d}K^*L)(\sqrt{d}K^*L)^*=dK^*LL^*K^*=dK^*K=dI_d,\]
so $\sqrt{d}K^*L$ is an Hadamard matrix. Conversely, if $\sqrt{d}K^*L$ is an Hadamard matrix, then the fact that its entries are of modulus $1$ implies that $|\langle e_i,f_j\rangle|=1/\sqrt{d}$.\qedhere
\end{proof}

In view of Lemma \ref{lemma-MUBMatrix}, we say that two unitary $d\times d$ matrices $L$ and $K$ are \textit{mutually unbiased} if and only if $\sqrt{d}K^*L$ is an Hadamard matrix.
Let $M(d)$ be the maximal cardinality of a set of MUBs in $\C^d$. It can be shown, see Section 12.4 of \cite{Bengtsson-Zyczkowski-GeoQuantumStates}, that $M(d)\leq d+1$. A set of MUBs in $\C^d$ of cardinality $d+1$ is called a \textit{complete set of MUBs}\index{mutually unbiased bases!complete set of}. In prime power dimensions $q=p^n$ there exists a complete set of MUBs, we will show this in Section \ref{sec-BHDoublyEven} for  $q=p$.

\begin{theorem}[McNulty-Weigert, Theorem 3 \cite{McNulty-Weigert-IsolatedHadamardMUBs}]\label{thm-McNultyWeigert} Let $H$ be an Hadamard matrix of order $n$. Let $K_1,\dots,K_n$, and $L_1,\dots,L_n$ be two sets of $d\times d$ unitary matrices such that $K_i$ is unbiased to $L_j$, then the matrix
\[
M=\sqrt{d}
\left[
\begin{array}{cccc}
h_{11} K_1^*L_1 & h_{12} K_1^*L_2 & \dots &h_{1n} K_1^*L_n\\
h_{21} K_2^*L_1 & h_{22} K_2^*L_2 & \dots &h_{2n} K_2^*L_n\\
\vdots &\vdots & \ddots &\vdots\\
h_{n1} K_n^*L_1 & h_{n2} K_n^*L_2 & \dots &h_{nn} K_n^*L_n\\
\end{array}
\right],
\]
is an Hadamard matrix of order $nd$.
\end{theorem}
\begin{proof}
By Lemma \ref{lemma-MUBMatrix}, we have that for each $1\leq i,j\leq n$, the matrix $H_{ij}=\sqrt{d}K_i^*L_j$ is an Hadamard matrix. In particular, the entries of $h_{ij}K_i^*L_j$ are of modulus $1$. Direct computation shows that the inner product of the $r$-th row block of $M$ with the $s$-th row block of $M$ is 
\begin{align*}
\sum_j (\sqrt{d}h_{rj}K_r^*L_j)(\sqrt{d}h_{sj}K_s^*L_j)^*&=d\sum_j h_{rj}\overline{h_{sj}}K_r^*L_jL_j^*K_s\\
&=d\sum_j h_{rj}\overline{h_{sj}}K_r^*K_s\\
&=dK_r^*\left(\sum_j h_{rj}\overline{h_{sj}}\right)K_s\\
&=nd\delta_{rs}K_r^*K_s\\
&=nd\delta_{rs}I_{d}.
\end{align*}
Therefore, $MM^*=ndI_{nd}$.\qedhere
\end{proof}

The main advantage of the construction in Theorem \ref{thm-McNultyWeigert} is that the matrices $K_i$ and $L_j$ need not be Hadamard.


\index{Hadamard matrix!generalised}
\begin{definition}\normalfont\label{Def-GHM} Let $G$ be a finite group of order $n$. An $n\times n$ matrix $H$ with entries in $G$ is a \textit{Generalised Hadamard matrix} over the group $G$, or $\GH(nt,G)$ if and only if for every $i\neq j$ the list of quotients $[h_{ik}h_{jk}^{-1}:k=1,\dots, nt]$ contains every element of $G$ exactly $t$ times.
\end{definition}
Let $\Z[G]$ be the group ring of $G$, and denote by $H^{*}$ the transpose of the entrywise inverse of $H$, i.e. $(H^{*})_{ij}=h_{ji}^{-1}$. Then $H$ is a $\GH(nt,G)$ if and only if in $\Mat_n(\Z[G])$ 
\[HH^*=(nt-t[G]) I_{nt}+t[G]J_{nt}.\]
where $[G]=\sum_{g\in G}g$. Let $\mathcal{I}=\langle [G]\rangle$ be the principal two-sided ideal generated by $[G]$, then in the quotient ring $\Z[G]/\mathcal{I}$ a generalised Hadamard matrix satisfies the usual orthogonality equation

\[HH^* = ntI_{nt} \pmod{\mathcal{I}}.\]

In the case that $G=C_p$ is the cyclic group of order $p$ where $p$ is a prime number, the concept of $\BH(n,p)$ matrices and $\GH(n,C_p)$ matrices coincide. To see this simply map a generator $\gamma$ of $C_p$ to $\zeta_p$ a primitive $p$-th root of unity. More generally if $G=C_m$ is the cyclic group of order $m$ then every $\GH(n,C_m)$ is a $\BH(n,m)$ but the converse does not necessarily hold. For example for $m=6$ there exists a $\BH(7,6)$ but for a $\GH(n,C_6)$ to exist it is necessary that $6\mid n$. In terms of group rings, the ring homomorphism
\begin{align*}
\phi: \Z[C_m] & \rightarrow \Z[\zeta_m]\\
 \sum_i a_i x^i &\mapsto \sum_i a_i \zeta_m^i
\end{align*}
has a kernel $\mathcal{J}$ which consists of the two-sided ideal generated by all elements of $\Z[C_m]$ which map to vanishing sums of $m$-th roots of unity. The ideal $\mathcal{I}=\langle [G]\rangle$ is clearly contained in $\mathcal{J}$ since $\sum_{i=0}^{m-1}\zeta_m^i=0$, and unless $m$ is prime $\mathcal{I}$ is properly contained in $\mathcal{J}$. For more on GHMs and their relationship to projective planes see Appendix \ref{app-GHMs}.\\

\index{Scarpis construction}
The following result is due to Scarpis \cite{Scarpis}, who proved in 1898 that if an Hadamard matrix of order $p+1$ exists for $p$ prime, then there is an Hadamard matrix of order $p(p+1)$. His construction seems to have been motivated by an analysis of Hadamard's construction of $\pm 1$ Hadamard matrices at orders $12$ and $20$ \cite{Eliahou-Hadamardblog}. We mention that the Scarpis Construction seems to have been largely unnoticed in the literature. In 2012 William Orrick wrote an expository article \cite{Orrick-Scarpis} about it, and some years after Đokovi\'{c} \cite{Djokovic-Scarpis} generalised the Scarpis result to prime powers. Seberry essentially rediscovered the Scarpis Construction in \cite{Seberry-GeneralisedHadamard-1980}, based on ideas developed in Rajkundlia's PhD thesis \cite{Rajkundlia-Thesis, Rajkundlia-Article}. She stated her result in terms of generalised Hadamard matrices and only stated existence of $\GH(q(q+1),\EA(q))$ matrices provided that $q$ and $q+1$ are both prime powers, where $\EA(q)$ denotes the elementary abelian group of order $q$. It appears that Scarpis' technique had never been applied to obtain the existence of general Butson Hadamard matrices before. We state here our own version of this result, which is more general than the ones previously found in the literature. First we set some notation:
\begin{itemize}
\item If $v$ is an $n$-vector, then $D=\diag(v)$ denotes the $n\times n$ diagonal matrix such that $d_{ii}=v_i$.
\item $A^{(r\times s)} =  J_{r,s}\otimes A=[J_{r,s}a_{ij}]_{i,j}$,
\item  Let $G$ be a group of order $n$ and $\rho: G\rightarrow\GL_n(\C)$ be the regular representation of $G$. If $M$ is an $a\times b$ matrix with entries in $G$ and $A$ is an $n\times r$ matrix then $A_{(M)}$ is the $an\times br$ block matrix whose block $(i,j)$ is $\rho(m_{ij})A$. Namely
\[A_{(M)} = \left[
\begin{array}{ccc}
\rho(m_{11})A &\dots & \rho(m_{1b})A\\
\vdots & &\vdots\\
\rho(m_{a1})A &\dots & \rho(m_{ab})A
\end{array}
\right]
\]
\end{itemize}
Any Hadamard matrix is equivalent to one whose first row and column consist of the all-ones vector. Such an Hadamard matrix is called \textit{dephased}\index{Hadamard matrix!dephased}, i.e. if $H$ is a dephased Hadamard matrix, then 
\[\left[
\begin{array}{c|c}
1 & \mathbf{1}^{\intercal}  \\
\hline
\mathbf{1} & C  \\
\end{array}
\right]
\]
The matrix $C$ is called the \textit{core} of $H$.\index{Hadamard matrix! core of an }

\index{Scarpis construction}
\begin{theorem}[Scarpis' Construction] \label{thm-ScarpisConstruction} Let $H$ be a $\BH(n+1,m)$, and suppose that there is a $\GH(n,G)$ where $|G|=n$. Then there is a $\BH(n(n+1),m)$ matrix.
\end{theorem}
\begin{proof}
Without loss of generality, suppose that $H$ is dephased and let $C$ be the core of $H$. Let $M$ be a $\GH(n,G)$, without loss of generality we may assume that the first row of $M$ has all entries equal to $1_G$. Denote by $c_i$ the $i$-th row of $C$ and by $k_i$ the $i$-th column of $C$. Then the matrix
\[
K:=\left[
\begin{array}{cc}
1^{(n\times n)} & C^{(1\times n)}\\
C^{(n\times 1)} &C_{(M)}
\end{array}
\right]=\left[
\begin{array}{c|ccc}
J_n & \diag(k_1)J_n& \dots  &\diag(k_n)J_n\\
\hline
 J_{n,1}\otimes c_1  & \rho(m_{11})C & \dots & \rho(m_{1n})C\\
 \vdots & \vdots & \ddots & \vdots\\
 J_{n,1}\otimes c_n & \rho(m_{n1})C & \dots & \rho(m_{nn})C
\end{array}
\right]
\]
is a $\BH(n(n+1),m)$. Since $C$ is the core of an Hadamard matrix of order $n+1$ we find that $CC^*=(n+1)I_{n}-J_{n}$ and $CJ_n=J_nC=-J_n$. It is easy to see that $KK^*=n(n+1)I_{n(n+1)}$ computing the product by blocks. The inner product of the first row block with any other row block is 
\begin{align*}
(\mathbf{1}_n\cdot c_i)J_n + \sum_{j} \diag(k_j)J_nC^*\rho(m_{ij}^{-1})&=-J_n+\sum_j [\diag(k_j)(-J_n)\rho(m_{ij}^{-1})]\\
&=-J_n-\left(\sum_{j}\diag(k_j)\right)J_n=-J_n+J_n=0.
\end{align*}
The inner product of two distinct row-blocks different than the first is 
\begin{align*}(c_i\cdot c_j)J_n + \sum_k \rho(m_{ik})CC^*\rho(m_{jk}^{-1})&=-J_n + \sum_k [(n+1)\rho(m_{ik})\rho(m_{jk}^{-1})-J_n]\\
&=-J_n+(n+1)J_n-nJ_n=0.
\end{align*}
Finally the inner product of the first block-row with itself is $nJ_n+\sum CC^*=nJ_n+n(n+1)I_n-nJ_n=n(n+1)I_n$. And the inner product of any other block-row with itself is 
\[(c_i\cdot c_i)J_n+\sum_j \rho(m_{ij})CC^*\rho(m_{ij}^{-1})=nJ_n+\sum_j[(n+1)\rho(m_{ij}m_{ij}^{-1})-J_n]=n(n+1)I_{n}.\qedhere\]
\end{proof}

\subsection{de Launey's construction}
In the early 1980s Warwick de Launey introduced a construction for $\GH(q^t(q+1),\EA(q+1))$ where $t\geq 1$ and both $q$ and $q+1$ are prime powers, see \cite{DeLauney-GHMSurvey}. This is a generalisation of the Scarpis Construction as proved by Seberry \cite{Seberry-GeneralisedHadamard-1980}. However this result by de Launey was never published, and instead appeared in a preprint which the present author has not been able to find. A particularly interesting corollary to this result is that there is a $\BH(2^t \cdot 3, 3)$ for every $t\geq 0$. As we will see in the following section, there is evidence to believe that $\BH(2^t p,p)$ exists for every prime $p$. Even more strongly it  appears likely that $\BH(hp,p)$ should exist whenever $h$ is the order of a real Hadamard matrix. \\

The key idea of the construction of de Launey is to generate a recursive sequence of ``\textit{generalised cores}'', and by this we mean matrices that play the analogous role of the core $K$ of a $\BH(q+1,q+1)$ in the Scarpis construction. We illustrate this idea in the following example with $q=2$.\\

The core, say $K_1$, of a $\BH(3,3)$ satisfies the Gram matrix equation
\[K_1K_1^*=\begin{bmatrix}
2 & -1\\
-1 & 2
\end{bmatrix}.\]
We wish to define a sequence of matrices $K_t$ of order $2^t$ with entries in $\{1,\omega,\omega^2\}$ satisfying
\begin{align*}K_t K_t^* &= (-1)^t (J_2-I_2) \otimes J_{2^{t-1}}+2\diag(K_{t-1}K_{t-1}^*,K_{t-1}K_{t-1}^*)
\\
&=
\left[
\begin{array}{c|c}
2 K_{t-1}K_{t-1}^* & (-1)^tJ_{2^{t-1}}\\
\hline
(-1)^t J_{2^{t-1}} & 2 K_{t-1}K_{t-1}^*
\end{array}
\right].
\end{align*}
Thus letting $K_0=1$, after $K_0K_0^*=[1]$, we have the following sequence of Gram equations:
\begin{align*}
K_1K_1^*=\begin{bmatrix}
2 & -1\\
-1 & 2
\end{bmatrix},
K_2K_2^*=\begin{bmatrix}
4 & -2 & 1 & 1\\
-2& 4 & 1 & 1\\
1 & 1 & 4 & -2\\
1 & 1 & -2 & 4
\end{bmatrix},
K_3K_3^*=\left[\begin{smallmatrix}
8 & -4 & 2 & 2 & -1 & -1 & -1 & -1\\
-4& 8 & 2 & 2 & -1 & -1 & -1 & -1\\
2 & 2 & 8 & -4& -1 & -1 & -1 & -1\\
2 & 2 &-4 & 8 & -1 & -1 & -1 & -1\\
-1&-1 &-1 &-1 & 8 & -4 & 2 & 2\\
-1&-1 &-1 &-1 & -4 & 8 & 2 & 2\\
-1&-1 &-1 &-1 & 2 & 2 & 8 & -4\\
-1&-1 &-1 &-1 & 2 & 2 & -4 & 8
\end{smallmatrix}\right],
\end{align*}
and so on. Another property of the core of an Hadamard matrix is that $KJ=JK=-J$. In our case this is equivalent to $K_1J_2=J_2K_1=K_1(K_0\otimes J_2)=(-1)J_2,$ and this property generalises to $K_t(K_{t-1}^*\otimes J_2)=(K_{t-1}^*\otimes J_2)K_t=(-1)^tJ_{2^t}$.

\begin{lemma} If $\{K_t\}$ is a sequence of matrices of size $2^t$ with $K_0=1$ satisfying $K_t(K_{t-1}^*\otimes J_2)=(K_{t-1}^*\otimes J_2)K_t=(-1)^t J_{2^t}$, then
\[K_t J_{2^t}=J_{2^t}K_t=\alpha_t J_{2^t}\]
where $\alpha_t$ satisfies the recurrence $\alpha_0=1$ and $\alpha_t=(-1)^t2^{t-1}/\alpha_{t-1}$.
\end{lemma}
\begin{proof}
We proceed by induction. First notice that $K_0=1$ implies $K_0J_{2^0}=1=\alpha_0J_{2^0}$. Now assume that $K_{t-1}J_{2^{t-1}}=J_{2^{t-1}}K_{t-1}=\alpha_{t-1}J_{2^{t-1}}$, then we find that $K_{t-1}^*J_{2^{t-1}}=(J_{2^{t-1}}K_{t-1})^*=K_{t-1}J_{2^{t-1}}=\alpha_{t-1}J_{2^{t-1}}$ and
\begin{align*}
(-1)^t2^{t}J_{2^t}&=(K_t(K_{t-1}^*\otimes J_2))J_{2^{t}}\\
&=K_t(K_{t-1}^*\otimes J_2)(J_{2^{t-1}}\otimes J_2)\\
&=2K_t(K_{t-1}J_{2^{t-1}}\otimes J_2)\\
&=2\alpha_{t-1}K_tJ_{2^t}.
\end{align*}
Hence $K_tJ_{2^t}=\alpha_t J_{2^t}$ with $\alpha_t=(-1)^t2^{t-1}/\alpha_{t-1}$. In a similar way we have
\begin{align*}
(-1)^{t}2^tJ_{2^t}&=J_{2^t}(K_{t-1}^*\otimes J_2)K_t\\
&=(J_{2^{t-1}}\otimes J_2)(K_{t-1}^*\otimes J_2)K_t\\
&=2(J_{2^{t-1}}K_{t-1}^*\otimes J_2)K_t\\
&=2\alpha_{t-1}J_{2^{t}}K_t.
\end{align*}
From which it follows that $J_{2^t}K_t=K_tJ_{2^t}=\alpha_t J_{2^t}$.\qedhere
\end{proof}
In particular all of our matrices $K_t$ have constant row-sum and the row-sum is given by the sequence
\[\alpha_t:\ 1,-1,-2,2,4,-4,-8,8,16,-16,-32,32,\dots\]
\begin{proposition}
If there is a sequence of matrices $K_t$ of order $2^t$  for $t\geq 0$ with entries in $\{1,\omega,\omega^2\}$, satisfying $K_t(K_{t-1}^*\otimes J_2)=(-1)^{t}J_{2^{t}}$, and the following recurrent Gram matrix equations $K_0K_0^*=1$,
\[K_tK_t^*= \left[
\begin{array}{cc}
2 K_{t-1}K_{t-1}^* & (-1)^t J_{2^{t-1}}\\
(-1)^t J_{2^{t-1}} & 2K_{t-1}K_{t-1}^*
\end{array}
\right]\text{ for } t\geq 1.\]
 Then the matrix
\[H_t=\left[
\begin{array}{cc}
K_{t-2}\otimes J_2 & K_{t-1}\otimes J_{1,2}\\
K_{t-1}\otimes J_{2,1} & K_t
\end{array}
\right],
\]
is a $\BH(2^t\cdot 3, 3)$ for every $t\geq 2$.
\end{proposition}
\begin{proof}
We prove by induction that the following relations hold for $t\geq 2$:
\begin{align}
&(K_{t-1}K_{t-1}^*)\otimes J_2 + K_{t}K_{t}^* = (2^t\cdot 3) I_{2^{t}}, \text{ and }\label{rel1}\\
&(K_{t-1}\otimes J_{2,1})(K_{t-2}^*\otimes J_2) + K_t(K_{t-1}^*\otimes J_{2,1})=0.\label{rel2}
\end{align}
To prove (\ref{rel1}) notice that when $t=1$ we have
\[(K_0K_0^*)\otimes J_2 + K_1K_1^*=J_2+K_1K_1^*=\begin{bmatrix}
1 & 1\\
1 & 1
\end{bmatrix}
+\begin{bmatrix}
2 & -1\\
-1 & 2
\end{bmatrix}
=3I_2.
\]
Now assume that the first relation is true for some $t\geq 1$, then 
\begin{align*}
(K_{t}K_{t}^*)\otimes J_2 + K_{t+1}K_{t+1}^* &= \begin{bmatrix}
2(K_{t-1}K_{t-1}^*)\otimes J_2 & (-1)^{t-1}J_{2^{t-1}}\otimes J_2\\
(-1)^{t-1}J_{2^{t-1}}\otimes J_2 & 2(K_{t-1}K_{t-1}^*)\otimes J_2
\end{bmatrix}
+\begin{bmatrix}
2K_tK_t^*  & (-1)^{t}J_{2^t}\\
(-1)^{t}J_{2^t} & 2K_tK_t^*
\end{bmatrix}\\
&=\begin{bmatrix}
2((K_{t-1}K_{t-1}^*)\otimes J_2 + K_tK_t^*) & \mathbf{0}\\
\mathbf{0} &  2((K_{t-1}K_{t-1}^*)\otimes J_2 + K_tK_t^*)
\end{bmatrix}\\
&=2^{t+1}\cdot 3 I_{2^{t+1}}.
\end{align*}

To prove (\ref{rel2}) we note first that since $K_t(K_{t-1}^*\otimes J_2)=(-1)^{t}J_{2^{t}}$, then $K_{t}(K_{t-1}^*\otimes J_{2,1})=(-1)^{t}J_{2^{t},2^{t-1}}$. This is because every column of $K_{t-1}^*\otimes J_{2,1}$ is a column of $K_{t-1}^*\otimes J_{2}$. Likewise 
\[(K_{t-1}\otimes J_{2,1})(K_{t-2}^*\otimes J_2)=(-1)^{t-1}J_{2^{t},2^{t-1}},\]
follows from $K_{t-1}(K_{t-2}^*\otimes J_2)=(-1)^{t-1}J_{2^{t-1}}$ since every row of $K_{t-1}\otimes J_{2,1}$ is a row of $K_{t-1}$. Therefore we find that
\[(K_{t-1}\otimes J_{2,1})(K_{t-2}^*\otimes J_2)+K_{t}(K_{t-1}^*\otimes J_{2,1})=(-1)^{t-1}J_{2^{t},2^{t-1}}+(-1)^tJ_{2^{t},2^{t-1}}=0.\]
The two relations (\ref{rel1}) and (\ref{rel2}) that we just showed imply that $H_t$ is an Hadamard matrix.\qedhere
\end{proof}

\begin{theorem}[de Launey, \cite{DeLauney-GHMSurvey}]\label{thm-DeLauneyConstruction} For every $t\geq 0$ there exists a $\BH(2^t\cdot 3,3)$.
\end{theorem}

More strongly, de Launey shows the following using the same technique
\begin{theorem}[de Launey, \cite{DeLauney-GHMSurvey}]\label{thm-DeLauneyTheorem} If both $q$ and $q+1=p^f$ are prime powers, then for every $t\geq 0$ there exists a $\BH(q^t(q+1),p)$.
\end{theorem}

de Launey obtained the result above by constructing a sequence of matrices $K_t$ as specified above. The matrices constructed by de Launey consist of $2\times 2$ blocks given by the following plug-in construction for matrices with entries in the set $\{\pm 1,\pm \omega,\pm \omega^2\}$:
\begin{alignat*}{2}
1 & \mapsto \begin{bmatrix}
\omega & \omega^2\\
\omega^2 & \omega
\end{bmatrix} && \qquad\qquad -1 \mapsto \begin{bmatrix}
\omega^2 & \omega\\
\omega & \omega^2
\end{bmatrix} \\
\omega & \mapsto \begin{bmatrix}
\omega^2 & 1\\
1 & \omega^2
\end{bmatrix} && \qquad\qquad -\omega \mapsto \begin{bmatrix}
1 & \omega^2\\
\omega^2 & 1
\end{bmatrix} \\
\omega^2 & \mapsto \begin{bmatrix}
1 & \omega\\
\omega & 1
\end{bmatrix} && \qquad\qquad -\omega^2 \mapsto \begin{bmatrix}
\omega & 1\\
1 & \omega
\end{bmatrix}
\end{alignat*}

In this way $K_t$ is specified by a pair of matrices $A_t $ and $B_t$ where $A_t$ has entries in the third roots of unity and $B_t$ is a $\pm 1$ matrix. If we denote by $\varphi$ the mapping above and by $M^{\varphi}$ the block matrix $[\varphi(m_{ij})]_{ij}$ where $m_{ij}\in\{\pm 1,\pm \omega, \pm \omega^2\}$ then $K_t=(A_t\circ B_t)^{\varphi}$.\\

In his survey \cite{DeLauney-GHMSurvey}, de Launey determines the value of $A_t$ in terms of $K_{t-2}$ but leaves $B_t$ unspecified, and it is claimed that for the given $A_t$ there is a choice of $B_t$ that makes $K_t$ satisfy the required Gram matrix equation. de Launey omits the proof of the result in \cite{DeLauney-GHMSurvey}. We carried a computer search to find a complete list of values for $B_t$ that will give examples of matrices $K_t$ at small orders. An interesting thing to remark is that in our exhaustive search, all solutions for $B_t$ are real Hadamard matrices, furthermore the number of solutions found always turned out to be a power of $2$. See Appendix \ref{app-MatrixTables} for examples of a $\BH(12,3)$, a $\BH(24,3)$, and a $\BH(48,3)$.

\begin{research-problem}\normalfont\label{resp-genDeLauney}
Seberry's Construction \cite{Seberry-GeneralisedHadamard-1980}, assumed that both $q$ and $q+1$ are prime powers to show existence of $\GHM(q(q+1),\EA(q+1))$ matrices, and de Launey's result develops further this idea. We have seen (Theorem \ref{thm-ScarpisConstruction}) that, more generally, we only need the existence of a $\GHM(n,n)$ to show that there is a $\BH(n(n+1),m)$ whenever there is a $\BH(n+1,m)$. Generalise the de Launey Construction to remove the assumption that $q+1$ is a prime power. In other words show that there is a $\BH(q^t(q+1),m)$ whenever there is a $\BH(q+1,m)$ and $q$ is a prime power.
\end{research-problem}

\section{Morphisms of Hadamard matrices}
\index{morphism!complete}
Let $\mathcal{X}$ and $\mathcal{Y}$ be two families of Hadamard matrices, a \textit{complete morphism} from $\mathcal{X}$ to $\mathcal{Y}$ is a mapping $\mathcal{X}\rightarrow\mathcal{Y}$. The examples of morphisms that we will consider come from embeddings of matrix algebras, and most involve infinite families of Butson Hadamard matrices. For example, given a fixed $M\in \BH(n,k)$, the tensor product construction can be seen as a complete morphism
\begin{align*}\bullet \otimes M :\BH(m,\ell)&\rightarrow\BH(nm,\lcm(\ell,k))\\
H &\mapsto H\otimes M
\end{align*}

A \textit{partial morphism} \index{morphism!partial} from $\mathcal{X}$ to $\mathcal{Y}$ is a function from a subset of $\mathcal{X}$ into $\mathcal{Y}$. One of the most important examples of a morphism of Hadamard matrices is the Turyn morphism

\begin{theorem}[Turyn, \cite{Turyn-Morphism}] \label{thm-TurynMorphism}\index{morphism!Turyn} There is a complete morphism from $\BH(n,4)$ to $\BH(2n,2)$.
\end{theorem}
\begin{proof}
Let $H$ be a $\BH(n,4)$, then every entry of $H$ is in the set $\{\pm 1,\pm i\}$, and we can write 
\[H=A+iB,\]
where $A$ and $B$ are $(0,\pm 1)$-matrices, and $A\circ B=0$. Furthermore, from the equation $HH^*=nI_n$ we find
\[(A+iB)(A^{\intercal}-iB^{\intercal})=AA^{\intercal}+BB^{\intercal}+i(-AB^{\intercal}+BA^{\intercal})=nI_n.\]
Therefore, $AB^{\intercal}=BA^{\intercal}$, and $AA^{\intercal}+BB^{\intercal}=nI_n$. Now let
\[
M=
\left[
\begin{array}{c|c}
A+B & -A+B \\
\hline
A-B & A+B \\
\end{array}
\right].
\]
Direct computation shows
\[MM^{\intercal}=
\left[
\begin{array}{c|c}
A+B & -A+B \\
\hline
A-B & A+B \\
\end{array}
\right]
\left[
\begin{array}{c|c}
A^{\intercal}+B^{\intercal} & A^{\intercal}+B^{\intercal} \\
\hline
-A^{\intercal}+B^{\intercal} & A^{\intercal}+B^{\intercal} \\
\end{array}
\right]=
\left[
\begin{array}{c|c}
2(AA^{\intercal}+BB^{\intercal}) & 2(-AB^{\intercal}+BA^{\intercal}) \\
\hline
2(AB^{\intercal}-BA^{\intercal}) & 2(AA^{\intercal}+BB^{\intercal}) \\
\end{array}
\right].
\]
This implies $MM^{\intercal}=2nI_{2n}$, and since $M$ is a $\pm 1$ matrix, it follows that $M\in \BH(2n,2)$.\qedhere
\end{proof}

The Turyn morphism can also be specified by a mapping as follows: Let $\varphi$ be the mapping from the set $\{\pm 1,\pm i\}$ to $2\times 2$ matrices with entries $\pm 1$ given by,
\begin{align*}
1 &\mapsto \begin{bmatrix}
1 & -\\
1 & 1
\end{bmatrix} &
i &\mapsto \begin{bmatrix}
1 & 1\\
- & 1
\end{bmatrix}\\
-1 &\mapsto\begin{bmatrix}
- & 1\\
- & -
\end{bmatrix} &
-i &\mapsto\begin{bmatrix}
- & -\\
1 & -
\end{bmatrix}
\end{align*}
Then, for every $H\in\BH(n,4)$, the block matrix $H^{\varphi}$ obtained by applying $\varphi$ to $H$ entrywise is a $\BH(2n,2)$. The relationship between the construction of Theorem \ref{thm-TurynMorphism} and the plug-in construction using $\varphi$ is established by the \textit{Kronecker shuffle} matrix. Recall that given two square matrices $A$ and $B$ of orders $n$ and $m$, the Kronecker products $A\times B$ and $B\times A$ are similar, in particular there exists a permutation matrix $P_{mn}$ such that
\[P_{mn}(A\otimes B)P_{mn}^{-1}=B\otimes A.\]
\begin{proposition}[Rose, \cite{Rose-IdentitiesFourier}]\normalfont For any $n,m\in\N$ the $nm\times nm$ Kronecker shuffle matrix $P_{mn}$ is 
\[P_{mn}=\left[\delta(i,\lfloor j/n\rfloor + m\cdot j_n)\right]_{0\leq i,j\leq mn-1},\]
where $\delta(x,y)$ is the Kronecker delta, and $j_n\in\{0,\dots,n-1\}$ satisfies $j_n\equiv j\pmod{n}$.
\end{proposition}
If $M$ is an $mn\times mn$ matrix consisting of diagonal blocks, then $P_{mn}MP_{mn}^{-1}$ has $m\times m$ blocks in the diagonal and zeros in every other block. If we let $M$ be the matrix obtained from $H\in \BH(n,4)$ as in Theorem \ref{thm-TurynMorphism}, and $M'=H^{\varphi}$, then $P_{2n} M'P_{2n}^{\intercal}=M$.

\begin{definition}\normalfont\index{Hadamard matrix!unreal}
An Hadamard matrix $H$ is called \textit{unreal} if each entry of $H$ is not in $\R$, i.e. if $h_{ij}\in\C-\R$ for all $i,j$.
\end{definition}

\begin{example}\normalfont 
The matrix 
\[\begin{bmatrix}
\omega &\omega^2 & \omega^2\\
\omega^2 & \omega & \omega^2\\
\omega^2 & \omega^2 & \omega
\end{bmatrix},\]
where $\omega^2+\omega+1=0$, is an unreal $\BH(3,3)$.
\end{example}

Define a map $\pi:\Q[\omega]\rightarrow \Mat_n(\Q)$ by $\pi(a+b\omega+c\omega^2)=2aI_4+bH+cH^{\intercal}$, where
\[H=\begin{bmatrix}
- & 1 & 1 & 1\\
- & - & 1 & -\\
- & - & - & 1\\
- & 1 & - & -
\end{bmatrix}.
\]

\begin{theorem}[Compton-Craigen-DeLauney, \cite{Compton-Craigen-DeLauney}]\label{thm-CCDLMorphism} There is a partial morphism from $\BH(n,6)$ to $\BH(n,2)$. Namely, if $H$ is an unreal $\BH(n,6)$, the matrix $H^{\pi}$ is a $\BH(4n,2)$.
\end{theorem}

The morphism in Theorem \ref{thm-CCDLMorphism} comes from a $\Q$-algebra isomorphism: Notice that $H^2=2H^{\intercal}$ and that $(H+I)^{\intercal}=-(H+I)$, therefore $(\frac{1}{2}H)^2+\frac{1}{2}H+I=0$. In other words, the minimal polynomial of $\frac{1}{2}H$ is $T^2+T+1$, which coincides with the minimal polynomial of a primitive third-root of unity $\omega$. Then, we have the following  $\Q$-algebra isomorphism
\[\Q[\frac{1}{2}H]\simeq \Q[T]/(T^2+T+1)\simeq \Q[\omega].\]
This isomorphism is given explicitly by
\[\omega\mapsto \frac{1}{2}H, \text{ and } \omega^2\mapsto(\frac{1}{2}H)^2=\frac{1}{2}H^{\intercal}.\]
Multiplying by $2$, we recover the mapping $\pi$. Clearly, to avoid zeros in the resulting matrix, we have to restrict ourselves to unreal $\BH(n,6)$ matrices.\\

Egan and Ó Catháin found a general construction for morphisms between Butson Hadamard matrices that includes both the morphisms of Theorem \ref{thm-TurynMorphism} and Theorem \ref{thm-CCDLMorphism}.

\begin{definition}\normalfont  Let $X,Y\subseteq\mu_k=\{1,\zeta_k,\dots,\zeta_k^{k-1}\}$. Let $H\in\BH(n,k)$, and suppose that every entry of $H$ is contained in $X$. Let $M\in\BH(m,\ell)$ be such that every eigenvalue of $\frac{1}{\sqrt{m}}M$ is contained in $Y$. The pair $(H,M)$ is called $(X,Y)$-\textit{sound} if and only if
\begin{itemize}
\item[(i)] For each $\zeta_k^i\in X$, we have $\sqrt{m}(\frac{1}{\sqrt{m}}M)^i\in\BH(m,\ell)$.
\item[(ii)] For each $\zeta_k^j\in Y$, $H^{(j)}\in \BH(n,k)$,
\end{itemize}
where $H^{(j)}$ is the entrywise $j$-th power of $H$. A pair $(H,M)$ is called \textit{sound} if and only if there exist $X,Y\subseteq \mu_k$ such that $(H,M)$ is $(X,Y)$-sound.
\end{definition}

\begin{theorem}[Egan - Ó Catháin, Theorem 4 \cite{Egan-OCathain-Morphisms}]\label{thm-SoundMorphisms} Let $H\in\BH(n,k)$ and $M\in \BH(m,\ell)$. Let $\phi$ be the mapping given by
\[\zeta_k^{i}\mapsto\sqrt{m}\left(\frac{1}{\sqrt{m}}M\right)^i.\]
 If $(H,M)$ is a sound pair, then $H^{\phi}\in\BH(mn,\ell)$. 
\end{theorem}

\begin{example} \normalfont Let 
\[M_8=\begin{bmatrix}
1 & 1\\
- & 1
\end{bmatrix}.
\]
Then, the set of eigenvalues of $\frac{1}{\sqrt{2}}M_8$ is $Y=\{\zeta_8,\zeta_8^7\}\subset \mu_8$. Additionally, the matrices $\sqrt{2}(\frac{1}{\sqrt{2}}M_8)^3$, $\sqrt{2}(\frac{1}{\sqrt{2}}M_8)^5$, and $\sqrt{2}(\frac{1}{\sqrt{2}}M_8)^7$ are all $\BH(2,2)$ matrices.  Given any matrix $H\in \BH(n,4)$, we have that $(\zeta_8 H)^{(7)}=\overline{(\zeta_8 H)}$, where $\overline{H}$ denotes the entrywise complex conjugate of $H$. Clearly, both $\zeta_8 H$ and $\overline{(\zeta_8 H)}$ are $\BH(n,8)$ matrices, so the pair $(\zeta_8 H,M_8)$ is $(\{\zeta_8,\zeta_8^3,\zeta_8^5,\zeta_8^7\},\{\zeta_8,\zeta_8^7\})$-sound. Using Theorem \ref{thm-SoundMorphisms}, we find that $(\zeta_8 H)^{\phi}$ is a $\BH(2n,2)$ matrix, so we recover the Turyn morphism $\BH(n,4)\rightarrow\BH(2n,2)$.
\end{example}
Similarly, one can recover the morphism of Theorem \ref{thm-CCDLMorphism} using this technique. Additionally, in \cite{Egan-OCathain-Morphisms}, Egan and Ó Catháin obtain the following morphism
\begin{corollary}\normalfont The matrix 
\[
M_5 =
\begin{bmatrix}
-1 & -1 & -1 & -1 \\
\phantom{-}1 & -1 & \phantom{-}1 & -1 \\
\phantom{-}i & \phantom{-}i & -i & -i \\
\phantom{-}i & -i & -i & \phantom{-}i \\
\end{bmatrix}
\]
induces a partial morphism $\BH(n,5)\mapsto \BH(4n,4)$ defined on unreal $\BH(n,5)$ matrices. 
\end{corollary}
Combining this with the Turyn morphism we obtain a partial morphism $\BH(n,5)\mapsto\BH(8n,2)$.\\

The results in \cite{Egan-OCathain-Morphisms} where further developed by \"{O}steg\aa rd and Paavola in \cite{Ostergard-Paavola-Morphisms} and in subsequent papers of Egan, Ó Catháin and Swartz \cite{Egan-OCathain-Swartz-Spectra, OCathain-Swartz-Morphisms}. We mention the following interesting result:
\begin{theorem}[Ó Catháin-Swartz, \cite{OCathain-Swartz-Morphisms}] Let $k=mt$ and suppose that each prime divisor of $k$ also divides $t$. Then, there is a complete morphism $\BH(n,mt)\rightarrow\BH(mn,t)$.
\end{theorem}

In a paper in collaboration with Heikoop, Pugmire and Ó Catháin \cite{QUH-paper}, we found the first example of a morphism from a non-Butson family of Hadamard matrices to real Hadamard matrices. Recall that a $\QUH(n,m)$ matrix is an Hadamard matrix with entries in the set\index{Hadamard matrix!quaternary unit}
\[\left\{\frac{1\pm \sqrt{-m}}{\sqrt{m+1}},\frac{-1\pm\sqrt{-m}}{\sqrt{m+1}}\right\}.\]
A \textit{skew Hadamard matrix} is a real Hadamard matrix such that $H-I$ is a skew matrix, i.e. $(H-I)^{\intercal}=-(H-I)$.\index{Hadamard matrix!skew}
\begin{theorem}[\cite{QUH-paper}]\label{thm-QUHMorphism}
If there exists a skew Hadamard matrix of order $m+1$, then there is a morphism $\QUH(n,m)\rightarrow\BH(n(m+1),2)$.
\end{theorem}
\begin{proof}
Let $H$ be a $\QUH(n,m)$, then we can write
\[H=\frac{1}{\sqrt{m+1}}A+\frac{\sqrt{-m}}{\sqrt{m+1}}B,\]
where $A$ and $B$ are $\pm 1$ matrices of order $n$. From $HH^*=nI_n$, it follows that
\[AB^{\intercal}=BA^{\intercal},\text{ and } AA^{\intercal}+mBB^{\intercal}=n(m+1)I_n.\]
Let $S=(s_{ij})$ be a skew Hadamard matrix of order $m+1$. Let $M$ be the block matrix with blocks equal to $A$ along the diagonal, and whose off-diagonal block in position $[i,j]$ is $s_{ij}B$ for $1\leq i,j\leq n$, i.e. we have
\[
M=
\begin{bmatrix}
    A & s_{12}B & \cdots & s_{1n}B \\
    -s_{12}B & A & \cdots & s_{2n}B \\
    \vdots & \vdots & \ddots & \vdots \\
    -s_{1n}B & -s_{2n}B & \cdots & A
\end{bmatrix}
\]
The inner product of any row block with itself $AA^{\intercal}+mBB^{\intercal}=n(m+1)I_n$. And the inner product two distinct row blocks, indexed $r$ and $t$ is
\[s_{rt}AB^{\intercal}+s_{tr}BA^{\intercal}+\sum_{j\neq r,t}s_{rj}s_{tj}BB^{\intercal}=0.\]
To see this, notice that from the skewness of $S$, we know that $s_{rt}=-s_{tr}$, so the term $s_{rt}AB^{\intercal}+s_{tr}BA^{\intercal}$ vanishes. Finally, since $s_{rt}+s_{tr}=0$, the orthogonality of distinct rows of $S$ implies that $0=s_{rt}+s_{tr}+\sum_{j\neq r,t}s_{rj}s_{tj}=\sum_{j\neq r,t}s_{rj}s_{tj}$. Therefore we have that $MM^{\intercal}=n(m+1)I_{n(m+1)}$.\qedhere
\end{proof}

The morphism above can also be found from a matrix algebra isomorphism. If $H$ is a skew Hadamard matrix of order $m+1$, then $(H-I)^{\intercal}=H^{\intercal}-I=-(H-I)$. Multiplying this equation by $H$ we find
\[(m+1)I-H= HH^{\intercal}-H=-H^2+H,\]
from which it follows that
\[H^2=2H-(m+1)I.\]
From here, one can deduce that the minimal polynomial of $\frac{1}{\sqrt{m+1}}H$ is equal to $T^4+\frac{2(m-1)}{m+1}T^2+1$, which coincides with the minimal polynomial of the entries $\{\frac{\pm 1\pm\sqrt{-m}}{\sqrt{m+1}}\}$. This establishes the $\Q$-algebra isomorphism
\[\Q\left[\frac{1}{\sqrt{m+1}}H\right]\simeq\Q[T]/(T^4+\frac{2(m-1)}{m+1}T^2+1)\simeq \Q\left[\frac{\pm 1\pm \sqrt{-m}}{\sqrt{m+1}}\right].\]
From this isomorphism, the existence of the morphism follows, see \cite{QUH-paper}.\\

Fender, Kharaghani and Suda in \cite{Fender-Kharaghani-Suda}, provide a construction for $\QUH(q^t,q)$ for all $t\geq 1$, where $q\equiv 3\pmod{4}$ is a prime power. Using the existence of QUH matrices at those orders we obtain an infinite family of real Hadamard matrices, first discovered by Mukhopadhyay in \cite{Mukhopadhyay-Hadamard} using different methods.
\begin{corollary}[\cite{QUH-paper}]\normalfont Let $q\equiv 3\pmod 4$ be an odd prime power. Then, for any integer $t\geq 1$ there exists a real Hadamard matrix of order $q^{t+1}+q^{t}$.
\end{corollary}
\begin{proof}
This is a consequence of the fact that the (type 1) Paley matrix is a skew-Hadamard matrix of order $q+1$ for any prime power $q\equiv 3\pmod{4}$: Let $Q_q$ be the following matrix indexed by elements of $\F_q$,
\[[Q_q]_{xy}=
\begin{cases}
+1 & \text{ if } x-y\text{ is a square in } \F_q^{\times}\\
-1 &\text{ if } x-y\text{ is not a square in } \F_q^{\times}\\
0 & \text{ if } x-y=0\\
\end{cases}.
\]
Then for $q\equiv 3\pmod{4}$, we have that $-1$ is not a square in $\F_q$, and so $x-y$ is a non-zero square in $\F_q$ if and only if $y-x$ is a non-square in $\F_q$. This implies that $Q_q^{\intercal}=-Q_q$. Bordering the matrix $Q_q+I_q$ with a row of $+1$s and a column of $-1$s we obtain a skew Hadamard matrix of order $q+1$, see Lemma 2.4. of \cite{Horadam-HadamardBook}. In \cite{Fender-Kharaghani-Suda}, the authors show that there is a $\QUH(q^t,q)$ for all $q\equiv 3\pmod{4}$ and $t$ a positive integer. Therefore, the morphism in Theorem \ref{thm-QUHMorphism} implies the existence of a $\BH(q^t(q+1),2)$, i.e. there exists a real Hadamard matrix of order $q^{t+1}+q^{t}$.\qedhere
\end{proof}

Notice that this settles the case $m=2$ in Research problem \ref{resp-genDeLauney}.

\section{BH matrices at doubly even orders}\label{sec-BHDoublyEven}
In this section we give an account of results of existence for  $\BH(hp,p)$ matrices, where $p$ is a prime and $h$ is the order of a real Hadamard matrix. The first result of this type appeared in 1962 and is due to Butson \cite{Butson}, who showed the existence of $\BH(2p,p)$ matrices. When looking for Butson matrices one needs a strategy for obtaining cancellation in the inner products of rows. In the previous section we saw methods that make use of cores of Hadamard matrices. The main idea introduced by Butson was to obtain cancellations by splitting the vanishing sum $\sum_k \zeta_p^k$ into two parts: one involving quadratic residues in $\F_p$ and one involving quadratic non-residues. Butson's result was rediscovered in 1979 by Jungnickel in \cite{Jungnickel-GHMs}. Jungnickel's proof expresses the matrices involved in the construction in terms of polynomial functions in two variables. These polynomial functions express twists of the Fourier matrix $F_p$ (which is represented by the polynomial $xy$) by quadratic residues and non-residues so that the cancellation occurs in the same fashion as in Butson's proof. The method of Jungnickel was subsequently expanded by Dawson in 1985 \cite{Dawson-BH4p} where he showed the existence of $\BH(4p,p)$ matrices for all primes $p$. There is a nice early account of these results given in the survey by de Launey on Generalised Hadamard matrices \cite{DeLauney-GHMSurvey}. These investigations were further expanded by de Launey and Dawson in \cite{DeLauney-Dawson-BH8p} where they showed existence of $\BH(8p,p)$ matrices for all $p>19$, and culminated in 1994 with their result on the asymptotic existence of Butson Hadamard matrices \cite{DeLauney-Dawson-AsymptoticExistence}. More recently in 2013, Szöllősi described a new approach to these results using the language of mutually unbiased bases and Gauss sums \cite{Szollosi-MUBs-BH}. The approach using MUBs is more conceptual, and it is easier to see how the problem of constructing a $\BH(hp,p)$ matrix can be reduced to a number-theoretical problem. Namely, determining the occurrence of certain square and non-square patterns over the finite field $\F_p$. \\

Our contribution in this section is a computational result that shows the existence of $\BH(12p,p)$ matrices for all $p>263$, which is a significant improvement over the best previously known lower bound of $p>(10\cdot 2^{10})^2$.

\subsection{Gauss sums}
We begin with a few preliminary results on \textit{Gauss sums}, see Chapters 6 and 8 of \cite{Ireland-Rosen} or the survey by Berndt and Evans \cite{Berndt-Evans-GaussSumsSurvey} for more on the subject.\\\index{Gauss sums}

Over a finite field $\F_p$ of prime order $p$, we can define linear characters similar to those in Definition  \ref{def-MultCharacter}.\index{character} Namely, a \textit{character} of $\F_q$ is a character of the multiplicative group $\F_q^{\times}$, i.e. a homomorphism $\chi:\F_q^{\times}\rightarrow \C^{\times}$. The only difference with characters over finite groups is that we extend the domain of $\chi$ from $\F_q^{\times}$ to the whole $\F_q$.
The character $\varepsilon(x)=1$ for all $x\in \F_p$ is called the \textit{trivial character} of $\F_p$. The non-trivial characters of $\F_q^{\times}$ are extended to the whole of $\F_p$ by letting $\chi(0)=0$. If a character $\chi\neq \varepsilon$ satisfies $\chi^k=\varepsilon$, we say that $\chi$ is a character of \textit{order} $k$. If in addition $\chi^i\neq \varepsilon$ for all $1\leq i\leq k-1$, we say that $\chi$ is a \textit{primitive character} of order $k$.\index{character!primitive}

\begin{example}\normalfont The \textit{Legendre symbol}\index{Legendre symbol}
\[\legendre{a}{p}= \begin{cases}
+1 & \text{ if }  a \text{ is a quadratic residue modulo } p\\
-1 & \text{ if }  a \text{ is a quadratic nonresidue modulo } p\\
0 & \text{ if } a=0
\end{cases}
\]
gives a primitive character of order $2$ on the field $\F_p$ for every odd prime $p$, by letting $\chi(a)=\legendre{a}{p}$. This character is called the \textit{quadratic character} of $\F_p$.\index{character!quadratic}
\end{example}

 Throughout this subsection we let $k$ be an integer and $p\equiv 1\pmod{k}$ a prime number. There are two definitions of Gauss sums of order $k$ in the literature, namely the sums
\[\mathcal{G}(k)=\sum_{x\in\F_p} \zeta_p^{x^k},\]
and the sums
\[G(\chi) = \sum_{x\in \F_p}\chi(x)\zeta_p^x,\]
where $\chi$ is a primitive character of $\F_p$ of order $k$. The latter sums are better behaved; for example one can show that $|G(\chi)|=\sqrt{p}$ for any such sum. We include a proof of this for completeness, for more details see Chapter 8 of \cite{Ireland-Rosen}. We introduce the following notation,
\[G(\chi; a)=\sum_{x\in\F_p}\chi(x)\zeta_p^{ax}.\]
In particular, $G(\chi)=G(\chi; 1)$, and by Lemma \ref{lemma-VanishingCharSum} we have that $G(\chi; 0)=0$. It is easy to check that $G(\chi; a)= \chi(a^{-1})G(\chi)$, for $a\neq 0$.
\begin{proposition}[cf. Proposition 8.2.2. \cite{Ireland-Rosen}] \normalfont If $\chi$ is a non-trivial character of $\F_p$, then $|G(\chi)|=\sqrt{p}$.
\end{proposition}
\begin{proof} We compute the expression $\sum_{a\in\F_p}G(\chi;a)\overline{G(\chi;a)}$ in two ways. On the one hand, we have that
for $a\neq 0$,
\[G(\chi; a)\overline{G(\chi;a)}=\chi(a^{-1})\overline{\chi(a)}G(\chi)\overline{G(\chi)}=|G(\chi)|^2.\]
Therefore,
\[\sum_{a\in\F_p}G(\chi; a)\overline{G(\chi;a)}=(p-1)|G(\chi)|^2.\]
On the other hand, 
\[\sum_{a\in \F_p}G(\chi; a)\overline{G(\chi;a)}=\sum_a \sum_x\sum_y\chi(x)\overline{\chi(y)}\zeta_p^{a(x-y)}.\]
From Lemma \ref{lemma-VanishingCharSum}, we have that $\sum_a \zeta_p^{a(x-y)}=(p-1)\delta_{xy}$. Therefore,
\[\sum_{a\in\F_p}G(\chi; a)\overline{G(\chi;a)}=\sum_x\sum_y\chi(x)\overline{\chi(y)}(p-1)\delta_{xy}=\sum_{x\in\F_p}(p-1)=p(p-1).\]
It follows that $(p-1)|G(\chi)|^2=p(p-1)$, therefore $|G(\chi)|=\sqrt{p}$.\qedhere
\end{proof}

Even though, we just saw that $G(\chi)$ always has the same modulus, the modulus of $\mathcal{G}(k)$ depends on $k$. We note that although it was straightforward to determine $|G(\chi)|$, determining its phase is a much harder task, and even after centuries of attempts we do not yet have simple descriptions of the precise value of the Gauss sums $G(\chi)$ except at a few orders. For the quadratic Gauss sum, for example, we have the following.

\begin{theorem}[Gauss, Chapter 6, Theorem 1 \cite{Ireland-Rosen}] \label{thm-QuadraticGaussSum} Let $\chi$ be the quadratic character of $\F_p$, where $p$ is an odd prime. Then
\[G(\chi)=\begin{cases}
\sqrt{p} & \text{ if } p\equiv 1 \pmod{4}\\
i\sqrt{p} & \text{ if } p\equiv 3\pmod{4}
\end{cases}
\]\index{Gauss sums!quadratic}
\end{theorem}

Rather surprisingly, character sums count the number of points in certain algebraic varieties over $\F_p$. For more on this direction, see the paper by Weil \cite{Weil-EqsFF}, and Babai's lecture notes \cite{Babai-FourierAnalysisFG}.

\begin{lemma}[Proposition 8.1.5, \cite{Ireland-Rosen}]\normalfont \label{lemma-CharacterCounting} Let $a$ be an element of $\F_p$. Then the number of solutions $N(x^n=a)$ in $\F_p$ to the equation $x^n=a$ is given by 
\[N(x^n=a)=\sum_{\chi^n=\varepsilon} \chi(a),\]
where the sum is over all characters of $\F_p$ of order $n$.
\end{lemma}

\begin{proposition}[Proposition 8.1.3 \cite{Ireland-Rosen}]\normalfont\label{prop-DualGroupCyclicFF} Let $p$ be a prime, then the group of characters of $\F_p$ is cyclic of order $p-1$.
\end{proposition}
By Proposition \ref{prop-DualGroupCyclicFF}, if $\chi$ is a primitive character of order $n$ in $\F_p$, Lemma \ref{lemma-CharacterCounting} implies
\[N(x^n=a)=\sum_{i=0}^{n-1}\chi^i(a).\]

 With this we can obtain the following equation, relating both notions of Gauss sums.
 \begin{lemma}[cf. \cite{Berndt-Evans-GaussSumsSurvey}]\normalfont \label{lemma-TwoGaussEq} Let $\chi$ be a primitive character of order $k$ in $\F_p$, then
\[\mathcal{G}(k)=\sum_{x\in\F_p} (1+\chi(x)+\chi^2(x)+\dots+\chi^{k-1}(x))\zeta_p^{x}=\sum_{i=0}^{k-1}G(\chi^{i}).\]
\end{lemma}
\begin{proof}
The sum $\mathcal{G}(k)$ only involves summands $\zeta_p^x$ where $x=y^k$ for some $y\in\F_p$, therefore
\[\mathcal{G}(k)=\sum_{x\in\F_p}\zeta_p^{x^k}=\sum_{x\in\F_p}N(y^k=x)\zeta_p^x.\]
By Lemma \ref{lemma-CharacterCounting} it follows that,
\[\mathcal{G}(k)=\sum_{x\in\F_p}\zeta_p^{x^k}=\sum_{x\in\F_p}N(y^k=x)\zeta_p^x=\sum_{x\in\F_p}\left(\sum_{i=0}^{k-1}\chi^i(x)\right)\zeta_p^x=\sum_{i=0}^{k-1}G(\chi^i).\]
as we wanted to show.\qedhere
\end{proof}
In analogy with $G(\chi;a)$, we define $\mathcal{G}(k; a)$ as 
\[\mathcal{G}(k;a)=\sum_{x\in\F_p}\zeta_p^{ax^k}.\]
Recall that since $G(\chi; a)=\sum_x\chi(x)\zeta_p^{ax}$, we have for $a\neq 0$
\[\chi(a)G(\chi;a)=\chi(a)\sum_x\chi(x)\zeta_p^{ax}=\sum_x\chi(ax)\zeta_p^{ax}=G(\chi).\]
Similarly if we let $\mathcal{G}(k;a)=\sum_x\zeta_p^{ax^k}$ then we have
\[\mathcal{G}(k;a)=\sum_i G(\chi^i;a)=\sum_i\overline{\chi(a)}G(\chi^i).\] 

This implies that quadratic Gauss sums have the following additional property. 
\begin{lemma}\normalfont Let $\chi$ be the quadratic character of $\F_p$, then for $a\neq 0$ in $\F_p$
\[\mathcal{G}(2;a)=\legendre{a}{p}G(\chi)=G(\chi;a).\]
\end{lemma}
\begin{proof}
We have that $\mathcal{G}(2;a)=\sum_{x\in\F_p}\zeta_p^{ax^2}$. The number of elements $y\in\F_p$ such that $y=ax^2$ is equal to $N(x^2=a^{-1}y)$, therefore
\begin{align*}
\mathcal{G}(2;a)&=\sum_{x\in\F_p}\zeta_p^{ax^2}\\
&=\sum_{y\in\F_p}\left(1+\legendre{a^{-1}y}{p}\right)\zeta_p^{y}\\
&=\sum_{y\in\F_p}\zeta_p^{y}+\legendre{a^{-1}}{p}\sum_{y\in\F_p}\chi(y)\zeta_p^y\\
&=\legendre{a^{-1}}{p}G(\chi).
\end{align*} 
Using the fact that $\legendre{a^{-1}}{p}=\legendre{a}{p}$, the result follows.\qedhere
\end{proof}

This property does not generalise for characters of higher degrees, and hence the following calculation is particular to the quadratic case. Denote
\[\sigma_p=
\begin{cases}
\sqrt{p} & \text{ if } p\equiv 1\pmod{4}\\
i\sqrt{p} & \text{ if } p\equiv 3\pmod{4}\\
\end{cases}.
\]

\begin{lemma}[cf. \cite{Szollosi-MUBs-BH}]\normalfont \label{lemma-GeneralQGaussSum} Let $p$ be an odd prime, $a\in\F_p^{\times}$ and $b\in \F_p$. Then,
\[\sum_{x\in\F_p}\zeta_p^{ax^2+bx}=
\zeta_p^{-2^{p-3}a^{p-1}b^2}\legendre{a}{p}\sigma_p.
\]
\end{lemma}
\begin{proof}
Since $a\neq 0$, we can complete the square in the expression $ax^2+bx$, to find
\[ax^2+bx=a\left(x+\frac{b}{2a}\right)^2-\frac{b^2}{2^2a}.\]
Thus, 
\[ \sum_{x\in\F_p}\zeta_p^{ax^2+bx}=\zeta_p^{-\frac{b^2}{2^2a}}\sum_{x\in\F_p}\zeta_p^{a(x+\frac{b}{2a})^2}.\]
Now, the mapping $x\mapsto x+b/2a$ is invertible in $\F_q$, from which it follows that $\sum_{x} \zeta_p^{a(x+b/2a)^2}=\sum_x\zeta_p^{ax^2}$, so
\[\sum_{x\in\F_p}\zeta_p^{ax^2+bx}=\zeta_p^{-\frac{b^2}{2^2a}}\sum_{x\in\F_p}\zeta_p^{ax^2}=\zeta_p^{-\frac{b^2}{2^2a}}\mathcal{G}(2;a)=\zeta_p^{-\frac{b^2}{2^2a}}\legendre{a}{p}G(\chi),\]
thus by Theorem \ref{thm-QuadraticGaussSum} $G(\chi)=\sigma_p$, and
\[\mathcal{G}(2;a)=
\zeta_p^{-\frac{b^2}{2^2a}}\legendre{a}{p}\sigma_p
\]
From the equation $a^{-1}\equiv a^{p-2}\pmod{p}$ and $2^{-2}\equiv 2^{p-3}\pmod{p}$ we can write the result as in the statement.\qedhere
\end{proof}

\subsection{Butson's Theorem}


We present here Szöll\H{o}si's approach to the existence of $\BH(hp,p)$ matrices, where $h$ is the order of a real Hadamard matrix, see \cite{Szollosi-MUBs-BH}.\\

Throughout this section, we define $\Delta$ to be the following diagonal matrix,
\[\Delta=\diag(1,\zeta_p^{(1)^2},\zeta_p^{(2)^2},\dots,\zeta_p^{(p-1)^2}).\]
\begin{lemma}[cf. \cite{Szollosi-MUBs-BH}]\normalfont \label{lemma-GaussMUBs} Let $p$ be an odd prime number. Then,
\[\left\{I_p,\frac{1}{\sqrt{p}}F_p,\frac{1}{\sqrt{p}}\Delta F_p,\dots,\frac{1}{\sqrt{p}}\Delta^{p-1}F_p\right\},\]
gives a complete set of MUBs in $\C^p$.
\end{lemma}
\begin{proof}
The matrix $F_p$ satisfies $F_pF_p^{*}=pI_p$, so every matrix in the set is unitary, and unbiased to $I_p$. By Lemma \ref{lemma-MUBMatrix}, it suffices to show that
\[H_{x-y}=\sqrt{p}(\frac{1}{\sqrt{p}}\Delta^xF_p)^*(\frac{1}{\sqrt{p}}\Delta^yF_p)=\frac{1}{\sqrt{p}}F_p^*\Delta^{x-y}F_p,\]
is an Hadamard matrix for all $0\leq x<y\leq p-1$, or equivalently that $H_a=\frac{1}{\sqrt{p}}F_p^*\Delta^aF_p$ is Hadamard for each $a\in \F_p^{\times}$. We have,
\[H_aH_a^*=\frac{1}{p}F_p^*\Delta^aF_pF_p^*\Delta^{-a}F_p=F_p^*F_p=pI_p.\]
So we only need to show that the entries of $H_a$ are unimodular. Direct computation shows
\[[H_a]_{ij}=\frac{1}{\sqrt{p}}\sum_{r,s}[F_p^*]_{ir}[\Delta^a]_{rs}[F_p]_{sj}=\frac{1}{\sqrt{p}}\sum_{r,s}[F_p^*]_{ir}\zeta_p^{ar^2}\delta_{rs}[F_p]_{sj}=\frac{1}{\sqrt{p}}\sum_r\zeta_p^{ar^2+(j-i)r}.\]
Now we apply Lemma \ref{lemma-GeneralQGaussSum} with $b=(j-i)$, to find
\[\frac{1}{\sqrt{p}}\sum_{r}\zeta_p^{ar^2+(j-i)r}=\frac{1}{\sqrt{p}}\zeta_p^{f(a,j-i)}\legendre{a}{p}\sigma_p=\begin{cases}
\legendre{a}{p}\zeta_p^{f(a,j-i)} & \text{ if } p\equiv 1\pmod{4}\\
\legendre{a}{p}i\zeta_p^{f(a,j-i)} & \text{ if } p\equiv 3\pmod{4}
\end{cases},
\]
where $f(a,b)=a^{p-2}2^{p-3}b^{2}$. In any case, we find that the modulus of the entries of $H_a$ is $1$.\qedhere
\end{proof}

\begin{remark}\normalfont \label{rem-RootPEntries}
Notice that from the proof of Lemma \ref{lemma-GaussMUBs}, it follows that the entries of $(1/\sigma_p)F_p^*\Delta^aF_p$ are
\[\frac{1}{\sigma_p}[F_p^*\Delta^aF_p]_{ij}=\frac{\sqrt{p}}{\sigma_p}[H_a]_{ij}= \legendre{a}{p}\zeta_p^{f(a,j-i)},\]
with $f(a,b)=a^{p-2}2^{p-3}b^2$. In particular, the entries of $\frac{1}{\sigma_p}F_p^*\Delta^aF_p$ are in $\{1,\zeta_p,\dots,\zeta_p^{p-1}\}$ or $\{-1,-\zeta_p,\dots,-\zeta_p^{p-1}\}$ depending on the value of $\legendre{a}{p}$. This is the key observation in the construction of Butson matrices that we present below.
\end{remark}

\begin{theorem}[Butson, \cite{Butson}]\label{thm-ButsonConstruction} Let $p$ be an odd prime, and $s$ a non-square in $\F_p^{\times}$. Then, the matrix
\[H=
\begin{bmatrix}
I_p & 0\\
0 & \frac{1}{\sigma_p}F_p^*\Delta
\end{bmatrix}
\begin{bmatrix}
F_p & \Delta^{s-1}F_p\\
F_p & -\Delta^{s-1}F_p
\end{bmatrix}=
\begin{bmatrix}
\Delta F_p & \Delta^{s-1} F_p\\
\frac{1}{\sigma_p}F_p^*\Delta F_p & -\frac{1}{\sigma_p}F_p^*\Delta^s F_p
\end{bmatrix},
\]
where $\sigma_p$ is the $p$-th quadratic Gauss sum, is a $\BH(2p,p)$.
\end{theorem}
\begin{proof}
The entries of each block of the matrix $H$ are $p$-th roots of unity. This is clear for the blocks $\Delta F_p$ and $\Delta^{s-1}F_p$. Since $1$ is a square in $\F_p^{\times}$, Remark \ref{rem-RootPEntries} implies that the entries of $(1/\sigma_p)F_p^{*}\Delta Fp$ are $p$-th roots of unity, and likewise since $s$ is a non-square in $\F_p^{\times}$, the entries of $(1/\sigma_p)F_p^{*}\Delta^{s} F_p$ are negatives of $p$-th roots of unity. So it suffices to show that $HH^*=2pI_p$. This follows easily from a direct computation of $HH^*$ by blocks. We notice however, that this also follows from the fact that 
\[M=\begin{bmatrix}
F_p & \Delta^{s-1}F_p\\
F_p & -\Delta^{s-1}F_p
\end{bmatrix}=\sqrt{p}\begin{bmatrix}
\frac{1}{\sqrt{p}}F_p & \frac{1}{\sqrt{p}}\Delta^{s-1}F_p\\
\frac{1}{\sqrt{p}}F_p & -\frac{1}{\sqrt{p}}\Delta^{s-1}F_p
\end{bmatrix},
\]
is the McNulty-Weigert Construction (Theorem \ref{thm-McNultyWeigert}) applied to the real Hadamard matrix  of order $2$ $H_2=\begin{bmatrix}
1 & 1\\
1 & -
\end{bmatrix}$, and the mutually unbiased unitaries $\frac{1}{\sqrt{p}}F_p$ and $\frac{1}{\sqrt{p}}\Delta^{s-1}F_p$ (Lemma \ref{lemma-GaussMUBs}). So $MM^*=2pI_p$, and then
\[HH^*=\begin{bmatrix}
I_p & 0\\
0 & \frac{1}{\sigma_p}F_p^*\Delta
\end{bmatrix}MM^*
\begin{bmatrix}
I_p & 0\\
0 & \frac{1}{\overline{\sigma_p}}\Delta^*F_p
\end{bmatrix}
=2p\begin{bmatrix}
I_p & 0\\
0 & \frac{1}{|\sigma_p|^2}F_p^*F_p
\end{bmatrix}=2_pI_p,
\]
since $1/(|\sigma_p|^2)=1/p$, and $F_pF_p^*=pI_p$.\qedhere
\end{proof}

\begin{corollary}\normalfont There is a $\BH(2^ip^j,p)$ for all primes $p$ and $1\leq i\leq j$.
\end{corollary}
\begin{proof}
This follows from the existence of the Fourier matrix at order $p$, which is a $\BH(p,p)$ and the existence of $\BH(2p,p)$ matrices. Taking Kronecker products of these matrices (Proposition \ref{prop-SylvesterConstruction}), we can construct BH matrices over the $p$-th roots at orders $2^{i}p^{j}$ where $1\leq i\leq j$.\qedhere
\end{proof}

\begin{equation*}
\left[
\begin{array}{cccccc}
0&0&0&0&0&0\\
0&1&2&1&2&0\\
0&2&1&1&0&2\\
0&2&2&0&1&1\\
2&0&2&1&0&1\\
2&2&0&1&1&0\\
\end{array}
\right]
\end{equation*}
\captionof*{table}{A $\BH(6,3)$ matrix obtained with the method of Theorem \ref{thm-ButsonConstruction}.}

\subsection{Asymptotic existence of Butson-type Hadamard matrices}

The existence of $\BH(4p,p)$ for all primes $p$ was settled by Dawson. Here, in analogy with the proof of existence of $\BH(2p,p)$ matrices, we use a template of signs given by a real Hadamard matrix, namely
\[H_4=
\begin{bmatrix}
1 & 1 & 1 & 1\\
1 & 1 & - & -\\
1 & - & 1 & -\\
1 & - & - & 1
\end{bmatrix}.
\]

\begin{proposition}[Szöll\H{o}si, \cite{Szollosi-MUBs-BH}]\label{prop-DawsonConstruction} \normalfont
Let $p$ be an odd prime number. If there exist a triple $\alpha,\beta,\gamma\in\F_p^{\times}$ such that 
\begin{align*}
&\legendre{\alpha+1}{p}=\legendre{\beta+4}{p}=\legendre{\gamma+9}{p}=+1, \text{ and }\\
&\legendre{\alpha+4}{p}=\legendre{\alpha+9}{p}=\legendre{\beta+1}{p}=\legendre{\beta+9}{p}=\legendre{\gamma+1}{p}=\legendre{\gamma+4}{p}=-1,
\end{align*}
then the matrix
\begin{align*}
H&=\begin{bmatrix}
I_p  \\
&\frac{1}{\sigma_p} F_p^* \Delta \\
& &\frac{1}{\sigma_p} F_p^* \Delta^4 \\
& & &\frac{1}{\sigma_p}F_p^* \Delta^9
\end{bmatrix}
\begin{bmatrix}
F_p & \phantom{-}\Delta^{\alpha}F_p & \phantom{-}\Delta^{\beta}F_p & \phantom{-}\Delta^{\gamma}F_p\\
F_p & \phantom{-}\Delta^{\alpha}F_p & -\Delta^{\beta}F_p & -\Delta^{\gamma}F_p\\
F_p & -\Delta^{\alpha}F_p & \phantom{-}\Delta^{\beta}F_p & -\Delta^{\gamma}F_p\\
F_p & -\Delta^{\alpha}F_p & -\Delta^{\beta}F_p & \phantom{-}\Delta^{\gamma}F_p
\end{bmatrix}\\
&=\begin{bmatrix}
F_p & \phantom{-}\Delta^{\alpha}F_p & \phantom{-}\Delta^{\beta}F_p & \phantom{-}\Delta^{\gamma}F_p\\
\frac{1}{\sigma_p}F_p^*\Delta F_p & \phantom{-}\frac{1}{\sigma_p}F_p^*\Delta^{\alpha+1}F_p & -\frac{1}{\sigma_p}F_p^*\Delta^{\beta+1}F_p & -\frac{1}{\sigma_p}F_p^*\Delta^{\gamma+1}F_p\\
\frac{1}{\sigma_p}F_p^*\Delta^{4}F_p & -\frac{1}{\sigma_p}F_p^*\Delta^{\alpha+4}F_p & \phantom{-}\frac{1}{\sigma_p}F_p^*\Delta^{\beta+4}F_p & -\frac{1}{\sigma_p}F_p^*\Delta^{\gamma+4}F_p\\
\frac{1}{\sigma_p}F_p^*\Delta^9 F_p & -\frac{1}{\sigma_p}F_p^*\Delta^{\alpha+9}F_p & -\frac{1}{\sigma_p}F_p^*\Delta^{\beta+9}F_p & \phantom{-}\frac{1}{\sigma_p}F_p^*\Delta^{\gamma+9}F_p
\end{bmatrix},
\end{align*}
where $\sigma_p$ is the $p$-th quadratic Gauss sum, is a $\BH(4p,p)$.
\end{proposition}
\begin{proof}The proof is analogous to that of Theorem \ref{thm-ButsonConstruction}. The fact that the entries of $H$ belong to the set of $p$-th roots of unity follows from Remark \ref{rem-RootPEntries}. The orthogonality follows from the McNulty-Weigert construction, Theorem \ref{thm-McNultyWeigert}, and an analogous computation to the one in the proof of Theorem \ref{thm-ButsonConstruction}.
\end{proof}

To show that there are indeed $\BH(4p,p)$ matrices for every odd prime $p$, it remains to be shown that the system of residue conditions of Proposition \ref{prop-DawsonConstruction} has a solution for all but finitely many values of $p$, so that the existence for all $p$ will follow by settling finitely many sporadic cases.\\

A few remarks: First notice that the choice of square powers of $\Delta$ along the diagonal in Proposition \ref{prop-DawsonConstruction}, is to ensure that the first block-column of $H$ consists of entries in the $p$-th roots of unity. This reduces the number of quadratic residue equations from $16-4=12$ to $12-3=9$. We can do even better than $9$ equations by considering a single residue $r$ and letting $\alpha=r+a$, $\beta=r+b$ and $\gamma=r+c$, be shifts of $r$ for some integers $a,b,c$. For example, if $r=\alpha-1=\beta-4=\gamma-9$, then we reduce the $9$ equations to $7$ since the condition $\legendre{\alpha+1}{p}=\legendre{\beta+4}{p}=\legendre{\gamma+9}{p}=1$ reduces to $\legendre{r}{p}=1$. This is not the best choice of $r$ however. For example, we can obtain an improvement exploiting the fact that the template matrix $H_4$ has a symmetric core. For example letting $r=\alpha+1$, so that $a=-1$, we choose the values of $b$ and $c$, so that $\alpha+4=\beta+1$ and $\alpha+9=\gamma+1$, (and we know the residue character of these values must coincide by the symmetry of the template). We find that the choice $\alpha=r-1$, $\beta=r-2\cdot 1+4=r+2$, and $\gamma=r-2\cdot 1+9=r+7$, and this immediately implies that $\beta+9=\gamma+4$. Therefore, we require only $6$ equations, namely:

\begin{align*}
&\legendre{r}{p}=\legendre{r+6}{p}=\legendre{r+16}{p}=+1,\text{ and }\\
&\legendre{r+3}{p}=\legendre{r+8}{p}=\legendre{r+11}{p}=-1.
\end{align*}

 To find a lower bound on the values of $p$ for which such $r$ exists we can follow Hudson's approach (see Theorem 2 of \cite{Hudson-ResiduePatterns}) using the Weil bounds:

\begin{theorem}[Weil, \cite{Weil-CourbesAlgebriques}]\label{thm-WeilBounds}\index{Weil bound}
Let $p$ be an odd prime number. Then for any integer $m$ with $1\leq m\leq p-1$, and $a_1,\dots, a_m\in\F_p$, we have
\[\left|\sum_{r=1}^{p}\prod_{i=1}^{m}\legendre{r+a_i}{p}\right|\leq (m-1)\sqrt{p}
.\] 
\end{theorem}

\begin{proposition}[cf. Hudson, Theorem 2 \cite{Hudson-ResiduePatterns}]\normalfont \label{prop-HudsonComputations} Let $p$ be a prime number. There exists an integer $r$, $1\leq r\leq p-17$ with 
\begin{align*}
&\legendre{r}{p}=\legendre{r+6}{p}=\legendre{r+16}{p}=+1,\text{ and }\\
&\legendre{r+3}{p}=\legendre{r+8}{p}=\legendre{r+11}{p}=-1,
\end{align*}
if and only if $p\in\{7,29,31,41,47,59,61\}$ or $p\geq 71$.
\end{proposition}
\begin{proof}
The idea is to show that the sum
\begin{align*}S=\sum_{r=1}^{p-17}\bigg[&\left(1+\legendre{r}{p}\right)\left(1+\legendre{r+6}{p}\right)\left(1+\legendre{r+16}{p}\right)\cdot \\
&\cdot \left(1-\legendre{r+3}{p}\right)\left(1-\legendre{r+8}{p}\right)\left(1-\legendre{r+11}{p}\right)\bigg]
\end{align*}
is non-zero. Each term in the sum is either $0$ or $64$, so $S>0$ implies the existence of an integer $r$ satisfying the properties of the statement. Expanding the product in each summand of $S$ we find that $S$ is split into $7=6+1$ sums each involving the product of $i$ Legendre symbols, for $0\leq i\leq 6$. More explicitly, let $\mathcal{S}=\{0,3,6,8,16,11\}$ and index a subset $\mathcal{A}\subseteq\mathcal{S}$ as $\mathcal{A}=\{a_1,\dots,a_{|\mathcal{A}|}\}$, then we have
\[S=\sum_{\substack{\mathcal{A}\subseteq\mathcal{S}\\ |\mathcal{A}|\geq 0}}(-1)^{\xi(\mathcal{A})}\sum_{r=1}^{p-17}\prod_{i=1}^{|\mathcal{A}|}\legendre{r+a_i}{p},\]
where $\xi(\mathcal{A})$ is $0$ or $1$. It is easy to calculate the summands when there are $0$ and $1$ Legendre symbols involved, we have when $|\mathcal{A}|=0$ a sum 
\[\sum_{r=1}^{p-17}(1)^6=p-17. \]
When $|\mathcal{A}|=1$ we have $6$ sums, namely
\[\sum_{r=1}^{p-17}\legendre{r}{p}\cdot (1)^5=\sum_{r=1}^{p-17}\legendre{r}{p},\ \sum_{r=1}^{p-17}\legendre{r+6}{p},\ \dots,\text{ }-\sum_{r=1}^{p-17}\legendre{r+11}{p}.\]
Now we can estimate each of these sums. For example,
\[\left|\sum_{r=1}^{p-17}\legendre{r}{p}\right|=\left|\left(\sum_{r=0}^{p-1}\legendre{r}{p}\right)-\legendre{p-16}{p}-\legendre{p-15}{p}-\dots-\legendre{p-1}{p}\right|\leq 16.\]
And similarly for the rest of sums with $|\mathcal{A}|=1$. For $k\neq \ell$ the following identity holds (see Theorem 2 of \cite{Hudson-ResiduePatterns}),
\[\sum_{r=1}^{p-1}\legendre{(r+k)(r+\ell)}{p}=-1.\]
So we can also obtain good estimates of the $15={6\choose 2}$ sums with $|\mathcal{A}|=2$, for example
\[\left|\sum_{r=1}^{p-17}\legendre{r}{p}\legendre{r+6}{p}\right|=\left|\left(\sum_{r=1}^{p-1}\legendre{r(r+6)}{p}\right)-\legendre{(p-16)(p-10)}{p}+\dots+\legendre{(p-1)(p+7)}{p}\right|\leq 16,\]
and likewise with the rest of sums with $|\mathcal{A}|=2$. The sums with $|\mathcal{A}|\geq 3$ can be estimated using the Weil bounds, Theorem \ref{thm-WeilBounds},
\[\left|(-1)^{\xi(a)}\sum_{r=1}^{p-17}\prod_{i=1}^{|\mathcal{A}|}\legendre{r+a_i}{p}\right|\leq (|\mathcal{A}|-1)\sqrt{p}+17.\]
Therefore,
\[|S-(p-17)|\leq 6\cdot 16+15\cdot 16+ \sum_{i=3}^{6}{6\choose i}((i-1)\sqrt{p}+17)=114\sqrt{p}+1050.\]
If $p-17<S$, then clearly $S>0$ for $p\geq 17$. Otherwise, we have that
\[S\geq (p-17)-114\sqrt{p}-1050,\]
and $(p-17)-114\sqrt{p}-1050>0$ for all primes $p\geq 15061$. That the statement is true for $p\in\{7,29,31,41,47,59,61\}$ and for $71\leq p \leq 15061$ can easily be verified by computer.\qedhere
\end{proof}

\begin{theorem}[Dawson, \cite{Dawson-BH4p}]\label{thm-DawsonThm} There exists a $\BH(4p,p)$ for every prime number $p$.
\end{theorem}
\begin{proof}
By Proposition \ref{prop-DawsonConstruction} and Proposition \ref{prop-HudsonComputations} it follows that there is a $\BH(4p,p)$ for all $p\in\{7,29,31,41,47,59,61\}$. Furthermore, there exist triples $(\alpha,\beta,\gamma)$ satisfying the hypotheses of Proposition \ref{prop-DawsonConstruction} for the following primes, given in the table below,
\begin{table}[H]
\centering
\begin{tabular}{|c|c||c|c||c|c|}
\hline
$p$ & $(\alpha,\beta,\gamma)$ & $p$ & $(\alpha,\beta,\gamma)$ & $p$ & $(\alpha,\beta,\gamma)$ \\
\hline
$11$ & $(4,1,6)$ & $19$ & $(4,1,11)$ & $43$ & $(3,11,1)$ \\

$13$ & $(2,6,1)$ & $23$ & $(1,21,16)$ & $53$ & $(10,11,1)$  \\

$17$ & $(1,5,6)$ & $37$ & $(9,5,1)$ & $67$ & $(3,2,1)$  \\
\hline
\end{tabular}
\caption{Triples giving $\BH(4p,p)$ matrices via Proposition \ref{prop-DawsonConstruction}.}
\end{table}
After considering these values, the only sporadic cases remaining are $\BH(12,3)$ and $\BH(20,5)$, but we have existence for both these matrices via the de Launey construction, Theorem \ref{thm-DeLauneyConstruction}, and the Scarpis construction, Theorem \ref{thm-ScarpisConstruction}.\qedhere
\end{proof}

After this result of Dawson, de Launey conjectured in \cite{DeLauney-GHMSurvey} that there exist $\BH(4tp,p)$ for $t\geq 1$.
\begin{research-problem}\normalfont Prove the conjecture of de Launey on the existence of $\BH(4tp,p)$ matrices for $t\geq 1$.
\end{research-problem}

In fact, de Launey and Dawson made a significant contribution supporting this conjecture by generalising Dawson's methods to confirm the asymptotic existence of $\BH(hp,p)$ matrices where $h$ is the order of a real Hadamard matrix. Recall that the $p$-th \textit{Paley core}, or $p$-th \textit{Jacobsthal} matrix, is the $p\times p$ matrix $Q_p$ given by\index{Paley!core}
\[[Q_p]_{ij}=\legendre{i-j}{p}.\]

\begin{theorem}[cf. de Launey - Dawson \cite{DeLauney-Dawson-AsymptoticExistence}]\label{thm-SubmatrixPaleyBH}
Let $Q_p$ be the $p$-th Paley core. If there is an Hadamard submatrix $H$ of order $h$ in $Q_p$, then there exists a $\BH(hp,p)$.
\end{theorem}
\begin{proof}
Suppose that there is an Hadamard submatrix $H$ of order $h$ in $Q_p$, then there exist row indices $\mathcal{I}=\{i_1,\dots,i_h\}$ and column indices $\mathcal{J}=\{j_1,\dots,j_h\}$ such that 
\[H_{rs}=\legendre{i_r-j_s}{p}.\]
We show that the matrix
\begin{align*}M_p(\mathcal{I};\mathcal{J})=&\diag\left[\frac{1}{\sigma_p}F_p^*\Delta^{i_1},\frac{1}{\sigma_p}F_p^*\Delta^{i_2},\dots,\frac{1}{\sigma_p}F_p^*\Delta^{i_h}\right]\cdot\\
\cdot&\left(H\otimes\left[\Delta^{-j_1}F_p,\Delta^{-j_2}F_p,\dots,\Delta^{-j_h}F_p\right]\right),
\end{align*}
is a $\BH(hp,p)$. The block in position $(r,s)$ of $M_p(\mathcal{I}; \mathcal{J})$ is 
\[H_{rs}\frac{1}{\sigma_p}F_p^*\Delta^{i_{r}-j_{s}}F_p,\]
which by Remark \ref{rem-RootPEntries}, has entries in the $p$-th roots of unity. It suffices to prove orthogonality, but this follows from Lemma \ref{lemma-GaussMUBs} and the McNulty-Weigert construction, Theorem \ref{thm-McNultyWeigert}.
\end{proof}

Theorem \ref{thm-SubmatrixPaleyBH} provides us then with an effective program to show the existence of $\BH(hp,p)$ matrices for all $p$. Namely,

\begin{itemize}
\item[(i)] Show that for $p$ large enough, a given real Hadamard matrix of order $h$ is guaranteed to exist as a submatrix of the $p$-th Paley core $Q_p$.
\item[(ii)] Use computational methods, or other techniques, to lower the bounds on $p$.
\item[(iii)] Find constructions for a small number of sporadic examples.
\end{itemize}

As in the proof of Dawson's theorem, Theorem \ref{thm-DawsonThm}, Szöll\H{o}si shows in \cite{Szollosi-MUBs-BH} that the Weil bounds can be applied to show step (i) for any real Hadamard matrix. More strongly, it can be shown that any pattern of signs $+1$ and $-1$ can be found in a large enough Paley matrix.
The reason for this is that quadratic residues exhibit a \textit{pseudorandom} behaviour, see Theorem 6.8 in Babai's notes \cite{Babai-FourierAnalysisFG}. Heuristically, this tells us that we can expect that the entries of the Paley matrix will behave as if they were taken randomly to be $+1$ or $-1$ with probability $1/2$. Therefore, we can expect to observe any pattern of signs in the matrix for large enough values of $p$. The Weil bounds are sufficient to obtain an asymptotic result on existence, but the corresponding bound turns out to be rather weak, and it should be possible to do better with more specialised techniques. To conclude this subsection, we present the current status of existence of $\BH(8p,p)$ and two lower bounds for the asymptotic existence of $\BH(hp,p)$ matrices:

\begin{theorem}[DeLauney - Dawson, \cite{DeLauney-Dawson-BH8p}] There exists a $\BH(8p,p)$ for all $p>19$. 
\end{theorem}

Notice that by the de Launey construction, Theorem \ref{thm-DeLauneyConstruction}, a $\BH(24,3)$ matrix exists. So to settle the existence of $\BH(8p,p)$ matrices it suffices to give an answer to the following:¨

\begin{research-problem}\normalfont Decide the existence or non-existence of $\BH(8p,p)$ for $p\in\{5,7,11,13,17\}$.
\end{research-problem}

Using an analogous method to the one in Proposition \ref{prop-DawsonConstruction}, Szöll\H{o}si obtains the following:
\begin{theorem}[Szöll\H{o}si, \cite{Szollosi-MUBs-BH}] \label{thm-SzollosiBound} Suppose there exists a real Hadamard matrix of order $h$. Then for every prime $p>2^{2h^2+1}$, there exists a $\BH(hp,p)$.
\end{theorem}

Szöll\H{o}sis's bound gives a lower bound on $p$ such that a particular Hadamard submatrix of order $h$ will be guaranteed to exist in the Paley core $Q_p$. Dawson and de Launey give an alternate approach in \cite{DeLauney-Dawson-AsymptoticExistence}, where instead they found lower bounds on $p$ that guarantee that any $\pm 1$ vector of length $h$, or its negation, can be found as a subvector of a row of the Paley core $Q_p$. Even if this may seem like a stronger condition to impose on $p$, it turns out that it requires a lesser number of restrictions, and the lower bounds are several orders of magnitude lower.

\begin{theorem}[de Launey - Dawson \cite{DeLauney-Dawson-AsymptoticExistence}]\label{thm-DDLBound}
Suppose there exists a real Hadamard matrix of order $h$. Then for all primes $p\geq ((h-2)2^{h-2})^2$, there exists a $\BH(hp,p)$.
\end{theorem}  

It would be very interesting to study the asymptotic existence of Butson-type Hadamard matrices at orders $mp$, where $m\equiv 2\pmod{4}$. For example,

\begin{research-problem}\normalfont Study the asymptotic existence of $\BH(6p,p)$, matrices for $p$ prime.
\end{research-problem}

Szöll\H{o}si pointed out in \cite{Szollosi-MUBs-BH} that this problem seems to require new ideas. Since there exists a $\BH(6,4)$ matrix, one may be lead to consider instead quartic residues, but the identity in Lemma 
\ref{lemma-GeneralQGaussSum} appears to be unique to the quadratic character.

\subsection{The existence of BH(12p, p) matrices}

In this subsection we study lower bounds on $p$ for the existence of $\BH(12p,p)$ matrices. First we illustrate how the lower bounds on $p$ can be improved according to the choice of template matrix. After that we outline a computational approach to improve the theoretical lower bounds. With this approach we were able to reduce the lower bound on $p$ for the existence of $\BH(12p,p)$ matrices to $p>263$.\\

 As in the case of Proposition \ref{prop-DawsonConstruction}, we can reduce the number of constraints by using an Hadamard matrix of order $12$ with a symmetric core. For example we can take the following back-circulant matrix,
\[
H_{12}=
\left[
\begin{array}{cccccccccccc}
1 & 1 & 1 & 1 & 1 & 1 & 1 & 1 & 1 & 1 & 1 & 1 \\
1 & 1 & - & - & - & 1 & - & - & 1 & - & 1 & 1 \\
1 & - & - & - & 1 & - & - & 1 & - & 1 & 1 & 1 \\
1 & - & - & 1 & - & - & 1 & - & 1 & 1 & 1 & - \\
1 & - & 1 & - & - & 1 & - & 1 & 1 & 1 & - & - \\
1 & 1 & - & - & 1 & - & 1 & 1 & 1 & - & - & - \\
1 & - & - & 1 & - & 1 & 1 & 1 & - & - & - & 1 \\
1 & - & 1 & - & 1 & 1 & 1 & - & - & - & 1 & - \\
1 & 1 & - & 1 & 1 & 1 & - & - & - & 1 & - & - \\
1 & - & 1 & 1 & 1 & - & - & - & 1 & - & - & 1 \\
1 & 1 & 1 & 1 & - & - & - & 1 & - & - & 1 & - \\
1 & 1 & 1 & - & - & - & 1 & - & - & 1 & - & 1 \\
\end{array}
\right]
\]
In analogy with Proposition \ref{prop-DawsonConstruction} we have the following result.
\begin{proposition}[cf. \cite{Szollosi-MUBs-BH}]\normalfont \label{prop-12Template} Suppose there is a $11$-tuple $(\alpha_1,\dots,\alpha_{11})$ taken from $\F_p^{\times}$ such that 
\[\legendre{\alpha_j+i^2}{p}=[H_{12}]_{(i+1),(j+1)},\text{ for all } 1\leq i,j\leq 11.\]
Then the matrix 
\begin{align*}M_p(\alpha_1,\dots,\alpha_{11})=&\diag\left[I_p,\frac{1}{\sigma_p}F_p^*\Delta,\frac{1}{\sigma_p}F_p^*\Delta^{2^2},\dots,\frac{1}{\sigma_p}F_p^*\Delta^{11^2}\right]\cdot\\
\cdot&\left(H_{12}\otimes\left[F_p,\Delta^{\alpha_1}F_p,\Delta^{\alpha_2}F_p,\dots,\Delta^{\alpha_{11}}F_p\right]\right),
\end{align*}
is a $\BH(12p,p)$.
\end{proposition}

We could attempt to reduce the number of constraints from $121=11^2$ to $66={12\choose 2}=11+{11\choose 2}$, by letting
\[\alpha_1=r-1,\alpha_2=r+2,\dots,\alpha_i=r+(j^2-2),\dots, \alpha_{11}=119,\]
as we did in the $4\times 4$ case. The matrix $(r+(j^2-2)+i^2)_{i,j}$ is clearly symmetric, so the symmetry of the core will reduce the number of constraints. However, we would encounter problems by taking this choice of shifts of $r$. The reason is that the sets $\{1,2^2,\dots, 11^2\}$ and $\{-1,2,\dots,119\}=\{1^2-2,2^2-2,\dots 11^2-2\}$, are not \textit{Sidon pairs}.
\begin{definition}\normalfont We call a pair of subsets $A,B\subset G$ of an abelian group $(G,+)$ a \textit{Sidon pair} if and only if the list of pairwise sums
\[\left[a_i+b_j: 1\leq i\leq |A|,\text{ and } 1\leq j\leq |B|\right],\]
contains no repetitions.
\end{definition}
 Notice however, that $r+(10^2-2)+5^2=r+123$, and $r+(11^2-2)+2^2=r+123$, but $[H_{12}]_{6,11}=-1$ and $[H_{12}]_{3,12}=+1$ which would force $r+123$ to be both a square and a non-square modulo $p$. This is the only obstruction found. There are more collisions, however these cause no issues since the entries of $H_{12}$ at those coincide, so in reality a weaker condition than a Sidon pair may suffice.\\
 
 Here we will take the Sidon pair $A=\{i^2+10i: i \in\{1,\dots,11\}\}$ and $B=\{j^2+10j-11: j\in \{1,\dots,11\}\}$, which gives the following summation table:
 \[
\begin{array}{ccccccccccc}
11 & 24 & 39 & 56 & 75 & 96 & 119 & 144 & 171 & 200 & 231 \\
   & 37 & 52 & 69 & 88 & 109 & 132 & 157 & 184 & 213 & 244 \\
   &    & 67 & 84 & 103 & 124 & 147 & 172 & 199 & 228 & 259 \\
   &    &    & 101 & 120 & 141 & 164 & 189 & 216 & 245 & 276 \\
   &    &    &     & 139 & 160 & 183 & 208 & 235 & 264 & 295 \\
   &    &    &     &     & 181 & 204 & 229 & 256 & 285 & 316 \\
   &    &    &     &     &     & 227 & 252 & 279 & 308 & 339 \\
   &    &    &     &     &     &     & 277 & 304 & 333 & 364 \\
   &    &    &     &     &     &     &     & 331 & 360 & 391 \\
   &    &    &     &     &     &     &     &     & 389 & 420 \\
   &    &    &     &     &     &     &     &     &     & 451 \\
\end{array}
\]
This is simply the upper triangular part of the symmetric matrix $(i^2+j^2+10(i+j)-11)_{i,j}$. It is easy to check that there are no repeated elements in the list above. We can use these numbers as a pattern of shifts of $r$ to obtain a $\BH(12p,p)$. Note however, that by adding the linear term $10(i+j)$ we cannot use the template of Proposition \ref{prop-12Template}, and additionally we must ensure that the terms $r+i^2+10i$ for $i=1,\dots,11$ are all squares modulo $p$, these shifts are
\[0, 13, 28, 45, 64, 85, 108, 133, 160, 189,\text{ and } 220.\]
Once we introduce these new terms, we find two coincidences with the numbers in the summation table above, namely $160=5^2+6^2+10\cdot(5+6)-11$ and $189=4^2+8^2+10\cdot(4+8)-11$, however we have that $[H_{12}]_{6,7}=+1$ and $[H_{12}]_{5,9}=+1$ so we encounter no contradictions. The total number of constraints we find is ${12\choose 2}+11-2=75$. We have just shown that the matrix
\begin{align*}M_p(r)=&\diag\left[I_p,\frac{1}{\sigma_p}F_p^*\Delta^{1^2+10},\frac{1}{\sigma_p}F_p^*\Delta^{2^2+20},\dots,\frac{1}{\sigma_p}F_p^*\Delta^{11^2+110}\right]\cdot\\
\cdot&\left(H_{12}\otimes\left[F_p,\Delta^{r-11}F_p,\Delta^{r+1^2+10-11}F_p,\dots,\Delta^{r+11^2+110-11}F_p\right]\right),
\end{align*}
is a $\BH(12p,p)$ provided that the quadratic character of $r+(i^2+j^2+10(i+j)-11)$ modulo $p$ coincides with the corresponding entry of $H_{12}$.\\
 
 Then, in analogy with Proposition \ref{prop-HudsonComputations}, we found lower bounds on $p$ such that the above condition on $r$ is satisfied, namely
\begin{proposition}\normalfont Let $p>2^{150}$ be a prime. Then there is a integer $r$, $1\leq r\leq p-452$, such that
\begin{align*}
\legendre{r+a}{p}=
\begin{cases}
+1 & \text{ if } a\in \mathcal{S}^{+}\\
-1 & \text{ if } a\in \mathcal{S}^{-}
\end{cases},
\end{align*}
where $\mathcal{S}^{+}$ and $\mathcal{S}^{-}$ are given in the tables below
\begin{center}
\begin{tabular}{|P{0.05\linewidth}|P{0.3\linewidth}|}
\hline
\vspace{38pt}$\mathcal{S}^+$ & $0, 11, 13, 28, 45, 64,$\newline $67, 69, 75, 85, 108, 120,$
\newline $124, 132, 133, 144, 160, 164,$
\newline $172, 181, 183, 184, 189, 199,$
\newline $200, 204, 208, 213, 216, 220,$
\newline $228, 231, 244, 304, 308, 316,$
\newline $389, 391, 451$\\
\hline
\vspace{30pt}$\mathcal{S}^-$ & $24, 37, 39, 52, 56, 84,$\newline $88, 96, 101, 103, 109, 119,$
\newline $139, 141, 147, 157, 171, 227,$
\newline $229, 235, 245, 252, 256, 259,$
\newline $264, 276, 277, 279, 285, 295,$
\newline $331, 333, 339, 360, 364, 420$\\
\hline
\end{tabular}
\end{center}
\end{proposition}
\begin{proof}
We follow the same argument as in the proof of Proposition \ref{prop-HudsonComputations}. Let $S$ be the sum
\[S=\sum_{r=1}^{p-452}\prod_{a\in\mathcal{S}^{+}\cup\mathcal{S}^{-}}\left(1+(-1)^{\eta(a)}\legendre{a+s}{p}\right),\]
where $\eta(a)=0$ if $a\in \mathcal{S}^{+}$ and $\eta(a)=1$ if $a\in\mathcal{S}^{-}$. Then we can bound the absolute value of the terms involving $1$-fold and $2$-fold products of Legendre symbols by $451$, and the terms involving $k$-fold products for $k\geq 3$ are estimated with the Weil bounds, giving an upper bound of $(k-1)\sqrt{p}+452$. Recall that we have a total of $75$ constraints, so we obtain the bound 
\[|S-(p-452)|\leq {75\choose 1}451+{75\choose 2}451+\sum_{i=3}^{75}\left[{75\choose i}\sqrt{p}+452\right]\]
For $p\geq 452$, if $S>p-452$ then the claim holds trivially. Otherwise, we have that 
\[S\geq p - 37778931862957161706717\sqrt{p} - 1318798.\]
It is easy to show that if $p>2^{150}$, then $S>0$.\qedhere
\end{proof}

 The bound we obtained is of the order $2^{150}$, which is a significant improvement over the bound $2^{284}$ using Szöll\H{o}si's more general choice of $r$. However the bound obtained by de Launey and Dawson is still several orders of magnitude better, of value $(10\cdot 2^{10})^2$.\\
 
 However, with computational methods we are able to show that $\BH(12p,p)$ matrices exist for all primes $p>263$.
 
 \begin{definition}\normalfont Let $\mathcal{R}=\{r_1,\dots,r_m\}\subset \{+1,-1\}^{n}$ be a set of $m$ row vectors of length $n$ with entries $\pm 1$. The \textit{orthogonality graph} of $\mathcal{R}$, is the graph $G=(V,E)$ on $m$ vertices, such that vertices $i$ and $j$ are adjacent if and only if rows $r_i$ and $r_j$ are orthogonal, i.e. if and only if $r_ir_j^{\intercal}=0$.
 \end{definition}\index{graph!orthogonality}
 
 Our computational methods is as follows: Let $p$ be a prime number, and $h$ the order of an Hadamard matrix, 
 \begin{itemize}
 \item[(i)] Construct the Paley matrix $Q_p$.
 \item[(ii)] Randomly select a set of column indices $\mathcal{C}=\{c_1,\dots,c_h\}$ of size $h$.
 \item[(iii)] Create a set of $p-h$ rows $\mathcal{R}$ by taking all rows in $Q_p$ whose indices are not in $\mathcal{C}$, and restricting those rows to their entries in $\mathcal{C}$ (to avoid the appearance of zeros in the submatrix).
 \item[(iv)] Create the orthogonality graph $G$ corresponding to the set of rows $\mathcal{R}$.
 \item[(v)] Find a $K_h$ subgraph in $G$.
 \end{itemize}
 
 We implemented this method in \texttt{C}, making use of the library \texttt{cliquer} by Patric Österg\aa rd and Sampo Niskanen \cite{cliquer}, which provides fast routines to find cliques in a given graph.\\
 
 For small values of $p$ we performed the search as described above. But for values of $p>1000$ we found that a better approach is to randomly select a small number $r$ of rows, instead of considering the full orthogonality graph on the $p-h$ rows of $Q_p$. A small choice of $r$ will result in more random trials of rows and columns needed until an Hadamard submatrix is found, and a large choice of $r$ will result in excessive time spent in creating the orthogonality graph and searching for a clique. For $h=12$, we found that restricting to around $r=700$ random rows of the $Q_p$ produced the fastest search results. With this we obtained the following:
 
 \begin{theorem}
 There is a $\BH(12p,p)$ matrix for all primes $p>197$. 
 \end{theorem}
\begin{proof}
We checked computationally that there is an Hadamard submatrix of order $12$ in the Paley core $Q_p$ for every prime $197<p\leq 104857600=(10\cdot 2^{10})^2.$ Theorem \ref{thm-SubmatrixPaleyBH} implies that for each such prime there is a $\BH(12p,p)$.  By the de Launey and Dawson theorem on the asymptotic existence of $\BH$ matrices, Theorem \ref{thm-DDLBound}, there is a $\BH(12p,p)$ matrix for all $p>(10\cdot 2^{10})^2$.  Additionally we were able to find Hadamard submatrices of order $12$ in $Q_p$ for $p$ in the set indicated in the statement, which shows the existence of the corresponding $\BH(12p,p)$.\qedhere
\end{proof}

It may be possible to find Hadamard submatrices of order $12$ in $Q_p$ for further values of $p$. However, our randomised approach is not adequate for smaller orders. This suggests the following problem.

\begin{research-problem}\normalfont Carry a deterministic computer search to determine which set of primes $p\leq 197$ has the property that $Q_p$ contains an Hadamard submatrix of order $12$.
\end{research-problem}

We observed, that for $h=12$ starting from primes $p>300$ the probability to find an Hadamard submatrix of order $12$ was very high. This suggests that the event of finding an Hadamard submatrix in a $k\times h$ random $\pm 1$ matrix may have a \textit{sharp threshold}. By this we mean that there is a critical value of $k$ for which the probability of finding an Hadamard submatrix in a $m\times h$ matrix is close to $0$ for $m<k$ and close to $1$ for $m>k$.

\begin{research-problem}\normalfont Give an heuristic argument that shows that the event of finding an Hadamard submatrix of order $h$ in a random $k\times h$ $\pm 1$ matrix has a sharp threshold. Give an estimate of the critical value of $k$ for a given $h$.  
\end{research-problem} 
 
 We attempted the same search for the case $h=16$. The critical value of $k$ seems to be somewhere between $4000$ and $4025$. In fact we were unable to find Hadamard submatrices for primes $p<4000$, and primes for primes $p>4025$ these are found easily. However, the threshold of $4025$ is much too large and finding an Hadamard submatrix of order $16$ in a $\pm 1$ matrix of with over $4025$ rows is computationally costly. For this reason it seems infeasible to verify de Launey's conjecture up to the de Launey-Dawson bound for $h=16$ using our methods. 

\section{Tables of existence of BH matrices}
We conclude this chapter with a summary of results, and tables on existence and classification of $\BH(n,m)$ matrices for $3\leq m\leq 6$. The tables should be interpreted as follows: below every order $n$ appears either a number, the symbol \blue{?}, the symbol \green{E}, or the symbol \green{H}. A number indicates that BH matrices have been classified at the corresponding order, either by showing non-existence (in which case the number \red{0} appears) or by complete enumeration of isomorphism classes. The symbol \blue{?} indicates that the existence or non-existence is currently unknown. The symbol \green{E} indicates that existence is known, but we that do not have a complete classification. For $m$ even, the symbol \green{H} means that there is a real Hadamard matrix at order $n$, which in particular implies that there is a $\BH(n,m)$.\\

The non-existence results are obtained by means of Theorem \ref{thm-ButsonNonEx}. In particular, $\BH(n,3)$ matrices and $\BH(n,6)$ matrices cannot exist whenever $n$ is odd and $p\equiv 5\pmod 6$ divides the square-free part of $n$, and $\BH(n,5)$ matrices cannot exist whenever $p\equiv 3,7,9\pmod{10}$ divides the square-free part of $n$. Recall as well, that when $p$ is prime then a $\BH(n,p)$ can only exist if $p\mid n$, see Lemma \ref{lemma-BHpDiv}. The classification results have been taken from the paper by Lampio, Österg\aa rd, and Szöll\H{o}si \cite{Lampio-Ostergard-Szollosi-Butson}, see also \cite{Lampio-Ostergard-Szollosi-Quaternary}. To obtain the remaining existence results, we have used the following methods
\begin{itemize}
\item[(i)] Sylvester's construction, Proposition \ref{prop-SylvesterConstruction}, to obtain a $\BH(ab,m)$ whenever $\BH(a,m)$ and $\BH(b,m)$ exist.
\item[(ii)] The Scarpis construction, Theorem \ref{thm-ScarpisConstruction}, to obtain a $\BH(qn,m)$, whenever $q=n-1$ is a prime power. For example, the existence of $\BH(10,6)$ implies the existence of $\BH(90,6)$ since $9=10-1$ is a prime power. To the best of our knowledge, the existence of $\BH(90,6)$ was previously unknown.
\item[(iii)] Real Hadamard matrices have been shown to exist at orders $4n$ for all $1\leq n<167$. These give in turn existence of $\BH(4n,2m)$ for any integer $m\geq 1$.
\item[(iv)] de Launey's Construction in Theorem \ref{thm-DeLauneyConstruction} gives us the existence of $\BH(2^t\cdot 3,3)$ for every $t\geq 1$, which account for the orders $n=6,12,24,48,$ and $96$ in the first table. More generally, Theorem \ref{thm-DawsonThm} gives the existence of $\BH(4^t\cdot 5,5)$ matrices for all $t\geq 1$, and this accounts for $n=20$ and $n=80$ in the third table.
\item[(v)] For every odd prime power $q$, the Paley Construction gives a $\BH(q+1,4)$. See Theorem \ref{thm-ComplexPaley}, and Lemma 2.4 of \cite{Horadam-HadamardBook}. 
\item[(vi)] For every prime $p$ there is a $\BH(p^2,6)$ via a construction of Craigen and Szöll\H{o}si, see Theorem 1.4.41 of \cite{Szollosi-PhDThesis}.
\item[(vii)] In the table of $\BH(n,6)$ matrices we included some sporadic examples of the type $\BH(2p,6)$ for $p$ prime. See section 1.4. of \cite{Szollosi-PhDThesis}, for more details.
\end{itemize}

\begin{table}[H]
\centering
\begin{tabular}{cccccccc}
3 & 6 & 9 & 12 &15 & 18 & 21 & 24\\
\hline
\green{1} & \green{1} & \green{3} & \green{2} &\red{0} & \green{85} & \green{72} & \green{E}\\

27 & 30 & 33 & 36 & 39 & 42 & 45 & 48\\
\hline
\green{E}& \green{E}& \red{0}&\green{E} &\blue{?} & \blue{?} & \red{0} & \green{E}\\

51 & 54 & 57 & 60 & 63 & 66 & 69 & 72\\
\hline
\red{0} & \green{E} & \blue{?} & \blue{?}&\green{E}&\blue{?}&\red{0}&\green{E}\\
75 & 78 & 81 & 84 & 87 & 90 & 93 & 96\\
\hline
\blue{?}&\blue{?}&\green{E}&\blue{?}&\red{0}&\green{E}&\blue{?}&\green{E}\\
99 & 102 & 105 & 108 & 111 & 114 & 117 & 120\\
\hline
\red{0} & \blue{?} & \red{0} & \green{E} &\blue{?} & \blue{?} & \blue{?} & \blue{?}
\end{tabular}
\caption{Existence and classification of $\BH(n,3)$ for $n\leq 120$.}\label{tab-BH3}
\end{table}

\begin{table}[H]
\centering
\begin{tabular}{cccccc}
2 & 4 & 6 & 8 & 10 & 12\\
\hline
\green{1} &\green{2} &\green{1}&\green{15}&\green{10}&\green{319}\\
14&16&18&20&22&24\\
\hline
\green{752}& \green{1786763}&\green{E}&\green{H}&\green{E}&\green{H}\\
26&28&30&32&34&36\\
\hline
\green{E} &\green{H} & \green{E} & \green{H}&\green{E} &\green{H}\\

38&40&42&44&46&48\\
\hline
\green{E} &\green{H} & \green{E} & \green{H}&\green{E} &\green{H}\\

50&52&54&56&58&60\\
\hline
\green{E} &\green{H} & \green{E} & \green{H}&\green{E} &\green{H}\\

62&64&66&68&70&72\\
\hline
\green{E} &\green{H} & \green{E} & \green{H}&\blue{?} &\green{H}\\
74 & 76 & 78 & 80 & 82 & 84\\
\hline
\green{E} & \green{H} & \blue{?} &\green{H}&\blue{?} & \green{H}\\
86 & 88 & 90 & 92 & 94 & 96\\
\hline
\blue{?} &\green{H} & \green{E} & \green{H} &\blue{?} & \green{H}\\
98 & 100 & 102 & 104 & 106 & 108\\
\hline
\green{E} & \green{H} & \green{E} &\green{H} & \blue{?} & \green{H}
\end{tabular}
\caption{Existence and classification of $\BH(n,4)$ for $n\leq 108$.}\label{tab-BH4}
\end{table}

\begin{table}[H]
\centering
\begin{tabular}{ccccc}
5 & 10 & 15 & 20 & 25\\
\hline
\green{1}&\green{1}&\red{0}&\green{E}&\green{E}\\
30& 35 & 40 & 45 & 50\\
\hline
\blue{?} &\red{0} & \blue{?} & \red{0} & \green{E}\\

55& 60 & 65 & 70 & 75\\
\hline
\blue{?} & \blue{?} & \red{0} & \blue{?} & \red{0}\\
80& 85 & 90 & 95 &100\\
\hline
\green{E} & \red{0} & \green{E} & \red{0} & \green{E}
\end{tabular}
\caption{Existence and classification of $\BH(n,5)$ for $n\leq 100$.}\label{tab-BH5}
\end{table}

\pagebreak
{\phantom{line}}\\\vspace{84pt}
\begin{table}[H]
\centering
\begin{tabular}{cccccccccccc}
1 & 2 & 3 & 4 & 5 & 6 & 7 & 8 & 9 & 10 & 11 & 12\\
\hline
\green{1}&\green{1}&\green{1}&\green{2}&\red{0}&\green{4}&\green{2}&\green{36}&\green{17}&\green{34}&\red{0}&\green{8703}\\
13 & 14 & 15 & 16 & 17 & 18 & 19 & 20 & 21 & 22 &23 & 24\\
\hline
\green{436} & \green{16776} & \red{0} &\green{H} & \red{0} &\green{E} & \green{E} &\green{H} &\green{E} &\green{E}&\red{0}&\green{H}\\
25 & 26 & 27 & 28 & 29 & 30 & 31 &32 &33 &34&35&36\\
\hline
\green{E} &\green{E}&\green{E}&\green{H}&\red{0}&\green{E}&\blue{?}&\green{H}&\red{0}&\green{E}&\red{0}&\green{H}\\
37 & 38 & 39 & 40& 41& 42&43&44&45&46&47&48\\
\hline
\blue{?} & \green{E} &\green{E} &\green{H} &\red{0} &\green{E} &\blue{?} &\green{H} & \red{0} &\blue{?} &\red{0}&\green{H}\\
49 & 50 & 51&52&53&54&55&56&57&58&59&60\\
\hline
\green{E}&\green{E}&\red{0}&\green{H}&\red{0}&\green{E}&\red{0}&\green{H}&\green{E}&\green{E}&\red{0}&\green{H}\\
61 & 62 & 63 & 64 & 65 & 66 & 67 & 68 & 69 & 70 & 71 & 72\\
\hline
\blue{?}&\blue{?}&\green{E}&\green{H} & \red{0} & \green{E}&\blue{?}&\green{H}&\red{0} & \green{E} &\red{0} & \green{H}\\
73& 74& 75 & 76 & 77 & 78 & 79 & 80 & 81 & 82& 83 & 84\\
\hline
\blue{?} & \blue{?} &\green{E}&\green{H}&\red{0}&\green{E}&\blue{?} &\green{H}&\green{E}&\blue{?}&\red{0}&\green{H}\\ 
85 & 86 & 87 & 88 & 89 & 90 & 91 & 92 & 93 & 94 & 95 & 96\\
\hline
\red{0} & \blue{?} & \red{0} & \green{H} & \red{0} & \green{E}&\green{E} &\green{H}&\blue{?}&\blue{?}&\red{0} & \green{H}\\
97 & 98 & 99 & 100 & 101 & 102 & 103 &104 & 105 & 106 & 107 & 108\\
\hline
\blue{?} & \green{E} & \red{0} & \green{H}&\red{0}&\green{E}&\blue{?}&\green{E}&\red{0}&\blue{?}&\red{0}&\green{E}
\end{tabular}
\caption{Existence and classification of $\BH(n,6)$ for $n\leq 108$.}\label{tab-BH6}
\end{table}

\begin{research-problem}\normalfont Determine the existence or non-existence of the first open cases in each table, namely $\BH(31,3)$, $\BH(39,3)$, $\BH(70,4)$, $\BH(30,5)$ and $\BH(37,6)$.
\end{research-problem}

\cleardoublepage
\chapter{Complex Maximal Determinant Matrices}\label{chap-Maxdet}
Hadamard's theorem, Theorem \ref{thm-HadamardTheorem}, shows that Hadamard matrices achieve the largest determinant possible among matrices with entries of absolute value $1$. In the last chapter, we saw that for every integer $n$ the Fourier matrix is an Hadamard matrix of order $n$. However, if we consider Hadamard matrices with restricted entries, such as real Hadamard matrices,  then these do not exist for certain values of $n$. In his paper \cite{Hadamard-Determinants}, J. J. Hadamard noticed that $\pm 1$ Hadamard matrices of order $n>2$ can only exist if $n$ is a multiple of $4$. This led him to pose the following problem:

\begin{research-problem}[Hadamard's maximal determinant problem]\normalfont
Find the maximal value of the determinant of a $\pm 1$ matrix of order $n$, for all $n\geq 1$.\index{Hadamard's maximal determinant problem}\index{Hadamard's maximal determinant problem}
\end{research-problem}

Similar to real Hadamard matrices, $\BH(n,m)$ matrices do not exist for all values of $n$. This may happen because there are no vanishing sums of $m$-th roots at order $n$, or  for subtler reasons as Theorem \ref{thm-ButsonNonEx} shows. We propose an extension of Hadamard's maximal determinant problem to the class of matrices with entries over the $m$-th roots of unity.

\begin{research-problem}\normalfont For an integer $n\geq 1$, find the maximal absolute value of the determinant of a matrix of order $n$ with entries in the set $\mu_m$ of $m$-th roots of unity.
\end{research-problem}

In this chapter, we will study general upper and lower bounds for matrices with entries in the $m$-th roots. We will show that there are interesting similarities between the $\pm 1$ problem and the cases when $m=3,4$ or $6$. The case $m=4$ was first studied by J.H.E Cohn in \cite{Cohn-ComplexDOptimal}, where he showed that using the Turyn morphism, Theorem \ref{thm-TurynMorphism} one can ``lift'' certain $\pm 1$ maximal determinant matrices to maximal determinant matrices with entries in $\mu_4$. We use this to show that Spence's family of $\pm 1$ maximal determinant matrices \cite{Spence-GoethalsSeidel} yields a, previously unknown, family of maximal determinant matrices over the fourth roots. We also find sporadic examples computationally.\\

 Among the cases $m=3,4,$ and $6$, we consider $m=3$ to be the most interesting and challenging case, so we devote more attention to it. In particular, we find a determinantal lower bound at orders $n\equiv 2\pmod{3}$ using cyclotomy, and we find several examples of small maximal determinant matrices. In the case $m=6$, there is much evidence to believe that the only obstruction to the existence of $\BH(n,6)$ matrices is the condition of Theorem \ref{thm-ButsonNonEx}. So most of the interesting results on maximal determinant matrices over the sixth roots have been discussed in Chapter \ref{chap-BHMats}, and in \cite{Szollosi-PhDThesis}. Because of the existence of the morphism of Theorem \ref{thm-CCDLMorphism}, a more interesting problem that we propose is the following
\begin{research-problem}\normalfont For all integers $n\geq 1$, determine the maximal value of the determinant of a matrix with entries in the set $\{\pm \omega, \pm \omega^2\}$, where $\omega$ is a primitive third-root of unity.
\end{research-problem}

Throughout this chapter, we will denote the value of the Hadamard bound as
\[h(n)=n^{n/2}.\]
The maximal absolute value of the determinant of an $n\times n$ matrix with entries in $\mu_m=\{1,\zeta_m,\zeta_m^2,\dots,\zeta_m^{m-1}\}$ will be denoted by $\gamma_m(n)$. When $m=2$, we abbreviate $\gamma(n):=\gamma_2(n)$.\\

\section{Hadamard's maximal determinant problem}
Recall that Hadamard's determinant bound, Theorem \ref{thm-HadamardTheorem}, states that for a square matrix $M$ of order $n$ with complex entries of modulus $1$,
\[|\det M|\leq n^{n/2}.\]
Furthermore, $M$ meets Hadamard's bound with equality if and only if $MM^*=nI_n$, i.e. if and only if $M$ is an Hadamard matrix. In Lemma \ref{lemma-BHpDiv} we showed that a $\BH(n,p)$ can only exist whenever $p\mid n$. So, if a real Hadamard matrix of order $n$ exists, then $n$ must be even. More strongly, we have the following
\begin{lemma}[cf. Hadamard, \cite{Hadamard-Determinants}]\normalfont If there is a real Hadamard matrix of order $n>2$, then $4\mid n$.
\end{lemma}
\begin{proof}
Suppose that there is an Hadamard matrix $H$ of order $n>2$, then $H$ has at least $3$ mutually orthogonal rows, and up to monomial equivalence (Definition \ref{def-MonomialEquivalence}), we may assume that the first three rows of $H$ are 

\[
\begin{array}{cccccccccccc}
1 & \dots &1&1&\dots & 1& 1&\dots & 1 & 1 & \dots &1\\
1 &\dots &1 &1 &\dots & 1 &- &\dots &- &-&\dots & -\\
\undermat{a}{1&\dots&1}&\undermat{b}{-&\dots&-}&\undermat{c}{1&\dots&1}&\undermat{d}{-&\dots&-}\\
\end{array}
\]
\vspace{12pt}\\Taking pairwise inner products, we find 
\[
\begin{cases}
& a+b+c+d=n\\
& a+b-c-d=0\\
& a-b+c-d=0\\
& a-b-c+d=0
\end{cases},
\]
Solving this linear system of equations, one finds the unique solution 
\[a=b=c=d=n/4.\]
And since $a,b,c$ and $d$ are integers, it follows that $n$ must be divisible by $4$.\qedhere
\end{proof}

In view of this result, Hadamard posed the maximal determinant problem in his paper \cite{Hadamard-Determinants}. Hadamard's maximal determinant problem has been well-studied, partly due to the application of \index{matrix!maximal determinant} maximal determinant matrices in statistics, where they are known as $D$-optimal designs\index{design!D-optimal}. Currently we have at our disposal strengthened upper and lower bounds for the determinant, infinite families of maximal determinant matrices, and computational results for matrices of small order. Here we give a summary of results for maximal determinant matrices of order $n$, split into the different congruence classes of $n$ modulo $4$. For additional details, we refer the reader to the recent survey on Hadamard's maximal determinant problem \cite{Padraig-MaxDetSurvey}.

\subsection{Real Hadamard matrices}
Let $n>2$ be an integer such that $n\equiv 0\pmod{4}$. Then a $\pm 1$ matrix $M$ satifies
\[|\det(M)|\leq n^{n/2},\]
with equality if and only if $M$ is real Hadamard. No counterexamples to the existence of a real Hadamard matrix have been found, and there is a strong reason to believe that real Hadamard matrices exist at all orders $n=4m$. For example, Seberry \cite{Seberry-Asymptotic} and Craigen \cite{Craigen-Asymptotic} obtained the following asymptotic existence results for real Hadamard matrices:

\begin{theorem}[Seberry, Theorem 17 \cite{Seberry-Asymptotic}, 1975] If $m>3$ is an integer, then there exists an Hadamard matrix of order $2^tm$, where $t=\lfloor 2\log_2(m-3)\rfloor$.
\end{theorem}
\begin{theorem}[Craigen, Theorem 9 \cite{Craigen-Asymptotic}, 1993] If $m$ is an odd positive integer, then there is an Hadamard matrix of order $2^tm$, where $t=4\lceil \frac{1}{6}\log_2((m-1)/2)\rceil +2$. 
\end{theorem}
In particular, these results show that the existence of real Hadamard matrices of order $2^tm$, where $m$ is any given odd number, is settled for all but finitely many values of $t$. Furthermore, several infinite families of real Hadamard matrices have been found, we summarise some of these constructions below:

\begin{enumerate}
\item $2^t$ for $ t\geq 0$ \cite{Sylvester-InverseOrthogonal}.
\item $q+1$ where $q\equiv 3\pmod{4}$ is a prime power \cite{Paley}.
\item $2(q+1)$ where $q\equiv 1\pmod{4}$ is a prime power \cite{Paley}.
\item $p(p+2)+1$, where $p$ and $p+2$ are twin primes, see \cite{Stanton-Sprott-TwinPrimes}, and Example 2.1.1 (3) in \cite{Horadam-HadamardBook}.
\end{enumerate}

The smallest undecided order for the existence of a real Hadamard matrix is $n=668$, \cite{Kharaghani-Hadamard428}. For more information on real Hadamard matrices see \cite{Horadam-HadamardBook}.

\subsection{Barba matrices}
Let $n$ be an odd integer. Then a $\pm 1$ matrix $M$ of order $n$ cannot meet the Hadamard bound. In \cite{Barba-Bound}, Guido Barba found a strengthened upper bound for odd order $\pm 1$ matrices
\begin{theorem}[Barba, \cite{Barba-Bound}] \label{thm-RealBarbaBound}Let $n$ be an odd and $M$ a $\pm 1$ matrix of order $n$. Then\index{determinant inequality!Barba}
\[|\det(M)|\leq \sqrt{2n-1}(n-1)^{(n-1)/2}.\]
Furthermore $M$ meets the bound with equality if and only if $M$ is monomially equivalent to a matrix $B$ such that
\[BB^{\intercal}=(n-1)I_n+J_n.\]
\end{theorem}

Suppose $M$ meets the Barba bound with equality, then $\det(M)$ is an integer and this implies that $2n-1$ is a square.
\begin{lemma}\normalfont\label{lemma-2n1Square} Let $n$ be an odd integer. Then $2n-1$ is a perfect square if and only if $n$ is the sum of two consecutive squares.
\end{lemma}
\begin{proof}
 Suppose that $2n-1=k^2$ for some integer $k$. Then $k$ is odd, then $k^2\equiv 1\pmod{4}$ and there is an integer $m$ such that $2n-1=k^2=4m+1$, and thus $4m=k^2-1=(k+1)(k-1)$. There is an integer $t$ such that $2t=k+1$ and $2(t-1)=k-1$.  This implies that $m=t(t-1)$, and substituting into $n=2m+1$ we find
\begin{align*}
n&=2t^2-2t+1\\
&=t^2+(t^2-2t+1)\\
&=t^2+(t-1)^2.
\end{align*}
Conversely, if $n=t^2+(t-1)^2$ for some integer $t$, then $2n-1=4t^2-4t+1=(2t-1)^2.$\qedhere
\end{proof}

\begin{corollary}\normalfont If a $\pm 1$ matrix $M$ of odd order $n$ meets the Barba bound with equality, then $n$ is a sum of consecutive squares, and in particular $n\equiv 1\pmod{4}$.
\end{corollary}
\begin{proof}
If $B$ meets the Barba bound with equality, then $2n-1$ is a perfect square, which by Lemma \ref{lemma-2n1Square} implies that $n=t^2+(t-1)^2=4t^2-4t+1\equiv 1\pmod{4}$.\qedhere
\end{proof}

\begin{definition}\normalfont\label{def-RealBarbaMat}
A matrix $B$ of order $n$, with entries of modulus $1$, is called a \textit{Barba matrix}\index{matrix!Barba}  if and only if 
\[BB^{*}=(n-1)I_n+J_n.\]
\end{definition}
In the real case, every Barba matrix is maximal determinant. In the following section we will show that complex Barba matrices are maximal determinant only over the third and fourth roots of unity.\\

With an argument inspired by the paper \cite{Chan-Godsil} by Chan and Godsil, we can show that real Barba matrices are equivalent to certain symmetric designs. 

\begin{theorem} \label{thm-RealBarba} Let $B$ be a $\pm 1$  matrix of order $v$ with constant row-sum. Then $B$ is a Barba matrix if and only if $N=(J_v+B)/2$ is the $(0,1)$-incidence matrix of a symmetric $2$-$(v,k,k-(v-1)/4)$ design.
\end{theorem}
\begin{proof}Suppose that $B$ has constant row-sum, say $\rho$. Then, the matrix $N$ also has constant row-sum, since
\[2NJ_v=(J_v+B)J_v=(v+\rho)J_v.\]
We may then write $NJ_v=kJ_v=N^{\intercal}J_v$, where $k=(v+\rho)/2\in\Z$. Now, assume that $B$ is a Barba matrix, then $BB^{\intercal}=(v-1)I_v+J_v$. On the other hand, since $B=2N-J_v$, we have
\begin{align*}
(v-1)I_v+J_v=BB^{\intercal}&=
4NN^{\intercal}-2(NJ_v)^{\intercal}
-2NJ_v+vJ_v\\
&=4NN^{\intercal}-(4k-v)J_v.
\end{align*}
Dividing by $4$ and rearranging terms we find that
\[NN^{\intercal}=
\frac{v-1}{4}I_v+\left(k-\frac{v-1}{4}\right)J_v.
\]
Since $N$ is a square $(0,1)$ matrix of order $v$, this implies that $N$ is the incidence matrix of a $(v,k,k-(v-1)/4)$ design. Conversely, if $N$ is the incidence matrix of a symmetric $(v,k,k-(v-1)/4)$ design, then a straightforward calculation shows that $B=2N-J_v$ is a real Barba matrix.
\end{proof}

\begin{corollary}\normalfont \label{cor-BarbaDesignLK} There exists a real Barba matrix of order $v$ if and only if there exists a symmetric $2$-$(v,k,\lambda)$ design, with $\lambda=k-(v-1)/4$.
\end{corollary}
\begin{proof}
By Theorem \ref{thm-RealBarbaBound}, if there is a Barba matrix $M$, then $M$ is monomially equivalent to a normal Barba matrix $B$, see Theorem 18 in \cite{Padraig-MaxDetSurvey}. Therefore by associativity
\[B(BB^{\intercal})=B(B^{\intercal}B)=(BB^{\intercal})B.\]
Since $BB^{\intercal}=(n-1)I_n+J_n$, then $B$ commutes with $J_n$, and this implies that $B$ has constant row sum. By Theorem \ref{thm-RealBarba}, the existence of $B$ implies the existence of a symmetric $2$-$(v,k,k-(v-1)/4)$ design. The converse is immediate from Theorem \ref{thm-RealBarba}.\qedhere
\end{proof}

\begin{remark*}\normalfont In Theorem \ref{thm-BarbaConstantRowSum}, we will show that Barba matrices are monomially-equivalent to normal constant row sum Barba matrices also in the complex case.
\end{remark*}

The following lemma characterises the parameters of designs satisfying the condition above, and will help us identify an infinite family of designs satisfying the conditions of Corollary \ref{cor-BarbaDesignLK}.
 
\begin{lemma}\normalfont Suppose that $(v,k,\lambda)$, with $\lambda=k-(v-1)/4$, are the parameters of a symmetric $2$-$(v,k,\lambda)$ design. Then there is an integer $t$ such that $(v,k,\lambda)=(t^2+(t+1)^2,t^2,{t\choose 2})$, or $(v,k,\lambda)=(t^2+(t-1)^2,t^2,{t+1\choose 2})$.
\end{lemma}
\begin{proof}
Since $k>0$ is a natural number, there is a positive real number $t\in\R$ such that $k=t^2$. By assumption, we have that $4\lambda=4k-(v-1)=4t^2-x$, where $x:=v-1$. Because $(v,k,\lambda)$ are the parameters of a $2$-design, we have by Lemma \ref{general-blockcount} that $(v-1)\lambda=k(k-1)$. Therefore, $4x\lambda=4k(k-1)=4t^2(t^2-1)$. Now, substituting $4\lambda=4t^2-x$ we find
\[x^2-4t^2x+4t^2(t^2-1)=0.\]
This quadratic equation on $x$ has two solutions, namely $x=2t^2\pm 2t$, which implies that $v=x+1=t^2+(t\pm 1)^2$. Also note that $2\lambda=2k-\frac{x}{2}=t^2\mp t$, and this implies $\lambda={t\choose 2}$ or $\lambda={t+1\choose 2}$ respectively. It just remains to be shown that $t$ is an integer: We know that $t^2=k$ is an integer, and since $\lambda$ is an integer, then $t^2\mp t=2\lambda\in\Z$. But since $t^2\in\Z$, this implies that $t\in \Z$.\qedhere
\end{proof}
Conversely, if $(v,k,\lambda)=(t^2+(t+1)^2,t^2,{t\choose 2})$ or $(v,k,\lambda)=(t^2+(t-1)^2,t^2,{t+1\choose 2})$, then $\lambda=k-(v-1)/4$. The only known infinite family of $\pm 1 $ Barba matrices was found by Neubauer and Radcliffe in \cite{Neubauer-Radcliffe}, here they construct matrices at orders $q^2+(q-1)^2$ where $q$ is a prime power. Using a family of designs attributed to R.M. Wilson, and constructed by Brouwer in \cite{Brouwer-FamilyOfDesigns}, we can give a simplified proof of the existence of such Barba matrices.
\begin{theorem}[Wilson - Brouwer \cite{Brouwer-FamilyOfDesigns}]\label{thm-BrouwerConstruction}
Let $q$ be an odd prime power, and $h>0$ an integer. Then, there exists a symmetric $2$-$(v,k,\lambda)$ design with
\begin{align*}
&v=2(q^h+q^{h-1}+\dots+q)+1,\\
&k=q^h,\text{ and }\\
&\lambda=\frac{1}{2}q^{h-1}(q-1).
\end{align*} 
\end{theorem}
\begin{theorem}[cf. Neubauer - Radcliffe \cite{Neubauer-Radcliffe}]\label{thm-BarbaConstruction}  For every odd prime power $q$, there exists a $\pm 1$ Barba matrix of order $q^2+(q+1)^2$.
\end{theorem}
\begin{proof}
Letting $h=2$ in Theorem \ref{thm-BrouwerConstruction} one finds the existence of a $2$-$(q^2+(q+1)^2,q^2,{q\choose 2})$-design, say $\mathcal{D}$. Denoting $v=q^2+(q+1)^2$, $k=q^2$ and $\lambda={q\choose 2}$, it is easy to see that $\lambda=k-(v-1)/4$. Let $N$ be the incidence matrix of $\mathcal{D}$. Then, by Theorem \ref{thm-RealBarba}, the matrix $B=2N-J_v$, is a Barba matrix.\qedhere
\end{proof}

When $n\equiv 1\pmod{4}$ is not a sum of consecutive squares, the Barba bound cannot be met. In this case, computational methods are required to guarantee the maximality of the determinant of a candidate matrix, see for example \cite{Greek-17}.

\subsection{Ehlich-Wojtas matrices}
If for $n>2$, $n\equiv 2\pmod{4}$, then a $\pm 1$ matrix of order $n$ cannot achieve either the Hadamard bound nor the Barba bound. In this case there is a strengthened determinant upper bound which was obtained independently by Wojtas \cite{Wojtas-Determinants}, and Ehlich \cite{Ehlich-DeterminantsI}.

\begin{theorem}[Ehlich - Wojtas, \cite{Ehlich-DeterminantsI,Wojtas-Determinants}]\label{thm-EhlichWojtas}\index{determinant inequality!Ehlich-Wojtas} Let $M$ be a $\pm 1$ matrix of order $n\equiv 2\pmod{4}$. Then,
\[
\det(M)\leq (2n-2)(n-2)^{(n-2)/2}.
\]
Furthermore, $M$ achieves the bound if and only if $M$ is monomially equivalent to a matrix $W$ such that 
\[
WW^{\intercal}
=
\left[
\begin{array}{c|c}
(n-2)I_{n/2}+2J_{n/2} & 0\\
\hline
0 & (n-2)I_{n/2}+2J_{n/2}
\end{array}
\right].
\]
\end{theorem}

\begin{definition}\normalfont\label{def-EWMatrix}
A $\pm 1$ matrix $W$ of order $n$ is called an \textit{Ehlich-Wojtas matrix}, or EW matrix if and only if 
\index{matrix!EW}
\[WW^{\intercal}
=
\left[
\begin{array}{c|c}
(n-2)I_{n/2}+2J_{n/2} & 0\\
\hline
0 & (n-2)I_{n/2}+2J_{n/2}
\end{array}
\right].
\]
\end{definition}

The following terminology is due to J. H. E. Cohn \cite{Cohn-NumberDOptimal}.
\begin{definition}\normalfont 
\label{def-SkewBlockMatrix} A block-matrix $M$ of the type
\[M=\begin{bmatrix}
A & B\\
-B & A
\end{bmatrix},
\]
is called \textit{skew}.
\end{definition}
\begin{theorem}[Theorem 18, \cite{Padraig-MaxDetSurvey}]
If $M$ is a $\pm 1$ matrix of order $n$ meeting the bound of Theorem \ref{thm-EhlichWojtas} with equality, then $2n-2$ is the sum of two squares.
\end{theorem}
The existence of skew EW matrices is particularly interesting, since in Section \ref{sec-Maxdet4} we will show that they can be used to construct maximal determinant matrices over the fourth roots. We present here two constructions of skew EW matrices, and note that to the best of our knowledge these are the only two known infinite families.

\begin{lemma} \normalfont \label{lemma-Barba2EW} Let $B$ be a Barba matrix of order $n$, then the matrix
\[W=\begin{bmatrix}
B & B\\
-B & B
\end{bmatrix},\]
is a skew EW matrix of order $2n$.
\end{lemma}
\begin{proof}
This follows from a direct computation of $WW^{\intercal}$ by blocks, using the fact that $BB^{\intercal}=(n-1)I_n+J_n$.\qedhere
\end{proof}

\begin{corollary}\normalfont There is an EW matrix of order $2(q^2+(q+1)^2)$ for every odd prime power $q$.
\end{corollary}
\begin{proof}
By Theorem \ref{thm-BarbaConstruction}, there exists a $\pm 1$ Barba matrix $B$ of order $q^2+(q+1)^2$ for every odd prime power $q$ . Apply the construction of Lemma \ref{lemma-Barba2EW} to $B$.\qedhere
\end{proof}

Following an argument of Koukouvinos, Kounias, and Seberry \cite{KKS-DSDOptimal}, we can use a result of Spence \cite{Spence-GoethalsSeidel}, to obtain a construction for EW matrices.

\begin{lemma}[Koukouvinos - Kounias - Seberry, \cite{KKS-DSDOptimal}]\label{lemma-KKSLemma}\normalfont Let $n\equiv 1\pmod{4}$, and let $R$ and $S$ be two commuting matrices with entries $\pm 1$ such that $RR^{\intercal}+SS^{\intercal}=(2n-2)I_n+2J_n$. Then, the matrix
\[W=
\begin{bmatrix}
R & S\\
-S^{\intercal} & R^{\intercal}
\end{bmatrix}
\]
is an EW matrix of order $2n$.
\end{lemma}
\begin{proof}
Using the fact that $RS=SR$, and $RR^{\intercal}+SS^{\intercal}=(2n-2)I_n+2J_n$, direct computation of $WW^{\intercal}$ by rows shows that 
\[WW^{\intercal}=
\begin{bmatrix}
RR^{\intercal}+SS^{\intercal} & -RS+SR\\
-S^{\intercal}R^{\intercal}+R^{\intercal}S^{\intercal}&RR^{\intercal}+SS^{\intercal}
\end{bmatrix}
=\begin{bmatrix}
(2n-2)I_n+2J_n & 0 \\
0 & (2n-2)I_n+2J_n
\end{bmatrix}.\qedhere
\]
\end{proof}

Spence's Theorem will give us a pair of matrices $R$ and $S$ for the construction above. To state it, we introduce a bit of terminology.
\begin{definition}[cf. Marshall Hall \cite{Hall-CyclicProjectivePlanes}]\normalfont A projective plane $\mathcal{P}$ of order $n$ is called \textit{cyclic} if $\mathcal{P}$ admits an automorphism $\sigma$ of order $n^2+n+1$ acting transitively on the points of $\mathcal{P}$.
\end{definition}
Let $\pi_n$ be the $n\times n$ permutation matrix given by the permutation $(1,2,\dots,n)$. In terms of the Kronecker delta, $\pi_n$ can be written as $[\pi_n]_{i,j}=\delta_{i,j-1}$, where the indices are interpreted modulo $n$. For example,
\[\pi_3=\begin{bmatrix}
0 & 1 & 0\\
0 & 0 & 1\\
1 & 0 & 0
\end{bmatrix}
.\]
In the physics literature, the matrix $\pi_n$ is sometimes called the \textit{shift matrix}. A closely related matrix is the \textit{clock matrix} $\Delta_n=\diag(1,\zeta_n,\dots,\zeta_n^{n-1})$. Both matrices are related by the following well-known lemma

\begin{lemma}[cf. Theorem 3.2.1 \cite{Davis-CirculantMatrices}]\normalfont \label{lemma-FourierIntertwining} Let $F_n$ be the Fourier matrix of order $n$, then
\[\pi_nF_n=F_n\Delta_n.\]
\end{lemma}
\begin{proof}
Conceptually, this is a consequence of the fact that $\pi_n$ is the regular representation of the cyclic group $C_n$, and $\Delta_n$ is the direct sum of all irreducible representations of $C_n$. Since $F_n$ is the character table of $C_n$, then the identity holds. It is also straightforward to check the identity $F_n$ directly using the fact that $[F_n]_{ij}=\zeta_n^{ij}$.\qedhere
\end{proof}
\begin{definition}\normalfont A matrix $A$ of order $n$ is called \textit{circulant} if and only if
\[A=\sum_{i=0}^{n-1}a_i\pi_n^{i},\]
for some scalars $a_0,a_1,\dots,a_{n-1}$.
\end{definition}
For example, a generic $3\times 3$ circulant matrix has the shape
\[A=a_0I_3+a_1\pi_3+a_2\pi_3^2=\begin{bmatrix}
a_0 & a_1 & a_2\\
a_2 & a_0 & a_1\\
a_1 & a_2 & a_0
\end{bmatrix}.\]
Clearly, any pair of circulant matrices $A$ and $B$ of the same order commute with each other. This is because both $A$ and $B$ are expressed as polynomials on the matrix $\pi$. Notice that by Lemma \ref{lemma-FourierIntertwining}, all circulant matrices are simultaneously diagonalisable by the Fourier matrix $F_n$.

\begin{theorem}[Spence, \cite{Spence-GoethalsSeidel}]\label{thm-SpenceConstruction}If there exists a cyclic projective plane of order $n^2$, then there exist two $\pm 1$ matrices $R$ and $S$, both circulant and of order $n^2+n+1$, such that
\[RR^{\intercal}+SS^{\intercal}=(2n^2 +2n)I_{n^2+n+1}+2J_{n^2+n+1}.\]
\end{theorem}

\begin{theorem}[Singer, cf. Theorem 8.1. \cite{Moore-Pollatsek-DifferenceSets}]\label{thm-Singer}
For every prime power $q$, the Desarguesian projective plane of order $q$ is cyclic.
\end{theorem}

From this, the following construction of EW matrices follows easily.
\begin{theorem}[Koukouvinos - Kounias - Seberry, \cite{KKS-DSDOptimal}]\label{thm-EWConstruction} For every prime power $q$, there exists an EW matrix of order $2(q^2+q+1)$.
\end{theorem}
\begin{proof}
Let $q$ be a prime power. By Singer's Theorem (Theorem \ref{thm-Singer}), the projective plane of order $q^2$ is cyclic. Therefore, Spence's construction, Theorem \ref{thm-SpenceConstruction} implies that there exists a pair of circulant matrices $R$ and $S$, such that 
\[RR^{\intercal}+SS^{\intercal}=(2q^2+2q)I_{q^2+q+1}+2J_{q^2+q+1}.\]
Since any pair of circulant matrices of the same order commute with each other, Theorem \ref{thm-EWConstruction} implies that 
\[W=
\begin{bmatrix}
R & S\\
-S^{\intercal} & R^{\intercal}
\end{bmatrix},
\]
is an EW matrix.\qedhere
\end{proof}

We conclude this section by showing that the EW matrices of Lemma \ref{lemma-KKSLemma} are monomially equivalent to skew EW matrices.

\begin{lemma}[cf. Cohn, \cite{Cohn-NumberDOptimal}]\label{lemma-Circulant2Skew} \normalfont Let $P$ be the back-diagonal matrix of order $n$, i.e. $P_{ij}=\delta_{n+1-i,j}$. If $R$ and $S$ are circulant matrices of order $n$, then
\[
\begin{bmatrix}
P & 0\\
0 & I_n
\end{bmatrix}
\begin{bmatrix}
R & S\\
-S^{\intercal} & R^{\intercal}
\end{bmatrix}
\begin{bmatrix}
P & 0\\
0 & I_n
\end{bmatrix},
\]
is a skew matrix.
\end{lemma}
\begin{proof}
Direct computation shows that 
\[
\begin{bmatrix}
P & 0\\
0 & I_n
\end{bmatrix}
\begin{bmatrix}
R & S\\
-S^{\intercal} & R^{\intercal}
\end{bmatrix}
\begin{bmatrix}
P & 0\\
0 & I_n
\end{bmatrix}=
\begin{bmatrix}
PRP & PS\\
-S^{\intercal}P & R^{\intercal}
\end{bmatrix}.
\]
It suffices show that $PRP=R^{\intercal}$ and $PS=S^{\intercal}P$, whenever $R$ and $S$ are circulant. We show that $P\pi_n=\pi_n^{\intercal}P$. On the one hand, we have that 
\[[P\pi_n]_{ij}=\sum_{k}P_{ik}[\pi_n]_{kj}=\sum_k \delta_{n+1-i,k}\delta_{k+1,j}=\delta_{n+1-i,j-1}.\]
On the other hand,
\[[\pi_n^{\intercal}P]_{ij}=\sum_k [\pi_n]_{k,i}P_{kj}=\sum_k\delta_{k+1,i}\delta_{n+1-k,j}=\delta_{n+1-i+1,j}.\]
Since $\delta_{n+1-i+1,j}=\delta_{n+1-i,j-1}$, this implies that $P\pi_n=\pi_n^{\intercal}P$. Therefore $PS=S^{\intercal}P$, and $PR=R^{\intercal}P$. Now, since $P^2=I_n$, it follows that $PRP=R^{\intercal}$.\qedhere
\end{proof}

\begin{corollary}\normalfont \label{cor-EWSkew} For every prime power $q$, there is a skew EW matrix of order $2(q^2+q+1)$.
\end{corollary}
\begin{proof}
By Theorem \ref{thm-EWConstruction} there is an EW matrix, say $M$, of order $2(q^2+q+1)$, and $M$ has the following structure
\[M=\begin{bmatrix}
R & S\\
-S^{\intercal} &R^{\intercal}
\end{bmatrix}.
\]
Let $X$ be the permutation matrix given by 
\[X=\begin{bmatrix}
P & 0\\
0 & I
\end{bmatrix},
\]
where $P$ is the back-diagonal matrix of order $n$. We have by Lemma \ref{lemma-Circulant2Skew}, that $W=XMX$ is a $\pm 1$ skew matrix. Computing the Gram matrix of $W$ we find that
\begin{align*}
WW^{\intercal}&=X(MM^{\intercal})X\\
&=\begin{bmatrix}
P & 0\\
0 & I
\end{bmatrix}
\begin{bmatrix}
G & 0\\
0 & G
\end{bmatrix}
\begin{bmatrix}
P & 0\\
0 & I
\end{bmatrix}\\
&=\begin{bmatrix}
PGP & 0\\
0 & G
\end{bmatrix},
\end{align*}
where $G=(2q^2+2q)I_{q^2+q+1}+2J_{q^2+q+1}$. Now, since $G$ is a circulant matrix, it follows from the proof of Lemma \ref{lemma-Circulant2Skew} that $PGP=G^{\intercal}=G$. Therefore, $W$ is a skew EW matrix. \qedhere
\end{proof}

\subsection{Ehlich matrices}
 As we saw in Lemma \ref{lemma-2n1Square}, the Barba bound cannot be met unless $n\equiv 1\pmod{4}$. In \cite{Ehlich-3mod4}, Ehlich found a sharper upper bound for the determinant of a $\pm 1$ matrix of order $n\equiv 3\pmod{4}$.

\begin{theorem}[Ehlich, cf. Satz 3.3. \cite{Ehlich-3mod4}] \label{thm-EhlichBound}\index{determinant inequality!Ehlich} Let $n\geq 63$ be an integer with $n\equiv 3\pmod{4}$. Then a $\pm 1$ matrix of order $n$ satisfies
\[|\det(M)|^2\leq \frac{4\cdot 11^6}{7^7}n(n-1)^6(n-3)^{n-7}. \]
Equality is achieved in the bound if and only if $n=7m$, and 
\[MM^{\intercal}=D(m),\]
where $D(m):=((7m-3)I_m+4J_m)\otimes I_7 -J_{7m}.$
\end{theorem}
In all other congruence classes above, we saw that the general upper bounds obtained are always achievable, and hence they cannot be improved. This is unclear for Ehlich's bound in the case $n\equiv 3\pmod{4}$, as there are no known examples of matrices meeting the bound with equality for $n>3$. Recall also that in Section \ref{sec-MaxdetApp} we proved Tamura's result \cite{Tamura-DOptimal}, showing that the smallest order at which this bound could be attained is $n=511$. This makes the case $n\equiv 3\pmod {4}$ significantly more challenging.\\

We give a high-level picture of the analysis of Ehlich, for more details the reader can consult Section 6 of \cite{Padraig-MaxDetSurvey}, or the original paper by Ehlich \cite{Ehlich-3mod4}. The proof strategy for Ehlich's bound consists in an analysis of the determinant of matrices in the set of $m\times m$ matrices: 
\[\mathcal{C}_m=\{M\ :\ m_{ii}=n,\ m_{ij}\equiv 3\pmod{4},\ |m_{ij}|<n\}, \]
where $n$ is a fixed positive integer. In particular, Ehlich studies the matrices $C_m^*$ for which
\[\det C_m^*=\max\{\det M:M\in\mathcal{C}_m\}.\]
The reason for this is that given a matrix $A$ of order $n\equiv 3\pmod{4}$ with entries in $\pm 1$, we have that $AA^{\intercal}\in \mathcal{C}_n$. And so, the determining the value of $\det C_n^*$ gives us an upper bound for the determinant of $A$.\\

The first key result that Ehlich shows is the following:

\begin{proposition}[Ehlich, Satz 2.2. \cite{Ehlich-3mod4}] \normalfont Let $C_m^*\in\mathcal{C}_m$ be a matrix achieving the maximal determinant among all matrices in $\mathcal{C}_m$. Then, the off-diagonal entries of $C_m^*$ belong to the set $\{-1,+3\}$.
\end{proposition}

It is enough then to understand the position of the elements $3$ and $-1$ along the matrices $C_m^*$ to give a general upper bound.

\begin{definition}\normalfont \label{def-EhlichBlock}  A matrix $M$ in $\mathcal{C}_m$ has an \textit{Ehlich block} of length $r$ if up to a symmetric permutation of rows and columns, there is an index $a$ such that
\[
\begin{cases}
m_{ij}=3 & \text{ for } i\neq j,\text{ and } i,j\in\{a+1,\dots,a+r\}\\
m_{ij}=-1& \text{ for } i\in \{a+1,\dots,a+r\},\text{ and } j\not\in \{a+1,\dots,a+r\}
\end{cases}.
\]
In other words, if $M$ has an Ehlich block of length $r$, then $M$ is can be symmetrically rearranged into the matrix
\[
\begin{bmatrix}
* & -J_{a,r} & *\\
-J_{r,a} & (n-3)I_{r}+3J_{r} & -J_{r,m-(a+r)}\\
* & J_{m-(a+r),r} & *
\end{bmatrix}=
\left[
\begin{array}{ccc|cccc|ccc}
&&&-1 & -1 & \dots & -1 &&&\\
&*&&\vdots&\vdots&&\vdots&&*&\\
&&&-1 & -1 & \dots & -1 &&&\\
\hline
-1 & \dots & -1 & n & 3 & \dots & 3 & -1 & \dots & -1\\
-1 & \dots & -1 & 3 & n & \dots & 3 & -1 & \dots & -1\\
\vdots &  & \vdots & \vdots & \vdots & \ddots & \vdots & \vdots &  & \vdots\\
-1 & \dots & -1 & 3 & 3 & \dots & n& -1 & \dots & -1\\
\hline
&&&-1 & -1 & \dots & -1 &&&\\
&*&&\vdots&\vdots&&\vdots&&*&\\
&&&-1 & -1 & \dots & -1 &&&\\
\end{array}
\right].
\]
A matrix $M$ is an \textit{Ehlich block matrix} if and only if $M$ is equivalent to a matrix consisting only of Ehlich blocks, by a series of symmetric row/column permutations.\index{matrix! Ehlich block}
\end{definition}

Notice that Ehlich block matrices of order $n$ are indexed by partitions of $n$. That is, for any partition of the number $n$ one obtains a unique Ehlich block matrix up to symmetric row/column permutations.

\begin{proposition}[Ehlich, Satz 2.3. \cite{Ehlich-3mod4}]\normalfont
Let $C_m^*\in\mathcal{C}_m$ be a matrix achieving the maximal determinant among all matrices in the set $\mathcal{C}_m$. Then $C_m^*$ is an Ehlich block matrix.
\end{proposition}

This shows, that to give an upper bound for the determinant of the matrices $C_n^*$, it is sufficient to find the maximum value of the determinant of an Ehlich block matrix among all possible partitions of $n$. For this purpose, Ehlich provides the following useful computation

\begin{lemma}[Ehlich, Satz 3.1 \cite{Ehlich-3mod4}]\normalfont Let $M$ be a n $m\times m$ Ehlich block matrix with $s$ blocks of length $r_i$, $i=1,\dots,s$, so that $r_1+r_2+\dots+r_s=m$. Then,
\[
\det(M)=(n-3)^{m-s}\prod_{i=1}^s(n-3+4r_i)\left(1-\sum_{i=1}^s\frac{r_i}{n-3+4r_i}\right).
\]
\end{lemma}
Setting $m=n$, Ehlich finds the optimal values of $s$ and $r_i$ so that the determinant in the lemma above is maximised. From here, Ehlich's bound in Theorem \ref{thm-EhlichBound} follows.

\subsection{Small real maximal determinant matrices}

We conclude this section with a summary of Hadamard's maximal determinant problem at small orders. William Orrick's website \cite{Orrick-Website} contains a database on Hadamard's maximal determinant problem. In particular, it includes all the known maximal determinant matrices, and record lower bounds of the determinant for matrices of order $n<120$, prior to 2012. To the best of our knowledge the only new maximal determinant matrix, not present in Orrick's website, is a matrix of order $n=22$ proved to be maximal by Chasiotis, Kounias, and Farmakis \cite{CKF-DOptimal22,CKF-DO22Corrigendum}. The following table is taken from \cite{Orrick-Website}, with the confirmed value of the maximal determinant at order $n=22$. 

\begin{table}[H]

\begin{tabular}{|ccc||ccc||ccc||ccc|}
\hline
$n$ & $\det/2^{n-1}$ & R & $n$ & $\det/2^{n-1}$ & R & $n$ & $\det/2^{n-1}$ & R & $n$ & $\det/2^{n-1}$ & R \\
\hline
 &  &  & 1 & $1$ & $1$ & 2 & $1$ & $1$ & 3 & $1$ & $1$ \\
4 & $2 \times 1^1$ & 1 & 5 & $3 \times 1^1$ & 1 & 6 & $5 \times 1^1$ & 1 & 7 & $9 \times 1^1$ & .98 \\
8 & $4 \times 2^3$ & 1 & 9 & $7 \times 2^3$ & .85 & 10 & $18 \times 2^3$ & 1 & 11 & $40 \times 2^3$ & .94 \\
12 & $6 \times 3^5$ & 1 & 13 & $15 \times 3^5$ & 1 & 14 & $39 \times 3^5$ & 1 & 15 & $105 \times 3^5$ & .97 \\
16 & $8 \times 4^7$ & 1 & 17 & $20 \times 4^7$ & .87 & 18 & $68 \times 4^7$ & 1 & 19 & $833 \times 4^6$ & .98 \\
20 & $10 \times 5^9$ & 1 & 21 & $29 \times 5^9$ & .91 & 22 & $100 \times 5^9$ & .95 & 23 & \textcolor{wpicrimson}{$42411 \times 5^6$??} & \textcolor{wpicrimson}{.93} \\
24 & $12 \times 6^{11}$ & 1 & 25 & $42 \times 6^{11}$ & 1 & 26 & $150 \times 6^{11}$ & 1 & 27 & \textcolor{wpicrimson}{$546 \times 6^{11}$??} & \textcolor{wpicrimson}{.97} \\
28 & $14 \times 7^{13}$ & 1 & 29 & \textcolor{wpicrimson}{$320 \times 7^{12}$??} & \textcolor{wpicrimson}{.87} & 30 & $203 \times 7^{13}$ & 1 & 31 & \textcolor{wpicrimson}{$784 \times 7^{13}$??} & \textcolor{wpicrimson}{.94} \\
32 & $16 \times 8^{15}$ & 1 & 33 & \textcolor{wpicrimson}{$441 \times 8^{14}$??} & \textcolor{wpicrimson}{.85} & 34 & \textcolor{wpicrimson}{$256 \times 8^{15}$??} & \textcolor{wpicrimson}{.97} & 35 & \textcolor{wpicrimson}{$1064 \times 8^{15}$??} & \textcolor{wpicrimson}{.94} \\
36 & $18 \times 9^{17}$ & 1 & 37 & $72 \times 9^{17}$ & .94 & 38 & $333 \times 9^{17}$ & 1 & 39 & \textcolor{wpicrimson}{$1440 \times 9^{17}$??} & \textcolor{wpicrimson}{.95} \\
\hline
\end{tabular}
\caption{The status of Hadamard's maximal determinant problem for $n<40$.}
\end{table}

In each column, the reader can find listed
\begin{enumerate}
\item The order $n$.
\item The value of the maximal determinant of a $\pm 1$ matrix of order $n$, divided by $2^{n-1}$.
\item The ratio of the maximal determinant at order $n$ to the corresponding bound in its congruence class. Namely, the Hadamard bound for $n\equiv 0\pmod{4}$, the Barba bound for $n\equiv 1\pmod{4}$, the Ehlich-Wojtas bound for $n\equiv 2\pmod{4}$, and the Ehlich bound for $n\equiv 3\pmod{4}$.
\end{enumerate}

The symbol `\textcolor{wpicrimson}{??}' is used to indicate that there is no proof yet that the value given is maximal. For the particular matrices attaining these values of the determinant, see  \cite{Orrick-Website}.

\section{General upper and lower determinantal bounds}

We turn now to more general results for matrices with entries in $\mu_m=\{1,\zeta_m,\dots,\zeta_m^{m-1}\}$.\\

In this section we will study general upper and lower bounds for the determinant of a matrix with entries in the set $\mu_m$ of $m$-th roots of unity. For this, we give a brief account of results on determinant theory. For the history of determinant theory the reader can consult Muir's four-volume treatise \cite{Muir-DeterminantBooks}. The survey \cite{Padraig-MaxDetSurvey} contains accessible proofs of some of the determinant bounds presented. Krattenthaller's paper  \cite{Krattenthaller-AdvancedDeterminantCalculus} discusses an interesting series of techniques to evaluate determinants.\\

\subsection{The generalised Barba bound}

\begin{theorem}[Muir-Kelvin bound, Theorem 7.8.1 \cite{Horn-Johnson}]\label{thm-MuirKelvinBound} Let $G$ be an $n\times n$ positive-definite matrix. Then\index{determinant inequality!Muir-Kelvin}
\[|\det G|\leq \prod_{i=1}^n g_{ii}.\]
Furthermore, $G$ meets the bound with equality if and only if $G$ is diagonal.
\end{theorem}
\begin{proof}
$G$ is positive-definite then letting $e_i$ be the $i$-th canonical basis vector, $g_{ii}=e_i^*Ge_i>0$, so its diagonal entries $g_{ii}$ must be real and positive. Let $\Delta=\diag(\sqrt{g_{11}},\sqrt{g_{22}},\dots,\sqrt{g_{nn}})$, and let
\[C=\Delta^{-1} G \Delta^{-1}.\]
Then $C$ is also Hermitian and positive definite, and $\tr(C)=n$. Let $\lambda_1,\dots,\lambda_n$ be the eigenvalues of $C$. Then by the inequality of arithmetic and geometric means
\[\det(C)=\lambda_1\dots\lambda_n\leq \left(\frac{\lambda_1+\dots+\lambda_n}{n}\right)^n=\left(\frac{1}{n}\tr(C)\right)^n=1.\]
Thus
\[\det(G)=\det(\Delta)^2\det(C)\leq\det(\Delta)^2=g_{11}g_{22}\dots g_{nn}.\]
Finally, notice that equality in the arithmetic-geometric mean inequality happens if and only if all $\lambda_i$'s are equal, i.e. when $C=I_n$. This in turn implies that equality in the bound occurs if and only if $G$ is diagonal.\qedhere
\end{proof}

From this result, Hadamard's determinant bound follows easily: Let $M$ be a matrix with entries of modulus $1$, then for the Gram matrix $G=MM^*$ the diagonal entries are $g_{ii}=n$. Hadamard's inequality then implies that
\[|\det(H)|^2=\det(HH^*)=\det(G)\leq\prod_{i=1}^n g_{ii}=n^n.\]
And equality holds if and only if $G$ is diagonal, which implies $HH^*=G=nI_n$.\\

When Hadamard's bound cannot be achieved, we saw in the $\pm 1$ case that the Barba bound, Theorem \ref{thm-BarbaBound}, is sharper. Here we adapt a matrix-theoretic proof of the Barba bound due to Wojtas \cite{Wojtas-Determinants}, and extend it to complex matrices. We show that this bound applies in the same form of Theorem \ref{thm-BarbaBound} to matrices with entries in $\mu_m$ if and only if $m=2,3,4$ or $6$. This generalised bound  had previously been obtained in the case $m=4$ by J.H.E Cohn with analytical methods. However, we state it here for the first time for $m$ arbitrary. For Cohn's analytic approach we refer the reader to Cohn's papers \cite{Cohn-Determinants-I, Cohn-ComplexDOptimal}.

\begin{proposition}[cf. \cite{Wojtas-Determinants,Padraig-MaxDetSurvey}]\normalfont\label{prop-DetBoundOffDiag}
Let $B$ an Hermitian positive-definite matrix of the following form
\[
B=
\begin{bmatrix}
m & g_{12} & \dots & g_{1k} & b_1\\
g_{12}^* & m & \dots & g_{2k} & b_2\\
\vdots & \vdots & \ddots & \vdots & \vdots\\
g_{1k}^* & g_{2k}^* &\dots & m & b_k\\
b_1^* & b_2^* & \dots & b_k^* & b
\end{bmatrix}.
\]
If $0<b\leq |b_i|$ for all $1\leq i\leq k$, then $\det(B)\leq b(m-b)^k$. Furthermore, $B$ achieves this bound with equality if and only if $|b_i|=b$ and $g_{ij}=b_ib_j^*/b$.
\end{proposition}
\begin{proof}
A series of simultaneous (Hermitian) elementary row and column operations shows that the matrix $B$ is equivalent to
\[B'=\begin{bmatrix}
m-|b_1|^2/b &g_{12}-b_1b_2^*/b &\dots & g_{1k}-b_1b_k^*/b & 0\\
g_{12}^*-b_2b_1^*/b & m-|b_2|^2/b & \dots &g_{2k}-b_2b_k^*/b & 0\\
\vdots & \vdots & \ddots & \vdots & \vdots\\
g_{1k}^*-b_kb_1^*/b& g_{2k}^*-b_kb_2^*/b & \dots & m-|b_k|^2/b & 0\\
0 & 0 & \dots & 0 & b
\end{bmatrix}\]
Let $D$ be the $k\times k$ principal submatrix of $B'$, then $\det(B)=b\det(D)$. By Sylvester's criterion, Theorem \ref{thm-SylvesterCriterion}, the matrix $D$ is Hermitian and positive-definite, so we can apply the Muir-Kelvin bound (Theorem \ref{thm-MuirKelvinBound}) to $D$ to obtain 
\[\det(B)=b\det(D)\leq b\prod_{i=1}^k\left(m-\frac{|b_i|^2}{b}\right)\leq b\prod_{i=1}^k(m-|b_i|)\leq b(m-b)^k.\]
The first inequality is an equality if and only if $D$ is diagonal, i.e. $g_{ij}=b_ib_j^*/b$ for all $i,j$, and the last inequality is an equality if and only if $|b_i|=b$ for all $i$.\qedhere
\end{proof}

\begin{theorem}[cf. \cite{Wojtas-Determinants, Padraig-MaxDetSurvey}]  \label{thm-GenBarba} Let $G$ be an $n\times n$ Hermitian positive-definite matrix, with diagonal entries $m$. If $b$ is a positive real number such that $b\leq |g_{ij}|$ for all off-diagonal entries $g_{ij}$. Then
\[\det G\leq (m+(n-1)b)(m-b)^{n-1}.\]
\end{theorem}
\begin{proof}
We prove the result by induction. For the base case $n=2$ we have that
\[\det
\begin{bmatrix}
m & g_{12}\\
g_{12}^* &m 
\end{bmatrix}=m^2-|g_{12}|^2\leq m^2-b^2=(m+b)(m-b).
\] 
Now, let
\[G=\begin{bmatrix}
m & g_{12}  & \dots  & g_{1n}\\
g_{12}^* & m  &\dots  & g_{2n}\\
\vdots & \vdots &\ddots &\vdots \\
g_{1,n}^* & g_{2,n}^*  &\dots & m 
\end{bmatrix},\]
and assume that the statement is true for all matrices of order $n-1$ satisfying the hypotheses. By linearity of the determinant on the rows of $G$ we have,
\begin{align*}
\det G &= 
\det\undermat{A}{\begin{bmatrix}
m & g_{12}  & \dots &g_{1,n-1} & g_{1n}\\
g_{12}^* & m  &\dots & g_{2,n-1} & g_{2n}\\
\vdots & \vdots &\ddots &\vdots &\vdots\\
g_{1,n-1}^* & g_{2,n-1}^*  &\dots & m & g_{n,n-1}\\
0 & 0  &\dots & 0 & m-b
\end{bmatrix}}
+
\det\undermat{B}{\begin{bmatrix}
m & g_{12}  & \dots &g_{1,n-1} & g_{1n}\\
g_{12}^* & m  &\dots & g_{2,n-1} & g_{2n}\\
\vdots & \vdots &\ddots &\vdots &\vdots\\
g_{1,n-1}^* & g_{2,n-1}^*  &\dots & m & g_{n,n-1}\\
g_{1n}^* & g_{2n}^*  &\dots & g_{n,n-1}^* & b
\end{bmatrix}}.
\end{align*}
\vspace{12pt}\\If the $\det(B)>0$, then by Sylvester's criterion (Theorem \ref{thm-SylvesterCriterion}) the matrix $B$ is positive-definite, in which case we can apply Proposition \ref{prop-DetBoundOffDiag} to obtain
\[\det(G)\leq (m-b)\det(G_{n-1})+b(m-b)^{n-1},\]
where $G_{n-1}$ is the $(n-1)\times(n-1)$ principal submatrix of both $G$, hence of $A$ as well. If instead $\det(B)\leq 0$, then we have
\[\det(G)\leq (m-b)\det(G_{n-1})\leq (m-b)\det(G_{n-1})+b(m-b)^{n-1}.\]
Therefore, in any case the induction hypothesis applied to $G_{n-1}$ implies
\begin{align*}
\det(G)&\leq (m-b)\det(G_{n-1})+b(m-b)^{n-1}\\
&\leq (m+(n-2)b)(m-b)^{n-1}+b(m-b)^{n-1}\\
&=(m+(n-1)b)(m-b)^{n-1},
\end{align*}
and this concludes the proof. \qedhere
\end{proof}

\begin{corollary}\normalfont \label{cor-BarbaGram}An Hermitian positive-definite matrix $G$ of order $n$, with $g_{ii}=n$ and $|g_{ij}|\geq b>0$ for all $i\neq j$ satisfies $\det(G)=(n+(n-1)b)(n-b)^{n-1}$ if and only if  there is a diagonal matrix $\Delta$, with diagonal entries of modulus $1$ such that
\[\Delta^* G \Delta=(n-b)I_n+bJ_n.\]
\end{corollary}
\begin{proof}
Suppose we have the equality $\det(G)=(n+(n-1)b)(n-b)^{n-1}$.  Then, following the notation of the proof of Theorem \ref{thm-GenBarba}, we must have the equality $\det B=b(m-b)^{n-1}$. It follows from Proposition \ref{prop-DetBoundOffDiag} that $|g_{in}|=b$ for $1\leq i\leq n$, and  $g_{ij}=g_{in}g_{jn}^*/b$ for $1\leq i,j\leq n-1$. 
Since $|g_{in}|^2=g_{in}g_{in}^*=b^2$, letting $\Delta=\diag(g_{1n}/b,\dots,g_{n-1,n}/b,1)$, we have that $G'=\Delta^* G\Delta$ satisfies $g'_{i,n}=b$ for $1\leq i\leq n-1$. Furthermore, the non-zero entries of $\Delta$ have modulus $1$, which implies that $\det(G')=\det(G)$, and that $|g'_{ij}|=|b|$. Now apply Proposition \ref{prop-DetBoundOffDiag} to $G'$ to find: $g'_{ij}=g'_{in}{(g_{jn}')}^*/b=b^2/b=b$ for all $i\neq j$, so
\[\Delta^*G\Delta=G'=(n-b)I_n+bJ_n.
\]
Conversely, if $\Delta^*G\Delta=(n-b)I_n+bJ_n$, then
\[\det(G)=\det((n-b)I_n+bJ_n)=(n+(n-1)b)(n-b)^{n-1}.\qedhere\]
\end{proof}

\begin{definition} \normalfont The $m$-th \textit{minimal root-sum} of order $n$, $\sigma_m(n)$ is defined as follows
\[\sigma_m(n)=\min\left\{\left|\sum_{i=1}^{n}\zeta_m^{a_i}\right|: a_i\in \{0,\dots,m-1\}, \text{ for } 1\leq i \leq n\right\},\]
in other words, $\sigma_m(n)$ is the minimal absolute value of the sum of the elements of an $n$-subset of the set $\mu_m$ of $m$-th roots of unity.
\end{definition}

\begin{example}\normalfont For $m=5$, the values of $\sigma_5(n)$ can become arbitrarily small. For example, letting $\zeta_5=e^{2\pi i/5}$, we have that
\[|\zeta_5+\zeta_5^4|=\frac{1}{\varphi},\]
where $\varphi=\frac{1+\sqrt{5}}{2}$ is the golden ratio. Notice that by the binomial theorem, the element $(\zeta_5+\zeta_5^4)^n$ can be interpreted as a sum of $2^n$ fifth roots of unity. Therefore $\sigma_5(2^n)\leq 1/\varphi^n$.
\end{example}

\begin{theorem}[cf. Barba, \cite{Barba-Bound}] \label{thm-BarbaBound} Let $M$ be an $n\times n$ matrix with entries in the set $\mu_m$ of $m$-th roots of unity. Suppose that the $m$-th minimal root-sum $\sigma_m(n)$ is positive. Then,
\[|\det M|\leq \sqrt{(n+(n-1)\sigma_m(n))}(n-\sigma_m(n))^{(n-1)/2}.\]
Furthermore there is equality in the bound if and only if there exists a diagonal matrix $\Delta$ with non-zero entries of modulus $1$, such that $B=\Delta^* M$ satisfies $BB^*=(n-\sigma_m(n))I_n+\sigma_m(n)J_n$.
\end{theorem}
\begin{proof}
Let $G=MM^*$, then $G$ is Hermitian positive-definite. Since every entry of $M$ has modulus $1$, the diagonal entries of $G$ are all $n$. Furthermore, since $\mu_m$ is closed under multiplication, we have that the off-diagonal entries $g_{ij}$ are sums of $m$-th roots of unity. Therefore, by definition $|g_{ij}|\geq \sigma_m(n)$ for $i\neq j$. We can then apply Theorem \ref{thm-GenBarba} to $G$ to obtain,
\[|\det(M)|^2=\det(MM^*)=(n+(n-1)\sigma_m(n))(n-\sigma_m(n))^{n-1}.\]
Hence, taking square-roots the result follows. By Corollary \ref{cor-BarbaGram}, we have that $M$ meets the bound with equality if and only if there is a diagonal matrix $\Delta$ with non-zero entries of modulus $1$ such that
\[\Delta^*(MM^*)\Delta=(n+(n-1)\sigma_m(n))(n-\sigma_m(n))^{n-1}.\]
Therefore, the matrix $B=\Delta^* M$ is as required.\qedhere
\end{proof}

\begin{remark}\normalfont \label{rem-BarbaGramIssue} The reader may have noticed that in Theorem \ref{thm-BarbaBound}, when we talk about the existence of a matrix $B$ with entries of modulus $1$ satisfying 
\[BB^*=(n-\sigma_m(n))I_n+\sigma_m(n)J_n,\]
we never say that $B$ is monomially equivalent to $M$, see Definition \ref{def-MonomialEquivalence}. The reason for this is that the entries of the diagonal matrix $\Delta$ are not necessarily $m$-th roots of unity, so the matrix $B$ is not guaranteed to have entries in $\mu_m$. Every non-zero entry of $\Delta$ is of the type $a/\sigma_m(n)$, where $a$ is a sum of $n$ $m$-th roots of unity satisfying $|a|=\sigma_m(n)$. Even assuming that there is a sum of $n$ roots of unity, say $b$, whose value is exactly $\sigma_m(n)$, the element $a/\sigma_m(n)=a/b$ can only be guaranteed to belong to $\Q[\zeta_m]$, i.e. it is not necessarily an algebraic integer.
\end{remark}

We would like to have a characterisation for complex Barba matrices in terms of their Gram matrix up to monomial equivalence, similar to the real case in Theorem \ref{thm-BarbaBound}. To find such a characterisation will require a bit of number theory and some technicality. This effort is worthwhile, since without a canonical form for Barba matrices a theoretical study of existence and non-existence becomes much harder.\\

We show that whenever $\sigma_m(n)=1$, the issue in Remark \ref{rem-BarbaGramIssue} does not arise. We saw this already when $m=2$ and $n$ is odd, since clearly all odd sums of elements $\pm 1$ have absolute value at least $1$. Recall the following result of Kronecker.

\begin{theorem}[Kronecker, cf. \cite{Greiter-KroneckerTheorem}]\label{thm-KroneckerThm} Let $f$ be an irreducible monic polynomial with integer coefficients. Assume that all roots of $f$ have modulus $1$. Then all roots of $f$ are roots of unity.
\end{theorem}
\begin{proof}
We follow an elementary, matrix-theoretic  proof due to Greiter, see \cite{Greiter-KroneckerTheorem}. Let $f(x)=x^n+a_{n-1}x^{n-1}+\dots+a_1x+a_0$ be an irreducible monic polynomial with integer coefficients. Let
\[
A=\begin{bmatrix}
0 & 0 & \dots &0& -a_0\\
1 & 0 & \dots &0& -a_1\\
0 & 1 & \dots &0& -a_2\\
\vdots & \vdots & \ddots&\vdots& \vdots\\
0 & 0 & \dots & 1 & -a_{n-1}
\end{bmatrix},
\]
be the companion matrix of $f$. By elementary row operations in the determinant $\det(xI-A)$, it is easy to check that the characteristic polynomial of $A$ is $f$. Furthermore, the matrix $A$ is diagonalisable as a complex matrix, see Theorem 3.3.14 and Corollary 3.3.10 of \cite{Horn-Johnson}. Therefore, there exists a matrix $V$ with complex entries such that $A=VDV^{-1}$ where $D=\diag(\alpha_1,\dots,\alpha_n)$ and $\alpha_i$ are the roots of $f$. By assumption $|\alpha_i|=1$ for all $i$. Denote by $|M|$ the matrix obtained by applying the absolute value to $M$ entrywise. Then, $|D|=I_n$ which implies $|VD|=|V|$. Therefore, the entries of the matrices in the set 
\[X=\{A^t=VD^tV^{-1}:t\in \N\},\]
are bounded. Since $X$ is a subset of the set $\Mat_n(\Z)$ of matrices with integer coefficients, it follows that $X$ is a finite set. But then there are $s,t\in\N$ such that $A^t=A^{t+s}$, this implies $D^{t}=D^{t+s}$ and then $\alpha_i^s=1$ for all $i$.\qedhere
\end{proof}
\begin{corollary}\normalfont \label{cor-CyclotomicIntegersAbs1} If $\alpha\in \Z[\zeta_m]$ has modulus $1$, then $\alpha$ is an $m$-th root of unity if $m$ is even, or a $2m$-th root of unity if $m$ is odd.
\end{corollary}
\begin{proof}
Since $\alpha\in\Z[\zeta_m]$, then $\alpha$ is an algebraic integer, i.e. $\alpha$ is a root of an irreducible monic polynomial $f$ with integer coefficients. Let $G=\Gal(\Q[\zeta_m]/\Q)$, then $f(T)=\prod_{\sigma\in G}(T-\sigma(\alpha))$. Since the Galois group $G$ is cyclic, all Galois automorphisms commute with complex conjugation. Let $\tau\in G$ be the Galois automorphism induced by complex conjugation. Then for all $\sigma\in G$
\[|\alpha^{\sigma}|^2=(\alpha^{\sigma})(\alpha^{\sigma})^{\tau}=(\alpha^{\sigma})(\alpha^{\tau})^{\sigma}=(\alpha\alpha^{\tau})^{\sigma}=(|\alpha|^2)^{\sigma}=1.\]
 So we can apply Kronecker's theorem, Theorem \ref{thm-KroneckerThm}, to $f$. This implies that $\alpha$ is a root of unity. Since $\alpha\in\Z[\zeta_m]$, then $\alpha=\pm \zeta_m^i$ for some $i$. This shows that $\alpha$ is an $m$-th root of unity or a $2m$-th root of unity depending on the parity of $m$. \qedhere
\end{proof}

\begin{corollary}\normalfont\label{cor-BarbaLattice} Suppose that $\sigma_{m}(n)=1$. If $m$ is even, then for an $n\times n$ matrix $M$ with entries in $\mu_m$, $|\det(M)|$ meets the Barba bound with equality if and only if $M$ is monomially equivalent to a matrix $B$ with entries in the $m$-th roots of unity satisfying 
\[BB^*=(n-1)I_n+J_n.\]
\end{corollary}
\begin{proof}
By Theorem \ref{thm-BarbaBound}, we know that $|\det(M)|=\sqrt{2n-1}(n-1)^{(n-1)/2}$ if and only if $M=\Delta^* B\Delta$, where 
\[BB^*=(n-1)I_n+J_n.\]
The entries of $\Delta$ are furthermore of the type $\alpha/\sigma_m(n)$, where $|\alpha|=\sigma_m(n)$. By assumption, $\sigma_m(n)=1$, so $\alpha\in\Z[m]$ satisfies $|\alpha|=1$. By Corollary \ref{cor-CyclotomicIntegersAbs1}, we have that $\alpha$ is an $m$-th root of unity. Therefore, $\Delta$ is a diagonal matrix whose non-zero entries are in $\mu_m$. This implies that the entries of $B$ also belong to $\mu_m$.\qedhere
\end{proof}

We now characterise the values of $m$ for which $\sigma_m(n)=1$ or $\sigma_m(n)=0$ for all $n$.

\begin{definition}\label{def-Discrete-Subring} \normalfont A \textit{discrete subring} of $\C$ is a subring $R$ of $\C$ where every element of $R$ is isolated with respect to the Euclidean topology of $\C$. Namely, for every $x\in R$ there exists a real number $\varepsilon>0$ such that the ball $B_{\varepsilon}(x)$ of radius $\varepsilon$ centred at $x$ satisfies $B_{\varepsilon}(x)\cap R=\{x\}$.
\end{definition}

It is easy to check that $R$ is a discrete subring of $\C$ if and only if the distance between any pair of elements of $R$ is at least $1$. We will show that $\Z[\zeta_m]$ is a discrete subring if and only if $m=1,2,3,4$ or $6$. This shows that for these values of $m$ all sums of roots of unity are either vanishing or of absolute value at least $1$.  We will need a few results from the theory of Diophantine approximation, see Chapter 7 of \cite{Apostol-ModularFunctions}.

\begin{theorem}[Dirichlet's approximation theorem, Theorem 7.9, \cite{Apostol-ModularFunctions}]\label{thm-DirichletApproximation}
For any real number $\theta$ and any integer $N>0$, there exist integers $a$ and $b$ with $0\leq b\leq N$ such that
\[|b\theta-a|<\frac{1}{N}.\]
\end{theorem}
\begin{proof}
Let $\{x\}=x-[x]$ be the fractional part of $x\in\R$. Consider the set of $N+1$ real numbers
\[X:=\{0,\{\theta\},\{2\theta\},\dots,\{N\theta\}.\}\]
Then, all elements of $X$ lie in the interval $I=[0,1)$. Dividing $I$ into $N$ subintervals of length $1/N$, we have by the pigeonhole principle that there are two elements $\{r\theta\},\{s\theta\}\in X$ with $0\leq s<r\leq N$ which are in the same subinterval. Therefore,
\[|\{r\theta\}-\{s\theta\}|<\frac{1}{N}.\]
We have that,
\[\{r\theta\}-\{s\theta\}=(r-s)\theta-([r\theta]-[s\theta]).\]
Letting $b=(r-s)\in\Z$, and $a=[r\theta]-[s\theta]\in\Z$, we find
\[|b\theta-a|<1/N, \text{ and } 0<b\leq N.\qedhere\]
\end{proof}
\begin{corollary}[cf. Theorem 7.12, \cite{Apostol-ModularFunctions}]\normalfont\label{cor-IrrationalPeriods} Let $\omega_1$ and $\omega_2$ be two complex numbers such that the ratio $\omega_2/\omega_1$ is real and irrational. Then for every $\varepsilon>0$, there is an element $z\in \omega_1\Z \oplus \omega_2\Z$ with $0<|z|<\varepsilon$.
\end{corollary}
\begin{proof}
Apply Dirichlet's approximation Theorem to $\theta=\omega_2/\omega_1\in\R-\Q$, and $N>|\omega_1|/\varepsilon$ an integer. Then, for any $\varepsilon>0$, there exist integers $a$ and $b$ such that
\[|b\theta-a|<\frac{1}{N}<\frac{\varepsilon}{|\omega_1|}.\]
Multiplying by $|\omega_1|$ we find
\[|b\omega_2-a\omega_1|<\varepsilon.\]
Letting $z=b\omega_2-a\omega_1\in \Z[\omega_1,\omega_2]$, we find that $|z|<\varepsilon$. Finally $z\neq 0$, since otherwise $\theta=a/b$, but $\theta$ is irrational by assumption.\qedhere
\end{proof}

Using this corollary, the following theorem follows from a straightforward, although a bit lengthy, case analysis.

\begin{theorem}[Theorem 7.13. \cite{Apostol-ModularFunctions}]\label{thm-3periods} Let $\omega_1,\omega_2,$ and $\omega_3$ be complex numbers which are linearly independent over $\Z$. Then, for every $\varepsilon>0$, there is an element  $z\in \omega_1\Z\oplus\omega_2\Z\oplus\omega_3\Z$ such that $0<|z|<\varepsilon$.
\end{theorem}

\begin{theorem} \label{thm-DSRCharacterisation}  $\Z[\zeta_m]$ is a discrete subring of $\C$ if and only if $m=1,2,3,4$ or $6$.
\end{theorem}
\begin{proof}
If the degree of the field extension $\Q\subset\Q[\zeta_m]$ is $\geq 3$, then  $\Z[\zeta_m]$ has at least $3$ linearly independent elements over $\Z$. Therefore, Theorem \ref{thm-3periods} implies that for every $\varepsilon>0$ there is a $z\in\Z[\zeta_m]$ with $0<|z|<\varepsilon$. Hence, $0$ is not isolated in $\Z[\zeta_m]$. The cyclotomic extensions of degree $2$ are exactly $\Q[\zeta_3]=\Q[\zeta_6]$ and $\Q[\zeta_4]$, the cyclotomic extensions of degree $1$ are $\Q=\Q[\zeta_1]=\Q[\zeta_2]$. Conversely, it is easy to show that no pair of distinct elements of $\Z[\zeta_m]$ is at distance $<1$, whenever $m=1,2,3,4$ or $6$. \qedhere
\end{proof}

From this, it is easy to check the following:

\begin{corollary}\normalfont\label{cor-MinSumLattice} Let $\sigma_m(n)$ be the minimal sum of $m$-th roots at order $n$. Then,
\begin{itemize}
\item[(i)] For $m=3$, $\sigma_3(n)=0$ if $3\mid n$ and $\sigma_3(n)=1$ otherwise.
\item[(ii)] For $m=4$, $\sigma_4(n)=0$ if $n$ is even, and $\sigma_4(n)=1$ otherwise.
\item[(iii)] For $m=6$, $\sigma_6(n)=0$ for all $n>1$.
\end{itemize}
\end{corollary}
%
%
%
We immediately deduce the following generalisation of the Barba bound:

\begin{theorem}[Barba bound over the fourth roots, cf. Cohn \cite{Cohn-ComplexDOptimal}]\label{thm-BarbaBound4Roots} If $n$ is odd, then the determinant of a matrix $M$ with entries in $\{\pm 1,\pm i\}$ satisfies
\[|\det M|\leq \sqrt{2n-1}(n-1)^{(n-1)/2}.\]
Furthermore, equality is achieved if and only if $M$ is monomially equivalent to a matrix $B$, with entries in $\{\pm 1,\pm i\}$ such that 
\[BB^*=(n-1)I_n+J_n.\]
\end{theorem}
\begin{proof}
This follows directly from Corollary \ref{cor-BarbaLattice}, and Corollary \ref{cor-MinSumLattice}.
\end{proof}

With a bit more work, we can show that the Barba bound also holds over the third roots.

\begin{lemma}\normalfont \label{lemma-HexLatticeColouring}
Let $\sigma$ be a sum of third roots of unity of length $n$, with $3\nmid n$. Suppose $|\sigma|=1$, then
\[
\begin{cases}
\sigma\in \{1,\omega,\omega^2\} & \text{ if } n\equiv 1\pmod{3}\\
\sigma\in \{-1,-\omega,-\omega^2\} & \text{ if } n\equiv 2\pmod{3}\\
\end{cases}
\]
\end{lemma}
\begin{proof}
Let $(1-\omega)$ be the principal ideal generated by the element $1-\omega\in\Z[\omega]$. Then,
\[\Z[\omega]/(1-\omega)\simeq \Z/3\Z.\]
Indeed, for $a+b\omega\in\Z[\omega]$ we have $a+b\omega\equiv a+b\pmod{(1-\omega)}$, since $a+b=(a+b\omega)+b\cdot(1-\omega)$. Now, $(1-\omega^2)(1-\omega)=3$, so if $a+b\equiv c\pmod{3}$ where $c\in\{0,1,2\}$, then $a+b\equiv r\pmod{(1-\omega)}$ as well. So each element in $\Z[\omega]/(1-\omega)$ has a unique representative in the set $\{0,1,2\}$ which establishes the isomorphism. Now, we show that $1,\omega$ and $\omega^2$ are all congruent to $1$ modulo $(1-\omega)$. We have,
\begin{align*}
&1\equiv 1\pmod{(1-\omega)},\\
&\omega=1+(1-\omega)\cdot(-1)\equiv 1\pmod{(1-\omega)}, \text{ and }\\
&\omega^2=\omega+(1-\omega)\cdot(-\omega)\equiv \omega\equiv 1\pmod{(1-\omega)}.
\end{align*}
Since $-1\equiv 2\pmod{3}$, it follows that $-1\equiv 2\pmod{(1-\omega)}$, and hence $-1,-\omega$, and $-\omega^2$ are all congruent to $2$ modulo $(1-\omega)$. The elements $1$, $\omega$ and $\omega^2$ are all congruent to $1$, so any sum of third roots of unity is congruent to its length modulo $(1-\omega)$.\qedhere
\end{proof}

\begin{theorem}[Barba bound over the third roots]\label{thm-Barba3} Let $n>2$ be an integer not divisible by $3$, then the determinant of a matrix $M$ with entries over the third roots $\{1,\omega,\omega^2\}$ satisfies
\[|\det M|\leq \sqrt{2n-1}(n-1)^{(n-1)/2}.\]
Furthermore, equality is achieved if and only if $n\equiv 1\pmod{3}$ and $M$ is monomially equivalent to a matrix $B$, with entries in $\{1,\omega,\omega^2\}$ such that $BB^*=(n-1)I_n+J_n$.
\end{theorem}
\begin{proof}
Since $n$ is not divisible by $3$, we have from Corollary \ref{cor-MinSumLattice} that $\sigma_3(n)=1$. Theorem \ref{thm-GenBarba} then implies that a matrix $M$ with entries in $\mu_3$ satisfies
\[|\det M|\leq \sqrt{2n-1}(n-1)^{(n-1)/2}.\]
Let $G=MM^*$, then the proof of  Corollary \ref{cor-BarbaGram} shows that letting $\Delta=(g_{1n},\dots,g_{n-1,n},1)$, the matrix $B=\Delta^*M$ satisfies
\[BB^*=(n-1)I_n+J_n.\]
The elements $g_{ij}$, $1\leq i<j\leq n$ are sums of third roots of length $n$. By Lemma \ref{lemma-HexLatticeColouring}, we have these elements belong to either $\{1,\omega,\omega^2\}$ or $\{-1,-\omega,-\omega^2\}$ according to the congruence class of $n$ modulo $3$. But since $n>2$, we have that $g_{12}=g_{13}g_{23}^*$, considering this equation modulo $(1-\omega)$ we find
\[n\equiv g_{12}=g_{13}g_{23}^*\equiv n^2\equiv 1\pmod{3}.\]
So this implies that $n\equiv 1 \pmod{3}$, and in particular all elements $g_{i,n}$, $1\leq i\leq n-1$, are third roots of unity. This implies that the entries of $B$ are also in $\{1,\omega,\omega^2\}$.
 \qedhere
\end{proof}

Since $\sigma_6(n)=0$ for all $n>1$, the Barba bound can never be applied for matrices on the sixth roots. The only obstruction to the existence of $\BH(n,6)$ matrices seems to be the determinant obstruction of Theorem \ref{thm-ButsonNonEx}. In fact, in Table \ref{tab-BH6} we can see that this is confirmed for all but $12$ orders $n\leq 100$.\\

We conclude this subsection with the following result

\begin{theorem}[cf. Theorem 18 \cite{Padraig-MaxDetSurvey}, and Theorem 2 \cite{Cohn-ComplexDOptimal}] \label{thm-BarbaConstantRowSum} Suppose there is a Barba matrix of order $n$, with entries over the $m$-th roots of unity, $m\in\{2,3,4\}$. Then, there is a normal Barba matrix of order $n$ over the $m$-th roots of unity with constant row-sum.
\end{theorem}
\begin{proof}
Let $B$ be a Barba matrix, then by Theorem \ref{thm-BarbaBound} there is a monomial matrix $Q_1$ such that
\[Q_1BB^*Q_1^*=(n-1)I_n+J_n.\]
Since $|\det(B^*)|=|\det(B)|$, then $B^* $ is also a Barba matrix, and again there is a $\mu_m$ monomial matrix $Q_2$ such that $Q_2B^*BQ_2^* = (n-1)I_n+J_n.$ Letting $N=Q_1BQ_2^*$, we find
\[
NN^*=(n-1)I_n+J_n=N^*N.
\]
Since $(NN^*)N=N(N^*N)$, it follows that $N$ commutes $J_n$, i.e. $NJ_n=J_nN$ so $N$ must have constant row and column sum.\qedhere
\end{proof}

\subsection{Determinant lower bounds from Bush-type matrices}

\begin{proposition}\normalfont \label{prop-ConstRSBoundComp} Let $H$ be an Hadamard matrix of order $n$ with constant row sum. Let $M$ be the following bordered matrix,
\[M=\left[
\begin{array}{cc}
1 & \mathbf{1}_n^{\intercal}\\
\mathbf{1}_n & H
\end{array}
\right].
\]
Then
\[|\det M|\geq (\sqrt{n}+1)n^{n/2}. \]
\end{proposition}
\begin{proof}
Since $H$ has constant row sum, there exists a complex number $s$ such that $HJ_n=sJ_n$. Taking the conjugate transpose we see that $J_nH^*=s^*J_n$, therefore
\[s^*J_nH=J_nH^*H=J_nHH^*=nJ_n.\]
It follows that $J_nH=(n/s^*)J_n$, so $H$ also has constant column sum. Therefore by comparing the row sum and column sum, we find that the \textit{excess} of $H$, i.e. the sum of all entries of $H$ has value
\[ns=\frac{n^2}{s^*},\]
and hence $ss^*=n$, which also implies that the column sum of $H$ is $s$. Using this fact, we can compute the Gram matrix of $M$ as follows,
\[MM^*=
\begin{bmatrix}
n+1 & (1+s)\mathbf{1}_n^{\intercal}\\
(1+s)\mathbf{1}_n & nI_n+J_n
\end{bmatrix}.
\]
Let $\alpha=1+s$, a series of elementary row and column operations implies the following similarity of matrices,
\[
\begin{bmatrix}
n+1 & \alpha^*\mathbf{1}\\
\alpha\mathbf{1} & nI_n+J_n
\end{bmatrix}
=\left[
\begin{array}{cc|c}
n+1 & \alpha & \alpha\mathbf{1}_{n-1}^{\intercal}\\
\alpha^* &n+1 & \mathbf{1}_{n-1}^{\intercal}\\
\hline
\alpha^*\mathbf{1}_{n-1}&\mathbf{1}_{n-1} & nI_{n-1}+J_{n-1}
\end{array}
\right]
\simeq
\left[
\begin{array}{cc|c}
n+1 & \alpha^* & \mathbf{0}_{n-1}^{\intercal}\\
n\alpha & 2n & \mathbf{0}_{n-1}^{\intercal}\\
\hline
\alpha\mathbf{1}_{n-1} & \mathbf{1}_{n-1} & nI_{n-1}
\end{array}\right].
\]
Taking determinants, and using the fact that $|\alpha|^2=\alpha\alpha^*=n+1+2\re(s)$, we find
\[|\det(M)|^2=(2(n+1)-|\alpha|^2)n^{n}=(n+1-2\re(s))n^n.\]
From the fact that $\re(s)\leq \sqrt{n}$, it follows that
\[|\det(M)|^2\geq (n+1-2\sqrt{n})n^n=(\sqrt{n}+1)^2n^n,\]
and taking square-roots the result follows. \qedhere
\end{proof}
\begin{remark}\normalfont From the proof in Proposition \ref{prop-ConstRSBoundComp}, we also find that if $n$ is the order of a $\BH(n,m)$ with constant row-sum, then
\[n=ss^*,\]
for some $s\in \Z[\zeta_m]$. In other words, $n$ is a norm in the quadratic extension $\Q[\zeta_m+\zeta_m^{-1}]\subset \Q[\zeta_m]$. In Chapter \ref{chap-HermitianForms} we characterised the integers $n$ that are norms in this extension. In particular in the case of $\pm 1$ matrices, we find that $n$ must be a square.
\end{remark}

 Proposition \ref{prop-ConstRSBoundComp} says that whenever we have a $\BH(n,m)$ matrix with constant row-sum, then there is a large-determinant matrix of order $n+1$ with entries in $\mu_m$. The following construction takes an arbitrary $\BH(n,m)$ matrix and produces a $\BH(n^2,m)$ of \textit{Bush-type}. Bush-type Hadamard matrices are block Hadamard matrices with diagonal blocks equal to $J_n$, and off-diagonal blocks with zero row-sum, so Bush-type matrices have constant row-sum $n$. This type of Hadamard matrix was introduced by Bush in \cite{Bush-HadamardProjPlanes}, where the author was interested in the non-existence of these matrices in relation to the existence question of projective planes. The following result showing existence is well-known, see for example Kharaghani's paper \cite{Kharaghani-BushType}.
\begin{theorem}[cf. \cite{Kharaghani-BushType}] \normalfont\label{thm-BushTypeConst} Suppose that $H$ is a dephased $\BH(n,m)$. Let $r_i$ be the $i$-th row of $H$, and let $E_i=r_i^*r_i$ be the rank-$1$ projection matrix onto the subspace spanned by $r_i$. Then the block-circulant matrix $M=[E_{i-j}]_{ij}$, i.e.
\[
M=\begin{bmatrix}
E_0 & E_1& E_2 & \dots &E_{n-1}\\
E_{n-1} & E_0 & E_1 & \dots & E_{n-2}\\
E_{n-2} & E_{n-1} & E_0 & \dots & E_{n-3}\\
\vdots & \vdots & \vdots & \ddots & \vdots\\
E_{1} & E_2 & E_3 & \dots & E_0
\end{bmatrix},
\]

 is a $\BH(n^2,m)$ with constant row sum $n$.
\end{theorem}
\begin{proof}
Since $H$ is dephased, the first row of $H$ consists of the all-ones vector, $\mathbf{1}_n^{\intercal}$. Then, $E_0=\mathbf{1}_n\mathbf{1}_n^{\intercal}=J_n$, and $E_0J_n=nJ_n$. It is enough to check that $E_iE_j^*=0$ for $i\neq j$ and that $\sum_{i}E_iE_i^*=n^2I_n$. From the fact that $r_ir_j^*=\delta_{ij}n$, it follows
\[E_iE_j^*=(r_i^*r_i)(r_j^*r_j)=n\delta_{ij}r_i^*r_j=n\delta_{ij}E_i.\]
In particular, we have that $E_iJ_n=E_iE_0=0$ for $i\neq 0$, which implies that $M$ has constant row sum $n$. To show that $\sum_i E_iE_i^*=n^2 I_n$ we show that $\sum_{i}E_i=nI_n$. Notice that $\{r_0^*,\dots, r_{n-1}^*\}$ forms a basis for an $n$-dimensional vector space. Since
\[\left(\sum_i E_i\right)r_{j}^*=\sum_{i}r_i^*r_ir_j^*=\sum_{i}r_i^*n\delta_{ij}=nr_j^*,\]
it follows that $\sum_{i}E_i=nI_n$. Therefore,
\[\sum_i E_iE_i^*=\sum_{ij}E_iE_j^*=\left(\sum_i E_i\right)\left(\sum_j E_j\right)^*=n^2 I_n.\]
This shows that $MM^*=n^2I_{n^2}$. Finally, the entries of each $E_i=r_i^*r_i$ are $m$-th roots of unity, since each $r_i$ consists of $m$-th roots of unity, thus $M$ is a $\BH(n^2,m)$.\qedhere
\end{proof}

We remark that Bush type matrices may also exist at other square orders. For example, Janko \cite{Janko-Bush36} showed that there is a Bush-type $\BH(36,2)$, yet no $\BH(6,2)$ exists.\\ 

Combining the results above we obtain the following,
\begin{theorem}\label{thm-BushLowerBound}
If there is a $\BH(n,m)$, then there is a matrix of order $n^2+1$ with entries in the $m$-th roots of unity $M$ such that
\[|\det(M)|\geq (n+1)n^{n^2}.\]
\end{theorem}
\begin{proof}
The existence of a $\BH(n,m)$ implies the existence of a $\BH(n^2,m)$, say $H$, with constant row sum by Theorem \ref{thm-BushTypeConst}. Apply Proposition \ref{prop-ConstRSBoundComp} to $H$ to obtain the lower bound in the statement.\qedhere
\end{proof}
\begin{corollary}\normalfont \label{cor-BushLowerBound}For all $m$ and $t\geq 1$, 
\[\gamma_m(m^{2t}+1)\geq (m^{t}+1)m^{tm^{2t}}. \]
\end{corollary}
\begin{proof}
The Fourier matrix $F_m$ is a $\BH(m,m)$ matrix. Sylvester's construction, Proposition \ref{prop-SylvesterConstruction}, implies that there exists a $\BH(m^t,m)$ for all $t\geq 1$. Then Theorem \ref{thm-BushLowerBound} with $n=m^t$ yields the result.\qedhere
\end{proof}
Notice that 
\[\lim_{t}\frac{\gamma_m(m^{2t}+1)}{h(m^{2t}+1)}\geq\lim_{t}\frac{(m^t+1)m^{tm^{2t}}}{(m^{2t}+1)^{(m^{2t}+1)/2}}=\frac{1}{\sqrt{e}},\]
so our construction achieves at least 60\% of the Hadamard bound infinitely often.\\

Conversely one can consider a $\BH(n+1,m)$ matrix and take its core after dephasing:
\begin{lemma}\normalfont\label{lemma-CoreDet} Let $H$ be a $\BH(n+1,m)$, then there is a matrix $C$ of order $n$ with entries in the $m$-th roots, such that
\[|\det C|=(n+1)^{(n-1)/2}.\]
\end{lemma} 
\begin{proof}
Let $C$ be the core of $H$ after dephasing, then 
\[CC^*=(n+1)I_n-J_n,\]
which implies that $|\det C|^2=(n+1)^{n-1}$.\qedhere
\end{proof}
However, we have that
\[\lim_n \frac{(n+1)^{n-1}}{n^n}=0,\]
so the ratio between the determinant of $C$ and the Hadamard bound decays to $0$ as $n$ grows larger. Nonetheless, $C$ has a large determinant value for small values of $n$.

\subsection{Generalised Paley cores}
Here we present a construction which generalises the Paley core to matrices with entries over the $m$-th roots. This construction can be used to build matrices with large determinants. In Section \ref{sec-Maxdet3} we will show an example of this on the third roots, but for now we state the construction in full generality.

\begin{proposition}[Generalised Paley cores]\normalfont\label{prop-GenPaleyCore} \index{generalised Paley core}Let $m>1$ be an integer, and $q$ a prime power such that $q\equiv 1\pmod{m}$. Then there is a matrix $Q$ of order $q$, called a \textit{generalised Paley core}, with entries in $\{0,1,\zeta_m,\zeta_m^2,\dots,\zeta_m^{p-1}\}$ such that
\begin{enumerate}
\item $QQ^*=qI_q-J_q$, and
\item $QJ_q=0$.
\end{enumerate}
\end{proposition}
\begin{proof}
The group $\F_q^{\times}$ is cyclic, so let $\gamma$ be a generator. Then, since $q\equiv 1\pmod{m}$, i.e. $m\mid (q-1)$, there is a non-trivial subgroup of $m$-th powers in $\F_q^{\times}$ given by 
\[H=\{\gamma^{am}: a\in \{0,1,\dots,(q-1)/m\}\}.\]
A complete set of cosets of $H$ is given by
\[\{H,\gamma H,\gamma^2 H,\dots,\gamma^{m-1}H\}.\]
For every $x\in\F_q^{\times}$ define $\chi(x)=\zeta_m^i$ if and only if $x\in \gamma^i H$ for $0\leq i\leq m-1$. Additionally let $\chi(0)=0$, then it is easy to check that $\chi$ is a  character of $\F_q$ of order $m$.\\

Let $Q$ be the matrix indexed by elements of $\F_q$ given by 
\[Q_{xy}=\chi(x-y).\]
Then $Q$ is a $q\times q$ matrix with entries in the set $\{0,1,\zeta_m,\dots,\zeta_m^{m-1}\}$. We have that
\[[QQ^*]_{xx}=\sum_{y}\chi(x-y)\overline{\chi(x-y)}=\sum_{x\neq 0} 1=q-1.\]
Now, if $x\neq y$ then
\[[QQ^*]_{xy}=\sum_z \chi(x-z)\overline{\chi(y-z)}=\sum_{z\neq y}\chi\left(\frac{x-z}{y-z}\right).\]
Do the change of variables over $\F_q$ given by $c=(x-z)/(y-z)$, so that $z=(yc-x)/(c-1)$ and $c\neq 1$. Therefore
\[[QQ^*]_{xy}=\sum_{c\neq 1}\chi(c)=-\chi(1)+\sum_c\chi(c)=-\chi(1)=-1.\]
Here, we used the fact that the sum over all $c\in\F_q$ of the values of a non-trivial character at $c$ is $0$, see Lemma \ref{lemma-VanishingCharSum}. This shows that $QQ^*=(q+1)I_q-J_q$. To show that $QJ_q=0$, notice that the row-sum of $Q$ is a sum of the type $\sum_{x\in\F_q}\chi(x)$, which again is vanishing by Lemma \ref{lemma-VanishingCharSum}. 

\end{proof}

From the generalised Paley cores, we can obtain a new family of \textit{generalised weighing matrices} (GWMs). For another new family of GWMs that we found see Appendix \ref{app-GWMs}.

\begin{definition}\normalfont A \textit{generalised weighing matrix}\index{matrix!generalised weighing} of order $n$ and weight $w$ over the $m$-th roots is a matrix $W$ with entries either $0$ or in $\mu_m$, such that
\[WW^*=wI_n.\]
Such a matrix is denoted $\GW(n,w;m)$. A matrix $\GW(n,w;2)$ is simply called a \textit{weighing matrix}, and denoted as $\W(n,w)$.
\end{definition}

\begin{theorem}\label{thm-PaleyGW}Let $m>1$ be an integer, and $q$ a prime power with $q\equiv 1\pmod{p}$. Then there is a $\GW(q+1,q;m)$, i.e. there is a matrix $W$ of order $q+1$ and entries in $\{0,1,\zeta_m,\dots,\zeta_m^{m-1}\}$ such that 
\[WW^*=qI_{q+1}.\]
\end{theorem}
\begin{proof}
The hypotheses of proposition \ref{prop-GenPaleyCore} are satisfied, so there is a matrix $Q$ of order $q$ with entries in $\{0,1,\zeta_m,\dots,\zeta_m^{m-1}\}$ such that $QQ^*=(q+1)I_q-J_q$, and $QJ_q=0$. Let $W$ be the following block matrix
\[W=
\left[
\begin{array}{c|c}
0 & \mathbf{1}_q^{\intercal}\\
\hline
\mathbf{1}_q & Q
\end{array}
\right].
\]
Then direct computation shows that
\[WW^*=
\left[
\begin{array}{c|c}
q & 0\\
\hline
0 & QQ^*+J_q
\end{array}
\right]=qI_{q+1}.\qedhere
\]
\end{proof}

Because of the presence of zero elements in the diagonal of the matrix $W$ constructed in Theorem \ref{thm-PaleyGW}, we cannot immediately extract a lower bound for $\gamma_m(q+1)$ from them. To obtain a lower bound, one can consider a perturbation of $W$ by a constant diagonal. However, to bound the value of the determinant of such a perturbation, or to compute it explicitly can be a very challenging task. The following lemma will be useful

\begin{lemma} \normalfont \label{lemma-RNFGPaleyW}
Let $m>1$ be an integer, and $q=p$ a prime number. Let $F_p$ be the Fourier matrix of order $p$. If 
\[W=
\left[
\begin{array}{c|c}
0 & \mathbf{1}_p^{\intercal}\\
\hline
\mathbf{1}_p & Q
\end{array}
\right],
\]
 is the $\GW(p+1,p;m)$ matrix of Theorem \ref{thm-PaleyGW}, we have that
\[F^{-1}WF=
\left[
\begin{array}{cc|c}
0 & p & \mathbf{0}_{p-1}^{\intercal}\\
1 & 0 & \mathbf{0}_{p-1}^{\intercal}\\
\hline
\mathbf{0}_{p-1}&\mathbf{0}_{p-1} & \Delta
\end{array}
\right],
\]
where $\Delta$ is the diagonal matrix consisting of the non-zero eigenvalues of $Q$ and,
\[F=
\left[
\begin{array}{c|c}
1 & \mathbf{0}^{\intercal}\\
\hline
\mathbf{0} & F_p
\end{array}
\right].
\]
\end{lemma}
\begin{proof}
We have that $W=A+B$, where
\[A=
\left[
\begin{array}{c|c}
0 & \mathbf{1}^{\intercal}\\
\hline
\mathbf{1} & \mathbf{0}
\end{array}
\right],\text{ and }
B=
\left[
\begin{array}{c|c}
0 & \mathbf{0}^{\intercal}\\
\hline
\mathbf{0} & Q
\end{array}
\right].
\]
 Recall that $F_p$ is given by $[F_q]_{ij}=\zeta_p^{ij}$, where the indices are interpreted modulo $p$. This implies that the first row and column of $F_p$ consist of the all-ones vector. Therefore, $F_p\mathbf{1}_p=(p,0,\dots,0)^{\intercal}$, and since $F_p^{-1}=\frac{1}{p}F_p$, we find
\[F^{-1}AF= 
\left[
\begin{array}{c|c}
0 & \mathbf{1}^{\intercal}F_p\\
\hline
F_p^{-1}\mathbf{1} & \mathbf{0}
\end{array}
\right]
=
\left[
\begin{array}{cc|c}
0 & p & \mathbf{0}_{p-1}^{\intercal}\\
1 & 0 & \mathbf{0}_{p-1}^{\intercal}\\
\hline
\mathbf{0}_{p-1} & \mathbf{0}_{p-1} & \mathbf{0}_{p-1,p-1}
\end{array}
\right].
\]
 Since $p$ is prime, the group $(\F_p,+)$ is isomorphic to the cyclic group $\Z/p\Z$, and the matrix $Q$ is circulant. So, by Lemma \ref{lemma-FourierIntertwining}, we have that
\[F_p^{-1}QF_p=\left[
\begin{array}{c|c}
0 & \mathbf{0}^{\intercal}\\
\hline
\mathbf{0} & \Delta
\end{array}
\right],\]
where $\Delta$ is a diagonal matrix consisting of the non-zero eigenvalues of $Q$. Therefore,
\[F^{-1}BF=
\left[
\begin{array}{c|c}
0 & \mathbf{0}^{\intercal}\\
\hline
\mathbf{0} & F_p^{-1}QF_p
\end{array}
\right]
=
\left[
\begin{array}{cc|c}
0 & 0 & \mathbf{0}_{p-1}^{\intercal}\\
0 & 0 & \mathbf{0}_{p-1}^{\intercal}\\
\hline
\mathbf{0}_{p-1}&\mathbf{0}_{p-1} & \Delta
\end{array}
\right].
\]
It follows that
\[F^{-1}WF=F^{-1}AF+F^{-1}BF=
\left[
\begin{array}{cc|c}
0 & p & \mathbf{0}_{p-1}^{\intercal}\\
1 & 0 & \mathbf{0}_{p-1}^{\intercal}\\
\hline
\mathbf{0}_{p-1}&\mathbf{0}_{p-1} & \Delta
\end{array}
\right].\qedhere
\]
\end{proof}

\begin{corollary}\normalfont\label{cor-BorderedPaleyDet} The determinant of the generalised Paley matrix $W+\alpha I_p$ is equal to 
\[\det(W+\alpha I_{p+1})=\frac{\alpha^2-p}{\alpha}\cdot \det(Q+\alpha I_p).\]
\end{corollary}
\begin{proof}
From Lemma \ref{lemma-RNFGPaleyW}, we have that 
\[F^{-1}(W+\alpha I_{p+1})F=
\left[
\begin{array}{cc|c}
\alpha & p & \mathbf{0}_{p-1}^{\intercal}\\
1 & \alpha & \mathbf{0}_{p-1}^{\intercal}\\
\hline
\mathbf{0}_{p-1}&\mathbf{0}_{p-1} & \Delta +\alpha I_{p-1}
\end{array}
\right].
\]
Since $\det(Q+\alpha I_{p})=\alpha\cdot \det(\Delta + \alpha I_{p-1})$, we find by the multiplicativity of the determinant that
\[\det(W+\alpha I_{p+1})=\det\begin{bmatrix}
\alpha & p\\
1 & \alpha
\end{bmatrix}\det(\Delta +\alpha I_{p-1})=\frac{\alpha^2-p}{\alpha}\det(Q+\alpha I_{p}).\qedhere\]
\end{proof}

Therefore, to calculate the determinant of $W+\alpha I_{p+1}$ it is sufficient to calculate the determinant of $Q+\alpha I_p$. This latter task can be tackled in some cases using the theory of cyclotomy. We will do this analysis in the case $m=3$ in the next section.

\section{Maximal determinants over the third roots}\label{sec-Maxdet3}
In the previous section, we showed that Barba matrices over the third roots may only exist at orders $n\equiv 1\pmod{3}$. Using the techniques developed in Chapter \ref{chap-HermitianForms}, we can find further restrictions.\\

Throughout this section, we let $\omega$ be a primitive third root of unity, so that $\omega^2+\omega+1=0$. 

\begin{lemma}[cf. Greaves and Yatsina \cite{Greaves-Yatsyna-Eq17Seidel}]\label{lemma-Root3Div}\normalfont Let $M$ be a matrix of order $n$  with entries in $\{1,\omega,\omega^2\}$. Then  $|\det(M)|^2\in\Z$, and $3^{n-1}$ divides $|\det(M)|^2$.
\end{lemma}
\begin{proof}
There exists a diagonal matrix $D$ with non-zero entries in $\{1,\omega,\omega^2\}$ such that $MD$ has diagonal equal to the all-ones vector. Therefore, $MD$ has the shape
\[
\begin{bmatrix}
1 & \omega^{a_{12}} & \omega^{a_{13}}  & \dots & \omega^{a_{1n}}\\
\omega^{a_{21}} & 1 & \omega^{a_{13}} & \dots & \omega^{a_{1n}}\\
\omega^{a_{31}} &  \omega^{a_{32}} & 1 & \dots & \omega^{a_{1n}}\\
\vdots & \vdots & \vdots &\ddots & \vdots\\
\omega^{a_{n1}} & \omega^{a_{n2}} & \omega^{a_{n3}} & \dots & 1
\end{bmatrix}
\]
By a series of elementary row operations, we see that
\[\det(MD)=\det
\begin{bmatrix}
1 & \omega^{a_{12}} & \dots & \omega^{a_{1n}}\\
0 & 1-\omega^{a_{12}+a_{21}} &\dots & \omega^{a_{2n}}-\omega^{a_{1n}+a_{21}}\\
\vdots & \vdots & \ddots &\vdots\\
0 &\omega^{a_{n2}}-\omega^{a_{12}+a_{n1}} &\dots &1-\omega^{a_{1n}+a_{n1}}
\end{bmatrix}.
\]
The element $(1-\omega)$ divides the last $n-1$ rows of the matrix in the right-hand side. To see this, notice that $(\omega^i-\omega^j)=(1-\omega^{j-i})\omega^i$ and $i$ is a unit. It is sufficient to show that $(1-\omega)$ divides $(1-\omega^n)$ for all $n$, but this is a consequence of the fact that the polynomial $X-1$ always divides $X^n-1$. Therefore, $\det(MD)=(1-\omega)^{n-1}\alpha$, where $\alpha\in\Z[\omega]$. It follows that,
\[|\det(M)|^2=[(1-\omega)\cdot(1-\omega^2)]^{n-1}|\alpha|^2=3^{n-1}|\alpha|^2, \]
and since $\alpha\in\Z[\omega]$, then $|\alpha|^2=\alpha\cdot \overline{\alpha}\in \Z[\omega]\cap\Q=\Z$. Thus, $3^{n-1}\mid |\det(M)|^2$ over the integers.\qedhere
\end{proof}

\begin{theorem}\normalfont \label{thm-Barba3NonEx} Let $n\equiv 1\pmod{3}$ be an integer. Write $2n-1=a^2\cdot 3^{t}\cdot r$ and $n-1=b^2\cdot 3^{\ell}\cdot s$. Suppose there is a prime number $p\equiv 2\pmod{3}$ such that one of the following holds:
\begin{itemize}
\item $n$ is odd and $p\mid r$, or
\item $n$ is even and $p\mid r$ or $p\mid s$,
\end{itemize}
then there is no Barba matrix of order $n$ over the third roots.
\end{theorem}
\begin{proof}
If a Barba matrix of order $n$ over the third roots exists, then 
\[\det(M)\overline{\det(M)}=(2n-1)(n-1)^{n-1},\]
so the number $\alpha_n:=(2n-1)(n-1)^{n-1}$ must be a norm in the quadratic extension $\Q\subset\Q[\omega]$. If $n$ is odd, then $(n-1)^{n-1}$ is a square, so we may write
\[\alpha_n=c^2\cdot 3^{t}\cdot r.\]
By Proposition \ref{prop-DetNormCondition}, we have that if $p\equiv 2\pmod{3}$ is a prime dividing $r$, then $\alpha_n$ cannot be a norm, and we arrive at a contradiction.\\

Note that $2n-1$ and $n-1$ have a disjoint set of prime factors. If $p$ is a prime such that $p\mid 2n-1$ and $p\mid n-1$, then $p\mid (2n-1)-(n-1)=n$.  This implies that $p\mid  n-(n-1)=1$ which is a contradiction. So in particular $(r,s)=1$, and $r\cdot s$ is a square-free number. Now, suppose that $n$ is even, then $\alpha_n$ is written as
\[\alpha_n=(ab^{n-1})^2\cdot 3^{t+(n-1)\ell}\cdot r\cdot s.\]
Again, by Proposition \ref{prop-DetNormCondition} if there is a prime factor $p\equiv 2\pmod{3}$ of $r$ or of $s$, then $\alpha_n$ is not a norm and we have a contradiction.\qedhere
\end{proof}

The following is the list of unattainable orders $n\equiv 1\pmod{3}$ for Barba matrices over the third roots, where $n<150$:
\[16, 28, 34, 43, 46, 52, 58, 70, 73, 88, 94, 100, 103, 106, 118, 124, 127, 133,
136, 142, 148.\]

Given that the Barba bound can never be attained at orders $n\equiv 2 \pmod{3}$, the maximal determinant problem over the third roots splits naturally into congruence classes, in the same way as the $\pm 1$ maximal determinant problem does. At orders $n\equiv 0\pmod{3}$, we have the results of existence of Hadamard matrices for $\BH(n,3)$ matrices that we presented in Chapter \ref{chap-BHMats}. See in particular the Table \ref{tab-BH3}. We have lower bounds at infinitely many orders $n\equiv 1\pmod{3}$ given by Theorem \ref{thm-BushLowerBound}. Additionally, we will compute a lower bound for certain orders $n\equiv 2\pmod{3}$ using cyclotomy.

\subsection{Structured Barba matrices over the third roots}
The problem of finding Barba matrices over the third roots appears to be more difficult than for $\pm 1$ matrices. Our first attempt will be to consider an analogue of Theorem \ref{thm-RealBarba}. For this we require the following auxiliary lemma.

\begin{lemma} \normalfont \label{lemma-Dioph3} The equation
\[x^2-(3y+1)x+3y^2=0,\]
has a finite number of solutions for $x,y\in \N$. The list of such solutions is $(x,y)=(0,0)$, $(1,0)$, $(1,1)$, $(3,1)$, $(3,2)$ and $(4,2)$.
\end{lemma}
\begin{proof}
Rewrite the equation as $x(x-1)=3y(x-y)$, then for $x,y\geq 0$ we have that $x-y\geq 0$, so $x\geq y$. And for $y\geq 1$ we have that $x(x-1)\leq 3y(x-1)$ so that $x\leq 3y$. Thus all solutions with $y\geq 1$ satisfy $y\leq x\leq 3y$. The discriminant of $x^2-(3y+1)x+3y^2$ as a polynomial in $x$ is $\Delta=(3y+1)^2-12y^2=-3y^2+6y+1$, which is nonnegative only for $0\leq y\leq 2$. Thus the set of solutions can be easily checked to be the one claimed.
\end{proof}

The following classifies Barba matrices with entries in $\{1,\omega,\omega^2\}$ having two distinct entries.

\begin{theorem}\label{thm-Barba3TwoEntries}
Let $B=J_v+(\omega-1)N$, where $N$ is a $\{0,1\}$-matrix of order $v$. Then $B$ is a Barba matrix if and only if $N$ is the incidence matrix of a symmetric $2$-$(v,k,k-(v-1)/3)$ design, where $v\equiv 1\pmod{3}$. 
\end{theorem}
\begin{proof}
Suppose that $B$ is a Barba matrix of order $v$, then 
$BB^*=(v-1)I_v+J_v$. Using the fact that $B=J_v+(\omega-1)N$ we have
\[BB^*=vJ_v+(\omega^2-1)J_vN^{\intercal}+(\omega-1)NJ_v+3NN^{\intercal}.\]
Since $\omega^2=-1-\omega$ the above can be rewritten as
\[BB^*=vJ_v-2J_vN^{\intercal}-NJ_v+\omega(NJ_v-J_vN^{\intercal})+3NN^{\intercal}.\]
From the condition $BB^*=(v-1)I_v+J_v$ it follows that $NJ_v-J_v N^{\intercal}=0$, and hence $NJ_v=J_vN^{\intercal}=(NJ_v)^{\intercal}$. Since $NJ_v$ is symmetric, there exists a natural number $k$ such that
\[NJ_v=J_vN^{\intercal}=kJ.\]
Using this fact we obtain
\[(v-3k)J_v+3NN^{\intercal}=BB^*=(v-1)I_v+J_v.\]
From which follows that 
\[NN^{\intercal}=\frac{v-1}{3}I_v+(k-\frac{v-1}{3})J_v.\]
Since $N$ is a $\{0,1\}$-matrix it is necessary that $v\equiv 1 \pmod{3}$. It follows that $N$ is the incidence matrix of a symmetric $2$-$(v,k,k-(v-1)/3)$ design with $v\equiv 1 \pmod{3}$. Conversely, let $N$ be the incidence matrix of a symmetric $2$-$(v,k,k-(v-1)/3)$ design with $v\equiv 1 \pmod{3}$, then $NJ_v=J_vN^{\intercal}=kJ_v$ and
\[3NN^{\intercal}=(v-1)I+(3k-(v-1))J_v,\]
and a straightforward calculation shows that if $B=J_v+(\omega-1)N$, then
\[BB^*=(v-1)I_v+J_v.\qedhere\]
\end{proof}
\begin{corollary}\normalfont The only ternary Barba matrices with exactly two distinct entries occur at orders $4$ and $7$, and correspond to the symmetric designs given by $1$-subsets (or $3$-subsets) of a $4$-set, and the projective plane of order $2$.
\end{corollary}
\begin{proof} By Theorem \ref{thm-Barba3TwoEntries}, if $M=J_v+(\omega-1)N$ is a Barba matrix, then $N$ is the incidence matrix of a symmetric $2$-$(v,k,k-(v-1)/3)$-design. Therefore we have that
\[(v-1)(k-(v-1)/3)=k(k-1).\]
Letting $x:=(v-1)/3$ the above equation can be rewritten as 
\[k^2-(3x+1)k+3x^2=0.\]
By Lemma \ref{lemma-Dioph3} the list of solutions with $x,k\geq 1$ is $(k,x)=(1,1)$, $(3,1)$, $(3,2)$ and $(4,2)$. Which yield the parameters $(v,k,\lambda)=(4,1,0)$, $(4,3,2)$, $(7,3,1)$ and $(7,4,2)$. These parameters are realised by the $1$-subsets of a $4$-set, the projective plane of order $2$ and their complements.
\end{proof}

Similarly, we can classify ternary Barba matrices on strongly regular graphs, see Definition \ref{def-SRG}.
\begin{theorem}\label{thm-Barba3SRG}
Let $M=I_v+\omega A+\omega^2(J_v-I_v-A)$, where $A$ is a $01$-matrix satisfying $A\circ I_v=0$ and $AJ_v=J_vA^{\intercal}=kJ_v$ for some $k\in\N$. If $M$ is a Barba matrix, then $A$ is the adjacency matrix of a strongly regular graph of parameters $(v,k,\lambda,\mu)$ with $v\equiv 1\pmod{3}$ and $\lambda=\mu-1=k-(v-1)/3$.
\end{theorem}
\begin{proof}
Let $M=I_v+\omega A+\omega^2(J_v-I_v-A)$, then 
\begin{align*}
MM^*=&\ [I_v+\omega A+\omega^2(J_v-I_v-A)][I_v+\omega^2A^{\intercal}+\omega(J_v-I_v-A^{\intercal})]\\
=&\ 2I_v+(v-2(k+1))J_v+2AA^{\intercal}+A+A^{\intercal}\\
&+\omega(A-2A^{\intercal}-AA^{\intercal}-I_v+(k+1)J_v)\\
&+\omega^2(A^{\intercal}-2A-AA^{\intercal}-I_v+(k+1)J_v).
\end{align*}
$M$ is a Barba matrix, so $MM^*$ is a matrix with integer coefficients. Therefore the coefficients of $\omega$ and $\omega^2$ must coincide. This implies that $A=A^{\intercal}$, so $A$ is symmetric. Using this fact we find that 
\begin{align*}MM^*&=2I_v+(v-2(k+1))J_v+2AA^{\intercal}+2A-(-A-AA^{\intercal}-I_v+(k+1)J_v)\\
&=3I_v+(v-3(k+1))J_v+3AA^{\intercal}+3A.
\end{align*}
Now from $MM^*=(n-1)I_v+J_v$ we find that 
\begin{align*}
3A^2=3AA^{\intercal}&=(v-4)I_v-3A+(3(k+1)-(v-1))J_v\\
&=3kI_v+(3k-(v-1))A+(3(k+1)-(v-1))(J_v-I_v-A).
\end{align*}
Since $k$ is an integer and $A$ a $01$-matrix it follows that $v\equiv 1\pmod{3}$. Furthermore, the hypothesis $I_v\circ A=0$ together with $A=A^{\intercal}$ and the equation for $A^2$ imply that $A$ is a strongly regular graph with the sought parameters.
\end{proof}
\begin{corollary}\normalfont The only strongly regular graphs whose Bose-Mesner algebra contains a Barba matrix over the third roots of unity are the \textit{Petersen graph}, the complement of the Petersen graph, and the \textit{Paley graph} of order $13$.
\end{corollary}
\begin{proof}
Every strongly regular graph satisfies the relation $(v-k-1)\mu=k(k-\lambda-1)$. Substituting $\lambda=\mu-1=k-(v-1)/3$ in this relation, and letting $x:=(v-1)/3$ we find
\[(3x-k)(k-x+1)=k(x-1).\]
 Thus $x$ satisfies $3x^2-3(k+1)x+k^2=0$, and reducing the equation modulo $3$ we find that $k^2\equiv 0\pmod{3}$ and hence $k=3y$ for some integer $y$. Substituting, we obtain the equation
\[x^2-(3y+1)x+3y^2=0.\]
By Lemma \ref{lemma-Dioph3}, the list of natural number solutions is $(x,y)=(0,0),(1,0),(1,1),(3,1),(3,2)$ and $(4,2)$. These solutions yield the feasible parameters $(v,k,\lambda,\mu)=(10,3,0,1),(10,6,3,4)$ and $(13,6,2,3)$ which correspond to the Petersen graph, its complement, and the Paley graph of order $13$, respectively.
\end{proof}

\begin{research-problem}\normalfont Find an infinite family of Barba matrices over the third roots, or show that there are only finitely many such matrices.
\end{research-problem}

We conclude this subsection with a remark on lower bounds at orders $n\equiv 1\pmod{3}$. By Corollary \ref{cor-BushLowerBound}, we have that 
\[\gamma_3(3^{2t}+1)\geq (3^t+1)3^{t3^{2t}}.\]
All orders $3^{2t}+1$ are congruent to $1$ modulo $3$, and the limit of the ratio of $\gamma_3$ and the Barba bound is
\[\lim_{t}\frac{\gamma_3(3^{2t}+1)}{\sqrt{2(3^{2t}+1)-1}\cdot 3^{t3^{2t}}}\geq \lim_t \frac{3^t+1}{\sqrt{2\cdot 3^{2t}+1}}=\frac{1}{\sqrt{2}}.\]
So we achieve approximately $70\%$ of the Barba bound infinitely often at orders $\equiv 1\pmod{3}$.

\subsection{Determinant lower bounds from cyclotomy}
Here, we further analyse the generalised Paley core $Q$, defined in Proposition \ref{prop-GenPaleyCore}. In particular, we will compute $\det(Q+\alpha I)$ where $\alpha$ is a third root of unity. For this, we will make use of the theory of cyclotomy.\index{cyclotomy} Classical reference texts for this area are the book by Storer \cite{Storer-Cyclotomy}, or Section 11.6 of Hall's \textit{Combinatorial Theory} \cite{Hall-CombinatorialTheory}.\\

\begin{definition}\normalfont Let $q$ be a prime power, and let $e$ and $f$ be integers such that $q=ef+1$. Let $\gamma$ be a primitive element of $\F_q$. Then, the \textit{$e$-th cyclotomic classes} are the cosets of the subgroup $H_0$ of $e$-th powers in $\F_q^{\times}$, which has index $e$. In other words, the cyclotomic classes are 
\[H_i=\{\gamma^{ea+i}: a\in \{0,1\dots,f-1\}\},\]
for $0\leq i\leq e-1$.
\end{definition} 

From here on, we assume that $q$ is a fixed prime power, and $q=ef+1$ for integers $e$ and $f$.

\begin{definition}\normalfont Over the field $\F_q$, for $q$ a prime power, the \textit{$e$-th cyclotomic number} $(i,j)$ is defined as the number of elements $x_i\in H_i$ such that
\[x_i+1\in H_j.\]
Equivalently, the cyclotomic number $(i,j)$ is the number of solutions $(x,y)$ to the equation 
\[\gamma^{ex+i}+1=\gamma^{ey+j}.\]
\end{definition}

The cyclotomic numbers exhibit the following elementary symmetries.
\begin{lemma}[cf. Part 1, Lemma 3 \cite{Storer-Cyclotomy}]\label{lemma-CyclotomicNumsSymmetry} \normalfont Let $(i,j)$ be the $e$-th cyclotomic number. Then
\begin{enumerate}
\item $(i+ae,j+be)=(i,j)$ for any integers $a,b$,
\item $(i,j)=(e-i,j-i)$,
\item $(i,j)=(j,i)$ if $f$ is even, and $(i,j)=(j+e/2,i+e/2)$ if $f$ is odd,
\item
\[\sum_{j=0}^{e-1}(i,j)=
\begin{cases}
f-1 & \text{ if } f\text{ is even, and } i=0\\
f-1 & \text{ if } f\text{ is odd, and } i=e/2\\
f & \text{ otherwise }
\end{cases}
\]
\item
\[\sum_{i=0}^{e-1}(i,j)=
\begin{cases}
f-1 & \text{ if } j=0\\
f & \text{ otherwise }
\end{cases}
\]
\end{enumerate}
\end{lemma}
\begin{example}[Quadratic cyclotomy]\normalfont \label{ref-QuadraticCyclotomy}
Suppose that $e=2$, and that $q=2f+1$ is an odd prime. By part (2) of Lemma \ref{lemma-CyclotomicNumsSymmetry}, we have that $(1,0)=(1,1)$. If $f$ is even, i.e. if $q\equiv 1\pmod{4}$, then part (3) of Lemma \ref{lemma-CyclotomicNumsSymmetry} implies that $(0,1)=(1,0)$. Therefore, letting $A=(0,0)$, and $B=(0,1)=(1,0)=(1,1)$, we have the following table of cyclotomic numbers,

\[
\begin{blockarray}{ccc}
&0 & 1\\
\begin{block}{c[cc]}
0 &\phantom{-}A & B \phantom{-}\\
1 &\phantom{-}B & B \phantom{-}\\
\end{block}
\end{blockarray},
\]
where by part (4) of Lemma \ref{lemma-CyclotomicNumsSymmetry} we must have
\begin{align*}
&A+B=f-1,\\
&2B=f.
\end{align*}
Therefore, $B=f/2=(q-1)/4$, and $A=f-1-B=(q-1)/4 - 1$. If $f$ is odd, so that $q\equiv 3\pmod{4}$, we find analogously the following table of cyclotomic numbers
\[
\begin{blockarray}{ccc}
&0 & 1\\
\begin{block}{c[cc]}
0 &\phantom{-}A & B \phantom{-}\\
1 &\phantom{-}A & A \phantom{-}\\
\end{block}
\end{blockarray},
\]
where $A+B=f$ and $2A=f-1$, so that $A=(q-3)/4$ and $B=(q-3)/4 +1$.
\end{example}

Cubic cyclotomic numbers can be determined with similar techniques, although in this case we need some number-theoretical results and counting arguments.

\begin{theorem}[Chapter 8, Theorem 2 \cite{Ireland-Rosen}] Suppose that $q\equiv 1\pmod{3}$. Then, there are integers $c$ and $d$ such that $4q=c^2+27d^2$. If we require $c\equiv 1\pmod{3}$, then $c$ is uniquely determined.
\end{theorem}

\begin{theorem}[cf. Part 1, Lemma 7 \cite{Storer-Cyclotomy}]\label{thm-CubicCyclotomy}
The cyclotomic numbers for $e=3$ are given by the table
\[\begin{blockarray}{cccc}
&0 & 1 & 2\\
\begin{block}{c[ccc]}
0 &\phantom{-}A & B & C \phantom{-}\\
1 &\phantom{-}B & C & D \phantom{-}\\
2 &\phantom{-}C & D & B \phantom{-}\\
\end{block}
\end{blockarray},\]
where 
\begin{align*}
9A &= q-8+c\\
18B&=2q-4-c-9d\\
18C&=2q-4-c+9d\\
9D &= q+1+c,
\end{align*}
and $4q=c^2+27d^2$ with $c\equiv 1\pmod{3}$.
\end{theorem}
So the cubic cyclotomic numbers $A=(0,0)$ and $D=(1,2)=(2,1)$ are completely determined. The numbers $B$ and $C$ are determined only up to the sign of $d$. In our case, we do not need to resolve this indeterminacy in order to carry our computations.
\begin{lemma}[cf. Part 1, Section 3 \cite{Storer-Cyclotomy}]\normalfont \label{lemma-TripleSum} Let $q=3f+1$ be a prime power, and $H_i$ be the cubic cyclotomic classes in $\F_q$. Then, the number of solutions $N$ to the equation
\[1+x_0+x_1+x_2=0,\]
where $x_i\in H_i$, $i=0,1,2$, satisfies
\[N=AD+B^2+C^2=BC+BD+CD=\frac{1}{3^3}(q^2-3q-c),\]
where $4q=c^2+27d^2$, and $c\equiv 1\pmod{3}$.
\end{lemma}

From the $e$-th cyclotomic classes we can define an $e$-class association scheme, known as the \textit{cyclotomic scheme}. First some notation: Let $G=(\F_q,+)$ be the additive group of $\F_q$. Let $\C[G]$ be the complex group algebra of $G$, i.e. the algebra generated by elements of the type $\sum_{x\in G}\alpha_x[x]$, where each $\alpha_x\in \C$. For each $x,y\in G$, the product of $[x]$ and $[y]$ in $G$ is $[x]\cdot [y]=[x+y]$. An element $\alpha\in \F_q^{\times}$ induces an automorphism of the group $G=(\F_q,+)$ by $x\mapsto \alpha x$. In turn, $\alpha\in \F_q^{\times}$ induces an automorphism of $\C[G]$ given by $[x]^{\alpha}:=[\alpha x]$. Denote 
\[K_i:=\sum_{x\in H_i}[x]=\sum_{a=0}^{f-1}[\gamma^{ae+i}]\in\C[G], \]
for $0\leq i\leq e-1$, where $\gamma$ is a generator of the cyclic group $\F_q^{\times}$. Then, we have

\begin{proposition}[cf. \cite{Munemasa-Cyclotomic}]\normalfont \label{prop-CycloConstants} Let $H_r$ be the unique $e$-th cyclotomic class containing the element $-1$. The product in $\C[G]$ of $K_i$ and $K_j$ is 
\[K_iK_j=f\delta_{i',j}[0]+\sum_{k=0}^{e-1}(j-i,k-i)K_k,\]
where $i'=i+r$.
\end{proposition}
\begin{proof}
First we compute $K_0\cdot K_i$, for $0\leq i\leq e-1$. We have
\[
K_0\cdot K_i=\left(\sum_{x\in H_0}[x]\right)\left(\sum_{y\in H_i}[y]\right)=\sum_{x\in H_0}\sum_{y\in H_i}[x+y].
\]
Suppose $x=\gamma^{ce}\in H_0$, then for a fixed $0\leq j\leq e-1$, the number of solutions $(a,b)$ to the equation $\gamma^{ae+i}+\gamma^{ce}=\gamma^{be+j}$ is the cyclotomic number $(i,j)$. Since dividing by $\gamma^{ce}$, this equation is equivalent to 
\[\gamma^{(a-c)e+i}+1=\gamma^{(b-c)e+j}.\]
Therefore, for a fixed $z\in H_j$, the number of solutions $(x,y)\in H_{0}\times H_{i}$ to $x+y=z$ is the cyclotomic number $(i,j)$. Since $1\in H_0$, then the number of solutions to the equation $x+y=0$ with $x\in H_0$ and $y\in H_i$ is either $0$ if $-1\not \in H_i$ and $f$ if $-1\in H_i$. By definition of $r$, the number of solutions is $f\delta_{r,i}$.  From here it follows that 
\[K_0\cdot K_i=\sum_{x\in H_0}\sum_{y\in H_i}[x+y]=f\delta_{r,i}[0]+\sum_{j=0}^{e-1}\sum_{z\in H_j}(i,j)z=f\delta_{r,i}[0]+\sum_{j=0}^{e-1}(i,j)K_j.\]
Applying the automorphism $\gamma^{j}\in \F_q^{\times}$ to $K_i$, we find
\[(K_i)^{\gamma^j}=\sum_{x\in \gamma^iH_0}[x]^{\gamma^{j}}=\sum_{x\in \gamma^i H_0}[\gamma^{j}x]=K_{i+j}.\]
Therefore, we have that
\begin{align*}
K_{i}K_j=(K_0K_{j-i})^{\gamma^i}&=f\delta_{r,j-i}[0]+\sum_{k}(j-i,k)K_k^{\gamma^i}\\
&=f\delta_{r+i,j}[0]+\sum_{k}(j-i,k)K_{k+i}\\
&=f\delta_{i',j}[0]+\sum_{k}(j-i,k-i)K_k.\qedhere
\end{align*}
\end{proof}
An immediate consequence of this result is:
\begin{corollary}[cf. \cite{Munemasa-Cyclotomic}] \normalfont The vector space $\Span\{[0],K_0,K_1,\dots,K_{e-1}\}$ is a $\C$-subalgebra of $\C[G]$ whose structure constants are given by  cyclotomic numbers.
\end{corollary}

\begin{theorem}[cf. \cite{Munemasa-Cyclotomic}] Let $q=ef+1$ be a prime power, and let $H_i$ for $0\leq i\leq e-1$ be the $e$-th cyclotomic classes. Define the matrices $A_i$ for $i=0,1,\dots,e$ by $A_0=I_q$, and 
\[
[A_{i+1}]_{xy}=
\begin{cases}
1 & \text{ if } x-y\in H_{i}\\
0 & \text{ otherwise }
\end{cases}.
\]
Then, $\Span_{\C}\{A_0,A_1,\dots,A_e\}$ is the Bose-Mesner algebra of an $e$-class association scheme.
\end{theorem}
\begin{proof}
Let $G=(\F_q,+)$ be the additive group of $\F_q$, and let $\rho:G\rightarrow\GL_q(\C)$ be the regular representation of $G$. Extend $\rho$ to $\C[G]$ by letting $\rho(\sum_{x}\alpha_x[x])=\sum_x \alpha_x\rho(x)$. We show that $\rho(K_i)=A_{i+1}$, which is equivalent to showing that $\sum_{z\in H_i}\rho(z)=A_{i+1}$. For $x,y\in\F_q$, we have
\[\sum_{z\in H_i}[\rho(z)]_{xy}=\sum_{z\in H_{i}}\delta_{x,z+y}.\]
This sum takes the value $1$ if there is a $z\in H_i$ such that $x-y=z$, and $0$ otherwise, in particular this shows that $\rho(K_i)=\sum_{z\in H_i}[z]=A_{i+1}$. Trivially, $A_0=I_q=\rho([0])$. Using this fact we can show that the matrices $A_i$ satisfy the axioms of adjacency matrices of an association scheme.
\begin{itemize}
\item[(i)] $\sum_{i=0}^{e}A_i=\rho([0])+\sum_{i=0}^{e}\rho(K_i)=\rho\left([0]+\sum_{i=0}^{e-1} K_i\right)=\rho\left(\sum_{x\in G}[x]\right)=J_q.$
\item[(ii)] $A_i^{\intercal}=\rho(K_{i-1})^{\intercal}=\rho(-K_{i-1})=A_{i'}$, for some $i'\in \{0,1,\dots,e-1\}$.
\item[(iii)]
\[A_{i+1}A_{j+1}=\rho(K_{i}K_{j})=f\delta_{i',j}I_q+\sum_{k=0}^{e-1}(j-i,k-i)A_k.\]
\end{itemize}
Finally, since $G$ is an abelian group, then $K_iK_j=K_jK_i$ for all $i,j$, which implies that the matrices $A_i$ commute.\qedhere
\end{proof}
The association scheme above is known as the $e$-th \textit{cyclotomic scheme}. Since all matrices $A_i$ are normal and commuting, they are simultaneously diagonalisable. Therefore, there is another basis for the Bose-Mesner algebra $\{E_0,E_1,\dots,E_e\}$, consisting of orthogonal idempotents. Additionally, there is a matrix $P$, known as the \textit{first eigenmatrix} of the scheme, such that
\[A_i=\sum_{j=0}^{e}P_{ji}E_j,\]
where $P_{ji}$ is the eigenvalue of $A_i$ in the eigenspace spanned by the columns of $E_j$. Dually, there is a matrix $Q$, known as the \textit{second eigenmatrix} of the scheme, such that
\[E_i=\frac{1}{q}\sum_{j=0}^{e}q_{ij}A_j.\]
\begin{definition}\normalfont Let $p=ef+1$ be a prime. Then, the $e$-th \textit{Gaussian periods} are defined for $0\leq i\leq e-1$ as 
\[\eta_i=\sum_{x\in H_i} \zeta_p^{x}.\]
\end{definition}

\begin{proposition}[cf. \cite{Munemasa-Cyclotomic}]\normalfont \label{prop-PMatrixCyclotomic} Let $p=ef+1$ be a prime. The first eigenmatrix $P$ of the $e$-th cyclotomic scheme is 
\[
P=
\begin{bmatrix}
1 & f & f &\dots &f\\
1 & \eta_0 & \eta_1& \dots & \eta_{e-1}\\
1 & \eta_1 & \eta_2 & \dots & \eta_{0}\\
\vdots & \vdots & \vdots & \ddots & \vdots\\
1 & \eta_{e-1} & \eta_{0} & \dots & \eta_{e-2}
\end{bmatrix},
\]
where $\eta_i$ are the $e$-th Gaussian periods.
\end{proposition}
\begin{proof}
Since $p$ is prime, the adjacency matrices $A_i$ of the cyclotomic scheme are circulant. Therefore, by Lemma \ref{lemma-FourierIntertwining}, they are simultaneously diagonalised by the Fourier matrix. Let $\rho:G\rightarrow \GL_p(\C)$ be the regular representation of the group $G=(\F_p,+)$. Then, we know that the eigenvalues of the circulant matrix $\pi_n^x=\rho(x)$ for $0\leq x\leq p-1$, are given by the values of the linear characters of $G$ at $x$.  Since $A_0$ is the identity matrix, the first column of $P$ is the all-ones vector. For the trivial character $\varepsilon$, we have that $\varepsilon(0)=1$, and $\varepsilon(K_i)=\sum_{x\in H_i}\varepsilon(x)=|H_i|=f$. This implies that the first row of $P$ is as claimed. Let $\gamma$ be a generator of the group $\F_p^{\times}$, then every non-trivial linear character of $G$ is of the type $\chi_j$, where $\chi_j(1)=\zeta_p^{\gamma^j}$, for $0\leq j\leq p-1$. We have that
\[\chi_{j}(K_i)=\sum_{x\in H_i}\chi_j(x)=\sum_{x\in H_i}\chi_{1}(\gamma^jx)=\sum_{x\in H_{i+j}}\chi_1(x)=\sum_{x\in H_{i+j}}\zeta_p^{x}=\eta_{i+j}.\qedhere\]
\end{proof}

Using basic identities in the theory of association schemes, it is easy to show that for the cyclotomic scheme $Q=\overline{P}$. This property is known as \textit{formal self-duality}.

\begin{remark}\normalfont
In the case $e=3$, we have that $-1\in H_0$, so all matrices $A_i$ of the cubic cyclotomic scheme are symmetric. In particular, all Gaussian periods are real, and $Q=P$.
\end{remark}

\begin{proposition}\normalfont\label{prop-GramComputationPaley} Let $q=3f+1$ be a prime power,  $Q$ be the generalised Paley core of order $q$ over the third roots, and $\rho:(\F_q,+)\rightarrow \GL_p(\C)$ be the regular representation of the additive group $(\F_q,+)$. Then the Gram matrix of $Q+\alpha I$, for $\alpha\in\C$ is 
\[\rho[(\alpha^2+q-1)[0]+(2\re(\alpha) -1)K_{0}+(2\re(\alpha\omega^2)-1)K_{1}+
(2\re(\alpha\omega-1))K_{2}].\]
\end{proposition}
\begin{proof}
Let $G=(\F_q,+)$. The matrix $Q+\alpha I $ is given by 
\[Q+\alpha I = \rho(\alpha[0]+ K_0+\omega K_1+\omega^2 K_2).\] 
Since $-1\in H_0$, we have that $-H_i=H_i$, and the matrix $(Q+\alpha I)^*$ is given by the element $\overline{\alpha}[0]+K_0+\omega^2 K_1 + \omega K_2\in\C[G]$. Computing the product in $\C[G]$ we find:
 
\begin{align*}
&(\alpha [0]+ K_0+\omega K_1+\omega^2 K_2)(\overline{\alpha} [0] + K_0+ \omega^2 K_1+ \omega K_2)\\
= &|\alpha|^2[0]+\alpha K_0+\alpha\omega^2 K_1+\alpha \omega K_2\\
&\overline{\alpha}K_0+K_0^2 + \omega^2 K_0K_1+\omega K_0K_2\\
&\overline{\alpha}\omega K_1+\omega K_1K_0+K_1^2+\omega^2 K_1K_2\\
&\overline{\alpha}\omega^2 K_2+\omega^2 K_2K_0+\omega K_2K_1+K_2^2.
\end{align*}
We evaluate this expression. First we find by Proposition \ref{prop-CycloConstants} and Lemma \ref{lemma-CyclotomicNumsSymmetry}, that
\[K_0^2+K_1^2+K_2^2 = 3f[0] +\sum_k \left(\sum_i(0,k-i)\right)K_k =3f[0]+(f-1)(K_0+K_1+K_2).\]
Since $3f=(q-1)$ we rewrite this as $K_0^2+K_1^2+K_2^2= (p-1)[0]+(f-1)(K_0+K_1+K_2)$.
Next we evaluate $\omega(K_0K_2+K_1K_0+K_2K_1)$ and $\omega^2(K_0K_1+K_1K_2+K_2K_0)$. It is easy to check that
\begin{align*}
K_0K_2+K_1K_0+K_2K_1&=f(K_0+K_1+K_2),\text{ and}\\
K_0K_1+K_1K_2+K_2K_0&=f(K_0+K_1+K_2).
\end{align*}
This implies that $\omega(K_0K_2+K_1K_0+K_2K_1)+\omega^2(K_0K_1+K_1K_2+K_2K_0)=f(\omega+\omega^2)(K_0+K_1+K_2)=-f(K_0+K_1+K_2).$ Therefore the Gram matrix is given by taking the regular representation of the element  
\[(|\alpha|^2+(q-1))[0]+(2\re(\alpha)-1)K_0+(2\re(\alpha\omega^2)-1)K_1+(2\re(\alpha\omega)-1)K_2,\]
 of $\C[G]$, as we wanted to show.\qedhere
\end{proof}
\begin{corollary}\label{cor-EigValsGramPaley} \normalfont Let $p=3f+1$ be a prime, and let $Q$ be the generalised Paley core of order $p$ over the third roots. Then, the eigenvalues of the Gram matrix of $Q+\alpha I$ where $\alpha\in\{1,\omega,\omega^2\}$ are 
\begin{itemize}
\item[(i)] $p-3f=1$, which occurs with multiplicity $1$, and
\item[(ii)] $(p+2)+3\eta_i$ with multiplicity $f$, for $i=0,1,2$.
\end{itemize}
\end{corollary}
\begin{proof}
Computing the Gram matrix $M$ of $Q+\alpha I$ with Proposition \ref{prop-GramComputationPaley}, and using the fact that $\alpha\in \{1,\omega,\omega^2\}$, we have that $M$ is given by the group algebra element
\begin{align*}
&(|\alpha^2+(p-1))[0]+(2\re(\alpha)-1)K_0+(2\re(\alpha\omega^2)-1)K_1+(2\re(\alpha\omega^2)-1)K_2\\
&=p[0]+K_{i}-2K_{i+1}-2K_{i+2},
\end{align*}
for some $i=0,1,2$. Analogously as in the proof of Proposition \ref{prop-PMatrixCyclotomic}, we have that since $p$ is a prime, the Gram matrix $M$ is circulant, and its eigenvalues are given by the evaluation of $p[0]+K_{i}-2K_{i+1}-2K_{i+2}$ at each linear character of the additive group $(\F_p,+)$. For the trivial character $\varepsilon$, we find
\[\varepsilon(p[0]+K_{i}-2K_{i+1}-2K_{i+2})=p-3|K_0|=p-3f=1.\]
All non-trivial linear characters are of the type $\chi_j(x)$, $1\leq j\leq p-1$, where $\chi_j(1)=\zeta_p^{\gamma^j}$, and $\gamma$ is a primitive element of $\F_p$. We have that 
\begin{align*}
\chi_j(p[0]+K_i-2K_{i+1}-2K_{i+2})&=
p+\chi_j(K_i)-2\chi_j(K_{i+1})
-2\chi_j(K_{i+1})\\
&=p+\eta_{i+j}-2\eta_{i+j+1}-2\eta_{i+j+2}\\
&=p+3\eta_k -2(\eta_0+\eta_1+\eta_2),
\end{align*}
for some $k\in\{0,1,2\}$. Since $\sum_{x\in\F_p}\zeta_p^x=0$, we have that $\eta_0+\eta_1+\eta_2=-1$, and the eigenvalues are
\[p+2+3\eta_i,\]
with multiplicity $f=(p-1)/3$ for each $i=0,1,2$.\qedhere
\end{proof}

\begin{theorem}\label{thm-PaleyDeterminant} Let $p=3f+1$ be a prime, $Q$ be the generalised Paley core over the third roots, and $\alpha\in\{1,\omega,\omega^2\}$. Then, the absolute value of the determinant of $Q+\alpha I_p$ is
\begin{align*}
|\det(Q+\alpha I)|&=\left(\prod_{i=0}^2((p+2)+3\eta_i)\right)^{f/2}\\
&=\left[(p+2)^3-3(p+2)^2-3(p-1)(p+2)+(3+c)p-1\right]^{(p-1)/6},
\end{align*}
where $4p=c^2+27d^2$, and $c\equiv 1\pmod{3}$.
\end{theorem}
\begin{proof}
By Corollary \ref{cor-EigValsGramPaley}, the determinant of $(Q+\alpha I)(Q+\alpha I)^*$ is 
\[[(p+2+3\eta_0)(p+2+3\eta_1)(p+2+3\eta_2)]^f.\]
Let $G=(\F_p,+)$. Instead of computing the product above, we can equivalently compute the product using the elements $K_i$ in the quotient ring $\C[G]/(S_G)$, where $S_G:=\sum_{x\in G}[x]$. The product $\prod_{i} ((p+2)[0]+3K_i)$ expands as
\begin{align*}
&(p+2)^3[0]^3 +3(p+2)^2[0]^2(K_0+K_1+K_2)+
3^2(p+2)[0](K_0K_1+K_0K_2+K_1K_2)+
3^3K_0K_1K_2\\
&=(p+2)^3[0]+3(p+2)^2(K_0+K_1+K_2)
+3^2(p+2)(K_0K_1+K_0K_2+K_1K_2)
+3^3K_0K_1K_2.
\end{align*}
We have that $K_0K_1+K_0K_2+K_1K_2=f(K_0+K_1+K_2)\equiv -f[0] \pmod{S_G}$, so
\[\prod_{i=0}^2((p+2)[0]+3K_i)\equiv ((p+2)^3 -3(p+2)^2-3^2(p+2)f)[0]+ 3^3K_0K_1K_2\pmod{S_G}.\]
It suffices to compute the term $K_0K_1K_2$. Using the notation of Theorem \ref{thm-CubicCyclotomy}, we have
\begin{align*}
K_0K_1K_2&= K_0(DK_0+BK_1+CK_2)\\
&=fD\cdot [0]+ADK_0+BDK_1+CDK_2\\
&\phantom{=fD\cdot [0]\ }
+B^2K_0+BC K_1+ BDK_2\\
&\phantom{=fD\cdot [0]\ }
+C^2K_0+CDK_1+CBK_2.
\end{align*}
By Lemma \ref{lemma-TripleSum}, we have that $AD+B^2+C^2=BC+BD+CD=N=(p^2-3p-c)/3^3$.
Therefore,
\begin{align*}
K_0K_1K_2&=fD[0]+N(K_0+K_1+K_2)\\
&\equiv (fD-N)[0]\pmod{S_G}.
\end{align*}
Substituting $N=(q^2-3q-c)/3^3$, and using the fact that $f=(p-1)/3$ and $3^2D=(p+1+c)$, we have that $3^2(p+2)f=3(p-1)(p+2)$, and $3^3(fD-N)=3p-1+pc$. Therefore,
\[\prod_{i=0}^2((p+2)[0]+3K_i)\equiv ((p+2)^3 -3(p+2)^2-3(p-1)(p+2)+(3+c)p-1)[0]\pmod{S_G}.\]
Evaluating this expression at a non-trivial character of $\F_p$, the result follows.\qedhere
\end{proof}
\begin{corollary}\normalfont For every prime number $p\equiv 1\pmod{3}$, there is a matrix $M$ of order $p+1\equiv 2\pmod{3}$ over the third roots of unity such that
\[|\det M|=\sqrt{p^2+p+1}\cdot \left[(p+2)^3-3(p+2)^2-3(p-1)(p+2)+(3+c)p-1\right]^{(p-1)/6},\]
where $4p=c^2+27d^2$, and $c\equiv 1\pmod{3}$.
\end{corollary}
\begin{proof}
By Corollary \ref{cor-BorderedPaleyDet}, the determinant of $W+\alpha I_{p+1}$ is
\[|\det(W+\alpha I_{p+1})|=\left|\frac{\alpha^2-p}{\alpha}\right||\det(Q+\alpha I_p)|.\]
If $\alpha=1$, then $|\alpha^2-p|/|\alpha|=p-1$. On the other hand, if $\alpha\in\{\omega,\omega^2\}$ then $|\alpha-p|/|\alpha|=\sqrt{p^2+p+1}>p-1$, so the largest value of the determinant is obtained with $\alpha=\omega$ or $\alpha=\omega^2$. By Theorem \ref{thm-PaleyDeterminant}, we have that
\[|\det(W+\omega I_{p+1})|=\sqrt{p^2+p+1}\cdot \left[(p+2)^3-3(p+2)^2-3(p-1)(p+2)+(3+c)p-1\right]^{(p-1)/6}.\qedhere\] 
\end{proof}
In particular, this gives an infinite family of matrices of order $n\equiv 2\pmod{3}$ that achieve a constant ratio of the Barba bound. We have

\[\lim_{p\rightarrow\infty} \frac{\sqrt{p^2+p+1}\cdot \left[(p+2)^3-3(p+2)^2-3(p-1)(p+2)+(3+c)p-1\right]^{(p-1)/6}}{\sqrt{2p+1}\cdot p^{p/2}}=\frac{1}{\sqrt{2}}.\]
So our construction achieves approximately $70\%$ of the Barba bound in the limit.

\subsection{Small maximal determinant  matrices over the third roots}
We conclude this section with a summary of results, and tables of maximal determinant matrices and putative maximal determinant matrices. Some of these matrices are presented here,  for additional examples see Appendix \ref{app-MatrixTables}. For matrices over $\{1,\omega,\omega^2\}$, we have by the results in Chapter \ref{chap-BHMats} that the Hadamard bound is achieved infinitely often at orders $n\equiv 0\pmod{3}$. We found examples of Barba matrices over the third roots at small orders $n\equiv 1\pmod{3}$, which show that the bound is sharp. Additionally, we have that a fraction of approximately $70\%$ of the Barba bound is achieved infinitely often at orders $n\equiv 1,2\pmod{3}$.\\

In the table below, $n$ indicates the order of the matrix. The columns labelled $|\det|^2/3^{n-1}$ include the values of the determinant of the Gram matrix divided by $3^{n-1}$, see Lemma \ref{lemma-Root3Div}. The symbol `\textcolor{wpicrimson}{??}' is used to indicate that we currently have no certificate of maximality of the given determinant. The columns R indicate the ratio of the record determinant to the Hadamard bound if $n\equiv 0\pmod{3}$, and to the Barba bound if $n\equiv 1,2\pmod{3}$. 
\begin{table}[H]
\centering
\begin{tabular}{|ccc||ccc||ccc|}
\hline
$n$ & $|\det|^2/3^{n-1}$ & R & $n$ & $|\det|^2/3^{n-1}$ & R & $n$ & $|\det|^2/3^{n-1}$ & R\\
\hline
& & & 1 & $1$ & $1$ & 2 & $1$ & $1$ \\
3 & $3$ & $1$ & 4 & $7$ & $1$ & 5 & $3\times 7$ & $0.86$\\
6 & $2^6\times 3$ & $1$ &
7 & $2^6\times 13$ & 1 & 8 & $2^{12}$ \textcolor{wpicrimson}{??}& \textcolor{wpicrimson}{$0.85$}\\
9& $3^{10}$ & 1 & 10 & $3^{9}\times 19$ & 1 & 11 & $3^9\times 7\times 19$\textcolor{wpicrimson}{??}& \textcolor{wpicrimson}{$0.86$} \\
12 & $2^{24}\times 3$ & 1 &
13 & $2^{24}\times 5^2$ & $1$ &
14 & $2^{24}\times 223$\textcolor{wpicrimson}{??} &\textcolor{wpicrimson}{$0.85$} \\
15 & $2^{22}\times 3^6\times 19$\textcolor{wpicrimson}{??} & \textcolor{wpicrimson}{$0.79$} &16 &$2^{24}\times 3^8\times 7$ \textcolor{wpicrimson}{??}& \textcolor{wpicrimson}{$0.90$} &17 &$ 13^5\times 67^4$\textcolor{wpicrimson}{??}& \textcolor{wpicrimson}{$0.72$}\\
18 &$2^{18}\times 3^{19}$ & $1$ &19 & $ 13 \times 37^2\times 342037^2$\textcolor{wpicrimson}{??} & \textcolor{wpicrimson}{$0.74$} &20 &$7^6\times 37^6\times 127$\textcolor{wpicrimson}{??} &\textcolor{wpicrimson}{$0.76$}\\
\hline
\end{tabular}
\caption{Maximal determinants and record determinants for matrices over the third roots.}\label{tab-Maxdet3}
\end{table}

We remark that at order $2$, the determinants of $3I_2-J_2$ and $I_2+J_2$ coincide so the Barba bound at order $2$ is met with equality by the matrix
\[\begin{bmatrix}
1 & 1\\
\omega & \omega^2
\end{bmatrix}.
\]
With Theorem \ref{thm-Barba3TwoEntries}, and Theorem \ref{thm-Barba3SRG} we find Barba matrices at orders $n=4,7,10,$ and $13$, see Appendix \ref{app-MatrixTables} for examples of these matrices. Since Theorem \ref{thm-Barba3NonEx} rules out the existence of a Barba matrix at the order $n=16$, the next open case is $n=19$.

\begin{research-problem}\normalfont Find a Barba matrix of order $19$ over the third roots of unity, or prove that no such matrix exists.
\end{research-problem}

 Some of the examples of large determinant matrices that we present have been obtained using genetic algorithms. For example, at order $8$ the determinant in Theorem \ref{thm-PaleyDeterminant} achieves $75\%$ of the Barba bound, at order $14$ this same construction achieves $68\%$, and we could find better examples with genetic algorithms. At the orders $11$, $15$, $16$, $17$, and $19$ a genetic search in the whole space of matrices tends to converge rapidly to local minima. Upon review of a first draft of this dissertation, Adam Zsolt Wagner proposed the following greedy approach to the present author, improving on the lower  bounds obtained with genetic algorithms:
\begin{enumerate}
\item Create a set of random matrices $\mathcal{M}$, and label all matrices in $\mathcal{M}$ as unexplored.
\item Select an unexplored matrix $M$ with largest determinant among all unexplored matrices in $\mathcal{M}$.
\item Create the list of matrices at Hamming distance $1$ from $M$ and include it to the set $\mathcal{M}$, label $M$ as explored.
\end{enumerate} 
This algorithm will eventually generate all possible matrices. In practice we have limited memory, so to address this issue we discard a portion of the matrices with lowest determinant once the memory is full. With this approach, Wagner reported matrices of large determinant at orders $n\in\{11,14,15,16\}$. We include the matrices of order $n=11,14$ and $16$ in Appendix \ref{app-MatrixTables}, the matrix at order $n=15$ is included below. We also carried searches for circulant matrices. The  first row of the best circulant matrices we could find at each order are tabulated below
 \begin{align*}
 &11: [0 0 2 0 2 0 0 1 2 2 1]\\
 &15: [0 1 2 2 2 2 1 2 0 2 2 1 0 2 0]\\
 &16: [2 2 2 1 2 0 2 2 2 0 0 1 0 2 2 0]\\
 &17: [0 0 1 1 0 0 0 2 0 1 2 2 2 1 0 2 0]\\
 &19: [0 1 0 0 0 1 0 1 1 2 2 0 1 1 0 2 2 1 1]
 \end{align*}

At order $n=20$ the matrix achieving the determinant indicated is the one of Theorem \ref{thm-PaleyDeterminant}. We conclude with some interesting observations: A maximal determinant matrix at order $5$ can be obtained from a generalised Paley matrix of order $q=4$ over the third roots. The following matrix, written logarithmically,
\[
M_5=\begin{bmatrix}
1 & 0 & 0 & 0 & 0\\
0 & 1 & 0 & 1 & 2\\
0 & 0 & 1 & 2 & 1\\
0 & 1 & 2 & 1 & 0\\
0 & 2 & 1 & 0 & 1
\end{bmatrix}
\]
achieves the maximal determinant of value $1701$. After a permutation of rows and columns of $M_5$ we obtain a matrix whose Gram matrix has the form
\[
\begin{bmatrix}
5 & 2 & - & - & -\\
2 & 5 & - & - & -\\
- & - & 5 & 2 & -\\
- & - & 2 & 5 & -\\
- & - & - & - & 5
\end{bmatrix}.
\]
So the structure of the Gram matrix is analogous to that of the Ehlich blocks in the real case, see Definition \ref{def-EhlichBlock}. Similarly, the candidate matrix
\[
M_8=
\begin{bmatrix}
0&1&1&1&1&2&2&0\\
1&0&0&1&2&2&2&0\\
1&1&0&0&1&0&0&2\\
0&0&1&0&2&0&0&2\\
1&2&1&2&0&2&1&2\\
1&2&1&0&0&1&2&1\\
2&1&2&2&2&1&2&2\\
2&1&2&0&2&2&1&1
\end{bmatrix},
\]
has the following Gram matrix
\[
\begin{bmatrix}
8 & 2 & - & - & - & - & - & -\\
2 & 8 & - & - & - & - & - & -\\
- & - & 8 & 2 & - & - & - & -\\
- & - & 2 & 8 & - & - & - & -\\
- & - & - & - & 8 & 2 & - & -\\
- & - & - & - & 2 & 8 & - & -\\
- & - & - & - & - & - & 8 & 2\\
- & - & - & - & - & - & 2 & 8
\end{bmatrix},
\]
which again has an Ehlich-block type structure. We observe the same pattern for the matrices of order $n=11$ and $n=14$ in Appendix \ref{app-MatrixTables}. This suggests the following:

\begin{research-problem}\normalfont Extend the analysis of Ehlich to find a sharpened upper bound for matrices with entries in $\{1,\omega,\omega^2\}$ at orders $n\equiv 2\pmod{3}$.
\end{research-problem}

The best circulant matrix we obtained at order $15$ is the following:
\[[0 1 2 2 2 2 1 2 0 2 2 1 0 2 0].\]
This circulant matrix is permutation-equivalent to one satisfying the equation $MM^*=((15-3)I_3+3J_3)\otimes I_5$, which gives a Gram matrix of a similar structure to that of EW matrices, see Definition \ref{def-EWMatrix}. With the greedy search described above one can obtain the following matrix:
\[
M_{15}=
\left[
\begin{array}{*{15}{c}}
0&0&1&0&0&0&1&0&2&1&0&2&0&2&0\\
2&0&0&2&2&1&2&2&0&1&1&2&0&2&1\\
2&0&1&1&1&0&0&1&1&1&2&2&2&0&1\\
0&2&2&2&0&0&0&1&0&0&0&1&0&0&1\\
1&0&2&1&2&2&0&0&0&1&0&2&2&0&2\\
0&2&0&2&1&1&2&0&1&1&0&0&2&0&0\\
0&1&0&0&2&1&1&1&0&1&2&1&2&0&0\\
0&1&2&1&1&2&0&2&1&1&1&1&0&2&0\\
2&0&1&2&1&2&1&2&0&2&0&1&1&0&0\\
0&0&0&1&1&0&2&0&0&2&2&1&0&1&2\\
0&0&0&1&2&2&1&1&1&0&0&0&1&2&1\\
2&1&0&1&0&1&0&0&2&1&0&1&1&1&1\\
1&1&1&2&1&2&1&0&0&0&2&0&0&1&1\\
2&1&0&0&1&2&2&1&2&0&0&2&0&0&2\\
1&1&0&0&1&0&0&2&0&2&0&0&2&2&1
\end{array}
\right],
\]
which yields the record reported in Table \ref{tab-Maxdet3}. Given that there is no $\BH(15,3)$ matrix, the following is a very interesting computational problem:

\begin{research-problem}\normalfont
Determine the maximal determinant of a $\{1,\omega,\omega^2\}$ matrix of order $15$. Is $M_{15}$ a maximal determinant matrix?
\end{research-problem}

For more on how to approach this problem see Section \ref{sec-MaximalityCertificates} in this chapter.

\section{Maximal determinants over the fourth roots}\label{sec-Maxdet4}
Let $i=\sqrt{-1}$. In this section we study maximal determinant matrices with entries in $\{\pm 1,\pm i\}$. 

\begin{lemma}\normalfont\label{lemma-Root4Div} Let $M$ be a matrix with entries in the set $\{\pm 1,\pm i\}$, then $|\det(M)|^2\in\Z$ is an integer and $2^{n-1}$ divides $|\det(M)|^2$.
\end{lemma}
\begin{proof}
The proof is analogous to that of Lemma \ref{lemma-Root3Div}. Now instead of having the factor $(1-\omega)^{n-1}$, we have the factor $(1-i)^{n-1}$. Since $(1-i)\overline{(1-i)}=2$, the result follows.\qedhere
\end{proof}
Since $\{\pm 1\}\subset\{\pm 1,\pm i\}$, any real maximal determinant matrix at orders $n\equiv 0,1\pmod{4}$ is also a maximal determinant matrix over the fourth roots. In particular, all real Hadamard matrices are $\BH(n,4)$ matrices. For $\BH(n,4)$ matrices at orders $n\equiv 2\pmod{4}$, we have the following construction

\begin{theorem}[cf. Paley \cite{Paley}]\label{thm-ComplexPaley} Let $q\equiv 1\pmod{4}$ be a prime power, and let $Q$ be the (quadratic) Paley core of order $q$. Then, the matrix $iQ-I$, bordered with a row and column of ones, is a $\BH(q,4)$.
\end{theorem}
\begin{proof} The matrix $Q$ has entries $\pm 1$ and satisfies $QQ^{\intercal}=qI_q-J_q$, and $QJ_q=0$. Additionally, since $q\equiv 1\pmod{4}$, then $x-y$ is a square in $\F_q$ if and only if $y-x$ is a square in $\F_q$, so $Q=Q^{\intercal}$. This implies that 
\[(iQ-I)(iQ-I)^*=(q+1)I_q-J_q.\]
Therefore, letting
\[H=
\left[
\begin{array}{c|c}
1 & \mathbf{1}_q^{\intercal}\\
\hline
\mathbf{1}_q & iQ -I
\end{array}
\right],
\]
it follows easily that $HH^*=(q+1)I_{q+1}$.\qedhere
\end{proof}

We showed in Theorem \ref{thm-BarbaBound4Roots}, that the Barba bound applies to matrices over the fourth roots at odd orders. The following gives further restrictions.
\begin{theorem}[Cohn, cf. Theorem 2 \cite{Cohn-ComplexDOptimal}]  If there is a Barba matrix of order $n$, with entries in the fourth roots of unity, then $2n-1$ is a sum of two integer squares.
\end{theorem}
\begin{proof}
By Theorem \ref{thm-BarbaConstantRowSum}, the existence of a Barba matrix over the fourth roots implies the existence of a normal Barba matrix $B$ with constant row-sum. Then, there are integers $a,b\in\Z$ such that  $BJ_n=(a+bi)J_n$. This implies,
\[|a+bi|^2J_n=BB^*J_n=((n-1)I_n+J_n)J_n=(2n-1)J_n.\]
Therefore, $2n-1=|a+bi|=a^2+b^2$.\qedhere
\end{proof}

 If $n\equiv 1\pmod{4}$, then the construction of Barba matrices in Theorem \ref{thm-BarbaConstruction} gives maximal determinant matrices over the fourth roots. The following result due to Cohn establishes a fundamental relationship between the maximal determinant problems over $\{\pm 1\}$ and over $\{\pm 1, \pm i\}$. We include the proof here for completeness.

\begin{theorem}[Cohn, Theorem 1 \cite{Cohn-ComplexDOptimal}] \label{thm-MorphismDeterminant}
One has $\gamma(2n)\geq 2^n\gamma_4(n)^2$, with equality if and only if there is a skew matrix $M$ satisfying $|\det M|=\gamma(2n)$. 
\end{theorem} 
\begin{proof}
Let $N$ be a matrix of order $n$, with entries over the fourth roots of unity, such that $|\det(N)|=\gamma_4(n)$. We apply the Turyn morphism, Theorem \ref{thm-TurynMorphism}, to $N$. Writing $N=A+iB$, where $A$ and $B$ are $\{0,\pm 1\}$-matrices, we let $R=A+B$, $S=-A+B$ and
\[M=
\begin{bmatrix}
R & S\\
-S & R
\end{bmatrix}
=
\begin{bmatrix}
A+B & -A+B\\
A-B & A+B
\end{bmatrix}.
\]
It is easy to check that 
\[
\begin{bmatrix}
I_n & 0\\
iI_n & I_n
\end{bmatrix}
\begin{bmatrix}
R & S\\
-S & R
\end{bmatrix}
\begin{bmatrix}
I_n & 0\\
-iI_n & I_n
\end{bmatrix}
=
\begin{bmatrix}
R-iS & S\\
0 & R+iS
\end{bmatrix}.
\]
Now, $R-iS=(1-i)(A+iB)=(1-i)N$ and $R+iS=(1+i)(A-iB)=(1+i)\overline{N}.$
Therefore,
\begin{align*}
\det(M)&=\det(R-iS)\det(R+iS)\\
&=(1-i)^n\det(A+iB)\cdot (1+i)^n\det(A-iB)\\
&=2^n|\det(N)|^2\\
&=2^n\gamma_4(n)^2.
\end{align*}
This implies that $\gamma(2n)\geq 2^n\gamma_4(n)^2$, with equality if and only if $M$ is maximal determinant. So equality occurs if and only if there is a skew maximal determinant $\pm 1$ matrix at order $2n$.\qedhere
\end{proof}
  In general, we have
  \begin{lemma}[cf. Cohn, \cite{Cohn-ComplexDOptimal, Cohn-NumberDOptimal}]\normalfont \label{lemma-CirculantEW2Barba} If there is an EW matrix of order $2n$ having the shape
\[M=
\begin{bmatrix}
A & B\\
-B^{\intercal} & A^{\intercal}
\end{bmatrix},
\]
where $A$ and $B$ are circulant. Then there is a $\{\pm 1,\pm i\}$ Barba matrix of order $n$.
\end{lemma}
\begin{proof}
From Lemma \ref{lemma-Circulant2Skew} we know that the existence of $M$ implies the existence of a skew EW matrix $W$, having the shape
\[W=
\begin{bmatrix}
R & S\\
-S & R
\end{bmatrix},
\]
where $RS^{\intercal}=SR^{\intercal}$ and $RR^{\intercal}+SS^{\intercal}=2(n-1)I_n+2J_n$. So the matrix $B=\frac{1}{2}(R-S)+\frac{i}{2}(R+S)$ satisfies $BB^*=(n-1)I_n+J_n$.\qedhere
\end{proof}

In particular, we have the following infinite family of Barba matrices.
\begin{theorem}\label{thm-BarbaEW4Roots}
Let $q$ be a prime power, then there is a Barba matrix of order $q^2+q+1$ over the fourth roots.
\end{theorem}
\begin{proof} This follows directly from Theorem \ref{thm-EWConstruction}, and Lemma \ref{lemma-CirculantEW2Barba}. Alternatively, from Corollary \ref{cor-EWSkew} to Theorem \ref{thm-EWConstruction}, we have that for every $q$ a prime power there exists a skew EW matrix of order $2(q^2+q+1)$. Let $n:=q^2+q+1$, Theorem \ref{thm-MorphismDeterminant} implies that 
\[2^n\gamma_4(n)^2=\gamma(2n)=(4n-2)(2n-2)^{(2n-2)/2}=2^n(2n-1)(n-1)^{n-1}.\]
Hence, $\gamma_4(n)=\sqrt{2n-1}(n-1)^{(n-1)/2}$, and by Theorem \ref{thm-BarbaBound4Roots} there is a Barba matrix over the third roots of order $n=q^2+q+1$.\qedhere
\end{proof}

It is unclear to the present author whether or not the result above was known to Cohn at the time of the publication of \cite{Cohn-ComplexDOptimal}. The results of Koukouvinos, Kounias, and Seberry in \cite{KKS-DSDOptimal} had already been published, however Cohn makes no mention of this family of EW matrices. To the best of our knowledge this existence result appears for the first time in this dissertation.

\begin{theorem}\label{thm-Barba4TwoEntries}
Let $B=J_v + (i-1) N$, where $N$ is a $\{0,1\}$-matrix of order $v$. Then $B$ is a Barba matrix if and only if $N$ is the incidence matrix of a $2$-$(v,k,k-(v-1)/2)$ design.
\end{theorem}
\begin{proof}
The argument is analogous to the one in the proof of Theorem \ref{thm-Barba3TwoEntries}.\qedhere
\end{proof}
\begin{corollary}\normalfont There is a unique Barba matrix with entries in $\{1,i\}$, up to monomial equivalence. Namely
\[B_3=
\begin{bmatrix}
i & 1 & 1\\
1 & i & 1\\
1 & 1 & i
\end{bmatrix}
\]
\end{corollary}
\begin{proof}
Suppose there is a Barba matrix $B$ with entries in $\{1,i\}$. Then, letting $B=J_v+(i-1)N$, Theorem \ref{thm-Barba4TwoEntries} implies that $N$ is the incidence matrix of a symmetric $2$-$(v,k,\lambda)$ design with $\lambda=k-(v-1)/2$. The parameters of a design satisfy
\[(v-1)\lambda=k(k-1),\]
so letting $x=(v-1)/2$, we find that $2x(k-x)=k^2-k$, and rearranging
\[x^2-kx+\frac{k(k-1)}{2}=0.\]
Therefore, 
\[v-1 = k\pm \sqrt{-k^2+2k}.\]
Since $-k^2+2k$ must be an integer, we have that $1\leq k\leq 2$. If $k=1$, then $v=3$, and the design is trivial. If $k=2$, then again $v=3$ and the design is trivial. The matrix $N$ can be taken to be $N=I_3$.\qedhere
\end{proof}

In Chapter \ref{chap-ASMaxdet} we will consider Barba matrices with entries in $\{\pm 1 ,\pm i\}$ in the Bose-Mesner algebra of a strongly regular graph, see Theorem \ref{thm-Barba4SRG}.

\subsection{Small maximal determinants over the fourth roots}

Below we include a list of maximal determinant matrices and our records for candidate maximal determinant matrices over the fourth roots. Here $n$ indicates the order of the matrix. The even columns are labelled $|\det|^2$,  and include the square absolute value of the determinant. The odd columns are labelled $|\det|^2$, and include the square absolute value of the determinant, divided by $2^{n-1}$, see Lemma \ref{lemma-Root4Div}. The columns labelled R include the ratio of the record determinant value with the applicable upper bound at each order, i.e. the Hadamard bound for even orders and the Barba bound for odd orders.

\begin{table}[H]
\centering
\begin{tabular}{|ccc||ccc||ccc||ccc|}
\hline
$n$ & $|\det|^2$ & R & $n$ & $|\det|^2/2^{n-1}$ & R & $n$ & $|\det|^2$ & R & $n$ & $|\det|^2/2^{n-1}$ & R\\
\hline
&&&1 & $1$ & $1$ & 2 & $2^2$ & $1$ & 3 & $ 5$ & $1$\\
4 & $4^4$ & $1$ & 5 & $2^4\times 3^2$ & $1$ & 6 & $6^6$ & $1$ & 7 & $3^6\times 13$ & $1$\\
8 & $8^8$ & $1$ & 9 & $4^{8}\times 17$ & $1$ & 10 & $10^{10}$ & $1$ & 11 & $2^2\times 5^{11} $\textcolor{wpicrimson}{??} & \textcolor{wpicrimson}{$0.97$}\\ 
12& $12^{12}$ & $1$ &13 & $ 6^{12}\times 5^2$ & $1$ & 14 & $14^{14}$ & $1$ & 
15 & $7^{14}\times 29$ & $1$\\
16 & $16^{16}$ & $1$ & 
17 &$13\times 137^4\times 1327^2$\textcolor{wpicrimson}{??} &\textcolor{wpicrimson}{$0.93$} &
18 & $18^{18}$ & $1$ &
 19 &
$3^{36}\times 37$ 
  & $1$ \\
20 & $20^{20}$ & $1$ & 21 & $10^{20}\times 41$ & $1$ & 22 & $22^{22}$ & $1$ & 23 &$3^2 \times 5\times 11^{22}$ &$1$\\
24 & $24^{24}$ & $1$ & 
$25$ & $2^{48}\times 3^{24}\times 7^2$ & $1$ &
26 & $26^26$ & $1$ & $27$ & $13^{26} \times 53$ & $1$\\
\hline
\end{tabular}
\caption{Maximal determinants and record determinants for matrices over the fourth roots.}\label{tab-Maxdet4}
\end{table}

Using Theorem \ref{thm-BarbaConstruction}, we can find real Barba matrices at orders $n=5=1^2+2^2$, and $n=13=2^2+3^2$. Using Theorem \ref{thm-BarbaEW4Roots} we can find Barba matrices over the fourth roots at orders $7=2^2+2+1$, $13=3^2+3+1$, and $21=4^2+4+1$.\\

The first sporadic example of a Barba matrix over the fourth roots is at order $n=9$. There is no real Barba matrix at order $n=9$ since $9$ is not the sum of two consecutive squares. To find such a matrix, we used a variation of  the method of Lampio, Österg\aa rd, and Szöll\H{o}si in \cite{Lampio-Ostergard-Szollosi-Butson}. We followed the following steps

\begin{enumerate}
\item We exhaustively construct a complete set of representatives (under monomial equivalence) of $k\times n$ matrix $M$ with the property
\[MM^*=(n-1)I_k+J_k.\]
To create these matrices, we use a technique of \textit{orderly generation}, see McKay \cite{McKay-GraphIsoI}: We assume that the first row is the all-ones vector, and we ensure that each row $r_i$ is lexicographically ordered on each interval of columns $I=\{a,a+1,\dots,a+r\}$ where $r_{i-1,j}$ is constant for all $j\in I$.
\item For each matrix $M$ as above, we generate the complete set of rows $r$ of length $n$ with entries in $\{\pm 1,\pm i\}$ which have inner product $1$ with all rows in $M$, and are lexicographically larger than all rows of $M$. Call this set $R_M$.
\item We create the \textit{compatibility graph}\index{graph!compatibility} of the set of rows $R_M$. This is a graph with $m=|R_M|$ vertices with an edge between rows $u$ and $v$ if and only if $u\cdot v=1$. Call this graph $G_M$.
\item We search for a clique of size $n$ in $G_M$ using \texttt{cliquer}, \cite{cliquer}.
\end{enumerate}

In our case, letting $k=3$ and $n=9$, we find a total of $190$ equivalence classes of matrices $M$ as above. From the first such matrix (written logarithmically)
\[M=
\begin{bmatrix}
0&0&0&0&0&0&0&0&0\\
0&0&0&0&0&2&2&2&2\\
0&0&0&2&2&0&0&2&2\\
\end{bmatrix},
\]
we find the following Barba matrix
\[
B_9=
\begin{bmatrix}
0&0&0&0&0&0&0&0&0\\
0&0&0&0&0&2&2&2&2\\
0&0&0&2&2&0&0&2&2\\
0&0&2&1&3&1&3&0&2\\
0&2&0&2&0&1&3&3&1\\
1&3&3&1&0&0&1&3&2\\
2&0&0&1&3&2&0&3&1\\
3&1&3&0&1&1&0&3&2\\
3&3&1&1&0&1&0&2&3
\end{bmatrix}.
\]
An equivalent normal Barba matrix with constant row-sum of value $i-4$ is the following:
\[
\begin{bmatrix}
-1 &-1 &-1 &\phantom{-}i &-1 &\phantom{-}i &-1 &-i &\phantom{-}1\\
-1 &-1 &-1 &\phantom{-}i &-1 &-i &\phantom{-}1 &\phantom{-}i &-1\\
-1 &-1 &-1 &-i &\phantom{-}1 &\phantom{-}i &-1 &\phantom{-}i &-1\\
-1 &-1 &\phantom{-}1 &-1 &\phantom{-}i &-1 &\phantom{-}i &-i &-1\\
-1 &\phantom{-}1 &-1 &-i &-1 &-1 &\phantom{-}i &-1 &\phantom{-}i\\
-i &\phantom{-}i &\phantom{-}i &-1 &-1 &\phantom{-}i &-i &-1 &-1\\
\phantom{-}1 &-1 &-1 &-1 &\phantom{-}i &-i &-1 &-1 &\phantom{-}i\\
\phantom{-}i &-i &\phantom{-}i &\phantom{-}i &-i &-1 &-1 &-1 &-1\\
\phantom{-}i &\phantom{-}i &-i &-1 &-1 &-1 &-1 &\phantom{-}i &-i
\end{bmatrix}.
\]

The next sporadic example of a Barba matrix we find is at order $n=15$. Here, we have that $2\cdot 15-1=29=5^2+2^2$. Since the search space is much larger than in the case $n=9$ we restrict the search to the set rows with entries in $\{\pm 1,\pm i\}$ with row sum equal to $2+5i$. Applying the algorithm described above with this restriction, we find the following normal Barba matrix with constant row sum equal to $2+5i$.

\[
B_{15}=
\left[
\begin{array}{*{15}{c}}
0&0&0&0&0&0&1&1&1&1&1&2&2&2&2\\
0&0&0&1&1&1&0&2&2&3&3&1&1&1&1\\
0&0&1&0&2&2&2&0&2&1&1&0&0&1&1\\
0&1&2&2&0&1&1&0&1&1&3&0&1&2&0\\
0&2&1&2&1&0&0&1&3&1&1&3&1&0&1\\
1&1&0&2&3&0&1&3&1&2&1&1&0&0&1\\
1&1&1&3&1&3&0&1&1&0&2&1&3&1&0\\
1&1&3&0&1&1&3&3&0&1&1&2&1&1&0\\
1&2&0&1&2&0&2&1&1&0&0&0&1&1&3\\
1&3&1&1&0&2&1&1&0&2&0&2&0&1&0\\
1&3&2&1&1&1&0&0&1&0&1&1&0&3&2\\
2&0&1&1&1&0&2&0&0&1&3&1&2&0&1\\
2&1&1&0&0&2&0&2&1&1&0&0&1&0&2\\
2&1&3&1&0&1&1&1&3&0&1&0&3&1&1\\
3&1&1&0&2&1&1&1&0&3&1&1&1&3&0
\end{array}
\right].
\]

The Barba matrices of orders $19$, $23$, $25$, and $27$ can be found by applying Lemma \ref{lemma-CirculantEW2Barba}. In Orrick's website \cite{Orrick-Website} there are several examples of EW matrices of order $38=2\times 19$,  $46=2\times 23$, $50=2\times 25$ and $54=2\times 27$ with the circulant block structure 
\[
W=\begin{bmatrix}
A & B\\
-B^{\intercal} & A^{\intercal}
\end{bmatrix},
\]
where $A$ and $B$ are circulant. For example we have the matrices
\begin{align*}
&A_{19}:[+++++---+-++-++++-+]\\
&B_{19}:[+++--+--+++-+-+++-+]\\
&A_{23}:[+++++++-++-+++---+-++-+]\\
&B_{23}:[+++---+--++++-+-+-++--+]\\
&A_{25}:[+++++-+-+-++----+++-+++-+]\\
&B_{25}:[+++++-++---+++-++-++-+--+]\\
&A_{27}:[++++--+++-+++-+-+--+-++++-+]\\
&B_{27}:[++++----+-+--++-+++++-++--+]
\end{align*}

To the best of our knowledge, all currently known EW matrices at orders $2n\geq 54$ are of the circulant block  form above, so all these yield Barba matrices at orders $n$. See the paper by Cohn \cite{Cohn-Determinants-I}, and the papers by Yang \cite{Yang-SomeDesigns,Yang-Construction,Yang-OnDesignsI,Yang-OnDesignsII}.\\

At orders $n=11$, and $n=17$, we have that $2n-1$ is not a sum of two squares, so the Barba bound cannot be achieved. The circulant matrices with largest determinant that we found are:
\begin{align*}
&A_{11}:[0 2 3 1 1 1 1 1 3 2 0]\\
&A_{17}:[0 1 2 1 0 1 0 0 2 1 3 3 3 1 0 1 3]
\end{align*}
In \cite{Cohn-ComplexDOptimal}, Cohn claimed the existence of a matrix of order $n=11$ with determinant $434976$. Adam Zsolt Wagner reported the following matrix, achieving the current record:

\[M_{11}=
\left[
\begin{array}{*{11}{c}}
3&0&1&3&2&2&3&0&1&2&0\\
2&1&3&0&1&0&1&2&1&2&0\\
3&3&3&3&0&1&0&2&3&2&1\\
0&1&2&2&0&0&2&3&0&2&1\\
0&2&0&2&3&2&1&1&1&2&3\\
0&3&3&2&2&3&0&3&1&0&0\\
0&0&0&1&0&1&3&1&1&0&0\\
1&0&3&1&3&2&3&3&2&2&3\\
1&3&0&1&2&3&3&1&0&2&1\\
0&3&1&1&1&0&0&0&2&2&3\\
0&1&0&3&2&0&3&2&3&2&3
\end{array}
\right]
\]

The Gram matrix of $M_{11}$ is the following:

\[
M_{11}M_{11}^*=\left[
\begin{array}{*{11}{c}}
11&-&-&-&1&1&1&1&1&1&1\\
-&11&1&1&-&-&-&-&-&-&1\\
-&1&11&1&-&-&-&-&-&-&1\\
-&1&1&11&-&-&-&-&-&-&1\\
1&-&-&-&11&1&1&1&1&1&1\\
1&-&-&-&1&11&1&1&1&1&-\\
1&-&-&-&1&1&11&1&1&1&-\\
1&-&-&-&1&1&1&11&1&1&1\\
1&-&-&-&1&1&1&1&11&1&1\\
1&-&-&-&1&1&1&1&1&11&1\\
1&1&1&1&1&-&-&1&1&1&11
\end{array}
\right]
\]

\begin{research-problem}\normalfont Find the maximal determinant of a $\{\pm 1, \pm i\}$ matrix at orders $n=11$ and $n=17$.
\end{research-problem}

\section{Certificates of maximality}
\label{sec-MaximalityCertificates}

In the cases where the general upper bounds cannot be met, we need a certificate of maximality for our candidate matrices. This procedure involves a great deal of computation, and good strategies are needed to traverse our search spaces. There has been much work done in  obtaining certificates of maximality for $\pm 1$ matrices: The first result of this type that we have knowledge of is the proof of maximality of a $\pm 1$ matrix of order $17$, due to Moyssiadis and Kounias \cite{Greek-17}. See also the case $n=21$ by Chadjipantelis, Moyssiadis and Kounias \cite{Greek-21}. To approach the more challenging cases where $n\equiv 3\pmod{4}$, Orrick \cite{Orrick-15}, Brent and Osborn \cite{Brent-Orrick-Osborn-19-37}, refined these methods. Some of their key improvements include the use of the techniques of orderly generation of McKay \cite{McKay-GraphIsoI}, as well as the introduction of particular upper bounds for the congruence class $3\pmod{4}$.\\

The following is a well-known generalisation of the Muir-Kelvin bound:
\begin{theorem}[Fischer's inequality, Theorem 7.8.5. \cite{Horn-Johnson}]\index{determinant inequality!Fischer} \label{thm-FischerInequality} Let
\[M=\begin{bmatrix}
A & B\\
B^* &C
\end{bmatrix},
\]
be an Hermitian positive-definite matrix. Then,
\[\det(M) \leq \det(A)\det(C).\]
\end{theorem}

The main result that we will use is the following generalisation of the determinant bound of Moyssiadis and Kounias:

\begin{theorem}[cf. Moyssiadis and Kounias \cite{Greek-17}]\label{thm-ExtendedKounias}
Let $\Phi$ be a finite subset of $\C$, and let $c>0$ be a real number such that $|x|\geq c$ for all $x\in\Phi$. Suppose that $D$ is a given Hermitian positive-definite matrix of order $r\geq 1$, with off-diagonal entries in $\Phi$ and with $d_{ii}=n$. Furthermore, let $M$ be an $\ell\times \ell$ Hermitian positive-definite matrix, with $\ell> r$, extending $D$ in the following way:
\[M=\left[\begin{array}{c|c}
D & B\\
\hline
B^{*} & A
\end{array}
\right],
\]
where $a_{ii}=n$ and all entries of $A$ and $B$ are in $\Phi$. If
\[\hat{d}=\det\left[
\begin{array}{c|c}
D & \hat{\gamma}\\
\hline
\hat{\gamma}^{*} & c
\end{array}
\right]=\max_{\gamma\in \Phi^r}\det
\left[
\begin{array}{c|c}
D & \gamma\\
\hline
\gamma^{*} & c
\end{array}
\right],
\]
then
\[|\det(M)|\leq (n-c)^{\ell-r-1}[(n-c)\det(D)+(\ell-r)\max(0,\hat{d})].\]
\end{theorem}
\begin{proof}
We prove this by induction on $\ell$: The base case is $\ell=r+1$. By linearity of the determinant on rows, we have
\[
\det M =\det \begin{bmatrix}
D & \gamma\\
\gamma^* & n
\end{bmatrix}=\det\begin{bmatrix}
D & \gamma\\
0 & n-c
\end{bmatrix}
+\det
\begin{bmatrix}
D & \gamma\\
\gamma^* & c
\end{bmatrix}
\leq (n-c)\det(D) + \max(0,\hat{d}).
\]
So the base case holds. Now, suppose that the statement is true for $\ell>r$, we show that it is true for $\ell+1$: Since $a_{ii}=n$, we may write $M$ in the following form:
\[M=
\begin{bmatrix}
D & B_1 & \gamma\\
B_1^* & A_1 &\delta\\
\gamma^*&\delta^* & n
\end{bmatrix},
\]
where $\gamma$ is a column vector of length $r$, and $\delta$ is a column vector of length $\ell-r$. By linearity of the determinant on rows, we have that
\[\det M=
\det\begin{bmatrix}
D & B_1 & \gamma\\
B_1^* & A_1 & \delta\\
0 & 0 & n-c
\end{bmatrix}
+
\det\begin{bmatrix}
D & B_1 & \gamma\\
B_1^* & A_1 & \delta\\
\gamma^* & \delta^* & c
\end{bmatrix}.
\]
Letting $F=\begin{bmatrix}
D & B_1 & \gamma\\
B_1^* & A_1 & \delta\\
\gamma^* & \delta^* & c
\end{bmatrix}$, we have that 
\[\det(M)=(n-c)\det\begin{bmatrix}
D & B_1\\
B_1^* & A_1
\end{bmatrix}+\det(F).\]
If $\det(F)\leq 0$, then $\det(M)\leq (n-c)\det M_1$, where $M_1=\begin{bmatrix}
D & B_1\\
B_1^* & A_1 
\end{bmatrix}$ is an $\ell\times \ell$ matrix. Applying the induction hypothesis to $M_1$ we find
\begin{align*}\det(M)&\leq (n-c)^{\ell-r}[(n-c)\det(D)+(\ell-r)\max(0,\hat{d})]\\
&\leq (n-c)^{\ell-r}[(n-c)\det(D)+(\ell-r+1)\max(0,\hat{d})].
\end{align*}
Suppose that $\det(F)>0$. A series of elementary row operations shows that 
\[\det(F)=\det
\begin{bmatrix}
D-\gamma\gamma^*/c & B_1-\gamma\delta^*/c &0\\
B_1-\delta\gamma^*/c & A_1-\delta\delta^*/c & 0\\
\gamma^* & \delta^* & c
\end{bmatrix}=c\det \begin{bmatrix}
D-\gamma\gamma^*/c & B_1-\gamma\delta^*/c \\
B_1-\delta\gamma^*/c & A_1-\delta\delta^*/c
\end{bmatrix}.
\]
Since $\det(F)>0$, Sylvester's Criterion, Theorem \ref{thm-SylvesterCriterion}, implies that $F$ is positive-definite. By Fischer's inequality, Theorem \ref{thm-FischerInequality},  it follows
\[\det(F)\leq c\det(D-\gamma\gamma^*/c)\det(A_1-\delta\delta^*/c).\]
 Again by Sylvester's criterion, $A_1$ is  Hermitian positive-definite. Applying the Muir-Kelvin bound, Theorem \ref{thm-MuirKelvinBound} we have
\[\det(A_1)\leq \prod_{i=1}^{\ell-r}\left(n-\frac{|\delta_{ii}|^2}{c}\right)\leq (n-c)^{\ell-r}. \]
On the other hand,
\[\det(D-\gamma\gamma^*/c)=\frac{1}{c}\det\begin{bmatrix}
D & \gamma\\
\gamma^* & c
\end{bmatrix}\leq \frac{\max(0,\hat{d})}{c}.\]
Therefore, if $\det(F)>0$, 
\begin{align*}
\det(F)\leq \max(0,\hat{d})(n-c)^{\ell-r}
\end{align*}
Applying the induction hypothesis to  $M_1$, we find:
\begin{align*}
\det(M) & \leq (n-c)\det(M_1) + \max(0,\hat{d})(n-c)^{\ell-r}\\
&\leq (n-c)^{\ell-r}[(n-c)\det(D)+(\ell-r)\max(0,\hat{d})]+\max(0,\hat{d})(n-c)^{\ell-r}\\
&=(n-c)^{\ell-r}[(n-c)\det(D)+(\ell+1-r)\max(0,\hat{d})].\qedhere
\end{align*}
\end{proof}
\begin{remark*}\normalfont Let $\Phi$ be the set of all sums of $m$-th roots of unity with length $n$. Then, the off-diagonal entries of $XX^*$ lie in $\Phi$ for any matrix $X$ of order $n$ with entries over $\mu_m$. If $m=2,3,4,6$, then by Corollary \ref{cor-MinSumLattice}, we can take $c=1$ in Theorem \ref{thm-ExtendedKounias}, and the bound takes the shape
\[\det(M)\leq (n-1)^{\ell-r-1}[(n-1)\det(D)+(\ell-r)\max(0,\hat{d})].\]
\end{remark*}

To prove that a certain matrix $X_0$ is of maximal determinant, Moyssiadis and Kounias \cite{Greek-17} proposed the strategy of constructing the set of ``potential Gram matrices'': Suppose that $X$ is a matrix with entries in $\mu_m$, and let $\Phi$ be the set of sums of $m$-th roots of unity of length $n$. Let $\mathcal{M}_{n,\ell}$ be the set of Hermitian positive-definite matrices of order $\ell$ with $n$'s along the diagonal, and whose off-diagonal elements are in $\Phi$. Since all matrices in $\mathcal{M}_{n,\ell}$ are Hermitian positive-definite, they form a poset. Let $\mathcal{M}_{\ell,n}(d)$ be the following subset
\[\mathcal{M}_{\ell,n}(d):=\{M\in\mathcal{M}_{k,n}: |\det(M)|\geq d\}.\]
Let $d_{0}=|\det(X_0X_0^*)|=|\det(X_0)|^2$. We can construct $\mathcal{M}_{\ell,n}(d_0)$ with a backtracking search as follows.

 \begin{itemize}
 \item[(1)] Initialise $\Phi_1:=\Phi$, $M_1:=(n)$, $k:=1$, and $i:=1$.
 \item[(2)] Given $M_k$ create all extended matrices $M_{\ell+1}^{(v)}$ by iterating over all possible vectors $v\in \Phi_i^\ell$ and letting
 \[M_{\ell+1}^{(v)}=\left[
 \begin{array}{c|c}
 M_{\ell} & v\\
 \hline
 v^{*} & n
 \end{array}
 \right].\]
 \begin{itemize}
 \item If $\ell+1=n$ and $|\det(M_{\ell+1}^{(v)})|\geq d_0$, then print $M_{\ell+1}^{(v)}$.
 \item If $\ell+1<n$: Apply Theorem \ref{thm-ExtendedKounias} with $m=\ell+1$ and $r=k$ to do a pruning step: If the bound from the theorem implies $\det (M_{\ell+1}^{v})< d_0$ then discard $M_{\ell+1}^{(v)}$. Let $\mathcal{A}$ be the subset of all $v\in \Phi_i^\ell$ such that $M_{\ell+1}^{(v)}$ survives the pruning step. After the pruning has been carried update $i\leftarrow i+1$ and build the set $\Phi_{i+1}$ by removing all the entries that do not appear in any $M_{\ell+1}^{(v)}$ from $\Phi_{i}$. For each remaining $M_{\ell+1}^{(v)}$ do a recursion step by going to step (2) with the updated values of $i$ and $\Phi_i$, with $\ell \leftarrow \ell+1$ and $M_{\ell+1}^{(v)}$ in place of $M_\ell$.
 \end{itemize}
 \end{itemize}
 
One of the advantages of searching for Gram matrices instead of matrices with entries in $\mu_m$ is that the action of a monomial matrix $P$ on columns of $X$ leaves the Gram matrix unaltered:
 \[(X P)(X P)^*=X PP^{*}X^*=XX^*.\]
 \begin{definition}\normalfont Two Hermitian positive-definite matrices $M_1$ and $M_2$ are $m$\textit{-isomorphic} if and only if there exists a monomial matrix $P$ with non-zero entries in the set $\mu_m$ such that
 \[P^*M_1P=M_2.\]
 \end{definition} 
 The generation of matrices in step (2) is bound to produce many isomorphic examples. Here is where the orderly generation techniques will be useful. For example, generating the matrices lexicographically and creating canonical forms enables us to greatly improve the efficiency of the search. Once a list of putative Gram matrices with larger determinant than $d_0$ has been generated, we use the methods described in Chapter \ref{chap-GramEquations} and Chapter \ref{chap-HermitianForms}, to rule out their decomposability as Gram matrices. If these methods do not succeed, then we can use integrality conditions such as the ones in Lemma \ref{lemma-Root3Div} or Lemma \ref{lemma-Root4Div}, which show that the determinant must be divisible by a large power of $3$ in the case of the third roots, or a large power of $2$ in the case of the fourth roots.
\begin{lemma}[Cauchy-Binet Formula, Theorem 4.2.16 \cite{Brualdi-Ryser-CombinatorialMatrixTheory}]\index{Cauchy-Binet formula} \label{lemma-CauchyBinet}\normalfont Let $A$ be a matrix of order $n$, and denote by $\wedge^k A$ the $k$-th exterior product of $A$, i.e. the matrix of order ${n\choose k}$ whose entries correspond to $k$-minors of $A$. Then, for any pair of matrices $A$ and $B$ of order $n$
\[\wedge^k(AB)=\wedge^kA\wedge^kB.\] 
\end{lemma} 
  The following is an extension of the result in Lemma \ref{lemma-Root3Div} based on an idea of Greaves and Yatsyna \cite{Greaves-Yatsyna-Eq17Seidel}, 
 and it imposes strong arithmetic conditions on the characteristic polynomial of a candidate Gram matrix.  
 \begin{proposition}\normalfont Let $M=XX^*$, where $X$ is an $n\times n$ matrix with entries in $\{1,\omega,\omega^2\}$. Let 
 \[p_M(x)=x^n-n^2x^{n-1}+a_2x^{n-2}+\dots+a_{n-1}x+a_n.\]
 Then, $a_i\in \Z$ for all $i=2,\dots,n$, and $3^{i-1}\mid a_i$.
 \end{proposition}
 \begin{proof}
 Up to sign, the $k$-th coefficient of the characteristic polynomial of $M$ is the sum of all principal $k$-minors of $M$. In the language of exterior products of $M$ this can be written as
 \[a_k=(-1)^k\tr(\wedge^k M).\]
 By the Cauchy-Binet formula, Lemma \ref{lemma-CauchyBinet}, we have that
 \[\wedge^k M=\wedge^k(XX^*)=\wedge^k X \wedge^k X^*=(\wedge^k X)(\wedge^k X)^*.\]
 By the proof of Lemma \ref{lemma-Root3Div}, we have that each $k$-minor of $X$ is divisible by $(1-\omega)$ in $\Z[\omega]$. Therefore, each entry of $\wedge^k M$ is divisible by $[(1-\omega)(1-\omega^2)]^{k-1}=3^{k-1}$. Furthermore, $(\wedge^k M)^* =\wedge^k M^*=\wedge^k M$, so the diagonal entries of $M$ must be in $\Z[\omega]\cap \Q=\Z$. Therefore, $\tr(\wedge^k M)\in\Z$ is divisible by $3^{k-1}$, and this concludes the proof.\qedhere
 \end{proof}
 \begin{remark*}\normalfont The result above can be easily extended to matrices with entries in $\{\pm 1\}$ or $\{\pm 1,\pm i\}$. In these cases, the factor $3^{k-1}$ is replaced with $4^{k-1}$ and $2^{k-1}$, respectively.
 \end{remark*}
 
 \subsection{The maximal determinant at order 5 over the third roots}
 
 Here we prove that the matrix
\[
M_5=\begin{bmatrix}
 1 & \omega & 1 & \omega & \omega^2\\
 1 & \omega & \omega^2 & \omega & 1\\
 1 & 1 & \omega & \omega^2 & \omega\\
 1 & \omega^2 & \omega & 1 & \omega\\
 \omega & 1 & 1 & 1 & 1
\end{bmatrix}
\]
with Gram matrix
\[M_5M_5^*=\begin{bmatrix}
5 & 2 & - & - & -\\
2 & 5 & - & - & -\\
- & - & 5 & 2 & -\\
- & - & 2 & 5 & -\\
- & - & - & - & 5
\end{bmatrix}
\]
is a maximal determinant matrix. Notice  that $\det(M_5M_5^*)=1701=3^5\cdot 7$. As described at the beginning of this section, we recursively construct all candidate Gram matrices of a matrix with entries in $\{1,\omega,\omega^2\}$ and determinant $\geq 1701$. Take $\Phi$ to be the set of all possible inner products of two vectors of size $5$. For example
\[\Phi_1 =\{-1,-\omega,-\omega^2,2\omega,2\omega^2,\omega-2\omega^2,\dots,5,5\omega,5\omega^2\}.\]
Applying Theorem \ref{thm-ExtendedKounias} with $r=2$ we find that the only off-diagonal elements taken from $\Phi$ that produce an extended matrix of determinant $\geq 1701$ are
\[\Phi_2=\{-1,-\omega,-\omega^2,2,2\omega,2\omega^2,\omega-2\omega^2,\omega^2-2\omega,-2\omega+1,-2\omega^2+1\},\]
thus it is enough to consider only $\Phi_1$ in what follows. For the case $r=3$ this set is further reduced to
\[\Phi_2=\{-1,-\omega,-\omega^2,2,2\omega,2\omega^2\},\]
since all possible submatrices of size $3$ with entries taken from $\Phi_1$ that extend to a $5\times 5$ positive definite matrix of determinant at least $1701$ are
\[
\begin{array}{c c c c}
\vspace{12pt}
1:
\begin{bmatrix}
5 & 2\omega & 2\omega\\
2\omega^2 & 5 &2\\
2\omega^2 & 2 & 5
\end{bmatrix},
&
2:
\begin{bmatrix}
5 & 2\omega & 2\omega\\
2\omega^2 & 5 &-\omega\\
2\omega^2 & -\omega^2 & 5
\end{bmatrix},
&
3:
\begin{bmatrix}
5 & 2\omega & 2\\
2\omega^2 & 5 &-\omega\\
2 & -\omega^2 & 5
\end{bmatrix},
&
4:
\begin{bmatrix}
5 & 2\omega & 2\\
2\omega^2 & 5 &-1\\
2 & -1 & 5
\end{bmatrix},\\
\vspace{12pt}
5:
\begin{bmatrix}
5 & 2\omega & -\omega\\
2\omega^2 & 5 &-\omega\\
-\omega^2 & -\omega^2 & 5
\end{bmatrix},
&
6:
\begin{bmatrix}
5 & 2\omega & -\omega\\
2\omega^2 & 5 &-1\\
-\omega^2 & -1 & 5
\end{bmatrix},
&
7:
\begin{bmatrix}
5 & 2\omega & -1\\
2\omega^2 & 5 &-1\\
-1 & -1 & 5
\end{bmatrix},
&
8:
\begin{bmatrix}
5 & 2 & 2\\
2 & 5 & 2\\
2 & 2 & 5
\end{bmatrix},
\\
\vspace{12pt}
9:
\begin{bmatrix}
5 & 2 & 2\\
2 & 5 & -\omega\\
2 & -\omega^2 & 5
\end{bmatrix},
&
10:
\begin{bmatrix}
5 & 2 & -\omega\\
2 & 5 & -\omega\\
-\omega^2 & -\omega^2 & 5
\end{bmatrix},
&
11:
\begin{bmatrix}
5 & 2 & -\omega\\
2 & 5 & -1\\
-\omega^2 & -1 & 5
\end{bmatrix},
&
12:
\begin{bmatrix}
5 & 2 & -1\\
2 & 5 & -1\\
-1 & -1 & 5
\end{bmatrix},\\
13:
\begin{bmatrix}
5 & -\omega & -\omega\\
-\omega^2 & 5 & -\omega\\
-\omega^2 & -\omega^2 & 5
\end{bmatrix},
&
14:
\begin{bmatrix}
5 & -\omega & -\omega\\
-\omega^2 & 5 & -1\\
-\omega^2 & -1 & 5
\end{bmatrix},
&
15:
\begin{bmatrix}
5 & -\omega & -1\\
-\omega^2 & 5 & -1\\
-1 & -1 & 5
\end{bmatrix},
&
16:
\begin{bmatrix}
5 & -1 & -1\\
-1 & 5 & -1\\
-1 & -1 & 5
\end{bmatrix}.
\end{array}
\]
The corresponding determinant bounds for each matrix, obtained via Theorem \ref{thm-ExtendedKounias} are given in the following table:

\begin{table}[H]
\centering
\begin{tabular}{c|*{8}c}
$i$ & $1$ & $2$ & $3$ & $4$ & $5$ & $6$ & $7$ & $8$\\
\hline
Theorem \ref{thm-ExtendedKounias} Bound & $1728$ & $1752$ & $1752$ & $1752$ & $1896$ & $2016$ & $1824$ & $1728$
\end{tabular}\\

\begin{tabular}{c|*{8}c}
$i$ & $9$ & $10$ & $11$ & $12$ & $13$ & $14$ & $15$ & $16$\\
\hline
Theorem \ref{thm-ExtendedKounias} Bound & $1752$ & $2016$ & $1896$ & $2016$ & $2184$ & $2160$ & $2184$ & $2016$
\end{tabular}
\end{table}

We can further reduce the list of $3\times 3$ submatrices by considering equivalence of Gram matrices by monomial matrices with entries in $\{1,\omega,\omega^2\}$. In this way, we can carry step $r=4$ using only candidates $7,8,9,11,12,15$ and $16$. Proceeding in this way we find a total of $42$ candidate Gram matrices with determinant $>1701$. Among these matrices, $37$ have determinants that do not contain factors $p\equiv 2\pmod{3}$ in their square-free part, see Proposition \ref{prop-DetNormCondition}. These determinants are $1728, 1809,$ and $1971$, and $3^{n-1}=3^4$ does not divide any of them, so Lemma \ref{lemma-Root3Div} implies that none of the $37$ matrices can be Gram matrices of a  matrix over the third roots of unity. Therefore, the matrix $M_5$ is maximal determinant.\\

For a small order like $n=5$ the proof of maximality can also be done by brute force, but for larger values of $n$ the method we outlined above is more efficient. We note that applying this method to the $\pm 1$ maximal determinant matrix of order $17$ we were able to confirm the results of Moyssiadis and Kounias in \cite{Greek-17}. In the case $\{1,\omega,\omega^2\}$ at order $n=8$ the increased number of phase factors make the same approach infeasible without an isomorphism check at every stage $k=2,3,\dots,8$. For this purpose, \texttt{nauty} \cite{McKay-Piperno-GraphIsoII} may provide the necessary tools.
\begin{research-problem}\normalfont Develop computational techniques to prove maximality (or otherwise) of the open cases in Tables \ref{tab-Maxdet3} and \ref{tab-Maxdet4}.
\end{research-problem}

\cleardoublepage
\chapter{Maximal Determinants in Association Schemes}\label{chap-ASMaxdet}
In Chapter \ref{chap-Maxdet}, we characterised certain types of maximal determinant matrices by their Gram matrices. For example, Hadamard matrices are characterised by the equation $HH^*=nI_n$, and Barba matrices by $BB^*=(n-1)I_n+J_n$. Furthermore, we saw in Theorem \ref{thm-BarbaConstantRowSum} that Barba matrices have constant row sum, and because of this, the Bose-Mesner algebra of an association scheme is a good place to search for these types of matrices.\\

Here we consider a variation of the questions we investigated in Chapter \ref{chap-GramEquations}, where we solve the equation $XX^*=M$ for given $M$, under the assumption that both $M$ and $X$ are in the Bose-Mesner algebra of an association scheme. Using the basic properties of association schemes, particularly the interplay between the matrix product and the Schur product, we extract conditions on the entries of $X$ that characterise the solvability of $XX^*=M$. These conditions are given by a system of real quadratic polynomials, which can be studied using Gröbner bases. Since we are interested in maximal determinant matrices, we will also add the condition that the entries of $X$ are unimodular, although we mention that our methods do not require this assumption and they can also be used to find more general types of matrices, such as type-II matrices, see \cite{Chan-Godsil}. \\

With our approach, we reproduce part of the results in \cite{Chan-ComplexHadamard}, and \cite{Ikuta-Munemasa-Bordered}. Furthermore, we find new families of Barba matrices, and classify Hadamard matrices in $2$-class asymmetric association schemes. 
\section{Gram matrices in association schemes}

 Throughout this section we will consider a $d$-class association scheme $\mathcal{X}$, not necessarily symmetric, with Bose-Mesner algebra $\mathcal{A}$. Recall the following notation: The incidence matrices of $\mathcal{X}$ are denoted as $\{A_0=I_n,A_1,\dots,A_d\}$. We denote by $i'$ the index in $\{0,1,\dots,d\}$ such that $A_i^{\intercal}=A_{i'}$. The intersection numbers $p_{ij}^k$ are defined by 
\[A_iA_j=\sum_{k=0}^{d}p_{ij}^k A_k.\]
The primitive idempotents of $\mathcal{X}$ are denoted $\{E_0=\frac{1}{n}J_n,E_1,\dots,E_d\}$. The Schur product, or entrywise product, of two matrices $A$ and $B$ is the matrix $A\circ B$, given by $[A\circ B]_{ij}=A_{ij}B_{ij}$, and clearly $A_{i}\circ A_j=\delta_{ij}A_i$.\\

We have the following well-known fact:
\begin{lemma}\normalfont \label{lemma-ComplexGramAS} Let $M$ be a matrix in the Bose-Mesner algebra $\mathcal{A}$ of a $d$-class association scheme. Then $M=XX^*$ for some  $X\in\GL_n(\C)$ if and only if there exist real numbers $\lambda_i> 0$ such that
\[M=\sum_{i=0}^d \lambda_i E_i.\]
\end{lemma}
\begin{proof}
Let $M=XX^*$ for some $X\in\GL_n(\C)$. Since $X$ is invertible, then $M$ is positive-definite, so all eigenvalues $\lambda_i$ of $M$ are real and positive. This implies
\[M=\sum_{i=0}^{d} \lambda_iE_i.\]
Conversely, if $M=\sum_{i=0}^d\lambda_i E_i$, then $M$ is positive-definite and by Theorem \ref{thm-GramPD} there exists a matrix $X\in\GL_n(\C)$ such that $M=XX^*$.\qedhere
\end{proof}

Lemma \ref{lemma-ComplexGramAS} shows that the matrices $M$ in a Bose-Mesner algebra $\mathcal{A}$ that split as $XX^*=M$ with $X\in\GL_n(\C)$ are precisely the matrices in the \textit{positive-definite cone} of the scheme. However, this does not guarantee that the matrix $X$ belongs to $\mathcal{A}$.
 
\begin{definition}\normalfont Let $\mathcal{X}$ be an association scheme. For a fixed $k$, we define the $k$-th \textit{symmetric intersection} matrix as \index{matrix!symmetric intersection}
\[P_k=(p_{ij}^k)_{ij}.\] 
\end{definition}
Note the difference with the usual intersection matrices $B_i=(p_{ij}^k)_{jk}$, corresponding to the regular representation of $\mathcal{X}$, where instead of $k$ one fixes $i$, see \cite{Bannai-Ito}. The following basic results hold
\begin{lemma} \normalfont Let $\mathcal{X}$ be an association scheme. Then
\begin{itemize}
\item[(i)] The matrices $P_k$ are symmetric.
\item[(ii)] The $j$-th column of $P_k$ is the $k$-th column of $B_j$. 
\end{itemize} 
\end{lemma}
\begin{proof}
The first claim is a consequence of the commutativity of $\mathcal{A}$, i.e. $p_{ij}^k=p_{ji}^k$. The $j$-th column of $P_k$ is the vector $(p_{ij}^k)_i$, and the $k$-th column of $B_j=(p_{ji}^k)_{ik}=(p_{ij}^k)_{ik}$ is the vector $(p_{ij}^k)_i$ as well, hence (ii) follows.  
\end{proof}

\begin{theorem} \label{thm-GramAssociationScheme}
Let $M=\sum_{k=0}^{d} \alpha_kA_k$ be a matrix in the Bose-Mesner algebra $\mathcal{A}$ of a $d$-class association scheme. Then, $M=NN^*$ where $N=\sum_k \beta_k A_k$ if and only if for all $k=0,1,\dots, d$,
\[\beta^*(WP_k)\beta = \alpha_k,\]
where $\beta=(\beta_0,\beta_1,\dots,\beta_d)^{\intercal}$, and $W$ is the permutation matrix given by the involution $i\mapsto i'$.
\end{theorem}
\begin{proof}
With the notation in the statement, we have that  $M=NN^*$ if and only if $M=(\sum_i \beta_iA_i)(\sum_j \overline{\beta_{j'}} A_j)=\sum_{ij}\beta_i\overline{\beta_{j'}}A_iA_j$. Therefore,
\[\alpha_kA_k=M\circ A_k=\left(\sum_{ij}\beta_i\overline{\beta_{j'}}\sum_{\ell}p_{ij}^{\ell}A_{\ell}\right)\circ A_{k}=\left(\sum_{ij}\overline{\beta_{j'}}p_{ij}^k \beta_i\right)A_k.\]
Thus $\alpha_k=\sum_{ij}\overline{\beta_{j'}}p_{ij}^k\beta_i$, and this can be rewritten as $\alpha_k=\beta^*P_k\beta$ in the symmetric case and as $\alpha_k=\beta^*WP_k\beta$ in the asymmetric case.\qedhere
\end{proof}

 Since the condition $|\alpha|^2=\alpha\overline{\alpha}=1$ for a complex number $\alpha$ is not polynomial in $\alpha$ we write instead $\alpha=a+bi$ and obtain a quadratic constraint $a^2+b^2-1=0$. Hence, a matrix $M=\sum_{k}\alpha_k A_k$ in the Bose-Mesner algebra $\mathcal{A}$ can be written as $XX^*=M$ for some $X=\sum_{i}\beta_iA_i\in\mathcal{A}$ with unimodular entries if and only if 
 \[
\begin{cases}
\beta^* WP_k\beta =\alpha_k & \text{ for all } k=0,\dots, d\\
x_k^2+y_k^2 = 1 & \text{ for all } k=0,\dots,d
\end{cases},
 \]
 where $\beta_k =x_k+iy_k$, and $\beta=(\beta_0,\dots,\beta_d)$.\\
 
  For an association scheme of order $n$ with Bose-Mesner algebra $\mathcal{A}$, and the problem in Theorem \ref{thm-GramAssociationScheme}, the following holds:
  
\begin{itemize}
	\item[(a)] The existence of an Hadamard matrix in $\mathcal{A}$ is equivalent to a solution with $\alpha_0=n$, and $\alpha_i=0$ for all $i>0$.
	\item[(b)] The existence of a Barba matrix in $\mathcal{A}$ is equivalent a solution with $\alpha_0=n$, and $\alpha_i=1$ for all $i>0$.
	\item[(c)] The existence of a Bordered Hadamard matrix, with core in $\mathcal{A}$, is equivalent to a solution with $\alpha_0=n$, and $\alpha_i=-1$ for all $i>0$.
\end{itemize}
 
 All these are systems of quadratic equations, and can be studied with a technique known as \textit{Gröbner bases}, see the book by Cox, Little and O'Shea \cite{Cox-Little-OShea-IdealsVarietiesAlgorithms}. The only fact that we will need is that Gröbner bases can be used to find the \textit{primary decomposition} of an ideal in the polynomial ring $\Q[x_1,\dots,x_n]$.

\section{Primary ideal decompositions} 
 Here we introduce some notions from commutative algebra, and summary of useful results, a good reference is Chapters 4 and 7 of \cite{Atiyah-Macdonald}. All results in this section are well-known material in commutative algebra and basic algebraic geometry.
 \begin{definition} \normalfont \label{def-PrimaryIdeal} An ideal $\q$ in a (commutative) ring $R$ is called \textit{primary} if and only if $ab\in \q$ implies that $a\in \q$ or $b^n \in\q$ for some integer $n\geq 1$.
 \end{definition}
 The \textit{radical} of an ideal $I$ in a ring $R$ is the set
 \[\sqrt{I}=\{x\in R: x^n\in I, \text{ for some integer } n\geq 1\}.\]
 It is easy to check the following well-known fact:
 \begin{lemma} If $I$ is an ideal in $R$, then $\sqrt{I}$ is also an ideal in $R$.
 \end{lemma}
 \begin{proof}
 If $x,y\in\sqrt{I}$, then there are integers $n,m\geq 1$ such that $x^n\in I$ and $y^m\in I$. By the binomial theorem
 \[(x+y)^{n+m}=\sum_{i=0}^{n+m}{n+m\choose i}x^i y^{n+m-i}.\]
 For every $i=0,\dots,n+m$, we have that if $i<n$, then $n+m-i>m$ so either $x^i\in I$ or $y^{n+m-i}\in I$. Since $I$ is an ideal this implies that $x^iy^{n+m-i}\in I$, and $(x+y)^{n+m}\in I$. By definition of the radical, $x+y\in \sqrt{I}$. Finally, given an arbitrary element $r\in R$, and $x\in \sqrt{I}$ we have $x^n\in I$ for some integer $n\geq 1$, so $(rx)^n=r^nx^n\in I$, which implies $rx\in\sqrt{I}$.\qedhere
 \end{proof}
 \begin{proposition}[cf. Proposition 4.1. \cite{Atiyah-Macdonald}]\normalfont \label{prop-RadicalPrime}The radical $\mathfrak{p}=\sqrt{\q}$ of a primary ideal $\q$ is the smallest prime ideal containing $\q$.
 \end{proposition}

\begin{definition}\normalfont \label{def-PrimaryDecomposition} Let $I$ be an ideal in $R$. A \textit{primary ideal decomposition} of $I$ is an expression of $I$ as  a finite intersection of primary ideals in $R$, for example
\[I=\bigcap_{i=1}^r \q_i,\]
where each $\q_i$ is primary.
\end{definition}
 
 \begin{definition}\normalfont A ring $R$ is called \textit{Noetherian} if and only if every ascending chain of ideals 
 \[I_0\subseteq I_1\subseteq \dots\subseteq I_n\subseteq I_{n+1}\subset\dots\]
is stationary, i.e. there is an integer $m\geq 0$ such that $I_{n+m}=I_{m}$ for all $n$.
 \end{definition}
 \begin{theorem}[Hilbert's basis theorem, Theorem 7.5. \cite{Atiyah-Macdonald}]\label{thm-HilbertBasis}
 Let $R$ be a Noetherian ring, then the polynomial ring $R[x]$ is Noetherian.
 \end{theorem}
 \begin{corollary}\normalfont If $K$ is a field, then  $K[x_1,\dots,x_n]$ is Noetherian.
 \end{corollary}
 \begin{proof}
 Since $K$ is a field, then it is Noetherian as its only ideals are the zero ideal $(0)$ and $(1)=K$. Therefore, by Hilbert's basis theorem, $K[x_1]$ is Noetherian. Using an induction argument, we can show that $K[x_1,\dots,x_n]$ is Noetherian.\qedhere
 \end{proof}
 
 \begin{theorem}[Lasker-Noether, Theorem 7.13 \cite{Atiyah-Macdonald}] In a Noetherian ring $A$, every ideal has a primary ideal decomposition.
 \end{theorem} 
 
 In particular, we have that every ideal in the ring $\Q[x_1,\dots,x_n]$ has a primary ideal decomposition. These decompositions are useful because primary ideals allow us to identify and parametrise solutions to a system of equations. See Chapter 4, Section 7 of \cite{Cox-Little-OShea-IdealsVarietiesAlgorithms}, for more information on Gröbner bases and primary ideal decompositions.\\
 
 Given an ideal $I$ in $K[x_1,\dots,x_n]$, we define the \textit{zero set} of $I$ as the set
 \[V(I)=\{x\in K^n: f(x)=0\text{ for all } f\in I\}.\]
 \begin{proposition}[cf. Chapter 4 \cite{Cox-Little-OShea-IdealsVarietiesAlgorithms}]\normalfont \label{prop-ZeroSetProps} Let $I_1$ and $I_2$ be ideals in $K[x_1,\dots,x_n]$, then
 \begin{itemize}
 \item[(a)] If $I_1\subseteq I_2$ then $V(I_1)\supseteq V(I_2)$,
 \item[(b)] $V(I_1I_2)=V(I_1)\cup V(I_2)$,
 \item[(c)] $V(I_1\cap I_2)=V(I_1)\cup V(I_2)$.
 \end{itemize}
 \end{proposition}
 \begin{proof}
 Part (a) follows easily from the definition: let $x\in V(I_2)$, then $f(x)=0$ for all $f\in I_2\supset I_1$, in particular $f(x)=0$ for all $f\in I_1$, so $x\in V(I_1)$. 
 To prove (b), let $x\in V(I_1I_2)$, then $f(x)g(x)=0$ for all $f\in I_1$, and $g\in I_2$. Therefore, $f(x)=0$ or $g(x)=0$ so in either case $x\in V(I_1)\cup V(I_2)$. Conversely, if $x\in V(I_1)\cup V(I_2)$, then either $f(x)=0$ for all $f\in I_1$ or $g(x)=0$ for all $g\in I_2$, in either case $f(x)g(x)=0$ for all $f\in I_1$, and $g\in I_2$. Thus, $x\in V(I_1I_2)$. 
 Finally, to prove (c), let $x\in V(I_1)\cup V(I_2)$, then $x\in V(I_1)$ or $x\in V(I_2)$. Without loss of generality, $x\in V(I_1)$, then $f(x)=0$ for all $f\in I_1$, so in particular we have $f(x)=0$ for all $f\in I_1\cap I_2$. This shows $V(I_1)\cup V(I_2)\subseteq V(I_1\cap I_2)$. On the other hand, $I_1I_2\subseteq I_1\cap I_2$, so $V(I_1\cap I_2)\subseteq V(I_1I_2)=V(I_1)\cup V(I_2)$ by part (b).\qedhere
 \end{proof}
 
 The zero set of an ideal $I\in F[x_1,\dots,x_n]$ is also known as the \textit{affine} algebraic variety of the ideal $I$, and it is a subset of the vector space $F^n$. The \textit{Zariski topology} is the topology in $F^n$ whose closed sets are given by the affine algebraic sets. For a set $S\subset F^n$ let
 \[I(S)=\{f\in F[x_1,\dots,x_n]: f(x)=0,\text{ for all } x\in S\}.\]
 It is easy to check that $I(S)$ is an ideal. The relationship between affine algebraic varieties and ideals is given by the following theorem.
 \begin{theorem}[Hilbert's Nullstellensatz, cf. Chapter 4, Theorem 2 \cite{Cox-Little-OShea-IdealsVarietiesAlgorithms}] Let $F$ be an algebraically closed field, then for any ideal $I\subset F[x_1,\dots,x_n]$,
 \[I(V(I))=\sqrt{I}.\]
 \end{theorem}
 On the other hand, we have for an arbitrary field $F$ and $S\subset F^n$ that
 \[V(I(S))=\overline{S},\]
 where $\overline{S}$ is the closure of the set $S$ under the Zariski topology in $F^n$. This implies that, over an algebraically closed field, affine algebraic varieties are in one to one correspondence with radical ideals. An affine algebraic variety $V$ is called \textit{reducible} if and only if there exist two proper subsets $A,B\subsetneq V$ such that both $A$ and $B$ are affine algebraic varieties, and $V=A\cup B$. If $V$ is not reducible, then it is called \textit{irreducible}.
 \begin{proposition} \normalfont If $F$ is an algebraically closed field, then any affine algebraic variety $V\subseteq F$ has a decomposition
 \[V=V_1\cup\dots V_r,\]
 into finitely many irreducible components.
 \end{proposition}
 \begin{proof}
 Let $V=V(I)$. By the Lasker-Noether theorem, $I$ has a decomposition into finitely many primary ideals
 \[I=\q_1\cap \dots \cap \q_r.\]
 Taking zero sets, we find by Proposition \ref{prop-ZeroSetProps} (c) that,
 \[V=V(\q_1)\cup \dots \cup V(\q_r).\]
 We have that $I(V(\q_k))=\sqrt{\q_k}$, so $V(\sqrt{\q_k})=\overline{V(\q_k)}=V(\q_k)$. From Proposition \ref{prop-RadicalPrime}, the ideal $\sqrt{\q_k}$ is prime, and this implies that $V(\sqrt{\q_k})=V(\q_k)$ is irreducible.\qedhere
 \end{proof}
 
 Therefore, the primary ideal decomposition can be essentially interpreted as a decomposition of the variety defined by the ideal into irreducible components. Although the situation may be more subtle in non-algebraically closed fields.
 
 \begin{example}\normalfont
 We parametrise solutions to the system of matrix equations
 \[v^*\begin{bmatrix}
 0 & 1\\
 1 & 2
 \end{bmatrix}v=n,\text{ and } v^*\begin{bmatrix}
 1 & 1 \\
 1 & 0
 \end{bmatrix}v=1,
 \]
 where $v=(v_1,v_2)$ has entries of modulus $1$. First, we let $v_1=x+iy$, and $v_2=z+it$, and we compute 
 \begin{align*}
&v^*\begin{bmatrix}
 0 & 1\\
 1 & 2
 \end{bmatrix}v=2xz + 2yt + 2z^2 + 2t^2,\text{ and }\\
&v^*\begin{bmatrix}
 1 & 1 \\
 1 & 0
 \end{bmatrix}v=x^2 + 2xz + y^2 + 2yt.
 \end{align*} 
 We have the additional conditions $|x+iy|=x^2+y^2=1$, and $|z+it|=z^2+t^2=1$. The system of equations defines the following ideal
 \[I=\langle 2xz+2yt+2z^2+2t^2-n,x^2 + 2xz + y^2 + 2yt-1,x^2+y^2-1,z^2+t^2-1\rangle.\]
 Using the \texttt{MAGMA} computer algebra system \cite{MAGMA}, or a similar tool, we can find a primary decomposition of the $I$. We have,
 \[I=\q_1\cup\q_2=\langle x + t,
        y - z,
        z^2 + t^2 - 1,
        n - 2\rangle
       \cap\langle
       x - t,
        y + z,
        z^2 + t^2 - 1,
        n - 2
        \rangle.
\]
A solution to the system of equations is a point $(x,y,z,t,n)$ in the zero set $V(I)$ of $I$. From Proposition \ref{prop-ZeroSetProps}, it follows that
\[V(I)=V(\q_1)\cup V(\q_2).\]
So in any case we find that $n$ must be equal to $2$, and $(x,y)=(\mp t, \pm z)$, where $\phi = t+iz$ is an arbitrary complex number with $|\phi|=1$. Hence, we find that there is a solution if and only if $n=2$, in which case $v=(\mp\overline{\phi},\phi)$ is a uniparametric family of solutions.
 \end{example}

 \section{Maximal determinants on 2-class association schemes}
  Hadamard matrices belonging to the Bose-Mesner algebra of a strongly regular graph have been completely classified by Chan in \cite{Chan-ComplexHadamard}, see also the related work \cite{Chan-Godsil,Chan-Hosoya}. Ikuta and Munemasa classified the bordered Hadamard matrices in strongly regular graphs in their paper \cite{Ikuta-Munemasa-Bordered}, see also their work \cite{Ikuta-Munemasa-BoseMesner,Ikuta-Munemasa-Galois,Ikuta-Munemasa-Nonsymmetric}. Their arguments involve an analysis of the eigenvalues of the strongly regular graph.\\
 
 Recall that the adjacency matrix $A$ of a strongly regular graph, Definition \ref{def-SRG}, satisfies 
 \[A^2=kI_v+\lambda A+\mu (J_v-I_v-A).\]
 Furthermore, we have the following relations between parameters.
 \begin{proposition}[cf. Chapter 1 \cite{Brouwer-VanMaldeghem}]\normalfont \label{prop-SRGRelations} Let $(v,k,\lambda,\mu)$ be the parameters of a strongly regular graph, and let $r>s$ be its restricted eigenvalues, i.e. those eigenvalues with eigenvectors orthogonal to $\mathbf{1}_v$. Then,
 \begin{itemize}
 \item[(i)] $(v-k-1)\mu = k(k-\lambda-1)$,
 \item[(ii)] $\lambda=\mu+r+s$, $k-\mu=rs$,
 \item[(iii)] $(k-r)(k-s)=\mu v$.
 \end{itemize}
 \end{proposition}
This gives the following:
\begin{lemma}\normalfont \label{lemma-SRGSIM} For a strongly regular graph with parameters $(v,k,\lambda,\mu)$, the symmetric intersection matrices are given by
\begin{align*}
P_0=
\begin{bmatrix}
1 & 0 & 0\\
0 & k & 0\\
0 & 0 & v-(k+1)
\end{bmatrix},\ 
&P_1=
\begin{bmatrix}
0 & 1 & 0\\
1 & \lambda & k-(\lambda+1)\\
0 & k-(\lambda+1) & v-2k+\lambda
\end{bmatrix},\text{ and }\\
&P_2=\begin{bmatrix}
0 & 0 & 1\\
0 & \mu & k-\mu\\
1 & k-\mu & v-2(k+1)+\mu
\end{bmatrix}.
\end{align*}
\end{lemma} 
\begin{proof}
The Bose-Mesner algebra of a strongly regular graph is spanned by $\{I_v,A,J_v-I_v-A\}$, so it is enough to compute $A(J_v-I_v-A)$, and $(J_v-I_v-A)^2$. On the one hand,
\begin{align*}
A(J_v-I_v-A)&=kJ_v-A-A^2\\
&=kJ_v-A-(kI_v+\lambda A+\mu (J_v-I_v-A))\\
&=(\mu-k)I_v+(\mu-(\lambda+1))A+(k-\mu)J_v\\
&=(k-(\lambda+1))A+(k-\mu)(J_v-I_v-A).
\end{align*}
On the other hand,
\begin{align*}
(J_v-I_v-A)^2&=vJ_v+I_v+A^2-2(k+1)J_v+2A\\
&=(k+1)I_v+(\lambda+2)A+\mu(J_v-I_v-A)+(v-2(k+1))J_v\\
&=(k+1-\mu)I_v+(\lambda-\mu+2)A+(v-2(k+1)+\mu)J_v\\
&=(v-(k+1))I_v+(v-2k+\lambda)A+(v-2(k+1)+\mu)(J_v-I_v-A).
\end{align*}
The result follows from the definition of symmetric intersection matrices.\qedhere
\end{proof}

With Lemma \ref{lemma-SRGSIM} we can classify several types of matrices in symmetric $2$-class association schemes. We note that our method can reproduce some of the results of Chan \cite{Chan-ComplexHadamard}, and Ikuta and Munemasa \cite{Ikuta-Munemasa-Bordered}, but fails to give a complete classification due to the high complexity of some of the primary ideals in the decompositions. For example, classifying Hadamard matrices over a general strongly regular graph as done in \cite{Chan-ComplexHadamard}, is infeasible with our method. However,  using Theorem \ref{thm-GramAssociationScheme} is very effective for searching for matrices in individual association schemes, and in given families of association schemes with a fixed number of classes. Another advantage  of our method is that we have complete control over the entries $\alpha$ and $\beta$ of the matrix 
\[I+\alpha A+\beta(J-I-A).\]
This allows us to easily classify matrices with prescribed entries. So this provides a complementary tool to the previous analyses done in the literature. All the computations that follow have been carried out in \texttt{MAGMA}, \cite{MAGMA}.\\

The results in Theorem \ref{thm-Barba3SRG} can be easily recovered using Gröbner bases. The following is another example application to matrices with prescribed entries:
\begin{theorem}\label{thm-Barba4SRG} There are no Barba matrices with entries in $\{\pm 1,\pm i\}$ in the Bose-Mesner algebra of a strongly regular graph, provided that at least one entry is non-real.
\end{theorem}
\begin{proof}
If $B$ is a Barba matrix then $\overline{B}$ is also a Barba matrix. Therefore, up to taking the complement of the graph, it is enough to consider the cases
\begin{itemize}
\item[(i)] $B=I_v+iA-i(J_v-I_v-A)$, and
\item[(ii)] $B=I_v-A+i(J_v-I_v-A)$.
\end{itemize}
To study case (i): Let $I$ be the ideal in $\Q[v,k,\lambda,\mu]$ generated by the basic relation $(v-k-1)\mu=k(k-\lambda-1)$, between parameters of strongly regular graphs, and the polynomials defining the system of equations
\[x^*P_0x = v,\ x^*P_1x=x^*P_2x=1,\]
where the matrices $P_i$ are the symmetric intersection matrices of Lemma \ref{lemma-SRGSIM}, and $x=(1,i,-i)^{\intercal}$. Using Gröbner bases we find that $I$ is primary, and given by generators as
\[I=
\langle
v + 4\mu - 4k - 3,
    \lambda - \mu + 1,
    \mu^2 - (k + 1/2)\mu + k^2/4
\rangle.
\]
From the condition $\mu^2-(k+1/2)\mu +k^2/4$ we find that
\[\mu = \frac{k+1/2\pm \sqrt{k+1/4}}{2}.\]
The radical $\sqrt{k+1/4}$ must be a rational square. Let $a$ and $b$ be coprime integers such that $k+1/4=(a/b)^2$, then
\[4k+1=(2a/b)^2.\]
Since $4k+1$ is an integer, we must have either $b=1$ or $b=2$. In any case, we find that $4k+1=t^2$ for some integer $t$. Adding this condition to the ideal $I$ in $\Q[v,k,\lambda,\mu,t]$, we find that $I$ is primary and has generators
\begin{align*}
I&=
\langle
v -(t+1)^2/2-1,
        \lambda - (t-1)^2/8+1,
        \mu - (t-1)^2/8,
        k - (t^2-1)/4
\rangle\\
&\cap
\langle
v -(t-1)^2/2-1,
        \lambda - (t+1)^2/8+1,
        \mu - (t+1)^2/8,
        k - (t^2-1)/4
\rangle
\end{align*}
The corresponding putative parameters $(v_+,k,\lambda_+,\mu_+)=((t+1)^2/2+1,(t^2-1)/4,(t-1)^2/8-1,(t-1)^2/8)$ and $(v_-,k,\lambda_-,\mu_-)=((t-1)^2/2+1,(t^2-1)/4,(t+1)^2/8-1,(t+1)^2/8)$ are complementary whenever the orders $v_+$ and $v_-$ coincide. Hence, we may consider only $(v,k,\lambda,\mu):=(v_+,k,\lambda_+,\mu_+)$. In this case we have that $8$ must divide $(t-1)^2$, and hence $t\equiv 1\pmod{4}$. The eigenvalues $r$ and $s$ of a strongly regular graph with such parameters are 
\[r=-1/2+\sqrt{\frac{1}{8}(t^2-2t-1)},\text{ and } s=-1/2-\sqrt{\frac{1}{8}(t^2-2t-1)}.\]
Let $f$ and $g$ be the multiplicities of $r$ and $s$, respectively. If $f=g$, then $f=g=(v-1)/2$ and
\[0=k+fr+gs=k+\frac{v-1}{2}(r+s)=k-\frac{v-1}{2},\]
which is a contradiction as $k\neq (v-1)/2$ for our parameters. Hence, $f\neq g$ and this implies that $r,s\in\Z$, see 1.1.4 in \cite{Brouwer-VanMaldeghem}. However, this is impossible: assume that $r$ is an integer, then $r+1/2=\frac{1}{2}\sqrt{(t^2-2t-1)/2}$, hence $\sqrt{(t^2-2t-1)/2}=2r+1\in\Z$. Since $t\equiv 1\pmod{4}$ we may let $t=4a+1$ for some integer $a$, and we find that
\[(2r+1)^2=\frac{t^2-2t-1}{2}=8a^2-1.\]
This is a contradiction, since $8a^2-1\equiv 3\pmod{4}$ and $(2r+1)^2\equiv 1\pmod{4}$. This shows that there are no Barba matrices of the type in case (i).\\
To study case (ii) we construct an ideal $I$ in $\Q[v,k,\lambda,r,s]$ analogous to the one above, and we include additionally the conditions $r^2+(\mu-\lambda)r+(\mu-k)=0$, together with $r+s=-(\mu-\lambda)$ and $rs=\mu-k$, so that the variables $r$ and $s$ correspond to the eigenvalues of the strongly regular graph. Using Gröbner bases we find that $I$ is primary, and that
\[I=\langle
v - 2s^2 - 3,
        \lambda - k + s^2,
        \mu - k + s^2,
        k^2 - (2s^2 + 3)k + 2s^4 + 2s^2,
        r + s
\rangle .
\]
Since $s=-r$, we must have that $f\neq g$, otherwise $k+rf+sg=0$ would imply that $k=0$. This implies by 1.1.4 in \cite{Brouwer-VanMaldeghem}, that $r$ and $s$ are integers. From the condition
\[k^2 - (2s^2 + 3)k + 2s^4 + 2s^2=0,\]
we find that the only integer value of $s<0$ that makes $k$ real and positive is $s=-1$. In which case we find that the parameters of the strongly regular graph must be $(v,k,\lambda,\mu)=(5,4,3,3)$, but this would correspond to the complete graph $K_5$ which is not a strongly regular graph by definition.
\qedhere
\end{proof}

\begin{proposition}[cf. Ikuta and Munemasa \cite{Ikuta-Munemasa-Bordered}, and \cite{Singh-Dubey}]\normalfont \label{prop-RealBorderedSRG} Let $\{I_v,A,J_v-I_v-A\}$ be the adjacency matrices of a strongly regular graph $G$ of parameters $(v,k,\lambda,\mu)$. Then, 
\[I_v-A+(J_v-I_v-A),\]
is the core of a bordered (real) Hadamard matrix if and only if 
\[(v,k,\lambda,\mu)=(4r^2-1,2r^2,r^2,r^2),\]
where $r$ is the largest restricted eigenvalue of $G$.
\end{proposition}
\begin{proof}
Let $I$ be the ideal in $\Q[v,k,\lambda,\mu,r,s]$ generated by the relations of Proposition \ref{prop-SRGRelations}, and the equations
\[x^*P_0x = v,\ x^*P_1x=x^*P_2x=-1,\]
where the matrices $P_i$ are the symmetric intersection matrices of Lemma \ref{lemma-SRGSIM}, and $x=(1,-1,1)^{\intercal}$. Using Gröbner bases we find that the ideal $I$ is primary, and can be expressed with the generators
\[I=\langle
  v - 4r^2 + 1,
  \lambda - r^2,
  \mu - r^2,
  k - 2r^2,
  r + s
\rangle . \]
The result follows immediately.\qedhere
\end{proof}

The family of \textit{conference graphs} contains several interesting matrices in their Bose-Mesner algebra.

\begin{definition}\normalfont A \textit{conference graph} is a strongly regular graph with parameters 
\[(v,k,\lambda,\mu)=(v,(v-1)/2,(v-5)/4,(v-1)/4).\]
\end{definition}
For example, Paley graphs are a subfamily of conference graphs.

\begin{proposition}\normalfont Let $\{I,A,J-I-A\}$ be the adjacency matrices of a conference graph of order $v$. Let 
\[M=I+\alpha A+\beta(J-I-A).\]
Then,
\begin{itemize}
\item[(i)]  $M$ is the core of a bordered Hadamard matrix if and only if $\alpha=\pm i$ and $\beta=\mp i$ or $\alpha=\overline{\beta}$ has the minimal polynomial 
\[p(x)=x^2+\frac{2}{t}x+1,\]
where $t=k=(v-1)/2$, (cf. Ikuta and Munemasa \cite{Ikuta-Munemasa-Bordered}).
\item[(ii)] $M$ is a Barba matrix if and only if 
\[\alpha=\frac{-1\pm i\sqrt{t^2-1}}{t},\]
and $\beta=\overline{\alpha}$, where $t^2+(t+1)^2=v$.
\item[(iii)] $M$ is an Hadamard matrix if and only if 
\[\alpha=\frac{-1\pm i\sqrt{t^2-1}}{t},\]
and $\beta=\overline{\alpha}$, where $(t+1)^2=v$.
\end{itemize}
\end{proposition}
\begin{proof}
For part (i) we construct the ideal $I$ generated by the relations in Proposition \ref{prop-SRGRelations}, and the equations of Theorem \ref{thm-GramAssociationScheme} with $\alpha_0=v$, and $\alpha_1=\alpha_2=-1$. We split the analysis into two cases: Let $\alpha=x_0+ix_1$ and $\beta=y_0+iy_1$, we will consider the case where $x_0=0$, and the case where $x_0\neq 0$ separately. When $x_0=0$, we consider the ideal $I+\langle x_0\rangle$, and using Gröbner bases we find the decomposition
\begin{align*}
I+\langle x_0\rangle&=\langle
x_0,
        x_1 + 1,
        y_0,
        y_1 - 1,
        v - 2k - 1,
        \lambda - k/2 + 1,
        \mu - k/2
        \rangle\\
&\cap \langle
x_0,
        x_1 - 1,
        y_0,
        y_1 + 1,
        v - 2k - 1,
        \lambda - k/2 + 1,
        \mu - k/2
\rangle\\
&\cap
\langle
   x_0,
        x_1 + 1,
        y_0,
        y_1 + 1,
        v - 1,
        \lambda + 1,
        \mu,
        k
\rangle\\
&\cap\langle
 x_0,
        x_1 - 1,
        y_0,
        y_1 - 1,
        v - 1,
        \lambda + 1,
        \mu,
        k
\rangle
\end{align*}
The last two primary factors give infeasible values for the parameters $\lambda$ and $k$, so they yield no matrices. The first two primary ideals correspond to the matrices
\[I\pm i A\mp i(J-I-A).\]
Now we consider $x_0\neq 0$, a way to incorporate this condition is by introducing a new variable $t$, and adding the relation $tx_0+1$ to $I$. Gröbner bases yield the primary decomposition:
\begin{align*}
I+\langle tx_0+1\rangle &=\langle
tx_0 +1,
        x_1 + y_1,
        y_0-x_0,
        y_1^2t^2 - t^2 + 1,
        v - 2t - 1,
        \lambda - t/2 + 1,
        \mu - t/2,
        k - t
        \rangle\\
        &\cap
        \langle
        x_0 + k,
        x_1 - y_1,
        y_0 + k,
        y_1^2 + k^2 - 1,
        v - 2k - 1,
        \lambda - k/2 + 1,
        \mu - k/2,
        kt - 1
\rangle
\end{align*}
The second primary factor yields no matrices, since we have the condition $y_1^2=-k^2+1$, hence $k=1$ and this implies $\lambda=-1/2$, which is impossible. From the first primary factor we find that $y_1^2=(t^2-1)/t^2$, and from the other relations we have
\[\alpha=\frac{-1\pm i\sqrt{t^2-1}}{t},\text{ and } \beta=\frac{-1\mp i\sqrt{t^2-1}}{t}.\]
From here it follows that the minimal polynomial of $\alpha$ and $\beta$ is 
\[x^2+\frac{2}{t}x +1.\]
From the condition $k-t=0$, we find the claimed minimal polynomial. For parts (ii) and (iii) we modify the ideals with $\alpha_1=\alpha_2=1$ and $\alpha_1=\alpha_2=0$ respectively. In each case, we find that there are no solutions with $x_0=0$. Adding the relation $tx_0+1$, we find the primary factors
\[
\langle tx_0 +1,
        x_1 + y_1,
        y_0 - x_0,
        y_1^2t^2 - t^2 + 1,
        v - (t^2+(t+1)^2),
        \lambda - t^2/2 - t/2 + 1,
        \mu - t^2/2 - t/2,
        k - t^2 - t
\rangle,
\]
and
\[
\langle tx_0 +1,
        x_1 + y_1,
        y_0 - x_0,
        y_1^2t^2 - t^2 + 1,
        v - (t+1)^2,
        \lambda - t^2/4 - t/2 + 1,
        \mu - t^2/4 - t/2,
        k - t^2/2 - t
\rangle,
\]
respectively. In either case, the expression of $\alpha$ and $\beta$ is the same as in case (i). The only difference is in the relationship between $t$ and $v$, which is $t^2+(t+1)^2=v$, and $(t+1)^2=v$ respectively.\qedhere

\end{proof}

\begin{example}
\normalfont We give some concrete examples of the Hadamard and Barba matrices above. When $t=1$, we find two degenerate examples of Barba matrices and Hadamard matrices, where $\alpha=\beta=-1$. In case (ii) we have that $v=1^2+2^2=5$, and in case (iii) we have $v=(1+1)^2=4$. These correspond to the Paley graphs of order $5$ and $4$ and give the circulant Barba matrix $B_5$ and the circulant Hadamard matrix $H_4$ below:
\[
B_5=
\begin{bmatrix}
1 & - & - & - &-\\
- & 1 & - & - &-\\
- & - & 1 & - &-\\
- & - & - & 1 &-\\
- & - & - & - &1
\end{bmatrix},
H_4=
\begin{bmatrix} 
1 & - & - & -\\
- & 1 & - & -\\
- & - & 1 & -\\
- & - & - & 1
\end{bmatrix}.
\]
When $t=2$, we find that $\alpha=\omega$, and $\beta=\omega^2$ for some primitive root of unity $\omega$. In case (ii) $v=2^2+3^2=13$, and in case (iii), $v=(2+1)^2=9$. These correspond to the Paley graphs of order $13$ and $9$ respectively. The first matrix is the Barba matrix of order $13$ in Appendix \ref{app-MatrixTables}, and the second one is the $\BH(9,3)$ matrix:
\[
H_9=\left[
\begin{array}{*{9}{c}}
0&1&1&1&2&2&1&2&2\\
1&0&2&2&2&1&1&2&1\\
1&2&0&1&1&1&2&2&2\\
1&2&1&0&2&2&2&1&1\\
2&2&1&2&0&1&1&1&2\\
2&1&1&2&1&0&2&2&1\\
1&1&2&2&1&2&0&1&2\\
2&2&2&1&1&2&1&0&1\\
2&1&2&1&2&1&2&1&0
\end{array}\right].
\]
\end{example}

Given the classification of complex Hadamard matrices in strongly regular graphs in \cite{Chan-ComplexHadamard}, it is natural to ask the following questions:
\begin{research-problem}\normalfont What is the maximal value of the determinant of a matrix with entries of absolute value $1$ in the Bose-Mesner algebra of an strongly regular graph of parameters $(v,k,\lambda,\mu)$?
\end{research-problem}

\begin{research-problem}\normalfont Classify complex Hadamard matrices, in a symmetric $3$-class association scheme.
\end{research-problem}

\subsection{Matrices in asymmetric 2-class association schemes}
 Asymmetric $2$-class association schemes are more rigid than their symmetric counterpart.

\begin{definition}\normalfont A \textit{tournament} of order $v$ is a directed graph obtained by assigning an orientation to each of the edges of the undirected complete graph $K_v$.\index{doubly regular tournament} A\textit{ doubly regular tournament} $T$ of order $v$ with parameters $(m_1,m_2)$, or $(m_1,m_2)$-DRT of order $v$, is a tournament of order $v$ satisfying
\begin{itemize}
\item[(i)] For every vertex $x$ of $T$, $\outdeg(x)=m_1$, and
\item[(ii)] For every pair of vertices $(x,y)$ with $x\neq y$, the number of vertices dominated by both $x$ and $y$ is $m_2$.
\end{itemize}
\end{definition}

\begin{example}\normalfont Let $q\equiv 3\pmod{4}$ be a prime power. Then the matrix $Q$ given by 
\[Q_{xy}=
\begin{cases}
+1 & \text{ if } x-y \text{ is a nonzero square in }\F_q\\
0 & \text{otherwise}
\end{cases}
\]
is the $\{0,1\}$ adjacency matrix of a doubly regular tournament of order $q$.
\end{example}
\begin{remark}\normalfont Notice that above we consider an adjacency matrix with entries in $\{0,1\}$, as opposed to the more common $\pm 1$ adjacency matrices for tournaments.
\end{remark}
\begin{lemma}\normalfont\label{lemma-ParamsDRT} An $(m_1,m_2)$-DRT of order $v$ satisfies $v=2m_1+1$ and $m_1=2m_2+1$.
\end{lemma}
\begin{proof}
Let $\Gamma$ be an $(m_1,m_2)$-DRT of order $v$, with vertex set $V$ and edge set $E$. By definition we have that $\outdeg(x)=m_1$ for each vertex $x$ of $\Gamma$, and since $\Gamma$ is a tournament $\indeg(x)=v-1-m_1$ for all $x\in V$. By the Handshaking Lemma we have that $\sum_x \indeg(x)=\sum_x \outdeg(x)$, and so
\[v(v-1-m_1)=vm_1,\]
which implies that $v=2m_1+1$. Let $x$ be a fixed vertex of $\Gamma$, counting the number of elements of the set
\[\{(y,z): x\neq y,\text{ and }(x,z),(y,z)\in E\},\]
in two different ways we obtain,
\[(v-1)m_2=m_1(m_1-1).\]
Since $v=2m_1+1$ it follows that $m_1=2m_2+1$.
\end{proof}
\begin{proposition} \normalfont \label{prop-DRT2AsymEq} An $(m_1,m_2)$-DRT gives an asymmetric $2$-class association scheme with parameters $(v,k,\lambda,\mu)=(2m_1+1,m_1,m_1-m_2-1,m_1-m_2)$. Conversely a $(v,k,\lambda,\mu)$ asymmetric $2$-class association scheme is a $(v,k,k-\mu)$-DRT.
\end{proposition}
\begin{proof}
Let $A$ be the $\{0,1\}$ adjacency  matrix of an $(m_1,m_2)$-DRT. We show that $\{I_v,A,A^{\intercal}\}$ generates the Bose-Mesner algebra of an asymmetric $2$-class association scheme with parameters $(v,k,\lambda,\mu)=(2m_1+1,m_1,m_1-m_2-1,m_1-m_2)$. Since $A$ is the incidence matrix of a tournament, we have that $A^{\intercal}=J_v-I_v-A$, hence $I_v+A+A^{\intercal}=J_v$. By definition $AJ_v=m_1J_v$ and so the valency of $A$ is $k:=m_1$, similarly $A^{\intercal}J_v=(v-1-m_1)J_v$. Notice that $(AA^{\intercal})_{ij}=\sum_k\delta_{i\rightarrow k}\delta_{j\rightarrow k}$, where $\delta_{i\rightarrow j}$ takes the value $1$ if $(i,j)$ is a directed edge of the DRT, and $0$ otherwise.  Therefore by definition of DRT, $(AA^{\intercal})_{ii}=m_1$ and if $i\neq j$ then $(AA^{\intercal})=m_2$. Hence $AA^{\intercal}=m_1I_v+m_2(A+A^{\intercal})$, and
\begin{align*}
A^2=A(J_v-I_v-A^{\intercal})&=m_1J_v-A-(m_1I_v+m_2A+m_2A^{\intercal})\\
&=m_1(I_v+A+A^{\intercal})-A-m_1I_v-m_2A-m_2A^{\intercal}\\
&=(m_1-m_2-1)A+(m_1-m_2)A^{\intercal}.
\end{align*}
From Lemma \ref{lemma-ParamsDRT}, we know that $\indeg(x)=v-1-m_1=2m_1+1-m_1-1=m_1$ for every vertex $x$, which implies that $AJ_v=J_vA=m_1J_v$. Therefore
\[A^{\intercal}A=(J_v-I_v-A)A=A(J_v-I_v-A)=AA^{\intercal},\]
so commutativity holds. Conversely, let $A$ be the incidence matrix of an asymmetric $2$-class association scheme. Consider $A$ as the incidence matrix of a tournament, then by definition  the out-degree of every vertex is $m_1:=k$. The value $\outdeg(i,j)$ for any two vertices $i\neq j$ is given by the coefficient of $A^{\intercal}$ in the expression for $AA^{\intercal}$ in the basis $\{I_v,A,A^{\intercal}\}$, which is precisely $k-\mu$.
\end{proof}
\begin{corollary}\label{cor-Params2ClassAsym}\normalfont Every asymmetric $2$-class association scheme, with parameters $(v,k,\lambda,\mu)$ satisfies $v=4r+3$, $k=2r+1$, $\lambda=r$ and $\mu=r+1$, for some natural number $r$.
\end{corollary}
\begin{proof}
By Proposition \ref{prop-DRT2AsymEq}, we have that $(v,k,\lambda,\mu)=(2m_1+1,m_1,m_1-m_2-1,m_1-m_2)$, where $(m_1,m_2)$ are the parameters of a doubly regular tournament. Let $m_2=r$, then by Lemma \ref{lemma-ParamsDRT}, we have that $m_1=2r+1$, and $v=2m_1+1=4r+3$, $k=m_1=2r+1$, $\lambda=m_1-m_2-1=r$, and $\mu=r+1$.\qedhere
\end{proof}

\begin{lemma} For an asymmetric $2$-class associations scheme with parameters $(4r+3,2r+1,r,r+1)$, the symmetric intersection matrices are given by
\begin{align*}P_0=
\begin{bmatrix}
1 & 0 & 0\\
0 & 0 & 2r+1\\
0 & 2r+1 & 0
\end{bmatrix},\  &P_1=\begin{bmatrix}
0 & 1 & 0\\
1 & r & r\\
0 & r & r+1
\end{bmatrix}, \text{ and }\\
&P_2=\begin{bmatrix}
0 & 0 & 1\\
0 & r+1 & r\\
1 & r & r
\end{bmatrix}.
\end{align*}
\end{lemma}
\begin{proof}
The Bose-Mesner algebra of the scheme is generated by $\{I_v,A,A^{\intercal}\}$. Using Corollary \ref{cor-Params2ClassAsym}, suppose that the parameters of the scheme are $(v,k,\lambda,\mu)=(4r+3,2r+1,r,r+1)$. In the proof of Proposition \ref{prop-DRT2AsymEq}, we showed that
\[A^2=rA+(r+1)A^{\intercal}.\]
Therefore,
\begin{align*}
AA^{\intercal}&=A(J_v-I_v-A)\\
&=kJ_v-A-A^2\\
&=(2r+1)J_v-A-rA-(r+1)A^{\intercal}\\
&=(2r+1)(I_v+A+A^{\intercal})-(r+1)(A+A^{\intercal})\\
&=(2r+1)I_v+r(A+A^{\intercal}).
\end{align*}
Finally,
\begin{align*}
(A^{\intercal})^2=(A^2)^{\intercal}=(r+1)A + r A^{\intercal}.
\end{align*}
Using the definition of the symmetric intersection matrices, the result follows.\qedhere
\end{proof}

\begin{lemma}\normalfont \label{lemma-DRTSkewHadamard} If $\{I_v,A,A^{\intercal}\}$ are the adjacency matrix of an asymmetric $2$-class association scheme, then
\[C = I_v+A-A^{\intercal},\]
is the core of a bordered skew Hadamard matrix of order $v+1$.
\end{lemma}
\begin{proof}
Direct computation shows that
\begin{align*}
CC^{\intercal}&=(I_v+A-A^{\intercal})(I_v-A+A^{\intercal})\\
&=I_v+2AA^{\intercal}-A^2-(A^{\intercal})^2\\
&=I_v+2[(2r+1)I_v+r(A+A^{\intercal})]-(rA+(r+1)A^{\intercal})-((r+1)A+rA^{\intercal})\\
&=(4r+3)I_v+2r(A+A^{\intercal})-(2r+1)(A+A^{\intercal})\\
&=(4r+3)I_v-(J_v-I_v).
\end{align*}
Hence, $CC^{\intercal}=(v+1)I_v-J_v$. Furthermore $CJ=(I+A-A^{\intercal})J=(1+(2r+1)-(2r+1))J=J$.
Therefore, the matrix
\[H=
\begin{bmatrix}
1 & \mathbf{1}^{\intercal}\\
-\mathbf{1} & C
\end{bmatrix},
\]
is a real Hadamard matrix of order $v+1$. \qedhere
\end{proof}
In fact, Reid and Brown proved that doubly regular tournaments are equivalent to skew Hadamard matrices \cite{Reid-Brown-DRTHadamard}. Using Gröbner basis it is easy to show that there is a complex Hadamard matrix in every asymmetric $2$-class association scheme:
\begin{theorem}
Let $\mathcal{X}$ be an asymmetric $2$-class association scheme with parameters $(v,k,\lambda,\mu)=(4r+3,2r+1,r,r+1)$. Let $\{I,A,A^{\intercal}\}$ be the Schur idempotents of the Bose-Mesner Algebra of $\mathcal{X}$, then the matrix
\[H=I+\alpha A+ \beta A^{\intercal},\]
is a complex Hadamard matrix if and only if

\begin{itemize}
\item[(i)] One of $\alpha$ or $\beta$ has value $1$, and the other has minimal polynomial
\[p_r(t)=t^2+\frac{2r+1}{r+1}t+1.\]
\end{itemize}
\item[(ii)] $H=I_3+\omega (J_3-I_3)$, where $\omega$ is a primitive third root of unity.
\end{theorem}
\begin{proof}
We pose the problem of Theorem \ref{thm-GramAssociationScheme} with $\alpha_0=v=4r+3$, $\alpha_1=0$, and $\alpha_2=0$. From the system of equations given by
\[u^* WP_i u=\alpha_i,\]
where $u=(1,\alpha,\beta)^{\intercal}$, we extract an ideal $I$, whose zero set $V(I)$ is in one-to-one correspondence with the sought Hadamard matrices. Letting $\alpha=x_0+ix_1$, and $\beta=y_0+iy_1$, this ideal $I$ is given by the following generators:
\[I:
\begin{cases}
&2x_0^2r + x_0^2 + 2x_1^2r + x_1^2 + 2y_0^2r + y_0^2 + 2y_1^2r + y_1^2 - 4r 
        - 2,\\
&x_0^2r + 2x_0y_0r + x_0y_0 + x_0 + x_1^2r + 2x_1y_1r + x_1y_1 + y_0^2r + y_0 +
        y_1^2r,\\
&x_0y_1 - x_1y_0 + x_1 - y_1\\
&x_0^2 + x_1^2 - 1,\\
&y_0^2 + y_1^2 - 1.
\end{cases}
\]
The primary ideal decomposition of $I$ as an ideal in $\Q[x_0,x_1,y_1,y_2,r]$ is 
\begin{align*}
I&=\langle x_0 + 2r + 1/2,
        x_1 - y_1,
        y_0 + 2r + 1/2,
        y_1^2 + 4r^2 + 2r - 3/4\rangle\\
&\cap \langle
  x_0 + 2x_1^2(r + 1) - 1,
        x_1^2(r+1)^2 - r - 3/4,
        y_0 - 1,
        y_1
\rangle\\
&\cap
\langle
x_0 - 1,
        x_1,
        y_0 + 2y_1^2(r +1) - 1,
        y_1^2(r+1)^2 - r - 3/4\rangle.\\
&=\q_1\cap \q_2\cap \q_3.
\end{align*}
In $\q_1$ we find the condition $y_1^2+4r^2+2r-3/4=0$, which implies that $-4r^2-2r+3/4\geq 0$. Since $r$ is an integer this only occurs for the value $r=0$. Substituting, we find that $y_1^2=3/4$, so $y_1=\pm\sqrt{3}/2$. Also, $x_1=y_1$, and $x_0=y_0=-1/2$, hence 
\[\alpha=\beta=\frac{-1\pm i\sqrt{3}}{2},\]
so $\alpha$ and $\beta$ are both equal to a fixed primitive third root of unity $\omega$, and $I+\alpha A+\beta A^{\intercal}=I_3+\omega (J_3-I_3)$. In $\q_2$, the relation $x_1^2(r+1)^2 - r - 3/4$ implies that we must have $x_1^2(r+1)^2=r+3/4=(4r+3)/4$, hence
\[x_1=\pm\frac{\sqrt{4r+3}}{2(r+1)}.\]
The relation $x_0+2x_1^2(r+1)-1$ implies that
\[x_0=1-\frac{4r+3}{2(r+1)}=\frac{-2r-1}{2(r+1)}.\]
Therefore, the elements of the zero set $V(\q_2)$ are given as a uniparamateric family in terms of $r$ as:
\[\alpha=\frac{-(2r+1)\pm i\sqrt{4r+3}}{2(r+1)},\text{ and } \beta=1.\]
Then, the minimal polynomial of $\alpha$ is
\[p(x)=x^2+\frac{2r+1}{r+1}x+1.\]
In $V(\q_3)$ the roles of $x_i$ and $y_i$, and hence $\alpha$ and $\beta$, are exchanged.
\end{proof}

In particular, this result implies that there is always an Hadamard matrix in an asymmetric $2$-class association scheme. Therefore, the Hadamard bound can always be achieved by unimodular matrices in these schemes.

\begin{research-problem}\normalfont
Classify complex Hadamard matrices in asymmetric $3$-class association schemes. 
\end{research-problem}

\cleardoublepage
\chapter{User-Private Information Retrieval and Finite Geometry}\label{chap-UPIR}

This chapter is quite different in spirit from the others in this dissertation, and it is based on our paper \cite{UPIR-paper} in collaboration with Gnilke, Greferath, Hollanti, Ó Catháin, and Swartz.\\

Here we will study an application of finite geometries to user-private information retrieval (UPIR). The setting of UPIR consists of a network of users who wish to retrieve information from a databased stored in a server. UPIR provides means for the users to retrieve the information without revealing their identity to the server. The way this can be achieved is by having the users act as proxies of each other, i.e. requesting the information on their behalf. It is easy to show \cite{Swanson-Stinson-UPIR-I}, that if the proxies are chosen uniformly at random then privacy against the server is achieved. However, the identity of users can be compromised by a set of eavesdroppers within the network.\\

To motivate UPIR, we will begin with an introduction to one of its precursors: private information retrieval (PIR). A PIR scheme provides a mechanism by which a user can retrieve one bit $x_i$ of an $N$-bit database $x\in\{0,1\}^N$, modelled as a binary vector. We will discuss several shortcomings to PIR, the most important of all being that it requires cooperation from the server, in the sense that the server must act in compliance with a protocol designed to preserve the user's privacy. This additionally imposes restrictions on the ways that the server retrieves information from the database. These assumptions are often unrealistic, and UPIR instead provides a system that assumes nothing about the behaviour of the server, or the encoding of the database.\\

Previous UPIR schemes in the literature were based on projective planes and BIBDs.  We find that the condition that any pair of users can establish direct communication is a great vulnerability. And therefore, we propose schemes where this condition does not hold. We study schemes based on generalised quadrangles (GQs), and show that they provide a much higher level of privacy. To study GQs we will require some of the theory of quadratic forms. Hence, we assume that the reader is familiar with the results in Chapter \ref{chap-GramEquations} and Chapter \ref{chap-HermitianForms}, particularly with Section \ref{sec-QF} and Section \ref{sec-HF}.

\section{Private information retrieval}\label{sec-PIR}
Private information retrieval (PIR) was introduced by Chor, Goldreich, Kushilevitz and Sudan in \cite{CGKS-PIR}. The classical setting of PIR consists of\index{private information retrieval}
\begin{itemize}
\item[(i)] a set of $k$ \textit{servers} $\mathcal{S}=\{S_1,\dots,S_k\}$ storing a (replicated) database $x$, which is modelled as a binary string of a given length $n$, and
\item[(ii)] a user $\mathcal{U}$ who wishes to retrieve the $i$-th bit of information $x_i$ from $x$ without revealing the position $i$ to the servers.
\end{itemize}
A PIR scheme consists of a collection of algorithms (or protocols) that provide communication between the user $\mathcal{U}$ and the servers $\mathcal{S}$ in such a way that $\mathcal{U}$ can privately retrieve $x_i$, subject to certain assumptions on the way that the servers operate. To achieve this, the user $\mathcal{U}$ sends one randomised query to each of the $k$ servers in $\mathcal{S}=\{S_1,\dots,S_k\}$, and each server sends back a reply to $\mathcal{U}$. To retrieve the $i$-th bit $x_i$, the user uses a \textit{reconstruction function} that takes as input the $k$ responses and returns $x_i$. To define a PIR scheme more formally, we introduce some notation: Let $\{0,1\}^n$ be the set of binary strings of length $n$, and $\{0,1\}^*$ the set of finite binary strings, i.e. $\{0,1\}^*=\{e\}\cup \bigcup_{\ell=1}^{\infty}\{0,1\}^{\ell}$.

\begin{definition}[\cite{CGKS-PIR}]\normalfont Let $\ell_{r}$ and $\ell_q$ be positive integers. A $k$-server \textit{PIR scheme} for a database of length $n$ consists of a tuple $(\mathcal{Q},\mathcal{A},R)$, where 
\begin{itemize}
\item[(i)] $\mathcal{Q}$ is a set consisting of $k$ \textit{query functions}
\[Q_j:[n]\times\{0,1\}^{\ell_r}\rightarrow \{0,1\}^{\ell_q}\]
for $1\leq j\leq k$. These take an index $i\in[n]:=\{1,\dots,n\}$ and a random string $r\in\{0,1\}^{\ell_r}$, and produce a query $q\in\{0,1\}^{\ell_q}$ targetted to server $j$.
\item[(ii)] $\mathcal{A}$ is a set of $k$ \textit{answer functions}
\[A_j:\{0,1\}^n\times \{0,1\}^{\ell_q}\rightarrow\{0,1\}^{*}\]
for $1\leq j\leq k$. These take a database $x\in\{0,1\}^n$ and a query $q\in\{0,1\}^{\ell_q}$ and produce an answer $a$ with variable (finite) length depending on the query $q$ received by server $j$.
\item[(iii)] $R$ is a \textit{reconstruction function}
\[R:[n]\times \{0,1\}^{\ell_r}\times(\{0,1\}^{*})^k\rightarrow \{0,1\}.\]
\end{itemize}
These functions must satisfy the following pair of axioms:
\begin{itemize}
\item[]\textit{Correctness:} For every $x\in\{0,1\}^n$, $i\in [n]$ and $r\in \{0,1\}^{\ell_r}$
\[R(i;r;A_1(x,Q_1(i,r)),A_2(x,Q_2(i,r)),\dots,A_k(x,Q_k(i,r)))=x_i.\]
In other words, the reconstruction function retrieves the information $x_i$ from the replies of the servers for any given randomised queries for $i$.
\item[]\textit{Privacy:} For every $i,j\in[n]$, $1\leq s\leq k$ and $q\in \{0,1\}^{\ell_q}$
\[\p(Q_s(i,r)=q)=\p(Q_s(j,r)=q),\]
here the probability is taken over $r\in\{0,1\}^{\ell_r}$ uniformly distributed. In other words, a single server cannot infer the position $i$ from the randomised queries sent by the user.
\end{itemize}
\end{definition}

The sense in which $i$ is hidden above is \textit{information-theoretic}, meaning that each individual server obtains no information about the location of interest $i$ from its communications with $\mathcal{U}$.

\begin{example}[Trivial PIR scheme]\normalfont \label{ex-TrivialPIR} Let $\mathcal{S}=\{S\}$ consist of a single server. For a database of length $n$, the \textit{trivial PIR scheme} is the scheme where $\mathcal{U}$ requests the entire database from $S$. Formally we define:

\begin{itemize}
\item[(i)] For all $i\in [n]$ and $r\in\{0,1\}^{\ell_r}$
 let $Q(i,r):=Q_1(i,r)=\mathbf{1}_n$, where $\mathbf{1}_n$ is the all-ones binary string of length $n$.
 \item[(ii)] $A(x,q):=A_1(x,q)=[x_i: q_i\neq 0]$, i.e. $A$ answers with a binary string which consists all bits $x_i$ of $x$ where $q_i\neq 0$.
 \item[(iii)] For all $i\in[n]$, $r\in\{0,1\}^{\ell_r}$ and $a\in\{0,1\}^n$, define $R(i,r,a)=a_i$.
 \end{itemize}
 We have that $A(x,Q(i,r))=A(x,\mathbf{1}_n)=x$, therefore
 \[R(i,r,A(x,Q(i,r)))=R(i,r,x)=x_i,\]
 and the scheme is correct. Since $Q(i,r)=\mathbf{1}_n$ is independent of $r$ and $i$, the scheme is private.
\end{example}

The example above shows that PIR is possible, however the trivial scheme is far from practical. A real-world database may contain several terabytes of data, and $\mathcal{U}$ would have to download the entire database to retrieve a single bit privately. One of the main goals of PIR is to achieve privacy with a low \textit{communication complexity} (or \textit{communication overhead}). The communication complexity of a PIR scheme is defined as the total number of bits transferred between $\mathcal{U}$ and the servers in $\mathcal{S}$ during the execution of the protocol. The communication complexity for a PIR scheme is thus computed as 
\[\sum_{j=1}^k \ell(q_j)+\ell(a_j),\]
where $\ell(q_j)$ is the length of the query sent to server $S_j$ and $\ell(a_j)$ is the length of the answer sent by server $S_j$. For example, the trivial PIR scheme has a communication complexity of $2n$ bits, which is of asymptotic order $\Theta(n)$. The following theorem shows that this is the best possible communication complexity if there is only one database in the PIR scheme.
\begin{theorem}[Section 5.1, \cite{CGKS-PIR}] A single-database PIR scheme has communication complexity $\Omega(n)$.
\end{theorem}
This implies that in order to achieve PIR with sublinear communication complexity, two or more servers are required. An additional assumption of \textit{non-collusion} is typically imposed, namely it is assumed that no pair of servers will exchange information with the purpose of violating the privacy of $\mathcal{U}$.\\

In what follows we will assume that the queries are based on \textit{linear summations} (or \texttt{xor} queries). Namely, we interpret our queries and database $q,x\in\{0,1\}^{n}$ as elements of the vector space $\F_2^n$. The answer function of each of our servers $S_j$ is
\[A_j(x,q)=x\cdot q=\sum_{i=1}x_iq_i,\]
where the product and summation are interpreted in $\F_2$. In the following PIR schemes, we depart from the rather cumbersome formal description of PIR scheme as we did in Example \ref{ex-TrivialPIR}, and instead leave the details of this formalisation to the interested reader.

\begin{example}[Toy example for $2$-database PIR] \normalfont Suppose two servers $S_1$ and $S_2$ replicate the same $n$-bit database $x$, and that user $\mathcal{U}$ wishes to retrieve $x_i$. The PIR scheme proceeds as follows:
\begin{itemize}
\item[(i)] Let $\mathcal{U}$ choose a vector $q\in \F_2^n$ uniformly at random.
\item[(ii)] $\mathcal{U}$ sends $q$ to $S_1$ and $q+e_i$ to $S_2$, where $e_i$ is the $i$-th standard basis vector in $\F_2^n$.
\item[(iii)] $S_1$ replies with the bit $x\cdot q$ and $S_2$ replies with the bit $x\cdot (q+e_i)=x\cdot q+a_i$. 
\item[(iv)]$\mathcal{U}$ retrieves $x_i$ by adding both replies, i.e.
\[x\cdot q + (x\cdot q+x_i)=x_i.\]
Note that this equation is valid since the summation occurs in $\F_2$.
\end{itemize}

\end{example}
The communication complexity of the above scheme is also $\Theta(n)$, however this toy example can serve as the basis of more efficient schemes.

\begin{example}[Sublinear $2$-database PIR]\normalfont \label{ex-Sublinear2DBPIR} Suppose that we have a database $x$ of size $mn$ replicated on two servers $S_1$ and $S_2$. Interpret the database $x$ as an $m\times n$ matrix. Suppose that user $\mathcal{U}$ wants to retrieve the bit in the $(j,i)$ position of this matrix. The PIR scheme proceeds as follows:
\begin{itemize}
\item[(i)] Let $\mathcal{U}$ choose $q\in\F_2^n$ uniformly at random.
\item[(ii)] $\mathcal{U}$ sends $q$ to $S_1$ and $q+e_i$ to $S_2$.
\item[(iii)] For each row $r_{\ell}$ of $x$, $1\leq \ell\leq m$, server $S_1$ computes
\[a_\ell=r_\ell\cdot q,\]
and sends $a=[a_1,a_2,\dots,a_m]$ to $\mathcal{U}$. On the other hand $S_2$ computes $b_{\ell}=r_{\ell}\cdot(q+e_i)=r_{\ell}\cdot q+x_{\ell i}=a_{\ell}+x_{\ell i}$, and sends $b=[b_1,b_2,\dots,b_m]$ to $\mathcal{U}$.
\item[(iv)] $\mathcal{U}$ retrieves $x_{ji}$ by adding the responses $a_j$ and $b_j$
\[x_{ji}=a_j+b_j=a_j+(a_j+x_{ji}).\]
\end{itemize}
Therefore, the PIR scheme is correct, and by the non-collusion hypothesis it is also private since the queries sent to both $S_1$ and $S_2$ are distributed uniformly in the space $\{0,1\}^n$ of possible queries. 
The total communication complexity is of $2\cdot(n+m)$ bits. In particular for a database of size $n$ regarded as a $\sqrt{n}\times \sqrt{n}$ matrix, we find a communication complexity of $2\sqrt{n}=\Theta(\sqrt{n})$.
\end{example}
Example \ref{ex-Sublinear2DBPIR} is a particular instance of a more general result presented in Section 3.4 of \cite{CGKS-PIR}. The authors show that for a PIR scheme $\mathcal{P}$ where the user sends $p(n,k)$ bits of information to the servers and the total information received from the servers is $s(n,k)$ bits; given a database of size $nm$ one can apply $\mathcal{P}$ by rows to obtain a scheme of communication complexity
\[p(n,k)+ms(n,k).\]
This idea can be extended to $t$-fold tensors, recall that a $t$-fold tensor is an object of the type
\[\sum_{i_1,\dots,i_t=1}^nx_{i_1,\dots,i_t}(e_{i_1}\otimes\dots\otimes e_{i_t}),\]
where $e_{i}$ is the $i$-th basis vector in $\F_2^n$. In particular, matrices are in bijection with $2$-tensors
$\sum_{i,j=1}^nx_{ij}(e_i\otimes e_j)$, and higher order tensors can be interpreted as multi-dimensional matrices.
 With a variation of the method in Example \ref{ex-Sublinear2DBPIR} in the multi-dimensional case one can obtain
\begin{theorem}[Section 3.2 \cite{CGKS-PIR}]\label{thm-HyperCubePIR} Let $k=2^t$ where $t>0$ is an integer, then there is a $k$-database PIR scheme with communication complexity $\Theta(ktn^{1/t})=\Theta(k\log(k)n^{1/\log(k)})$.
\end{theorem}

Rather than giving the general protocol in Theorem \ref{thm-HyperCubePIR}, we illustrate the idea in the $3$-dimensional case with $k=2^3=8$ servers.

\begin{example}[$8$-server PIR of complexity $\Theta(n^{1/3})$]\normalfont Suppose we have a database of size $n=\ell^3$ replicated in $k=2^3$ servers. Interpret $x$ as a $3$-tensor in $\F_2$ of dimensions $\ell\times\ell\times \ell$,

\[x=\sum_{i,j,k=1}^{\ell}x_{ijk}(e_i\otimes e_j\otimes e_k).\]

 Equivalently, $x$ can be thought of as a cube grid of side length $\ell$ with vertices labelled $0$ or $1$. Label the $8$ servers in $\mathcal{S}$ by binary strings of length $3$, namely
\begin{align*}
\mathcal{S}=\{&S_{000},S_{001},S_{010},S_{011},\\
&S_{100},S_{101},S_{110},S_{111}\}.
\end{align*} 
  Suppose that user $\mathcal{U}$ wishes to retrieve item $x_{ijk}$ in position $(i,j,k)$ from the database $x$. The PIR scheme proceeds as follows
  \begin{itemize}
  \item[(i)] $\mathcal{U}$ chooses three queries $q_a^{(0)},q_b^{(0)},q_c^{(0)}\in\{0,1\}^{\ell}$ uniformly at random, and produces three additional queries
  \[q_{a}^{(1)}=q_a^{(0)}+e_i,\ q_{b}^{(1)}=q_b^{(0)}+e_j,\text{ and } q_c^{(1)}=q_c^{(0)}+e_k.\]
  \item[(ii)] $\mathcal{U}$ sends $(q_a^{\alpha},q_b^{\beta},q_c^{\gamma})$ to server $S_{\alpha\beta\gamma}$, for each $(\alpha,\beta,\gamma)\in\{0,1\}^3$.
  \item[(iii)] Server $S_{\alpha\beta\gamma}$ computes 
  \[a_{\alpha\beta\gamma}=\sum_{r,s,t=1}^{\ell}x_{rst}(q_a^{\alpha})_r(q_b^{\beta})_s(q_c^{\gamma})_t\in\{0,1\},\]
  and sends $a_{\alpha\beta\gamma}$ to $\mathcal{U}$.
  \item[(iv)] $\mathcal{U}$ retrieves $x_{ijk}$ by adding all responses $a_{\alpha\beta\gamma}$, we have
  \[x_{ijk}=\sum_{\alpha,\beta,\gamma\in\{0,1\}} a_{\alpha\beta\gamma}.\]
  \end{itemize}
  To show correctness we only need to prove the validity of the equation above. This is a consequence of the fact that to the tensor $x=\sum_{r,s,t} x_{rst}(e_r\otimes e_s \otimes e_t)$ corresponds a trilinear form $T:\F_2^3\rightarrow \F_2$. For each $r,s,t$ one has a trilinear form $[e_r\otimes e_s\otimes e_t]:\F_2^{\ell}\times \F_2^{\ell}\times \F_2^{\ell}\rightarrow \F_2$ by extending
  \[[e_r\otimes e_s\otimes e_r](e_i, e_j, e_k)=\delta_{ir}\delta_{sj}\delta_{rk},\]
  linearly on each of the three arguments. Letting $T=\sum_{r,s,t}x_{r,s,t}[e_r\otimes e_s\otimes e_t]$, we find that
  \begin{align*}T(q_{a}^{\alpha},q_{b}^{\beta},q_c^{\gamma})&=\sum_{r,s,t=1}^{\ell}x_{rst}[e_r\otimes e_s\otimes e_t](q_a^{\alpha},q_b^{\beta},q_c^{\gamma})\\
  &=\sum_{r,s,t=1}^{\ell}x_{rst}(q_a^{\alpha})_r(q_b^{\beta})_s(q_c^{\gamma})_t\\
  &=a_{\alpha\beta\gamma}.
  \end{align*}
  Therefore, using the multilinearity of $T$
 \begin{align*}\sum_{\alpha,\beta,\gamma\in\{0,1\}} a_{\alpha\beta\gamma}&=\sum_{\alpha,\beta,\gamma\in\{0,1\}}T(q_a^{\alpha},q_b^{\beta},q_c^{\gamma})\\
 &=T(q_a^{(0)}+q_a^{(1)},q_b^{(0)}+q_b^{(1)},q_c^{(0)}+q_c^{(1)})\\
 &=T(e_i,e_j,e_k)\\
 &=\sum_{r,s,t=1}^{\ell}x_{rst}[e_r\otimes e_s\otimes e_t](e_i,e_j,e_k)\\
 &=\sum_{r,s,t=1}^{\ell}x_{rst}\delta_{ri}\delta_{sj}\delta_{tk}\\
 &=x_{ijk}.
 \end{align*}
 The PIR scheme is private under the assumption of non-collusion, since each individual server receives a triple of queries uniformly distributed in $\{0,1\}^{\ell}$. In total, each server receives a query of $3\ell$ bits and replies with a single bit, so the communication complexity is $8(3\ell+1)=\Theta(n^{1/3})$.
\end{example}

We have illustrated some of the techniques one can use to create PIR schemes. There have been many subsequent improvements to the communication complexity in Theorem \ref{thm-HyperCubePIR} (see the survey on PIR by Gasarch \cite{Gasarch-PIRSurvey}). We mention some breakthrough results,
\begin{itemize}
\item Beimel, Ishai, Kushilevitz and Raymond in 2002 \cite{BIKR-BreakPIR} found a PIR scheme of communication complexity $n^{O(\log\log(k)/k\log(k))}$. This was the first improvement over schemes of complexity $O(n^{1/(2k-1)})$. One of the strategies that the authors use is to interpret the database $x$ as a multivariate polynomial over $\F_2$.
\item Yekhanin in 2008 \cite{Yekhanin-3ServerLCD} found the first subpolynomial PIR scheme, under the assumption that there are infinitely many Mersenne primes. This $3$-server PIR scheme was obtained via a relationship between PIR and \textit{locally decodable codes} (or LCDs), established by Katz and Trevisan in \cite{Katz-Trevisan-LCD-PIR} (see also the survey by Yekhanin \cite{Yekhanin-LCDSurvey}).
\item  Efremenko in 2009 \cite{Efremenko-ProceedingsPIR,Efremenko-PIR} found a subpolynomial PIR without conjectural assumptions for $k\geq 3$ servers.
\item Dvir and Gopi in 2016 \cite{Dvir-Gopi-2ServerPIR} found the first subpolynomial $2$-server PIR scheme. We remark that $2$-server PIR had seen no improvements from the best known $O(n^{1/3})$ communication complexity since the 1995 paper by Chor et. al. \cite{CGKS-PIR}.
\end{itemize}

There are many variations to the problem of PIR. For example, the non-collusion assumption has been relaxed to preserve privacy up to $T$ colluding servers \cite{BIW-LCD-tPIR,BIK-GeneralConstructionsPIR}. PIR has also been considered over coded databases instead of replicating databases \cite{HGHK-CodedPIR}.\\

We highlight the variant known as \textit{computational PIR},\index{private information retrieval!computational} or \textit{CPIR}. In this variation, the privacy assumption is relaxed, and our assumption is that the servers in $\mathcal{S}$ can infer $i$ only if they can solve a specific computational problem, which is conjecturally computationally expensive. In their work \cite{Chor-Gilboa-CPIR}, the authors propose a CPIR scheme with sublinear communication complexity. In this scheme, the server can determine the value of $i$ only if it can solve the \textit{quadratic residuosity problem}, which involves deciding whether an integer $a$ is a square residue modulo an integer $N$. It is widely believed that this problem is difficult to solve for a large non-prime value of $N$.\\

Despite the communication complexity advantages of CPIR over PIR, CPIR suffers from a practical limitation that may be insurmountable: it is faster to send one bit of information than performing an operation on it. This is an important issue since, in order to maintain privacy for the user, a CPIR scheme must perform operations on every bit of the database, otherwise the server could narrow down the search for $i$. In \cite{Sion-Carbunar-CPIRIssues} the authors demonstrate empirically that carrying a single-database CPIR protocol is more time-consuming than using the trivial PIR scheme.  They further predict that this effect is likely to be amplified in the future due to the greater rate of increase in communication speed compared to computing speed.

\section{User-private information retrieval}

PIR has several practical limitations. For example, most PIR schemes assume that users already know the position $i$ they want to retrieve. However, this assumption is often unrealistic, and a more practical scenario would be to conduct keyword-based searches: see for example \cite{CGN-KeywordPIR}. Nevertheless, the most significant drawback of PIR is its dependence on server cooperation. By this we mean that the server must willingly provide a PIR system and adhere to a protocol that guarantees user privacy. Unfortunately, this assumption may not hold true in many situations.\\

To preserve user privacy in cases where the server is unwilling to cooperate, a complementary approach known as \textit{User-Private Information Retrieval (UPIR) }can be adopted. In UPIR, instead of considering a ``game'' between a user and one or multiple databases where the user attempts to hide the requested information, we consider multiple users playing against one or more databases. In this scenario, the objective is not to conceal the requested item $i$ but rather to hide the identity of the user who made the request.\\

In a UPIR system, we consider a set of users $\mathcal{U}$ and we assume that all users in $\mathcal{U}$ have access to a database through a server or collection of servers $\mathcal{S}$. We do not make any assumptions on the way $\mathcal{S}$ encodes the database, or on the protocol followed in the communications between the users and the servers. Instead, we can consider $\mathcal{S}$ as a ``black-box'' function $\mathcal{A}:\mathcal{Q}\rightarrow\mathcal{X}$, where $\mathcal{X}$ is the set of items of the database and $\mathcal{Q}$ is the set of admissible queries.\\

To ensure user privacy, the UPIR system employs a random selection process, where a \textit{proxy} $v\in \mathcal{U}$ is chosen to request the desired information on behalf of user $u\in \mathcal{U}$. It can be shown that selecting proxies uniformly at random is necessary and sufficient to guarantee that $\mathcal{S}$ is unable to trace the queries back to the originating users, \cite{Swanson-Stinson-UPIR-I}. Consequently, $\mathcal{S}$ can only gather information about the overall query patterns of the network. In a sufficiently large network, this limitation minimizes the server's ability to create user profiles based on the collected data.\\

Therefore, in a UPIR system, anonymity with respect to the server is easy to obtain, however the users within the network may be able to do inference that would allow them to identify the sources of certain queries. The goal of UPIR is then to protect against malicious users in the network. The main tool to increase privacy is to restrict the traffic of information in the network by means of \textit{message spaces}, where the information is recorded and retrieved. We will see how the structure of the message spaces is crucial to the level of anonymity of users in the network.\\

Formally we define a UPIR system as follows:\index{UPIR!system}

\begin{definition}\normalfont
An \textit{UPIR system} is defined as a bipartite graph $(\mathcal{U}\cup\mathcal{M},E)$, where $\mathcal{U}$ denotes the set of \textit{users} and $\mathcal{M}$ denotes the set of \textit{message spaces}. A user $u\in\mathcal{U}$ is said to \textit{have access} to the message space $M\in\mathcal{M}$ if $(u,M)\in E$. We say that the UPIR system is \textit{connected} whenever the bipartite graph $(\mathcal{U}\cup\mathcal{M},E)$ is connected.
\end{definition}\index{message space}

\begin{figure}[H]
  \centering
  \begin{tikzpicture}[every node/.style={circle, draw}]
    \node[draw,rectangle] (s) at (2,2) {$\mathcal{S}$};
    \node (u1) at (0,0) {$u_1$};
    \node (u2) at (2,0) {$u_2$};
    \node (u3) at (4,0) {$u_3$};
    
    \node (m1) at (0,-2) {$M_1$};
    \node (m2) at (2,-2) {$M_2$};
    \node (m3) at (4,-2) {$M_3$};
    
    \draw[dashed] (s) -- (u1);
    \draw[dashed] (s) -- (u2);
    \draw[dashed] (s) -- (u3);
    \draw (u1) -- (m1);
    \draw (u1) -- (m2);
    \draw (u2) -- (m1);
    \draw (u2) -- (m2);
    \draw (u2) -- (m3);
	\draw (u3) -- (m2);
	\draw (u3) -- (m3);    
    
  \end{tikzpicture}
  \caption{Visualisation of a UPIR system}
\end{figure}

\begin{remark*}\normalfont From an incidence structure $\mathcal{D}$ with points $\mathcal{P}$ and blocks $\mathcal{B}$, we can construct a UPIR system $(\mathcal{P}\cup\mathcal{B},E)$ by taking the \textit{incidence graph}, also known as the \textit{Levi graph}, of $\mathcal{D}$. In this graph, an edge $(p,b)\in E$ exists if and only if point $p$ is incident to block $b$.\index{graph!Levi}
\end{remark*}

We measure the distance between two users $u,v\in \mathcal{U}$ in a UPIR system $(\mathcal{U}\cup\mathcal{M},E)$ as $d(u,v)/2$, where $d(u,v)$ is the shortest distance between $u$ and $v$ in the bipartite graph.\\

In a UPIR system, message spaces serve as the means of communication among users, where messages are both written and read. Queries intended for the server and the corresponding replies are communicated through these message spaces. We assume that the content of a message space $M\in \mathcal{M}$ is only visible to users $u\in \mathcal{U}$ who have access to $M$. To enable communication within the UPIR system, a protocol is required. We provide an explicit example of such a protocol below. We refer to the combination of a UPIR system and a protocol as a \textit{UPIR scheme}. Distinguishing between the UPIR system (bipartite graph) and the protocol helps illustrate the interplay between the graph's combinatorics and the privacy properties of the protocol.

\begin{protocol}\normalfont\label{prot-BasicUPIR}
Let $(\mathcal{U}\cup\mathcal{M},E)$ be a connected UPIR system. Suppose that user $u$ wants to retrieve the response to query $Q$ from the server $\mathcal{S}$\index{UPIR!scheme}
\begin{enumerate}
\item User $u$ chooses a user $v$ uniformly at random from the set of all users.
\item If $u = v$, $u$ requests $Q$ directly from the database, receiving response $R$.
\item Otherwise, $u$ chooses uniformly at random a shortest path $(u, M_1, u_1, \ldots, M_n, v)$ from $u$ to $v$ in the bipartite graph.
\item User $u$ writes a request $[(u_1, M_2, \ldots, M_n, v), Q]$ onto $M_1$.
\item For $i = 1, 2, \ldots, n-1$, user $u_i$ observes the request $[(u_i, M_{i+1}, \ldots, M_n, v), Q]$ addressed to him in $M_i$. $u_i$ writes a new request $[(u_{i+1}, M_{i+2}, \ldots, M_n, v), Q]$ to $M_{i+1}$ and remembers $M_i$ and $Q$.
\item When the message $[(v), Q]$ reaches message space $M_n$, user $v$ sees it and forwards $Q$ to the server. User $v$ writes the response $R$ from the server as $[Q, R]$ in $M_n$.
\item User $u_j$, upon seeing the response $[Q, R]$ in $M_{j+1}$, writes the response $[Q, R]$ to $M_j$.
\item User $u$ receives the response $R$ to his query after $u_1$ writes $[Q, R]$ to $M_1$.
\end{enumerate}
\end{protocol}

UPIR was introduced by Domingo-Ferrer, Bras-Amorós, Wu and Manjón in \cite{DBWM-UPIR}. Here, the authors presented a protocol where users write queries to message spaces without specifying a proxy. Swanson and Stinson also developed a special case of this protocol \cite{Swanson-Stinson-UPIR-I,Swanson-Stinson-UPIR-II}. Both groups of authors focused on UPIR systems where every pair of users share a common message space. In such cases, Protocol \ref{prot-BasicUPIR} can be implemented to ensure that every path between two users has a length of at most 2, allowing any user to write requests directly to their chosen proxy.\\

Stokes and Bras-Amorós \cite{Stokes-Amoros-UPIR} addressed the problem of constructing a UPIR system while imposing restrictions such as a constant degree for all message spaces $M$. This requirement aims to balance the load among message spaces. They also stipulated that every pair of users shares precisely one message space. After eliminating degenerate solutions where message spaces have sizes of 1, 2, $n-1$, or $n$, the authors identified the class of finite projective planes as the optimal configuration. \\

Swanson and Stinson \cite{Swanson-Stinson-UPIR-I} analysed attacks on UPIR systems based on projective planes and proposed UPIR systems constructed from balanced incomplete block designs (BIBD) and pairwise balanced designs (PBD). We recall the definition of PBD bellow:

\begin{definition}\normalfont\label{def-PBD} Let $X$ be a set of points with cardinality $v$, let $K\subseteq [v]:=\{1,\dots,v\}$. A pair $(X,\mathcal{B})$ where $\mathcal{B}$ is a family of subsets of $X$ is called a $(v,K,\lambda)$-\textit{pairwise balanced design}, or PBD, if\index{design!pairwise balanced}
\begin{itemize}
\item[(i)] $|B|\in K$ for all $B\in \mathcal{B}$,
\item[(ii)] Every pair of distinct elements of $V$ is in a unique block of $\mathcal{B}$.
\end{itemize}
\end{definition}
In particular, a projective plane of order $n$ is an $(n^2+n+1,\{n+1\},1)$-PBD. See \cite{Beth-Jungnickel-Lenz} for a reference on design theory. 

\section{Privacy in UPIR schemes}
Under the assumption that message spaces can only be accessed by the users in the network and not by the server, one can establish privacy against the server and more generally against any \textit{external observer}. By external observer we mean some entity which can read communications between users and servers, but has no access to the message spaces and their information. To formalize these notions, we introduce some notation and definitions. Let $\mathcal{Q}$ be a finite set of possible queries and $\mathcal{X}$ be a finite set of possible database items, recall that the servers respond according to a function $\mathcal{A}:\mathcal{Q}\rightarrow\mathcal{X}$. As our setting involves a finite set of users $\mathcal{U}$, we can model all events using finite probability spaces and finite probability distributions.\\

When a user $u \in \mathcal{U}$ requests a query $Q \in \mathcal{Q}$ through the UPIR scheme, we refer to $u$ as the \textit{source} of $Q$. We define the event $\mathbf{s}(Q)=u$ as the occurrence of $u$ being the source of query $Q$. Similarly, when a user $v \in \mathcal{U}$ acts as a proxy sending query $Q$ to the server $\mathcal{S}$, we denote it as $\mathbf{p}(Q)=v$. We assume that the server has a prior finite probability distribution for each user's query, denoted as $\p(\mathbf{s}(Q)=u)$. Similarly, we assume that the server has a prior distribution for the event $\mathbf{p}(Q)=v$, denoted as $\p(\mathbf{p}(Q)=v)$. 

\begin{definition}\normalfont\label{def-PrivateEO}
We define a UPIR scheme to be \textit{private against external observers} if, for any pair of users $u,v \in \mathcal{U}$, the conditional probability $\p(\mathbf{s}(Q)=u|\mathbf{p}(Q)=v)$ is equal to $\p(\mathbf{s}(Q)=u)$.
\end{definition}
In Definition \ref{def-PrivateEO}, we adopt a Bayesian approach, assuming that an external observer holds a prior probability distribution representing their degree of belief that user $u$ will request item $\mathcal{Q}$. According to Definition \ref{def-PrivateEO}, privacy against external observers implies that the observer's posterior probability, after observing an execution of the UPIR scheme where $u$ requests $Q$, remains equal to the prior probability. In other words, the observer gains no new information on the likelihood that $u$ will request $Q$. Notice that Definition \ref{def-PrivateEO} does not consider message spaces or internal information of the UPIR scheme, hence the name ``privacy against external observers''. Formally we could define a general \textit{observer} $\omega$ to a UPIR scheme as a joint probability distribution
\[\p_{\omega}(\mathbf{s}(Q)=u,\mathbf{p}(Q')=v,Q''\in M),\]
for $Q,Q',Q''\in\mathcal{Q}$, $u,v\in\mathcal{U}$ and $M\in \mathcal{M}$.  The marginal probability distributions derived from this joint distribution provide us with the observer's degree of belief regarding all possible events occurring within the UPIR scheme.

\begin{theorem}[Swanson and Stinson, Theorems 6.1, 6.2 \cite{Swanson-Stinson-UPIR-I}] A connected UPIR scheme is private against external observers if each user chooses proxies uniformly at random, and the proxies for distinct queries are chosen independently.
\end{theorem}
In particular, Protocol \ref{prot-BasicUPIR} ensures privacy against external observers. Therefore, the main problem of UPIR is not to guarantee anonymity against the server, but rather to ensure that the identity of users in the network cannot be compromised by a coalition of \textit{honest-but-curious} colluding users within the network. By honest-but-curious, we mean that the users in this coalition will act according to the UPIR scheme protocol but may attempt to determine the source of the queries they observe.\\

\begin{definition}\normalfont Let $(\mathcal{U}\cup\mathcal{M},E)$ be a UPIR system equipped with a communication protocol. Consider a coalition $\mathcal{C}\subset\mathcal{U}$ consisting of users collaborating to identify the source of the messages transmitted within the UPIR scheme. Users $u,v\in\mathcal{U}$ are \textit{pseudonymous} with respect to $\mathcal{C}$ if and only if $\p(\mathbf{s}(Q)=u)=0$ is equivalent to $\p(\mathbf{s}(Q)=v)=0 $, and for any pair $(Q,M)$ where $Q\in\mathcal{Q}$ and $M\in\mathcal{M}$ is a message space accessible to a user in $\mathcal{C}$, the following holds:
\[\p(\mathbf{s}(Q)=u|Q\in M)\p(\mathbf{s}(Q)=v)=\p(\mathbf{s}(Q)=v|Q\in M)\p(\mathbf{s}(Q)=u).\]
 
\end{definition}\index{pseudonymity}
Here the probability distributions correspond to the beliefs of coalition $\mathcal{C}$. Informally, the concept of pseudonymity of two users $u$ and $v$ with respect to a coalition $\mathcal{C}$ means that the information that $\mathcal{C}$ can gain from observing the message spaces they have access to is not sufficient to distinguish between $u$ and $v$. Notice that, unlike privacy against external observers, pseudonymity with respect to coalitions depends on the structure of the UPIR system.
\begin{lemma}[cf. \cite{UPIR-paper}]\normalfont Pseudonymity with respect to a coalition $\mathcal{C}$ is an equivalence relation on the set of users $\mathcal{U}$ of a UPIR scheme.
\end{lemma}
\begin{proof}
 Reflexivity and symmetry are both trivial. To show transitivity holds, let $u,v\in\mathcal{U}$ be pseudonymous with respect to $\mathcal{C}$ and likewise with $v$ and $w$, then we have
\begin{align*}
& \p(\mathbf{s}(Q)=u|Q\in M)\p(\mathbf{s}(Q)=v)=\p(\mathbf{s}(Q)=v|Q\in M)\p(\mathbf{s}(Q)=u),\text{ and }\\
&\p(\mathbf{s}(Q)=v|Q\in M)\p(\mathbf{s}(Q)=w)=\p(\mathbf{s}(Q)=w|Q\in M)\p(\mathbf{s}(Q)=v).
\end{align*}
Direct computation shows that
\begin{align*}
\p(\mathbf{s}(Q)=u|Q\in M)\p(\mathbf{s}(Q)=w)\p(\mathbf{s}(Q)=v)&=\p(\mathbf{s}(Q)=v|Q\in M)\p(\mathbf{s}(Q)=u)\p(\mathbf{s}(Q)=w)\\
&=\p(\mathbf{s}(Q)|Q\in M)\p(\mathbf{s}(Q)=u)\p(\mathbf{s}(Q)=v).
\end{align*}
Therefore,
\[[\p(\mathbf{s}(Q)=u|Q\in M)\p(\mathbf{s}(Q)=w)-\p(\mathbf{s}(Q)=w|Q\in M)\p(\mathbf{s}(Q)=u)]\p(\mathbf{s}(Q)=v)=0.\]
If $\p(\mathbf{s}(Q)=v)\neq 0$ then we are done. Otherwise, if $\p(\mathbf{s}(Q)=v)=0$, then from the equivalence of  $\p(\mathbf{s}(Q)=v)=0$ to both $\p(\mathbf{s}(Q)=u)=0$ and $\p(\mathbf{s}(Q)=w)=0$ we have that $u$ and $w$ are pseudonymous.\qedhere
\end{proof}

 We use the following definition of security against coalitions

\begin{definition}\normalfont\label{def-SecureUPIR}
Let $(\mathcal{V}_i)$ be a family of UPIR schemes indexed by $i\in\N$, where scheme $i$ has exactly $n_i$ users. We say that the family $\mathcal{V}_i$ has a \textit{secure against $t$-coalitions} if and only if, for any $\epsilon\in (0,1)$, there exists $N_{\epsilon}\in \N$ such that for $n_i>N_{\epsilon}$ and any coalition $\mathcal{C}\subseteq \mathcal{U}(\mathcal{V}_i)$ of size $t$ in $\mathcal{V}_i$, there exists a pseudonymity class $P$ with respect to $\mathcal{C}$ such that the union of all other pseudonymity classes has a size of $O(n_{i}^{1-\epsilon})$. A family of UPIR schemes is called \textit{secure} if and only if it is secure against $t$-coalitions for all $t\in\N$.
\end{definition}\index{security against coalitions}
In other words, a family $(\mathcal{V}_i)$ of UPIR schemes is secure against $t$-coalitions if for large enough members of $(\mathcal{V}_i)$, an arbitrary coalition of size $t$ can identify only a negligible portion of the UPIR system. A secure family of UPIR schemes is one that is secure against coalitions of users of bounded size.\\

An \textit{identifying set} is a coalition $\mathcal{C}\subseteq\mathcal{U}$ such that the pseudonymity classes with respect to $\mathcal{C}$ are all of size $1$. In order to evaluate the level of anonymity provided by a UPIR scheme within the network, we use the concept of \textit{linked queries} introduced by Swanson and Stinson in \cite{Swanson-Stinson-UPIR-I}. Linked queries refer to a group of queries that can be traced back to a single source, such as queries related to a very niche topic. In our analysis, we adopt a conservative approach by considering a worst-case scenario where each user attaches a unique identifier to each of their queries, e.g. their IP or MAC address. We emphasise that this does not mean that the true identity of the user $u$ is exposed, but rather that the queries originate from the same user $u$.
\begin{definition}\normalfont In a UPIR scheme, two queries $Q$ and $Q'$ are linked if and only if
\[\p(\mathbf{s}(Q)=\mathbf{s}(Q'))=1.\]
In other words, the coalition $\mathcal{C}$ has complete belief that the source of query $Q$ and query $Q'$ is the same. We say that $u$ makes \textit{sufficiently many linked queries} whenever $u$ requests a series of linked queries $\{Q_1,\dots,Q_N\}$ repeatedly until all possible combinations of proxies and paths to request $Q\in\{Q_1,\dots,Q_N\}$ in the UPIR scheme have been used.\index{linked queries}
\end{definition}

\begin{theorem}[Theorem 10, \cite{UPIR-paper}]
In a PBD-UPIR scheme using Protocol \ref{prot-BasicUPIR}, a single eavesdropper can identify any user who makes a sufficiently large number of linked queries. In other words, any coalition of size one is an identifying set.
\end{theorem}
\begin{proof}
Let $u$ be a user who makes a sufficiently large number of linked queries. Then, an eavesdropper $c$ will observe linked queries in the unique message space $M$ to which both $c$ and $u$ have access. Then, $c$ will note that $u$ is never written as a proxy for a linked query $Q$ in $M$ since $u$ acts as its own proxy. Therefore, $c$ can identify $u$ as the source of the linked queries with a probability of $1$.
\end{proof}

\begin{corollary}\normalfont A PBD-UPIR scheme using Protocol \ref{prot-BasicUPIR} is not secure.
\end{corollary}
\begin{proof}
Since a coalition $\mathcal{C}$ of size one is an identifying set, each pseudonymity class is a singleton. Let $P$ be an arbitrary pseudonymity class with respect to $\mathcal{C}$ in a PBD-UPIR scheme with $n_i$ users, then the union of all pseudonymity classes distinct from $P$ is of size $n_i-1=O(n_i)$. \qedhere
\end{proof}

Note that $u$ must act as its own proxy as often as any other user, otherwise we lose privacy against external observers which is the main priority in a UPIR scheme. To circumvent the vulnerability of $u$ not writing in his message spaces to be the proxy, one may suppose that $u$ writes $[Q,u]$ randomly in one of the message spaces he has access to. In this case, a frequency analysis by $c$, observing the query patterns and analysing the frequency of $[Q,u]$ in the message spaces would compromise the identity of $u$.\\

We consider an encrypted version of Protocol \ref{prot-BasicUPIR}, using the technique of \textit{onion routing} \cite{OnionRouting} to ensure that only the source and the proxy can read the query $Q$.

\begin{protocol}\normalfont\label{prot-EncryptedUPIR}
Let $(U \cup M, E)$ be a UPIR system where the distance between any pair of users is at most $2$, equivalently the diameter of the bipartite graph is $2$ or $4$. Suppose furthermore that a public key infrastructure is in place, and a public key for every user is available. Namely, for each user $u\in\mathcal{U}$ there is a unique encryption function $\varphi_u$  accessible to all users in $\mathcal{U}$, and a unique decryption function $\delta_u$ accessible only to $u$, such that $\delta_u(\varphi_u(x))=x$ for all possible messages $x$.  User $u$ wishes to retrieve the response to the query $Q$ from the server.\index{UPIR!scheme}

\begin{enumerate}
  \item $u$ chooses a user $v$ uniformly at random from the set of all users, and generates a secret key $\psi$ for a symmetric cipher.
  \item If $u = v$, $u$ requests $Q$ directly from the server, receiving response $R$.
  \item If $d(u, v) = 1$, then user $u$ encrypts both the query $Q$ and the key $\psi$ using $v$'s public key $\varphi_v$, and writes the request $[v, \varphi_v(Q), \varphi_v(\psi)]$ to a message space that they share.
  \item Otherwise, $u$ chooses a shortest path to $v$, say $[u, M_1, u_1, M_2, v]$. $u$ writes the query $[(u_1, M_2, v), \varphi_v(Q), \varphi_v(\psi)]$ to $M_1$.
  \item When $v$ receives the request, he forwards $Q$ to the database, receives response $R$, and writes the response $[(v), \varphi_v(Q), \psi(R)]$ to the message space in which the query was observed. The response is returned to user $u$ as in Protocol 1.
\end{enumerate}
\end{protocol}

The encryption in Protocol \ref{prot-EncryptedUPIR} offers a significant advantage in that only proxies are able to observe the content of the query $\mathcal{Q}$. However, PBD-UPIR schemes remain insecure when using Protocol \ref{prot-EncryptedUPIR}. We show this below with an analysis based on \textit{intersection attacks}. These are attacks in which members of a coalition exploit knowledge of the incidences in the system, observing linked queries, and determining the source as a user in the ``intersection'' of two or more message spaces.
\begin{theorem}[\cite{UPIR-paper}]
In a PBD-UPIR scheme using Protocol \ref{prot-EncryptedUPIR}, there is an identifying set of size $3$.
\end{theorem}
\begin{proof}
Let $u$ be a user, and assume that $u$ makes sufficiently many linked queries. We show that a coalition $\mathcal{C}=\{c_1,c_2,c_3\}$ of three users, not all sharing a common message space, is an identifying set. Users $u$ and $c_1$ share access to a unique message space $M$. The spy $c_1$ can identify $M$, as linked queries addressed to $c_1$ will only be written in $M$. Since the users in $\mathcal{C}$ do not all share the same message space, there is a user $c\in \{c_2,c_3\}$ that does not have access to message space $M$,  without loss of generality we may assume that $c_2=c$. Let $U(M)$ be the set of users with access to message space $M$. Then, $c_2$ shares exactly one message space with each user in $U(M)$, and $c_2$ will observe linked queries only in the unique message space shared with $u$ and in no other message space $c_2$ has access to. In this way, the coalition can identify any user $u$ with probability $1$.
\end{proof}

This vulnerability in PBD-UPIR schemes arises from the fact that every pair of users shares a message space. We present secure UPIR schemes based on incidence structures where this condition no longer holds.

\section{Generalised quadrangles}\label{sec-GQ}
In this section we introduce \textit{generalised quadrangles} (GQs). Here we assume that the reader is familiar with the concepts of Chapter \ref{chap-GramEquations} and Chapter \ref{chap-HermitianForms}. For more on generalised quadrangles see \cite{Payne-Thas-GQs}.

\begin{definition}\normalfont \label{def-GQ} A generalised quadrangle is an incidence structure consisting of points and lines that satisfy the following properties:
\begin{enumerate}\index{generalised quadrangle}
\item Each block is incident to $1+s$ points for some $s\geq 1$ and two distinct lines are incident to at most one point.
\item Each point is incident to $1+t$ lines for some $t\geq 1$ and two distinct points are incident with at most one line.
\item For any point-line pair $(x,L)$ where $x$ is not in $L$, there is a unique point $x'$ in $L$ that shares a line with $x$.
\end{enumerate}
A generalised quadrangle with parameters $s$ and $t$ is denoted by $\GQ(s,t)$, and the tuple $(s,t)$ is the \textit{order} of the generalised quadrangle. There is a point-line duality in the definition of GQs, if we exchange the words point and line in the Definition \ref{def-GQ}, we obtain the definition of a $\GQ(t,s)$.
\end{definition}
A $\GQ(s,t)$ is said to be \textit{trivial} if $s=1$ or $t=1$. If $s=1$, then there are two points in every line so the GQ is a graph, in this case the GQ axioms force the graph to be bipartite. If $t=1$, then there are $2$ blocks through every point, so the GQ is a \textit{grid}\index{generalised quadrangle!grid} and points can be labelled as $x_{ij}$ for $0\leq i,j\leq s$ and lines consist of points sharing a common subscript in the same position.
\begin{example}\normalfont Consider the set $[6]=\{1,2,3,4,5,6\}$ with six elements. Let $\mathcal{P}=\binom{[6]}{2}$ be the set of unordered pairs from $[6]$, resulting in $|\mathcal{P}|=15$ elements. Let $\mathcal{L}$ be the set of partitions of $[6]$ into three disjoint unordered pairs. We define an incidence structure with points in $\mathcal{P}$ and lines in $\mathcal{L}$. The incidence relation is defined by the occurrence of a pair of $\mathcal{P}$ in a partition of $\mathcal{L}$. For example, the pair $12$ occurs in exactly three partitions, namely $\{12,34,56\}$, $\{12,35,46\}$ and $\{12,36,45\}$. This incidence structure is a $\GQ(2,2)$, pictured below
\begin{figure}[H]\label{fig-GQ(2,2)}
\centering
\begin{tikzpicture}
\pgfmathsetmacro{\Ar}{2.5} 
\coordinate (A1) at (90:\Ar);
\coordinate (A2) at (162:\Ar);
\coordinate (A3) at (234:\Ar);
\coordinate (A4) at (306:\Ar);
\coordinate (A5) at (378:\Ar);

\pgfmathsetmacro{\Br}{2} 
\coordinate (B1) at (126:\Br);
\coordinate (B2) at (198:\Br);
\coordinate (B3) at (270:\Br);
\coordinate (B4) at (342:\Br);
\coordinate (B5) at (414:\Br);

\pgfmathsetmacro{\Cr}{0.8} 
\coordinate (C1) at (126:\Cr);
\coordinate (C2) at (198:\Cr);
\coordinate (C3) at (270:\Cr);
\coordinate (C4) at (342:\Cr);
\coordinate (C5) at (414:\Cr);

\draw[thick] (A1) .. controls (B1) and (B1) .. (A2);
\draw[thick] (A2) .. controls (B2) and (B2) .. (A3);
\draw[thick] (A3) .. controls (B3) and (B3) .. (A4);
\draw[thick] (A4) .. controls (B4) and (B4) .. (A5);
\draw[thick] (A5) .. controls (B5) and (B5) .. (A1);

\draw[thick,blue] (A1) .. controls (C3) and (C3) .. (B3);
\draw[thick,blue] (A2) .. controls (C4) and (C4) .. (B4);
\draw[thick,blue] (A3) .. controls (C5) and (C5) .. (B5);
\draw[thick,blue] (A4) .. controls (C1) and (C1) .. (B1);
\draw[thick,blue] (A5) .. controls (C2) and (C2) .. (B2);

\draw[thick,red] (C2) ..controls (156:3) and (96:3) .. (C5); 
\draw[thick,red] (C3) ..controls (228:3) and (168:3) .. (C1); 
\draw[thick,red] (C4) ..controls (300:3) and (240:3) .. (C2); 
\draw[thick,red] (C5) ..controls (372:3) and (312:3) .. (C3); 
\draw[thick,red] (C1) ..controls (444:3) and (384:3) .. (C4);

\fill(A1) circle (0.1cm) node[above]{$46$};
\fill(A2)circle (0.1cm) node[above left]{$35$};
\fill(A3)circle (0.1cm) node[below left]{$24$};
\fill(A4)circle (0.1cm) node[below right]{$36$};
\fill(A5)circle (0.1cm) node[above right]{$25$};

\fill[blue] (B1) circle (0.1cm) node[above left]{$12$};
\fill[blue] (B2) circle (0.1cm) node[below left]{$16$};
\fill[blue] (B3) circle (0.1cm) node[below]{$15$};
\fill[blue] (B4) circle (0.1cm) node[below right]{$14$};
\fill[blue] (B5) circle (0.1cm) node[above right]{$13$};

\fill[red] (C1) circle (0.1cm) node[above left]{$45$};
\fill[red] (C2) circle (0.1cm) node[below left]{$34$};
\fill[red] (C3) circle (0.1cm) node[below]{$23$};
\fill[red] (C4) circle (0.1cm) node[below right]{$26$};
\fill[red] (C5) circle (0.1cm) node[above right]{$56$};
\end{tikzpicture}
\caption{The smallest non-trivial generalised quadrangle.}
\end{figure}
\end{example}

\begin{lemma}[1.2.1 \cite{Payne-Thas-GQs}]\normalfont\label{lemma-GQCounting} In a $\GQ(s,t)$ the number of points is $v=(s+1)(st+1)$ and, dually, the number of lines is $b=(t+1)(st+1)$. The total number of points at distance $1$ from a given point $x$ is $s(t+1)$ and the total number of points at distance $2$ from $x$ is $s^2t$.
\end{lemma}
\begin{proof}
To see that $v=(s+1)(st+1)$ one can fix a point $x$ in the GQ and count the total number of points sharing a line with $x$, and not sharing a line with $x$. Since $t+1$ lines pass through $x$ and $s+1$ points in each line, there are exactly $s(t+1)$ points sharing a line wiht $x$. For each point $y\neq s$ in a line $\ell$ through $x$, there are $t$ lines distinct from $\ell$ passing through $y$. For each such line, there are $s$ points distinct from $y$, and each such point shares no line with $x$, giving a total of $s^2t$ points. By axiom (iii) in Definition \ref{def-GQ}, this accounts for all points of the GQ not sharing a line with $x$. Therefore, the total number of points is 
\[v=s^2t+s(t+1)+1=(s+1)(st+1).\]
Exchanging the roles of points and lines, duality shows that $b=(t+1)(st+1)$.
\end{proof}

Generalised quadrangles are a particular case of a wider family known as \textit{generalised polygons}. A generalised $m$-gon is an incidence structure whose incidence graph has diameter $m$ and girth $2m$. In particular, generalised quadrangles contain no triangles.\index{generalised polygon}

\begin{theorem}[Feit and G. Higman, \cite{Feit-Higman-NonExGP}]
A finite generalised $m$-gon of order $(s,t)$ with $s,t>1$ satisfies $m\in\{2,3,4,6,8\}$.

\end{theorem}

The class of finite generalised $3$-gons (or generalised triangles) coincides with the class of finite projective planes. Namely, a finite generalised triangle of order $(s,t)$ satisfies $s=t=n$, and is a projective plane of order $n$. For generalised quadrangles the values of $s$ and $t$ may be different, but they must obey Higman's inequality:
\begin{theorem}[D.G. Higman, 1.2.3  \cite{Payne-Thas-GQs}]\label{thm-HigmanBound} In a generalised quadrangle, if $s>1$ and $t>1$ then $t\leq s^2$, and dually $s\leq t^2$.
\end{theorem}
Because of the common framework of generalised polygons, generalised quadrangles are a natural structure to consider in the setting of UPIR, as an extension of previous work focused on projective planes and PBDs.

\subsection{Quadrics and Hermitian varieties over finite fields}
The main infinite families of finite generalised quadrangles are the \textit{classical generalised quadrangles}, which arise from bilinear or sesquilinear forms. Before we can introduce the classical generalised quadrangles we need to discuss \textit{quadrics} and \textit{Hermitian varieties} on finite fields.

\begin{definition}\normalfont Let $\F_q$ denote the finite field on $q$ elements. The $n$-dimensional \textit{affine space} on $\F_q$ is the set $\AG(n,q)=\F_q^{n}$. The $n$-dimensional \textit{projective space} on $\F_q$ is the quotient set $(\F_q^{n+1}-\{0\})/\sim$ of non-zero vectors of $\F_q^{n+1}$ by the equivalence relation $\sim$, where $x\sim y$ if and only if there is a non-zero scalar $\lambda\in\F_q^{\times}$ such that $x=\lambda y$. We denote the $n$-dimensional projective space on $\F_q$ by $\PG(n,q)$.\index{affine space}\index{projective space}
\end{definition}
The projective geometry $\PG(2,q)$ is precisely a projective plane of order $q$. In projective space $\PG(n,q)$, \textit{projective subspaces} of dimension $d$ are in bijective correspondence to subspaces of dimension $d+1$ of the underlying affine space. For example, in $\PG(2,q)$ each point corresponds to a $1$-dimensional subspace of $\F_q^{3}$.\\

Notice that over a finite field, every function $f:\F_q^n\rightarrow \F_q$ is a polynomial on $n$ variables with coefficients in $\F_q$, i.e. $f\in\F_q[x_1,\dots,x_n]$. This is a consequence that by Lagrange interpolation we can construct a polynomial that takes the same values as $f$ on finitely many points, since $\F_q$ is a finite field we can find a polynomial that agrees with $f$ everywhere on $\F_q^n$.
\begin{definition}\normalfont If $f\in \F_q[x_1,\dots,x_n]$ is a polynomial, the \textit{affine variety} defined by $f$ is the subset 
\[V(f)=\{x\in\AG(n,q): f(x)=0\}\subset \AG(n,q).\]
If $f$ is a \textit{form} on $n+1$ variables, i.e. if $f\in\F_q[x_0,x_1,\dots,x_n]$ is a homogeneous polynomial, then the \textit{projective variety} defined by $f$ is the subset
\[V(f)=\{x\in\PG(n,q):f(x)=0\}\subset \PG(n,q).\]
\end{definition}
It is necessary to require that $f$ is homogeneous in order to define a projective variety. Since otherwise we may have $f(x)=0$ yet $f(\lambda x)\neq 0$, and the zeros of $f$ would not be well-defined as elements of the quotient space $(\F_q^{n+1}-\{0\})/\sim$.\\

A \textit{quadric} in projective space $\PG(n,q)$ is a variety of the type $Q=V(\phi)$ where $\phi$ is a quadratic form on $n+1$ variables. An \textit{Hermitian variety} in projective space $\PG(n,q)$ is a variety of the type $H=V(h)$ where $h$ is an Hermitian form on $n+1$ variables. The quadric $Q$ (resp. Hermitian variety $H$) is called \textit{non-singular} if and only if the quadratic form $\phi$ (resp. hermitian form $h$) is regular. Recall that a bilinear or sesquilinear form $f$ on a finite-dimensional vector space $V$ is regular if and only if the matrix $A_f$ of $f$ with respect to a basis of $V$ has determinant $\neq 0$.\\


\begin{remark}\normalfont \label{rem-HFSquare}In order to define an Hermitian form over $\F_q$, we first need to have an involutory automorphism $\tau:\F_q\rightarrow\F_q$. All automorphisms of $\F_q$, where $q=p^f$ for some prime $p$, are powers of the \textit{Frobenius automorphism}, $F:\F_q\rightarrow\F_q$ given by $F(x)=x^p$ for all $x\in\F_q$. Write $\tau=F^a$, so that $\tau(x)=x^{p^a}$, then $\tau(\tau(x))=x$ if and only if $(x^{p^a})^{p^a}=x=x^{q}$, but then $p^{2a}=q$, which implies that $q$ is a perfect square. In what follows we write $\tau$ in exponential notation, i.e. $x^{\tau}:=\tau(x)$.
\end{remark}

If two quadratic forms $\phi$ and $\phi'$ are equivalent then there is a projectivity taking $V(\phi)$ into $V(\phi')$, i.e. a bijective linear mapping $\sigma: \F_q^{n+1}\rightarrow\F_q^{n+1}$ such that $\sigma(V(\phi))=V(\phi')$.  The same claim holds for Hermitian forms. Both quadratic and Hermitian forms over finite fields can easily be classified from the fact that any element in $\F_q^{\times}$ is the sum of two squares:

\begin{lemma}\normalfont \label{lemma-binaryFormsFF} Two binary quadratic forms over a finite field $\F_q$ of characteristic $\neq 2$ are equivalent if and only if their discriminants are equal, i.e.
\[\langle a,b\rangle \simeq \langle c,d\rangle,\]
if and only if $ab\equiv cd$ in the square class group $\Gamma(\F_q)=\F_q^{\times}/(\F_q^{\times})^2$.
\end{lemma}
\begin{proof}
Theorem \ref{thm-2x2HM} in Chapter \ref{chap-BRC} implies that two binary forms $\langle a,b\rangle$ and $\langle c,d\rangle$ are equivalent if and only if $(a,b)_{\F_q}=(c,d)_{\F_q}$ and $ab\equiv cd$ in $\Gamma(\F_q)$. We also showed in Chapter \ref{chap-BRC} (see Example \ref{ex-FFSymbol} ) that every element of $\F_q$ is a sum of two squares, from which it follows that $(a,b)_{\F_q}=1$ for all $a,b\in\F_q^{\times}$, so the first requirement is vacuous and we find that $\langle a,b\rangle\simeq \langle c,d\rangle$ if and only if $ab\equiv cd$ in $\Gamma(\F_q)$.\qedhere
\end{proof}

\begin{theorem}[cf. Chapter 2, Theorem 3.8. \cite{Scharlau-QuadraticHermitian}]\label{thm-FqForms} Let $\F_q$ be a finite field of characteristic $\neq 2$, then there are exactly two isometry classes of regular quadratic forms on $\F_q$ of dimension $n$, namely
\begin{align*}
&\langle 1,1,\dots,1\rangle,\text{ and }\\
&\langle \varepsilon,1,\dots,1\rangle,
\end{align*}
where $\varepsilon$ is a non-square element in $\F_q^{\times}$. In particular, the dimension and the discriminant form a complete set of invariants for quadratic forms over finite fields.
\end{theorem}
\begin{proof}
 Let $\phi$ be regular quadratic form of degree $n$, then by the polarisation identity (Theorem \ref{thm-Polarisation}) we may assume that there are $a_1,\dots,a_n\in\F_q^{\times}$ such that
\[\phi\simeq \langle a_1,\dots,a_n\rangle \equiv a_1x_1^2+\dots+a_nx_n^2.\]
 Without loss of generality we may assume that $a_1,\dots,a_r$ are non-squares in $\F_q^{\times}$ and that $a_{r+1},\dots,a_{n}$ are square. Then since the square-class group $\Gamma(\F_q)=\F_q^{\times}/(\F_{q}^{\times})^2$ has order $2$, we find that
 \[\phi\simeq \langle \varepsilon,\dots,\varepsilon,1,\dots,1\rangle\equiv\varepsilon(x_1^2+\dots+x_r^2)+x_{r+1}^2+\dots+x_{n}^2,\]
 where $\varepsilon\in\F_q^{\times}$ is a non-square residue. By Lemma \ref{lemma-binaryFormsFF} we have that $\langle\varepsilon,\varepsilon\rangle\simeq \langle 1,1\rangle$, since $\varepsilon^2\equiv 1$ in $\Gamma(\F_q)$. Therefore applying the equivalence $\langle 1,1\rangle\simeq \langle \varepsilon,\varepsilon\rangle$ one pair of variables at a time we find,
 \[\phi\simeq \langle 1,\dots,1\rangle \text{ if } r \text{ is even, and  } \phi\simeq \langle \varepsilon,1,\dots,1\rangle\text{ if } r \text{ is odd.}\]
 Notice that these two possible forms are inequivalent since they have discriminant $1$ and $\varepsilon$ respectively, which are different elements in $\Gamma(\F_q)$.\qedhere
\end{proof}

From the classification of quadratic forms on $\F_q$ of Theorem \ref{thm-FqForms}, we can classify quadrics on the projective space $\PG(q,n)$.

\begin{theorem}[cf. Theorem 5.2.4 \cite{Hirschfeld-PGFF}]\label{thm-ClassificationQuadrics}
Let $q$ be an odd prime power, and $Q$ be a quadric on $\PG(n,q)$, then up to equivalence
\begin{itemize}
\item[(i)] If $n=2s$ is even, then
\[Q=V(x_0^2 +x_1x_2+x_3x_4+\dots+x_{2s-1}x_{2s}).\]
\item[(ii)] If $n=2s-1$ is odd, then 
\begin{align*}
&Q=V(x_0x_1+x_2x_3+\dots x_{2s-2}x_{2s-1})\text{, or }\\
&Q=V(f(x_0,x_1)+x_2x_3+\dots+x_{2s-2}x_{2s-1}),
\end{align*}
where $f(x,y)$ is a binary quadratic form which is inequivalent to $xy$.
\end{itemize}
\end{theorem}
\begin{proof}
To prove (i) let $\phi$ be a regular quadratic form of dimension $n+1$, where $n=2s$. Then, by Theorem \ref{thm-FqForms}, $\phi$ is equivalent to either $\langle 1,\dots,1\rangle\equiv x_0^2+x_1^2+\dots x_{2s}^2$ or $\langle \varepsilon,1,\dots,1\rangle\equiv \varepsilon x_0^2+x_1^2+\dots+x_{2s}^2$. We show the identity $V(x_0^2+x_1^2+\dots+x_{2s}^2)=V(\varepsilon x_0^2+x_1^2+\dots+x_n^2)$ holds. To see this, notice that from the equivalence $\langle\varepsilon,\varepsilon\rangle\simeq \langle 1,1\rangle$ we have that $\langle \varepsilon,1\dots,1\rangle\simeq \langle\varepsilon,\varepsilon,\dots,\varepsilon\rangle$, since there are exactly $2s$ ones after the first coefficient in $\langle\varepsilon,1\dots,1\rangle$. Therefore,
\[V(\varepsilon x_0^2+x_1^2+\dots+x_{2s}^2)=V(\varepsilon(x_0^2+x_1^2+\dots+x_n^2))=V(x_0^2+x_1^2+\dots+x_{2s}^2).\]
So in any case we find that an arbitrary regular quadratic form of dimension $n+1=2s+1$ satisfies
\[V(\phi)=V(x_0^2+x_1^2+\dots+x_{2s}^2).\]
Over the field $\F_q$ with $q$ odd, the binary quadratic form $x^2-y^2 \equiv \langle 1,-1\rangle$ is equivalent to the quadratic form $xy$ since
\[\begin{bmatrix}
1 & -1\\
1 & 1
\end{bmatrix}^{\intercal}\begin{bmatrix}
0 & 1/2\\
1/2 & 0
\end{bmatrix}
\begin{bmatrix}
1 & -1\\
1 & 1
\end{bmatrix}=\begin{bmatrix}
1 & 0\\
0 & -1
\end{bmatrix},\]
over any field of characteristic $\neq 2$. Theorem \ref{thm-FqForms} implies $\langle 1,1\rangle\simeq \langle -1,-1\rangle$  over $\F_q$. Applying this equivalence for $s$ pairs $\langle 1,1\rangle$, we have 
\[\langle 1;1,\dots,1\rangle\simeq \langle 1;1,\dots, 1;-1,\dots,-1\rangle,\]
where there are exactly $s$ occurrences of $1$ and $s$ of $-1$ after the first coefficient in the right-hand-side. Collecting the terms $+1$ and $-1$ in pairs, and using the fact that $\langle 1,-1\rangle \simeq xy$, we find
\[V(\phi)=V(x_0^2+x_1x_2+\dots+x_{2s-1}x_{2s}).\]
In particular any quadric $Q$ is equal to $V(x_0^2+x_1x_2+\dots+x_{2s-1}x_{2s})$. 
The proof of (ii) is similar, and we refer the reader to Theorem 5.2.4 of \cite{Hirschfeld-PGFF} for the details.\qedhere
\end{proof}

\begin{remark*} \normalfont The classification of quadrics above also holds for $q$ even, but the proof is entirely different. We refer the reader to Theorem 5.1.7 of Hirschfeld's book \cite{Hirschfeld-PGFF}.
\end{remark*}

The representatives of quadratic forms chosen above are more convenient than the ones in Theorem \ref{thm-FqForms} since they do not depend on the parity of $q$. In addition, we will see in the next subsection that, with these representatives, it is easier to identify the dimension of a maximal linear subspace contained in the quadric.\\

Recall that by Remark \ref{rem-HFSquare} we can only consider Hermitian forms over fields $\F_{q^2}$ where $q$ is a prime power. We can easily characterise Hermitian varieties:
\begin{theorem}[cf. Theorem 5.1.5 \cite{Hirschfeld-PGFF}]
Over the field $\F_{q^2}$ there is only one isometry class of regular Hermitian forms of dimension $n+1$, in particular all such Hermitian forms are isomorphic to the form
\[x_0^{\tau}x_0+x_1^{\tau}x_1+\dots+x_n^{\tau}x_n=x_0^{q+1}+x_1^{q+1}+\dots+x_n^{q+1}.\]
\end{theorem}
\begin{proof}
By Jacobson's reduction (Theorem \ref{thm-JacobsonReduction}), two Hermitian forms $h$ and $h'$ over $\F_{q^2}$ are equivalent if and only if their trace forms $\phi_h$ and $\phi_{h'}$ are equivalent as quadratic forms on $\F_q$. Since $\F_{q}\subset\F_{q^2}$ is a quadratic field extension, we can identify $\F_{q^2}=\F_{q}[\varepsilon]$, where $\varepsilon$ is a non-square in $\F_q^{\times}$. After polarisation, an arbitrary Hermitian form $h$ over $\F_{q^2}$ can be written as $\phi\equiv a_0x_0^{q+1}+a_1x_1^{q+1}+\dots+a_nx_n^{q+1}$ for some $a_0,a_1,\dots,a_n\in\F_{q}^{\times}$, in which case its trace form $\phi_h$ is isomorphic to $\langle a_0,a_1,\dots,a_n; \varepsilon a_0,\varepsilon a_1,\dots,\varepsilon a_n\rangle$. Therefore, the discriminant of the trace form $\phi$, of an arbitrary Hermitian form $h$, is  $\varepsilon^{n+1}$. By Theorem \ref{thm-FqForms}, two regular quadratic forms over $\F_q$ are equivalent if and only if they have the same dimension and discriminant. This shows that all regular Hermitian forms of dimension $n+1$ are equivalent.\qedhere
\end{proof}

\begin{corollary} \normalfont In $\PG(n,q^2)$ there is only one Hermitian variety, namely 
\[H=V(x_0^{q+1}+x_1^{q+1}+\dots+x_n^{q+1}).\]
\end{corollary}

\subsection{Classical families of generalised quadrangles}

We can obtain families of generalised quadrangles from quadrics by restricting the incidence structure of points and lines in $\PG(n,q)$ to a given quadric. First we note that, in order to obtain a generalised quadrangle from a quadric $Q$, the largest projective dimension of a linear subspace contained in $Q$ must be $1$ (or equivalently affine dimension at most $2$). Otherwise, there is a projective plane contained in $Q$, and the incidence structure of a projective plane contains triangles.\\

Suppose $Q$ is a quadric given as the zero set of some regular quadratic form $\phi$ on a vector space $V$ over $\F_q$. Let $b$ be the bilinear form associated to $\phi$, then a linear subspace $W$ contained in $Q$ is equivalent to a subspace of $V$ satisfying $\phi(x)=b(x,x)=0$ for all $x\in W$. In the theory of quadratic forms, a space $W$ is said to be \textit{totally isotropic} if and only if $\phi(x)=b(x,x)=0$ for all $x\in W$, equivalently $W\subset W^{\perp}$. So a linear subspace $W$ of a quadric $Q=v(\phi)$ is equivalent to a totally isotropic subspace of the quadratic space $(V,\phi)$. From the inclusion $W\subset W^{\perp}$ it follows that 
\[\dim V=\dim W+\dim W^{\perp}\geq 2\dim W,\]
so a totally isotropic subspace $W$ satisfies $\dim W\leq \dim V/2$. A \textit{maximal totally isotropic} subspace $W$ of $V$ is a totally isotropic subspace of $V$ that is not contained in any other isotropic subspace.
\begin{theorem}[Witt decomposition, 5.11. Chapter 1, \cite{Scharlau-QuadraticHermitian}]\label{thm-IsotropyDimension}
Let $(V,\phi)$ be a regular quadratic space of dimension $n$. Then, the dimension of a maximal totally isotropic subspace in $V$ is $m\leq n/2$ if and only if $V\simeq H_1\oplus\dots\oplus H_m\oplus V_1$, where $H_i\simeq \langle 1,-1\rangle$ and $V_1$ is \textit{anisotropic}. In other words,
\[\phi\simeq \langle 1,\dots,1; -1\dots,-1\rangle\oplus \psi\equiv x_1x_2+\dots+x_{2m-1}x_{2m}+\psi(y_1,\dots,y_{n-2m}),\]
where $\psi(y)\neq 0$ for all $y\in V_1$.
\end{theorem}

\begin{remark}\normalfont The integer $m$ is an invariant of the quadratic space $(V,\phi)$, known as the Witt index of the space.
\end{remark}

\begin{theorem}[cf. Section 3.3.1 \cite{Payne-Thas-GQs}]\label{thm-GQQuadrics} Let $Q\subset \PG(n,q)$ be a quadric containing no linear subspaces of projective dimension $\geq 2$. Then the incidence structure of points in $Q$ and lines of $\PG(n,q)$ contained in $Q$ is a finite generalised quadrangle.
\end{theorem}
\begin{proof}
To show Axiom 1. in Definition \ref{def-GQ}, notice that every line in the quadric $Q$ has exactly $s+1=q+1\geq 3$ points. We now show that if $Q$ contains no linear subspaces of projective dimension $\geq 2$, then there are no triangles in the point-line incidence structure of $Q$. Let $Q=V(\phi)$, and $b$ be the bilinear form corresponding to $\phi$. Then for $x,y\in Q$ with $x$ and $y$ distinct, if the line given by $x$ and $y$ is contained in $Q$, then 
\[0=\phi(\lambda x+\mu y)=\lambda^2\phi(x)+2\lambda\mu b(x,y)+\mu^2\phi(y)=2\lambda\mu b(x,y).\]
So for $\lambda,\mu\in\F_q^{\times}$ and $q$ odd, this implies that $b(x,y)=0$. Therefore, a triangle in $Q$ implies the existence of $x,y,z\in Q$ with $\phi(x)=\phi(y)=\phi(z)=0$, and $b(x,y)=b(x,z)=b(y,z)=0$, but then
\begin{align*}
\phi(\alpha_0x+\alpha_1 y+\alpha_2 z)&= \alpha_0^2\phi(x)+\alpha_1^2\phi(y)+\alpha_2^2\phi(z)\\
&+2\alpha_0\alpha_1b(x,y)+2\alpha_0\alpha_2b(x,z)+2\alpha_1\alpha_2b(y,z)\\
&=0.
\end{align*}
This implies that there is a $\PG(2,q)$ embedded in $Q$, contradicting the assumption that $Q$ does not contain linear subspaces of projective dimension $2$ or larger. Now, to prove Axiom 3. it suffices to show that for a non-collinear pair of line and point $(\ell, z)$ there is at least one point $w$ in $\ell$ collinear to $x$: Let $\ell$ be an arbitrary line in $Q$ and $z$ a point in $Q$ not incident to $\ell$. Let $x,y$ be two distinct points incident to $\ell$, so that every point in $\ell$ can be written as a linear combination of $x$ and $y$. Now, we find a point $w$ in $\ell$ such that $b(z,w)=0$, in this case every point of the type $\lambda z+\mu w$ will be isotropic, implying that $z$ and $w$ share a line contained in $Q$. If either $b(z,x)=0$ or $b(z,y)=0$ we are done. Otherwise, we may assume $b(z,y)\neq 0$, and let $w=x+\alpha y$, where $\alpha=-b(z,x)/b(z,y)$, so that
\begin{align*}
b(z,w)=b(z,x)-\frac{b(z,x)}{b(z,y)}\cdot b(z,y)=0.
\end{align*}
Finally, to show Axiom 2. let $1+\ell(x)$ be the number of lines in $Q$ passing through a point $x$ in the quadric. Then, using Axioms 1. and 3. the total number of points in the is equal to $(1+s)(1+s\ell(x))$. This quantity is independent of $x$, which implies that there is a constant $t$ such that each line in the quadric contains $1+t$ points. For the case where $q$ is even, we refer the reader to \cite{Payne-Thas-GQs}.\qedhere
\end{proof}
Theorem \ref{thm-IsotropyDimension} together with Theorem \ref{thm-ClassificationQuadrics} and Theorem \ref{thm-GQQuadrics} imply that the point-line incidence structure of a quadric yields a generalised quadrangle only for a quadric $Q$ in dimensions $3$, $4$ or $5$ equivalent to 
\begin{enumerate}
\item $V(x_0x_1+x_2x_3)$ in $\PG(3,q)$,
\item $V(x_0^2+x_1x_2+x_3x_4)$ in $\PG(4,q)$, or
\item $V(f(x_0,x_1)+x_2x_3+x_4x_5)$ in $\PG(5,q)$, where $f(x,y)$ is a binary quadratic form inequivalent to $xy$.
\end{enumerate}
These are the quadrics on $\F_q$ which contain no totally isotropic subspaces of projective dimension $\geq 2$. These quadrics turn out to induce generalised quadrangles:
\begin{corollary}\normalfont
Let $q$ be a prime power, then the following are families of generalised quadrangles embedded in $\PG(n,q)$:
\begin{align*}
&\text{Q}(3,q):\  s=q, t=1, v=(q+1)^2, b=2(q+1),\text{ when  } n=3,\\
&\text{Q}(4,q):\  s=t=q, v=b=(q+1)(q^2+1),\text{ when } n=4,\text{ and }\\
&\text{Q}(5,q):\  s=q, t=q^2, v=(q+1)(q^3+1), b=(q^2+1)(q^3+1),\text{ when } n=5.
\end{align*}
\end{corollary}
\begin{proof}
The family $\text{Q}(3,q)$ is given by the quadric $Q=V(x_0x_1+x_2x_3)$ in $\PG(3,q)$. We show that the parameters of $\text{Q}(3,q)$ are as stated. Suppose that $(x_0,x_1,x_2,x_3)\in \F_q^{4}$ is a non-zero solution to the equation
\[x_0x_1+x_2x_3=0.\]
If $x_0=0$, then $x_2=0$ or $x_3=0$ and $x_1$ may be chosen freely. This gives a total of $2(q^2-1)$ possible solutions, since in projective space two vectors are identified if and only if they are non-zero multiples of each other, the total number of points in $\PG(3,q)$ with $x_0=0$ satisfying the equation is $2(q^2-1)/(q-1)=2(q+1)$. 
If $x_0\neq 0$, then letting $x_1=-x_2x_3/x_0$ we have a solution to the equation, so  $x_2$ and $x_3$ can be chosen freely. This gives a total of $(q-1)(q^2-1)$ solutions in $\F_q^4$, which is equivalent to $q^2-1$ solutions in $\PG(3,q)$. Therefore, the number of points in $Q$ is $v=q^2-1+2(q+1)=(q+1)^2$. Since each line in $\PG(3,q)$ comprises $q+1$ points, the number of points in a line contained in $Q$ is also $q+1$, so $s=q$. By Lemma \ref{lemma-GQCounting} we know that $v=(s+1)(st+1)$, so from the knowledge of $s$ and $v$ we may find $t$, we have
\[(q+1)^2=v=(s+1)(st+1)=q^2t+qt+q+1=q(q+1)t+(q+1).\]
From here it follows that $t=1$, and $b=(t+1)(st+1)=2(q+1)$.\\

Similarly, counting the number of points in $Q=V(x_0^2+x_1x_2+x_3x_4)$ is sufficient to determine the parameters of the family $\text{Q}(4,q)$, likewise for $Q=V(f(x_0,x_1)+x_2x_3+x_4x_5)$ and $\text{Q}(5,q)$, where $f(x,y)$ is a binary quadratic form inequivalent to $xy$. For the details of this computation we refer the reader to Theorem 5.1.8 of \cite{Hirschfeld-PGFF}.
\end{proof}

 We conclude this section by mentioning two more families of generalised quadrangles. For details we refer the reader to Chapter 3 of \cite{Payne-Thas-GQs}.

\begin{theorem}[$H$-family, 3.3.1 (ii)\cite{Payne-Thas-GQs}] Let $H$ be a non-degenerate Hermitian variety on $\PG(n,q^2)$ where $n=3$ or $4$. Then the points of $H$ with the lines on $H$ form a generalised quadrangle, denote $H(n,q^2)$, with parameters
\begin{align*}
&H(3,q^2):\ s=q^2,\ t=q,\ v=(q^2+1)(q^3+1),\ b=(q+1)(q^3+1),\text{ and }\\
&H(4,q^2):\ s=q^2,\ t=q^3,\ v=(q^2+1)(q^5+1).
\end{align*}
In either case, $H$ is equivalent to $V(x_0^{q+1}+x_1^{q+1}+\dots+x_n^{q+1})$.

\end{theorem}

Recall that a \textit{symplectic form} on a $k$-vector space $V$ is a function $f:V\times V\rightarrow k$ satisfying
\begin{itemize}
\item[(i)] $f(x+x',y)=f(x,y)+f(x',y)$, and
\item[(ii)] $f(x,y)=-f(y,x)$.
\end{itemize}
In particular $f(x,x)=0$, for all $x\in V$.

\begin{theorem}[$W$-family, 3.3.1 (iii)] The points of $\PG(3,q)$ together with the totally isotropic lines in $\PG(3,q)$ with respect to a symplectic form constitute a generalised quadrangle, denoted $W(q)$, with parameters
\[s=t=q,\ v=b=(q+1)(q^2+1).\]
\end{theorem}

\section{Privacy in GQ-UPIR schemes}
We return to UPIR schemes. To analyse pseudonymity relations in a GQ-UPIR scheme we will require the concept of \textit{hyperbolic lines} on a GQ.
\begin{definition}\normalfont
Let $x$ be a point in a GQ. We denote by $B_1(x)$ the set of points collinear to $x$. If $\mathcal{X}$ is a subset of points of a $GQ$, we write $B_1(\mathcal{X})=\bigcap_{x\in\mathcal{X}}B_1(x)$ for the set of points collinear to every point in $\mathcal{X}$. The set
\[\ssp(\mathcal{X})=B_1(B_1(\mathcal{X}))=\bigcap_{z\in B_1(\mathcal{X})}B_1(z),\]
is called the \textit{span} of $\mathcal{X}$. When $\mathcal{X}=\{x,y\}$ for two non-collinear points $x,y$, the set $\ssp(x,y):=\ssp(\{x,y\})$ is called the \textit{hyperbolic line} defined by $x$ and $y$.\index{hyperbolic line}
\end{definition}
Hyperbolic lines satisfy similar incidence relations to those of ordinary lines in a GQ.
\begin{lemma}[cf. Lemma 16 \cite{UPIR-paper}]\normalfont\label{lemma-HLLineProp} If $a\in \ssp(x,y)$ then $\ssp(a,x)=\ssp(x,y)$.
\end{lemma}
\begin{proof}
 Let $a\in \ssp(x,y)=B_1(B_1(\{x,y\}))$. Since $a$ is collinear to every point in $B_1(x,y)$ we have $B_1(\{x,y\})=B_1(\{a,x,y\})$ and clearly $B_1(\{a,x,y\})\subseteq B_1(a,x)$. Through the point $x$ pass $t+1$ lines, for each such line $\ell$ there is a unique point $z$ in $\ell$ such that $z$ and $y$ are collinear. Therefore for any two non-collinear points $x,y$ we have that $|B_1(\{x,y\})|=t+1$. Therefore,
\[t+1=|B_1(\{x,y\})|=|B_1(a,x,y)|\leq |B_1(a,x)|=t+1.\]
Which shows $B_1(\{x,y\})=B_1(a,x)$, hence $\ssp(x,y)=\ssp(a,x)$.
\end{proof}
\begin{corollary}\normalfont If $|\ssp(x,y)\cap \ssp(w,z)|>1$, then $\ssp(x,y)=\ssp(w,z)$.
\end{corollary}
\begin{proof}
Let $\{a,b\}\subseteq \ssp(x,y)\cap(w,z)$. Then by the lemma above $\ssp(x,y)=\ssp(a,x)=\ssp(a,b)=\ssp(w,a)=
\ssp(w,z)$.
\end{proof}

\begin{table}[H]
\centering
\begin{tabular}{|l|l|l|}
\hline
$Q$ & Order & Span size \\
\hline
$W(q)$,  $q$ odd & $(q, q)$ &$\vert \ssp(x, y)\vert = q + 1$ \\
$Q(4, q)$, $q$ even & $(q, q)$ & $\vert \ssp(x, y)\vert = q + 1$ \\
$Q(4, q)$, $q$ odd & $(q, q)$ &$\vert \ssp(x, y)\vert = 2$ \\
$Q(5, q)$ & $(q, q^2)$ & $\vert \ssp(x, y)\vert = 2$ \\
$H(3, q^2)$ & $(q^2, q)$ & $\vert \ssp(x, y)\vert = q + 1$ \\
$H(4, q^2)$ & $(q^2, q^3)$ & $\vert \ssp(x, y)\vert = q + 1$ \\

\hline
\end{tabular}
\caption{Sizes of hyperbolic lines in the classical generalised quadrangles. Here $q$ is a prime power,  $x$, and $y$ are non-collinear points. See \cite{Payne-Thas-GQs}: Chapter 1 contains the relevant definitions, and the values of $|\ssp(x,y)|$ can be inferred from 2.5.1. and Section 3.3}\label{tab-GQSpans}
\end{table}
In Table \ref{tab-GQSpans}, the structures $W(q)$ and $Q(4,q)$ are dual for $q$ odd, $W(q)\simeq Q(4,q)$ is self-dual for $q$ even, and $H(3,q^2)$ and $Q(5,q)$ are dual.

\begin{proposition}[\cite{UPIR-paper}]\normalfont\label{prop-GQUPIR-1-UE}
In a GQ-UPIR scheme using the unencrypted Protocol \ref{prot-BasicUPIR}, the pseudonymity classes with respect to a single eavesdropper $c$ are singleton classes for users at distance $1$ from $c$, and are of the form $\ssp(c,u)-\{c\}$ for any user $u$ at distance $2$ from $c$.
\end{proposition} 
\begin{proof}
Suppose that $u$ sends sufficiently many linked queries. Then $c$ will observe these linked queries only in the unique message space, or line, shared between $c$ and $u$, and $u$ is the unique user that never acts as a proxy for a linked query in said message space. This shows that $c$ can identify $u$ when $c$ and $u$ share a message space.\\

If $c$ and $u$ do not share a message space, then the GQ axiom implies that for each message space $M$ that $c$ has access to, there is a unique user $u_1$ that shares a message space with $u$. Since $u$ sends sufficiently many linked queries, $c$ will observe every user in $M$ act as a proxy of a linked query except for $u_1$. Therefore, $c$ can identify the set $\mathcal{X}=B_1(c)\cap B_1(u)$. All candidates in $B_1(\mathcal{X})-\{c\}=\ssp(u,c)-\{c\}$ are then pseudonymous, since for every user $v\in\ssp(u,c)-\{c\}$ we have that $\ssp(u,c)-\{c\}=\ssp(v,c)-\{c\}$ by Lemma \ref{lemma-HLLineProp}.
\end{proof}

In particular, a single user can identify every user in the GQ-UPIR scheme if and only if every hyperbolic line of the GQ has size $2$. The infinite families $\text{Q}(4,q)$ with $q$ odd and $\text{Q}(5,q)$ have this property.

\begin{proposition}[\cite{UPIR-paper}]\normalfont\label{prop-D2Pseudonymous} In a GQ-UPIR scheme using the encrypted Protocol \ref{prot-EncryptedUPIR}, all users at distance $2$ from every member of a coalition $\mathcal{C}$ are pseudonymous with respect to $\mathcal{C}$.
\end{proposition}
\begin{proof}
First we consider a single user $c_1$. Then, $c_1$ can identify if the source $u$ of a series of linked queries is at distance $1$ or distance $2$. If $u$ is at distance $1$, then $c_1$ will observed linked queries only in the unique message space shared by $c_1$ and $u$. If $u$ is at distance $2$, then $c_1$ will observe linked queries uniformly at random on all the message spaces it has access to. Since $c_1$ does not observe queries addressed to other users, $c_1$ does not find information about $B_1(u)$, hence no information about the hyperbolic line $\ssp(c_1,u)$. The only information $c_1$ learns is that $d(c_1,u)=2$, and in particular all users at distance $2$ form $c_1$ are in the same pseudonymity class with respect to $c_1$. If $u$ is at distance $2$ from every member of a coalition $\mathcal{C}=\{c_1,\dots,c_m\}$, then the only information that each member $c_i$ observes is a uniform random distribution of linked queries address to each of their message spaces. This would be stil the case if $v$ is the source of the linked queries, provided that $v$ is again at distance $2$ from each coalition member. Therefore, the set of users at distance $2$ from $\mathcal{C}$ is contained in a single pseudonymity class.
\end{proof}

\begin{theorem}[\cite{UPIR-paper}] A GQ-UPIR scheme with $s>1$ and using Protocol \ref{prot-EncryptedUPIR} is secure against coalitions of users of size $O(s^{1-\epsilon})$ for any $\epsilon>0$. Therefore, any such family of GQ-UPIR schemes is secure according to Definition \ref{def-SecureUPIR}.

\end{theorem}
\begin{proof}
Let $\mathcal{C}$ be a coalition of size $O(s^{1-\varepsilon})$. By Proposition \ref{prop-D2Pseudonymous} every user at distance $2$ from $\mathcal{C}$ forms a single pseudonymity class. By Lemma \ref{lemma-GQCounting} the number of users at distance $1$ from a given user in $\mathcal{C}$ is $s(t+1)$. Therefore, we have that the number of users at distance $1$ from a member of $\mathcal{C}$ is at most
\begin{equation}\label{eq-PSBound}
|\mathcal{C}|s(t+1)\leq s^{2-\epsilon}(t+1).
\end{equation}
Again, by Lemma \ref{lemma-GQCounting} the total number of users in the GQ-UPIR scheme is $(s+1)(st+1)$. If $t=1$, then the GQ is a grid, and by Equation \ref{eq-PSBound} the coalition $\mathbb{C}$ is at distance $1$ from at most $O(s^{2-\epsilon})$ users. The users at distance $2$ from $\mathcal{C}$ form a single pseudonymity class, hence the union of all other pseudonymity classes has size at most $O(s^{2-\epsilon})=O(v_{s}^{1-\epsilon})$, where $v_s=(s+1)^2$ is the total number of users in the UPIR scheme. Therefore, a grid GQ-UPIR scheme is secure in the sense of definition \ref{def-SecureUPIR}.\\
Suppose now that $t>1$ then applying Higman's bound (Theorem \ref{thm-HigmanBound}) we find that $s\geq t^{1/2}$. Therefore the number of users at distance one from a coalition $\mathbb{C}$ is at most $s^{2-\epsilon}(t+1)=O(s^{4-2\epsilon})$. The total number of users in the GQ is $v_{s,t}=(s+1)(st+1)=O(s^4)$. Therefore, the GQ-UPIR scheme is secure.
\end{proof}

\begin{corollary}\normalfont \label{cor-GridSecure}For all $\epsilon>0$, there is an $N_{\epsilon}\in\N$ such that any grid GQ-UPIR scheme using Protocol \ref{prot-EncryptedUPIR}, and having $n> N_{\epsilon}$ users, has no identifying sets of size $O(n^{1/2-\epsilon})$. 
\end{corollary}
\begin{proof}
The number of points in a grid GQ of order $(s,1)$ is $n=(s+1)^2$. Let $\mathcal{C}$ be a coalition of size $O(s^{1-\epsilon})=O(n^{1/2-\epsilon})$ where $\epsilon>0$ is arbitrary. Then, the number of users at distance $1$ from a member of $\mathcal{C}$ is at most
\[|\mathcal{C}|s\leq s^{2-\epsilon}=O(n^{1-\epsilon}).\]
Therefore, for $n$ large enough, the set of users at distance $2$ from all members of the coalition $\mathcal{C}$ has more than one element. By Proposition \ref{prop-D2Pseudonymous}, the users at distance $2$ from $\mathcal{C}$ form a pseudonymity class, hence $\mathcal{C}$ is not an identifying set.\qedhere
\end{proof}

Note that the grid GQ-UPIR scheme requires only $2\sqrt{n}$ message spaces, while still achieving security. In contrast to PBD-UPIR schemes, which are insecure and require at least as many message spaces as there are users.\\

In analogy to Corollary \ref{cor-GridSecure} we can show that, among the classical generalised quadrangles, the most secure GQ family is $H(3,q^2)$ which is secure against coalitions of size $O(n^{2/5-\epsilon})$, while the least secure is given by $Q(5,q)$ which is secure against coalitions of size $O(n^{1/4-\epsilon})$.

\begin{research-problem}\normalfont
Consider a UPIR scheme in which the path between a source user $u$ and a proxy $v$ is of a fixed length $t$, which may exceed the diameter of the underlying bipartite graph. The question arises: what is the minimum size of an identifying set in such a scenario? For instance, what is the smallest possible size of an identifying set in a GQ-UPIR scheme where the message passing requires exactly 3 steps? Similarly, what about projective planes UPIR schemes with a restriction of 2 steps?
\end{research-problem}

We can also consider the following generalisation of the problem of finding a minimal identifying set:
\begin{research-problem}\normalfont In a GQ-UPIR scheme, determine the smallest value of $t$ such that the average size of a pseudonymity class, with respect to an arbitrary coalition $\mathcal{C}$ of size $t$, is at most $2$ (or more generally, at most $k$).
\end{research-problem}

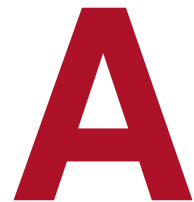
\appendix
\cleardoublepage
\chapter{Generalised Hadamard Matrices and Projective Planes}\label{app-GHMs}
\renewcommand*{\thepage}{A-\arabic{page}}
\setcounter{page}{1}

This appendix is a companion to the survey in Chapter \ref{chap-BHMats}. Here we present known results, but with a new exposition including several concrete examples.\\

There is a close connection between GHMs and projective planes. Namely, one can build a projective plane of order $n$ from a $\GH(n,G)$ where $|G|=n$. In addition projective planes which are obtained from a generalised Hadamard matrix have an astonishingly concise description, instead of requiring a binary matrix of order $n^2+n+1$ we require only an $n\times n$ matrix with entries over $G$. In particular the Fourier construction shows that we can encode a projective plane of order $p$ (for $p$ prime) in a $p\times p$ array. An interesting question arises which is to determine if all projective planes of prime order can be obtained from a GHM. This is related to two well-known open problems

\begin{research-problem}\normalfont Is every $\BH(p,p)$ matrix equivalent to the Fourier matrix $F_p$?
\end{research-problem}
\begin{research-problem}\normalfont
Is every projective plane of prime order Desarguesian?
\end{research-problem}

It was shown in \cite{Hirasaka-UniquenessOfBH} that the existence of a $\BH(p,p)$ which is not isomorphic to $F_p$ gives rise to a non-Desarguesian projective plane, and so if Problem 2 above has an affirmative answer then so does Problem 1. For a nice account on non-Desarguesian projective planes see C. A. Weibel's survey \cite{Weibel-SurveyNonDesarguesian}.\\

We recall below some basic facts about affine and projective planes. A good reference in the subject can be found in the book by Hughes and Piper \cite{Hughes-Piper-ProjectivePlanes} or in Chapter 3 of Dembowski's book \cite{Dembowski}.

\begin{definition}\normalfont\label{Def-AffinePlane} An \textit{affine plane} is an incidence structure consisting of a set of points $\mathcal{P}$ and a set of lines $\mathcal{L}$ such that the following axioms hold
\begin{itemize}
\item[A1.] There is a unique line through every pair of points.
\item[A2.] For any pair $(p,\ell)$ of point $p\in\mathcal{P}$ and line $\ell\in\mathcal{L}$ such that $p$ is not in $\ell$ there is a unique line $\ell'$ through $p$ such that $\ell$ and $\ell'$ have no points in common.
\item[A3.] There are three non-collinear points.
\end{itemize}
\end{definition}\index{affine plane}

If a line of an affine plane has exactly $n$ points then it follows from the axioms that every line has exactly $n$ points, and we say that the \textit{order} of the affine plane is $n$. From the definition of affine planes one can define an equivalence relation $||$ of parallelism of lines. Two lines $\ell$ and $\ell'$ are said to be parallel, denoted by $\ell||\ell'$, if and only if they have no common points. An equivalence class of parallel lines in an affine plane is called a \textit{parallel class} or \textit{pencil}. An affine plane of order $n$ has a partition of its lines into exactly $n+1$ parallel classes.

\begin{example}\normalfont \label{Ex-AffinePlane} Let $\F_q$ be a finite field, where $q$ is a prime power. We define an affine plane by letting the set of points be $\F_q\times \F_q$ and the set of lines consist of all sets of the type
\[\ell_{abc}=\{(x,y)\in \F_q\times\F_q: ax+by=c\}.\]
where $a,b,c\in\F_q$, and either $a$ or $b$ are non-zero. This construction gives an example of an affine plane of order $q$, as it is easy to check that axioms A1-A3 hold and that each line has $q$ points.\\

As a particular example take $q=2$. Then our set of points consists of all four binary tuples $\{00,01,10,11\}$, here written in shorthand notation. The line $\ell_{110}=\{(x,y)\in\F_2\times\F_2: x+y=0\}$ consists of the tuples whose coordinates add to zero, i.e. the two points $00$ and $11$. The full set of lines is given below
\begin{align*}
&\ell_{100}:\{00,01\},\ \ell_{101}:\{10,11\},\\
&\ell_{010}:\{00,10\},\ \ell_{011}: \{01,11\},\\
&\ell_{110}:\{00,11\},\ \ell_{111}:\{01,10\}.
\end{align*}
Note that each row above represents a parallel class of lines.
\end{example}

Let $\mathcal{I}$ be a finite incidence structure consisting of points and lines. The \textit{line-point} incidence matrix of $\mathcal{I}$ is the matrix $M$ with rows indexed by lines and columns indexed by points defined by
\[M_{\ell,p}=\begin{cases}
1 & \text{ if } \ell \text{ passes through the point } p\\
0 & \text{ otherwise }
\end{cases}.
\]
For example, the affine plane of order $2$ that we constructed in Example \ref{Ex-AffinePlane} has the following line-point incidence matrix
\[
\begin{blockarray}{ccccc}
&00 & 01 & 10 & 11\\
\begin{block}{c[cccc]}
\ell_{100} &1 & 1 & 0 & 0\\
\ell_{101} &0 & 0 & 1 & 1\\
\ell_{010} &1 & 0 & 1 & 0\\
\ell_{011} &0 & 1 & 0 & 1\\
\ell_{110} &1 & 0 & 0 & 1\\
\ell_{111} &0 & 1 & 1 & 0\\
\end{block}
\end{blockarray}
\]
\begin{definition}\normalfont\label{Def-ProjectivePlane} A \textit{projective plane} is an incidence structure consisting of a set of points $\mathcal{P}$ and a set of lines $\mathcal{L}$ such that the following axioms hold
\begin{itemize}
\item[P1.] There is a unique line through every pair of points.
\item[P2.] Any two distinct lines have a unique point in common.
\item[P3.] There are four points of which no three lie in the same line. 
\end{itemize}
\end{definition}

\begin{proposition}\normalfont Let $\mathcal{P}$ be a projective plane and assume that a line $\ell$ of $\mathcal{P}$ has exactly $n+1$ points. Then 
\begin{itemize}
\item[(i)] Each line of $\mathcal{P}$ contains exactly $n+1$ points.
\item[(ii)] Each point is on exactly $n+1$ lines.
\item[(iii)] $\mathcal{P}$ consists of $n^2+n+1$ points and $n^2+n+1$ lines.
\end{itemize}
\end{proposition}
\begin{proof}
See Theorem 3.5 on Chapter III of \cite{Hughes-Piper-ProjectivePlanes}.
\end{proof}
\begin{example}\normalfont \label{Ex-ProjectivePlane} Similarly as in Example \ref{Ex-AffinePlane} we can construct a projective plane of order $q$ from any finite field $\F_q$. This time the set of points is the set of one-dimensional vector subspaces $\langle p\rangle$ of $\F_q^3$ where $p$ is a non-zero vector in $\F_q^3$. Lines, in turn, are defined to be the two-dimensional vector subspaces of $\F_q^3$. And a point $\langle p\rangle$ is in a line $\ell$ if and only if $\langle p \rangle$ is a vector-subspace of $\ell$. It is easy to check that axioms P1-P3 are satisfied, and that every line contains exactly $q+1$ points.\\

As a particular example take $q=2$. Since the only multiples of any vector in $\F_2^3$ are the zero vector and itself, the set of points is given by the $7$ non-zero elements of $\F_2^3$ namely $\mathcal{P}=\{001,010,011,100,101,110,111\}$. Any given line is of the type $\{x,y,x+y\}$ where $x,y\in\mathcal{P}$, therefore the complete list of lines is 
\begin{align*}
&\{001,010,011\},\\
&\{001,100,101\},\{001,110,111\},\\
&\{010,100,110\},\{010,101,111\},\\
&\{011,100,111\},\{011,101,110\}.
\end{align*}
Notice that the affine plane of order $2$ constructed in Example \ref{Ex-AffinePlane} can be embedded in this projective plane. Take the mapping $(x,y)\mapsto (1,x,y)$, and notice that each row corresponds to the parallel classes of the affine plane where each line has now an additional point. The line in the first row is incident to all these additional points, in the context of this embedding the line $\{001,010,011\}$ is called the \textit{line at infinity} and the points $001$, $010$ and $011$ are called \textit{points at infinity}. The line-point incidence matrix of this projective plane is given below
\[
\left[
\begin{array}{ccc|cccc}
1 & 1 & 1 & 0 & 0 & 0 & 0\\
\hline
1 & 0 & 0 & 1 & 1 & 0 & 0\\
1 & 0 & 0 & 0 & 0 & 1 & 1\\
0 & 1 & 0 & 1 & 0 & 1 & 0\\
0 & 1 & 0 & 0 & 1 & 0 & 1\\
0 & 0 & 1 & 1 & 0 & 0 & 1\\
0 & 0 & 1 & 0 & 1 & 1 & 0
\end{array}
\right].
\]
Notice that the lower-right block corresponds to the line-point incidence matrix of the affine plane of order $2$ in the previous example.
\end{example}
 The previous example is a hint at the fact that an affine plane is essentially a projective plane with a distinguished line. The general construction is the following:\\
 
 From an affine plane one can obtain a unique projective plane up to isomorphism (See Chapter III of \cite{Dembowski}). We include $(n+1)$ additional points (one for each parallel class) and one additional line incident to each of these new points. These are the so-called points at infinity and line at infinity respectively. If $M$ is the incidence matrix of our projective plane then up to a re-indexing of the lines $M$ has block shape 
\[
M=\left[
\begin{array}{c}
M_0\\
\hline
M_1\\
\hline
\vdots\\
\hline
M_n
\end{array}
\right]
\]
where the rows of each $M_i$ are indexed by lines in the same parallel class. Under this assumption the corresponding incidence matrix for the projective plane is given as
\[
\left[
\begin{array}{c|c}
\mathbf{1}_{n+1} &\mathbf{0}_{n^2}\\
\hline
R_0 & M_0\\
\hline
R_1 & M_1\\
\hline
\vdots &\vdots\\
\hline
R_n & M_n
\end{array}
\right],
\]
where $\mathbf{1}_m$ and $\mathbf{0}_m$ represent the all-ones and all-zeroes vector of length $m$ respectively, and $R_i$ is the $n\times (n+1)$ rectangular matrix whose $i$-th column is the all-ones vector and every other entry is zero (see the matrix in Example \ref{Ex-ProjectivePlane}).\\

Conversely, given a projective plane $\mathcal{P}$ we may choose a line, say $\ell_{\infty}$, and construct an affine plane $\mathcal{A}$ by taking as set of points all points which are not incident to $\ell_{\infty}$. Two points in $\mathcal{A}$ are defined to be incident if and only if they are incident in $\mathcal{P}$. If $N$ is the point-line incidence matrix of $\mathcal{P}$, then up to a re-indexing of the points of $\mathcal{P}$ we may assume that the first row of $N$ is given by $(\mathbf{1}_{n+1}|\mathbf{0}_{n^2})$. Therefore taking the submatrix of $N$ consisting of the last $n^2$ columns of all rows but the first we obtain the line-point incidence matrix of an affine plane. Two affine planes obtained in this manner by taking different choices of $\ell_{\infty}$ may not be isomorphic, see \cite{Dembowski}.\\

The following definition is taken from Bruck's paper \cite{Bruck-Nets-I} although the study of nets started much earlier, a good general reference for finite geometry can be found in Dembowski \cite{Dembowski}, for connections of nets to group theory see the article by Baer \cite{Baer-NetsAndGroups}.

\begin{definition} \normalfont \label{Def-Net} Let $r$ and $n$ be positive integers with $r\geq 3$. An $r$-\textit{net} of order $n$ is an incidence structure consisting of a set of lines $\mathcal{L}$ and a set of points $\mathcal{P}$ satisfying the following axioms

\begin{itemize}
\item[(N1)] The set of lines $\mathcal{L}$ contains $r$ non-empty classes $\mathcal{L}_1,\dots,\mathcal{L}_r$.
\item[(N2)] Two lines $a\in\mathcal{L}_i$ and $b\in \mathcal{L}_j$ in distinct classes $i\neq j$ have a unique common point.
\item[(N3)] Each point $p\in\mathcal{P}$ is in a unique line $\ell\in \mathcal{L}_i$ for every class $i$.
\item[(N4)] There is a line with exactly $n$ distinct points.
\end{itemize}
\end{definition}

\begin{proposition}\normalfont \label{Net-Prop} Let $\mathcal{N}$ be an $r$-net of order $n$. Then
\begin{itemize}
\item[(i)] Every line of $\mathcal{N}$ has exactly $n$ distinct points.
\item[(ii)] Every class $\mathcal{L}_i$ of lines consists of $n$ distinct lines.
\item[(iii)] $\mathcal{N}$ has exactly $n^2$ points and exactly $rn$ lines.
\item[(iv)] $n=1$ or $r\leq n+1$.
\end{itemize}
\end{proposition}

\begin{proof}
The proof consists of a straightforward application of the axioms. The reader is invited to prove these facts for themselves, but we include a proof here for completeness. Notice that if $n=1$ the only possible $r$-net consists of $r$ lines through a point and claims (i)-(iv) follow trivially, hence we may assume that $n>1$.\\

To prove (i) let $\ell_0$ be a line with exactly $n$ points. Assume $\ell_0$ is in the parallel class $\mathcal{L}_i$ and let $\ell$ be an arbitrary line in a distinct parallel class, say $\mathcal{L}_j$. 
Then $\ell$ and $\ell_0$ meet at a unique point $t$. Let $\mathcal{L}_{k}$ be a third class distinct from both $\mathcal{L}_i$ and $\mathcal{L}_j$, this class exists by the assumption that $r\geq 3$. Now for every point $p\in\ell_0$ distinct from $t$, there is by (N3) a unique line $\ell_p$ through $p$ in $\mathcal{L}_k$ and $\ell_p$ meets $\ell$ at a unique point $q$. Now the $p\mapsto q$ for $p\neq t$ extended by $t\mapsto t$, is an injective map from the points of $\ell_0$ to the points of $\ell$. Reversing the roles of $\ell$ and $\ell_0$ we find that $\ell$ has exactly $n$ distinct points. This shows that the lines of all parallel classes distinct from $\mathcal{L}_i$ have exactly $n$ points. Following the same argument with a line not in $\mathcal{L}_i$ it follows that all lines in $\mathcal{L}_i$ have exactly $n$ points as well.\\

To prove (ii) let $\mathcal{L}_i$ be an arbitrary class of lines, and $\ell$ be a line in a class $\mathcal{L}_j$ distinct from $\mathcal{L}_i$. Then by (i) there are exactly $n$ points in $\ell$, for each point $p$ in $\ell$ there is one and only one line in  $\ell_p\in \mathcal{L}_i$ passing through $p$, so $\mathcal{L}_i$ consists of at least $n$ distinct lines. Conversely if $\ell'$ is a line of $\mathcal{L}_i$ then $\ell'$ meets $\ell$ at a unique point, so by (N2) the number of lines of $\mathcal{L}_i$ is at most $n$. \\

Claim (iii) is a straightforward consequence of (i) and (ii). Let $p$ be an arbitrary point of $\mathcal{N}$ and $\mathcal{L}_1$ and $\mathcal{L}_2$ be two distinct classes of lines, then by (N3) there is a unique line $\ell_i\in \mathcal{L}_1$ and unique line $\ell_j\in \mathcal{L}_2$ each passing through $p$. This establishes an injection $p\mapsto (i,j)$ from the points of $p$ into tuples of integers from $1$ to $n$, so $\mathcal{N}$ has at least $n^2$ points. Conversely given $(i,j)\in\{1,\dots, n\}^2$ the lines $\ell_i\in \mathcal{L}_1$ and $\ell_j$ in $\mathcal{L}_2$ meet at a unique point of $\mathcal{N}$ by (N2), therefore the number of points of $\mathcal{N}$ is exactly $n^2$. Clearly there is a total of $rn$ lines in $\mathcal{N}$ since each class of lines contains exactly $n$ lines, and two such classes must necessarily be disjoint.\\

Finally to prove (iv) we know since $n>1$ that there are at least two distinct lines $\ell_0$ and $\ell_1$ in a class $\mathcal{L}_i$. If $p$ is an arbitrary point of $\ell_0$ then there are $r-1$ lines not in $\mathcal{L}_i$ which pass through $p$, and by (N2) each of these lines meet $\ell_1$ at a unique point. By (i) this implies that $r-1\leq n$, or equivalently $r\leq n+1$.
\end{proof}

Notice that an affine plane of order $n$ satisfies the net axioms N1-N4, with classes $\mathcal{L}_1,\dots,\mathcal{L}_{n+1}$ consisting of the parallel classes of the plane. Axioms (N1), (N3) and (N4) follow easily. It suffices to show (N2), so let $\ell_1\in \mathcal{L}_1$ and $\ell_2\in \mathcal{L}_2$ be two lines in distinct parallel classes $\mathcal{L}_i$ and $\mathcal{L}_j$. Then $\ell_1$ and $\ell_2$ have at least one point in common. If they have two points in common, say $x$ and $y$, then by A1 $\ell_1=\ell_2$ which is impossible since $\mathcal{L}_i$ is disjoint to $\mathcal{L}_j$. Therefore two lines in distinct parallel classes meet in \textit{exactly} one point.
 Hence affine planes are equivalent to $(n+1)$-nets of order $n$.\\

If $\mathcal{N}$ is an $r$-net then we may choose a relabelling of the lines of $\mathcal{N}$ so that the first $r$ rows of $N$ correspond to lines in $\mathcal{L}_1$, rows $r+1$ to $2r$ of $N$ correspond to lines in $\mathcal{L}_2$, and so on. This can be done by multiplication by a permutation matrix by the left. It follows that a $01$ matrix $N$ of shape $rn\times n^2$ is the point-line incidence matrix of an $r$-net of order $n$ if and only if there is a permutation matrix $Q$ of order $rn$ such that
\[NN^{\intercal}=Q^{\intercal}
\left[
\begin{array}{c|c|c}
nI_n &\dots & J_n\\
\hline
\vdots & \ddots &\vdots\\
\hline
J_n &\dots & nI_n
\end{array}
\right]Q.
\]

An $r$-net $\mathcal{N}$ of order $n$ meets the upper bound $r\leq n+1$ with equality if and only if $\mathcal{N}$ is an affine plane of order $n$. In other words an affine plane is equivalent to an $(n+1)$-net of order $n$. So in a way, the $(n+1)$-net structure captures all the essential information contained in the projective plane. In particular the value
\[r(n)=\max \{r: \text{ there is an } r\text{-net of order } n\}\]
quantifies how close one can get to building an affine plane (and in turn projective plane) of order $n$.\\

Now we introduce \textit{mutually orthogonal Latin squares} or \textit{MOLS}. These objects are equivalent to nets, yet despite the fact that the description of nets is geometric in nature Latin squares can be seen as purely combinatorial. Here we establish this relationship by interpreting permutations of $n$ objects as $n\times n$ permutation matrices and vice-versa. We refer the reader to Chapter III-3 of the Handbook of Combinatorial Designs \cite{HandbookOfDesigns} for more on the relationship between nets, MOLS and other equivalent objects.

\begin{definition}\normalfont\label{Def-LatinSquare} A \textit{Latin Square} of order $n$ is an $n\times n$ array $L$ of symbols taken from the set $[n]=\{1,2,\dots,n\}$ such that every symbol occurs exactly once in each row and column of $N$.\index{Latin square}
\end{definition}
Notice that by definition, every row and column of a Latin square consists of a permutation of the elements of $[n]$.
\begin{definition}\normalfont Two $n\times n$ arrays $L$ and $R$ with symbols taken from the set $[n]$ are  \textit{orthogonal} if the list of tuples $(L_{ij},R_{ij})$, $1\leq i,j\leq n$ contains every element of $[n]\times [n]$ exactly once. A set of Latin squares $\{L_1,L_2,\dots,L_m\}$ such that $L_i$ and $L_j$ are orthogonal for each $i\neq j$ is called a set of \textit{mutually orthogonal Latin squares}, or MOLS.
\end{definition}

\begin{example}\normalfont
The following arrays form a pair of orthogonal Latin Squares of order $3$
\[
\begin{array}{ccc}
1 & 3 & 2\\
2 & 1 & 3\\
3 & 2 & 1\\
\end{array}\text{\hspace{60pt} }
\begin{array}{ccc}
1 & 2 & 3\\
2 & 3 & 1\\
3 & 1 & 2
\end{array}
\]
Below is a set of four MOLS of order $5$,

\[
\begin{array}{ccccc}
2&3&4&5&1\\
3&4&5&1&2\\
4&5&1&2&3\\
5&1&2&3&4\\
1&2&3&4&5
\end{array}\text{\hspace{14pt}}
\begin{array}{ccccc}
3&5&2&4&1\\
4&1&3&5&2\\
5&2&4&1&3\\
1&3&5&2&4\\
2&4&1&3&5
\end{array}\text{\hspace{14pt}}
\begin{array}{ccccc}
4&2&5&3&1\\
5&3&1&4&2\\
1&4&2&5&3\\
2&5&3&1&4\\
3&1&4&2&5
\end{array}\text{\hspace{14pt}}
\begin{array}{ccccc}
5&4&3&2&1\\
1&5&4&3&2\\
2&1&5&4&3\\
3&2&1&5&4\\
4&3&2&1&5
\end{array}
\]
\end{example}

In the theorem below we will establish a bijection between MOLS and nets. A concise geometric proof of the construction of MOLS from nets can be found in Chapter 3 of \cite{Dembowski}. The proof we present here is longer yet it has the advantage of being completely explicit and computational. It also highlights the interplay between linear representations and permutation representations, which will be useful later on in establishing the connection between GHMs, nets and MOLS.

\begin{theorem}\label{Net-MOLS} An $r$-net of order $n$ exists if and only if there is a set of $r-2$ MOLS. In particular projective and affine planes of order $n$ exist if and only if there exist $n-1$ mutually orthogonal Latin squares.
\end{theorem}
\begin{proof}
 Let $\mathcal{N}$ be an $r$-net of order $n$, and let $N$ be its line-point incidence matrix. Then there is a choice of indexing of lines of $N$ such that
\[NN^{\intercal}=
\left[
\begin{array}{c|c|c}
nI_n &\dots & J_n\\
\hline
\vdots & \ddots &\vdots\\
\hline
J_n &\dots & nI_n
\end{array}.
\right]
\]
This Gram matrix equation is preserved under column permutations. Then by Proposition \ref{Net-Prop} we may permute the columns of $N$ (or equivalently relabel the points of $\mathcal{N}$) to assume that $N$ has the following block-shape
\[
N=\left[
\begin{array}{c|c|c|c}
E_1 & E_2 & \dots & E_n\\
\hline
N_{01} & N_{02} &\dots & N_{0n}\\
\hline
N_{11} & N_{12} & \dots & N_{1n}\\
\hline
\vdots & \vdots &  &\vdots\\
\hline
N_{r-2,1} & N_{r-2,2}&\dots&N_{r-2,n} 
\end{array}
\right],
\]
where each $N_{ij}$ is an $n\times n$ matrix and $E_i$ denotes the $n\times n$ matrix whose $i$-th row is the all-ones vector and every other row is the all-zeroes vector. In other words, since a parallel class of lines partitions the point set, we re-indexed the points so that the first line of the first parallel class contains the first $n$ points, the second line contains points $n+1$ to $2n$ and so on, the matrices $E_i$ represent these incidences.\\

We now show that each $N_{ij}$ is a permutation matrix. Notice that $E_kM$ is a matrix whose $k$-th row is given by the column sums of $M$ and every other row has zero entries. From the equation for $NN^{\intercal}$ we know that $\sum_{j} E_jN_{ij}^{\intercal}=J_n$. Therefore the row sum of every $N_{ij}$ is $1$, i.e. there is a unique non-zero entry in each row of $N_{ij}$. And since each $N_{ij}$ encodes part of the incidences of a line in class $\mathcal{L}_i$ the inner product of two distinct rows of $N_{ij}$ must be zero. This shows that $N_{ij}N_{ij}^{\intercal}=I_n$, so $N_{ij}$ is a permutation matrix.\\

Since each $E_j$ is invariant under permutations of columns, we may further permute the columns of $N$ so that $N_{0j}=I_n$ for all $j$. We showed so far that $N$ can be put in the shape
\[
N=\left[
\begin{array}{c|c|c|c}
E_1 & E_2 & \dots & E_n\\
\hline
I_n & I_n &\dots & I_n\\
\hline
N_{11} & N_{12} & \dots & N_{1n}\\
\hline
\vdots & \vdots &  &\vdots\\
\hline
N_{r-2,1} & N_{r-2,2}&\dots&N_{r-2,n} 
\end{array}
\right],
\]
where each $N_{ij}$ is a permutation matrix of order $n$.  Let $\sigma_{ij}$ be the permutation given by $N_{ij}$. Again from the Gram matrix equation for $NN^{\intercal}$ we find that taking the inner product of the row block $(I_n|I_n|\dots|I_n)$ with any of the subsequent row blocks gives us $\sum_j N_{ij}=J_n$. This implies that if $\sigma_{ij}(k)=\sigma_{ij'}(k)$ then $j=j'$ otherwise the $(\sigma_{ij}(k),k)$ coordinates of $N_{ij}$ and $N_{ij'}$ would both be $1$ contradicting $\sum_{j} N_{ij}=J_n$. This implies that we can build a Latin square $L_k$ from each row block $(N_{k1}|N_{k2}|\dots|N_{kn})$ by letting the  $(i,j)$ entry of $L_k$ be $\sigma_{k,i}(j)$. In other words, the $i$-th row of $L_k$ is given by the expression of $N_{k,i}$ as a permutation $n$ elements.\\

Finally we show that the Latin squares in the set $\{L_1,\dots,L_{r-2}\}$ are mutually orthogonal. To see this notice that $\sum_{i}N_{ki}N_{k'i}^{\intercal}=J_n$, and that the $(r,s)$ coordinate of $N_{ki}N_{k'i}^{\intercal}$ is equal to one if and only if $\sigma_{ki}\sigma_{k'i}^{-1}(s)=r$. The latter is equivalent to the existence of an $j$ such that $\sigma_{ki}(j)=r$ and $\sigma_{k'i}(j)=s$, and in terms of the Latin squares this means that the pair of  $(i,j)$ coordinates of $L_k$ and $L_{k'}$ is $((L_k)_{ij},(L_{k'})_{ij})=(r,s)$. Now from $\sum_{i}N_{ki}N_{k'i}^{\intercal}=J_n$ it follows that all possible pairs occur and so $L_k$ and $L_{k'}$ are mutually orthogonal. This concludes the proof that the existence of an $r$-net of order $n$ gives the existence of a set of $r-2$ MOLS of order $n$.\\

Conversely let $\{L_1,\dots,L_{r-2}\}$ be a set of MOLS of order $n$. We will construct a net by reversing the process above, namely we identify the $j$-th row of $L_i$ as a permutation $\sigma_{ij}$ of $n$ elements. Now, let $N_{ij}$ be the permutation matrix associated to $\sigma_{ij}$ and define
\[N:=\left[
\begin{array}{c|c|c|c}
E_1 & E_2 & \dots & E_n\\
\hline
I_n & I_n &\dots & I_n\\
\hline
N_{11} & N_{12} & \dots & N_{1n}\\
\hline
\vdots & \vdots &  &\vdots\\
\hline
N_{r-2,1} & N_{r-2,2}&\dots&N_{r-2,n} 
\end{array}
\right].\]
We show that $N$ is the line-point incidence matrix of an $r$-net of order $n$. Since the $N_{ij}$ are permutation matrices it follows immediately that $\sum_j E_j N_{ij}=J_n$ for all $i$. From the fact that $L_i$ is a Latin square it follows that $\sum_{j}N_{ij}=J_n$, so the product of the row block $(I_n|\dots|I_n)$ with any other block is $J_n$. And since $L_k$ and $L_{k'}$ are MOLS it follows that $\sum_{i}N_{ki}N_{k'i}^{\intercal}=J_n$, indeed following the same argument as above, the $(r,s)$ entry of $N_{ki}N_{k'i}$ is equal to one if and only if there is a unique $j\in [n]$ such that $\sigma_{ki}(j)=r$ and $\sigma_{k'i}(j)=s$. Hence, the orthogonality assumption implies that for given $(r,s)$ there is a unique value of $i$ and $j$ such that $(N_{ki},N_{k'i}^{\intercal})_{rs}=1$, in other words $\sum_{i}N_{ki}N_{k'i}^{\intercal}=J_n$. This shows that $N$ is the incidence matrix of an $r$-net of order $n$.
\end{proof}
In views of Theorem \ref{Net-MOLS} we have that $r(n)=N(n)+2$ where
\[N(n)=\max\{N: \text{ there is an } N \text{-set of mutually orthogonal Latin squares of order } n\}.\]
The determination of the number $N(n)$ has received much attention in the literature, see the surveys by Colbourn and Dinitz on constructions for MOLS \cite{Colbourn-Dinitz-MakingTheMOLSTable, Colbourn-Dinitz-MOLSSurvey}. The story begins with Euler's 36 officers problem, which asks whether or not $N(6)\geq 2$. Euler tackled this problem and showed that $N(2m+1)\geq 2$ and $N(4m)\geq 2$. Being unable to find a set of two MOLS of order $6$, Euler conjectured that $N(4m+2)=1$ for all $m$. Tarry \cite{Tarry-I,Tarry-II} showed with a proof by exhaustion that indeed $N(6)=1$ however, Euler's conjecture turned out to be false as shown by Bose and Shrikhande \cite{Bose-Shrikhande-EulerConjecture} when they constructed two MOLS of order $22$. More strongly the two last authors together with E. T. Parker showed that Euler's conjecture is false for all $n\geq 10$ \cite{Bose-Shrikhande-Parker-FurtherResultsEulerConjecture}. Shortly after, Chowla, Erd\H{o}s and Straus \cite{Chowla-Erdos-Straus-MOLS} showed using sieve methods that  $N(n)>\frac{1}{3}n^{1/91}$, for $n$ sufficiently large.  Subsequent improvements to this lower bound appeared in the literature, with a breakthrough result by R. M. Wilson in \cite{Wilson-MOLS}, where he improved the bound to $N(n)\geq n^{1/17}-2$. We include below a short summary of results,
\begin{itemize}
\item $N(n)\leq n-1$.
\item $N(q)=q-1$ if $q$ is a prime power.
\item $N(nm)\geq \min(N(n),N(m))$.
\item $N(n)>1$ for all $n\neq 1,2,6$.
\item $N(n)\geq 6$ for all $n>90$ \cite{Wilson-MOLS}.
\item $N(n)\geq n^{1/17}-2$ for $n$ large enough \cite{Wilson-MOLS}.
\end{itemize}

We illustrate below the equivalence between nets and MOLS with some examples
\begin{example}\normalfont
A $4$-net of order $3$ is given by the following point-line incidence matrix
\[
N=\left[
\begin{array}{ccc|ccc|ccc}
1 & 1 & 1 & 0 & 0 & 0 & 0 & 0 & 0\\
0 & 0 & 0 & 1 & 1 & 1 & 0 & 0 & 0\\
0 & 0 & 0 & 0 & 0 & 0 & 1 & 1 & 1\\
\hline
1 & 0 & 0 & 1 & 0 & 0 & 1 & 0 & 0\\
0 & 1 & 0 & 0 & 1 & 0 & 0 & 1 & 0\\
0 & 0 & 1 & 0 & 0 & 1 & 0 & 0 & 1\\
\hline
1 & 0 & 0 & 0 & 0 & 1 & 0 & 1 & 0\\
0 & 1 & 0 & 1 & 0 & 0 & 0 & 0 & 1\\
0 & 0 & 1 & 0 & 1 & 0 & 1 & 0 & 0\\
\hline
1 & 0 & 0 & 0 & 1 & 0 & 0 & 0 & 1\\
0 & 1 & 0 & 0 & 0 & 1 & 1 & 0 & 0\\
0 & 0 & 1 & 1 & 0 & 0 & 0 & 1 & 0\\
\end{array}
\right]
\]
The reader may check that $NN^{\intercal}$ has diagonal blocks $3I_3$ and off-diagonal blocks $J_3$. Now denote,
\[P_1=
\begin{bmatrix}
0 & 0 & 1\\
1 & 0 & 0\\
0 & 1 & 0
\end{bmatrix}\text{ and }
P_2=\begin{bmatrix}
0 & 1 & 0\\
0 & 0 & 1\\
1 & 0 & 0
\end{bmatrix}
\]
so that we may rewrite
\[N=\left[
\begin{array}{c|c|c}
E_1 & E_2 & E_3\\
\hline
I_3 & I_3 & I_3\\
\hline
I_3 & P_1 & P_2\\
\hline
I_3 & P_2 & P_1
\end{array}
\right].
\]
The permutation matrices $I_3$, $P_1$ and $P_2$ correspond to following the permutations written in two-line notation:
\[
\sigma_0=\begin{pmatrix}
1 & 2 & 3\\
1 & 2 & 3
\end{pmatrix},\ 
\sigma_1=\begin{pmatrix}
1 & 2 & 3\\
2 & 3 & 1
\end{pmatrix}\text{ and }
\sigma_2=
\begin{pmatrix}
1 & 2 & 3\\
3 & 1 & 2
\end{pmatrix}.\]
 Now reading the second and third row of $N$ by blocks we build a pair of MOLS by writing the second line of the permutations associated to each block:
\[L_1=
\begin{bmatrix}
\sigma_0\\
\sigma_1\\
\sigma_2
\end{bmatrix}=
\begin{bmatrix}
1 & 2 & 3\\
2 & 3 & 1\\
3 & 1 & 2
\end{bmatrix}, \text{ and }
L_2=
\begin{bmatrix}
\sigma_0\\
\sigma_2\\
\sigma_1
\end{bmatrix}=
\begin{bmatrix}
1 & 2 & 3\\
3 & 1 & 2\\
2 & 3 & 1
\end{bmatrix}.
\]
Note that the subgroup of $S_3$ generated by $\sigma_1$ and $\sigma_2$ is isomorphic to $C_3$ and that letting $G=\langle x: x^3=1\rangle \simeq C_3$ then
\[H=\begin{bmatrix}
1 & 1 & 1\\
1 & x & x^2\\
1 & x^2 & x
\end{bmatrix}\]
is a $\GH(3,G)$.\\

Conversely, consider the following set of three MOLS of order four

\[L_1=\begin{bmatrix}
1 & 2 & 3 & 4\\
2 & 1 & 4 & 3\\
3 & 4 & 1 & 2\\
4 & 3 & 2 & 1
\end{bmatrix},\ 
L_2=
\begin{bmatrix}
1 & 2 & 3 & 4\\
4 & 3 & 2 & 1\\
2 & 1 & 4 & 3\\
3 & 4 & 1 & 2
\end{bmatrix},\text{ and }
L_3=
\begin{bmatrix}
1 & 2 & 3 & 4\\
3 & 4 & 1 & 2\\
4 & 3 & 2 & 1\\
2 & 1 & 4 & 3
\end{bmatrix}.
\]
If we denote 
\[\iota=\begin{pmatrix}
1 & 2 & 3 & 4\\
1 & 2 & 3 & 4
\end{pmatrix},\ 
\alpha=\begin{pmatrix}
1 & 2 & 3 & 4\\
2 & 1 & 4 & 3
\end{pmatrix},\ 
\beta=\begin{pmatrix}
1 & 2 & 3 & 4\\
3 & 4 & 1 & 2
\end{pmatrix}, \text{ and }
\gamma=\begin{pmatrix}
1 & 2 & 3 & 4\\
4 & 3 & 2 & 1
\end{pmatrix},\]
then we can write $L_1=(\iota,\alpha,\beta,\gamma)^{\intercal}$, $L_2=(\iota,\gamma,\alpha,\beta)^{\intercal}$, and $L_3=(\iota,\beta,\gamma,\alpha)^{\intercal}$. The permutation matrices associated to $\alpha,\beta$ and $\gamma$ are 
\[A=\begin{bmatrix}
0 & 1 & 0 & 0\\
1 & 0 & 0 & 0\\
0 & 0 & 0 & 1\\
0 & 0 & 1 & 0
\end{bmatrix},\ 
B=
\begin{bmatrix}
0 & 0 & 1 & 0\\
0 & 0 & 0 & 1\\
1 & 0 & 0 & 0\\
0 & 1 & 0 & 0
\end{bmatrix},\text{ and }
C=\begin{bmatrix}
0 & 0 & 0 & 1\\
0 & 0 & 1 & 0\\
0 & 1 & 0 & 0\\
1 & 0 & 0 & 0
\end{bmatrix},
\]
respectively. So if we let 
\[N=\left[
\begin{array}{c|c|c|c}
E_1 & E_2 & E_3 & E_4\\
\hline
I_4 & I_4 & I_4 & I_4\\
\hline
I_4 & A & B & C\\
\hline
I_4 & C & A & B\\
\hline
I_4 & B & C & A
\end{array}
\right],\]
then $N$ is the line-point incidence matrix of a $5$-net of order $4$. Note that the subgroup of $S_4$ generated by $\alpha,\beta,$ and $\gamma$ is isomorphic to $C_2\times C_2$, and that letting $G=\langle a,b: a^2=b^2=1,\ ab=ba\rangle\simeq C_2\times C_2$, then
\[H=\begin{bmatrix}
1 & 1 & 1 & 1\\
1 & a & b & c\\
1 & c & a & b\\
1 & b & c & a
\end{bmatrix}
\]
where $c=ab=ba$ is a $\GH(4,G)$.
\end{example}

We already hinted in the previous examples how generalised Hadamard matrices are connected to nets and MOLS. Here we make this connection precise via the regular representation of the group $G$. This representation can be thought of as a linear representation $\rho: G\rightarrow GL_n(\C)$ or as a permutation representation $\rho: G\rightarrow S_n$. The former will give a net structure, while the latter gives us a set of MOLS.

\begin{theorem} \label{Thm-GHM-Nets} Let $G$ be a group of order $n$. If there is a $\GH(n,G)$ then there is an $(n+1)$-net of order $n$.
\end{theorem}
\begin{proof}
Let $\rho$ be the regular representation of $G$. Notice that for every $g\in G$, $\rho(g)$ is a permutation matrix and so $\rho(g)^{\intercal}=\rho(g)^{-1}=\rho(g^{-1})$. If $H$ is a $\GH(n,G)$, we define a  block-matrix $M$ whose $(i,j)$-th block is given by $\rho(h_{ij})$

\[M=\left[
\begin{array}{c|c|c}
\rho(h_{11}) &\dots & \rho(h_{1n})\\
\hline
\vdots & \ddots &\vdots\\
\hline
\rho(h_{n1})&\dots &\rho(h_{nn})
\end{array}
\right]\]

We compute $MM^{\intercal}$ by blocks, and we find that the block $(i,i)$ of $MM^{\intercal}$ is given by 
\[\sum_{j} \rho(h_{ij})\rho(h_{ij})^{\intercal}=\sum_j\rho(h_{ij})\rho(h_{ij}^{-1})=\sum_j\rho(h_{ij}h_{ij}^{-1})=\sum_j\rho(1_G)=nI_n.\]
On the other hand the block $(i,j)$ with $i\neq j$ of $MM^{\intercal}$ is given by
\[\sum_k\rho(h_{ik})\rho(h_{jk})^{\intercal}=\sum_k\rho(h_{ik}h_{jk}^{-1})=\sum_{g\in G}\rho(g)=J_n.\]
Therefore 
\[MM^{\intercal}=
\left[
\begin{array}{c|c|c}
nI_n &\dots & J_n\\
\hline
\vdots & \ddots &\vdots\\
\hline
J_n &\dots & nI_n
\end{array}
\right],\]
and this implies that $M$ is the line-point incidence matrix of an $n$-net of order $n$ where every row-block $(\rho(h_{i1})|\dots|\rho(h_{in}))$ corresponds to a class $\mathcal{L}_i$ of parallel lines. We can construct from $M$ an $(n+1)$-net of order $n$ by introducing an additional class of lines $\mathcal{L}_0$. Let $E_i$ be the $n\times n$ matrix whose $i$-th row is the all-ones vector and every other row consists of the all-zeroes vector. Then we define
\[N=\left[
\begin{array}{ccc}
E_1 & \dots & E_n\\
\hline
&M
\end{array}
\right]=
\left[
\begin{array}{c|c|c}
E_1 & \dots & E_n\\
\hline
\rho(h_{11}) &\dots & \rho(h_{1n})\\
\hline
\vdots & \ddots &\vdots\\
\hline
\rho(h_{n1})&\dots &\rho(h_{nn})
\end{array}
\right]
\]
Clearly $E_k\rho(h_{ij})^{\intercal}=E_k$ as each $E_k$ is invariant under a permutation of columns. Also $E_kE_k^{\intercal}=nE_{kk}$ where $E_{kk}$ is the $n\times n$ matrix with a one in its coordinate $(k,k)$ and zeroes everywhere else, so $\sum_k E_kE_k^{\intercal}=nI_n$ and of course $\sum_k E_k=J_n$. Therefore the $(n+1)n\times (n+1)n$ Gram matrix of $N$ is
\[NN^{\intercal}=
\left[
\begin{array}{c|c|c}
nI_n &\dots & J_n\\
\hline
\vdots & \ddots &\vdots\\
\hline
J_n &\dots & nI_n
\end{array}
\right]\]
and hence $N$ is the line-point incidence matrix of an $(n+1)$-net of order $n$, equivalently $N$ is the line-point incidence matrix of an affine plane of order $n$.
\end{proof}

\begin{corollary}Let $G$ be a group of order $n$, then the existence of a $\GH(n,G)$ implies the existence of a set of $n-1$ MOLS.
\end{corollary}
\begin{proof}
This follows directly from Theorem \ref{Thm-GHM-Nets} and Theorem \ref{Net-MOLS}.
\end{proof}

When Drake introduced Generalised Hadamard matrices in \cite{Drake}, he gave the following generalisation of nets 
\begin{definition}\normalfont \label{Def-MuNet} An $(s,r,\mu)$-net is an incidence structure consisting of $v=s^2\mu$ points and $b=sr$ blocks such that
\begin{itemize}
\item [$\mu$-N1.] The set of blocks is partitioned into $r$ non-empty parallel classes $\mathcal{B}_1,\dots,\mathcal{B}_r$,
\item[$\mu$-N2.] Two blocks in distinct parallel classes have $\mu$ points in common.
\item[$\mu$-N3.] Each point is in a unique block of $\mathcal{B}_i$ for each class $i$.
\item[$\mu$-N4.] Each block consists of $k=s\mu$ points. 
\end{itemize}
\end{definition}
Notice that an $r$-net of order $n$ is simply an $(n,r,1)$-net. We remark as well that $(s,r,\mu)$-nets are also known in the literature as \textit{affine resolvable balanced incomplete block designs} or ARBIBDs. To see this we note that ARBIBDs are $2$-$(v,k,\lambda)$ designs admitting a partition of the block set into $r$ parallel classes, with $r=k+\lambda$. It follows from this that two blocks in distinct classes have $\mu$ points in common, see Theorem 5.21  Stinson's book \cite{Stinson-Designs}. From here it follows easily that both definitions are equivalent.

\section{Orthogonal arrays and Shrikhande's Construction}

\begin{definition} \normalfont \label{Def-OrthogonalArray} An orthogonal array of $s$ symbols, $r$ constraints,  index $\mu$ and strength $t$ denoted $\OA_t(s,r,\mu)$ is a matrix $M$ of shape $t\times s^t \mu$ with entries taken from a set of $s$ symbols $S=\{0,\dots,s-1\}$ such that all $s^t$ ordered pairs of symbols appear exactly $\mu$ times in any choice of $t$ distinct rows of $M$.
\end{definition}
The book by Hedayat, Sloane and Stufken \cite{Sloane-OrthogonalArrays} is an excellent reference on the subject, see also sections 6 and 7 of Chapter III of the Handbook of Combinatorial Designs \cite{HandbookOfDesigns}. We note that our definition corresponds to the transpose of an orthogonal array in most of the literature.
\begin{example} \normalfont The following is an example of an $\OA_2(9,4,1)$,
\[
\begin{bmatrix}
0 & 0 & 0 & 1 & 1 & 1 & 2 & 2 & 2\\
0 & 1 & 2 & 0 & 1 & 2 & 0 & 1 & 2\\
0 & 1 & 2 & 2 & 0 & 1 & 1 & 2 & 0\\
0 & 1 & 2 & 1 & 2 & 0 & 2 & 0 & 1
\end{bmatrix}.
\]
It is easy to see that all ordered pairs of elements from $\{0,1,2\}$ appear exactly once comparing row $1$ to any other row. Comparing the second row with the third we find the pairs
\[(0,0),(1,1),(2,2),(0,2),(1,0),(2,1),(0,1),(1,2)\text{, and } (2,0).\]
So each possible tuple appears exactly one. The reader can check the rest of the cases similarly.
\end{example}

For orthogonal arrays of strength $2$ we have the following upper bound due to Bose and Bush \cite{Bose-Bush-OABound}.

\begin{theorem}
If there is an $\OA_2(s,r,\mu)$ then
\[r\leq \frac{\mu s^2-1}{s-1}.\]
\end{theorem}
An orthogonal array meeting the Bose-Bush bound with equality is said to be \textit{complete}. Now we can state Shrikhande's result \cite{Shrikhande-GHM}.

\begin{theorem}[Shrikhande]
Let $p$ be an odd prime, if there is a complete $\OA_2(p,r,\mu)$ then there is a $\BH(p^2\mu,p)$.
\end{theorem}
\begin{proof}
Let $A$ be a complete $\OA_2(p,r,\mu)$, so $A$ has shape $r\times \mu p^2$ where $(p-1)r+1=\mu p^2$. We create a $\mu p^2\times \mu p^2$ having the following block shape
\[H=\left[
\begin{array}{c}
\mathbf{1}_{\mu^2 p}^{\intercal}\\
\hline
A_1\\
\hline
A_2\\
\hline
\vdots\\
\hline
A_{p-1}
\end{array}
\right].
\]
Here $A_k$ is the $r\times\mu p^2$ matrix whose $(i,j)$-th entry is $(A_k)_{ij}=\zeta_p^{ka_{ij}}$, where $\zeta_p$ is a primitive $p$-th root of unity. Since $A$ is an orthogonal array the inner product of two  rows $r_i^k$ and $r_j^{\ell}$ in blocks $A_k$ and $A_{\ell}$, where $1\leq i<j\leq r$ is
\[r_i^{k}\cdot r_j^{\ell} = \sum_t \zeta_p^{ka_{it}-\ell a_{jt}}= p\mu \sum_{t} \zeta_p^{kt}=0.\]
The inner product of rows $r_i^{k}$ and $r_i^{\ell}$ in blocks $A_k$ and $A_{\ell}$ where $k<\ell$ is 
\[r_i^{k}\cdot r_i^{\ell}=\sum_t \zeta_p^{(k-\ell)a_{it}}=\mu \sum_{t} \zeta_p^{(k-\ell)t}=0.\]
Similarly the inner product of $\mathbf{1}_{\mu p^2}$ with any other row $r_i^k$ in block $A_k$ is
\[\mathbf{1}_{\mu p^2}\cdot r_i^k=\sum_t\zeta_p^{ka_{it}}=\mu \sum_t\zeta_p^{kt}=0.\]
This shows that $H$ is a $\BH(p^2\mu,p)$.\qedhere
\end{proof}

We note however that in part due to the relationship between orthogonal arrays, nets, and projective planes \cite{Drake}, constructing a complete orthogonal array is hard. So Shrikhande's Construction is not as useful for the construction of Butson-type Hadamard matrices. More interesting is the partial converse that Shrikhande gives in his paper \cite{Shrikhande-GHM}, that Butson-type Hadamard matrices can be used to construct orthogonal arrays. In particular Shrikhande finds using Butson's result on the existence of $\BH(2p^n,p)$ matrices that for every prime $p$ there is an $\OA_2(p,2\frac{p^{n+1}-1}{p-1}-1,2p^{n-1})$. For $p$ odd, the Bose-Bush bound for an orthogonal array in $s=p$ symbols and with index $\mu=2p^{n-1}$ is not integral, so it cannot be achieved. It turns out that for there values of $s$ and $\mu$ the value of $r=2\frac{p^{n+1}-1}{p-1}-1=1+2(p+p^2+\dots+p^{n})$ is the largest possible. Hence Shrikhande's Construction using Butson matrices gives examples of orthogonal arrays with largest possible constraint number.
\cleardoublepage
\chapter{Tables of Matrices}\label{app-MatrixTables}
\renewcommand*{\thepage}{B-\arabic{page}}
\setcounter{page}{1}
This appendix contains several examples of matrices. All matrices are written logarithmically, which means that if the entry $(i,j)$ of the matrix has value $k$, then it should be interpreted as $\zeta_m^{k}$, where $\zeta_m$ is a primitive $m$-th root of unity.

\section{Examples of de Launey's Construction}
All matrices here are Butson-type Hadamard matrices over the third roots, constructed using Theorem \ref{thm-DeLauneyConstruction}.
\begin{figure}[H]
\[
\left[
\begin{array}{*{12}{c}}
1&1&2&2&1&1&2&2&1&1&2&2\\
1&1&2&2&2&2&1&1&2&2&1&1\\
2&2&1&1&1&1&2&2&2&2&1&1\\
2&2&1&1&2&2&1&1&1&1&2&2\\
1&2&1&2&2&0&2&0&0&1&0&1\\
1&2&1&2&0&2&0&2&1&0&1&0\\
2&1&2&1&2&0&2&0&1&0&1&0\\
2&1&2&1&0&2&0&2&0&1&0&1\\
1&2&2&1&0&1&1&0&2&0&0&2\\
1&2&2&1&1&0&0&1&0&2&2&0\\
2&1&1&2&0&1&1&0&0&2&2&0\\
2&1&1&2&1&0&0&1&2&0&0&2\\
\end{array}
\right]
\]
\caption{A $\BH(12,3)$ matrix.}
\end{figure}
\pagebreak

\vspace*{\fill}
\begin{figure}[H]
\[
\left[
\begin{array}{*{24}{c}}
1&1&2&2&1&1&2&2&2&2&0&0&2&2&0&0&0&0&1&1&0&0&1&1\\
1&1&2&2&1&1&2&2&0&0&2&2&0&0&2&2&1&1&0&0&1&1&0&0\\
2&2&1&1&2&2&1&1&2&2&0&0&2&2&0&0&1&1&0&0&1&1&0&0\\
2&2&1&1&2&2&1&1&0&0&2&2&0&0&2&2&0&0&1&1&0&0&1&1\\
1&1&2&2&2&2&1&1&0&0&1&1&1&1&0&0&2&2&0&0&0&0&2&2\\
1&1&2&2&2&2&1&1&1&1&0&0&0&0&1&1&0&0&2&2&2&2&0&0\\
2&2&1&1&1&1&2&2&0&0&1&1&1&1&0&0&0&0&2&2&2&2&0&0\\
2&2&1&1&1&1&2&2&1&1&0&0&0&0&1&1&2&2&0&0&0&0&2&2\\
2&0&2&0&0&1&0&1&2&0&2&0&0&1&0&1&2&0&2&0&0&1&0&1\\
2&0&2&0&0&1&0&1&0&2&0&2&1&0&1&0&0&2&0&2&1&0&1&0\\
0&2&0&2&1&0&1&0&2&0&2&0&0&1&0&1&0&2&0&2&1&0&1&0\\
0&2&0&2&1&0&1&0&0&2&0&2&1&0&1&0&2&0&2&0&0&1&0&1\\
2&0&2&0&1&0&1&0&0&1&1&0&2&0&0&2&0&1&1&0&2&0&0&2\\
2&0&2&0&1&0&1&0&1&0&0&1&0&2&2&0&1&0&0&1&0&2&2&0\\
0&2&0&2&0&1&0&1&0&1&1&0&2&0&0&2&1&0&0&1&0&2&2&0\\
0&2&0&2&0&1&0&1&1&0&0&1&0&2&2&0&0&1&1&0&2&0&0&2\\
0&1&1&0&2&0&0&2&2&0&2&0&1&0&1&0&0&1&1&0&0&2&2&0\\
0&1&1&0&2&0&0&2&0&2&0&2&0&1&0&1&1&0&0&1&2&0&0&2\\
1&0&0&1&0&2&2&0&2&0&2&0&1&0&1&0&1&0&0&1&2&0&0&2\\
1&0&0&1&0&2&2&0&0&2&0&2&0&1&0&1&0&1&1&0&0&2&2&0\\
0&1&1&0&0&2&2&0&0&1&1&0&0&2&2&0&2&0&2&0&1&0&1&0\\
0&1&1&0&0&2&2&0&1&0&0&1&2&0&0&2&0&2&0&2&0&1&0&1\\
1&0&0&1&2&0&0&2&0&1&1&0&0&2&2&0&0&2&0&2&0&1&0&1\\
1&0&0&1&2&0&0&2&1&0&0&1&2&0&0&2&2&0&2&0&1&0&1&0\\
\end{array}
\right]
\]
\caption{A $\BH(24,3)$ matrix.}
\end{figure}
\vfill
\pagebreak

\begin{landscape}
\scriptsize
\[
\left[
\begin{array}{*{48}{c}}

2&2&0&0&2&2&0&0&0&0&1&1&0&0&1&1&2&2&0&0&2&2&0&0&0&0&1&1&0&0&1&1&2&2&0&0&2&2&0&0&0&0&1&1&0&0&1&1\\
2&2&0&0&2&2&0&0&0&0&1&1&0&0&1&1&0&0&2&2&0&0&2&2&1&1&0&0&1&1&0&0&0&0&2&2&0&0&2&2&1&1&0&0&1&1&0&0\\
0&0&2&2&0&0&2&2&1&1&0&0&1&1&0&0&2&2&0&0&2&2&0&0&0&0&1&1&0&0&1&1&0&0&2&2&0&0&2&2&1&1&0&0&1&1&0&0\\
0&0&2&2&0&0&2&2&1&1&0&0&1&1&0&0&0&0&2&2&0&0&2&2&1&1&0&0&1&1&0&0&2&2&0&0&2&2&0&0&0&0&1&1&0&0&1&1\\
2&2&0&0&2&2&0&0&1&1&0&0&1&1&0&0&0&0&1&1&1&1&0&0&2&2&0&0&0&0&2&2&0&0&1&1&1&1&0&0&2&2&0&0&0&0&2&2\\
2&2&0&0&2&2&0&0&1&1&0&0&1&1&0&0&1&1&0&0&0&0&1&1&0&0&2&2&2&2&0&0&1&1&0&0&0&0&1&1&0&0&2&2&2&2&0&0\\
0&0&2&2&0&0&2&2&0&0&1&1&0&0&1&1&0&0&1&1&1&1&0&0&2&2&0&0&0&0&2&2&1&1&0&0&0&0&1&1&0&0&2&2&2&2&0&0\\
0&0&2&2&0&0&2&2&0&0&1&1&0&0&1&1&1&1&0&0&0&0&1&1&0&0&2&2&2&2&0&0&0&0&1&1&1&1&0&0&2&2&0&0&0&0&2&2\\
0&0&1&1&1&1&0&0&2&2&0&0&0&0&2&2&2&2&0&0&2&2&0&0&1&1&0&0&1&1&0&0&0&0&1&1&1&1&0&0&0&0&2&2&2&2&0&0\\
0&0&1&1&1&1&0&0&2&2&0&0&0&0&2&2&0&0&2&2&0&0&2&2&0&0&1&1&0&0&1&1&1&1&0&0&0&0&1&1&2&2&0&0&0&0&2&2\\
1&1&0&0&0&0&1&1&0&0&2&2&2&2&0&0&2&2&0&0&2&2&0&0&1&1&0&0&1&1&0&0&1&1&0&0&0&0&1&1&2&2&0&0&0&0&2&2\\
1&1&0&0&0&0&1&1&0&0&2&2&2&2&0&0&0&0&2&2&0&0&2&2&0&0&1&1&0&0&1&1&0&0&1&1&1&1&0&0&0&0&2&2&2&2&0&0\\
0&0&1&1&1&1&0&0&0&0&2&2&2&2&0&0&0&0&1&1&1&1&0&0&0&0&2&2&2&2&0&0&2&2&0&0&2&2&0&0&1&1&0&0&1&1&0&0\\
0&0&1&1&1&1&0&0&0&0&2&2&2&2&0&0&1&1&0&0&0&0&1&1&2&2&0&0&0&0&2&2&0&0&2&2&0&0&2&2&0&0&1&1&0&0&1&1\\
1&1&0&0&0&0&1&1&2&2&0&0&0&0&2&2&0&0&1&1&1&1&0&0&0&0&2&2&2&2&0&0&0&0&2&2&0&0&2&2&0&0&1&1&0&0&1&1\\
1&1&0&0&0&0&1&1&2&2&0&0&0&0&2&2&1&1&0&0&0&0&1&1&2&2&0&0&0&0&2&2&2&2&0&0&2&2&0&0&1&1&0&0&1&1&0&0\\
2&0&2&0&0&1&0&1&2&0&2&0&0&1&0&1&0&1&0&1&1&2&1&2&0&1&0&1&1&2&1&2&1&2&1&2&2&0&2&0&1&2&1&2&2&0&2&0\\
2&0&2&0&0&1&0&1&2&0&2&0&0&1&0&1&1&0&1&0&2&1&2&1&1&0&1&0&2&1&2&1&2&1&2&1&0&2&0&2&2&1&2&1&0&2&0&2\\
0&2&0&2&1&0&1&0&0&2&0&2&1&0&1&0&0&1&0&1&1&2&1&2&0&1&0&1&1&2&1&2&2&1&2&1&0&2&0&2&2&1&2&1&0&2&0&2\\
0&2&0&2&1&0&1&0&0&2&0&2&1&0&1&0&1&0&1&0&2&1&2&1&1&0&1&0&2&1&2&1&1&2&1&2&2&0&2&0&1&2&1&2&2&0&2&0\\
2&0&2&0&0&1&0&1&0&2&0&2&1&0&1&0&1&2&1&2&0&1&0&1&2&1&2&1&1&0&1&0&2&0&2&0&1&2&1&2&0&2&0&2&2&1&2&1\\
2&0&2&0&0&1&0&1&0&2&0&2&1&0&1&0&2&1&2&1&1&0&1&0&1&2&1&2&0&1&0&1&0&2&0&2&2&1&2&1&2&0&2&0&1&2&1&2\\
0&2&0&2&1&0&1&0&2&0&2&0&0&1&0&1&1&2&1&2&0&1&0&1&2&1&2&1&1&0&1&0&0&2&0&2&2&1&2&1&2&0&2&0&1&2&1&2\\
0&2&0&2&1&0&1&0&2&0&2&0&0&1&0&1&2&1&2&1&1&0&1&0&1&2&1&2&0&1&0&1&2&0&2&0&1&2&1&2&0&2&0&2&2&1&2&1\\
0&1&1&0&2&0&0&2&0&1&1&0&2&0&0&2&0&1&0&1&2&1&2&1&0&1&0&1&2&1&2&1&2&0&2&0&2&1&2&1&0&2&0&2&1&2&1&2\\
0&1&1&0&2&0&0&2&0&1&1&0&2&0&0&2&1&0&1&0&1&2&1&2&1&0&1&0&1&2&1&2&0&2&0&2&1&2&1&2&2&0&2&0&2&1&2&1\\
1&0&0&1&0&2&2&0&1&0&0&1&0&2&2&0&0&1&0&1&2&1&2&1&0&1&0&1&2&1&2&1&0&2&0&2&1&2&1&2&2&0&2&0&2&1&2&1\\
1&0&0&1&0&2&2&0&1&0&0&1&0&2&2&0&1&0&1&0&1&2&1&2&1&0&1&0&1&2&1&2&2&0&2&0&2&1&2&1&0&2&0&2&1&2&1&2\\
0&1&1&0&2&0&0&2&1&0&0&1&0&2&2&0&1&2&1&2&1&0&1&0&2&1&2&1&0&1&0&1&1&2&1&2&0&2&0&2&1&2&1&2&0&2&0&2\\
0&1&1&0&2&0&0&2&1&0&0&1&0&2&2&0&2&1&2&1&0&1&0&1&1&2&1&2&1&0&1&0&2&1&2&1&2&0&2&0&2&1&2&1&2&0&2&0\\
1&0&0&1&0&2&2&0&0&1&1&0&2&0&0&2&1&2&1&2&1&0&1&0&2&1&2&1&0&1&0&1&2&1&2&1&2&0&2&0&2&1&2&1&2&0&2&0\\
1&0&0&1&0&2&2&0&0&1&1&0&2&0&0&2&2&1&2&1&0&1&0&1&1&2&1&2&1&0&1&0&1&2&1&2&0&2&0&2&1&2&1&2&0&2&0&2\\
2&0&2&0&1&0&1&0&0&1&1&0&0&2&2&0&1&2&2&1&2&0&0&2&2&0&0&2&1&2&2&1&0&1&1&0&1&2&2&1&1&2&2&1&0&1&1&0\\
2&0&2&0&1&0&1&0&0&1&1&0&0&2&2&0&2&1&1&2&0&2&2&0&0&2&2&0&2&1&1&2&1&0&0&1&2&1&1&2&2&1&1&2&1&0&0&1\\
0&2&0&2&0&1&0&1&1&0&0&1&2&0&0&2&1&2&2&1&2&0&0&2&2&0&0&2&1&2&2&1&1&0&0&1&2&1&1&2&2&1&1&2&1&0&0&1\\
0&2&0&2&0&1&0&1&1&0&0&1&2&0&0&2&2&1&1&2&0&2&2&0&0&2&2&0&2&1&1&2&0&1&1&0&1&2&2&1&1&2&2&1&0&1&1&0\\
2&0&2&0&1&0&1&0&1&0&0&1&2&0&0&2&2&0&0&2&1&2&2&1&2&1&1&2&0&2&2&0&1&2&2&1&0&1&1&0&1&0&0&1&2&1&1&2\\
2&0&2&0&1&0&1&0&1&0&0&1&2&0&0&2&0&2&2&0&2&1&1&2&1&2&2&1&2&0&0&2&2&1&1&2&1&0&0&1&0&1&1&0&1&2&2&1\\
0&2&0&2&0&1&0&1&0&1&1&0&0&2&2&0&2&0&0&2&1&2&2&1&2&1&1&2&0&2&2&0&2&1&1&2&1&0&0&1&0&1&1&0&1&2&2&1\\
0&2&0&2&0&1&0&1&0&1&1&0&0&2&2&0&0&2&2&0&2&1&1&2&1&2&2&1&2&0&0&2&1&2&2&1&0&1&1&0&1&0&0&1&2&1&1&2\\
0&1&1&0&0&2&2&0&2&0&2&0&1&0&1&0&1&2&2&1&0&2&2&0&0&2&2&0&1&2&2&1&1&2&2&1&1&0&0&1&0&1&1&0&2&1&1&2\\
0&1&1&0&0&2&2&0&2&0&2&0&1&0&1&0&2&1&1&2&2&0&0&2&2&0&0&2&2&1&1&2&2&1&1&2&0&1&1&0&1&0&0&1&1&2&2&1\\
1&0&0&1&2&0&0&2&0&2&0&2&0&1&0&1&1&2&2&1&0&2&2&0&0&2&2&0&1&2&2&1&2&1&1&2&0&1&1&0&1&0&0&1&1&2&2&1\\
1&0&0&1&2&0&0&2&0&2&0&2&0&1&0&1&2&1&1&2&2&0&0&2&2&0&0&2&2&1&1&2&1&2&2&1&1&0&0&1&0&1&1&0&2&1&1&2\\
0&1&1&0&0&2&2&0&0&2&0&2&0&1&0&1&2&0&0&2&2&1&1&2&1&2&2&1&0&2&2&0&0&1&1&0&2&1&1&2&2&1&1&2&0&1&1&0\\
0&1&1&0&0&2&2&0&0&2&0&2&0&1&0&1&0&2&2&0&1&2&2&1&2&1&1&2&2&0&0&2&1&0&0&1&1&2&2&1&1&2&2&1&1&0&0&1\\
1&0&0&1&2&0&0&2&2&0&2&0&1&0&1&0&2&0&0&2&2&1&1&2&1&2&2&1&0&2&2&0&1&0&0&1&1&2&2&1&1&2&2&1&1&0&0&1\\
1&0&0&1&2&0&0&2&2&0&2&0&1&0&1&0&0&2&2&0&1&2&2&1&2&1&1&2&2&0&0&2&0&1&1&0&2&1&1&2&2&1&1&2&0&1&1&0\\
\end{array}
\right]
\]
\captionof{figure}{A $\BH(48,3)$ matrix.}
\end{landscape}

\section{Barba matrices over the third roots}
The following are the known Barba matrices over the third roots, these are in particular maximal determinant matrices. See Table \ref{tab-Maxdet3} for reference.

\[
B_4=
\begin{bmatrix}
2 & 1 & 1 & 1\\
1 & 2 & 1 & 1\\
1 & 1 & 2 & 1\\
1 & 1 & 1 & 2
\end{bmatrix}
\ \ \ 
B_7=
\begin{bmatrix}
1&1&1&2&1&2&2\\
2&1&1&1&2&1&2\\
2&2&1&1&1&2&1\\
1&2&2&1&1&1&2\\
2&1&2&2&1&1&1\\
1&2&1&2&2&1&1\\
1&1&2&1&2&2&1
\end{bmatrix}
\]
\captionof{figure}{The two unique Barba matrices over the third roots with two distinct entries.}

\[
B_{10}=\left[
\begin{array}{*{10}{c}}
0&2&1&2&1&1&2&2&2&2\\
2&0&2&2&2&1&2&1&1&2\\
1&2&0&1&2&2&2&2&1&2\\
2&2&1&0&2&2&2&1&2&1\\
1&2&2&2&0&2&1&1&2&2\\
1&1&2&2&2&0&2&2&2&1\\
2&2&2&2&1&2&0&2&1&1\\
2&1&2&1&1&2&2&0&2&2\\
2&1&1&2&2&2&1&2&0&2\\
2&2&2&1&2&1&1&2&2&0
\end{array}
\right]
\]
\captionof{figure}{A Barba matrix of order $10$ in the Bose-Mesner algebra of the Petersen graph.}

\[
B_{13}=\left[\begin{array}{*{13}{c}}
0&1&2&2&1&2&2&1&1&1&1&2&2\\
1&0&2&2&2&1&1&1&1&2&2&1&2\\
2&2&0&1&2&1&2&2&1&1&1&1&2\\
2&2&1&0&2&2&1&1&1&1&2&2&1\\
1&2&2&2&0&2&1&2&2&1&1&1&1\\
2&1&1&2&2&0&2&1&2&2&1&1&1\\
2&1&2&1&1&2&0&2&1&2&2&1&1\\
1&1&2&1&2&1&2&0&2&1&2&2&1\\
1&1&1&1&2&2&1&2&0&2&1&2&2\\
1&2&1&1&1&2&2&1&2&0&2&1&2\\
1&2&1&2&1&1&2&2&1&2&0&2&1\\
2&1&1&2&1&1&1&2&2&1&2&0&2\\
2&2&2&1&1&1&1&1&2&2&1&2&0
\end{array}\right]
\]
\captionof{figure}{A Barba matrix of order $13$ in the Bose-Mesner algebra of the Paley graph.}

\section{Large determinant matrices over the third roots}
Below we include two matrices with entries on the third roots of unity achieving large values of the determinant at orders $n=11,$ $14$, and $16$. Notice that the Gram matrices at orders $n=11$ and $14$ have Ehlich block type structure. 
\vspace*{\fill}
\[
M_{11}=\left[
\begin{array}{*{11}{c}}
0&1&1&1&1&0&0&1&2&2&0\\
2&0&1&1&1&1&2&1&0&2&1\\
0&0&0&1&2&1&2&1&2&0&0\\
0&0&1&2&2&0&2&0&2&2&1\\
1&1&0&1&2&2&1&0&0&2&2\\
1&2&2&1&0&2&2&2&2&2&2\\
2&1&0&2&0&1&1&2&2&2&2\\
2&1&2&1&0&2&1&0&2&0&1\\
1&1&0&0&1&0&2&2&1&0&1\\
1&1&2&2&1&1&2&0&1&1&0\\
2&1&2&0&2&0&2&2&0&1&0
\end{array}
\right]
\]
\vspace*{48pt}
\[M_{11}M_{11}^*=
\left[
\begin{array}{*{11}{c}}
11&2&2&2&-&-&-&-&-&-&-\\
2&11&2&2&-&-&-&-&-&-&-\\
2&2&11&2&-&-&-&-&-&-&-\\
2&2&2&11&-&-&-&-&-&-&-\\
-&-&-&-&11&2&2&2&-&-&-\\
-&-&-&-&2&11&2&2&-&-&-\\
-&-&-&-&2&2&11&2&-&-&-\\
-&-&-&-&2&2&2&11&-&-&-\\
-&-&-&-&-&-&-&-&11&2&2\\
-&-&-&-&-&-&-&-&2&11&2\\
-&-&-&-&-&-&-&-&2&2&11
\end{array}
\right]
\]
\vfill
\newpage
\vspace*{\fill}
\[
M_{14}=\left[
\begin{array}{*{14}{c}}
0&2&1&0&0&1&1&2&2&0&2&2&0&0\\
1&2&0&0&0&2&1&1&0&1&2&1&0&2\\
2&1&0&0&1&2&0&2&2&1&2&2&0&1\\
1&1&0&1&0&0&1&2&2&1&1&0&0&0\\
2&1&0&0&0&2&1&0&1&0&2&0&1&0\\
2&2&0&2&0&2&2&2&2&0&0&0&0&2\\
0&1&2&2&1&2&1&2&0&2&0&1&2&0\\
0&0&1&2&1&2&1&1&1&2&1&0&0&1\\
2&0&1&1&2&2&1&2&0&0&1&1&2&1\\
0&0&0&2&2&1&1&2&1&1&0&1&1&1\\
0&2&2&1&0&2&0&2&1&2&1&1&1&1\\
1&1&1&2&0&0&0&1&0&0&0&2&1&1\\
1&1&1&1&1&1&0&0&1&0&0&1&0&2\\
0&1&0&1&1&1&2&1&0&0&1&2&1&2
\end{array}\right]
\]
\vspace*{48pt}
\[M_{14}M_{14}^*=
\left[
\begin{array}{*{14}{c}}
14&2&2&2&2&2&-&-&-&-&-&-&-&-\\
2&14&2&2&2&2&-&-&-&-&-&-&-&-\\
2&2&14&2&2&2&-&-&-&-&-&-&-&-\\
2&2&2&14&2&2&-&-&-&-&-&-&-&-\\
2&2&2&2&14&2&-&-&-&-&-&-&-&-\\
2&2&2&2&2&14&-&-&-&-&-&-&-&-\\
-&-&-&-&-&-&14&2&2&2&2&-&-&-\\
-&-&-&-&-&-&2&14&2&2&2&-&-&-\\
-&-&-&-&-&-&2&2&14&2&2&-&-&-\\
-&-&-&-&-&-&2&2&2&14&2&-&-&-\\
-&-&-&-&-&-&2&2&2&2&14&-&-&-\\
-&-&-&-&-&-&-&-&-&-&-&14&2&2\\
-&-&-&-&-&-&-&-&-&-&-&2&14&2\\
-&-&-&-&-&-&-&-&-&-&-&2&2&14
\end{array}
\right]
\]
\vfill
\newpage
\vspace*{\fill}
\[
M_{16}=\left[
\begin{array}{*{16}{c}}
1&0&0&0&2&1&1&0&0&1&0&1&0&1&1&2\\
0&2&0&0&0&0&2&0&2&2&1&0&2&2&2&2\\
0&0&2&1&0&0&1&1&2&1&1&2&0&0&1&2\\
1&2&1&0&2&1&1&2&2&2&0&2&2&0&2&0\\
0&0&1&0&2&2&2&2&2&1&2&0&0&2&1&0\\
0&2&2&1&2&2&1&1&0&2&2&1&2&1&2&2\\
0&2&2&0&0&1&0&1&1&1&0&0&0&1&2&0\\
1&0&1&1&2&0&0&2&1&2&1&0&2&1&1&2\\
1&2&2&1&2&1&2&1&2&0&0&0&1&2&1&2\\
0&2&1&0&0&0&1&2&0&0&1&1&1&1&1&0\\
0&0&1&2&0&1&1&1&0&2&0&0&2&2&1&1\\
1&2&1&2&2&0&2&1&2&1&1&1&0&1&2&1\\
2&0&2&0&1&0&2&1&2&2&0&1&2&1&1&0\\
2&2&1&1&1&0&1&2&0&1&0&0&0&2&2&2\\
1&1&2&0&2&0&1&1&0&1&1&0&2&2&0&0\\
0&1&1&1&0&1&2&2&2&1&0&1&2&1&0&2\\
\end{array}\right]
\]
\vspace*{48pt}
\[M_{16}M_{16}^*=
\left[
\begin{array}{*{16}{c}}
16&-2&1&1&1&1&1&1&1&1&1&1&1&1&1&1\\
-2&16&1&1&1&1&1&1&1&1&1&1&1&1&1&1\\
1&1&16&-2&1&1&1&1&1&1&1&1&1&1&1&1\\
1&1&-2&16&1&1&1&1&1&1&1&1&1&1&1&1\\
1&1&1&1&16&-2&1&1&1&1&1&1&1&1&1&1\\
1&1&1&1&-2&16&1&1&1&1&1&1&1&1&1&1\\
1&1&1&1&1&1&16&-2&1&1&1&1&1&1&1&1\\
1&1&1&1&1&1&-2&16&1&1&1&1&1&1&1&1\\
1&1&1&1&1&1&1&1&16&-2&1&1&1&1&1&1\\
1&1&1&1&1&1&1&1&-2&16&1&1&1&1&1&1\\
1&1&1&1&1&1&1&1&1&1&16&-2&1&1&1&1\\
1&1&1&1&1&1&1&1&1&1&-2&16&1&1&1&1\\
1&1&1&1&1&1&1&1&1&1&1&1&16&-2&1&1\\
1&1&1&1&1&1&1&1&1&1&1&1&-2&16&1&1\\
1&1&1&1&1&1&1&1&1&1&1&1&1&1&16&-2\\
1&1&1&1&1&1&1&1&1&1&1&1&1&1&-2&16
\end{array}
\right]
\]
\vfill
\cleardoublepage
\chapter{A Family of Generalised Weighing Matrices}\label{app-GWMs}
\renewcommand*{\thepage}{C-\arabic{page}}
\setcounter{page}{1}
In this appendix, we present  constructions of generalised weighing matrices from Hadamard matrices by using exterior products. The present author rediscovered the following theorems by Dandawate and Craigen:
\begin{theorem}[Dandawate, \cite{Dandawate-Thesis}]\normalfont
\label{Dandawate-theorem}
If there exists an Hadamard matrix of order $n$ then there is a $W({n\choose 2},\frac{n^2}{4})$.
\end{theorem}
\begin{proof}
Let $H$ be an Hadamard matrix of order $n$. All $2\times 2$ minors of a $\pm 1$ matrix have values $0$ or $\pm 2$. Therefore we have that $\frac{1}{2}\wedge^2 H$ is a $(0,\pm 1)$-matrix. By the Cauchy-Binet formula, Lemma \ref{lemma-CauchyBinet}, we have
\[(\frac{1}{2}\wedge^2 H)(\frac{1}{2}\wedge^2 H)^{\intercal}=\frac{1}{4}\wedge^2(HH^{\intercal})=\frac{n^2}{4}I_{n\choose 2}.\qedhere\]
\end{proof}
\begin{theorem}[Craigen, \cite{Craigen-Adjugation}] \normalfont
\label{Craigen-weighing-theorem}
If there exists an Hadamard matrix of order $n$ then there is a $W({n\choose 3},\frac{n^3}{16})$.
\end{theorem}
\begin{proof}
The proof is analogous to the one in Theorem \ref{Dandawate-theorem}. All $3\times 3$ minors of a $\pm 1$ matrix have value $0$ or $\pm 4$, from which the result follows.
\end{proof}

In addition, we found a new construction for generalised weighing matrices over the sixth roots of unity.
\begin{theorem}
\label{3cyclo-weighing}
If there exists a $BH(n,3)$ then there exists a $\GW({n\choose 2},\frac{n^2}{3};6)$.
\end{theorem}
\begin{proof}
Let $H$ be a $\BH(n,3)$. The set $\{1,\omega,\omega^2\}$ forms a multiplicative group. Therefore all $2\times 2$ minors of $H$ are of the shape $x-y$ where $x,y\in \{1,\omega,\omega^2\}$. So up to dephasing all possible minors are $\alpha=\omega-1$ and $\beta=\omega^2-1$. Both values have modulus $\sqrt{3}$ and the minimal polynomials of $\alpha/\sqrt{-3}$ and $\beta/\sqrt{-3}$ over $\Q$ are $X^2-X+1$ and $X^2+X+1$ respectively. It follows that $\alpha/\sqrt{-3}$ and $\beta/\sqrt{-3}$ are both $6$-th roots of unity. Therefore $\frac{1}{\sqrt{-3}}\wedge^2 H$ is a matrix whose nonzero entries belong to the $6$-th roots of unity. Finally, the Cauchy-Binet formula, Lemma \ref{lemma-CauchyBinet}, implies
\[(\frac{1}{\sqrt{-3}}\wedge^2 H)(\frac{1}{\sqrt{-3}}\wedge^2 H)^*=\frac{1}{3}\wedge^2(HH^*)=\frac{n^2}{3}I_{n\choose 2}.\qedhere\]
\end{proof}
\pagebreak
As an example take the symmetric $BH(6,3)$ matrix
\[H=
\begin{bmatrix}
1 & 1 & 1 & 1 & 1 & 1\\
1 & 1 & \omega & \omega^2 & \omega^2 &\omega\\
1 & \omega & 1 & \omega & \omega^2 & \omega^2\\
1 & \omega^2 & \omega & 1 & \omega & \omega^2\\
1 & \omega^2 & \omega^2 & \omega & 1 & \omega\\
1 & \omega & \omega^2 & \omega^2 & \omega & 1
\end{bmatrix},
\]

Then by Theorem \ref{3cyclo-weighing} the matrix

\[\frac{1}{\sqrt{-3}}\wedge^2 H=\left[\begin{smallmatrix}
0&-\omega^2&-\omega^2&\omega&\omega&-1&\omega&\omega&-1&0&-\omega^2&-\omega^2&0&1&1\\
-\omega^2&0&\omega^2&-\omega^2&0&-\omega^2&\omega&-1&\omega&-1&\omega&-1&\omega&-1&0\\
-\omega^2&\omega^2&-\omega&1&-\omega^2&0&0&\omega&-\omega&-\omega^2&-1&0&\omega^2&1&-\omega^2\\
\omega&-\omega^2&1&0&-\omega&\omega^2&-\omega^2&1&0&-\omega^2&\omega&0&-1&\omega&-1\\
\omega&0&-\omega^2&-\omega&\omega^2&-\omega^2&1&0&\omega&-\omega&-1&\omega&-\omega&0&1\\
-1&-\omega^2&0&\omega^2&-\omega^2&-\omega&1&-1&-\omega^2&0&0&\omega^2&\omega&-\omega&-\omega^2\\
\omega&\omega&0&-\omega^2&1&1&0&-\omega&-\omega&\omega^2&-\omega^2&1&1&0&-\omega^2\\
\omega&-1&\omega&1&0&-1&-\omega&\omega^2&0&\omega&0&-\omega^2&\omega&-\omega&\omega^2\\
-1&\omega&-\omega&0&\omega&-\omega^2&-\omega&0&\omega^2&-\omega^2&1&-1&0&\omega&-\omega\\
0&-1&-\omega^2&-\omega^2&-\omega&0&\omega^2&\omega&-\omega^2&-\omega&1&\omega^2&-1&-\omega^2&0\\
-\omega^2&\omega&-1&\omega&-1&0&-\omega^2&0&1&1&0&\omega^2&-\omega&-\omega&\omega^2\\
-\omega^2&-1&0&0&\omega&\omega^2&1&-\omega^2&-1&\omega^2&\omega^2&-\omega&-\omega^2&\omega&0\\
0&\omega&\omega^2&-1&-\omega&\omega&1&\omega&0&-1&-\omega&-\omega^2&\omega^2&0&\omega\\
1&-1&1&\omega&0&-\omega&0&-\omega&\omega&-\omega^2&-\omega&\omega&0&\omega^2&-\omega^2\\
1&0&-\omega^2&-1&1&-\omega^2&-\omega^2&\omega^2&-\omega&0&\omega^2&0&\omega&-\omega^2&-\omega\\
\end{smallmatrix}\right]
\]
is a $\GW(15,12;6)$. We conjecture that there are no other families of Butson-type matrices, or choices of $k$, such that the $k$-th exterior product construction yields a generalised weighing matrix.

\cleardoublepage
\renewcommand{\thepage}{}
\printindex

\cleardoublepage
\bibliographystyle{PhDbiblio-url}
\bibliography{../full-biblio.bib}

\end{document}